\documentclass[10pt,aps,prb, preprint,reqno]{amsart}
\usepackage{amsmath}
\usepackage{amssymb}
\usepackage{yfonts}
\usepackage{hyperref}
\usepackage{mathrsfs}
\usepackage {latexsym}
\usepackage{graphics}
\usepackage{tikz}
\usepackage{skull}
\usetikzlibrary{snakes}
\usetikzlibrary{arrows.meta}
\usetikzlibrary{decorations.pathmorphing}
\usetikzlibrary{er, positioning}
\allowdisplaybreaks

\DeclareMathSymbol{:}{\mathord}{operators}{"3A}

\newtheorem{theorem}{Theorem}[section]
\newtheorem{remark}{Remark}[section]
\newtheorem{lemma}[theorem]{Lemma}

\newtheorem{proposition}[theorem]{Proposition}

\newtheorem{define}{Definition}[section]  
\newtheorem{example}{Example}[section] 

\newcommand{\vertiii}[1]{{\left\vert\kern-0.25ex\left\vert\kern-0.25ex\left\vert #1 \right\vert\kern-0.25ex\right\vert\kern-0.25ex\right\vert}}

\allowdisplaybreaks 

\begin{document}
\title[Magnetohydrodynamics system]{Approximating three-dimensional magnetohydrodynamics system forced by space-time white noise}
 
\author{Kazuo Yamazaki}  
\address{Texas Tech University, Department of Mathematics and Statistics, Lubbock, TX, 79409-1042, U.S.A.; Phone: 806-834-6112; Fax: 806-742-1112; E-mail: (kyamazak@ttu.edu)}
\date{}
\maketitle

\begin{abstract} 
The magnetohydrodynamics system consists of the Navier-Stokes and Maxwell's equations, coupled through multiples of nonlinear terms. Such a system forced by space-time white noise has been studied by physicists for decades, and the rigorous proof of its solution theory has been recently established in Yamazaki (2019, arXiv:1910.04820 [math.AP]) using the theory of paracontrolled distributions and a technique of coupled renormalizations. When an equation is well-posed, and it is approximated by replacing the differentiation operator by reasonable discretization schemes with a parameter, it is widely believed that a solution of the approximating equation should converge to the solution of the original equation as the parameter approaches zero. We prove otherwise in the case of the three-dimensional magnetohydrodynamics system forced by space-time white noise. Specifically, it is proven that the limit of the solution to the approximating system with an additional 32 drift terms solves the original system. These 32 drift terms depend on the choice of approximations, can be calculated explicitly in the process of renormalizations, and essentially represent a spatial version of It$\hat{\mathrm{o}}$-Stratonovich correction terms. In particular, the proof relies on the technique of coupled renormalizations again, as well as taking advantage of the special structure of the magnetohydrodynamics system on many occasions. 
\vspace{5mm}

\end{abstract}
\footnote{2010MSC : 35B65; 35Q85; 35R60}

\tableofcontents 

\section{Introduction}\label{Introduction}

\subsection{Motivation from physics}\label{Motivation physics} 

The magnetohydrodynamics (MHD) system consists of the coupling of the Navier-Stokes (NS) equations from fluid mechanics and the Maxwell's equations from electromagnetism. Initiated by Alfv$\acute{\mathrm{e}}$n \cite{A42a, A42b} in 1942 (also \cite{B50, C51}), the study of the MHD concerns the properties of electrically conducting fluids. Lighthill \cite{L60} in 1960 extended the MHD system to the Hall-MHD system via addition of a Hall term, which is considered to be of fundamental importance in astrophysical plasma as it modifies small-scale turbulent activity, producing departure from MHD predictions. For instance, while fluid turbulence is often investigated through the NS equations, MHD turbulence occurs in laboratory settings such as fusion confinement devices (e.g., reversed field pinch), as well as astrophysical systems (e.g., earth interior \cite{GR95}, solar corona \cite{L34a}), and the conventional system of equations to which physicists, astronomers, and engineers turn for such a study is that of the MHD. 

While the mathematical analysis of the deterministic NS equations was pioneered by Leray \cite{L34} in 1934, the study of the NS equations forced by a noise that is white only in time was initiated by Bensoussan and Temam \cite{BT73} in 1973, and many followed suit to provide important works. On the other hand, a space-time white noise (STWN) $\xi$ is a distribution-based Gaussian field with a correlation of 
\begin{equation}\label{STWN}
\mathbb{E} [ \xi(t, x)\xi(s,w)] = \delta(t-s)\delta(x-w) 
\end{equation} 
where $\delta$ is a delta function and $\mathbb{E}$ is a mathematical expectation. Considering an STWN rather than a noise that is white only in time is not of purely mathematical interest. Indeed, a prominent example of a stochastic partial differential equation (SPDE) forced by STWN is the Kardar-Parisi-Zhang (KPZ) equation, which models interface growth \cite{KPZ86}: 
\begin{equation}\label{KPZ}
\frac{\partial h}{\partial t} = \nu \frac{\partial^{2} h}{(\partial x)^{2}} + (\frac{\partial h}{\partial x})^{2} + \xi 
\end{equation}
where $h(t,x) \in \mathbb{R}$ describes interface height, and $\nu \geq 0$ is a diffusivity constant. Hereafter, let us denote by $\partial_{t} \triangleq \frac{\partial}{\partial t}, D \triangleq \partial_{x} \triangleq \frac{\partial}{\partial x}, D_{j} \triangleq \partial_{x^{j}} \triangleq \frac{\partial}{\partial x^{j}}$ for $j \in \mathbb{N}$ and $dD$ to stand for ``$d$-dimensional,'' $A \lesssim_{a,b} B$ and $A \approx B$ in case there exists a non-negative constant $C = C(a,b)$ that depends on $a,b$ such that $A \leq C B$ and $A = CB$, respectively.  In this initial derivation \cite{KPZ86}, these physicists stated the condition \eqref{STWN} on \cite[pg. 889]{KPZ86}, took a spatial Fourier transform and solved the resulting equation using \eqref{STWN} (see also \cite{BG97, DDT07} on the KPZ equation). In fact, an abundance of such SPDEs forced by STWN may be readily found in the physics literature. Tracing back to 1956, Landau and Lifschitz \cite{LL56} had pioneered the investigation of hydrodynamic fluctuations with STWN. We also refer to \cite{ACHS81, GP75, HS92, SH77, ZS71} concerning thermal fluctuations through the Boussinesq system forced by STWN. Moreover, Forster, Nelson and Stephen \cite{FNS77} studied the behavior of velocity correlations generated by the NS equations and Burgers' equation forced by STWN for large distance and long time (also  \cite{YO86} on the NS equations forced by STWN). Finally, forcing the MHD and the Hall-MHD systems by noise is also an important tool in modeling various real-world phenomena. E.g., the generation of magnetic fields by dynamo activity plays an important role in astrophysical objects such as stars and clusters of galaxies; the gas in these objects is characterized by turbulent flows. In particular,  Camargo and Tasso \cite{CT92} applied the renormalization group theory to the MHD system forced by STWN and determined the effective viscosity and magnetic resistivity without solving the system. We also refer to \cite{GMD10, MM75} on the Hall-MHD system forced by a noise that is white in only time, and ferromagnets model forced by STWN, respectively. 

While physicists, astronomers, and engineers advanced their studies of these SPDEs forced by STWN for more than half a century, the most fundamental issue, specifically the existence of solutions to such SPDEs, did not receive sufficient attention. Actually, the recent rigorous analysis of this fundamental issue by mathematicians proved that solutions to these SPDEs by the classical definition do not exist. For clarity, let us explain the difficulty using Young's integral in the case of the 1D Burgers' equation:
\begin{equation}\label{Burgers' equation}
\partial_{t}u + u \partial_{x}u = \nu \partial_{x}^{2}u + \xi, 
\end{equation} 
where $u(t,x) \in \mathbb{R}$ denotes the velocity field and $\nu > 0$. We will also need the definition of the H$\ddot{\mathrm{o}}$lder space with negative degrees; for this purpose, let us recall the basic background of Besov spaces (\cite{BCD11, GIP15} and also \cite{I99} on the Littlewood-Paley theory on $\mathbb{T}^{3}$). Unless elaborated in detail, let us denote $\sum_{k \in \mathbb{Z}^{3}}$ by $\sum_{k}$. First, let us recall the Fourier transform 
\begin{equation}\label{Fourier}
\hat{f}(k) \triangleq \mathcal{F}_{\mathbb{T}^{3}}(f)(k) \triangleq \frac{1}{(2\pi)^{\frac{3}{2}}} \int_{\mathbb{T}^{3}} f(x) e^{-i x \cdot k} dx 
\end{equation} 
with its inverse denoted by $\mathcal{F}_{\mathbb{T}^{3}}^{-1}$, let $\mathcal{D}$ be the set of all smooth functions with compact support on $\mathbb{T}^{3}$, $\mathcal{D}'$ its dual and thus the set of all distributions on $\mathbb{T}^{3}$. We note that Zhu and Zhu on  \cite[pg. 36]{ZZ17} actually defined $\hat{f}(k) \triangleq (2\pi)^{-\frac{3}{2}} \int_{\mathbb{T}^{3}} f(x) e^{ix\cdot k}dx$; in Remark \ref{inconsistency} we will explain the reason for our choice. We let $\chi, \rho \in \mathcal{D}$ be non-negative, radial such that the support of $\chi$ is contained in a ball, while that of $\rho$ in an annulus, and satisfy 
\begin{equation*}
\begin{split}
&\chi(\xi) + \sum_{j \geq 0} \rho(2^{j} \xi) = 1 \hspace{1mm} \forall \hspace{1mm} \xi, \hspace{1mm} \text{supp} (\chi) \cap \text{supp} (\rho(2^{-j} \cdot)) = \emptyset \hspace{1mm} \forall \hspace{1mm} j \geq 1,\\
&\text{supp} (\rho(2^{-i} \cdot)) \cap \text{supp} (\rho(2^{-j} \cdot)) = \emptyset \text{ for } \lvert i - j \rvert > 1.
 \end{split}  
\end{equation*} 
We define $\rho_{j}(\cdot) \triangleq \rho(2^{-j}\cdot)$, and define Littlewood-Paley operator as $\Delta_{j} f(x) \triangleq  \mathcal{F}_{\mathbb{T}^{3}}^{-1}(\rho_{j} \mathcal{F}_{\mathbb{T}^{3}} (f))(x)$. Let us also write $S_{j} f \triangleq \sum_{i \leq j-1} \Delta_{j} f$. Now for $\alpha \in \mathbb{R}$ and $p, q \in [1,\infty]$, let us define the inhomogeneous Besov space 
\begin{equation}\label{Besov space}
B_{p,q}^{\alpha} (\mathbb{T}^{3}) \triangleq \{ f \in \mathcal{D}' (\mathbb{T}^{3}): \lVert f \rVert_{B_{p,q}^{\alpha} (\mathbb{T}^{3})} \triangleq  \lVert 2^{j\alpha}\lVert \Delta_{j} f \rVert_{L^{p}(\mathbb{T}^{3})} \rVert_{l^{q}(\{j \geq -1\})}  < \infty \}. 
\end{equation}  
The H$\ddot{\mathrm{o}}$lder-Besov space $\mathcal{C}^{\alpha}(\mathbb{T}^{3})$ is the special case when $p = q = \infty$, which hereafter will be denoted by $\mathcal{C}^{\alpha}$ when no confusion arises. We point out that 
\begin{equation}\label{regularity 0}
\lVert \cdot \rVert_{\beta} \lesssim \lVert \cdot \rVert_{L^{\infty}} \lesssim \lVert \cdot \rVert_{\mathcal{C}^{\alpha}} \hspace{1mm} \text{ if } \hspace{1mm} \beta \leq 0 \leq \alpha.
\end{equation} 
Informally, $\mathcal{C}^{\alpha}$ serves as an extension of the classical H$\ddot{\mathrm{o}}$lder space to negative degrees, and allows us to precisely evaluate the regularity of the STWN $\xi$ in \eqref{Burgers' equation} as 
\begin{equation}\label{regularity 1}
\xi \in \mathcal{C}^{\alpha}(\mathbb{T}^{d}) \text{ for } \alpha < - \frac{d+2}{2} 
\end{equation}
(e.g., \cite[Lemma 10.2]{H14}, also \cite{BK17}). Due to the diffusive term $\nu \partial_{x}^{2}u$ within \eqref{Burgers' equation}, one can expect that $u \in \mathcal{C}^{\alpha}(\mathbb{T})$ for $\alpha < \frac{1}{2}$. Aiming to attain at least the weakest notion of a solution, let us multiply the nonlinear term $u \partial_{x}u$ in \eqref{Burgers' equation} by a cut-off function $\psi$, integrate in space to obtain 
\begin{equation}\label{Young's integral}
\int_{\mathbb{T}} \psi(x) u(x) \partial_{x} u dx = \int_{\mathbb{T}} \psi(x) u(x) du(x). 
\end{equation}
In order to determine whether or not this integral is well-defined, let us turn to the result by Young \cite{Y36}, which states that a Stieltjes type integral $\int_{\mathbb{T}^{d}} f(x) dg(x)$ is well-defined if $f \in \mathcal{C}^{\alpha}(\mathbb{T}^{d}), g \in \mathcal{C}^{\beta} (\mathbb{T}^{d})$ for $\alpha, \beta \in (0, \infty),  \alpha + \beta > 1$ and there exists no common discontinuity, and thereby unfortunately confirms that the integral in \eqref{Young's integral} is ill-defined (cf. \cite{HW13}). This discussion already shows that SPDEs forced by STWN do not necessarily possess any solution by a conventional definition, which is the reason why they are called singular SPDEs, and thus raises an alarming question concerning the advancements on the SPDEs forced by STWN in the physics, astronomy and geosciences literature, all of which were based on the assumption of the existence of the solution. 

\subsection{Relevant mathematical results}

We denote by $b(t,x), j \triangleq \nabla \times b \in \mathbb{R}^{3}$ magnetic and current density fields, $\pi(t,x) \in \mathbb{R}$ pressure field, and the parameters $\eta, \iota \geq 0$ magnetic resistivity and Hall parameter, respectively. Under such notations, given $u_{0}(x) \triangleq u(0,x)$ and $b_{0}(x) \triangleq b(0,x)$ as initial data, the following system governs initial-value problems of the Hall-MHD system when $\iota > 0$ and the MHD system otherwise:
\begin{subequations}\label{Hall-MHD} 
\begin{align}
\partial_{t} u + (u\cdot\nabla) u + \nabla \pi =& \nu \Delta u + (b\cdot\nabla) b + \xi_{u}, \hspace{25mm} \nabla\cdot u = 0, \\
\partial_{t} b + (u\cdot\nabla) b =& \eta \Delta b + (b\cdot\nabla) u + \iota \nabla \times (j\times b) + \xi_{b}, \hspace{3mm} \nabla\cdot b = 0. 
\end{align}
\end{subequations} 
We point out that $(b\cdot\nabla) b$ and $\nabla\times (j\times b)$ respectively represent the Lorentz force and the Hall term, and that \eqref{Hall-MHD} reduces to the NS equations when $b \equiv 0$. As described in \eqref{STWN}, $\xi = (\xi_{u}, \xi_{b})$ in \eqref{Hall-MHD} is a six-dimensional vector-valued STWN and for all $i, j \in \{1,2,3\}, s, t \in \mathbb{R}_{+}$ and $x, w \in \mathbb{T}^{3}$, it is assumed to satisfy 
\begin{align}\label{STWN for MHD}
\mathbb{E}[\xi_{u}^{i} (t,x) \xi_{u}^{j}(s, w)] =& \mathbb{E} [ \xi_{u}^{i}(t,x) \xi_{b}^{j}(s,w)] \\
=& \mathbb{E}[\xi_{b}^{i}(t,x) \xi_{b}^{j} (s,w)] = \delta(t-s)\delta(x-w)\delta_{ij}, \nonumber
\end{align} 
where $\delta_{ij}$ will be crucial in deriving \eqref{covariance} (readers are referred to \cite[Remark 1.1]{Y19a} for details of this assumption). For brevity, let us write $y \triangleq (u,b), y_{0} \triangleq (u_{0}, b_{0})$.  

In order to further elaborate on the difficulty caused by the lack of regularity from STWN, let us recall the Bony's decomposition of a product $uf$ into three parts where the frequency of $u$ and $f$ are low, high and comparable, specifically 
\begin{align}\label{Bony's decomposition}
&uf = \sum_{i,j \geq -1} \Delta_{j} u \Delta_{i} f = \pi_{<}(u,f) + \pi_{>} (u,f) + \pi_{0}(u,f), \\
&\pi_{<}(u,f) \triangleq \sum_{j \geq -1} S_{j} u \Delta_{j} f, \pi_{>}(u,f) \triangleq \sum_{j \geq -1} \Delta_{j} u S_{j} f, \pi_{0}(u,f) \triangleq  \sum_{j, l \geq -1: \lvert l-j \rvert \leq 1} \Delta_{j} u \Delta_{l} f.\nonumber 
\end{align}
The terms $\pi_{<}(u,f)$ and $\pi_{>}(u,f)$ are called paraproducts while $\pi_{0}(u,f)$ the remainder. Concerning their estimates, we have the following important fact: 
\begin{lemma}\label{[Lemma 1.1][Y19a]}
\rm{(\cite[Lemma 2.1]{GIP15}, \cite[Proposition 2.3]{CC18})} Let $\alpha, \beta \in \mathbb{R}$. Then 
\begin{enumerate}
\item $\lVert \pi_{<}(u,f) \rVert_{\mathcal{C}^{\beta}} \lesssim \lVert u \rVert_{L^{\infty}} \lVert f \rVert_{\mathcal{C}^{\beta}} \text{ for } u \in L^{\infty}(\mathbb{T}^{3}), f \in \mathcal{C}^{\beta}(\mathbb{T}^{3}),$
\item $\lVert \pi_{>}(u,f) \rVert_{\mathcal{C}^{\alpha + \beta}} \lesssim \lVert u \rVert_{\mathcal{C}^{\alpha}} \lVert f \rVert_{\mathcal{C}^{\beta}} \text{ for } \beta < 0, f \in \mathcal{C}^{\alpha}(\mathbb{T}^{3}), g \in \mathcal{C}^{\beta}(\mathbb{T}^{3})$, 
\item $\lVert \pi_{0}(u,f) \rVert_{\mathcal{C}^{\alpha + \beta}} \lesssim \lVert u \rVert_{\mathcal{C}^{\alpha}} \lVert f \rVert_{\mathcal{C}^{\beta}} \text{ for } \alpha + \beta > 0, u \in \mathcal{C}^{\alpha}(\mathbb{T}^{3}),f \in \mathcal{C}^{\beta}(\mathbb{T}^{3})$. 
\item Consequently, $uf$ is well-defined for $u \in \mathcal{C}^{\alpha}(\mathbb{T}^{3}), f \in \mathcal{C}^{\beta}(\mathbb{T}^{3})$ if $\alpha + \beta > 0$ and $\lVert uf\rVert_{\mathcal{C}^{\text{min} \{\alpha, \beta\}}} \lesssim \lVert u \rVert_{\mathcal{C}^{\alpha}} \lVert f \rVert_{\mathcal{C}^{\beta}}$ (see \cite[Corollary 1.2]{Y19b} for details). 
\end{enumerate} 
\end{lemma} 
Concerning the 2D NS equations forced by STWN, Da Prato and Debussche considered $z$ that solves the Stokes equation so that $v \triangleq u - z$ and $q \triangleq \pi - p$ solve 
\begin{subequations}
\begin{align}
& \partial_{t} z = \nu \Delta z - \nabla p + \xi, \hspace{34mm} \nabla\cdot z = 0, \\
& \partial_{t} v = \nu \Delta v - \nabla q - \frac{1}{2} \text{div} [(v+z) \otimes (v+z)], \hspace{1mm} \nabla\cdot v = 0. 
\end{align}
\end{subequations} 
As discussed in \eqref{regularity 1}, because $\xi \in \mathcal{C}^{\alpha}(\mathbb{T}^{d})$ for $\alpha < - \frac{d+2}{2}$, it follows that $z \in \mathcal{C}^{\alpha}(\mathbb{T}^{d})$ for $\alpha < 1 - \frac{d}{2}$. In turn, this leads to $\text{div} (z \otimes z) \in \mathcal{C}^{\alpha}(\mathbb{T}^{d})$ for $\alpha < 1 - d$ and thus $v \in \mathcal{C}^{\alpha}(\mathbb{T}^{d})$ for $\alpha < 3-d$. Taking advantage of the fact that $v \otimes v$, as well as $v \otimes z$, are  both well-defined if $d = 2$ according to Lemma \ref{[Lemma 1.1][Y19a]} (4), Da Prato and Debussche \cite{DD02} were able to prove the well-posedness of the 2D NS equations forced by STWN (and similarly 2D stochastic quantization equations in \cite{DD03}). In contrast, both $v \otimes v$ and $v \otimes z$ are ill-defined in the 3D case, so that an extension of \cite{DD02} to the 3D case was an open problem for more than a decade. 

The resolution to this and many related problems was initiated by Lyons, the founding father of the rough path theory \cite{L98, LQ02}. A rough path is informally a continuous path on which a sequence of iterated path integrals may be constructed. While others realized the potential of the rough path theory (e.g., \cite{FV10b, GT10}), Hairer \cite{H11} was the first to directly apply the rough path theory and solve the Burgers' equation with STWN \eqref{Burgers' equation}. Subsequently, Hairer \cite{H13} provided a solution theory for the KPZ equation \eqref{KPZ} using rough path theory and Feynman diagrams (also \cite{HQ18} for more recent work on the KPZ equation). Then Hairer \cite{H14} and Gubinelli, Imkeller and Perkowski \cite{GIP15} independently provided two novel approaches to prove solution theory to general SPDEs forced by STWN, namely the theory of regularity structures, and the theory of paracontrolled distributions, respectively. In short, the theory of paracontrolled distributions is an elegant combination of techniques from harmonic analysis (e.g., commutator estimate Lemma \ref{Lemma 3.5}) and controlled rough path theory \cite{G04} and has been applied to $\Phi^{4}$ model from quantum field theory (QFT) \cite{CC18}, 3D NS equations \cite{ZZ15} and the 3D MHD system \cite{Y19a}. On the other hand, the theory of regularity structures from \cite{H14} (also \cite{H15}) is believed by many experts to be more general and have deeper potentials (e.g., \cite{CFG17, HM18}). In particular, using the theory of regularity structures, Hairer and Mattingly \cite{HM18a} provided a general approach to prove strong Feller property for some singular SPDEs such as the KPZ equation \eqref{KPZ} (also \cite{ZZ20} for the NS equations and \cite{Y20a} for the MHD system). 

\subsection{Statement of main result}

The issue of the singular SPDEs being ill-defined by a classical definition, as described using Young's integral \eqref{Young's integral} and Lemma \ref{[Lemma 1.1][Y19a]}, can be seen as an indirect consequence of the instability in approximations; let us explain this relationship through the KPZ equation \eqref{KPZ} as follows. Inspired by the Cole-Hopf transform, let us consider a multiplicative stochastic heat equation 
\begin{equation}\label{heat}
dZ = \nu \partial_{x}^{2} Z dt + Z dW \text{ where } \partial_{t}W = \xi, 
\end{equation}
let $Z_{\epsilon}$ solve the same equation as \eqref{heat} except with a mollified noise $\xi_{\epsilon}$ in which $k$-th Fourier component is multiplied by $\psi(k\epsilon)$ for a cut-off function $\psi$, and compute It$\hat{\mathrm{o}}$'s formula for $h_{\epsilon}(t,x) \triangleq \ln Z_{\epsilon}(t,x)$ (see \cite{M91} concerning the positivity of $Z_{\epsilon}$), in hope to attain \eqref{KPZ} as $\epsilon \searrow 0$. In fact, this approach shows that $h_{\epsilon}$ solves 
\begin{equation}\label{renormalization}
\partial_{t} h_{\epsilon} =  \nu \partial_{x}^{2} h_{\epsilon} + (\partial_{x} h_{\epsilon})^{2} - \frac{1}{2} \sum_{k\in \mathbb{Z}} \phi^{2}(k\epsilon) + \xi_{\epsilon}  \text{ where } \lim_{\epsilon \to 0} \sum_{k\in\mathbb{Z}} \phi^{2}(k\epsilon) = + \infty 
\end{equation} 
and therefore informally the equation \eqref{KPZ} minus infinity in the limit. This outcome clearly shows the necessity of the renormalization technique from QFT, which boils down to strategically subtracting off a large constant. It also indicates that for a singular SPDE, variations in approximations, even if they are reasonable, can lead to significantly distinct limits, which has been well-documented \cite{H12, HM12, HMW14, HV11}. In order to explain why this is so surprising, let us consider the following example of the 1D deterministic Burgers' equation \eqref{Burgers' equation} for simplicity. 

\begin{example}\label{Burgers' equation example}
Let us emphasize that the following computations are not intended to give any sharp result; its intent is only to give an idea for convenience of readers. Let us consider the Burgers' equation \eqref{Burgers' equation} with initial data of $u_{0}$, define 
\begin{equation}
D^{\epsilon} f(x) \triangleq \frac{ f(x+ \epsilon) - f(x)}{\epsilon}, \hspace{3mm} \Delta_{\epsilon} f(x) \triangleq \frac{ f(x+ \epsilon) - 2f(x) + f(x-\epsilon)}{\epsilon^{2}},
\end{equation}
and consider the following approximations of \eqref{Burgers' equation} parametrized by $\epsilon > 0$ with its initial data of $u_{0}^{\epsilon}$: 
\begin{equation}\label{Burgers' equation approximation}
\partial_{t} u^{\epsilon} + \frac{1}{2} D^{\epsilon} (u^{\epsilon})^{2}  = \Delta_{\epsilon} u^{\epsilon}. 
\end{equation}   
It is well-known that for any $T > 0$ fixed, both \eqref{Burgers' equation} with $\xi\equiv 0$ and \eqref{Burgers' equation approximation} have unique solutions $u, u^{\epsilon} \in C([0,T]; L^{\infty} (\mathbb{T}))$, respectively. Let us assume that $u_{0} - u_{0}^{\epsilon} \to 0$ in $L^{\infty}(\mathbb{T})$ as $\epsilon \searrow 0$. Then we can define $P_{t} \triangleq e^{t \Delta}$ and $P_{t}^{\epsilon} \triangleq e^{t \Delta_{\epsilon}}$, write both solutions of \eqref{Burgers' equation} with $\xi\equiv 0$ and \eqref{Burgers' equation approximation} in a mild formulation and estimate 
\begin{align}\label{estimate 1}
\lVert (u - u^{\epsilon})(t) \rVert_{L^{\infty}} \lesssim \lVert u_{0} - u_{0}^{\epsilon} \rVert_{L^{\infty}}  + \sum_{i=1}^{3} I_{i} 
\end{align} 
where 
\begin{align*}
&I_{1} \triangleq \int_{0}^{t} \lVert (P_{t-s} - P_{t-s}^{\epsilon}) \partial_{x} u^{2} \rVert_{L^{\infty}}ds, \hspace{3mm} I_{2} \triangleq \int_{0}^{t} \lVert P_{t-s}^{\epsilon} (\partial_{x} - D^{\epsilon}) u^{2} \rVert_{L^{\infty}}  ds, \\
&I_{3} \triangleq \int_{0}^{t}\lVert P_{t-s}^{\epsilon} D^{\epsilon} (u^{2}  -  (u^{\epsilon})^{2}) \rVert_{L^{\infty}} ds. 
\end{align*} 
First, we estimate 
\begin{align}
I_{1} \lesssim  \int_{0}^{t} \epsilon^{\frac{1}{4}} (t-s)^{-\frac{3}{4}} \lVert u^{2} \rVert_{C^{-\frac{1}{4}}} ds  \lesssim \epsilon^{\frac{1}{4}} t^{\frac{1}{4}} \sup_{s\in [0,T]} \lVert u(s) \rVert_{L^{\infty}}^{2}  \label{estimate 2}
\end{align} 
by Lemma \ref{Lemma 3.2} and \eqref{regularity 0}. Second, we estimate 
\begin{align}
I_{2} \lesssim \int_{0}^{t} (t-s)^{-\frac{3}{4}} \epsilon^{\frac{1}{8}} \lVert u^{2} \rVert_{\mathcal{C}^{0}} ds
\lesssim \epsilon^{\frac{1}{8}} t^{\frac{1}{4}} \sup_{s\in [0,T]} \lVert u(s) \rVert_{L^{\infty}}^{2} \label{estimate 3}
\end{align} 
by Lemmas \ref{Lemma 3.1} and \ref{Lemma 3.15}. Third, we see that there exists a uniform constant $C_{0} > 0$ such that 
\begin{align}
I_{3} \leq C_{0} t^{\frac{1}{4}} (\sup_{s\in [0,T]} \lVert (u, u^{\epsilon})(s) \rVert_{L^{\infty}}) \sup_{s\in [0,t]} \lVert (u - u^{\epsilon})(s) \rVert_{L^{\infty}} \label{estimate 4}
\end{align}
by Lemmas \ref{Lemma 3.1} and \ref{Lemma 3.3}. Applying these estimates \eqref{estimate 2}-\eqref{estimate 4} to \eqref{estimate 1} shows that for all $t < (2C_{0} \sup_{s \in [0,T]} \lVert (u, u^{\epsilon})(s) \rVert_{L^{\infty}})^{-4} \triangleq t_{0}$, $\sup_{s\in [0,t]} \lVert (u - u^{\epsilon})(t) \rVert_{L^{\infty}} \to 0$ as $\epsilon \searrow 0$ for all $t \in [0,  t_{0})$. We can restart from $t_{0}$, and show that $\sup_{s\in [0,t]} \lVert (u - u^{\epsilon})(t) \rVert_{L^{\infty}} \to 0$ as $\epsilon \searrow 0$ for all $t \in [0,T]$.
\end{example} 
As this example showed, when an equation is well-posed, even locally in time, then one expects that a solution to an approximating equation depending on a parameter $\epsilon > 0$ converges to the solution of the original equation as $\epsilon \searrow 0$. However, Hairer \cite{H12} showed that a certain family of solutions to a semilinear variation of a heat equation forced by STWN that depends on $\epsilon > 0$ converges as $\epsilon \searrow 0$ to a solution to the same equation with $\epsilon =0$, but with an addition of a correction term, akin to that of It$\hat{\mathrm{o}}$-Stratonovich correction (e.g., \cite[Theorem 2.3.5]{K90}). Moreover, Hairer and Voss \cite{HV11} argued that for the Burgers' equation \eqref{Burgers' equation}, different finite-difference schemes converge to different limiting processes as the mesh size tends to zero, and provided supporting evidence via numerical analysis; subsequently, a rigorous proof of their argument was given by Hairer and Maas \cite{HM12}. The results from \cite{HM12} were extended by Hairer, Maas and Weber \cite{HMW14} from the case when the nonlinearity is a gradient type and the noise is additive to a case when the nonlinearity is a non-gradient type and the noise is multiplicative. 

In this manuscript we also study the spatial approximations of the 3D MHD system \eqref{Hall-MHD} with $\iota = 0$ (see Remark \ref{Remark 1.5} in case of the Hall-MHD system \eqref{Hall-MHD} with $\iota > 0$) . Following \cite[Equation (11)]{Y19a} we apply Leray projection $\mathcal{P}$ to \eqref{Hall-MHD} to study for $i_{0} \in \{1,2,3\}$, 
\begin{subequations}\label{MHD}
\begin{align}
du^{i_{0}} - \Delta u^{i_{0}} dt =& \sum_{i = 1}^{3} \mathcal{P}^{i_{0}i} dW_{u}^{i}- \frac{1}{2} \sum_{i, j=1}^{3} \mathcal{P}^{i_{0}i} D_{j} (u^{i} u^{j}) dt \nonumber\\
& + \frac{1}{2} \sum_{i, j=1}^{3} \mathcal{P}^{i_{0}i} D_{j} (b^{i} b^{j}) dt, \hspace{3mm} \nabla\cdot u = 0, \\
db^{i_{0}} - \Delta b^{i_{0}} dt =& \sum_{i = 1}^{3} \mathcal{P}^{i_{0}i} dW_{b}^{i} - \frac{1}{2} \sum_{i, j=1}^{3} \mathcal{P}^{i_{0}i} D_{j} (b^{i} u^{j}) dt \nonumber\\
&+ \frac{1}{2} \sum_{i, j=1}^{3} \mathcal{P}^{i_{0}i} D_{j} (u^{i} b^{j}) dt, \hspace{3mm} \nabla\cdot b = 0, 
\end{align} 
\end{subequations}
where we follow \cite[pg. 779]{HMW14} and assume that $W_{u}$ and $W_{b}$ are both cylindrical Wiener processes on $L^{2}(\mathbb{T}^{3})$ which satisfy $\hat{W}_{u} (0) = \hat{W}_{b} (0) = 0$. For brevity let us denote by $W \triangleq (W_{u}, W_{b})$. In order to describe the approximation of the system \eqref{MHD}, let us define the following approximation operators. 

\begin{define}\label{Definition of approximation}
For some $L_{0}> 0$ fixed, we define 
\begin{equation}\label{f}
f(x) \triangleq 
\begin{cases}
\tilde{f}(x) & \text{ if } \max_{j\in \{1,2,3\}} \lvert x^{j} \rvert \leq L_{0}, \\
+ \infty& \text{ otherwise,}
\end{cases}
\end{equation} 
where $\tilde{f}: \hspace{1mm} \mathbb{R}^{3} \mapsto [0,\infty)$ is radial, satisfying $\tilde{f}(0) = 1$, and $\lvert D^{k} \tilde{f}(x) \rvert \lesssim \frac{1}{\lvert x \rvert^{\lvert k \rvert} - 1} + C$ for some constant $C > 0$ if $\max_{ j \in \{1,2,3\}} \lvert x^{j} \rvert \leq 3 L_{0}$ and $\lvert k \rvert \leq 5$, and 
\begin{equation}\label{lower bound}
\tilde{f}(x) \geq c_{f}> 0 \text{ for all } x. 
\end{equation}
Let us also define $\bar{c}_{f}\triangleq c_{f} \wedge 1$ where $A \wedge B \triangleq \min\{A, B\}$. We also define 
\begin{equation}\label{g}
g(x) \triangleq \frac{ e^{iax} - e^{-ibx}}{(a+b) x}  
\end{equation} 
for $x \neq 0, a, b \geq 0$ such that $a+ b > 0$. Finally, for some $\bar{L}_{0} \in (0, \frac{L_{0}}{2})$, we define $h_{u}, h_{b} \in C^{1} (\{x \in \mathbb{R}^{3}: \lvert x \rvert \leq \bar{L}_{0} \})$ to be bounded, radial, and satisfy $h_{u}(0) = h_{b}(0) = 1$, $supp(h_{u}), supp(h_{b}) \subset \{x \in \mathbb{R}^{3}: \lvert x \rvert \leq \frac{L_{0}}{2} \}$. With such $f, g$ $h_{u}$ and $h_{b}$ defined, for $\epsilon > 0$, we define our approximations $\Delta_{\epsilon}, D_{j}^{\epsilon}, H_{u,\epsilon}, H_{b,\epsilon}$ to satisfy 
\begin{subequations}
\begin{align}
&\widehat{\Delta_{\epsilon} v}(k) \triangleq - \lvert k \rvert^{2} f(\epsilon k) \hat{v}(k), \hspace{6mm} \widehat{D_{j}^{\epsilon} v}(k) \triangleq k^{j} g(\epsilon k^{j}) \hat{v}(k), \label{approximation 1}\\
&\widehat{H_{u,\epsilon} W_{u}}(k) \triangleq h_{u}(\epsilon k) \hat{W}_{u}(k), \hspace{3mm} \widehat{H_{b,\epsilon} W_{b}}(k) \triangleq h_{b}(\epsilon k) \hat{W}_{b}(k). \label{approximation 2} 
\end{align}
\end{subequations}  
\end{define}
From \eqref{g}, we know that $\lim_{x \searrow 0} g(x) = i$; indeed, 
\begin{equation}\label{limit of g}
\lim_{x\searrow 0} g(x) = \lim_{x\searrow 0} \frac{ e^{iax} - e^{-ibx}}{(a+b) x} = \lim_{x\searrow 0} \frac{ ia e^{iax} + ib e^{-ibx}}{a+b} = i.
\end{equation} 
It immediately follows that $g$ is bounded. 

\begin{remark}\label{inconsistency}
As we pointed out, our definition of the Fourier transform in \eqref{Fourier} varies from that in \cite[pg. 36]{ZZ17}. The reason for our choice is that if $\hat{f}(k) \triangleq (2\pi)^{-\frac{3}{2}} \int_{\mathbb{T}^{3}} f(x) e^{ix\cdot k}dx$, then the Fourier symbol of $D_{j}$ becomes $-ik^{j}$ instead of $ik^{j}$. This leads to an inconsistency because in \eqref{approximation 1} we define $D_{j}^{\epsilon}$ to satisfy $\widehat{D_{j}^{\epsilon} v}(k) = k^{j} g(\epsilon k^{j}) \hat{v}(k)$ and approximate $D_{j}$ so that we must have $\lim_{\epsilon \searrow 0} g(\epsilon k^{j}) = -i$, in contrast to \eqref{limit of g}. 
\end{remark}

We are now ready to write down the approximation of \eqref{MHD}: for $\epsilon > 0$, 
\begin{subequations}\label{MHD approximation}
\begin{align}
du^{\epsilon, i_{0}} &= (\Delta_{\epsilon} u^{\epsilon, i_{0}} - \frac{1}{2} \sum_{i, j=1}^{3} \mathcal{P}^{i_{0} i} D_{j}^{\epsilon} (u^{\epsilon, i} u^{\epsilon, j} - b^{\epsilon, i} b^{\epsilon, j}  \nonumber \\
+& \sum_{i_{1} =1}^{3} [C_{1,u}^{\epsilon, i i_{1} j} + \tilde{C}_{1,u}^{\epsilon, i i_{1} j} - C_{2,u}^{\epsilon, i i_{1} j} - \tilde{C}_{2,u}^{\epsilon, i i_{1} j} \nonumber\\
&\hspace{3mm} +  C_{1,u}^{\epsilon, ji_{1}i} + \tilde{C}_{1,u}^{\epsilon, j i_{1} i} - C_{2,u}^{\epsilon, j i_{1} i} - \tilde{C}_{2,u}^{\epsilon, j i_{1} i} ] u^{\epsilon, i_{1}} \nonumber\\
+& \sum_{i_{1} =1}^{3} [C_{1,b}^{\epsilon, i i_{1} j} + \tilde{C}_{1,b}^{\epsilon, i i_{1} j} - C_{2,b}^{\epsilon, i i_{1} j} - \tilde{C}_{2,b}^{\epsilon, i i_{1} j} \nonumber\\
&\hspace{3mm}+  C_{1,b}^{\epsilon, j i_{1} i} + \tilde{C}_{1,b}^{\epsilon, j i_{1} i} - C_{2,b}^{\epsilon, j i_{1} i} - \tilde{C}_{2,b}^{\epsilon, j i_{1} i} ] b^{\epsilon, i_{1}} ))dt + \sum_{i=1}^{3} \mathcal{P}^{i_{0} i} H_{u,\epsilon} dW_{u}^{i}, \\
db^{\epsilon, i_{0}} &= (\Delta_{\epsilon} b^{\epsilon, i_{0}} - \frac{1}{2} \sum_{i, j=1}^{3} \mathcal{P}^{i_{0} i} D_{j}^{\epsilon} (b^{\epsilon, i} u^{\epsilon, j} - u^{\epsilon, i} b^{\epsilon, j}  \nonumber \\
+& \sum_{i_{1} =1}^{3} [C_{3,u}^{\epsilon, i i_{1} j} + \tilde{C}_{3,u}^{\epsilon, i i_{1} j} - C_{4,u}^{\epsilon, i i_{1} j} - \tilde{C}_{4,u}^{\epsilon, i i_{1} j} \nonumber\\
&\hspace{3mm} + C_{3,u}^{\epsilon, j i_{1} i} + \tilde{C}_{3,u}^{\epsilon, j i_{1} i} - C_{4,u}^{\epsilon, j i_{1} i} - \tilde{C}_{4,u}^{\epsilon, j i_{1} i} ] u^{\epsilon, i_{1}} \nonumber\\
+& \sum_{i_{1} =1}^{3} [C_{3,b}^{\epsilon, i i_{1} j} + \tilde{C}_{3,b}^{\epsilon, i i_{1} j} - C_{4,b}^{\epsilon, i i_{1} j} - \tilde{C}_{4,b}^{\epsilon, i i_{1} j} \nonumber\\
& \hspace{3mm} + C_{3,b}^{\epsilon, j i_{1} i} + \tilde{C}_{3,b}^{\epsilon, j i_{1} i} - C_{4,b}^{\epsilon, j i_{1} i} - \tilde{C}_{4,b}^{\epsilon, j i_{1} i} ] b^{\epsilon, i_{1}} ))dt + \sum_{i=1}^{3} \mathcal{P}^{i_{0} i} H_{b,\epsilon} dW_{b}^{i}, 
\end{align} 
\end{subequations} 
where $C_{k,u}^{\epsilon, i i_{1} j}, C_{k,b}^{\epsilon, i i_{1} j}, \tilde{C}_{k,u}^{\epsilon, i i_{1} j}$ and $\tilde{C}_{k,b}^{\epsilon, i i_{1} j}$ are computed explicitly in \eqref{C2uepsilonii1j}, \eqref{C2bepsilonii1j}, \eqref{C2uepsilonii2j},  \eqref{C2bepsilonii2j} and \eqref{estimate 124}. For instance, 
\begin{align*}
\lim_{\epsilon \searrow 0} C_{2,u}^{\epsilon, i i_{1} j}(t) =& - \frac{(2\pi)^{-3}}{8 (a+b)} \sum_{i_{2}, i_{3} =1}^{3} \int_{\mathbb{R}^{3}} \frac{ \cos(ax^{i_{2}}) - \cos(bx^{i_{2}})}{\lvert x \rvert^{4} f(x)^{2}} \\
& \times h_{b}(x)^{2} \hat{\mathcal{P}}^{ii_{1}}(x) \hat{\mathcal{P}}^{i_{2} i_{3}}(x) \hat{\mathcal{P}}^{ji_{3}}(x) dx 
\end{align*} 
by \eqref{C2uepsilonii1jlimit}, which is well-defined due to the hypothesis on $f$. Let us also point out some interesting relationships that $C_{1,u}^{\epsilon, ii_{1} j} = C_{4,b}^{\epsilon, ii_{1} j}$, $C_{1,b}^{\epsilon, ii_{1} j} = C_{4,u}^{\epsilon, ii_{1} j}$, $C_{2,u}^{\epsilon, ii_{1} j} = C_{3,b}^{\epsilon, ii_{1} j}$, $C_{3,u}^{\epsilon, ii_{1} j} = C_{2,b}^{\epsilon, ii_{1} j}$, while $\tilde{C}_{1,u}^{\epsilon, ii_{1} j} = -\tilde{C}_{4,b}^{\epsilon, ii_{1} j}$, $\tilde{C}_{2,u}^{\epsilon, ii_{1} j} = -\tilde{C}_{3,b}^{\epsilon, ii_{1} j}$, $\tilde{C}_{1,b}^{\epsilon, ii_{1} j} = \tilde{C}_{2,b}^{\epsilon, ii_{1} j} = - \tilde{C}_{3,u}^{\epsilon, ii_{1} j} = -\tilde{C}_{4,u} ^{\epsilon, ii_{1} j}$ (see \eqref{estimate 124}). As it will be clear from the proof of Theorem \ref{Theorem 1.2}, the number 32 comes from a product $(2^{3}) 4$ where two is due to the nonlinear term being a product of two solutions such as $u^{i}$ and $u^{j}$, another two due to the solution being a pair $(u,b)$, another two from the actual system and its approximation, and finally four from the four nonlinear terms total in the MHD system.  As described in \cite[Example 1.1]{ZZ17} (see also \cite[pg. 1679]{HM12}), common examples of discretization schemes which satisfy Definition \ref{Definition of approximation} include the following: finite-difference discretization scheme with 
\begin{align*}
\Delta_{\epsilon} F(x) \triangleq \sum_{l=1}^{3}  \frac{ F(x^{1}, \hdots, x^{l} + \epsilon, \hdots, x^{3}) - 2F(x) + F(x^{1}, \hdots, x^{l} - \epsilon, \hdots, x^{3})}{\epsilon^{2}}, 
\end{align*} 
and 
\begin{equation*}
\tilde{f}(x) \triangleq \frac{4}{\lvert x \rvert^{2}} \left( \sin^{2} (\frac{x^{1}}{2}) + \sin^{2}(\frac{x^{2}}{2}) + \sin^{2}(\frac{x^{3}}{2}) \right), \hspace{3mm} h(x) \triangleq 1_{\{ \lvert y \rvert \leq \frac{L_{0}}{2} \}}(x); 
\end{equation*} 
Galerkin discretization scheme with $\tilde{f}(x) = 1$ and $h(x) \triangleq 1_{\{ \lvert y \rvert \leq \frac{L_{0}}{2} \}}(x)$. It can be directly verified that 
\begin{equation}\label{[Equation (1.3i)][ZZ17]}
D_{j}^{\epsilon} u(x) - \frac{ u(x^{1}, \hdots, x^{j} + a \epsilon, \hdots, x^{3}) - u(x^{1}, \hdots, x^{j} - b \epsilon, \hdots, x^{3})}{(a+b) \epsilon}. 
\end{equation} 

\begin{remark}\label{Difficulty of the MHD system}
Before we state our main result, let us emphasize that it is far from trivial to extend any result on the NS equations to the MHD system. There exists an abundance of results which have been known for the NS equations but remained open for the MHD system for many decades despite much effort by great mathematicians. E.g., although Yudovich \cite{Y63} in 1963 proved the global regularity of the 2D NS equations with zero viscous diffusivity, and hence the Euler equations, its extension to the 2D MHD system without viscous diffusion has remained open \cite{CWY14, FMMNZ14, JZ14a, Y18a}. Moreover, the extension of \cite{HM06}, which proved the existence of a unique invariant measure for the 2D NS equations forced by degenerate noise that is white only in time, to the 2D MHD system has remained open. Finally, the global solution theory of the 2D NS equations forced by STWN in \cite{DD02} has not been extended to the 2D MHD system, mainly because an analog of the identity $\int_{\mathbb{T}^{2}} (u\cdot\nabla) u \cdot \Delta u dx = 0$ for the 2D NS equations does not exist for the 2D MHD system. Due to such difficulty, the study of the MHD and the Hall-MHD systems forced by noise have lagged behind that of the NS equations. We only mention some prominent works on the stochastic MHD system: \cite{BD07, CM10, S10a, SS99, Y17}, in all of which noise is white only in time and not space. 
\end{remark}

At last, we are ready to state our main result. As in \cite{Y19a}, we define the solution to \eqref{MHD} as the limit $\epsilon \searrow 0$ of a solution $\bar{y}^{\epsilon} \triangleq (\bar{u}^{\epsilon}, \bar{b}^{\epsilon})$ that solves for $i_{0} \in \{1,2,3\}$, 
\begin{subequations}\label{[Equation (3.1)][ZZ17]} 
\begin{align}
d \bar{u}^{\epsilon, i_{0}} = \Delta \bar{u}^{\epsilon, i_{0}} dt &+ \sum_{i =1}^{3} \mathcal{P}^{i_{0}i} H_{u,\epsilon} dW_{u}^{i}- \frac{1}{2} \sum_{i, j=1}^{3} \mathcal{P}^{i_{0}i} D_{j} (\bar{u}^{\epsilon, i} \bar{u}^{\epsilon, j}) dt  \nonumber\\
&\hspace{10mm} + \frac{1}{2} \sum_{i, j=1}^{3} \mathcal{P}^{i_{0}i} D_{j} (\bar{b}^{\epsilon, i} \bar{b}^{\epsilon, j}) dt, \hspace{5mm}  \bar{u}^{\epsilon} (0,\cdot) = \mathcal{P} u_{0} (\cdot),\\
d \bar{b}^{\epsilon, i_{0}} = \Delta \bar{b}^{\epsilon, i_{0}} dt &+ \sum_{i =1}^{3} \mathcal{P}^{i_{0}i} H_{b,\epsilon} dW_{b}^{i}- \frac{1}{2} \sum_{i, j=1}^{3} \mathcal{P}^{i_{0}i} D_{j} (\bar{b}^{\epsilon, i} \bar{u}^{\epsilon, j}) dt  \nonumber\\
&\hspace{10mm} + \frac{1}{2} \sum_{i, j=1}^{3} \mathcal{P}^{i_{0}i} D_{j} (\bar{u}^{\epsilon, i} \bar{b}^{\epsilon, j}) dt, \hspace{5mm}  \bar{b}^{\epsilon} (0,\cdot) = \mathcal{P} b_{0} (\cdot).
\end{align} 
\end{subequations}

\begin{theorem}\label{Theorem 1.2} 
Let $\delta_{0} \in (0, \frac{1}{2})$ and then $z \in (\frac{1}{2}, \frac{1}{2} + \delta_{0})$. For any $u_{0}, b_{0} \in \mathcal{C}^{-z}$, we let $((u,b), \tau)$ denote the unique maximal solution to \eqref{MHD} that satisfies $\sup_{t\in [0, \tau) } \lVert (u,b)(t) \rVert_{\mathcal{C}^{-z}} = + \infty$, while for any $\epsilon \in (0,1)$, for any $u_{0}^{\epsilon}, b_{0}^{\epsilon} \in \mathcal{C}^{-z}$, $(u^{\epsilon}, b^{\epsilon})$ be the unique maximal solution to \eqref{MHD approximation}. If $(u_{0}^{\epsilon}, b_{0}^{\epsilon}) \to (u_{0}, b_{0})$ in $\mathcal{C}^{-z}$ as $\epsilon \searrow 0$, then there exists a sequence of random times $\{\tau_{L}\}_{L > 0}$ such that $\tau_{L}$ converges to the explosion time $\tau$ of $(u,b)$ and 
\begin{equation*}
\sup_{t\in [0, \tau_{L}]} \lVert (u^{\epsilon}, b^{\epsilon})(t) - (u,b) (t) \rVert_{\mathcal{C}^{-z}} \to 0 \text{ in } \mathbb{P} \text{ as } \epsilon \searrow 0.  
\end{equation*} 
\end{theorem}

\begin{remark}
The natural procedure to prove Theorem \ref{Theorem 1.2} is to demonstrate the well-posedness of \eqref{MHD} and \eqref{MHD approximation} with drift terms, and thereafter the convergence of their solutions. However, the first major difficulty is actually the lack of \emph{a priori} knowledge of the drift terms. In other words, even though the well-posedness of \eqref{MHD approximation} must be proven before the convergence of its solution to the solution of \eqref{MHD}, it is not even clear how the drift terms should appear within \eqref{MHD approximation}, and obviously one cannot prove the well-posedness of a system if he/she cannot even write down the system. Even the knowledge of the four drift terms in the case of the NS equations in \cite[Equation (1.3)]{ZZ17} did not allow the author to guess the appropriate drift terms in the case of the MHD system due to their significant differences. To circumvent this difficulty, as a first step toward the proof of Theorem \ref{Theorem 1.2}, the author actually went ahead to prove the convergence of the renormalization terms from \eqref{MHD} and \eqref{MHD approximation} in Subsection \ref{Subsection 2.2}. This approach is completely counter-intuitive and seems impossible to accomplish because at this point, the author had no knowledge of any appropriate drift term within \eqref{MHD approximation} and thus certainly no information of any renormalization terms within \eqref{MHD approximation}. However, here, the author was able to rely on his experience from \cite{Y19a} and guess the renormalization terms for not only \eqref{MHD} but also \eqref{MHD approximation}. Although it took several attempts and corrections, the author was eventually able to derive the 32 drift terms within \eqref{MHD approximation}.
\end{remark}

\begin{remark}\label{commutator}
Many times during the process of the convergence of the renormalization terms in Subsection \ref{Subsection 2.2}, it became critical to somehow couple appropriate terms and renormalize together, the technique which was named ``coupled renormalization'' in \cite{Y19a}. This technique becomes once again crucial upon proving the convergence of the terms in the second Wiener chaos of \eqref{fifth convergence} in Subsection \ref{Subsection 2.2} (see \eqref{estimate 125}). For simplicity, let us briefly explain in the case of an application of the Kato-Ponce commutator estimate \cite{KP88}  to the deterministic MHD system. First, if one wishes to estimate an $\dot{H}^{s}$-norm of the solution to the deterministic NS equations for $s \in \mathbb{R}_{+}$, then it suffices to define a fractional Laplacian $\Lambda^{s} \triangleq (-\Delta)^{\frac{s}{2}}$ as a Fourier operator with a Fourier symbol $\lvert \xi \rvert^{s}$, apply $\Lambda^{s}$ to \eqref{Hall-MHD} with $b\equiv 0$, realize that 
\begin{equation}\label{commutator 1}
\int_{\mathbb{T}^{3}} u\cdot \nabla \Lambda^{s} u \cdot \Lambda^{s} u dx = 0 
\end{equation} 
due to the divergence-free property of $u$ in \eqref{Hall-MHD}, and deduce 
\begin{align}\label{estimate 52}
\frac{1}{2} \partial_{t} \lVert u \rVert_{\dot{H}^{s}}^{2} + \nu \int_{0}^{t} \lVert \nabla u \rVert_{\dot{H}^{s}}^{2} ds =& - \int_{\mathbb{T}^{3}} (\Lambda^{s} ((u\cdot\nabla) u) - u\cdot\nabla \Lambda^{s} u) \cdot \Lambda^{s} u dx \nonumber\\
=& - \int_{\mathbb{T}^{3}} [\Lambda^{s}, u\cdot\nabla] u \cdot \Lambda^{s} u,  
\end{align} 
where $[A,B] \triangleq AB - BA$ is a commutator so that the Kato-Ponce commutator estimate from \cite{KP88} may be readily applied. In the case of the MHD system, we can repeat the same procedure for the nonlinear terms $(u\cdot\nabla) u$ and $(u\cdot\nabla) b$ within \eqref{Hall-MHD}. However, for $(b\cdot\nabla) b$ within \eqref{Hall-MHD}, we see that even though we may still write 
\begin{equation}\label{estimate 51}
\int_{\mathbb{T}^{3}} \Lambda^{s} (b\cdot\nabla) b) \cdot \Lambda^{s} u dx =  \int_{\mathbb{T}^{3}} (\Lambda^{s} ((b\cdot\nabla) b) - u\cdot\nabla \Lambda^{s} u) \cdot \Lambda^{s} u dx
\end{equation} 
due to \eqref{commutator 1}, this cannot be written in the form of a commutator as in \eqref{estimate 52}, and thus Kato-Ponce commutator estimate seems inapplicable. Similarly, for the nonlinear term $(b\cdot\nabla) u$, because $\int_{\mathbb{T}^{3}} b\cdot\nabla \Lambda^{s} b \cdot \Lambda^{s} b dx = 0$, one can write 
\begin{equation}\label{estimate 53}
\int_{\mathbb{T}^{3}} \Lambda^{s} (b\cdot\nabla) u) \cdot \Lambda^{s} b dx =  \int_{\mathbb{T}^{3}} (\Lambda^{s} ((b\cdot\nabla) u) - b\cdot\nabla \Lambda^{s} b) \cdot \Lambda^{s} b dx,
\end{equation} 
and yet this cannot be written as a commutator and thus Kato-Ponce's commutator estimate seems inapplicable. The trick is to actually combine these ``two bad'' terms in \eqref{estimate 51}-\eqref{estimate 53} to create ``one good'' term. Indeed, because 
\begin{equation}
\int_{\mathbb{T}^{3}} b\cdot\nabla \Lambda^{s} b \cdot \Lambda^{s} u + b \cdot \nabla \Lambda^{s} u \cdot \Lambda^{s} b dx = 0, 
\end{equation} 
we can write 
\begin{align}
&\int_{\mathbb{T}^{3}} \Lambda^{s} ((b\cdot\nabla) b) \cdot \Lambda^{s} u + \Lambda^{s} ((b\cdot\nabla) u) \cdot \Lambda^{s} b dx \nonumber \\
=& \int_{\mathbb{T}^{3}} [ \Lambda^{s}, b\cdot\nabla] b \cdot \Lambda^{s} u + [\Lambda^{s}, b \cdot \nabla] u \cdot \Lambda^{s} b dx. 
\end{align} 
Such a trick of taking advantage of the structure of the MHD system is crucial in the mathematical analysis of the MHD system, although there still remain many open problems that cannot be extended from the NS equations to the MHD system, as we emphasized in Remark \ref{Difficulty of the MHD system}. It is remarkable that similar tricks will be crucial in our stochastic setting as well, e.g., Remarks \ref{Remark 2.1}, \ref{another remark}, \ref{Remark 2.2}, \ref{yet another remark}, \ref{Remark 2.3} and \ref{Remark 2.6}. 
\end{remark}

\begin{remark}\label{Remark 1.5}
Let us finally mention that because the MHD system \eqref{MHD} is one of the most complex systems of equations in fluid mechanics, Theorem \ref{Theorem 1.2} and its proof lays clear guidelines to achieve analogous results for various SPDEs in fluid mechanics such as the Boussinesq system (e.g., \cite{FGRT15}). Let us also point out that extending Theorem \ref{Theorem 1.2} to the Hall-MHD system, \eqref{Hall-MHD} with $\iota > 0$, is a highly non-trivial problem because it is quasilinear instead of semilinear and not locally subcritical by the definition from \cite[Assumption 8.3]{H14} so that even the theory of regularity structures from \cite{H14}, which is considered by many to be more general than the theory of paracontrolled distributions, is inapplicable to the Hall-MHD system \eqref{Hall-MHD} with $\iota > 0$. 
\end{remark}

In the next section, we will prove Theorem \ref{Theorem 1.2}; for convenience of readers, we include in the Subsection \ref{Preliminaries} of the Appendix the preliminary results which are essential for our proof, as well as a table of essential notations in Subsection \ref{Subsection 3.3}. The proof is inspired by the work of Zhu and Zhu in \cite{ZZ17}; however, many of our computations differ from those of \cite{ZZ17}. While we will emphasize the key difference from the work of \cite{ZZ17}, we will collect computations which are similar to the work of \cite{ZZ17} in the Subsection \ref{Subsection 3.2} of the Appendix for completeness. 

\section{Proof of Theorem \ref{Theorem 1.2}}

\subsection{Constructions of solutions}\label{Subsection 2.1}

\subsubsection{Construction of solution to \eqref{MHD}}

We recall that we defined the solution to \eqref{MHD} as the limit of a solution $\bar{y}^{\epsilon}$ that solves \eqref{[Equation (3.1)][ZZ17]}. We write $\bar{y}_{1}^{\epsilon} \triangleq (\bar{u}_{1}^{\epsilon}, \bar{b}_{1}^{\epsilon})$, $\bar{y}_{2}^{\epsilon} \triangleq (\bar{u}_{2}^{\epsilon}, \bar{b}_{2}^{\epsilon})$, $\bar{y}_{3}^{\epsilon} \triangleq (\bar{u}_{3}^{\epsilon}, \bar{b}_{3}^{\epsilon})$, $\bar{y}_{4}^{\epsilon} \triangleq (\bar{u}_{4}^{\epsilon}, \bar{b}_{4}^{\epsilon})$, and split \eqref{[Equation (3.1)][ZZ17]} to the following four systems, similarly to \cite[Equations (16)-(21)]{Y19a} but slightly differently. For $i \in \{1,2,3\}$,  
\begin{equation}\label{[Equation (3.1aa)][ZZ17]}
d \bar{u}_{1}^{\epsilon, i} = \Delta \bar{u}_{1}^{\epsilon, i} dt + \sum_{i_{1} =1}^{3} \mathcal{P}^{ii_{1}} H_{u,\epsilon} dW_{u}^{i_{1}}, \hspace{3mm} 
d \bar{b}_{1}^{\epsilon, i} = \Delta \bar{b}_{1}^{\epsilon, i} dt + \sum_{i_{1} =1}^{3} \mathcal{P}^{ii_{1}} H_{b,\epsilon} dW_{b}^{i_{1}}, 
\end{equation} 
\begin{subequations}\label{[Equation (3.1ab)][ZZ17]}
\begin{align}
d \bar{u}_{2}^{\epsilon, i} =& \Delta \bar{u}_{2}^{\epsilon, i} dt - \frac{1}{2} \sum_{i_{1}, j=1}^{3} \mathcal{P}^{ii_{1}} D_{j} (\bar{u}_{1}^{\epsilon, i_{1}} \diamond \bar{u}_{1}^{\epsilon, j} - \bar{b}_{1}^{\epsilon, i_{1}} \diamond \bar{b}_{1}^{\epsilon, j}) dt, \hspace{3mm} \bar{u}_{2}^{\epsilon} (0, \cdot) \equiv 0, \\
d \bar{b}_{2}^{\epsilon, i} =& \Delta \bar{b}_{2}^{\epsilon, i} dt - \frac{1}{2} \sum_{i_{1}, j=1}^{3} \mathcal{P}^{ii_{1}} D_{j} (\bar{b}_{1}^{\epsilon, i_{1}} \diamond \bar{u}_{1}^{\epsilon, j} - \bar{u}_{1}^{\epsilon, i_{1}} \diamond \bar{b}_{1}^{\epsilon, j}) dt, \hspace{3mm} \bar{b}_{2}^{\epsilon} (0, \cdot) \equiv 0, 
\end{align}
\end{subequations} 
\begin{subequations}\label{[Equation (3.1ac)][ZZ17]}
\begin{align}
d \bar{u}_{3}^{\epsilon, i} =& \Delta \bar{u}_{3}^{\epsilon, i} dt - \frac{1}{2} \sum_{i_{1}, j=1}^{3} \mathcal{P}^{ii_{1}} D_{j} (\bar{u}_{1}^{\epsilon, i_{1}} \diamond \bar{u}_{2}^{\epsilon, j} + \bar{u}_{2}^{\epsilon, i_{1}} \diamond \bar{u}_{1}^{\epsilon, j} \nonumber\\
& \hspace{36mm} - \bar{b}_{1}^{\epsilon, i_{1}} \diamond \bar{b}_{2}^{\epsilon, j} - \bar{b}_{2}^{\epsilon, i_{1}} \diamond \bar{b}_{1}^{\epsilon, j}) dt, \hspace{3mm} \bar{u}_{3}^{\epsilon} (0, \cdot) \equiv 0, \\
d \bar{b}_{3}^{\epsilon, i} =& \Delta \bar{b}_{3}^{\epsilon, i} dt - \frac{1}{2} \sum_{i_{1}, j=1}^{3} \mathcal{P}^{ii_{1}} D_{j} (\bar{b}_{1}^{\epsilon, i_{1}} \diamond \bar{u}_{2}^{\epsilon, j} + \bar{b}_{2}^{\epsilon, i_{1}} \diamond \bar{u}_{1}^{\epsilon, j} \nonumber\\
& \hspace{36mm} - \bar{u}_{1}^{\epsilon, i_{1}} \diamond \bar{b}_{2}^{\epsilon, j} - \bar{u}_{2}^{\epsilon, i_{1}} \diamond \bar{b}_{1}^{\epsilon, j}) dt, \hspace{3mm} \bar{b}_{3}^{\epsilon} (0, \cdot) \equiv 0,
\end{align}
\end{subequations} 
\begin{subequations}\label{[Equation (3.2)][ZZ17]}
\begin{align}
\bar{u}_{4}^{\epsilon, i}(t) =& P_{t} ( \sum_{i_{1} =1}^{3} \mathcal{P}^{ii_{1}} u_{0}^{i_{1}} - \bar{u}_{1}^{\epsilon, i}(0)) \\
& - \frac{1}{2} \int_{0}^{t} P_{t-s} \sum_{i_{1}, j=1}^{3} \mathcal{P}^{ii_{1}} \nonumber\\
& \times D_{j} [ \bar{u}_{1}^{\epsilon, i_{1}} \diamond (\bar{u}_{3}^{\epsilon, j} + \bar{u}_{4}^{\epsilon, j}) + (\bar{u}_{3}^{\epsilon, i_{1}} + \bar{u}_{4}^{\epsilon, i_{1}}) \diamond \bar{u}_{1}^{\epsilon, j} + \bar{u}_{2}^{\epsilon, i_{1}} \diamond \bar{u}_{2}^{\epsilon, j} \nonumber\\
&\hspace{5mm} + \bar{u}_{2}^{\epsilon, i_{1}}( \bar{u}_{3}^{\epsilon, j} + \bar{u}_{4}^{\epsilon, j}) + \bar{u}_{2}^{\epsilon, j} ( \bar{u}_{3}^{\epsilon, i_{1}} + \bar{u}_{4}^{\epsilon, i_{1}}) + (\bar{u}_{3}^{\epsilon, i_{1}} + \bar{u}_{4}^{\epsilon, i_{1}})(\bar{u}_{3}^{\epsilon, j} + \bar{u}_{4}^{\epsilon, j})\nonumber \\
&\hspace{5mm} - \bar{b}_{1}^{\epsilon, i_{1}} \diamond (\bar{b}_{3}^{\epsilon, j} + \bar{b}_{4}^{\epsilon, j}) - (\bar{b}_{3}^{\epsilon, i_{1}} + \bar{b}_{4}^{\epsilon, i_{1}}) \diamond \bar{b}_{1}^{\epsilon, j} - \bar{b}_{2}^{\epsilon, i_{1}} \diamond \bar{b}_{2}^{\epsilon, j} \nonumber\\
&\hspace{5mm} - \bar{b}_{2}^{\epsilon, i_{1}}( \bar{b}_{3}^{\epsilon, j} + \bar{b}_{4}^{\epsilon, j}) - \bar{b}_{2}^{\epsilon, j} ( \bar{b}_{3}^{\epsilon, i_{1}} + \bar{b}_{4}^{\epsilon, i_{1}}) - (\bar{b}_{3}^{\epsilon, i_{1}} + \bar{b}_{4}^{\epsilon, i_{1}})(\bar{b}_{3}^{\epsilon, j} + \bar{b}_{4}^{\epsilon, j})]ds, \nonumber\\
\bar{b}_{4}^{\epsilon, i}(t) =& P_{t} ( \sum_{i_{1} =1}^{3} \mathcal{P}^{ii_{1}} b_{0}^{i_{1}} - \bar{b}_{1}^{\epsilon, i}(0)) \\
& - \frac{1}{2} \int_{0}^{t} P_{t-s} \sum_{i_{1}, j=1}^{3} \mathcal{P}^{ii_{1}} \nonumber\\
& \times D_{j} [ \bar{b}_{1}^{\epsilon, i_{1}} \diamond (\bar{u}_{3}^{\epsilon, j} + \bar{u}_{4}^{\epsilon, j}) + (\bar{b}_{3}^{\epsilon, i_{1}} + \bar{b}_{4}^{\epsilon, i_{1}}) \diamond \bar{u}_{1}^{\epsilon, j} + \bar{b}_{2}^{\epsilon, i_{1}} \diamond \bar{u}_{2}^{\epsilon, j} \nonumber\\
&\hspace{5mm} + \bar{b}_{2}^{\epsilon, i_{1}}( \bar{u}_{3}^{\epsilon, j} + \bar{u}_{4}^{\epsilon, j}) + \bar{u}_{2}^{\epsilon, j} ( \bar{b}_{3}^{\epsilon, i_{1}} + \bar{b}_{4}^{\epsilon, i_{1}}) + (\bar{b}_{3}^{\epsilon, i_{1}} + \bar{b}_{4}^{\epsilon, i_{1}})(\bar{u}_{3}^{\epsilon, j} + \bar{u}_{4}^{\epsilon, j})\nonumber \\
&\hspace{5mm} - \bar{u}_{1}^{\epsilon, i_{1}} \diamond (\bar{b}_{3}^{\epsilon, j} + \bar{b}_{4}^{\epsilon, j}) - (\bar{u}_{3}^{\epsilon, i_{1}} + \bar{u}_{4}^{\epsilon, i_{1}}) \diamond \bar{b}_{1}^{\epsilon, j} - \bar{u}_{2}^{\epsilon, i_{1}} \diamond \bar{b}_{2}^{\epsilon, j} \nonumber\\
&\hspace{5mm} - \bar{u}_{2}^{\epsilon, i_{1}}( \bar{b}_{3}^{\epsilon, j} + \bar{b}_{4}^{\epsilon, j}) - \bar{b}_{2}^{\epsilon, j} ( \bar{u}_{3}^{\epsilon, i_{1}} + \bar{u}_{4}^{\epsilon, i_{1}}) - (\bar{u}_{3}^{\epsilon, i_{1}} + \bar{u}_{4}^{\epsilon, i_{1}})(\bar{b}_{3}^{\epsilon, j} + \bar{b}_{4}^{\epsilon, j})]ds, \nonumber
\end{align}
\end{subequations} 
where we recall from Example \ref{Burgers' equation example} that $P_{t} = e^{t\Delta}$, and 
\begin{subequations}\label{[Equation (3.2aa)][ZZ17]}
\begin{align}
& \bar{u}_{1}^{\epsilon, i} \diamond \bar{u}_{1}^{\epsilon, j} \triangleq \bar{u}_{1}^{\epsilon, i} \bar{u}_{1}^{\epsilon, j} - \bar{C}_{01}^{\epsilon, ij}, \hspace{3mm} \bar{b}_{1}^{\epsilon, i} \diamond \bar{b}_{1}^{\epsilon, j} \triangleq \bar{b}_{1}^{\epsilon, i} \bar{b}_{1}^{\epsilon, j} - \bar{C}_{02}^{\epsilon, ij}, \\
& \bar{u}_{1}^{\epsilon, i} \diamond \bar{b}_{1}^{\epsilon, j} \triangleq \bar{u}_{1}^{\epsilon, i} \bar{b}_{1}^{\epsilon, j} - \bar{C}_{03}^{\epsilon, ij}, \hspace{4mm} \bar{b}_{1}^{\epsilon, i} \diamond \bar{u}_{1}^{\epsilon, j} \triangleq \bar{b}_{1}^{\epsilon, i} \bar{u}_{1}^{\epsilon, j} - \bar{C}_{03}^{\epsilon, ij}, 
\end{align}
\end{subequations} 
with $\bar{C}_{01}^{\epsilon, ij}, \bar{C}_{02}^{\epsilon, ij}, \bar{C}_{03}^{\epsilon, ij}$ explicitly found in \eqref{estimate 41}, \eqref{estimate 42}, 
\begin{subequations}\label{[Equation (3.2ab)][ZZ17]}
\begin{align}
& \bar{u}_{2}^{\epsilon, i} \diamond \bar{u}_{1}^{\epsilon, j} \triangleq \bar{u}_{1}^{\epsilon, j} \diamond \bar{u}_{2}^{\epsilon, i} \triangleq \bar{u}_{1}^{\epsilon, j} \bar{u}_{2}^{\epsilon, i}, \hspace{3mm}  \bar{b}_{2}^{\epsilon, i} \diamond \bar{b}_{1}^{\epsilon, j} \triangleq \bar{b}_{1}^{\epsilon, j} \diamond \bar{b}_{2}^{\epsilon, i} \triangleq \bar{b}_{1}^{\epsilon, j} \bar{b}_{2}^{\epsilon, i},\\
& \bar{b}_{2}^{\epsilon, i} \diamond \bar{u}_{1}^{\epsilon, j} \triangleq \bar{u}_{1}^{\epsilon, j} \diamond \bar{b}_{2}^{\epsilon, i} \triangleq \bar{u}_{1}^{\epsilon, j} \bar{b}_{2}^{\epsilon, i}, \hspace{4mm}  \bar{u}_{2}^{\epsilon, i} \diamond \bar{b}_{1}^{\epsilon, j} \triangleq \bar{b}_{1}^{\epsilon, j} \diamond \bar{u}_{2}^{\epsilon, i} \triangleq \bar{b}_{1}^{\epsilon, j} \bar{u}_{2}^{\epsilon, i},
\end{align}
\end{subequations} 
\begin{subequations}\label{[Equation (3.2ac)][ZZ17]}
\begin{align}
& \bar{u}_{2}^{\epsilon, i} \diamond \bar{u}_{2}^{\epsilon, j} \triangleq \bar{u}_{2}^{\epsilon, i} \bar{u}_{2}^{\epsilon, j} - \bar{\phi}_{21}^{\epsilon, ij} - \bar{C}_{21}^{\epsilon, ij}, \hspace{1mm} \bar{b}_{2}^{\epsilon, i} \diamond \bar{b}_{2}^{\epsilon, j} \triangleq \bar{b}_{2}^{\epsilon, i} \bar{b}_{2}^{\epsilon, j} - \bar{\phi}_{22}^{\epsilon, ij} - \bar{C}_{22}^{\epsilon, ij}, \\
& \bar{b}_{2}^{\epsilon, i} \diamond \bar{u}_{2}^{\epsilon, j} \triangleq \bar{b}_{2}^{\epsilon, i} \bar{u}_{2}^{\epsilon, j} - \bar{\phi}_{23}^{\epsilon, ij} - \bar{C}_{23}^{\epsilon, ij}, \hspace{2mm} \bar{u}_{2}^{\epsilon, i} \diamond \bar{b}_{2}^{\epsilon, j} \triangleq \bar{u}_{2}^{\epsilon, i} \bar{b}_{2}^{\epsilon, j} - \bar{\phi}_{24}^{\epsilon, ij} - \bar{C}_{24}^{\epsilon, ij}, 
\end{align}
\end{subequations} 
with $\bar{\phi}_{22}^{\epsilon, ij}(t)$ explicitly found in \eqref{estimate 46}, $\bar{C}_{22}^{\epsilon, ij}$ in \eqref{estimate 44}, 
\begin{subequations}\label{[Equation (3.2ad)][ZZ17]}
\begin{align}
 \bar{u}_{1}^{\epsilon, i} \diamond \bar{u}_{3}^{\epsilon, j} \triangleq& \pi_{<} (\bar{u}_{1}^{\epsilon, i},  \bar{u}_{3}^{\epsilon, j}) + \pi_{>} (\bar{u}_{1}^{\epsilon, i}, \bar{u}_{3}^{\epsilon, j}) + \pi_{0, \diamond} (\bar{u}_{3}^{\epsilon, j}, \bar{u}_{1}^{\epsilon, i}), \\
 \bar{b}_{1}^{\epsilon, i} \diamond \bar{b}_{3}^{\epsilon, j} \triangleq& \pi_{<} ( \bar{b}_{1}^{\epsilon, i}, \bar{b}_{3}^{\epsilon, j}) + \pi_{>} ( \bar{b}_{1}^{\epsilon, i}, \bar{b}_{3}^{\epsilon, j}) + \pi_{0,\diamond} (\bar{b}_{3}^{\epsilon, j}, \bar{b}_{1}^{\epsilon, i}), \\
  \bar{u}_{1}^{\epsilon, i} \diamond \bar{b}_{3}^{\epsilon, j} \triangleq& \pi_{<} ( \bar{u}_{1}^{\epsilon, i}, \bar{b}_{3}^{\epsilon, j}) + \pi_{>} ( \bar{u}_{1}^{\epsilon, i}, \bar{b}_{3}^{\epsilon, j}) + \pi_{0,\diamond} (\bar{b}_{3}^{\epsilon, j}, \bar{u}_{1}^{\epsilon, i}), \\
 \bar{b}_{1}^{\epsilon, i} \diamond \bar{u}_{3}^{\epsilon, j} \triangleq& \pi_{<} ( \bar{b}_{1}^{\epsilon, i}, \bar{u}_{3}^{\epsilon, j}) + \pi_{>} ( \bar{b}_{1}^{\epsilon, i}, \bar{u}_{3}^{\epsilon, j}) + \pi_{0,\diamond} (\bar{u}_{3}^{\epsilon, j}, \bar{b}_{1}^{\epsilon, i}), 
\end{align}
\end{subequations} 
with 
\begin{subequations}\label{[Equation (3.2b)][ZZ17]}
\begin{align}
 \pi_{0,\diamond } (\bar{u}_{3}^{\epsilon, i}, \bar{u}_{1}^{\epsilon, j}) \triangleq& \pi_{0}( \bar{u}_{3}^{\epsilon, i}, \bar{u}_{1}^{\epsilon, j}) - \bar{\phi}_{11}^{\epsilon, ij}  - \bar{C}_{11}^{\epsilon, ij}, \\
 \pi_{0,\diamond } (\bar{b}_{3}^{\epsilon, i}, \bar{b}_{1}^{\epsilon, j}) \triangleq& \pi_{0}( \bar{b}_{3}^{\epsilon, i}, \bar{b}_{1}^{\epsilon, j}) - \bar{\phi}_{12}^{\epsilon, ij}  - \bar{C}_{12}^{\epsilon, ij}, \\
  \pi_{0,\diamond } (\bar{u}_{3}^{\epsilon, i}, \bar{b}_{1}^{\epsilon, j}) \triangleq& \pi_{0}( \bar{u}_{3}^{\epsilon, i}, \bar{b}_{1}^{\epsilon, j}) - \bar{\phi}_{13}^{\epsilon, ij}  - \bar{C}_{13}^{\epsilon, ij}, \\
 \pi_{0,\diamond } (\bar{b}_{3}^{\epsilon, i}, \bar{u}_{1}^{\epsilon, j}) \triangleq& \pi_{0}( \bar{b}_{3}^{\epsilon, i}, \bar{u}_{1}^{\epsilon, j}) - \bar{\phi}_{14}^{\epsilon, ij}  - \bar{C}_{14}^{\epsilon, ij}, 
\end{align}
\end{subequations} 
and $\bar{\phi}_{13}^{\epsilon, ij}(t)$ explicitly found in \eqref{estimate 48}, $\bar{C}_{13}^{\epsilon, ij}$ in \eqref{estimate 47}, 
\begin{subequations}\label{[Equation (3.2ae)][ZZ17]}
\begin{align}
 \bar{u}_{1}^{\epsilon, i} \diamond \bar{u}_{4}^{\epsilon, j} \triangleq& \pi_{<} (\bar{u}_{1}^{\epsilon, i}, \bar{u}_{4}^{\epsilon, j}) + \pi_{>} (\bar{u}_{1}^{\epsilon, i}, \bar{u}_{4}^{\epsilon, j}) + \pi_{0,\diamond} (\bar{u}_{4}^{\epsilon, j}, \bar{u}_{1}^{\epsilon, i}), \\
 \bar{b}_{1}^{\epsilon, i} \diamond \bar{b}_{4}^{\epsilon, j} \triangleq& \pi_{<} (\bar{b}_{1}^{\epsilon, i}, \bar{b}_{4}^{\epsilon, j}) + \pi_{>} (\bar{b}_{1}^{\epsilon, i}, \bar{b}_{4}^{\epsilon, j}) + \pi_{0,\diamond} (\bar{b}_{4}^{\epsilon, j}, \bar{b}_{1}^{\epsilon, i}), \\
 \bar{u}_{1}^{\epsilon, i} \diamond \bar{b}_{4}^{\epsilon, j} \triangleq& \pi_{<} (\bar{u}_{1}^{\epsilon, i}, \bar{b}_{4}^{\epsilon, j}) + \pi_{>} (\bar{u}_{1}^{\epsilon, i}, \bar{b}_{4}^{\epsilon, j}) + \pi_{0,\diamond} (\bar{b}_{4}^{\epsilon, j}, \bar{u}_{1}^{\epsilon, i}), \\
 \bar{b}_{1}^{\epsilon, i} \diamond \bar{u}_{4}^{\epsilon, j} \triangleq& \pi_{<} (\bar{b}_{1}^{\epsilon, i},  \bar{u}_{4}^{\epsilon, j}) + \pi_{>} (\bar{b}_{1}^{\epsilon, i}, \bar{u}_{4}^{\epsilon, j}) + \pi_{0,\diamond} (\bar{u}_{4}^{\epsilon, j}, \bar{b}_{1}^{\epsilon, i}), 
\end{align}
\end{subequations}
with 
\begin{subequations}\label{[Equation (3.2c)][ZZ17]}
\begin{align} 
& \pi_{0,\diamond} (\bar{u}_{4}^{\epsilon, i}, \bar{u}_{1}^{\epsilon, j})  \triangleq \pi_{0}(\bar{u}_{4}^{\epsilon, i}, \bar{u}_{1}^{\epsilon, j}), \hspace{3mm} \pi_{0,\diamond} (\bar{b}_{4}^{\epsilon, i}, \bar{b}_{1}^{\epsilon, j})  \triangleq \pi_{0}(\bar{b}_{4}^{\epsilon, i}, \bar{b}_{1}^{\epsilon, j}), \\
& \pi_{0,\diamond} (\bar{b}_{4}^{\epsilon, i}, \bar{u}_{1}^{\epsilon, j})  \triangleq \pi_{0}(\bar{b}_{4}^{\epsilon, i}, \bar{u}_{1}^{\epsilon, j}),  \hspace{4mm} \pi_{0,\diamond} (\bar{u}_{4}^{\epsilon, i}, \bar{b}_{1}^{\epsilon, j})  \triangleq \pi_{0}(\bar{u}_{4}^{\epsilon, i}, \bar{b}_{1}^{\epsilon, j}), 
\end{align}
\end{subequations}
and $\bar{\phi}_{ik}^{\epsilon, ij}$ for $i \in \{1,2\}, k \in \{1,2,3,4\}$, converges to some $\bar{\phi}_{ik}$ with respect to the norm of $\sup_{t\in [0,T]} t^{\rho} \lvert \cdot (t) \rvert$ for every $\rho > 0$ (see \eqref{estimate 127}, \eqref{estimate 54} - \eqref{estimate 55}). Now let us define $\bar{K}_{u}^{\epsilon}, \bar{K}_{b}^{\epsilon}$ to satisfy 
\begin{equation}\label{estimate 56}
d \bar{K}_{u}^{\epsilon, i}  = (\Delta \bar{K}_{u}^{\epsilon, i} + \bar{u}_{1}^{\epsilon, i}) dt, \hspace{1mm} \bar{K}_{u}^{\epsilon, i} (0) = 0, \hspace{2mm} d \bar{K}_{b}^{\epsilon, i}  = (\Delta \bar{K}_{b}^{\epsilon, i} + \bar{b}_{1}^{\epsilon, i}) dt, \hspace{1mm} \bar{K}_{b}^{\epsilon, i} (0) = 0. 
\end{equation} 
Then from the work of \cite[Equations (35), (37) and (113)]{Y19a} we see that for all 
\begin{equation}\label{[Equation (3.2g)][ZZ17]}
\delta\in (0, \delta_{0} \wedge \frac{1-2\delta_{0}}{3} \wedge \frac{1-z}{4} \wedge (2z-1) ), 
\end{equation} 
where we recall that $\delta_{0} \in (0,\frac{1}{2})$ and $z\in (\frac{1}{2}, \frac{1}{2} + \delta_{0})$ by hypothesis, the following quantity is finite: 
\begin{align}\label{[Equation (3.2e)][ZZ17]}
\bar{C}_{W}^{\epsilon}(T) \triangleq& \sup_{t\in [0,T]} [ \sum_{i=1}^{3} \lVert (\bar{u}_{1}^{\epsilon, i}, \bar{b}_{1}^{\epsilon, i})(t) \rVert_{\mathcal{C}^{- \frac{1}{2} - \frac{\delta}{2}}} \\
& + \sum_{i,j=1}^{3} \lVert ( \bar{u}_{1}^{\epsilon, i} \diamond \bar{u}_{1}^{\epsilon, j}, \bar{b}_{1}^{\epsilon, i} \diamond \bar{b}_{1}^{\epsilon, j}, \bar{u}_{1}^{\epsilon, i} \diamond \bar{b}_{1}^{\epsilon, j}, \bar{b}_{1}^{\epsilon, i} \diamond \bar{u}_{1}^{\epsilon, j})(t) \rVert_{\mathcal{C}^{-1 - \frac{\delta}{2}}} \nonumber\\
&+  \sum_{i,j=1}^{3} \lVert ( \bar{u}_{1}^{\epsilon, i} \diamond \bar{u}_{2}^{\epsilon, j}, \bar{b}_{1}^{\epsilon, i} \diamond \bar{b}_{2}^{\epsilon, j}, \bar{u}_{2}^{\epsilon, i} \diamond \bar{b}_{1}^{\epsilon, j}, \bar{b}_{2}^{\epsilon, i} \diamond \bar{u}_{1}^{\epsilon, j})(t) \rVert_{\mathcal{C}^{-\frac{1}{2} - \frac{\delta}{2}}} \nonumber\\
&+  \sum_{i,j=1}^{3} \lVert ( \bar{u}_{2}^{\epsilon, i} \diamond \bar{u}_{2}^{\epsilon, j}, \bar{b}_{2}^{\epsilon, i} \diamond \bar{b}_{2}^{\epsilon, j}, \bar{b}_{2}^{\epsilon, i} \diamond \bar{u}_{2}^{\epsilon, j})(t) \rVert_{\mathcal{C}^{-\delta}} \nonumber\\
&+  \sum_{i,j=1}^{3} \lVert (\pi_{0,\diamond} (\bar{u}_{3}^{\epsilon, j}, \bar{u}_{1}^{\epsilon, i}), \pi_{0,\diamond} (\bar{b}_{3}^{\epsilon, j}, \bar{b}_{1}^{\epsilon, i}), \pi_{0,\diamond}( \bar{u}_{3}^{\epsilon, j}, \bar{b}_{1}^{\epsilon, i}), \pi_{0,\diamond} (\bar{b}_{3}^{\epsilon, j}, \bar{u}_{1}^{\epsilon, i})) \rVert_{\mathcal{C}^{-\delta}}\nonumber\\
&+  \sum_{i, i_{1}, j, j_{1} =1}^{3} \lVert (\pi_{0,\diamond} (\mathcal{P}^{ii_{1}} D_{j} \bar{K}_{u}^{\epsilon, j}, \bar{u}_{1}^{\epsilon, j_{1}}), \pi_{0,\diamond}(\mathcal{P}^{ii_{1}} D_{j} \bar{K}_{u}^{\epsilon, i_{1}}, \bar{u}_{1}^{\epsilon, j_{1}}), \nonumber\\
& \hspace{25mm} \pi_{0,\diamond} (\mathcal{P}^{ii_{1}} D_{j} \bar{K}_{b}^{\epsilon, j}, \bar{u}_{1}^{\epsilon, j_{1}}), \pi_{0,\diamond} (\mathcal{P}^{ii_{1}} D_{j} \bar{K}_{b}^{\epsilon, i_{1}}, \bar{u}_{1}^{\epsilon, j_{1}}), \nonumber\\
&\hspace{25mm} \pi_{0,\diamond} (\mathcal{P}^{ii_{1}} D_{j} \bar{K}_{u}^{\epsilon, j}, \bar{b}_{1}^{\epsilon, j_{1}}), \pi_{0,\diamond}(\mathcal{P}^{ii_{1}} D_{j} \bar{K}_{u}^{\epsilon, i_{1}}, \bar{b}_{1}^{\epsilon, j_{1}}), \nonumber\\
& \hspace{25mm} \pi_{0,\diamond} (\mathcal{P}^{ii_{1}} D_{j} \bar{K}_{b}^{\epsilon, j}, \bar{b}_{1}^{\epsilon, j_{1}}), \pi_{0,\diamond} (\mathcal{P}^{ii_{1}} D_{j} \bar{K}_{b}^{\epsilon, i_{1}}, \bar{b}_{1}^{\epsilon, j_{1}})) \rVert_{\mathcal{C}^{-\delta}}] \nonumber
\end{align} 
where 
\begin{subequations}\label{[Equation (3.2f)][ZZ17]}
\begin{align}
 \pi_{0,\diamond} ( \mathcal{P}^{ii_{1}} D_{j} \bar{K}_{u}^{\epsilon, j}, \bar{u}_{1}^{\epsilon, j_{1}}) \triangleq& \pi_{0} ( \mathcal{P}^{ii_{1}} D_{j} \bar{K}_{u}^{\epsilon, j}, \bar{u}_{1}^{\epsilon, j_{1}}), \\
\pi_{0,\diamond} ( \mathcal{P}^{ii_{1}} D_{j} \bar{K}_{u}^{\epsilon, i_{1}}, \bar{u}_{1}^{\epsilon, j_{1}}) \triangleq& \pi_{0} ( \mathcal{P}^{ii_{1}} D_{j} \bar{K}_{u}^{\epsilon, i_{1}}, \bar{u}_{1}^{\epsilon, j_{1}}), \\
 \pi_{0,\diamond} ( \mathcal{P}^{ii_{1}} D_{j} \bar{K}_{b}^{\epsilon, j}, \bar{u}_{1}^{\epsilon, j_{1}}) \triangleq& \pi_{0} ( \mathcal{P}^{ii_{1}} D_{j} \bar{K}_{b}^{\epsilon, j}, \bar{u}_{1}^{\epsilon, j_{1}}), \\
\pi_{0,\diamond} ( \mathcal{P}^{ii_{1}} D_{j} \bar{K}_{b}^{\epsilon, i_{1}}, \bar{u}_{1}^{\epsilon, j_{1}}) \triangleq& \pi_{0} ( \mathcal{P}^{ii_{1}} D_{j} \bar{K}_{b}^{\epsilon, i_{1}}, \bar{u}_{1}^{\epsilon, j_{1}}), \\
\pi_{0,\diamond} ( \mathcal{P}^{ii_{1}} D_{j} \bar{K}_{u}^{\epsilon, j}, \bar{b}_{1}^{\epsilon, j_{1}}) \triangleq& \pi_{0} ( \mathcal{P}^{ii_{1}} D_{j} \bar{K}_{u}^{\epsilon, j}, \bar{b}_{1}^{\epsilon, j_{1}}), \\
\pi_{0,\diamond} ( \mathcal{P}^{ii_{1}} D_{j} \bar{K}_{u}^{\epsilon, i_{1}}, \bar{b}_{1}^{\epsilon, j_{1}}) \triangleq& \pi_{0} ( \mathcal{P}^{ii_{1}} D_{j} \bar{K}_{u}^{\epsilon, i_{1}}, \bar{b}_{1}^{\epsilon, j_{1}}), \\
 \pi_{0,\diamond} ( \mathcal{P}^{ii_{1}} D_{j} \bar{K}_{b}^{\epsilon, j}, \bar{b}_{1}^{\epsilon, j_{1}}) \triangleq& \pi_{0} ( \mathcal{P}^{ii_{1}} D_{j} \bar{K}_{b}^{\epsilon, j}, \bar{b}_{1}^{\epsilon, j_{1}}), \\
\pi_{0,\diamond} ( \mathcal{P}^{ii_{1}} D_{j} \bar{K}_{b}^{\epsilon, i_{1}}, \bar{b}_{1}^{\epsilon, j_{1}}) \triangleq& \pi_{0} ( \mathcal{P}^{ii_{1}} D_{j} \bar{K}_{b}^{\epsilon, i_{1}}, \bar{b}_{1}^{\epsilon, j_{1}}). 
\end{align}
\end{subequations} 
Now similarly to \cite[Equations (38)-(39)]{Y19a}, we may estimate
\begin{align}\label{[Equation (3.2h)][ZZ17]}
\sup_{t \in [0,T]} \sum_{i=1}^{3}( \lVert (\bar{u}_{2}^{\epsilon, i}, \bar{b}_{2}^{\epsilon, i} ) (t) \rVert_{\mathcal{C}^{-\delta}} + \lVert (\bar{u}_{3}^{\epsilon, i}, \bar{b}_{3}^{\epsilon, i} )(t) \rVert_{\mathcal{C}^{\frac{1}{2} - \delta}}) \lesssim \bar{C}_{W}^{\epsilon}(T) T^{\frac{\delta}{4}} 
\end{align} 
by \eqref{[Equation (3.1ab)][ZZ17]}, \eqref{[Equation (3.1ac)][ZZ17]}, Lemmas \ref{Lemma 3.7} and \ref{Lemma 3.10} and \eqref{[Equation (3.2e)][ZZ17]}. Moreover, similarly to \cite[Equations (33)-(34)]{Y19a}, 
\begin{align}\label{[Equation (3.3)][ZZ17]}
\lVert ( \bar{K}_{u}^{\epsilon, i}, \bar{K}_{b}^{\epsilon, i})(t) \rVert_{\mathcal{C}^{\frac{3}{2} - \delta}} \lesssim \int_{0}^{t} \lVert P_{t-s} (\bar{u}_{1}^{\epsilon, i}, \bar{b}_{1}^{\epsilon, i})(s) \rVert_{\mathcal{C}^{\frac{3}{2} - \delta}} ds  \lesssim \bar{C}_{W}^{\epsilon}(t) t^{\frac{\delta}{4}} 
\end{align}
by \eqref{estimate 56}, Lemma \ref{Lemma 3.10} and \eqref{[Equation (3.2e)][ZZ17]}. Now the hypothesis of  Theorem \ref{Theorem 1.2} and  \eqref{[Equation (3.2g)][ZZ17]} allow us to take $\kappa > 0$ such that 
\begin{equation}\label{[Equation (3.3c)][ZZ17]}
\kappa \in (0, \frac{1}{3} ( \frac{1}{2} + \delta_{0} - z) \wedge \frac{\delta}{2} \wedge \frac{1-2\delta_{0} - 3 \delta}{7} \wedge \frac{1-z- 4 \delta}{5}). 
\end{equation} 
Then we can compute similarly to \cite[Equations (40)-(43)]{Y19a}, 
\begin{equation}
\sup_{t\in [0,T]} t^{\frac{ \frac{1}{2} - \delta_{0} + z + \kappa}{2}} \lVert \bar{y}_{4}^{\epsilon}(t) \rVert_{\mathcal{C}^{\frac{1}{2} - \delta_{0}}} \lesssim I_{T}^{1} + I_{T}^{2} 
\end{equation} 
where 
\begin{subequations}
\begin{align}
I_{T}^{1} &\triangleq  \sup_{t\in [0, T]} t^{ \frac{ \frac{1}{2} - \delta_{0} + z + \kappa}{2}} \sum_{i=1}^{3}  \lVert P_{t} (\sum_{i_{1} =1}^{3} \mathcal{P}^{ii_{1}} (y_{0}^{i_{1}} - \bar{y}_{1}^{\epsilon, i}(0)) \rVert_{\mathcal{C}^{\frac{1}{2} - \delta_{0}}}, \\
I_{T}^{2} &\triangleq  \sup_{t\in [0, T]} t^{ \frac{ \frac{1}{2} - \delta_{0} + z + \kappa}{2}} \sum_{i=1}^{3} \\
& \times [ \lVert ( \int_{0}^{t} P_{t-s} \sum_{i_{1}, j=1}^{3} \mathcal{P}^{ii_{1}} D_{j} [ \bar{u}_{1}^{\epsilon, i_{1}} \diamond (\bar{u}_{3}^{\epsilon, j} + \bar{u}_{4}^{\epsilon, j}) + (\bar{u}_{3}^{\epsilon, i_{1}} + \bar{u}_{4}^{\epsilon, i_{1}}) \diamond \bar{u}_{1}^{\epsilon, j} + \bar{u}_{2}^{\epsilon, i_{1}} \diamond \bar{u}_{2}^{\epsilon, j} \nonumber\\
&\hspace{2mm} + \bar{u}_{2}^{\epsilon, i_{1}}( \bar{u}_{3}^{\epsilon, j} + \bar{u}_{4}^{\epsilon, j}) + \bar{u}_{2}^{\epsilon, j} ( \bar{u}_{3}^{\epsilon, i_{1}} + \bar{u}_{4}^{\epsilon, i_{1}}) + (\bar{u}_{3}^{\epsilon, i_{1}} + \bar{u}_{4}^{\epsilon, i_{1}})(\bar{u}_{3}^{\epsilon, j} + \bar{u}_{4}^{\epsilon, j})\nonumber \\
&\hspace{2mm} - \bar{b}_{1}^{\epsilon, i_{1}} \diamond (\bar{b}_{3}^{\epsilon, j} + \bar{b}_{4}^{\epsilon, j}) - (\bar{b}_{3}^{\epsilon, i_{1}} + \bar{b}_{4}^{\epsilon, i_{1}}) \diamond \bar{b}_{1}^{\epsilon, j} - \bar{b}_{2}^{\epsilon, i_{1}} \diamond \bar{b}_{2}^{\epsilon, j} \nonumber\\
&\hspace{2mm} - \bar{b}_{2}^{\epsilon, i_{1}}( \bar{b}_{3}^{\epsilon, j} + \bar{b}_{4}^{\epsilon, j}) - \bar{b}_{2}^{\epsilon, j} ( \bar{b}_{3}^{\epsilon, i_{1}} + \bar{b}_{4}^{\epsilon, i_{1}}) - (\bar{b}_{3}^{\epsilon, i_{1}} + \bar{b}_{4}^{\epsilon, i_{1}})(\bar{b}_{3}^{\epsilon, j} + \bar{b}_{4}^{\epsilon, j})]ds, \nonumber\\
&\int_{0}^{t} P_{t-s} \sum_{i_{1}, j=1}^{3} \mathcal{P}^{ii_{1}} D_{j} [ \bar{b}_{1}^{\epsilon, i_{1}} \diamond (\bar{u}_{3}^{\epsilon, j} + \bar{u}_{4}^{\epsilon, j}) + (\bar{b}_{3}^{\epsilon, i_{1}} + \bar{b}_{4}^{\epsilon, i_{1}}) \diamond \bar{u}_{1}^{\epsilon, j} + \bar{b}_{2}^{\epsilon, i_{1}} \diamond \bar{u}_{2}^{\epsilon, j} \nonumber\\
&\hspace{2mm} + \bar{b}_{2}^{\epsilon, i_{1}}( \bar{u}_{3}^{\epsilon, j} + \bar{u}_{4}^{\epsilon, j}) + \bar{u}_{2}^{\epsilon, j} ( \bar{b}_{3}^{\epsilon, i_{1}} + \bar{b}_{4}^{\epsilon, i_{1}}) + (\bar{b}_{3}^{\epsilon, i_{1}} + \bar{b}_{4}^{\epsilon, i_{1}})(\bar{u}_{3}^{\epsilon, j} + \bar{u}_{4}^{\epsilon, j})\nonumber \\
&\hspace{2mm} - \bar{u}_{1}^{\epsilon, i_{1}} \diamond (\bar{b}_{3}^{\epsilon, j} + \bar{b}_{4}^{\epsilon, j}) - (\bar{u}_{3}^{\epsilon, i_{1}} + \bar{u}_{4}^{\epsilon, i_{1}}) \diamond \bar{b}_{1}^{\epsilon, j} - \bar{u}_{2}^{\epsilon, i_{1}} \diamond \bar{b}_{2}^{\epsilon, j} \nonumber\\
&\hspace{2mm} - \bar{u}_{2}^{\epsilon, i_{1}}( \bar{b}_{3}^{\epsilon, j} + \bar{b}_{4}^{\epsilon, j}) - \bar{b}_{2}^{\epsilon, j} ( \bar{u}_{3}^{\epsilon, i_{1}} + \bar{u}_{4}^{\epsilon, i_{1}}) - (\bar{u}_{3}^{\epsilon, i_{1}} + \bar{u}_{4}^{\epsilon, i_{1}})(\bar{b}_{3}^{\epsilon, j} + \bar{b}_{4}^{\epsilon, j})]ds) \rVert_{\mathcal{C}^{\frac{1}{2} - \delta_{0}}}].  \nonumber
\end{align}
\end{subequations} f
We estimate 
\begin{align}\label{estimate 59}
I_{T}^{1} \lesssim&\sup_{t\in [0, T]} t^{\frac{ \frac{1}{2} - \delta_{0} + z + \kappa}{2}} t^{- \frac{ \frac{1}{2} - \delta_{0} + z}{2}} \sum_{i, i_{1} =1}^{3} ( \lVert y_{0}^{i_{1}} \rVert_{\mathcal{C}^{-z}} + \lVert \bar{y}_{1}^{\epsilon, i}(0) \rVert_{\mathcal{C}^{-z}}) \lesssim_{T, \lVert y_{0} \rVert_{\mathcal{C}^{-z}}} 1
\end{align} 
by Lemmas \ref{Lemma 3.10} and \ref{Lemma 3.7}. Within $I_{T}^{2}$, we can estimate e.g., 
\begin{align}\label{estimate 57}
\sup_{t \in [0, T]} t^{\frac{ \frac{1}{2} - \delta_{0} + z + \kappa}{2}} \lVert \int_{0}^{t} P_{t-s} \sum_{i_{1}, j=1}^{3} \mathcal{P}^{ii_{1}} D_{j} ( \bar{b}_{2}^{\epsilon, i_{1}} \diamond \bar{b}_{2}^{\epsilon, j})(s) ds \rVert_{\mathcal{C}^{\frac{1}{2} - \delta_{0}}} \lesssim_{\bar{C}_{W}^{\epsilon}(T), T} 1 
\end{align} 
by Lemmas \ref{Lemma 3.10} and \ref{Lemma 3.7}, \eqref{[Equation (3.2e)][ZZ17]} and \eqref{[Equation (3.2g)][ZZ17]}. Within $I_{T}^{2}$, we can also estimate e.g., 
\begin{align}\label{estimate 58}
& \sup_{t \in [0, T]} t^{\frac{ \frac{1}{2} - \delta_{0} + z + \kappa}{2}}\lVert \int_{0}^{t} P_{t-s} \sum_{i_{1}, j=1}^{3} \mathcal{P}^{ii_{1}} D_{j} ( \bar{b}_{4}^{\epsilon, i_{1}} \bar{b}_{4}^{\epsilon, j}) ds \rVert_{\mathcal{C}^{\frac{1}{2} - \delta_{0}}}  \nonumber \\ 
\lesssim& \sup_{t \in [0, T]} t^{\frac{ \frac{1}{2} - \delta_{0} + z + \kappa}{2}} \int_{0}^{t} \sum_{i_{1}, j=1}^{3} (t-s)^{-\frac{1}{2}} \lVert \bar{b}_{4}^{\epsilon, i_{1}} \rVert_{\mathcal{C}^{\frac{1}{2} - \delta_{0}}} \lVert \bar{b}_{4}^{\epsilon, j} \rVert_{\mathcal{C}^{\frac{1}{2} - \delta_{0}}} ds \nonumber \\
\lesssim_{T}& ( \sup_{t \in [0, T]} t^{\frac{ \frac{1}{2} - \delta_{0} + z + \kappa}{2}}\lVert \bar{b}_{4}^{\epsilon} (t) \rVert_{\mathcal{C}^{\frac{1}{2} - \delta_{0}}})^{2} 
\end{align}
by Lemmas \ref{Lemma 3.10}, \ref{Lemma 3.7} and \ref{[Lemma 1.1][Y19a]} (4) and \eqref{[Equation (3.3c)][ZZ17]}. Computations which are similar to \eqref{estimate 57} and \eqref{estimate 58} on other terms within $I_{T}^{2}$, along with \eqref{estimate 59}, prove that for all $\epsilon \in (0,1)$ fixed, there exists a maximal time $T_{\epsilon} > 0$ and maximal $(\bar{u}_{4}^{\epsilon}, \bar{b}_{4}^{\epsilon}) \in C([0, T_{\epsilon}); \mathcal{C}^{\frac{1}{2} - \delta_{0}})$ such that 
\begin{equation}\label{[Equation (3.3a)][ZZ17]}
\sup_{t\in [0, T_{\epsilon}]} t^{\frac{ \frac{1}{2} - \delta_{0} + z + \kappa}{2}} \lVert \bar{y}_{4}^{\epsilon} (t) \rVert_{\mathcal{C}^{\frac{1}{2} - \delta_{0}}} = + \infty. 
\end{equation} 
We also define a paracontrolled ansatz of 
\begin{subequations}
\begin{align}
\bar{u}_{4}^{\epsilon, i} = -\frac{1}{2} \sum_{i_{1}, j=1}^{3} &\mathcal{P}^{ii_{1}} D_{j} [ \pi_{<} ( \bar{u}_{3}^{\epsilon, i_{1}} + \bar{u}_{4}^{\epsilon, i_{1}}, \bar{K}_{u}^{\epsilon, j}) + \pi_{<} ( \bar{u}_{3}^{\epsilon, j} + \bar{u}_{4}^{\epsilon, j}, \bar{K}_{u}^{\epsilon, i_{1}}) \nonumber\\
& -\pi_{<} (\bar{b}_{3}^{\epsilon, i_{1}} + \bar{b}_{4}^{\epsilon, i_{1}}, \bar{K}_{b}^{\epsilon, j}) - \pi_{<}(\bar{b}_{3}^{\epsilon, j} + \bar{b}_{4}^{\epsilon, j}, \bar{K}_{b}^{\epsilon, i_{1}})] + \bar{u}^{\epsilon, \sharp, i},\label{paracontrolled ansatz u} \\
\bar{b}_{4}^{\epsilon, i} = -\frac{1}{2} \sum_{i_{1}, j=1}^{3} &\mathcal{P}^{ii_{1}} D_{j} [ -\pi_{<} ( \bar{u}_{3}^{\epsilon, i_{1}} + \bar{u}_{4}^{\epsilon, i_{1}}, \bar{K}_{b}^{\epsilon, j}) + \pi_{<} ( \bar{u}_{3}^{\epsilon, j} + \bar{u}_{4}^{\epsilon, j}, \bar{K}_{b}^{\epsilon, i_{1}}) \nonumber\\
&+ \pi_{<} (\bar{b}_{3}^{\epsilon, i_{1}} + \bar{b}_{4}^{\epsilon, i_{1}}, \bar{K}_{u}^{\epsilon, j}) - \pi_{<}(\bar{b}_{3}^{\epsilon, j} + \bar{b}_{4}^{\epsilon, j}, \bar{K}_{u}^{\epsilon, i_{1}})] + \bar{b}^{\epsilon, \sharp, i} \label{paracontrolled ansatz b}
\end{align}
\end{subequations} 
for some $\bar{u}^{\epsilon, \sharp, i}$ and $\bar{b}^{\epsilon, \sharp, i}$ with their regularity to be derived subsequently. Let us set $\bar{y}^{\epsilon, \sharp} \triangleq (\bar{u}^{\epsilon, \sharp}, \bar{b}^{\epsilon, \sharp})$ and 
\begin{equation}\label{beta}
\beta \in ( \frac{\delta}{2}, (z+ 2 \delta - \frac{1}{2}) \wedge (\frac{1}{2} - 2 \delta - 3 \kappa) \wedge (\frac{1}{2} - \delta_{0} - \delta - \kappa) \wedge (2- 2 z - \frac{5\delta}{2} - 4 \kappa)) 
\end{equation} 
and estimate similarly to \cite[Equations (45)-(46)]{Y19a} e.g., 
\begin{align}\label{estimate 60}
& t^{\frac{ \frac{1}{2} + \beta + z + \kappa}{2}} \lVert \int_{0}^{t} P_{t-s} \sum_{i_{1}, j=1}^{3} \mathcal{P}^{ii_{1}} D_{j} (\bar{u}_{4}^{\epsilon, i_{1}} \bar{b}_{4}^{\epsilon, j}) ds \rVert_{\mathcal{C}^{\frac{1}{2} + \beta}} \nonumber\\
\lesssim& t^{\frac{ \frac{1}{2} + \delta_{0} - z - \kappa}{2}} ( \sup_{s \in [0,t]} s^{\frac{ \frac{1}{2} - \delta_{0} + z + \kappa}{2}} \lVert \bar{y}_{4}^{\epsilon}(s) \rVert_{\mathcal{C}^{\frac{1}{2} - \delta_{0}}})^{2}
\end{align} 
by Lemmas \ref{Lemma 3.10}, \ref{Lemma 3.7} and \ref{[Lemma 1.1][Y19a]} (4), \eqref{beta} and \eqref{[Equation (3.3c)][ZZ17]}. Considering \eqref{[Equation (3.3a)][ZZ17]} in \eqref{estimate 60}, we see that $(\bar{u}^{\epsilon, \sharp, i}, \bar{b}^{\epsilon, \sharp, i})(t) \in \mathcal{C}^{\frac{1}{2} + \beta}$ for all $t \in (0, T_{\epsilon})$ and all $i \in \{1,2,3\}$. Similarly to \cite[Equations (47)-(49)]{Y19a}, we can compute 
\begin{align}\label{[Equation (3.4)][ZZ17]}
& \lVert (\bar{u}_{4}^{\epsilon, i}, \bar{b}_{4}^{\epsilon, i})(t) \rVert_{\mathcal{C}^{\frac{1}{2} - \delta - \kappa}} \nonumber\\
\lesssim& \sum_{i_{1}, j=1}^{3} \lVert (\bar{u}_{3}^{\epsilon, i_{1}} + \bar{u}_{4}^{\epsilon, i_{1}})(t) \rVert_{\mathcal{C}^{\frac{1}{2} - \delta_{0}}} \lVert \int_{0}^{t} P_{t-s} \bar{u}_{1}^{\epsilon, j} ds \rVert_{\mathcal{C}^{\frac{3}{2} - \delta - \kappa}}  \nonumber\\
&+ \lVert (\bar{b}_{3}^{\epsilon, i_{1}} + \bar{b}_{4}^{\epsilon, i_{1}})(t) \rVert_{\mathcal{C}^{\frac{1}{2} - \delta_{0}}} \lVert \int_{0}^{t} P_{t-s} \bar{b}_{1}^{\epsilon, j} ds \rVert_{\mathcal{C}^{\frac{3}{2} - \delta - \kappa}} + \lVert \bar{u}^{\epsilon, \sharp, i}(t) \rVert_{\mathcal{C}^{\frac{1}{2} + \beta}} \nonumber\\
&+   \lVert (\bar{u}_{3}^{\epsilon, i_{1}} + \bar{u}_{4}^{\epsilon, i_{1}})(t) \rVert_{\mathcal{C}^{\frac{1}{2} - \delta_{0}}} \lVert \int_{0}^{t} P_{t-s} \bar{b}_{1}^{\epsilon, j} ds \rVert_{\mathcal{C}^{\frac{3}{2} - \delta - \kappa}}  \nonumber\\
&+ \lVert (\bar{b}_{3}^{\epsilon, i_{1}} + \bar{b}_{4}^{\epsilon, i_{1}})(t) \rVert_{\mathcal{C}^{\frac{1}{2} - \delta_{0}}} \lVert \int_{0}^{t} P_{t-s} \bar{u}_{1}^{\epsilon, j} ds \rVert_{\mathcal{C}^{\frac{3}{2} - \delta - \kappa}} + \lVert \bar{b}^{\epsilon, \sharp, i}(t) \rVert_{\mathcal{C}^{\frac{1}{2} + \beta}} \nonumber\\
\lesssim& t^{\frac{\kappa}{2}}\sum_{i_{1}, j=1}^{3} \lVert (\bar{u}_{3}^{\epsilon, i_{1}} + \bar{u}_{4}^{\epsilon, i_{1}}, \bar{b}_{3}^{\epsilon, i_{1}} + \bar{b}_{4}^{\epsilon, i_{1}})(t) \rVert_{\mathcal{C}^{\frac{1}{2} - \delta_{0}}} \sup_{s\in [0, t]} \lVert (\bar{u}_{1}^{\epsilon, j}, \bar{b}_{1}^{\epsilon,j})(s) \rVert_{\mathcal{C}^{-\frac{1}{2} - \delta}} \nonumber\\
& \hspace{70mm} + \lVert (\bar{u}^{\epsilon, \sharp, i}, \bar{b}^{\epsilon, \sharp, i})(t)\rVert_{\mathcal{C}^{\frac{1}{2} + \beta}} 
\end{align} 
by \eqref{paracontrolled ansatz u}-\eqref{paracontrolled ansatz b}, Lemmas \ref{Lemma 3.10}, \ref{Lemma 3.7} and  \ref{[Lemma 1.1][Y19a]} (1), \eqref{regularity 0} and \eqref{estimate 56}.  Now it can be verified using \eqref{Bony's decomposition}, \eqref{[Equation (3.1aa)][ZZ17]} -  \eqref{[Equation (3.2c)][ZZ17]}, that $\bar{u}^{\epsilon, \sharp, i}$ and $\bar{b}^{\epsilon,\sharp, i}$ defined by \eqref{paracontrolled ansatz u}-\eqref{paracontrolled ansatz b} solve the following equations if and only if $(\bar{u}^{\epsilon}, \bar{b}^{\epsilon}) = (\sum_{i=1}^{4} \bar{u}_{i}^{\epsilon}, \sum_{i=1}^{4} \bar{b}_{i}^{\epsilon})$ solves \eqref{[Equation (3.1)][ZZ17]}: 
\begin{align}\label{[Equation (3.5) for u][ZZ17]}
\bar{u}^{\epsilon, \sharp, i} =P_{t}(\sum_{i_{1} =1}^{3} \mathcal{P}^{ii_{1}} u_{0}^{i_{1}} - \bar{u}_{1}^{\epsilon, i}(0))  + \int_{0}^{t} P_{t-s} \bar{\phi}_{u}^{\epsilon, \sharp, i} ds + \bar{F}_{u}^{\epsilon, i}(t) 
\end{align} 
where 
\begin{align}\label{overlinephiuepsilonsharpi}
\bar{\phi}_{u}^{\epsilon, \sharp, i} \triangleq - \frac{1}{2} \sum_{i_{1}, j=1}^{3}\mathcal{P}^{ii_{1}} & D_{j} [\bar{u}_{2}^{\epsilon, i_{1}} \diamond \bar{u}_{2}^{\epsilon, j} + \bar{u}_{2}^{\epsilon, i_{1}} (\bar{u}_{3}^{\epsilon, j} + \bar{u}_{4}^{\epsilon, j}) \\
&+ \bar{u}_{2}^{\epsilon, j} (\bar{u}_{3}^{\epsilon, i_{1}} + \bar{u}_{4}^{\epsilon, i_{1}})  + (\bar{u}_{3}^{\epsilon, i_{1}} + \bar{u}_{4}^{\epsilon, i_{1}})(\bar{u}_{3}^{\epsilon, j} + \bar{u}_{4}^{\epsilon, j})  \nonumber\\
&+ \pi_{>} (\bar{u}_{3}^{\epsilon, i_{1}} + \bar{u}_{4}^{\epsilon, i_{1}}, \bar{u}_{1}^{\epsilon, j}) + \pi_{0,\diamond} (\bar{u}_{3}^{\epsilon, i_{1}}, \bar{u}_{1}^{\epsilon, j}) + \pi_{0,\diamond} (\bar{u}_{4}^{\epsilon, i_{1}}, \bar{u}_{1}^{\epsilon, j})\nonumber\\
&+ \pi_{>} (\bar{u}_{3}^{\epsilon, j} + \bar{u}_{4}^{\epsilon, j}, \bar{u}_{1}^{\epsilon, i_{1}}) + \pi_{0,\diamond} (\bar{u}_{3}^{\epsilon, j}, \bar{u}_{1}^{\epsilon, i_{1}}) + \pi_{0,\diamond} (\bar{u}_{4}^{\epsilon, j}, \bar{u}_{1}^{\epsilon, i_{1}}) \nonumber \\
& - \bar{b}_{2}^{\epsilon, i_{1}} \diamond \bar{b}_{2}^{\epsilon, j} - \bar{b}_{2}^{\epsilon, i_{1}} (\bar{b}_{3}^{\epsilon, j} + \bar{u}_{4}^{\epsilon, j}) \nonumber\\
&- \bar{b}_{2}^{\epsilon, j} (\bar{b}_{3}^{\epsilon, i_{1}} + \bar{b}_{4}^{\epsilon, i_{1}})  - (\bar{b}_{3}^{\epsilon, i_{1}} + \bar{b}_{4}^{\epsilon, i_{1}})(\bar{b}_{3}^{\epsilon, j} + \bar{b}_{4}^{\epsilon, j})  \nonumber\\
&- \pi_{>} (\bar{b}_{3}^{\epsilon, i_{1}} + \bar{b}_{4}^{\epsilon, i_{1}}, \bar{b}_{1}^{\epsilon, j}) - \pi_{0,\diamond} (\bar{b}_{3}^{\epsilon, i_{1}}, \bar{b}_{1}^{\epsilon, j}) - \pi_{0,\diamond} (\bar{b}_{4}^{\epsilon, i_{1}}, \bar{b}_{1}^{\epsilon, j})\nonumber\\
&- \pi_{>} (\bar{b}_{3}^{\epsilon, j} + \bar{b}_{4}^{\epsilon, j}, \bar{b}_{1}^{\epsilon, i_{1}}) - \pi_{0,\diamond} (\bar{b}_{3}^{\epsilon, j}, \bar{b}_{1}^{\epsilon, i_{1}}) - \pi_{0,\diamond} (\bar{b}_{4}^{\epsilon, j}, \bar{b}_{1}^{\epsilon, i_{1}})] \nonumber 
\end{align} 
and 
\begin{align}\label{overlineFuepsiloni}
\bar{F}_{u}^{\epsilon, i}(t) \triangleq& - \frac{1}{2} \int_{0}^{t} P_{t-s} \sum_{i_{1}, j=1}^{3} \mathcal{P}^{ii_{1}} D_{j} [ \pi_{<} (\bar{u}_{3}^{\epsilon, i_{1}} + \bar{u}_{4}^{\epsilon, i_{1}}, \bar{u}_{1}^{\epsilon, j}) + \pi_{<} (\bar{u}_{3}^{\epsilon, j} + \bar{u}_{4}^{\epsilon, j}, \bar{u}_{1}^{\epsilon, i_{1}}) \nonumber\\
& \hspace{30mm} - \pi_{<} (\bar{b}_{3}^{\epsilon, i_{1}} + \bar{b}_{4}^{\epsilon, i_{1}}, \bar{b}_{1}^{\epsilon, j}) - \pi_{<}(\bar{b}_{3}^{\epsilon, j} + \bar{b}_{4}^{\epsilon, j}, \bar{b}_{1}^{\epsilon, i_{1}})] ds \nonumber\\
&+ \frac{1}{2} \sum_{i_{1}, j=1}^{3} \mathcal{P}^{ii_{1}} D_{j} [ \pi_{<} (\bar{u}_{3}^{\epsilon, i_{1}} + \bar{u}_{4}^{\epsilon, i_{1}}, \bar{K}_{u}^{\epsilon, j}) + \pi_{<}(\bar{u}_{3}^{\epsilon, j} + \bar{u}_{4}^{\epsilon, j}, \bar{K}_{u}^{\epsilon, i_{1}}) \nonumber\\
& \hspace{20mm} - \pi_{<} (\bar{b}_{3}^{\epsilon, i_{1}} + \bar{b}_{4}^{\epsilon, i_{1}}, \bar{K}_{b}^{\epsilon, j} - \pi_{<} (\bar{b}_{3}^{\epsilon, j} + \bar{b}_{4}^{\epsilon, j}, \bar{K}_{b}^{\epsilon, i_{1}})], 
\end{align} 
and similarly 
\begin{align}\label{[Equation (3.5) for b][ZZ17]}
\bar{b}^{\epsilon, \sharp, i} = P_{t}(\sum_{i_{1} =1}^{3} \mathcal{P}^{ii_{1}} b_{0}^{i_{1}} - \bar{b}_{1}^{\epsilon, i}(0))  + \int_{0}^{t} P_{t-s} \bar{\phi}_{b}^{\epsilon, \sharp, i} ds + \bar{F}_{b}^{\epsilon, i}(t) 
\end{align} 
where 
\begin{align}\label{overlinephibepsilonsharpi}
\bar{\phi}_{b}^{\epsilon, \sharp, i} \triangleq - \frac{1}{2} \sum_{i_{1}, j=1}^{3}\mathcal{P}^{ii_{1}} D_{j} &[\bar{b}_{2}^{\epsilon, i_{1}} \diamond \bar{u}_{2}^{\epsilon, j} + \bar{b}_{2}^{\epsilon, i_{1}} (\bar{u}_{3}^{\epsilon, j} + \bar{u}_{4}^{\epsilon, j}) \\
&+ \bar{u}_{2}^{\epsilon, j} (\bar{b}_{3}^{\epsilon, i_{1}} + \bar{b}_{4}^{\epsilon, i_{1}})  + (\bar{b}_{3}^{\epsilon, i_{1}} + \bar{b}_{4}^{\epsilon, i_{1}})(\bar{u}_{3}^{\epsilon, j} + \bar{u}_{4}^{\epsilon, j})  \nonumber\\
&+ \pi_{>} (\bar{b}_{3}^{\epsilon, i_{1}} + \bar{b}_{4}^{\epsilon, i_{1}}, \bar{u}_{1}^{\epsilon, j}) + \pi_{0,\diamond} (\bar{b}_{3}^{\epsilon, i_{1}}, \bar{u}_{1}^{\epsilon, j}) + \pi_{0,\diamond} (\bar{b}_{4}^{\epsilon, i_{1}}, \bar{u}_{1}^{\epsilon, j})\nonumber\\
&+ \pi_{>} (\bar{u}_{3}^{\epsilon, j} + \bar{u}_{4}^{\epsilon, j}, \bar{b}_{1}^{\epsilon, i_{1}}) + \pi_{0,\diamond} (\bar{u}_{3}^{\epsilon, j}, \bar{b}_{1}^{\epsilon, i_{1}}) + \pi_{0,\diamond} (\bar{u}_{4}^{\epsilon, j}, \bar{b}_{1}^{\epsilon, i_{1}}) \nonumber \\
& - \bar{u}_{2}^{\epsilon, i_{1}} \diamond \bar{b}_{2}^{\epsilon, j} - \bar{u}_{2}^{\epsilon, i_{1}} (\bar{b}_{3}^{\epsilon, j} + \bar{u}_{4}^{\epsilon, j}) \nonumber\\
&- \bar{b}_{2}^{\epsilon, j} (\bar{u}_{3}^{\epsilon, i_{1}} + \bar{u}_{4}^{\epsilon, i_{1}})  - (\bar{u}_{3}^{\epsilon, i_{1}} + \bar{u}_{4}^{\epsilon, i_{1}})(\bar{b}_{3}^{\epsilon, j} + \bar{b}_{4}^{\epsilon, j})  \nonumber\\
&- \pi_{>} (\bar{u}_{3}^{\epsilon, i_{1}} + \bar{u}_{4}^{\epsilon, i_{1}}, \bar{b}_{1}^{\epsilon, j}) - \pi_{0,\diamond} (\bar{u}_{3}^{\epsilon, i_{1}}, \bar{b}_{1}^{\epsilon, j}) - \pi_{0,\diamond} (\bar{u}_{4}^{\epsilon, i_{1}}, \bar{b}_{1}^{\epsilon, j})\nonumber\\
&- \pi_{>} (\bar{b}_{3}^{\epsilon, j} + \bar{b}_{4}^{\epsilon, j}, \bar{u}_{1}^{\epsilon, i_{1}}) - \pi_{0,\diamond} (\bar{b}_{3}^{\epsilon, j}, \bar{u}_{1}^{\epsilon, i_{1}}) - \pi_{0,\diamond} (\bar{b}_{4}^{\epsilon, j}, \bar{u}_{1}^{\epsilon, i_{1}})] \nonumber 
\end{align} 
and 
\begin{align}\label{overlineFbepsiloni}
\bar{F}_{b}^{\epsilon, i}(t) \triangleq& - \frac{1}{2} \int_{0}^{t} P_{t-s} \sum_{i_{1}, j=1}^{3} \mathcal{P}^{ii_{1}} D_{j} [ \pi_{<} (\bar{b}_{3}^{\epsilon, i_{1}} + \bar{b}_{4}^{\epsilon, i_{1}}, \bar{u}_{1}^{\epsilon, j}) + \pi_{<} (\bar{u}_{3}^{\epsilon, j} + \bar{u}_{4}^{\epsilon, j}, \bar{b}_{1}^{\epsilon, i_{1}})\nonumber \\
& \hspace{30mm} - \pi_{<} (\bar{u}_{3}^{\epsilon, i_{1}} + \bar{u}_{4}^{\epsilon, i_{1}}, \bar{b}_{1}^{\epsilon, j}) - \pi_{<}(\bar{b}_{3}^{\epsilon, j} + \bar{b}_{4}^{\epsilon, j}, \bar{u}_{1}^{\epsilon, i_{1}})] ds \nonumber\\
&+ \frac{1}{2} \sum_{i_{1}, j=1}^{3} \mathcal{P}^{ii_{1}} D_{j} [ -\pi_{<} (\bar{u}_{3}^{\epsilon, i_{1}} + \bar{u}_{4}^{\epsilon, i_{1}}, \bar{K}_{b}^{\epsilon, j}) + \pi_{<}(\bar{u}_{3}^{\epsilon, j} + \bar{u}_{4}^{\epsilon, j}, \bar{K}_{b}^{\epsilon, i_{1}}) \nonumber\\
& \hspace{20mm} + \pi_{<} (\bar{b}_{3}^{\epsilon, i_{1}} + \bar{b}_{4}^{\epsilon, i_{1}}, \bar{K}_{u}^{\epsilon, j} - \pi_{<} (\bar{b}_{3}^{\epsilon, j} + \bar{b}_{4}^{\epsilon, j}, \bar{K}_{u}^{\epsilon, i_{1}})]. 
\end{align} 
Now we obtain some estimates on $\bar{\phi}_{u}^{\epsilon, \sharp, i}, \bar{\phi}_{b}^{\epsilon, \sharp, i}, \bar{F}_{u}^{\epsilon, i}$ and $\bar{F}_{b}^{\epsilon, i}$. 

\begin{proposition}\label{[Lemma 3.9][ZZ17]}
Let $\delta_{0}$ and $z$ satisfy the hypothesis of Theorem \ref{Theorem 1.2}. For all $i \in \{1,2,3\}$, $\delta$, $\kappa$ and $\beta$ defined respectively in  \eqref{[Equation (3.2g)][ZZ17]}, \eqref{[Equation (3.3c)][ZZ17]} and \eqref{beta}, $\bar{\phi}_{u}^{\epsilon, \sharp, i}$ and $\bar{\phi}_{b}^{\epsilon, \sharp, i}$ defined respectively in \eqref{overlinephiuepsilonsharpi} and \eqref{overlinephibepsilonsharpi} satisfy 
\begin{align}\label{[Equation (3.6)][ZZ17]}
&\lVert (\bar{\phi}_{u}^{\epsilon, \sharp, i}, \bar{\phi}_{b}^{\epsilon, \sharp, i})(t) \rVert_{\mathcal{C}^{-1 - 2 \delta - 2 \kappa}} \nonumber\\
\lesssim& (1+ ( \bar{C}_{W}^{\epsilon}(t))^{3}) (1+ \lVert \bar{y}^{\epsilon,\sharp}(t) \rVert_{\mathcal{C}^{\frac{1}{2} + \beta}} + \lVert \bar{y}_{4}^{\epsilon}(t) \rVert_{\mathcal{C}^{\frac{1}{2} - \delta_{0}}} ) + \lVert \bar{y}_{4}^{\epsilon}(t) \rVert_{\mathcal{C}^{\delta}}^{2}. 
\end{align} 
\end{proposition} 

\begin{proof}[Proof of Proposition \ref{[Lemma 3.9][ZZ17]}]
We compute similarly to how \cite[Equation (97)]{Y19a} is derived. First, within \eqref{overlinephiuepsilonsharpi} and \eqref{overlinephibepsilonsharpi} we can estimate
\begin{align}\label{estimate 71}
 \lVert ( \bar{u}_{2}^{\epsilon, i_{1}} \diamond \bar{u}_{2}^{\epsilon, j}, \bar{b}_{2}^{\epsilon, i_{1}} \diamond \bar{b}_{2}^{\epsilon, j}, \bar{b}_{2}^{\epsilon, i_{1}} \diamond \bar{u}_{2}^{\epsilon, j}, \bar{u}_{2}^{\epsilon, i_{1}} \diamond \bar{b}_{2}^{\epsilon, j}) \rVert_{\mathcal{C}^{-2\delta - 2 \kappa}}  \lesssim \bar{C}_{W}^{\epsilon}(t)
\end{align}
by \eqref{regularity 0} and \eqref{[Equation (3.2e)][ZZ17]}. Second, within \eqref{overlinephiuepsilonsharpi} and \eqref{overlinephibepsilonsharpi} we estimate 
\begin{align}\label{estimate 72}
& \lVert ( \bar{u}_{2}^{\epsilon, i_{1}} ( \bar{u}_{3}^{\epsilon, j} + \bar{u}_{4}^{\epsilon, j}), \bar{u}_{2}^{\epsilon, j} (\bar{u}_{3}^{\epsilon, i_{1}} + \bar{u}_{4}^{\epsilon, i_{1}}), \bar{b}_{2}^{\epsilon, i_{1}}(\bar{b}_{3}^{\epsilon, j} + \bar{b}_{4}^{\epsilon, j}), \bar{b}_{2}^{\epsilon, j}(\bar{b}_{3}^{\epsilon, i_{1}} + \bar{b}_{4}^{\epsilon, i_{1}}), \nonumber \\
&\bar{b}_{2}^{\epsilon, i_{1}} ( \bar{u}_{3}^{\epsilon, j} + \bar{u}_{4}^{\epsilon, j}), \bar{u}_{2}^{\epsilon, j} (\bar{b}_{3}^{\epsilon, i_{1}} + \bar{b}_{4}^{\epsilon, i_{1}}), \bar{u}_{2}^{\epsilon, i_{1}}(\bar{b}_{3}^{\epsilon, j} + \bar{b}_{4}^{\epsilon, j}), \bar{b}_{2}^{\epsilon, j}(\bar{u}_{3}^{\epsilon, i_{1}} + \bar{u}_{4}^{\epsilon, i_{1}})) \rVert_{\mathcal{C}^{-2\delta - 2 \kappa}} \nonumber \\
\lesssim& (\bar{C}_{W}^{\epsilon}(t))^{2} + \bar{C}_{W}^{\epsilon}(t) \lVert \bar{y}_{4}^{\epsilon} \rVert_{\mathcal{C}^{\frac{1}{2} - \delta_{0}}}
\end{align} 
by \eqref{regularity 0}, Lemma \ref{[Lemma 1.1][Y19a]} (4), \eqref{[Equation (3.2g)][ZZ17]}, \eqref{[Equation (3.2e)][ZZ17]} and \eqref{[Equation (3.2h)][ZZ17]}. Third, within \eqref{overlinephiuepsilonsharpi} and \eqref{overlinephibepsilonsharpi} we estimate 
\begin{align}\label{estimate 73}
& \lVert ( ( \bar{u}_{3}^{\epsilon, i_{1}} + \bar{u}_{4}^{\epsilon, i_{1}})(\bar{u}_{3}^{\epsilon, j} + \bar{u}_{4}^{\epsilon, j}), (\bar{b}_{3}^{\epsilon, i_{1}} + \bar{b}_{4}^{\epsilon, i_{1}})(\bar{b}_{3}^{\epsilon, j} + \bar{b}_{4}^{\epsilon, j}), \\
& \hspace{10mm} (\bar{b}_{3}^{\epsilon, i_{1}} + \bar{b}_{4}^{\epsilon, i_{1}})(\bar{u}_{3}^{\epsilon, j}+ \bar{u}_{4}^{\epsilon, j}), (\bar{u}_{3}^{\epsilon, i_{1}} + \bar{u}_{4}^{\epsilon, i_{1}})(\bar{b}_{3}^{\epsilon, j} + \bar{b}_{4}^{\epsilon, j}) ) \rVert_{\mathcal{C}^{-2\delta - 2 \kappa}} \nonumber \\
\lesssim& \lVert \bar{u}_{3}^{\epsilon} + \bar{u}_{4}^{\epsilon} \rVert_{\mathcal{C}^{\delta}}^{2} + \lVert \bar{b}_{3}^{\epsilon} + \bar{b}_{4}^{\epsilon} \rVert_{\mathcal{C}^{\delta}}^{2}  \lesssim (\bar{C}_{W}^{\epsilon}(t))^{2} + \lVert \bar{y}_{4}^{\epsilon}\rVert_{\mathcal{C}^{\delta}}^{2} \nonumber
\end{align} 
by \eqref{regularity 0}, Lemma \ref{[Lemma 1.1][Y19a]} (4), \eqref{[Equation (3.2g)][ZZ17]} and \eqref{[Equation (3.2h)][ZZ17]}. Fourth, within \eqref{overlinephiuepsilonsharpi} and \eqref{overlinephibepsilonsharpi} we estimate 
\begin{align}\label{estimate 74}
& \lVert ( \pi_{0,\diamond} (\bar{u}_{3}^{\epsilon, i_{1}}, \bar{u}_{1}^{\epsilon, j}), \pi_{0,\diamond}(\bar{u}_{3}^{\epsilon, j}, \bar{u}_{1}^{\epsilon, i_{1}}), \pi_{0,\diamond} (\bar{b}_{3}^{\epsilon, i_{1}}, \bar{b}_{1}^{\epsilon, j}), \pi_{0,\diamond} (\bar{b}_{3}^{\epsilon, j}, \bar{b}_{1}^{\epsilon, i_{1}}), \\
&\pi_{0,\diamond} (\bar{b}_{3}^{\epsilon, i_{1}}, \bar{u}_{1}^{\epsilon, j}), \pi_{0,\diamond}(\bar{u}_{3}^{\epsilon, j}, \bar{b}_{1}^{\epsilon, i_{1}}), \pi_{0,\diamond} (\bar{u}_{3}^{\epsilon, i_{1}}, \bar{b}_{1}^{\epsilon, j}), \pi_{0,\diamond} (\bar{b}_{3}^{\epsilon, j}, \bar{u}_{1}^{\epsilon, i_{1}})) \rVert_{\mathcal{C}^{-2\delta - 2\kappa}} \lesssim \bar{C}_{W}^{\epsilon}(t) \nonumber
\end{align}
by \eqref{regularity 0} and  \eqref{[Equation (3.2e)][ZZ17]}. Fifth, within \eqref{overlinephiuepsilonsharpi} and \eqref{overlinephibepsilonsharpi}  we estimate 
\begin{align}\label{estimate 75}
& \lVert ( \pi_{>} ( \bar{u}_{3}^{\epsilon, i_{1}} + \bar{u}_{4}^{\epsilon, i_{1}}, \bar{u}_{1}^{\epsilon, j}), \pi_{>} (\bar{u}_{3}^{\epsilon, j} + \bar{u}_{4}^{\epsilon, j}, \bar{u}_{1}^{\epsilon, i_{1}}), \pi_{>}(\bar{b}_{3}^{\epsilon, i_{1}} + \bar{b}_{4}^{\epsilon, i_{1}}, \bar{b}_{1}^{\epsilon, j}), \nonumber\\
& \hspace{5mm}  \pi_{>} (\bar{b}_{3}^{\epsilon, j} + \bar{b}_{4}^{\epsilon, j}, \bar{b}_{1}^{\epsilon, i_{1}}),   \pi_{>} ( \bar{b}_{3}^{\epsilon, i_{1}} + \bar{b}_{4}^{\epsilon, i_{1}}, \bar{u}_{1}^{\epsilon, j}), \pi_{>} (\bar{u}_{3}^{\epsilon, j} + \bar{u}_{4}^{\epsilon, j}, \bar{b}_{1}^{\epsilon, i_{1}}), \nonumber\\
& \hspace{35mm}   \pi_{>}(\bar{u}_{3}^{\epsilon, i_{1}} + \bar{u}_{4}^{\epsilon, i_{1}}, \bar{b}_{1}^{\epsilon, j}), \pi_{>} (\bar{b}_{3}^{\epsilon, j} + \bar{b}_{4}^{\epsilon, j}, \bar{u}_{1}^{\epsilon, i_{1}})) \rVert_{\mathcal{C}^{-2\delta - 2 \kappa}} \nonumber\\
\lesssim& ( \bar{C}_{W}^{\epsilon}(t) + \lVert (\bar{u}_{3}^{\epsilon} + \bar{u}_{4}^{\epsilon}, \bar{b}_{3}^{\epsilon} + \bar{b}_{4}^{\epsilon}) \rVert_{\mathcal{C}^{\frac{1}{2} - \delta_{0}}} \sup_{s \in [0, t]} \lVert \bar{y}_{1}^{\epsilon} \rVert_{\mathcal{C}^{-\frac{1}{2}- \delta}} + \lVert \bar{y}^{\epsilon, \sharp} \rVert_{\mathcal{C}^{\frac{1}{2} + \beta}} ) \bar{C}_{W}^{\epsilon}(t) \nonumber \\
\lesssim& (\bar{C}_{W}^{\epsilon}(t))^{2} + (\bar{C}_{W}^{\epsilon}(t))^{3} + (\bar{C}_{W}^{\epsilon}(t))^{2} \lVert \bar{y}_{4}^{\epsilon}  \rVert_{\mathcal{C}^{\frac{1}{2} - \delta_{0}}} +\bar{C}_{W}^{\epsilon}(t) \lVert \bar{y}^{\epsilon,\sharp} \rVert_{\mathcal{C}^{\frac{1}{2} + \beta}} 
\end{align}
by \eqref{regularity 0}, Lemma \ref{[Lemma 1.1][Y19a]} (2), \eqref{[Equation (3.2e)][ZZ17]}, \eqref{[Equation (3.2h)][ZZ17]}, \eqref{[Equation (3.4)][ZZ17]}, and \eqref{[Equation (3.2g)][ZZ17]}. Finally, within \eqref{overlinephiuepsilonsharpi} and \eqref{overlinephibepsilonsharpi} we estimate the following in $\mathcal{C}^{-\delta}$-norm because we will need this subsequently in the proof of \eqref{[Equation (3.7d)][ZZ17]}, and it suffices because $-2\delta - 2 \kappa \leq -\delta$. We compute 
\begin{align}\label{estimate 61}
& \lVert ( \pi_{0,\diamond} (\bar{u}_{4}^{\epsilon, i_{1}}, \bar{u}_{1}^{\epsilon, j}), \pi_{0,\diamond} (\bar{u}_{4}^{\epsilon, j}, \bar{u}_{1}^{\epsilon, i_{1}}), \pi_{0,\diamond} (\bar{b}_{4}^{\epsilon, i_{1}}, \bar{b}_{1}^{\epsilon, j}), \pi_{0,\diamond} (\bar{b}_{4}^{\epsilon, j}, \bar{b}_{1}^{\epsilon, i_{1}}), \\
& \pi_{0,\diamond} (\bar{b}_{4}^{\epsilon, i_{1}}, \bar{u}_{1}^{\epsilon, j}), \pi_{0,\diamond} (\bar{u}_{4}^{\epsilon, j}, \bar{b}_{1}^{\epsilon, i_{1}}), \pi_{0,\diamond} (\bar{u}_{4}^{\epsilon, i_{1}}, \bar{b}_{1}^{\epsilon, j}), \pi_{0,\diamond} (\bar{b}_{4}^{\epsilon, j}, \bar{u}_{1}^{\epsilon, i_{1}})) \rVert_{\mathcal{C}^{-\delta}} \nonumber\\
=& \lVert ( \pi_{0} (- \frac{1}{2} \sum_{i_{2}, j_{1} =1}^{3} \mathcal{P}^{i_{1} i_{2}} D_{j_{1}} [ \pi_{<} ( \bar{u}_{3}^{\epsilon, i_{2}} + \bar{u}_{4}^{\epsilon, i_{2}}, \bar{K}_{u}^{\epsilon, j_{1}}) + \pi_{<} (\bar{u}_{3}^{\epsilon, j_{1}} + \bar{u}_{4}^{\epsilon, j_{1}}, \bar{K}_{u}^{\epsilon, i_{2}}) \nonumber\\
& \hspace{20mm} - \pi_{<} (\bar{b}_{3}^{\epsilon, i_{2}} + \bar{b}_{4}^{\epsilon, i_{2}}, \bar{K}_{b}^{\epsilon, j_{1}}) - \pi_{<} (\bar{b}_{3}^{\epsilon, j_{1}} + \bar{b}_{4}^{\epsilon, j_{1}}, \bar{K}_{b}^{\epsilon, i_{2}})] + \bar{u}^{\epsilon, \sharp, i_{1}}, \bar{u}_{1}^{\epsilon, j}), \nonumber \\
& \hspace{5mm}  \pi_{0} (- \frac{1}{2} \sum_{i_{2}, j_{1} =1}^{3} \mathcal{P}^{j i_{2}} D_{j_{1}} [ \pi_{<} ( \bar{u}_{3}^{\epsilon, i_{2}} + \bar{u}_{4}^{\epsilon, i_{2}}, \bar{K}_{u}^{\epsilon, j_{1}}) + \pi_{<} (\bar{u}_{3}^{\epsilon, j_{1}} + \bar{u}_{4}^{\epsilon, j_{1}}, \bar{K}_{u}^{\epsilon, i_{2}}) \nonumber\\
& \hspace{20mm} - \pi_{<} (\bar{b}_{3}^{\epsilon, i_{2}} + \bar{b}_{4}^{\epsilon, i_{2}}, \bar{K}_{b}^{\epsilon, j_{1}}) - \pi_{<} (\bar{b}_{3}^{\epsilon, j_{1}} + \bar{b}_{4}^{\epsilon, j_{1}}, \bar{K}_{b}^{\epsilon, i_{2}})] + \bar{u}^{\epsilon, \sharp, j}, \bar{u}_{1}^{\epsilon, i_{1}}), \nonumber \\
& \hspace{5mm}  \pi_{0} (- \frac{1}{2} \sum_{i_{2}, j_{1} =1}^{3} \mathcal{P}^{i_{1} i_{2}} D_{j_{1}} [- \pi_{<} ( \bar{u}_{3}^{\epsilon, i_{2}} + \bar{u}_{4}^{\epsilon, i_{2}}, \bar{K}_{b}^{\epsilon, j_{1}}) + \pi_{<} (\bar{u}_{3}^{\epsilon, j_{1}} + \bar{u}_{4}^{\epsilon, j_{1}}, \bar{K}_{b}^{\epsilon, i_{2}}) \nonumber\\
& \hspace{20mm} + \pi_{<} (\bar{b}_{3}^{\epsilon, i_{2}} + \bar{b}_{4}^{\epsilon, i_{2}}, \bar{K}_{u}^{\epsilon, j_{1}}) - \pi_{<} (\bar{b}_{3}^{\epsilon, j_{1}} + \bar{b}_{4}^{\epsilon, j_{1}}, \bar{K}_{u}^{\epsilon, i_{2}})] + \bar{b}^{\epsilon, \sharp, i_{1}}, \bar{b}_{1}^{\epsilon, j}), \nonumber \\
& \hspace{5mm}  \pi_{0} (- \frac{1}{2} \sum_{i_{2}, j_{1} =1}^{3} \mathcal{P}^{j i_{2}} D_{j_{1}} [- \pi_{<} ( \bar{u}_{3}^{\epsilon, i_{2}} + \bar{u}_{4}^{\epsilon, i_{2}}, \bar{K}_{b}^{\epsilon, j_{1}}) + \pi_{<} (\bar{u}_{3}^{\epsilon, j_{1}} + \bar{u}_{4}^{\epsilon, j_{1}}, \bar{K}_{b}^{\epsilon, i_{2}}) \nonumber\\
& \hspace{20mm} + \pi_{<} (\bar{b}_{3}^{\epsilon, i_{2}} + \bar{b}_{4}^{\epsilon, i_{2}}, \bar{K}_{u}^{\epsilon, j_{1}}) - \pi_{<} (\bar{b}_{3}^{\epsilon, j_{1}} + \bar{b}_{4}^{\epsilon, j_{1}}, \bar{K}_{u}^{\epsilon, i_{2}})] + \bar{b}^{\epsilon, \sharp, j}, \bar{b}_{1}^{\epsilon, i_{1}}), \nonumber \\
& \hspace{5mm}  \pi_{0} (- \frac{1}{2} \sum_{i_{2}, j_{1} =1}^{3} \mathcal{P}^{i_{1} i_{2}} D_{j_{1}} [- \pi_{<} ( \bar{u}_{3}^{\epsilon, i_{2}} + \bar{u}_{4}^{\epsilon, i_{2}}, \bar{K}_{b}^{\epsilon, j_{1}}) + \pi_{<} (\bar{u}_{3}^{\epsilon, j_{1}} + \bar{u}_{4}^{\epsilon, j_{1}}, \bar{K}_{b}^{\epsilon, i_{2}}) \nonumber\\
& \hspace{20mm} + \pi_{<} (\bar{b}_{3}^{\epsilon, i_{2}} + \bar{b}_{4}^{\epsilon, i_{2}}, \bar{K}_{u}^{\epsilon, j_{1}}) - \pi_{<} (\bar{b}_{3}^{\epsilon, j_{1}} + \bar{b}_{4}^{\epsilon, j_{1}}, \bar{K}_{u}^{\epsilon, i_{2}})] + \bar{b}^{\epsilon, \sharp, i_{1}}, \bar{u}_{1}^{\epsilon, j}), \nonumber \\
& \hspace{5mm}  \pi_{0} (- \frac{1}{2} \sum_{i_{2}, j_{1} =1}^{3} \mathcal{P}^{j i_{2}} D_{j_{1}} [ \pi_{<} ( \bar{u}_{3}^{\epsilon, i_{2}} + \bar{u}_{4}^{\epsilon, i_{2}}, \bar{K}_{u}^{\epsilon, j_{1}}) + \pi_{<} (\bar{u}_{3}^{\epsilon, j_{1}} + \bar{u}_{4}^{\epsilon, j_{1}}, \bar{K}_{u}^{\epsilon, i_{2}}) \nonumber\\
& \hspace{20mm} - \pi_{<} (\bar{b}_{3}^{\epsilon, i_{2}} + \bar{b}_{4}^{\epsilon, i_{2}}, \bar{K}_{b}^{\epsilon, j_{1}}) - \pi_{<} (\bar{b}_{3}^{\epsilon, j_{1}} + \bar{b}_{4}^{\epsilon, j_{1}}, \bar{K}_{b}^{\epsilon, i_{2}})] + \bar{u}^{\epsilon, \sharp, i_{1}}, \bar{b}_{1}^{\epsilon, j}), \nonumber \\
& \hspace{5mm}  \pi_{0} (- \frac{1}{2} \sum_{i_{2}, j_{1} =1}^{3} \mathcal{P}^{i_{1} i_{2}} D_{j_{1}} [ \pi_{<} ( \bar{u}_{3}^{\epsilon, i_{2}} + \bar{u}_{4}^{\epsilon, i_{2}}, \bar{K}_{u}^{\epsilon, j_{1}}) + \pi_{<} (\bar{u}_{3}^{\epsilon, j_{1}} + \bar{u}_{4}^{\epsilon, j_{1}}, \bar{K}_{u}^{\epsilon, i_{2}}) \nonumber\\
& \hspace{20mm} - \pi_{<} (\bar{b}_{3}^{\epsilon, i_{2}} + \bar{b}_{4}^{\epsilon, i_{2}}, \bar{K}_{b}^{\epsilon, j_{1}}) - \pi_{<} (\bar{b}_{3}^{\epsilon, j_{1}} + \bar{b}_{4}^{\epsilon, j_{1}}, \bar{K}_{b}^{\epsilon, i_{2}})] + \bar{u}^{\epsilon, \sharp, i_{1}}, \bar{b}_{1}^{\epsilon, i_{1}}), \nonumber \\
& \hspace{5mm}  \pi_{0} (- \frac{1}{2} \sum_{i_{2}, j_{1} =1}^{3} \mathcal{P}^{j i_{2}} D_{j_{1}} [- \pi_{<} ( \bar{u}_{3}^{\epsilon, i_{2}} + \bar{u}_{4}^{\epsilon, i_{2}}, \bar{K}_{b}^{\epsilon, j_{1}}) + \pi_{<} (\bar{u}_{3}^{\epsilon, j_{1}} + \bar{u}_{4}^{\epsilon, j_{1}}, \bar{K}_{b}^{\epsilon, i_{2}}) \nonumber\\
& \hspace{20mm} + \pi_{<} (\bar{b}_{3}^{\epsilon, i_{2}} + \bar{b}_{4}^{\epsilon, i_{2}}, \bar{K}_{u}^{\epsilon, j_{1}}) - \pi_{<} (\bar{b}_{3}^{\epsilon, j_{1}} + \bar{b}_{4}^{\epsilon, j_{1}}, \bar{K}_{u}^{\epsilon, i_{2}})] + \bar{b}^{\epsilon, \sharp, j}, \bar{u}_{1}^{\epsilon, i_{1}})) \rVert_{\mathcal{C}^{-\delta}} \nonumber 
\end{align} 
by \eqref{[Equation (3.2c)][ZZ17]}, \eqref{paracontrolled ansatz u} and \eqref{paracontrolled ansatz b}. Among the eight terms in \eqref{estimate 61}, without loss of generality (w.l.o.g.) we show details of the eighth term as others are similar.  We compute 
\begin{align}\label{estimate 69}
& \lVert \pi_{0} ( -\frac{1}{2} \sum_{i_{2}, j_{1} =1}^{3} \mathcal{P}^{j i_{2}} D_{j_{1}} [- \pi_{<} ( \bar{u}_{3}^{\epsilon, i_{2}} + \bar{u}_{4}^{\epsilon, i_{2}}, \bar{K}_{b}^{\epsilon, j_{1}}) + \pi_{<} (\bar{u}_{3}^{\epsilon, j_{1}} + \bar{u}_{4}^{\epsilon, j_{1}}, \bar{K}_{b}^{\epsilon, i_{2}}) \nonumber\\
& \hspace{10mm} + \pi_{<} (\bar{b}_{3}^{\epsilon, i_{2}} + \bar{b}_{4}^{\epsilon, i_{2}}, \bar{K}_{u}^{\epsilon, j_{1}}) - \pi_{<} (\bar{b}_{3}^{\epsilon, j_{1}} + \bar{b}_{4}^{\epsilon, j_{1}}, \bar{K}_{u}^{\epsilon, i_{2}})] + \bar{b}^{\epsilon, \sharp, j}, \bar{u}_{1}^{\epsilon, i_{1}}) \rVert_{\mathcal{C}^{-\delta}}\nonumber\\
& \hspace{80mm} \lesssim \sum_{i_{2}, j_{1} =1}^{3} \sum_{k=1}^{9} II_{t, i_{1}i_{2}j j_{1}}^{k}
\end{align} 
by Leibniz rule where 
\begin{subequations}\label{estimate 62}
\begin{align}
II_{t, i_{1} i_{2} j j_{1}}^{1} \triangleq& \lVert \pi_{0} (\mathcal{P}^{ji_{2}} \pi_{<} ( \bar{u}_{3}^{\epsilon, i_{2}} + \bar{u}_{4}^{\epsilon, i_{2}}, D_{j_{1}} \bar{K}_{b}^{\epsilon, j_{1}}), \bar{u}_{1}^{\epsilon, i_{1}}) \rVert_{\mathcal{C}^{-\delta}},\\
II_{t, i_{1} i_{2} j j_{1}}^{2} \triangleq& \lVert \pi_{0} ( \mathcal{P}^{ji_{2}} \pi_{<} ( D_{j_{1}} [ \bar{u}_{3}^{\epsilon, i_{2}} + \bar{u}_{4}^{\epsilon, i_{2}}], \bar{K}_{b}^{\epsilon, j_{1}}), \bar{u}_{1}^{\epsilon, i_{1}}) \rVert_{\mathcal{C}^{-\delta}},\\
II_{t, i_{1} i_{2} j j_{1}}^{3} \triangleq& \lVert \pi_{0} (\mathcal{P}^{ji_{2}} \pi_{<} ( \bar{u}_{3}^{\epsilon, j_{1}} + \bar{u}_{4}^{\epsilon, j_{1}}, D_{j_{1}} \bar{K}_{b}^{\epsilon, i_{2}}), \bar{u}_{1}^{\epsilon, i_{1}}) \rVert_{\mathcal{C}^{-\delta}},\\
II_{t, i_{1} i_{2} j j_{1}}^{4} \triangleq&  \lVert \pi_{0} ( \mathcal{P}^{ji_{2}} \pi_{<} ( D_{j_{1}} [ \bar{u}_{3}^{\epsilon, j_{1}} + \bar{u}_{4}^{\epsilon, j_{1}}], \bar{K}_{b}^{\epsilon, i_{2}}), \bar{u}_{1}^{\epsilon, i_{1}}) \rVert_{\mathcal{C}^{-\delta}},\\
II_{t, i_{1} i_{2} j j_{1}}^{5} \triangleq& \lVert \pi_{0} (\mathcal{P}^{ji_{2}} \pi_{<} ( \bar{b}_{3}^{\epsilon, i_{2}} + \bar{b}_{4}^{\epsilon, i_{2}}, D_{j_{1}} \bar{K}_{u}^{\epsilon, j_{1}}), \bar{u}_{1}^{\epsilon, i_{1}}) \rVert_{\mathcal{C}^{-\delta}},\\
II_{t, i_{1} i_{2} j j_{1}}^{6} \triangleq&  \lVert \pi_{0} ( \mathcal{P}^{ji_{2}} \pi_{<} ( D_{j_{1}} [ \bar{b}_{3}^{\epsilon, i_{2}} + \bar{b}_{4}^{\epsilon, i_{2}}], \bar{K}_{u}^{\epsilon, j_{1}}), \bar{u}_{1}^{\epsilon, i_{1}}) \rVert_{\mathcal{C}^{-\delta}},\\
II_{t, i_{1} i_{2} j j_{1}}^{7} \triangleq&  \lVert \pi_{0} (\mathcal{P}^{ji_{2}} \pi_{<} ( \bar{b}_{3}^{\epsilon, j_{1}} + \bar{b}_{4}^{\epsilon, j_{1}}, D_{j_{1}} \bar{K}_{u}^{\epsilon, i_{2}}), \bar{u}_{1}^{\epsilon, i_{1}}) \rVert_{\mathcal{C}^{-\delta}},\\ 
II_{t, i_{1} i_{2} j j_{1}}^{8} \triangleq&  \lVert \pi_{0} ( \mathcal{P}^{ji_{2}} \pi_{<} ( D_{j_{1}} [ \bar{b}_{3}^{\epsilon, j_{1}} + \bar{b}_{4}^{\epsilon, j_{1}}], \bar{K}_{u}^{\epsilon, i_{2}}), \bar{u}_{1}^{\epsilon, i_{1}}) \rVert_{\mathcal{C}^{-\delta}},\\
II_{t, i_{1} j}^{9} \triangleq& \lVert \pi_{0} (\bar{b}^{\epsilon, \sharp, j}, \bar{u}_{1}^{\epsilon, i_{1}}) \rVert_{\mathcal{C}^{-\delta}}. 
\end{align}
\end{subequations}
We compute 
\begin{align}\label{estimate 66}
II_{t, i_{1} i_{2} j j_{1}}^{1} \lesssim& \sum_{l=1}^{3} II_{t, i_{1} i_{2} j j_{1}}^{1l} 
\end{align} 
where 
\begin{align*}
II_{t, i_{1}i_{2}jj_{1}}^{11} \triangleq&  \lVert \pi_{0} ( \mathcal{P}^{ji_{2}} \pi_{<} ( \bar{u}_{3}^{\epsilon, i_{2}} + \bar{u}_{4}^{\epsilon, i_{2}}, D_{j_{1}} \bar{K}_{b}^{\epsilon, j_{1}}), \bar{u}_{1}^{\epsilon, i_{1}}) \nonumber\\
& \hspace{10mm} - \pi_{0} (\pi_{<} (\bar{u}_{3}^{\epsilon, i_{2}} + \bar{u}_{4}^{\epsilon, i_{2}}, \mathcal{P}^{ji_{2}} D_{j_{1}} \bar{K}_{b}^{\epsilon, j_{1}}), \bar{u}_{1}^{\epsilon, i_{1}}) \rVert_{\mathcal{C}^{\frac{1}{2} - \frac{3\delta}{2} - \delta_{0}}},\\
II_{t, i_{1}i_{2}jj_{1}}^{12} \triangleq&  \lVert \pi_{0} (\pi_{<} (\bar{u}_{3}^{\epsilon, i_{2}} + \bar{u}_{4}^{\epsilon, i_{2}}, \mathcal{P}^{j i_{2}} D_{j_{1}} \bar{K}_{b}^{\epsilon, j_{1}}), \bar{u}_{1}^{\epsilon, i_{1}}) \nonumber\\
& \hspace{10mm} - (\bar{u}_{3}^{\epsilon, i_{2}} + \bar{u}_{4}^{\epsilon, i_{2}}) \pi_{0} (\mathcal{P}^{ji_{2}} D_{j_{1}} \bar{K}_{b}^{\epsilon, j_{1}}, \bar{u}_{1}^{\epsilon, i_{1}}) \rVert_{\mathcal{C}^{\frac{1}{2} - \frac{3\delta}{2} - \delta_{0}}},\\
II_{t, i_{1}i_{2}jj_{1}}^{13} \triangleq&  \lVert (\bar{u}_{3}^{\epsilon, i_{2}} + \bar{u}_{4}^{\epsilon, i_{2}}) \pi_{0} (\mathcal{P}^{ji_{2}} D_{j_{1}} \bar{K}_{b}^{\epsilon, j_{1}}, \bar{u}_{1}^{\epsilon, i_{1}}) \rVert_{\mathcal{C}^{-\delta}} 
\end{align*}
by \eqref{estimate 62} and \eqref{[Equation (3.2g)][ZZ17]}. First,  
\begin{align}\label{estimate 63}
II_{t, i_{1} i_{2} j j_{1}}^{11} \lesssim& \lVert \bar{u}_{3}^{\epsilon, i_{2}} + \bar{u}_{4}^{\epsilon, i_{2}} \rVert_{\mathcal{C}^{\frac{1}{2} - \delta_{0}}} \lVert D_{j_{1}} \bar{K}_{b}^{\epsilon, j_{1}} \rVert_{\mathcal{C}^{\frac{1}{2} - \delta}} \bar{C}_{W}^{\epsilon} (t) \nonumber\\
 \lesssim& (\bar{C}_{W}^{\epsilon}(t) + \lVert \bar{u}_{4}^{\epsilon, i_{2}} \rVert_{\mathcal{C}^{\frac{1}{2} - \delta_{0}}} ) ( \bar{C}_{W}^{\epsilon}(t))^{2} t^{\frac{\delta}{4}}
 \end{align} 
by Lemmas \ref{[Lemma 1.1][Y19a]} (3), \ref{Lemma 3.6} and \ref{Lemma 3.10}, \eqref{[Equation (3.2g)][ZZ17]}, \eqref{[Equation (3.2e)][ZZ17]} and \eqref{[Equation (3.2h)][ZZ17]}. Second, 
\begin{align}\label{estimate 64}
II_{t, i_{1} i_{2} j j_{1}}^{12}\lesssim (\bar{C}_{W}^{\epsilon}(t) + \lVert \bar{u}_{4}^{\epsilon, i_{2}} \rVert_{\mathcal{C}^{\frac{1}{2} - \delta_{0}}}) (\bar{C}_{W}^{\epsilon}(t))^{2} t^{\frac{\delta}{4}}  
\end{align}
by Lemmas \ref{Lemma 3.5}, \ref{Lemma 3.7} and \ref{Lemma 3.10}, \eqref{estimate 56}- \eqref{[Equation (3.2e)][ZZ17]} and \eqref{[Equation (3.2h)][ZZ17]}. Third, we estimate 
\begin{align}\label{estimate 65}
II_{t, i_{1} i_{2} j j_{1}}^{13} \lesssim (\bar{C}_{W}^{\epsilon}(t) + \lVert \bar{u}_{4}^{\epsilon, i_{2}} \rVert_{\mathcal{C}^{\frac{1}{2} - \delta_{0}}} ) \bar{C}_{W}^{\epsilon}(t) 
\end{align} 
by \eqref{regularity 0}, Lemma \ref{[Lemma 1.1][Y19a]} (4), \eqref{[Equation (3.2g)][ZZ17]} - \eqref{[Equation (3.2h)][ZZ17]}. Applying \eqref{estimate 63} - \eqref{estimate 65} to \eqref{estimate 66} gives 
\begin{align}\label{estimate 67}
II_{t, i_{1} i_{2} j j_{1}}^{1} \lesssim (1+  (\bar{C}_{W}^{\epsilon}(t))^{3})(1+ \lVert \bar{u}_{4}^{\epsilon, i_{2}} \rVert_{\mathcal{C}^{\frac{1}{2} - \delta_{0}}}).
\end{align} 
We can deduce analogous estimates for $II_{t, i_{1} i_{2} j j_{1}}^{k}$ for all $k \in \{2, \hdots, 8\}$. Finally, 
\begin{align}\label{estimate 68}
II_{t, i_{1} j}^{9} \lesssim \lVert \pi_{0}( \bar{b}^{\epsilon, \sharp, j}, \bar{u}_{1}^{\epsilon, i_{1}}) \rVert_{\mathcal{C}^{\beta - \frac{\delta}{2}}}  \lesssim& \lVert \bar{b}^{\epsilon, \sharp, j} \rVert_{\mathcal{C}^{\frac{1}{2} + \beta}} \lVert \bar{u}_{1}^{\epsilon, i_{1}} \rVert_{\mathcal{C}^{-\frac{1}{2} - \frac{\delta}{2}}} \nonumber\\
& \hspace{15mm} \lesssim \lVert \bar{b}^{\epsilon, \sharp} \rVert_{\mathcal{C}^{\frac{1}{2} + \beta}} \bar{C}_{W}^{\epsilon}(t)
\end{align} 
by \eqref{estimate 62}, \eqref{regularity 0}, \eqref{beta}, Lemma \ref{[Lemma 1.1][Y19a]} (3) and \eqref{[Equation (3.2e)][ZZ17]}. Therefore, \eqref{estimate 67} and \eqref{estimate 68} applied to \eqref{estimate 69} gives 
\begin{align}\label{estimate 70}
&  \lVert \pi_{0} ( -\frac{1}{2} \sum_{i_{2}, j_{1} =1}^{3} \mathcal{P}^{j i_{2}} D_{j_{1}} [- \pi_{<} ( \bar{u}_{3}^{\epsilon, i_{2}} + \bar{u}_{4}^{\epsilon, i_{2}}, \bar{K}_{b}^{\epsilon, j_{1}}) + \pi_{<} (\bar{u}_{3}^{\epsilon, j_{1}} + \bar{u}_{4}^{\epsilon, j_{1}}, \bar{K}_{b}^{\epsilon, i_{2}}) \nonumber\\
& \hspace{10mm} + \pi_{<} (\bar{b}_{3}^{\epsilon, i_{2}} + \bar{b}_{4}^{\epsilon, i_{2}}, \bar{K}_{u}^{\epsilon, j_{1}}) - \pi_{<} (\bar{b}_{3}^{\epsilon, j_{1}} + \bar{b}_{4}^{\epsilon, j_{1}}, \bar{K}_{u}^{\epsilon, i_{2}})] + \bar{b}^{\epsilon, \sharp, j}, \bar{u}_{1}^{\epsilon, i_{1}}) \rVert_{\mathcal{C}^{-\delta}} \nonumber\\
\lesssim& (1+ ( \bar{C}_{W}^{\epsilon})^{3}) (1+ \lVert \bar{y}_{4}^{\epsilon} \rVert_{\mathcal{C}^{\frac{1}{2} - \delta_{0}}} + \lVert \bar{b}^{\epsilon, \sharp} \rVert_{\mathcal{C}^{\frac{1}{2} + \beta}}).
\end{align}
Estimates which are similar to \eqref{estimate 70} applied to \eqref{estimate 61} gives
\begin{align}\label{[Equation (3.6a)][ZZ17]}
& \lVert ( \pi_{0,\diamond} (\bar{u}_{4}^{\epsilon, i_{1}}, \bar{u}_{1}^{\epsilon, j}), \pi_{0,\diamond} (\bar{u}_{4}^{\epsilon, j}, \bar{u}_{1}^{\epsilon, i_{1}}), \pi_{0,\diamond} (\bar{b}_{4}^{\epsilon, i_{1}}, \bar{b}_{1}^{\epsilon, j}), \pi_{0,\diamond} (\bar{b}_{4}^{\epsilon, j}, \bar{b}_{1}^{\epsilon, i_{1}}),\nonumber \\
& \pi_{0,\diamond} (\bar{b}_{4}^{\epsilon, i_{1}}, \bar{u}_{1}^{\epsilon, j}), \pi_{0,\diamond} (\bar{u}_{4}^{\epsilon, j}, \bar{b}_{1}^{\epsilon, i_{1}}), \pi_{0,\diamond} (\bar{u}_{4}^{\epsilon, i_{1}}, \bar{b}_{1}^{\epsilon, j}), \pi_{0,\diamond} (\bar{b}_{4}^{\epsilon, j}, \bar{u}_{1}^{\epsilon, i_{1}})) \rVert_{\mathcal{C}^{-\delta}} \nonumber\\ 
\lesssim&  (1+ ( \bar{C}_{W}^{\epsilon})^{3}) (1+ \lVert \bar{y}_{4}^{\epsilon} \rVert_{\mathcal{C}^{\frac{1}{2} - \delta_{0}}} + \lVert \bar{y}^{\epsilon, \sharp} \rVert_{\mathcal{C}^{\frac{1}{2} + \beta}}).
\end{align}  
At last, we are ready to conclude \eqref{[Equation (3.6)][ZZ17]} due to \eqref{overlinephiuepsilonsharpi},\eqref{overlinephibepsilonsharpi}, \eqref{estimate 71}-\eqref{estimate 75} and \eqref{[Equation (3.6a)][ZZ17]}.
\end{proof} 
The following computations differ slightly from those of \cite{ZZ17}. 
\begin{proposition}\label{Proposition on F}
Let $\delta_{0}$ and $z$ satisfy the hypothesis of Theorem \ref{Theorem 1.2} while $\delta, \kappa$ and $\beta$ satisfy \eqref{[Equation (3.2g)][ZZ17]}, \eqref{[Equation (3.3c)][ZZ17]} and \eqref{beta}, respectively. Additionally, let 
\begin{equation}\label{b1}
b_{1} \in ( \frac{1}{2} ( \frac{\delta}{2} + \beta + 2 \kappa), (1 - \delta - z - \kappa) \wedge \frac{1}{2} ( \frac{1}{2} - \frac{\delta}{2} - 2 \kappa)). 
\end{equation} 
Then $\bar{F}_{u}^{\epsilon, i}$ and $\bar{F}_{b}^{\epsilon, i}$ defined respectively by \eqref{overlineFuepsiloni}, \eqref{overlineFbepsiloni} satisfy the following estimates:
\begin{subequations}
\begin{align}
&\lVert (\bar{F}_{u}^{\epsilon}, \bar{F}_{b}^{\epsilon})(t) \rVert_{\mathcal{C}^{\frac{1}{2} + \beta}}\label{[Equation (3.9)][ZZ17]} \\
\lesssim& \bar{C}_{W}^{\epsilon} (t) [ t^{ - \frac{ \frac{\delta}{2} + \beta + z}{2} - \frac{3\kappa}{4}} \lVert y_{0} - \bar{y}_{1}^{\epsilon}(0) \rVert_{\mathcal{C}^{-z}} \nonumber \\
&+ \int_{0}^{t} (t-r)^{-\kappa - \frac{3}{4} - \frac{\delta}{2} - \frac{\beta}{2}} \nonumber \\
& \hspace{3mm} \times [(1+ ( \bar{C}_{W}^{\epsilon}(t))^{3}) (1+ \lVert \bar{y}^{\epsilon, \sharp} (r) \rVert_{\mathcal{C}^{\frac{1}{2} + \beta}} + \lVert \bar{y}_{4}^{\epsilon}(r) \rVert_{\mathcal{C}^{\frac{1}{2} - \delta_{0}}} ) + \lVert \bar{y}_{4}^{\epsilon}(r) \rVert_{\mathcal{C}^{\delta}}^{2} ] dr \nonumber\\
&+ t^{- [\frac{ \frac{\delta}{2} + \beta + \kappa}{2} - b_{1}]} ( \int_{0}^{t} (t-r)^{- \frac{ \frac{3}{2} + \frac{\delta}{2} + \kappa}{2(1-b_{1})}} \nonumber\\
& \hspace{3mm} \times [(1+ (\bar{C}_{W}^{\epsilon}(t))^{3} )(1+ \lVert \bar{y}^{\epsilon, \sharp}(r) \rVert_{\mathcal{C}^{\frac{1}{2} + \beta}} + \lVert \bar{y}_{4}^{\epsilon}(r) \rVert_{\mathcal{C}^{\frac{1}{2} - \delta_{0}}}) + \lVert \bar{y}_{4}^{\epsilon}(r) \rVert_{\mathcal{C}^{\delta}}^{2} ]^{\frac{1}{1-b_{1}}} dr )^{1- b_{1}} \nonumber \\
&+ t^{\frac{1}{4} - \frac{ \delta_{0} + \beta + \kappa + \frac{\delta}{2}}{2}} \lVert (\bar{y}_{3}^{\epsilon}  + \bar{y}_{4}^{\epsilon})(t) \rVert_{\mathcal{C}^{\frac{1}{2} - \delta_{0}}}], \nonumber\\
& \lVert (\bar{F}_{u}^{\epsilon}, \bar{F}_{b}^{\epsilon}) (t) \rVert_{\mathcal{C}^{\delta}}\label{[Equation (3.10)][ZZ17]} \\
\lesssim& \bar{C}_{W}^{\epsilon}(t) [ t^{\frac{1}{4} - \frac{z}{2} - \frac{3\delta}{4}} \lVert y_{0} - \bar{y}_{1}^{\epsilon}(0) \rVert_{\mathcal{C}^{-z}} \nonumber\\
&+ \int_{0}^{t} (t-r)^{- \frac{1+ 2 \delta + \kappa}{2}} \nonumber\\
& \hspace{3mm} \times [(1+ (\bar{C}_{W}^{\epsilon}(t))^{3})(1+ \lVert \bar{y}^{\epsilon, \sharp}(r) \rVert_{\mathcal{C}^{\frac{1}{2} + \beta}} + \lVert \bar{y}_{4}^{\epsilon}(r) \rVert_{\mathcal{C}^{\frac{1}{2} - \delta_{0}}} ) + \lVert \bar{y}_{4}^{\epsilon}(r) \rVert_{\mathcal{C}^{\delta}}^{2}] dr \nonumber\\
&\hspace{3mm} + t^{\frac{\kappa}{2}} \int_{0}^{t} (t-r)^{- \frac{1+ 2 \delta+ \frac{3\kappa}{2}}{2}} \nonumber\\
& \hspace{3mm} \times [(1+ (\bar{C}_{W}^{\epsilon}(t))^{3})(1+ \lVert \bar{y}^{\epsilon, \sharp}(r) \rVert_{\mathcal{C}^{\frac{1}{2} + \beta}} + \lVert \bar{y}_{4}^{\epsilon}(r) \rVert_{\mathcal{C}^{\frac{1}{2} - \delta_{0}}}) + \lVert \bar{y}_{4}^{\epsilon}(r) \rVert_{\mathcal{C}^{\delta}}^{2}] dr \nonumber\\
&+ t^{\frac{1-\delta}{4}} \lVert (\bar{y}_{3}^{\epsilon} + \bar{y}_{4}^{\epsilon})(t) \rVert_{\mathcal{C}^{\delta}}], \nonumber \\
& \lVert (\bar{F}_{u}^{\epsilon},\bar{F}_{b}^{\epsilon})(t) \rVert_{\mathcal{C}^{\frac{1}{2} - \delta_{0}}} \label{[Equation (3.17b)][ZZ17]} \\
\lesssim& \bar{C}_{W}^{\epsilon} (t) [ t^{\frac{ \delta_{0} - z - \frac{\delta}{2}}{2}} \lVert y_{0} - \bar{y}_{1}^{\epsilon}(0) \rVert_{\mathcal{C}^{-z}}  \nonumber\\
&\hspace{3mm} + \int_{0}^{t} (t-r)^{ \frac{\delta_{0}}{2} - \frac{\delta}{2} - \frac{\kappa}{4} -\frac{3}{4}}  \nonumber \\
& \hspace{15mm} \times [ (1+ ( \bar{C}_{W}^{\epsilon}(t))^{3})(1+ \lVert y^{\epsilon, \sharp}(r)\rVert_{\mathcal{C}^{\frac{1}{2} + \beta}} + \lVert \bar{y}_{4}^{\epsilon}(r) \rVert_{\mathcal{C}^{\frac{1}{2} - \delta_{0}}} ) + \lVert \bar{y}_{4}^{\epsilon}(r)\rVert_{\mathcal{C}^{\delta}}^{2}]  dr  \nonumber \\
&\hspace{3mm} + t^{\frac{\kappa}{2}}  \int_{0}^{t} (t-r)^{\frac{\delta_{0}}{2} - \frac{\delta}{2} - \frac{3\kappa}{4} - \frac{3}{4}}   \nonumber \\
& \hspace{15mm} \times [ (1+ ( \bar{C}_{W}^{\epsilon}(t))^{3})(1+ \lVert y^{\epsilon, \sharp}(r) \rVert_{\mathcal{C}^{\frac{1}{2} + \beta}} + \lVert \bar{y}_{4}^{\epsilon}(r) \rVert_{\mathcal{C}^{\frac{1}{2} - \delta_{0}}} ) + \lVert \bar{y}_{4}^{\epsilon}(r)\rVert_{\mathcal{C}^{\delta}}^{2}]  dr  \nonumber \\
&\hspace{3mm} + t^{\frac{1-\delta}{4}} \lVert (\bar{y}_{3}^{\epsilon} + \bar{y}_{4}^{\epsilon})(t) \rVert_{\mathcal{C}^{\frac{1}{2} - \delta_{0}}}], \nonumber\\
& \lVert (\bar{F}_{u}^{\epsilon},\bar{F}_{b}^{\epsilon})(t) \rVert_{\mathcal{C}^{-z}} 
\lesssim \bar{C}_{W}^{\epsilon} (t) [ \int_{0}^{t} (t-r)^{- \frac{ \frac{3}{2} - z + \frac{\delta}{2} + \kappa}{2}} ( \bar{C}_{W}^{\epsilon}(t) + \lVert \bar{y}_{4}^{\epsilon}(r) \rVert_{\mathcal{C}^{\delta}}) dr \label{[Equation (3.17e)][ZZ17]} \\
& \hspace{27mm} + t^{\frac{1}{4} + \frac{z}{2} - \frac{\delta}{4} - \frac{\kappa}{2}} (\bar{C}_{W}^{\epsilon}(t) + \lVert \bar{y}_{4}^{\epsilon} (t) \rVert_{\mathcal{C}^{\delta}}) + t^{\frac{1-\delta}{4}}\lVert (\bar{y}_{3}^{\epsilon} + \bar{y}_{4}^{\epsilon})(t) \rVert_{\mathcal{C}^{-z}}]. \nonumber 
\end{align}  
\end{subequations}
\end{proposition} 

\begin{proof}[Proof of Proposition \ref{Proposition on F}]
For $\lambda = \frac{1}{2} + \beta, \delta, \frac{1}{2} - \delta_{0}$ or $-z$, we can estimate 
\begin{equation}\label{estimate 109}
\lVert (\bar{F}_{u}^{\epsilon}, \bar{F}_{b}^{\epsilon})(t) \rVert_{\mathcal{C}^{\lambda}} \lesssim \sum_{i =1}^{3} \lVert \bar{F}_{u}^{\epsilon, i}(t) \rVert_{\mathcal{C}^{\lambda}} + \lVert \bar{F}_{b}^{\epsilon, i} (t)\rVert_{\mathcal{C}^{\lambda}} 
\end{equation} 
where 
\begin{align}\label{estimate 110}
\lVert \bar{F}_{u}^{\epsilon, i}(t) \rVert_{\mathcal{C}^{\lambda}}  =& \lVert \sum_{i_{1}, j_{1} =1}^{3} - \frac{1}{2} \int_{0}^{t} P_{t-s} \mathcal{P}^{ii_{1}} D_{j_{1}} \pi_{<} ( \bar{u}_{3}^{\epsilon, i_{1}}(s) + \bar{u}_{4}^{\epsilon, i_{1}}(s), \bar{u}_{1}^{\epsilon, j_{1}}(s)) \nonumber \\
& - P_{t-s}\mathcal{P}^{ii_{1}} D_{j_{1}} \pi_{<} ( \bar{u}_{3}^{\epsilon, i_{1}}(t) + \bar{u}_{4}^{\epsilon, i_{1}}(t), \bar{u}_{1}^{\epsilon, j_{1}}(s)) ds \nonumber \\
& - \frac{1}{2} \int_{0}^{t} P_{t-s} \mathcal{P}^{ii_{1}} D_{j_{1}} \pi_{<} (\bar{u}_{3}^{\epsilon, i_{1}}(t) + \bar{u}_{4}^{\epsilon, i_{1}}(t), \bar{u}_{1}^{\epsilon, j_{1}}(s)) ds \nonumber \\
&+ \frac{1}{2} \mathcal{P}^{ii_{1}} D_{j_{1}} \pi_{<} ( \bar{u}_{3}^{\epsilon, i_{1}} (t) + \bar{u}_{4}^{\epsilon, i_{1}}(t), \bar{K}_{u}^{\epsilon, j_{1}}(t)) \nonumber\\
&- \frac{1}{2} \int_{0}^{t} P_{t-s} \mathcal{P}^{ii_{1}} D_{j_{1}} \pi_{<} (\bar{u}_{3}^{\epsilon, j_{1}}(s) + \bar{u}_{4}^{\epsilon, j_{1}}(s), \bar{u}_{1}^{\epsilon, i_{1}}(s)) \nonumber\\
& - P_{t-s} \mathcal{P}^{ii_{1}} D_{j_{1}} \pi_{<}(\bar{u}_{3}^{\epsilon, j_{1}}(t) + \bar{u}_{4}^{\epsilon, j_{1}}(t), \bar{u}_{1}^{\epsilon, i_{1}}(s)) ds \nonumber\\
& - \frac{1}{2} \int_{0}^{t} P_{t-s} \mathcal{P}^{ii_{1}} D_{j_{1}} \pi_{<} (\bar{u}_{3}^{\epsilon, j_{1}}(t) + \bar{u}_{4}^{\epsilon, j_{1}}(t), \bar{u}_{1}^{\epsilon, i_{1}}(s)) ds \nonumber\\
&+ \frac{1}{2} \mathcal{P}^{ii_{1}} D_{j_{1}} \pi_{<} ( \bar{u}_{3}^{\epsilon, j_{1}} (t) + \bar{u}_{4}^{\epsilon, j_{1}}(t), \bar{K}_{u}^{\epsilon, i_{1}}(t)) \nonumber\\
& + \frac{1}{2} \int_{0}^{t} P_{t-s} \mathcal{P}^{ii_{1}} D_{j_{1}} \pi_{<} ( \bar{b}_{3}^{\epsilon, i_{1}}(s) + \bar{b}_{4}^{\epsilon, i_{1}}(s), \bar{b}_{1}^{\epsilon, j_{1}}(s)) \nonumber \\
& - P_{t-s}\mathcal{P}^{ii_{1}} D_{j_{1}} \pi_{<} ( \bar{b}_{3}^{\epsilon, i_{1}}(t) + \bar{b}_{4}^{\epsilon, i_{1}}(t), \bar{b}_{1}^{\epsilon, j_{1}}(s)) ds \nonumber \\
& + \frac{1}{2} \int_{0}^{t} P_{t-s} \mathcal{P}^{ii_{1}} D_{j_{1}} \pi_{<} (\bar{b}_{3}^{\epsilon, i_{1}}(t) + \bar{b}_{4}^{\epsilon, i_{1}}(t), \bar{b}_{1}^{\epsilon, j_{1}}(s)) ds \nonumber \\
&- \frac{1}{2} \mathcal{P}^{ii_{1}} D_{j_{1}} \pi_{<} ( \bar{b}_{3}^{\epsilon, i_{1}} (t) + \bar{b}_{4}^{\epsilon, i_{1}}(t), \bar{K}_{b}^{\epsilon, j_{1}}(t)) \nonumber\\
&+ \frac{1}{2} \int_{0}^{t} P_{t-s} \mathcal{P}^{ii_{1}} D_{j_{1}} \pi_{<} (\bar{b}_{3}^{\epsilon, j_{1}}(s) + \bar{b}_{4}^{\epsilon, j_{1}}(s), \bar{b}_{1}^{\epsilon, i_{1}}(s)) \nonumber\\
& - P_{t-s} \mathcal{P}^{ii_{1}} D_{j_{1}} \pi_{<}(\bar{b}_{3}^{\epsilon, j_{1}}(t) + \bar{b}_{4}^{\epsilon, j_{1}}(t), \bar{b}_{1}^{\epsilon, i_{1}}(s)) ds \nonumber\\
& + \frac{1}{2} \int_{0}^{t} P_{t-s} \mathcal{P}^{ii_{1}} D_{j_{1}} \pi_{<} (\bar{b}_{3}^{\epsilon, j_{1}}(t) + \bar{b}_{4}^{\epsilon, j_{1}}(t), \bar{b}_{1}^{\epsilon, i_{1}}(s)) ds \nonumber\\
&- \frac{1}{2} \mathcal{P}^{ii_{1}} D_{j_{1}} \pi_{<} ( \bar{b}_{3}^{\epsilon, j_{1}} (t) + \bar{b}_{4}^{\epsilon, j_{1}}(t), \bar{K}_{b}^{\epsilon, i_{1}}(t)) \rVert_{\mathcal{C}^{\lambda}} 
\end{align} 
and 
\begin{align}\label{estimate 76}
\lVert \bar{F}_{b}^{\epsilon, i}(t) \rVert_{\mathcal{C}^{\lambda}} =& \lVert \sum_{i_{1}, j_{1} =1}^{3} - \frac{1}{2} \int_{0}^{t} P_{t-s} \mathcal{P}^{ii_{1}} D_{j_{1}} \pi_{<} ( \bar{b}_{3}^{\epsilon, i_{1}}(s) + \bar{b}_{4}^{\epsilon, i_{1}}(s), \bar{u}_{1}^{\epsilon, j_{1}}(s)) \nonumber \\
& - P_{t-s}\mathcal{P}^{ii_{1}} D_{j_{1}} \pi_{<} ( \bar{b}_{3}^{\epsilon, i_{1}}(t) + \bar{b}_{4}^{\epsilon, i_{1}}(t), \bar{u}_{1}^{\epsilon, j_{1}}(s)) ds \nonumber \\
& - \frac{1}{2} \int_{0}^{t} P_{t-s} \mathcal{P}^{ii_{1}} D_{j_{1}} \pi_{<} (\bar{b}_{3}^{\epsilon, i_{1}}(t) + \bar{b}_{4}^{\epsilon, i_{1}}(t), \bar{u}_{1}^{\epsilon, j_{1}}(s)) ds \nonumber \\
&+ \frac{1}{2} \mathcal{P}^{ii_{1}} D_{j_{1}} \pi_{<} ( \bar{b}_{3}^{\epsilon, i_{1}} (t) + \bar{b}_{4}^{\epsilon, i_{1}}(t), \bar{K}_{u}^{\epsilon, j_{1}}(t)) \nonumber\\
&- \frac{1}{2} \int_{0}^{t} P_{t-s} \mathcal{P}^{ii_{1}} D_{j_{1}} \pi_{<} (\bar{u}_{3}^{\epsilon, j_{1}}(s) + \bar{u}_{4}^{\epsilon, j_{1}}(s), \bar{b}_{1}^{\epsilon, i_{1}}(s)) \nonumber\\
& - P_{t-s} \mathcal{P}^{ii_{1}} D_{j_{1}} \pi_{<}(\bar{u}_{3}^{\epsilon, j_{1}}(t) + \bar{u}_{4}^{\epsilon, j_{1}}(t), \bar{b}_{1}^{\epsilon, i_{1}}(s)) ds \nonumber\\
& - \frac{1}{2} \int_{0}^{t} P_{t-s} \mathcal{P}^{ii_{1}} D_{j_{1}} \pi_{<} (\bar{u}_{3}^{\epsilon, j_{1}}(t) + \bar{u}_{4}^{\epsilon, j_{1}}(t), \bar{b}_{1}^{\epsilon, i_{1}}(s)) ds \nonumber\\
&+ \frac{1}{2} \mathcal{P}^{ii_{1}} D_{j_{1}} \pi_{<} ( \bar{u}_{3}^{\epsilon, j_{1}} (t) + \bar{u}_{4}^{\epsilon, j_{1}}(t), \bar{K}_{b}^{\epsilon, i_{1}}(t)) \nonumber\\
& + \frac{1}{2} \int_{0}^{t} P_{t-s} \mathcal{P}^{ii_{1}} D_{j_{1}} \pi_{<} ( \bar{u}_{3}^{\epsilon, i_{1}}(s) + \bar{u}_{4}^{\epsilon, i_{1}}(s), \bar{b}_{1}^{\epsilon, j_{1}}(s)) \nonumber \\
& - P_{t-s}\mathcal{P}^{ii_{1}} D_{j_{1}} \pi_{<} ( \bar{u}_{3}^{\epsilon, i_{1}}(t) + \bar{u}_{4}^{\epsilon, i_{1}}(t), \bar{b}_{1}^{\epsilon, j_{1}}(s)) ds \nonumber \\
& + \frac{1}{2} \int_{0}^{t} P_{t-s} \mathcal{P}^{ii_{1}} D_{j_{1}} \pi_{<} (\bar{u}_{3}^{\epsilon, i_{1}}(t) + \bar{u}_{4}^{\epsilon, i_{1}}(t), \bar{b}_{1}^{\epsilon, j_{1}}(s)) ds \nonumber \\
&- \frac{1}{2} \mathcal{P}^{ii_{1}} D_{j_{1}} \pi_{<} ( \bar{u}_{3}^{\epsilon, i_{1}} (t) + \bar{u}_{4}^{\epsilon, i_{1}}(t), \bar{K}_{b}^{\epsilon, j_{1}}(t)) \nonumber\\
&+ \frac{1}{2} \int_{0}^{t} P_{t-s} \mathcal{P}^{ii_{1}} D_{j_{1}} \pi_{<} (\bar{b}_{3}^{\epsilon, j_{1}}(s) + \bar{b}_{4}^{\epsilon, j_{1}}(s), \bar{u}_{1}^{\epsilon, i_{1}}(s)) \nonumber\\
& - P_{t-s} \mathcal{P}^{ii_{1}} D_{j_{1}} \pi_{<}(\bar{b}_{3}^{\epsilon, j_{1}}(t) + \bar{b}_{4}^{\epsilon, j_{1}}(t), \bar{u}_{1}^{\epsilon, i_{1}}(s)) ds \nonumber\\
& + \frac{1}{2} \int_{0}^{t} P_{t-s} \mathcal{P}^{ii_{1}} D_{j_{1}} \pi_{<} (\bar{b}_{3}^{\epsilon, j_{1}}(t) + \bar{b}_{4}^{\epsilon, j_{1}}(t), \bar{u}_{1}^{\epsilon, i_{1}}(s)) ds \nonumber\\
&- \frac{1}{2} \mathcal{P}^{ii_{1}} D_{j_{1}} \pi_{<} ( \bar{b}_{3}^{\epsilon, j_{1}} (t) + \bar{b}_{4}^{\epsilon, j_{1}}(t), \bar{K}_{u}^{\epsilon, i_{1}}(t)) \rVert_{\mathcal{C}^{\lambda}} 
\end{align} 
due to \eqref{overlineFuepsiloni} and \eqref{overlineFbepsiloni}. 
\begin{remark}\label{Remark 2.1} 
Within $\bar{F}_{b}^{\epsilon, i}$ from \eqref{overlineFbepsiloni}, there are eight terms. Among such eight terms, we strategically paired up the first and seventh, second with sixth, third with fifth and fourth with eighth terms in \eqref{estimate 76}. This is similar to the necessity of taking advantage of the structure of the nonlinear terms of the MHD system that we mentioned in Remark \ref{commutator}. 
\end{remark}
W.l.o.g., we focus only on $\lVert \bar{F}_{b}^{\epsilon} \rVert_{\mathcal{C}^{\lambda}}$ as the estimates for $\lVert \bar{F}_{u}^{\epsilon} \rVert_{\mathcal{C}^{\lambda}}$ are similar. We use \eqref{estimate 56} and compute 
\begin{equation}\label{estimate 111}
\lVert \bar{F}_{b}^{\epsilon}(t) \rVert_{\mathcal{C}^{\lambda}} \lesssim \sum_{k=1}^{8} III_{t}^{k}
\end{equation}
where 
\begin{subequations} 
\begin{align} 
III_{t}^{1} \triangleq & \sum_{i, i_{1}, j_{1} =1}^{3} \lVert \int_{0}^{t} P_{t-s} \mathcal{P}^{ii_{1}} D_{j_{1}} \pi_{<} ( \bar{b}_{3}^{\epsilon, i_{1}}(s) + \bar{b}_{4}^{\epsilon, i_{1}}(s) \nonumber\\
& \hspace{20mm}  - (\bar{b}_{3}^{\epsilon, i_{1}}(t) + \bar{b}_{4}^{\epsilon, i_{1}}(t)), \bar{u}_{1}^{\epsilon, j_{1}}(s)) ds \rVert_{\mathcal{C}^{\lambda}},  \label{III1}\\
III_{t}^{2} \triangleq & \sum_{i, i_{1}, j_{1} =1}^{3} \lVert \int_{0}^{t} P_{t-s} \mathcal{P}^{ii_{1}} D_{j_{1}} \pi_{<}(\bar{b}_{3}^{\epsilon, i_{1}}(t) + \bar{b}_{4}^{\epsilon, i_{1}}(t), \bar{u}_{1}^{\epsilon, j_{1}}(s)) ds \nonumber\\
& \hspace{20mm} - \mathcal{P}^{ii_{1}}D_{j_{1}} \pi_{<} ( \bar{b}_{3}^{\epsilon, i_{1}}(t)+ \bar{b}_{4}^{\epsilon, i_{1}}(t), \int_{0}^{t} P_{t-s} \bar{u}_{1}^{\epsilon, j_{1}}(s) ds) \rVert_{\mathcal{C}^{\lambda}}, \label{III2}\\
III_{t}^{3} \triangleq&  \sum_{i, i_{1}, j_{1} =1}^{3} \lVert \int_{0}^{t} P_{t-s} \mathcal{P}^{ii_{1}} D_{j_{1}} \pi_{<} (\bar{u}_{3}^{\epsilon, j_{1}}(s) + \bar{u}_{4}^{\epsilon, j_{1}}(s) \nonumber\\
& \hspace{20mm} - (\bar{u}_{3}^{\epsilon, j_{1}}(t) + \bar{u}_{4}^{\epsilon, j_{1}}(t)), \bar{b}_{1}^{\epsilon, i_{1}}(s)) ds \rVert_{\mathcal{C}^{\lambda}}, \label{III3}\\
III_{t}^{4} \triangleq&  \sum_{i, i_{1}, j_{1} = 1}^{3} \lVert \int_{0}^{t} P_{t-s} \mathcal{P}^{ii_{1}} D_{j_{1}} \pi_{<} ( \bar{u}_{3}^{\epsilon, j_{1}}(t) + \bar{u}_{4}^{\epsilon, j_{1}}(t), \bar{b}_{1}^{\epsilon, i_{1}}(s)) ds \nonumber\\
& \hspace{20mm} - \mathcal{P}^{ii_{1}} D_{j_{1}} \pi_{<} (\bar{u}_{3}^{\epsilon, j_{1}}(t) + \bar{u}_{4}^{\epsilon, j_{1}}(t), \int_{0}^{t} P_{t-s} \bar{b}_{1}^{\epsilon, i_{1}}(s) ds ) \rVert_{\mathcal{C}^{\lambda}}, \label{III4}\\
III_{t}^{5} \triangleq& \sum_{i, i_{1}, j_{1} =1}^{3} \lVert \int_{0}^{t} P_{t-s} \mathcal{P}^{ii_{1}} D_{j_{1}} \pi_{<} ( \bar{u}_{3}^{\epsilon, i_{1}}(s) + \bar{u}_{4}^{\epsilon, i_{1}}(s) \nonumber\\
& \hspace{20mm} - (\bar{u}_{3}^{\epsilon, i_{1}}(t) + \bar{u}_{4}^{\epsilon, i_{1}}(t)), \bar{b}_{1}^{\epsilon, j_{1}}(s)) ds \rVert_{\mathcal{C}^{\lambda}},  \label{III5}\\
III_{t}^{6} \triangleq&  \sum_{i, i_{1}, j_{1} =1}^{3} \lVert \int_{0}^{t} P_{t-s} \mathcal{P}^{ii_{1}} D_{j_{1}} \pi_{<}(\bar{u}_{3}^{\epsilon, i_{1}}(t) + \bar{u}_{4}^{\epsilon, i_{1}}(t), \bar{b}_{1}^{\epsilon, j_{1}}(s)) ds \nonumber\\
& \hspace{20mm} - \mathcal{P}^{ii_{1}}D_{j_{1}} \pi_{<} ( \bar{u}_{3}^{\epsilon, i_{1}}(t)+ \bar{u}_{4}^{\epsilon, i_{1}}(t), \int_{0}^{t} P_{t-s} \bar{b}_{1}^{\epsilon, j_{1}}(s) ds) \rVert_{\mathcal{C}^{\lambda}},  \label{III6}\\
III_{t}^{7} \triangleq& \sum_{i, i_{1}, j_{1} =1}^{3} \lVert \int_{0}^{t} P_{t-s} \mathcal{P}^{ii_{1}} D_{j_{1}} \pi_{<} (\bar{b}_{3}^{\epsilon, j_{1}}(s) + \bar{b}_{4}^{\epsilon, j_{1}}(s) \nonumber\\
& \hspace{20mm} - (\bar{b}_{3}^{\epsilon, j_{1}}(t) + \bar{b}_{4}^{\epsilon, j_{1}}(t)), \bar{u}_{1}^{\epsilon, i_{1}}(s)) ds \rVert_{\mathcal{C}^{\lambda}},  \label{III7}\\
III_{t}^{8} \triangleq&  \sum_{i, i_{1}, j_{1} = 1}^{3} \lVert \int_{0}^{t} P_{t-s} \mathcal{P}^{ii_{1}} D_{j_{1}} \pi_{<} ( \bar{b}_{3}^{\epsilon, j_{1}}(t) + \bar{b}_{4}^{\epsilon, j_{1}}(t), \bar{u}_{1}^{\epsilon, i_{1}}(s)) ds \nonumber\\
& \hspace{20mm} - \mathcal{P}^{ii_{1}} D_{j_{1}} \pi_{<} (\bar{b}_{3}^{\epsilon, j_{1}}(t) + \bar{b}_{4}^{\epsilon, j_{1}}(t), \int_{0}^{t} P_{t-s} \bar{u}_{1}^{\epsilon, i_{1}}(s) ds ) \rVert_{\mathcal{C}^{\lambda}}.\label{III8} 
\end{align}
\end{subequations} 
W.l.o.g. we show the estimates only for $III_{t}^{1}$ and $III_{t}^{2}$ as those for others are similar. In order to prove \eqref{[Equation (3.9)][ZZ17]}, we now consider $\lambda = \frac{1}{2} + \beta$. First, 
\begin{align}\label{[Equation (3.7a)][ZZ17]}
III_{t}^{1} \lesssim  \bar{C}_{W}^{\epsilon}(t)\int_{0}^{t} (t-s)^{-1 - \frac{ \frac{\delta}{2} + \beta + \kappa}{2}} \lVert \bar{b}_{3}^{\epsilon}(s) + \bar{b}_{4}^{\epsilon}(s) - (\bar{b}_{3}^{\epsilon}(t) + \bar{b}_{4}^{\epsilon}(t)) \rVert_{\mathcal{C}^{\frac{\kappa}{2}}} ds 
\end{align} 
by \eqref{III1}, Lemmas \ref{Lemma 3.7}, \ref{Lemma 3.10} and \ref{[Lemma 1.1][Y19a]} (2), \eqref{regularity 0} and \eqref{[Equation (3.2e)][ZZ17]}. For $s \in (0,t)$, $i \in \{1,2,3\}$, 
\begin{equation}\label{b and b0}
b \in (\frac{\delta}{2} + \beta + 2 \kappa, \frac{1}{2} - 2 \kappa - \frac{\delta}{2}), \hspace{3mm} b_{0} \in (\frac{\delta}{4} + \frac{\beta}{2} + \kappa, 1 - \frac{z}{2} - \kappa), 
\end{equation}
we define $I$ to be an identity operator and $\bar{G}^{\epsilon} \triangleq (\bar{G}^{\epsilon, 1}, \bar{G}^{\epsilon, 2}, \bar{G}^{\epsilon, 3})$ where  
\begin{align}\label{[Equation (3.7c)][ZZ17]}
\bar{G}^{\epsilon, i}(t) &\triangleq  \sum_{i_{1}, j=1}^{3} \mathcal{P}^{ii_{1}} D_{j}[\bar{b}_{1}^{\epsilon, i_{1}} \diamond \bar{u}_{2}^{\epsilon, j} + \bar{b}_{2}^{\epsilon, i_{1}} \diamond \bar{u}_{1}^{\epsilon, j} - \bar{u}_{1}^{\epsilon, i_{1}} \diamond \bar{b}_{2}^{\epsilon, j} - \bar{u}_{2}^{\epsilon, i_{1}} \diamond \bar{b}_{1}^{\epsilon, j} \\
& \hspace{8mm} +  \bar{b}_{1}^{\epsilon, i_{1}} \diamond (\bar{u}_{3}^{\epsilon, j} + \bar{u}_{4}^{\epsilon, j} ) + (\bar{b}_{3}^{\epsilon, i_{1}} + \bar{b}_{4}^{\epsilon, i_{1}}) \diamond \bar{u}_{1}^{\epsilon, j} + \bar{b}_{2}^{\epsilon, i_{1}} \diamond\bar{u}_{2}^{\epsilon, j} \nonumber\\
& \hspace{8mm} + \bar{b}_{2}^{\epsilon, i_{1}} (\bar{u}_{3}^{\epsilon, j} + \bar{u}_{4}^{\epsilon, j}) + \bar{u}_{2}^{\epsilon, j} (\bar{b}_{3}^{\epsilon, i_{1}} + \bar{b}_{4}^{\epsilon, i_{1}})+ (\bar{b}_{3}^{\epsilon, i_{1}} + \bar{b}_{4}^{\epsilon, i_{1}}) (\bar{u}_{3}^{\epsilon, j} + \bar{u}_{4}^{\epsilon, j}) \nonumber \\
&  \hspace{8mm} -\bar{u}_{1}^{\epsilon, i_{1}} \diamond (\bar{b}_{3}^{\epsilon, j} + \bar{b}_{4}^{\epsilon, j} ) - (\bar{u}_{3}^{\epsilon, i_{1}} + \bar{u}_{4}^{\epsilon, i_{1}}) \diamond \bar{b}_{1}^{\epsilon, j} - \bar{u}_{2}^{\epsilon, i_{1}} \diamond\bar{b}_{2}^{\epsilon, j} \nonumber \\
& \hspace{8mm} - \bar{u}_{2}^{\epsilon, i_{1}} (\bar{b}_{3}^{\epsilon, j} + \bar{b}_{4}^{\epsilon, j}) - \bar{b}_{2}^{\epsilon, j} (\bar{u}_{3}^{\epsilon, i_{1}} + \bar{u}_{4}^{\epsilon, i_{1}})- (\bar{u}_{3}^{\epsilon, i_{1}} + \bar{u}_{4}^{\epsilon, i_{1}}) (\bar{b}_{3}^{\epsilon, j} + \bar{b}_{4}^{\epsilon, j})](t)\nonumber 
\end{align} 
and compute within the right hand side of \eqref{[Equation (3.7a)][ZZ17]}, for $i\in \{1,2,3\}$, 
\begin{align}\label{[Equation (3.7b)][ZZ17]}
& \lVert \bar{b}_{3}^{\epsilon, i}(t) + \bar{b}_{4}^{\epsilon, i}(t) - \bar{b}_{3}^{\epsilon, i}(s) - \bar{b}_{4}^{\epsilon, i}(s) \rVert_{\mathcal{C}^{\frac{\kappa}{2}}} \nonumber\\
\lesssim& \lVert (P_{t-s} - I) P_{s} (b_{0} - \bar{b}_{1}^{\epsilon} (0)) \rVert_{\mathcal{C}^{\frac{\kappa}{2}}} \nonumber\\
& + \lVert \int_{0}^{s} (P_{t-r} - P_{s-r} ) \bar{G}^{\epsilon}(r) dr \rVert_{\mathcal{C}^{\frac{\kappa}{2}}} + \lVert \int_{s}^{t} P_{t-r} \bar{G}^{\epsilon}(r) dr \rVert_{\mathcal{C}^{\frac{\kappa}{2}}} \nonumber \\
\lesssim& (t-s)^{b_{0}} s^{- \frac{ z + \frac{\kappa}{2} + 2b_{0}}{2}} \lVert b_{0} - \bar{b}_{1}^{\epsilon} (0) \rVert_{\mathcal{C}^{-z}} \nonumber\\
& + (t-s)^{\frac{b}{2}} \int_{0}^{s} (s-r)^{- \frac{ \kappa + b + \frac{3}{2} + \frac{\delta}{2}}{2}} \lVert \bar{G}^{\epsilon}(r) \rVert_{\mathcal{C}^{- \frac{3}{2} - \frac{\delta}{2} - \frac{\kappa}{2}}} dr \nonumber\\
&+ (t-s)^{b_{1}} \left( \int_{s}^{t} [ (t-r)^{- \frac{ \kappa + \frac{3}{2} + \frac{\delta}{2}}{2}} \lVert \bar{G}^{\epsilon}(r) \rVert_{\mathcal{C}^{-\frac{3}{2} - \frac{\delta}{2} - \frac{\kappa}{2}}}]^{\frac{1}{1-b_{1}}} dr\right)^{1-b_{1}}  
\end{align} 
by \eqref{[Equation (3.1ac)][ZZ17]}, \eqref{[Equation (3.2)][ZZ17]},  that $P_{t} - P_{s} = (P_{t-s} - I)P_{s}$, Lemmas \ref{Lemma 3.12} and \ref{Lemma 3.10}, \eqref{b and b0} and H$\ddot{\mathrm{o}}$lder's inequality.  Now, we first estimate within $\bar{G}^{\epsilon, i}(r)$ of \eqref{[Equation (3.7c)][ZZ17]}, for $i\in \{1,2,3\}$,  
\begin{align}\label{estimate 83}
\lVert \mathcal{P}^{ii_{1}} D_{j} (\bar{b}_{1}^{\epsilon, i_{1}} \diamond \bar{u}_{2}^{\epsilon, j} ) \rVert_{\mathcal{C}^{-\frac{3}{2} - \frac{\delta}{2} - \frac{\kappa}{2}}}\lesssim   \lVert \bar{b}_{1}^{\epsilon, i_{1}} \diamond \bar{u}_{2}^{\epsilon,j} \rVert_{\mathcal{C}^{-\frac{1}{2} - \frac{\delta}{2} }} \lesssim \bar{C}_{W}^{\epsilon}(t)
\end{align} 
by Lemma \ref{Lemma 3.7} and \eqref{[Equation (3.2e)][ZZ17]}. Similarly, 
\begin{align}\label{estimate 84}
\lVert \mathcal{P}^{ii_{1}} D_{j} (\bar{b}_{2}^{\epsilon, i_{1}} \diamond \bar{u}_{1}^{\epsilon, j} - \bar{u}_{1}^{\epsilon, i_{1}} \diamond \bar{b}_{2}^{\epsilon, j}  - \bar{u}_{2}^{\epsilon, i_{1}} \diamond \bar{b}_{1}^{\epsilon, j} ) \rVert_{\mathcal{C}^{-\frac{3}{2} - \frac{\delta}{2} - \frac{\kappa}{2}}}  \lesssim \bar{C}_{W}^{\epsilon}(t). 
\end{align} 
Second, we estimate within $\bar{G}^{\epsilon, i}(r)$ of \eqref{[Equation (3.7c)][ZZ17]}, 
\begin{align}\label{estimate 82}
 \lVert \mathcal{P}^{ii_{1}} D_{j} [\bar{b}_{1}^{\epsilon, i_{1}} \diamond (\bar{u}_{3}^{\epsilon, j} + \bar{u}_{4}^{\epsilon, j})] \rVert_{\mathcal{C}^{-\frac{3}{2} - \frac{\delta}{2}-\frac{\kappa}{2}}} \lesssim \sum_{k=1}^{3} IV_{t, i_{1} j}^{k}
\end{align}  
where 
\begin{subequations}\label{estimate 78}
\begin{align}
& IV_{t, i_{1} j}^{1} \triangleq   \lVert  \pi_{<} (\bar{b}_{1}^{\epsilon, i_{1}}, \bar{u}_{3}^{\epsilon, j} + \bar{u}_{4}^{\epsilon, j}) \rVert_{\mathcal{C}^{-\frac{1}{2} - \frac{\delta}{2} - \frac{\kappa}{2}}}, \\
& IV_{t, i_{1} j}^{2} \triangleq \lVert \pi_{0,\diamond} (\bar{b}_{1}^{\epsilon, i_{1}}, \bar{u}_{3}^{\epsilon, j} + \bar{u}_{4}^{\epsilon, j})\rVert_{\mathcal{C}^{-\frac{1}{2} - \frac{\delta}{2} - \frac{\kappa}{2}}}, \\
& IV_{t, i_{1} j}^{3} \triangleq \lVert \pi_{>} (\bar{b}_{1}^{\epsilon, i_{1}}, \bar{u}_{3}^{\epsilon, j} + \bar{u}_{4}^{\epsilon, j}) \rVert_{\mathcal{C}^{-\frac{1}{2} - \frac{\delta}{2} - \frac{\kappa}{2}}},
\end{align}
\end{subequations} 
by \eqref{[Equation (3.2ad)][ZZ17]}, \eqref{[Equation (3.2ae)][ZZ17]}. We estimate 
\begin{align}\label{estimate 81}
IV_{t, i_{1}j}^{1} \lesssim \bar{C}_{W}^{\epsilon}(t) ( \lVert \bar{u}_{3}^{\epsilon, j} \rVert_{\mathcal{C}^{\frac{1}{2} - \delta}} + \lVert \bar{u}_{4}^{\epsilon, j} \rVert_{\mathcal{C}^{\frac{1}{2} - \delta_{0}}}) 
\lesssim \bar{C}_{W}^{\epsilon}(t) ( \bar{C}_{W}^{\epsilon}(t) + \lVert \bar{u}_{4}^{\epsilon, j} \rVert_{\mathcal{C}^{\frac{1}{2} - \delta_{0}}}) 
\end{align} 
by \eqref{estimate 78}, Lemma \ref{[Lemma 1.1][Y19a]} (2), \eqref{[Equation (3.2e)][ZZ17]}, \eqref{[Equation (3.2g)][ZZ17]}, and \eqref{[Equation (3.2h)][ZZ17]}, 
\begin{align}\label{estimate 79}
IV_{t, i_{1} j}^{2}\lesssim \bar{C}_{W}^{\epsilon}(t) + (1+ (\bar{C}_{W}^{\epsilon}(t))^{3}) (1+ \lVert \bar{y}^{\epsilon, \sharp} \rVert_{\mathcal{C}^{\frac{1}{2} + \beta}} + \lVert \bar{y}_{4}^{\epsilon} \rVert_{\mathcal{C}^{\frac{1}{2} - \delta_{0}}}) 
\end{align} 
by \eqref{estimate 78}, \eqref{[Equation (3.2g)][ZZ17]}, \eqref{[Equation (3.6a)][ZZ17]}, and 
\begin{align}\label{estimate 80}
IV_{t, i_{1} j}^{3} 
\lesssim ( \lVert \bar{u}_{3}^{\epsilon, j}\rVert_{\mathcal{C}^{\frac{1}{2} - \delta}} + \lVert \bar{u}_{4}^{\epsilon, j} \rVert_{\mathcal{C}^{\frac{1}{2} - \delta_{0}}}) \bar{C}_{W}^{\epsilon}(t) 
\lesssim (\bar{C}_{W}^{\epsilon}(t) + \lVert \bar{u}_{4}^{\epsilon, j} \rVert_{\mathcal{C}^{\frac{1}{2} -\delta_{0}}}) \bar{C}_{W}^{\epsilon}(t) 
\end{align} 
by \eqref{estimate 78}, Lemma \ref{[Lemma 1.1][Y19a]} (2), \eqref{[Equation (3.2g)][ZZ17]}, \eqref{[Equation (3.2e)][ZZ17]} and \eqref{[Equation (3.2h)][ZZ17]}. Applying \eqref{estimate 81} - \eqref{estimate 80} to \eqref{estimate 82} gives 
\begin{align}\label{estimate 85}
&\lVert \mathcal{P}^{ii_{1}} D_{j} [\bar{b}_{1}^{\epsilon, i_{1}} \diamond (\bar{u}_{3}^{\epsilon, j} + \bar{u}_{4}^{\epsilon, j})] \rVert_{\mathcal{C}^{-\frac{3}{2} - \frac{\delta}{2}-\frac{\kappa}{2}}} \nonumber\\
\lesssim& (1+ ( \bar{C}_{W}^{\epsilon}(t))^{3}) (1+ \lVert \bar{y}^{\epsilon, \sharp} \rVert_{\mathcal{C}^{\frac{1}{2} + \beta}} + \lVert \bar{y}_{4}^{\epsilon} \rVert_{\mathcal{C}^{\frac{1}{2} - \delta_{0}}}). 
\end{align} 
Similarly, 
\begin{align}\label{estimate 86}
&\lVert \mathcal{P}^{ii_{1}} D_{j} [( \bar{b}_{3}^{\epsilon, i_{1}} + \bar{b}_{4}^{\epsilon, i_{1}}) \diamond \bar{u}_{1}^{\epsilon, j}  - \bar{u}_{1}^{\epsilon, i_{1}} \diamond (\bar{b}_{3}^{\epsilon, j} + \bar{b}_{4}^{\epsilon, j}) - (\bar{u}_{3}^{\epsilon, i_{1}} + \bar{u}_{4}^{\epsilon, i_{1}}) \diamond \bar{b}_{1}^{\epsilon, j}] \rVert_{\mathcal{C}^{-\frac{3}{2} - \frac{\delta}{2}-\frac{\kappa}{2}}} \nonumber\\
\lesssim& (1+ ( \bar{C}_{W}^{\epsilon}(t))^{3}) (1+ \lVert \bar{y}^{\epsilon, \sharp} \rVert_{\mathcal{C}^{\frac{1}{2} + \beta}} + \lVert \bar{y}_{4}^{\epsilon} \rVert_{\mathcal{C}^{\frac{1}{2} - \delta_{0}}}). 
\end{align}  
Third, we estimate within $\bar{G}^{\epsilon, i}(r)$ of \eqref{[Equation (3.7c)][ZZ17]}
\begin{align}\label{estimate 87}
\lVert \mathcal{P}^{ii_{1}} D_{j} (\bar{b}_{2}^{\epsilon, i_{1}} \diamond \bar{u}_{2}^{\epsilon, j}, \bar{u}_{2}^{\epsilon, i_{1}} \diamond \bar{b}_{2}^{\epsilon, j}) \rVert_{\mathcal{C}^{-\frac{3}{2} - \frac{\delta}{2} - \frac{\kappa}{2}}}
 \lesssim \bar{C}_{W}^{\epsilon}(t) 
\end{align} 
by \eqref{[Equation (3.2g)][ZZ17]} and \eqref{[Equation (3.2e)][ZZ17]}. Fourth,  we estimate within $\bar{G}^{\epsilon, i}(r)$ of \eqref{[Equation (3.7c)][ZZ17]}, 
\begin{align}\label{estimate 88}
 \lVert \mathcal{P}^{ii_{1}} D_{j} [ \bar{b}_{2}^{\epsilon, i_{1}} (\bar{u}_{3}^{\epsilon, j} + \bar{u}_{4}^{\epsilon, j})] \rVert_{\mathcal{C}^{-\frac{3}{2} - \frac{\delta}{2} - \frac{\kappa}{2}}} 
\lesssim& \bar{C}_{W}^{\epsilon}(t) (\bar{C}_{W}^{\epsilon}(t) + \lVert \bar{u}_{4}^{\epsilon} \rVert_{\mathcal{C}^{\frac{1}{2} - \delta_{0}}}) 
\end{align} 
by \eqref{[Equation (3.2g)][ZZ17]}, Lemma \ref{[Lemma 1.1][Y19a]} (4), \eqref{[Equation (3.2h)][ZZ17]}, and \eqref{[Equation (3.2e)][ZZ17]}. Similarly, 
\begin{align}\label{estimate 89}
& \lVert \mathcal{P}^{ii_{1}} D_{j} [ \bar{u}_{2}^{\epsilon, j} (\bar{b}_{3}^{\epsilon, i_{1}} + \bar{b}_{4}^{\epsilon, i_{1}}) - \bar{u}_{2}^{\epsilon, i_{1}} (\bar{b}_{3}^{\epsilon, j} + \bar{b}_{4}^{\epsilon, j}) - \bar{b}_{2}^{\epsilon, j} (\bar{u}_{3}^{\epsilon, i_{1}} + \bar{u}_{4}^{\epsilon, i_{1}}) ] \rVert_{\mathcal{C}^{-\frac{3}{2} - \frac{\delta}{2} - \frac{\kappa}{2}}} \nonumber\\
\lesssim&  \bar{C}_{W}^{\epsilon}(t) (\bar{C}_{W}^{\epsilon}(t) + \lVert \bar{y}_{4}^{\epsilon} \rVert_{\mathcal{C}^{\frac{1}{2} - \delta_{0}}}). 
\end{align} 
Fifth, we estimate within $\bar{G}^{\epsilon, i}(r)$ of \eqref{[Equation (3.7c)][ZZ17]}, 
\begin{align}\label{estimate 90}
& \lVert \mathcal{P}^{ii_{1}} D_{j} [ ( \bar{b}_{3}^{\epsilon, i_{1}} + \bar{b}_{4}^{\epsilon, i_{1}}) ( \bar{u}_{3}^{\epsilon, j} + \bar{u}_{4}^{\epsilon, j}) - (\bar{u}_{3}^{\epsilon, i_{1}} + \bar{u}_{4}^{\epsilon, i_{1}})(\bar{b}_{3}^{\epsilon, j} + \bar{b}_{4}^{\epsilon, j}) ] \rVert_{\mathcal{C}^{-\frac{3}{2} - \frac{\delta}{2} - \frac{\kappa}{2}}} \nonumber\\
\lesssim& \lVert \bar{b}_{3}^{\epsilon} + \bar{b}_{4}^{\epsilon} \rVert_{\mathcal{C}^{\delta}} \lVert \bar{u}_{3}^{\epsilon} + \bar{u}_{4}^{\epsilon} \rVert_{\mathcal{C}^{\delta}} \lesssim ( \bar{C}_{W}^{\epsilon}(t))^{2}+ \lVert \bar{y}_{4}^{\epsilon} \rVert_{\mathcal{C}^{\delta}}^{2} 
\end{align} 
by Lemma \ref{[Lemma 1.1][Y19a]} (4), \eqref{[Equation (3.2g)][ZZ17]}, and \eqref{[Equation (3.2h)][ZZ17]}. Hence, \eqref{estimate 83}-\eqref{estimate 84} and \eqref{estimate 85} - \eqref{estimate 90} give us 
\begin{align}\label{[Equation (3.7d)][ZZ17]}
&\lVert \bar{G}^{\epsilon}(r) \rVert_{\mathcal{C}^{-\frac{3}{2} - \frac{\delta}{2} - \frac{\kappa}{2}}} \nonumber\\
\lesssim& (1+ (\bar{C}_{W}^{\epsilon}(t))^{3})(1+ \lVert \bar{y}^{\epsilon,\sharp}(r) \rVert_{\mathcal{C}^{\frac{1}{2} + \beta}} + \lVert \bar{y}_{4}^{\epsilon}(r) \rVert_{\mathcal{C}^{\frac{1}{2} - \delta_{0}}}) + \lVert \bar{y}_{4}^{\epsilon}(r) \rVert_{\mathcal{C}^{\delta}}^{2}. 
\end{align} 
Now applying \eqref{[Equation (3.7b)][ZZ17]} to \eqref{[Equation (3.7a)][ZZ17]} leads to 
\begin{align}\label{estimate 93}
III_{t}^{1} &\lesssim \bar{C}_{W}^{\epsilon}(t) [ \int_{0}^{t} (t-s)^{- 1- \frac{ \frac{\delta}{2} + \beta + \kappa}{2} + b_{0}} s^{- \frac{ z + \frac{\kappa}{2} + 2b_{0}}{2}} \lVert b_{0} - \bar{b}_{1}^{\epsilon}(0) \rVert_{\mathcal{C}^{-z}} ds \\
& \hspace{3mm} + \int_{0}^{t} \int_{0}^{s} (t-s)^{-1 - \frac{ \frac{\delta}{2} + \beta + \kappa}{2} + \frac{b}{2}} (s-r)^{- \frac{ \kappa + b + \frac{3}{2} + \frac{\delta}{2}}{2}} \lVert \bar{G}^{\epsilon}(r) \rVert_{\mathcal{C}^{- \frac{3}{2} - \frac{\delta}{2} - \frac{\kappa}{2}}} dr ds \nonumber  \\
& \hspace{3mm} + \int_{0}^{t} (t-s)^{-1 - \frac{ \frac{\delta}{2} + \beta + \kappa}{2} + b_{1}} \left(\int_{s}^{t} (t-r)^{- \frac{ \kappa + \frac{3}{2} + \frac{\delta}{2}}{2(1-b_{1})}} \lVert \bar{G}^{\epsilon}(r) \rVert_{\mathcal{C}^{-\frac{3}{2} - \frac{\delta}{2} - \frac{\kappa}{2}}}^{\frac{1}{1-b_{1}}}  dr \right)^{1-b_{1}} ds] \nonumber
\end{align} 
in which 
\begin{align}\label{estimate 92}
&  \int_{0}^{t} \int_{0}^{s} (t-s)^{-1 - \frac{ \frac{\delta}{2} + \beta + \kappa}{2} + \frac{b}{2}} (s-r)^{- \frac{ \kappa + b + \frac{3}{2} + \frac{\delta}{2}}{2}} \lVert \bar{G}^{\epsilon}(r) \rVert_{\mathcal{C}^{- \frac{3}{2} - \frac{\delta}{2} - \frac{\kappa}{2}}} dr ds\\
=& \int_{0}^{t} (t-r)^{- \kappa - \frac{3}{4} -\frac{\delta}{2} - \frac{\beta}{2}} \int_{0}^{1} (1-l)^{-1 - \frac{ \frac{\delta}{2} + \beta + \kappa}{2} + \frac{b}{2}} l^{- \frac{ \kappa + b + \frac{3}{2} + \frac{\delta}{2}}{2}} dl \lVert \bar{G}^{\epsilon}(r) \rVert_{\mathcal{C}^{- \frac{3}{2} - \frac{\delta}{2} - \frac{\kappa}{2}}} dr \nonumber
\end{align} 
by Fubini theorem, while 
\begin{align}\label{estimate 91}
& \int_{0}^{t} (t-s)^{-1 - \frac{ \frac{\delta}{2} + \beta + \kappa}{2} + b_{1}} \left(\int_{s}^{t} (t-r)^{- \frac{ \kappa + \frac{3}{2} + \frac{\delta}{2}}{2(1-b_{1})}} \lVert \bar{G}^{\epsilon}(r) \rVert_{\mathcal{C}^{-\frac{3}{2} - \frac{\delta}{2} - \frac{\kappa}{2}}}^{\frac{1}{1-b_{1}}}  dr \right)^{1-b_{1}} ds \nonumber \\
\lesssim& \left( t^{- \frac{ \frac{\delta}{2} + \beta + \kappa}{2} + b_{1}} \right)^{b_{1}} \nonumber\\
& \times  \left( \int_{0}^{t} \int_{0}^{r} (t-s)^{-1 - \frac{ \frac{\delta}{2} + \beta + \kappa}{2} + b_{1}} (t-r)^{- \frac{\kappa + \frac{3}{2} + \frac{\delta}{2}}{2(1-b_{1})}} \lVert \bar{G}^{\epsilon}(r) \rVert_{\mathcal{C}^{-\frac{3}{2} - \frac{\delta}{2} - \frac{\kappa}{2}}}^{\frac{1}{1-b_{1}}} ds dr \right)^{1-b_{1}} \nonumber\\
\lesssim& t^{- \frac{ \frac{\delta}{2} + \beta + \kappa}{2} + b_{1}} \left( \int_{0}^{t} (t-r)^{- \frac{\kappa + \frac{3}{2} + \frac{\delta}{2}}{2(1-b_{1})}} \lVert \bar{G}^{\epsilon}(r) \rVert_{\mathcal{C}^{-\frac{3}{2} - \frac{\delta}{2} - \frac{\kappa}{2}}}^{\frac{1}{1-b_{1}}} dr \right)^{1-b_{1}}  
\end{align} 
by H$\ddot{\mathrm{o}}$lder's inequality, Fubini theorem and \eqref{b and b0}. Hence, by applying \eqref{estimate 92} - \eqref{estimate 91} to \eqref{estimate 93} we deduce 
\begin{align}\label{estimate 94}
III_{t}^{1} \lesssim& \bar{C}_{W}^{\epsilon}(t) [ t^{- \frac{ \frac{\delta}{2} + \beta + z}{2} - \frac{3\kappa}{4}} \lVert b_{0} - \bar{b}_{1}^{\epsilon}(0) \rVert_{\mathcal{C}^{-z}} \nonumber \\
& \hspace{3mm} + \int_{0}^{t} (t-r)^{- \kappa - \frac{3}{4} - \frac{\delta}{2} - \frac{\beta}{2}} \lVert \bar{G}^{\epsilon}(r) \rVert_{\mathcal{C}^{-\frac{3}{2} - \frac{\delta}{2} - \frac{\kappa}{2}}} dr \nonumber \\
& \hspace{3mm} + t^{- [ \frac{ \frac{\delta}{2} + \beta + \kappa}{2} - b_{1}]} \left( \int_{0}^{t} (t-r)^{- \frac{ \kappa + \frac{3}{2} + \frac{\delta}{2}}{2(1-b_{1})}} \lVert \bar{G}^{\epsilon}(r) \rVert_{\mathcal{C}^{- \frac{3}{2} - \frac{\delta}{2} - \frac{\kappa}{2}}}^{\frac{1}{1-b_{1}}} dr \right)^{1- b_{1}} ].
\end{align} 
At last, applying \eqref{[Equation (3.7d)][ZZ17]} to \eqref{estimate 94} gives us 
\begin{align}\label{[Equation (3.8)][ZZ17]}
III_{t}^{1} &\lesssim \bar{C}_{W}^{\epsilon} (t) [ t^{ - \frac{ \frac{\delta}{2} + \beta + z}{2} - \frac{3\kappa}{4}} \lVert b_{0} - \bar{b}_{1}^{\epsilon}(0) \rVert_{\mathcal{C}^{-z}}  \\
&+ \int_{0}^{t} (t-r)^{-\kappa - \frac{3}{4} - \frac{\delta}{2} - \frac{\beta}{2}} \nonumber \\
& \hspace{1mm} \times [(1+ ( \bar{C}_{W}^{\epsilon}(t))^{3}) (1+ \lVert \bar{y}^{\epsilon, \sharp} (r) \rVert_{\mathcal{C}^{\frac{1}{2} + \beta}} + \lVert \bar{y}_{4}^{\epsilon}(r) \rVert_{\mathcal{C}^{\frac{1}{2} - \delta_{0}}} ) + \lVert \bar{y}_{4}^{\epsilon}(r) \rVert_{\mathcal{C}^{\delta}}^{2} ] dr \nonumber\\
&+ t^{- [\frac{ \frac{\delta}{2} + \beta + \kappa}{2} - b_{1}]} ( \int_{0}^{t} (t-r)^{- \frac{ \frac{3}{2} + \frac{\delta}{2} + \kappa}{2(1-b_{1})}} \nonumber\\
& \hspace{1mm} \times [(1+ (\bar{C}_{W}^{\epsilon}(t))^{3} )(1+ \lVert \bar{y}^{\epsilon, \sharp}(r) \rVert_{\mathcal{C}^{\frac{1}{2} + \beta}} + \lVert \bar{y}_{4}^{\epsilon}(r) \rVert_{\mathcal{C}^{\frac{1}{2} - \delta_{0}}}) + \lVert \bar{y}_{4}^{\epsilon}(r) \rVert_{\mathcal{C}^{\delta}}^{2} ]^{\frac{1}{1-b_{1}}} dr )^{1- b_{1}}. \nonumber
\end{align}
On the other hand, 
\begin{align}\label{[Equation (3.7)][ZZ17]}
III_{t}^{2} \lesssim \bar{C}_{W}^{\epsilon}(t) t^{ \frac{1}{4} - \frac{\delta_{0}}{2} - \frac{\delta}{4} - \frac{\beta}{2} - \frac{\kappa}{2}} \lVert (\bar{b}_{3}^{\epsilon} + \bar{b}_{4}^{\epsilon})(t) \rVert_{\mathcal{C}^{\frac{1}{2} - \delta_{0}}} 
\end{align} 
by \eqref{III2}, \eqref{regularity 0}, \eqref{[Equation (3.3c)][ZZ17]}, Lemma \ref{Lemma 3.11}, \eqref{[Equation (3.2e)][ZZ17]}, and \eqref{beta}. Combining \eqref{[Equation (3.7)][ZZ17]} and \eqref{[Equation (3.8)][ZZ17]} finally gives us \eqref{[Equation (3.9)][ZZ17]}. The estimates \eqref{[Equation (3.10)][ZZ17]} - \eqref{[Equation (3.17e)][ZZ17]} are proven similarly and we leave details in the Appendix Subsection \ref{Subsection 3.2}  for readers' convenience. This completes the proof of Proposition \ref{Proposition on F}. 
\end{proof}

\begin{proposition}\label{[Lemma 3.10][ZZ17]}
Let $\delta_{0}$ and $z$ satisfy the hypothesis of Theorem \ref{Theorem 1.2}, while $\delta, \kappa$ and $\beta$ satisfy \eqref{[Equation (3.2g)][ZZ17]}, \eqref{[Equation (3.3c)][ZZ17]} and \eqref{beta}, respectively. Then there exists some $T_{0} > 0$ that satisfies $T_{0} \leq T_{\epsilon}$ and 
\begin{subequations}
\begin{align}
& \sup_{t\in [0,T_{0}]} [ t^{\delta + z + \kappa} \lVert \bar{y}^{\epsilon, \sharp} (t) \rVert_{\mathcal{C}^{\frac{1}{2} + \beta}} + t^{\frac{ \delta + z + \kappa}{2}} \lVert \bar{y}^{\epsilon, \sharp}(t) \rVert_{\mathcal{C}^{\delta}}] \lesssim_{T_{0}, \bar{C}_{W}^{\epsilon}(T_{0}), \lVert y_{0} \rVert_{\mathcal{C}^{-z}}} 1, \label{[Equation (3.11)][ZZ17]} \\
& \sup_{t\in [0,T_{0}]} t^{\frac{ \frac{1}{2} - \delta_{0}  + z + \kappa}{2}} \lVert \bar{y}_{4}^{\epsilon}(t) \rVert_{\mathcal{C}^{\frac{1}{2} - \delta_{0}}} \lesssim_{T_{0}, \bar{C}_{W}^{\epsilon}(T_{0}), \lVert y_{0} \rVert_{\mathcal{C}^{-z}}} 1\label{[Equation (3.17a)][ZZ17]}
\end{align}
\end{subequations} 
for any $\epsilon \in (0,1)$, and thus $T_{0}$ is independent of $\epsilon$. 
\end{proposition} 

\begin{proof}[Proof of Proposition \ref{[Lemma 3.10][ZZ17]}]
First, for any $i\in \{1,2,3\}$, 
\begin{align}
& \lVert (\bar{u}_{4}^{\epsilon, i}, \bar{b}_{4}^{\epsilon, i} )(t) \rVert_{\mathcal{C}^{\frac{1}{2} - \delta_{0}}} \nonumber\\
\lesssim& [t^{\frac{\delta}{4}}\bar{C}_{W}^{\epsilon}(t) + \sum_{i_{1}=1}^{3} \lVert (\bar{u}_{4}^{\epsilon, i_{1}}, \bar{b}_{4}^{\epsilon, i_{1}}) \rVert_{\mathcal{C}^{\frac{1}{2} - \delta_{0}}} ] t^{\frac{\delta}{4}} \bar{C}_{W}^{\epsilon}(t) + \lVert (\bar{u}^{\epsilon, \sharp, i}, \bar{b}^{\epsilon,\sharp, i}) \rVert_{\mathcal{C}^{\frac{1}{2} - \delta_{0}}} 
\end{align} 
by \eqref{paracontrolled ansatz u}, \eqref{paracontrolled ansatz b}, Lemmas \ref{Lemma 3.7} and \ref{[Lemma 1.1][Y19a]} (2), \eqref{[Equation (3.2g)][ZZ17]}, \eqref{[Equation (3.3)][ZZ17]} and \eqref{[Equation (3.2h)][ZZ17]}. Therefore, for $t > 0$ sufficiently small, we deduce 
\begin{equation}\label{[Equation (3.12)][ZZ17]} 
\lVert \bar{y}_{4}^{\epsilon}(t) \rVert_{\mathcal{C}^{\frac{1}{2} - \delta_{0}}} \lesssim (\bar{C}_{W}^{\epsilon}(t))^{2} + \lVert \bar{y}^{\epsilon,\sharp} \rVert_{\mathcal{C}^{\frac{1}{2} - \delta_{0}}}, 
\end{equation} 
which is slightly different from a component-wise estimate in \cite[Equation (3.12)]{ZZ17}, that we believe to be difficult due to $\mathcal{P}^{ii_{1}}$. Similarly, 
\begin{align}
& \lVert (\bar{u}_{4}^{\epsilon, i}, \bar{b}_{4}^{\epsilon, i} )(t) \rVert_{\mathcal{C}^{\delta}} \nonumber\\
\lesssim& [t^{\frac{\delta}{4}} \bar{C}_{W}^{\epsilon}(t) + \sum_{i_{1} =1}^{3} \lVert ( \bar{u}_{4}^{\epsilon, i_{1}}, \bar{b}_{4}^{\epsilon, i_{1}}) \rVert_{\mathcal{C}^{\delta}}] t^{\frac{\delta}{4}} \bar{C}_{W}^{\epsilon}(t) + \lVert ( \bar{u}^{\epsilon, \sharp, i}, \bar{b}^{\epsilon, \sharp, i}) \rVert_{\mathcal{C}^{\delta}} 
\end{align} 
by \eqref{paracontrolled ansatz u}, \eqref{paracontrolled ansatz b}, Lemma \ref{[Lemma 1.1][Y19a]} (2),  \eqref{[Equation (3.2g)][ZZ17]}, \eqref{regularity 0}, \eqref{[Equation (3.3)][ZZ17]} and \eqref{[Equation (3.2h)][ZZ17]}. Thus, for $t> 0$ sufficiently small, we obtain 
\begin{equation}\label{[Equation (3.13)][ZZ17]}
\lVert \bar{y}_{4}^{\epsilon}(t) \rVert_{\mathcal{C}^{\delta}} \lesssim (\bar{C}_{W}^{\epsilon} (t))^{2} + \lVert \bar{y}^{\epsilon, \sharp}(t) \rVert_{\mathcal{C}^{\delta}}. 
\end{equation}  
Now we first assume that $\sup_{\epsilon} \bar{C}_{W}^{\epsilon}(t) < \infty$ so that $t$ becomes independent of $\epsilon > 0$ and estimate for $t \in (0,1)$, 
\begin{align}
& t^{\delta + z + \kappa} \lVert \bar{y}^{\epsilon, \sharp} (t) \rVert_{\mathcal{C}^{\frac{1}{2} + \beta}} \\
\lesssim& t^{\kappa} \lVert y_{0} - \bar{y}_{1}^{\epsilon}(0) \rVert_{\mathcal{C}^{-z}}\nonumber\\
&+ t^{\delta + z + \kappa} \int_{0}^{t} (t-s)^{-\frac{3}{4} - \delta - \frac{\beta}{2} - \kappa}   [(1+ ( \bar{C}_{W}^{\epsilon}(t))^{3}) (1+ \lVert \bar{y}^{\epsilon,\sharp}(s) \rVert_{\mathcal{C}^{\frac{1}{2} + \beta}} + \lVert \bar{y}_{4}^{\epsilon}(s) \rVert_{\mathcal{C}^{\frac{1}{2} - \delta_{0}}} ) \nonumber\\
& \hspace{47mm} + \lVert \bar{y}_{4}^{\epsilon}(s) \rVert_{\mathcal{C}^{\delta}}^{2}] ds + t^{\delta + z + \kappa} \sum_{i=1}^{3} \lVert (\bar{F}_{u}^{\epsilon, i}, \bar{F}_{b}^{\epsilon, i})(t) \rVert_{\mathcal{C}^{\frac{1}{2} + \beta}} \nonumber
\end{align} 
by \eqref{[Equation (3.5) for u][ZZ17]}, \eqref{[Equation (3.5) for b][ZZ17]}, Lemmas \ref{Lemma 3.10} and \ref{Lemma 3.7}, \eqref{beta} and \eqref{[Equation (3.6)][ZZ17]}. We apply \eqref{[Equation (3.9)][ZZ17]} to deduce 
\begin{align}\label{estimate 95}
& t^{\delta + z + \kappa} \lVert \bar{y}^{\epsilon, \sharp} (t) \rVert_{\mathcal{C}^{\frac{1}{2} + \beta}} \\
\lesssim& t^{\kappa} \lVert y_{0} - \bar{y}_{1}^{\epsilon}(0) \rVert_{\mathcal{C}^{-z}} \nonumber\\
&+ t^{\delta + z + \kappa} \int_{0}^{t} (t-s)^{-\frac{3}{4} - \delta - \frac{\beta}{2} - \kappa}   [(1+ ( \bar{C}_{W}^{\epsilon}(t))^{3}) (1+ \lVert \bar{y}^{\epsilon,\sharp}(s) \rVert_{\mathcal{C}^{\frac{1}{2} + \beta}} + \lVert \bar{y}_{4}^{\epsilon}(s) \rVert_{\mathcal{C}^{\frac{1}{2} - \delta_{0}}} ) \nonumber\\
& \hspace{47mm} + \lVert \bar{y}_{4}^{\epsilon}(s) \rVert_{\mathcal{C}^{\delta}}^{2}] ds \nonumber\\
&+ t^{\delta+  z + \kappa} \bar{C}_{W}^{\epsilon} (t) [ t^{ - \frac{ \frac{\delta}{2} + \beta + z}{2} - \frac{3\kappa}{4}} \lVert y_{0} - \bar{y}_{1}^{\epsilon}(0) \rVert_{\mathcal{C}^{-z}} \nonumber \\
& \hspace{2mm} + \int_{0}^{t} (t-r)^{-\kappa - \frac{3}{4} - \frac{\delta}{2} - \frac{\beta}{2}} \nonumber \\
& \hspace{5mm} \times [(1+ ( \bar{C}_{W}^{\epsilon}(t))^{3}) (1+ \lVert \bar{y}^{\epsilon, \sharp} (r) \rVert_{\mathcal{C}^{\frac{1}{2} + \beta}} + \lVert \bar{y}_{4}^{\epsilon}(r) \rVert_{\mathcal{C}^{\frac{1}{2} - \delta_{0}}} ) + \lVert \bar{y}_{4}^{\epsilon}(r) \rVert_{\mathcal{C}^{\delta}}^{2} ] dr \nonumber\\
& \hspace{2mm} + t^{- [\frac{ \frac{\delta}{2} - \beta + \kappa}{2} + b_{1}]} ( \int_{0}^{t} (t-r)^{- \frac{ \frac{3}{2} + \frac{\delta}{2} + \kappa}{2(1-b_{1})}} \nonumber\\
& \hspace{5mm} \times [(1+ (\bar{C}_{W}^{\epsilon}(t))^{3} )(1+ \lVert \bar{y}^{\epsilon, \sharp}(r) \rVert_{\mathcal{C}^{\frac{1}{2} + \beta}} + \lVert \bar{y}_{4}^{\epsilon}(r) \rVert_{\mathcal{C}^{\frac{1}{2} - \delta_{0}}}) + \lVert \bar{y}_{4}^{\epsilon}(r) \rVert_{\mathcal{C}^{\delta}}^{2} ]^{\frac{1}{1-b_{1}}} dr )^{1- b_{1}} \nonumber \\
&\hspace{2mm} + t^{\frac{1}{4} - \frac{ \delta_{0} + \beta + \kappa + \frac{\delta}{2}}{2}} \lVert ( \bar{y}_{3}^{\epsilon}  + \bar{y}_{4}^{\epsilon})(t) \rVert_{\mathcal{C}^{\frac{1}{2} - \delta_{0}}}].\nonumber
\end{align} 
We estimate
\begin{align*}
\lVert (\bar{y}_{3}^{\epsilon} + \bar{y}_{4}^{\epsilon})(t) \rVert_{\mathcal{C}^{\frac{1}{2} - \delta_{0}}} 
\lesssim  \bar{C}_{W}^{\epsilon} (t)+ ( \bar{C}_{W}^{\epsilon}(t))^{2} + \lVert \bar{y}^{\epsilon, \sharp}(t) \rVert_{\mathcal{C}^{\frac{1}{2} + \beta}} 
\end{align*} 
by \eqref{[Equation (3.2g)][ZZ17]}, \eqref{[Equation (3.2h)][ZZ17]}, \eqref{[Equation (3.12)][ZZ17]} and \eqref{regularity 0} and rely on \eqref{beta}, \eqref{[Equation (3.12)][ZZ17]} and \eqref{[Equation (3.13)][ZZ17]} to deduce from \eqref{estimate 95}, 
\begin{align}\label{estimate 96}
&t^{\delta + z + \kappa} \lVert \bar{y}^{\epsilon, \sharp} (t) \rVert_{\mathcal{C}^{\frac{1}{2} + \beta}} \\
\lesssim& t^{\frac{\kappa}{4}} (1+ \bar{C}_{W}^{\epsilon}(t)) \lVert y_{0} - \bar{y}_{1}^{\epsilon}(0) \rVert_{\mathcal{C}^{-z}} \nonumber\\
&+ t^{\delta+ z + \kappa} \int_{0}^{t} (t-s)^{- \frac{3}{4} - \delta - \frac{\beta}{2} - \kappa} [(1+ (\bar{C}_{W}^{\epsilon}(t))^{4}) (1+ \lVert \bar{y}^{\epsilon, \sharp}(s) \rVert_{\mathcal{C}^{\frac{1}{2} + \beta}} + \lVert \bar{y}^{\epsilon, \sharp}(s) \rVert_{\mathcal{C}^{\delta}}^{2})] ds \nonumber\\
&+ t^{\delta + z + \kappa} \bar{C}_{W}^{\epsilon}(t) [ \int_{0}^{t} (t-r)^{- \kappa - \frac{3}{4} - \frac{\delta}{2} - \frac{\beta}{2}} \nonumber\\
& \hspace{40mm} \times [(1+ (\bar{C}_{W}^{\epsilon}(t))^{4} ) ( 1+ \lVert \bar{y}^{\epsilon, \sharp}(r) \rVert_{\mathcal{C}^{\frac{1}{2} + \beta}} + \lVert \bar{y}^{\epsilon, \sharp}(r) \rVert_{\mathcal{C}^{\delta}}^{2})] dr \nonumber\\
& \hspace{20mm} + t^{ - [ \frac{ ( \frac{\delta}{2} + \beta  + \kappa)}{2} - b_{1} ]} ( \int_{0}^{t} (t-r)^{- \frac{ \frac{3}{2} + \frac{\delta}{2} + \kappa}{2(1-b_{1})}} \nonumber\\
& \hspace{25mm} \times [(1+ (\bar{C}_{W}^{\epsilon}(t))^{4})(1+ \lVert \bar{y}^{\epsilon, \sharp}(r) \rVert_{\mathcal{C}^{\frac{1}{2} + \beta}} + \lVert \bar{y}^{\epsilon, \sharp}(r) \rVert_{\mathcal{C}^{\delta}}^{2})]^{\frac{1}{1-b_{1}}} dr )^{1-b_{1}}\nonumber\\
&\hspace{20mm} + t^{\frac{1}{4} - \frac{ \delta_{0} + \beta + \kappa + \frac{\delta}{2}}{2}} (1+ ( \bar{C}_{W}^{\epsilon}(t))^{2} + \lVert \bar{y}^{\epsilon, \sharp}(t) \rVert_{\mathcal{C}^{\frac{1}{2} + \beta}})]. \nonumber 
\end{align} 
Moreover, 
\begin{align*}
&  t^{ \delta + z + \kappa- [ \frac{ ( \frac{\delta}{2} + \beta  + \kappa)}{2} - b_{1} ]} ( \int_{0}^{t} (t-r)^{- \frac{ \frac{3}{2} + \frac{\delta}{2} + \kappa}{2(1-b_{1})}} \nonumber\\
& \hspace{27mm} \times [(1+ (\bar{C}_{W}^{\epsilon}(t))^{4})(1+ \lVert \bar{y}^{\epsilon, \sharp}(r) \rVert_{\mathcal{C}^{\frac{1}{2} + \beta}} + \lVert \bar{y}^{\epsilon, \sharp}(r) \rVert_{\mathcal{C}^{\delta}}^{2})]^{\frac{1}{1-b_{1}}} dr )^{1-b_{1}}\nonumber\\
\lesssim& C(t) + C(\bar{C}_{W}^{\epsilon}(t))t^{\frac{ \delta + z + \kappa}{1-b_{1}}} \int_{0}^{t} (t-r)^{- \frac{ \frac{3}{2} + \frac{\delta}{2} + \kappa}{2(1-b_{1})}} [ \lVert \bar{y}^{\epsilon, \sharp}(r) \rVert_{\mathcal{C}^{\frac{1}{2} + \beta}} + \lVert \bar{y}^{\epsilon, \sharp}(r) \rVert_{\mathcal{C}^{\delta}}^{2}]^{\frac{1}{1-b_{1}}} dr 
\end{align*} 
by Young's inequality, \eqref{b1} and \eqref{beta} so that \eqref{estimate 96} finally leads us to  
\begin{align}\label{estimate 98}
&t^{\delta + z + \kappa} \lVert \bar{y}^{\epsilon, \sharp} (t) \rVert_{\mathcal{C}^{\frac{1}{2} + \beta}} \\
\lesssim& C(t, \bar{C}_{W}^{\epsilon}(t)) \lVert y_{0} - \bar{y}_{1}^{\epsilon}(0) \rVert_{\mathcal{C}^{-z}} + C(t, \bar{C}_{W}^{\epsilon}(t)) \nonumber\\
&+ C(\bar{C}_{W}^{\epsilon}(t)) t^{\delta + z + \kappa} \int_{0}^{t} (t-r)^{-\frac{3}{4} - \delta - \frac{\beta}{2} - \kappa} [\lVert \bar{y}^{\epsilon, \sharp}(r) \rVert_{\mathcal{C}^{\frac{1}{2} + \beta}} + \lVert \bar{y}^{\epsilon, \sharp}(r) \rVert_{\mathcal{C}^{\delta}}^{2}] dr \nonumber \\
&+ C(\bar{C}_{W}^{\epsilon}(t))  t^{\delta + z + \kappa} \int_{0}^{t} (t-r)^{-\kappa - \frac{3}{4} - \frac{\delta}{2} - \frac{\beta}{2}} [ \lVert \bar{y}^{\epsilon, \sharp}(r) \rVert_{\mathcal{C}^{\frac{1}{2} + \beta}} + \lVert \bar{y}^{\epsilon, \sharp}(r) \rVert_{\mathcal{C}^{\delta}}^{2}] dr \nonumber\\
&+ C( \bar{C}_{W}^{\epsilon}(t)) t^{\frac{ \delta + z + \kappa}{1-b_{1}}} \int_{0}^{t}(t-r)^{- \frac{ \frac{3}{2} + \frac{\delta}{2} + \kappa}{2(1-b_{1})}} [\lVert \bar{y}^{\epsilon, \sharp}(r) \rVert_{\mathcal{C}^{\frac{1}{2} + \beta}} + \lVert \bar{y}^{\epsilon, \sharp}(r) \rVert_{\mathcal{C}^{\delta}}^{2}]^{\frac{1}{1-b_{1}}} dr. \nonumber
\end{align} 
Next, 
\begin{align*}
& t^{\frac{\delta + z + \kappa}{2}} \lVert \bar{y}^{\epsilon, \sharp} (t) \rVert_{\mathcal{C}^{\delta}} \\
\lesssim& t^{\frac{\kappa}{2}} \lVert y_{0} - \bar{y}_{1}^{\epsilon}(0) \rVert_{\mathcal{C}^{-z}} \nonumber \\
&+ t^{ \frac{\delta + z + \kappa}{2}} \int_{0}^{t} (t-s)^{- \frac{3\delta}{2} - \frac{1}{2} - \kappa} \nonumber\\
& \hspace{12mm} \times [(1+ (\bar{C}_{W}^{\epsilon}(t))^{3})(1+ \lVert \bar{y}^{\epsilon, \sharp}(s) \rVert_{\mathcal{C}^{\frac{1}{2}+ \beta}} + \lVert \bar{y}_{4}^{\epsilon}(s) \rVert_{\mathcal{C}^{\frac{1}{2} - \delta_{0}}}) + \lVert \bar{y}_{4}^{\epsilon}(s)\rVert_{\mathcal{C}^{\delta}}^{2}] ds \nonumber \\
&+ t^{\frac{\delta + z + \kappa}{2}} \bar{C}_{W}^{\epsilon}(t) [ t^{\frac{1}{4} - \frac{z}{2} - \frac{3\delta}{4}} \lVert y_{0} - \bar{y}_{1}^{\epsilon} (0) \rVert_{\mathcal{C}^{-z}} \nonumber \\
& \hspace{5mm} + \int_{0}^{t} (t-r)^{- \frac{1+ 2 \delta + \kappa}{2}} \nonumber\\
& \hspace{12mm} \times [(1+ (\bar{C}_{W}^{\epsilon}(t))^{3})(1+ \lVert \bar{y}^{\epsilon, \sharp}(r) \rVert_{\mathcal{C}^{\frac{1}{2} + \beta}} + \lVert \bar{y}_{4}^{\epsilon} (r)\rVert_{\mathcal{C}^{\frac{1}{2} -\delta_{0}}}) + \lVert \bar{y}_{4}^{\epsilon}(r) \rVert_{\mathcal{C}^{\delta}}^{2} ] dr \nonumber\\
& \hspace{5mm} + t^{\frac{\kappa}{2}} \int_{0}^{t} (t-r)^{- \frac{1+ 2 \delta + \frac{3\kappa}{2}}{2}} \nonumber\\
& \hspace{12mm} \times [(1+ (\bar{C}_{W}^{\epsilon} (t))^{3})(1+ \lVert \bar{y}^{\epsilon, \sharp}(r) \rVert_{\mathcal{C}^{\frac{1}{2} + \beta}} + \lVert \bar{y}_{4}^{\epsilon}(r) \rVert_{\mathcal{C}^{\frac{1}{2} - \delta_{0}}}) + \lVert \bar{y}_{4}^{\epsilon} (r)\rVert_{\mathcal{C}^{\delta}}^{2}] dr \nonumber\\
&+ t^{\frac{1-\delta}{4}} \lVert (\bar{y}_{3}^{\epsilon} + \bar{y}_{4}^{\epsilon})(t) \rVert_{\mathcal{C}^{\delta}}] \nonumber 
\end{align*} 
by \eqref{[Equation (3.5) for u][ZZ17]}, \eqref{[Equation (3.5) for b][ZZ17]}, Lemmas \ref{Lemma 3.7} and \ref{Lemma 3.10}, \eqref{[Equation (3.6)][ZZ17]} and \eqref{[Equation (3.10)][ZZ17]}. We continue to bound this by 
\begin{align}\label{estimate 97}
& t^{\frac{\delta + z + \kappa}{2}} \lVert \bar{y}^{\epsilon, \sharp} (t) \rVert_{\mathcal{C}^{\delta}} \nonumber\\
\lesssim& C(t, \bar{C}_{W}^{\epsilon}(t)) \lVert y_{0} - \bar{y}_{1}^{\epsilon}(0) \rVert_{\mathcal{C}^{-z}} + C(t, \bar{C}_{W}^{\epsilon}(t)) \nonumber\\
&+ t^{\frac{ \delta + z + \kappa}{2}} C(\bar{C}_{W}^{\epsilon}(t)) \int_{0}^{t} (t-r)^{-\frac{3\delta}{2} - \frac{1}{2} - \kappa} [ \lVert \bar{y}^{\epsilon, \sharp}(r) \rVert_{\mathcal{C}^{\frac{1}{2} + \beta}} + \lVert \bar{y}^{\epsilon, \sharp}(r) \rVert_{\mathcal{C}^{\delta}}^{2}] dr \nonumber \\
&+ t^{\frac{\delta + z + \kappa}{2}} C(\bar{C}_{W}^{\epsilon}(t)) \int_{0}^{t} (t-r)^{- \frac{1+ 2 \delta + \kappa}{2}} [ \lVert \bar{y}^{\epsilon, \sharp}(r) \rVert_{\mathcal{C}^{\frac{1}{2} + \beta}} + \lVert \bar{y}^{\epsilon, \sharp}(r) \rVert_{\mathcal{C}^{\delta}}^{2}] dr \nonumber\\
&+ t^{\frac{ \delta + z + 2\kappa}{2}} C(\bar{C}_{W}^{\epsilon}(t)) \int_{0}^{t} (t-r)^{- \frac{1+ 2 \delta + \frac{3\kappa}{2}}{2}} [ \lVert \bar{y}^{\epsilon, \sharp}(r) \rVert_{\mathcal{C}^{\frac{1}{2} + \beta}} + \lVert \bar{y}^{\epsilon, \sharp}(r) \rVert_{\mathcal{C}^{\delta}}^{2}] dr \nonumber\\
&+ t^{\frac{ \delta + z + \kappa}{2}} t^{\frac{1-\delta}{4}} \bar{C}_{W}^{\epsilon}(t) (t^{\frac{\delta}{4}} \bar{C}_{W}^{\epsilon}(t) + ( \bar{C}_{W}^{\epsilon}(t))^{2}\lVert \bar{y}^{\epsilon, \sharp} \rVert_{\mathcal{C}^{\delta}})
\end{align} 
by \eqref{[Equation (3.12)][ZZ17]}, \eqref{regularity 0}, \eqref{[Equation (3.13)][ZZ17]}, \eqref{[Equation (3.2g)][ZZ17]}, \eqref{[Equation (3.3c)][ZZ17]} and \eqref{[Equation (3.2h)][ZZ17]}. By interpolation and considering only $t \in (0,1)$ sufficiently small, we deduce from \eqref{estimate 97}
\begin{align}\label{estimate 99}
& t^{\frac{ \delta + z + \kappa}{2}} \lVert \bar{y}^{\epsilon, \sharp} (t) \rVert_{\mathcal{C}^{\delta}} \nonumber\\
\lesssim& C(t,\bar{C}_{W}^{\epsilon}(t)) \lVert y_{0} - \bar{y}_{1}^{\epsilon}(0) \rVert_{\mathcal{C}^{-z}} + C(t, \bar{C}_{W}^{\epsilon}(t)) \nonumber\\
&+ C( \bar{C}_{W}^{\epsilon}(t)) t^{\frac{ \delta + z+ \kappa}{2}} \int_{0}^{t} (t-r)^{- \frac{3\delta}{2} - \frac{1}{2} - \kappa} [ \lVert  \bar{y}^{\epsilon, \sharp}(r) \rVert_{\mathcal{C}{\frac{1}{2} + \beta}} + \lVert \bar{y}^{\epsilon, \sharp}(r) \rVert_{\mathcal{C}^{\delta}}^{2}] dr \nonumber\\
&+ C(\bar{C}_{W}^{\epsilon}(t)) t^{\frac{ \delta + z + \kappa}{2}} \int_{0}^{t} (t-r)^{- \frac{1+ 2 \delta + \kappa}{2}} [ \lVert  \bar{y}^{\epsilon, \sharp}(r) \rVert_{\mathcal{C}{\frac{1}{2} + \beta}} + \lVert \bar{y}^{\epsilon, \sharp}(r) \rVert_{\mathcal{C}^{\delta}}^{2}] dr.
\end{align} 
Now an application of Bihari's inequality Lemma \ref{Lemma 3.14} on \eqref{estimate 98} and \eqref{estimate 99}  implies the existence of $T_{0}$ that is independent of $\epsilon > 0$ such that \eqref{[Equation (3.11)][ZZ17]}  holds. Next, we estimate 
\begin{align}\label{[Equation (3.17)][ZZ17]}
 \sup_{t \in [0, T_{0}]} t^{\delta + z + \kappa} \lVert ( \bar{\phi}_{u}^{\epsilon, \sharp} , \bar{\phi}_{b}^{\epsilon, \sharp} )(t) \rVert_{\mathcal{C}^{-1 - 2 \delta - 2 \kappa}} \lesssim_{T_{0}, \bar{C}_{W}^{\epsilon}(T_{0}), \lVert y_{0} \rVert_{\mathcal{C}^{-z}}} 1 
\end{align}
by \eqref{[Equation (3.6)][ZZ17]}, \eqref{[Equation (3.12)][ZZ17]}, \eqref{[Equation (3.13)][ZZ17]}, \eqref{regularity 0} and \eqref{[Equation (3.11)][ZZ17]}. Next, 
\begin{align}\label{first estimate}
& \sup_{t\in [0, T_{0}]} t^{\frac{ \frac{1}{2} - \delta_{0} + z + \kappa}{2}} \lVert \bar{y}_{4}^{\epsilon}(t) \rVert_{\mathcal{C}^{\frac{1}{2} - \delta_{0}}} \\
\lesssim& \sup_{t\in [0,T_{0}]} t^{ \frac{ \frac{1}{2} - \delta_{0} + z + \kappa}{2}} [ (\bar{C}_{W}^{\epsilon}(T_{0}))^{2} \nonumber\\
& \hspace{13mm} + \sum_{i=1}^{3} \lVert (P_{t} (\sum_{i_{1} =1}^{3} \mathcal{P}^{ii_{1}} u_{0}^{i_{1}} - \bar{u}_{1}^{\epsilon, i} (0)), P_{t} (\sum_{i_{1} =1}^{3} \mathcal{P}^{ii_{1}} b_{0}^{i_{1}} - \bar{b}_{1}^{\epsilon, i} (0))) \rVert_{\mathcal{C}^{\frac{1}{2} - \delta_{0}}}\nonumber\\
& \hspace{13mm} + \sum_{i=1}^{3} \int_{0}^{t} \lVert (P_{t-s} \bar{\phi}_{u}^{\epsilon, \sharp, i}, P_{t-s} \bar{\phi}_{b}^{\epsilon, \sharp, i})(s) \rVert_{\mathcal{C}^{\frac{1}{2} - \delta_{0}}} ds + \sum_{i=1}^{3} \lVert (\bar{F}_{u}^{\epsilon, i}, \bar{F}_{b}^{\epsilon, i})(t) \rVert_{\mathcal{C}^{\frac{1}{2} - \delta_{0}}}] \nonumber
\end{align} 
by \eqref{[Equation (3.12)][ZZ17]}, \eqref{[Equation (3.5) for u][ZZ17]}, \eqref{[Equation (3.5) for b][ZZ17]}. First, 
\begin{equation}\label{estimate 100}
t^{\frac{ \frac{1}{2} - \delta_{0} + z + \kappa}{2}} (\bar{C}_{W}^{\epsilon}(T_{0}))^{2} \lesssim_{T_{0}, \bar{C}_{W}^{\epsilon}(T_{0})} 1 
\end{equation} 
as $\delta_{0} \in (0, \frac{1}{2})$ by hypothesis. Second, 
\begin{align}
& t^{\frac{\frac{1}{2} - \delta_{0} + z + \kappa}{2}} \lVert (P_{t} (\sum_{i_{1} =1}^{3} \mathcal{P}^{ii_{1}} u_{0}^{i_{1}} - \bar{u}_{1}^{\epsilon, i} (0)), P_{t}(\sum_{i_{1} =1}^{3} \mathcal{P}^{ii_{1}} b_{0}^{i_{1}} - \bar{b}_{1}^{\epsilon, i}(0))) \rVert_{\mathcal{C}^{\frac{1}{2} - \delta_{0}}} \nonumber \\\
\lesssim& t^{\frac{\frac{1}{2} - \delta_{0} + z + \kappa}{2}} t^{- \frac{ \frac{1}{2} - \delta_{0} + z}{2}} \lVert y_{0} - \bar{y}_{1}^{\epsilon}(0) \rVert_{\mathcal{C}^{-z}}  \approx t^{\frac{\kappa}{2}} \lVert y_{0} - \bar{y}_{1}^{\epsilon}(0) \rVert_{\mathcal{C}^{-z}}
\end{align} 
by Lemmas \ref{Lemma 3.7} and \ref{Lemma 3.10}. Third, 
\begin{align*}
& t^{\frac{ \frac{1}{2} - \delta_{0} + z + \kappa}{2}} \int_{0}^{t} \lVert (P_{t-s} \bar{\phi}_{u}^{\epsilon, \sharp, i}, P_{t-s} \bar{\phi}_{b}^{\epsilon, \sharp, i}) (s) \rVert_{\mathcal{C}^{\frac{1}{2} - \delta_{0}}} ds \\
\lesssim& [\sup_{s \in [0, T_{0}]} s^{\delta + z + \kappa} \lVert ( \bar{\phi}_{u}^{\epsilon, \sharp, i}, \bar{\phi}_{b}^{\epsilon, \sharp, i})(s) \rVert_{\mathcal{C}^{-1-2\delta - 2 \kappa}} ]t^{\frac{ \frac{1}{2} - \delta_{0} + z + \kappa}{2}} t^{\frac{1}{4} + \frac{\delta_{0}}{2} - 2 \delta - 2 \kappa - z} \nonumber 
\end{align*} 
by Lemma \ref{Lemma 3.10}, \eqref{[Equation (3.3c)][ZZ17]} and \eqref{[Equation (3.2g)][ZZ17]}. Therefore, 
\begin{align}\label{estimate 101}
t^{\frac{ \frac{1}{2} - \delta_{0} + z + \kappa}{2}}\int_{0}^{t} \lVert ( P_{t-s} \bar{\phi}_{u}^{\epsilon, \sharp, i}, P_{t-s} \bar{\phi}_{b}^{\epsilon, \sharp, i})(s) \rVert_{\mathcal{C}^{\frac{1}{2} - \delta_{0}}} ds \lesssim_{T_{0}, \bar{C}_{W}^{\epsilon}(T_{0}), \lVert y_{0} \rVert_{\mathcal{C}^{-z}}} 1 
\end{align} 
by \eqref{[Equation (3.11)][ZZ17]} and \eqref{[Equation (3.3c)][ZZ17]}. Lastly, we rely on \eqref{[Equation (3.17b)][ZZ17]}, \eqref{[Equation (3.3c)][ZZ17]},  \eqref{[Equation (3.2g)][ZZ17]}, \eqref{[Equation (3.12)][ZZ17]}, \eqref{[Equation (3.13)][ZZ17]}  and \eqref{[Equation (3.2h)][ZZ17]} to deduce 
\begin{align}
& t^{ \frac{ \frac{1}{2} - \delta_{0} + z + \kappa}{2}} \sum_{i=1}^{3} \lVert ( \bar{F}_{u}^{\epsilon, i}, \bar{F}_{b}^{\epsilon, i} )(t) \rVert_{\mathcal{C}^{\frac{1}{2} - \delta_{0}}} \nonumber\\
\lesssim_{\bar{C}_{W}^{\epsilon}(t)}& t^{\frac{\frac{1}{2} - \delta_{0} + z + \kappa}{2}} [ t^{\frac{ \delta_{0} - z - \frac{\delta}{2}}{2}} \lVert y_{0} - \bar{y}_{1}^{\epsilon}(0) \rVert_{\mathcal{C}^{-z}} + t^{\frac{\delta_{0}}{2} - \frac{\delta}{2} - \frac{3\kappa}{4} + \frac{1}{4}} \nonumber\\
& \hspace{5mm} + \int_{0}^{t} (t-r)^{\frac{\delta_{0}}{2} - \frac{\delta}{2} - \frac{\kappa}{4} - \frac{3}{4}}[ \lVert \bar{y}^{\epsilon, \sharp}(r) \rVert_{\mathcal{C}^{\frac{1}{2} + \beta}} + \lVert \bar{y}^{\epsilon, \sharp}(r)\rVert_{\mathcal{C}^{\delta}}^{2}] dr \nonumber \\
& \hspace{5mm} + t^{\frac{\kappa}{2}} \int_{0}^{t} (t-r)^{\frac{\delta_{0}}{2} - \frac{\delta}{2} - \frac{3\kappa}{4} - \frac{3}{4}} [ \lVert \bar{y}^{\epsilon,\sharp}(r) \rVert_{\mathcal{C}^{\frac{1}{2} + \beta}} + \lVert \bar{y}^{\epsilon,\sharp}(r) \rVert_{\mathcal{C}^{\delta}}^{2}] dr \nonumber\\
& \hspace{5mm} + t^{\frac{1-\delta}{4}}(t^{\frac{\delta}{4}} \bar{C}_{W}^{\epsilon}(t) + \lVert \bar{y}_{4}^{\epsilon}(t) \rVert_{\mathcal{C}^{\frac{1}{2} - \delta_{0}}})]. 
\end{align}
Relying on \eqref{[Equation (3.2g)][ZZ17]}, \eqref{[Equation (3.3c)][ZZ17]} and \eqref{[Equation (3.11)][ZZ17]} leads us to 
\begin{align}\label{estimate 102}
& t^{\frac{ \frac{1}{2} - \delta_{0} + z + \kappa}{2}} \sum_{i=1}^{3} \lVert (\bar{F}_{u}^{\epsilon, i}, \bar{F}_{b}^{\epsilon, i})(t) \rVert_{\mathcal{C}^{\frac{1}{2} - \delta_{0}}} \\
& \hspace{25mm} \lesssim_{\bar{C}_{W}^{\epsilon}(t)} C(T_{0}, \lVert y_{0} \rVert_{\mathcal{C}^{-z}}) + t^{\frac{ \frac{1}{2} - \delta_{0} + z + \kappa}{2}} t^{\frac{1-\delta}{4}} \lVert \bar{y}_{4}^{\epsilon}(t) \rVert_{\mathcal{C}^{\frac{1}{2} - \delta_{0}}}. \nonumber
\end{align} 
Therefore, by applying \eqref{estimate 100}-\eqref{estimate 101} and \eqref{estimate 102} to \eqref{first estimate} , for $T_{0}> 0$ sufficiently small we deduce \eqref{[Equation (3.17a)][ZZ17]} so that the proof of Proposition \ref{[Lemma 3.10][ZZ17]} is complete. 
\end{proof}

Next, we estimate 
\begin{align*}
& \sum_{i=1}^{3} \lVert ( \bar{u}_{4}^{\epsilon,i}, \bar{b}_{4}^{\epsilon, i})(t) \rVert_{\mathcal{C}^{-z}}\\
\leq& \frac{1}{2} \sum_{i_{1} =1}^{3} \lVert (\bar{u}_{4}^{\epsilon, i_{1}}, \bar{b}_{4}^{\epsilon, i_{1}}) (t) \rVert_{\mathcal{C}^{-z}} + C(T_{0}, \bar{C}_{W}^{\epsilon}(T_{0})) + C \sum_{i=1}^{3} \lVert (\bar{u}^{\epsilon, \sharp, i}, \bar{b}^{\epsilon,\sharp, i} )(t) \rVert_{\mathcal{C}^{-z}} 
\end{align*}
by \eqref{paracontrolled ansatz u}, \eqref{paracontrolled ansatz b}, Lemma \ref{[Lemma 1.1][Y19a]} (2),  \eqref{[Equation (3.2g)][ZZ17]}, \eqref{[Equation (3.2h)][ZZ17]}, \eqref{[Equation (3.3)][ZZ17]} and taking $T_{0} > 0$ smaller if necessary. Thus, subtracting $\frac{1}{2} \sum_{i_{1} =1}^{3} \lVert (\bar{u}_{4}^{\epsilon, i_{1}}, \bar{b}_{4}^{\epsilon, i_{1}}) \rVert_{\mathcal{C}^{-z}}$ from both sides gives 
\begin{equation}\label{estimate 5}
\sum_{i=1}^{3} \lVert ( \bar{u}_{4}^{\epsilon, i}, \bar{b}_{4}^{\epsilon, i}) \rVert_{\mathcal{C}^{-z}} \lesssim C(\bar{C}_{W}^{\epsilon}(T_{0}), T_{0}, \lVert y_{0} \rVert_{\mathcal{C}^{-z}}) + \sum_{i=1}^{3} \lVert (\bar{F}_{u}^{\epsilon, i}, \bar{F}_{b}^{\epsilon, i}) \rVert_{\mathcal{C}^{-z}} 
\end{equation} 
by \eqref{[Equation (3.5) for u][ZZ17]}, \eqref{[Equation (3.5) for b][ZZ17]}, Lemmas \ref{Lemma 3.7} and \ref{Lemma 3.10}, \eqref{[Equation (3.17)][ZZ17]} and \eqref{[Equation (3.3c)][ZZ17]}. Now in order to estimate $\sum_{i=1}^{3} \lVert (\bar{F}_{u}^{\epsilon, i}, \bar{F}_{b}^{\epsilon, i}) \rVert_{\mathcal{C}^{-z}}$, we compute from \eqref{[Equation (3.17e)][ZZ17]}
\begin{align}\label{estimate 103}
& \lVert ( \bar{F}_{u}^{\epsilon}, \bar{F}_{b}^{\epsilon})(t) \rVert_{\mathcal{C}^{-z}} \lesssim_{T_{0}, \bar{C}_{W}^{\epsilon}(T_{0}), \lVert y_{0} \rVert_{\mathcal{C}^{-z}}} 1 + \lVert (\bar{y}_{3}^{\epsilon}  + \bar{y}_{4}^{\epsilon})(t)\rVert_{\mathcal{C}^{-z}} t^{\frac{1-\delta}{4}}
\end{align}
by \eqref{[Equation (3.13)][ZZ17]}, \eqref{[Equation (3.2g)][ZZ17]}, and \eqref{[Equation (3.3c)][ZZ17]}. We apply \eqref{estimate 103} to the estimate \eqref{estimate 5} to obtain 
\begin{equation}\label{[Equation (3.18)][ZZ17]}
\sup_{t \in [0, T_{0}]} \sum_{i=1}^{3} \lVert ( \bar{u}_{4}^{\epsilon, i}, \bar{b}_{4}^{\epsilon, i})(t) \rVert_{\mathcal{C}^{-z}} \lesssim_{T_{0}, \bar{C}_{W}^{\epsilon}(T_{0}), \lVert y_{0} \rVert_{\mathcal{C}^{-z}}} 1 
\end{equation} 
by \eqref{[Equation (3.2g)][ZZ17]}, \eqref{[Equation (3.2h)][ZZ17]} and then taking $T_{0} > 0$ smaller if necessary. 
Now based on \eqref{[Equation (3.2e)][ZZ17]}, we define 
\begin{align}
\bar{\mathbb{Z}}(W^{\epsilon}) \triangleq& ( \bar{u}_{1}^{\epsilon}, \bar{b}_{1}^{\epsilon}, \bar{u}_{1}^{\epsilon} \diamond \bar{u}_{1}^{\epsilon}, \bar{b}_{1}^{\epsilon} \diamond \bar{b}_{1}^{\epsilon}, \bar{u}_{1}^{\epsilon} \diamond \bar{b}_{1}^{\epsilon}, \bar{b}_{1}^{\epsilon} \diamond \bar{u}_{1}^{\epsilon}, \bar{u}_{1}^{\epsilon} \diamond \bar{u}_{2}^{\epsilon}, \bar{b}_{1}^{\epsilon} \diamond \bar{b}_{2}^{\epsilon}, \bar{b}_{1}^{\epsilon} \diamond \bar{u}_{2}^{\epsilon}, \bar{b}_{2}^{\epsilon} \diamond \bar{u}_{1}^{\epsilon}, \nonumber\\
& \bar{u}_{2}^{\epsilon} \diamond \bar{u}_{2}^{\epsilon}, \bar{b}_{2}^{\epsilon} \diamond \bar{b}_{2}^{\epsilon}, \bar{b}_{2}^{\epsilon} \diamond \bar{u}_{2}^{\epsilon}, \pi_{0,\diamond} (\bar{u}_{3}^{\epsilon}, \bar{u}_{1}^{\epsilon}), \pi_{0,\diamond} (\bar{b}_{3}^{\epsilon}, \bar{b}_{1}^{\epsilon}), \pi_{0,\diamond}(\bar{u}_{3}^{\epsilon}, \bar{b}_{1}^{\epsilon}), \pi_{0,\diamond} (\bar{b}_{3}^{\epsilon}, \bar{u}_{1}^{\epsilon}), \nonumber\\
& \pi_{0,\diamond} (\mathcal{P} D \bar{K}_{u}^{\epsilon}, \bar{u}_{1}^{\epsilon}),  \pi_{0,\diamond} (\mathcal{P} D \bar{K}_{b}^{\epsilon}, \bar{u}_{1}^{\epsilon}),  \pi_{0,\diamond} (\mathcal{P} D \bar{K}_{u}^{\epsilon}, \bar{b}_{1}^{\epsilon}),  \pi_{0,\diamond} (\mathcal{P} D \bar{K}_{b}^{\epsilon}, \bar{b}_{1}^{\epsilon})) \nonumber\\
\in& \mathbb{X} \triangleq C([0,T_{0}]; \mathcal{C}^{-\frac{1}{2} - \frac{\delta}{2}})^{2} \times C([0,T]; \mathcal{C}^{-1- \frac{\delta}{2}})^{4} \nonumber\\
& \hspace{20mm} \times C([0,T]; \mathcal{C}^{-\frac{1}{2} - \frac{\delta}{2}})^{4} \times C([0,T_{0}]; \mathcal{C}^{-\delta})^{11},  
\end{align} 
equipped with product topology. Then similarly, we can show that for all $a > 0$, there exists $T_{0} > 0$ sufficiently small such that the mapping $(y_{0}, \bar{\mathbb{Z}}(W^{\epsilon})) \mapsto \bar{y}_{4}^{\epsilon}$ is Lipschitz in a norm $C([0,T_{0}]; \mathcal{C}^{-z})$ on the set $\{(y_{0}, \bar{\mathbb{Z}}(W^{\epsilon})): \max\{ \lVert y_{0} \rVert_{\mathcal{C}^{-z}}, \bar{C}_{W}^{\epsilon}(T_{0}) \} \leq a \}$. This implies that $\bar{y}_{4}^{\epsilon}$ on $[0, T_{0}]$ depends in a locally Lipschitz continuous way on $(y_{0}, \bar{\mathbb{Z}}(W^{\epsilon}))$. Similarly to \cite[Theorem 3.2]{Y19a}, we can prove the existence of $\gamma > 0$, $(\bar{u}_{1}, \bar{b}_{1}) \in C([0, T_{0}]; \mathcal{C}^{-\frac{1}{2} - \frac{\delta}{2}})^{2}$, $(\bar{u}_{2},\bar{b}_{2}) \in C([0,T_{0}]; \mathcal{C}^{-\delta})^{2}$, $(\bar{u}_{3}, \bar{b}_{3}) \in C([0,T_{0}]; \mathcal{C}^{\frac{1}{2} - \delta})^{2}$ such that for all $p \geq 1$, 
\begin{subequations}
\begin{align}
& \mathbb{E} [ \lVert ( \bar{u}_{1}^{\epsilon}, \bar{b}_{1}^{\epsilon}) - (\bar{u}_{1}, \bar{b}_{1}) \rVert_{C([0,T_{0}]; \mathcal{C}^{-\frac{1}{2} - \frac{\delta}{2}})}^{p}] \lesssim \epsilon^{\gamma p}, \\
& \mathbb{E} [ \lVert (\bar{u}_{2}^{\epsilon}, \bar{b}_{2}^{\epsilon}) - (\bar{u}_{2}, \bar{b}_{2}) \rVert_{C([0, T_{0}]; \mathcal{C}^{ - \delta})}] \lesssim \epsilon^{\gamma p}, \\
&  \mathbb{E} [ \lVert (\bar{u}_{3}^{\epsilon}, \bar{b}_{3}^{\epsilon}) - (\bar{u}_{3}, \bar{b}_{3}) \rVert_{C([0, T_{0}]; \mathcal{C}^{\frac{1}{2} - \delta})}] \lesssim \epsilon^{\gamma p}.
\end{align}
\end{subequations} 
Using these bounds, letting $\epsilon_{k} \triangleq 2^{-k}$, we can prove that for all $\lambda > 0$, 
\begin{equation}\label{[Equation (3.19)][ZZ17]}
\sum_{k=1}^{\infty} \mathbb{P} ( \{ \lVert ( \bar{u}_{1}^{\epsilon_{k}}, \bar{b}_{1}^{\epsilon_{k}}) - (\bar{u}_{1}, \bar{b}_{1}) \rVert_{C([0, T_{0}]; \mathcal{C}^{-\frac{1}{2} - \frac{\delta}{2}})} > \lambda \}) \lesssim  \frac{1}{\lambda} \sum_{k=1}^{\infty} 2^{-\kappa \gamma}  \lesssim_{\lambda} 1 
\end{equation} 
by Chebyshev's inequality. Due to Borel-Cantelli Lemma, we conclude that for all $i \in \{1,2,3\},$ $(\bar{u}_{1}^{\epsilon_{k}, i}, \bar{b}_{1}^{\epsilon_{k}, i}) \to (\bar{u}_{1}^{i}, \bar{b}_{1}^{i})$ in $C([0,T_{0}]; \mathcal{C}^{-\frac{1}{2} - \frac{\delta}{2}})^{2}$ $\mathbb{P}$-a.s. as $k\to\infty$ and similarly $(\bar{u}_{2}^{\epsilon_{k}, i}, \bar{b}_{2}^{\epsilon_{k}, i}) \to (\bar{u}_{2}^{i}, \bar{b}_{2}^{i})$, $(\bar{u}_{3}^{\epsilon_{k}, i}, \bar{b}_{3}^{\epsilon_{k}, i}) \to (\bar{u}_{3}^{i}, \bar{b}_{3}^{i})$ in $C([0,T_{0}]; \mathcal{C}^{-\delta})^{2}$ and $C([0,T_{0}]; \mathcal{C}^{\frac{1}{2} - \delta})^{2}$ $\mathbb{P}$-a.s. as $k\to\infty$, respectively. Hence, we have shown that $\sup_{k\in \mathbb{N}: \epsilon_{k} = 2^{-k}} \bar{C}_{W}^{\epsilon_{k}}(T_{0}) < \infty$ $\mathbb{P}$-a.s., $(\bar{u}_{4}, \bar{b}_{4}) = \lim_{k\to\infty} (\bar{u}_{4}^{\epsilon_{k}}, \bar{b}_{4}^{\epsilon_{k}})$ in $[0, T_{0}]$, $y = (u,b) = (\sum_{i=1}^{4} \bar{u}_{i}, \sum_{i=1}^{4} \bar{b}_{i})$ as the solution to \eqref{MHD} on $[0,T_{0}]$ where $T_{0} > 0$ is independent of $\epsilon_{k} > 0$ and 
\begin{equation*}
\sup_{t\in [0,T_{0}]} \lVert( \bar{u}^{\epsilon_{k}}, \bar{b}^{\epsilon_{k}})(t) - (u, b)(t) \rVert_{\mathcal{C}^{-z}} \to 0
\end{equation*} 
$\mathbb{P}$-a.s., as $k\to\infty$ where we recall that $(\bar{u}^{\epsilon_{k}}, \bar{b}^{\epsilon_{k}}) = (\sum_{i=1}^{4} \bar{u}_{i}^{\epsilon_{k}}, \sum_{i=1}^{4} \bar{b}_{i}^{\epsilon_{k}})$. We can now define 
\begin{equation}\label{estimate 104}
\tau_{L} \triangleq \inf\{t: \hspace{1mm} \lVert (u,b) (t) \rVert_{\mathcal{C}^{-z}} \geq L \} \wedge L \text{ and } \bar{\rho}_{L}^{\epsilon} \triangleq \inf\{t: \hspace{1mm} \bar{C}_{W}^{\epsilon}(t) \geq L \} 
\end{equation} 
for some $L \geq 0$ and show that for $L_{1} \geq 0$, 
\begin{equation}\label{[Equation (3.21a)][ZZ17]}
\sup_{t\in [0, \tau_{L} \wedge \bar{\rho}_{L_{1}}^{\epsilon} ]} \lVert (\bar{y}^{\epsilon} - y )(t) \rVert_{\mathcal{C}^{-z}} \to 0 
\end{equation} 
in probability similarly to \eqref{[Equation (3.11)][ZZ17]} using \eqref{[Equation (3.12)][ZZ17]}-\eqref{[Equation (3.13)][ZZ17]}. Then the proof that we can extend such a solution to the maximal solution that satisfies $\sup_{t\in [0, \tau)} \lVert y(t) \rVert_{\mathcal{C}^{-z}} = + \infty$ for some explosion time $\tau$ is standard and may be shown identically as \cite[Proof of Theorem 3.12]{ZZ15}; thus, we omit it here. 

\subsubsection{Construction of solution to \eqref{MHD approximation}}

Analogously to \eqref{[Equation (3.1aa)][ZZ17]}-\eqref{[Equation (3.2)][ZZ17]}, we write $y_{1}^{\epsilon} \triangleq (u_{1}^{\epsilon}, b_{1}^{\epsilon})$, $y_{2}^{\epsilon} \triangleq (u_{2}^{\epsilon}, b_{2}^{\epsilon})$, $y_{3}^{\epsilon} \triangleq (u_{3}^{\epsilon}, b_{3}^{\epsilon})$, $y_{4}^{\epsilon} \triangleq (u_{4}^{\epsilon}, b_{4}^{\epsilon})$ and split \eqref{MHD approximation} as follows: for $i \in \{1,2,3\}$, 
\begin{equation}\label{[Equation (3.21c)][ZZ17]}
du_{1}^{\epsilon, i} = \Delta_{\epsilon} u_{1}^{\epsilon, i} dt + \sum_{i_{1} =1}^{3} \mathcal{P}^{ii_{1}} H_{u,\epsilon} dW_{u}^{i_{1}}, \hspace{1mm} 
db_{1}^{\epsilon, i} = \Delta_{\epsilon} b_{1}^{\epsilon, i} dt + \sum_{i_{1} =1}^{3} \mathcal{P}^{ii_{1}} H_{b,\epsilon} dW_{b}^{i_{1}}, 
\end{equation} 
\begin{subequations}\label{[Equation (3.21d)][ZZ17]}
\begin{align}
du_{2}^{\epsilon, i} =& \Delta_{\epsilon} u_{2}^{\epsilon, i} dt - \frac{1}{2} \sum_{i_{1}, j=1}^{3} \mathcal{P}^{ii_{1}} D_{j}^{\epsilon} (u_{1}^{\epsilon, i_{1}} \diamond u_{1}^{\epsilon, j} - b_{1}^{\epsilon, i_{1}} \diamond b_{1}^{\epsilon, j}) dt, \hspace{3mm} u_{2}^{\epsilon} (0, \cdot) \equiv 0, \\
d b_{2}^{\epsilon, i} =& \Delta_{\epsilon} b_{2}^{\epsilon, i} dt - \frac{1}{2} \sum_{i_{1}, j=1}^{3} \mathcal{P}^{ii_{1}} D_{j}^{\epsilon} (b_{1}^{\epsilon, i_{1}} \diamond u_{1}^{\epsilon, j} - u_{1}^{\epsilon, i_{1}} \diamond b_{1}^{\epsilon, j}) dt, \hspace{3mm} b_{2}^{\epsilon} (0, \cdot) \equiv 0, 
\end{align}
\end{subequations} 
\begin{subequations}\label{[Equation (3.21e)][ZZ17]}
\begin{align}
d u_{3}^{\epsilon, i} =& \Delta_{\epsilon} u_{3}^{\epsilon, i} dt - \frac{1}{2} \sum_{i_{1}, j=1}^{3} \mathcal{P}^{ii_{1}} D_{j}^{\epsilon} (u_{1}^{\epsilon, i_{1}} \diamond u_{2}^{\epsilon, j} + u_{2}^{\epsilon, i_{1}} \diamond u_{1}^{\epsilon, j} \nonumber\\
& \hspace{36mm} - b_{1}^{\epsilon, i_{1}} \diamond b_{2}^{\epsilon, j} - b_{2}^{\epsilon, i_{1}} \diamond b_{1}^{\epsilon, j}) dt, \hspace{3mm} u_{3}^{\epsilon} (0, \cdot) \equiv 0, \\
d b_{3}^{\epsilon, i} =& \Delta_{\epsilon} b_{3}^{\epsilon, i} dt - \frac{1}{2} \sum_{i_{1}, j=1}^{3} \mathcal{P}^{ii_{1}} D_{j}^{\epsilon} (b_{1}^{\epsilon, i_{1}} \diamond u_{2}^{\epsilon, j} + b_{2}^{\epsilon, i_{1}} \diamond u_{1}^{\epsilon, j} \nonumber\\
& \hspace{36mm} - u_{1}^{\epsilon, i_{1}} \diamond b_{2}^{\epsilon, j} - u_{2}^{\epsilon, i_{1}} \diamond b_{1}^{\epsilon, j}) dt, \hspace{3mm} b_{3}^{\epsilon} (0, \cdot) \equiv 0,
\end{align}
\end{subequations} 
\begin{subequations}\label{[Equation (3.22)][ZZ17]}
\begin{align}
u_{4}^{\epsilon, i}(t) =& P_{t}^{\epsilon} ( \sum_{i_{1} =1}^{3} \mathcal{P}^{ii_{1}} u_{0}^{i_{1}} - u_{1}^{\epsilon, i}(0)) \\
& - \frac{1}{2} \int_{0}^{t} P_{t-s}^{\epsilon} \sum_{i_{1}, j=1}^{3} \mathcal{P}^{ii_{1}} \nonumber\\
& \times D_{j}^{\epsilon} [ u_{1}^{\epsilon, i_{1}} \diamond (u_{3}^{\epsilon, j} + u_{4}^{\epsilon, j}) + (u_{3}^{\epsilon, i_{1}} + u_{4}^{\epsilon, i_{1}}) \diamond u_{1}^{\epsilon, j} + u_{2}^{\epsilon, i_{1}} \diamond u_{2}^{\epsilon, j} \nonumber\\
&\hspace{5mm} + u_{2}^{\epsilon, i_{1}}( u_{3}^{\epsilon, j} + u_{4}^{\epsilon, j}) + u_{2}^{\epsilon, j} ( u_{3}^{\epsilon, i_{1}} + u_{4}^{\epsilon, i_{1}}) + (u_{3}^{\epsilon, i_{1}} + u_{4}^{\epsilon, i_{1}})(u_{3}^{\epsilon, j} + u_{4}^{\epsilon, j})\nonumber \\
&\hspace{5mm} - b_{1}^{\epsilon, i_{1}} \diamond (b_{3}^{\epsilon, j} + b_{4}^{\epsilon, j}) - (b_{3}^{\epsilon, i_{1}} + b_{4}^{\epsilon, i_{1}}) \diamond b_{1}^{\epsilon, j} - b_{2}^{\epsilon, i_{1}} \diamond b_{2}^{\epsilon, j} \nonumber\\
&\hspace{5mm} - b_{2}^{\epsilon, i_{1}}( b_{3}^{\epsilon, j} + b_{4}^{\epsilon, j}) - b_{2}^{\epsilon, j} ( b_{3}^{\epsilon, i_{1}} + b_{4}^{\epsilon, i_{1}}) - (b_{3}^{\epsilon, i_{1}} + b_{4}^{\epsilon, i_{1}})(b_{3}^{\epsilon, j} + b_{4}^{\epsilon, j})]ds, \nonumber\\
b_{4}^{\epsilon, i}(t) =& P_{t}^{\epsilon} ( \sum_{i_{1} =1}^{3} \mathcal{P}^{ii_{1}} b_{0}^{i_{1}} - b_{1}^{\epsilon, i}(0)) \\
& - \frac{1}{2} \int_{0}^{t} P_{t-s}^{\epsilon} \sum_{i_{1}, j=1}^{3} \mathcal{P}^{ii_{1}} \nonumber\\
& \times D_{j}^{\epsilon} [ b_{1}^{\epsilon, i_{1}} \diamond (u_{3}^{\epsilon, j} + u_{4}^{\epsilon, j}) + (b_{3}^{\epsilon, i_{1}} + b_{4}^{\epsilon, i_{1}}) \diamond u_{1}^{\epsilon, j} + b_{2}^{\epsilon, i_{1}} \diamond u_{2}^{\epsilon, j} \nonumber\\
&\hspace{5mm} + b_{2}^{\epsilon, i_{1}}( u_{3}^{\epsilon, j} + u_{4}^{\epsilon, j}) + u_{2}^{\epsilon, j} ( b_{3}^{\epsilon, i_{1}} + b_{4}^{\epsilon, i_{1}}) + (b_{3}^{\epsilon, i_{1}} + b_{4}^{\epsilon, i_{1}})(u_{3}^{\epsilon, j} + u_{4}^{\epsilon, j})\nonumber \\
&\hspace{5mm} - u_{1}^{\epsilon, i_{1}} \diamond (b_{3}^{\epsilon, j} + b_{4}^{\epsilon, j}) - (u_{3}^{\epsilon, i_{1}} + u_{4}^{\epsilon, i_{1}}) \diamond b_{1}^{\epsilon, j} - u_{2}^{\epsilon, i_{1}} \diamond b_{2}^{\epsilon, j} \nonumber\\
&\hspace{5mm} - u_{2}^{\epsilon, i_{1}}( b_{3}^{\epsilon, j} + b_{4}^{\epsilon, j}) - b_{2}^{\epsilon, j} ( u_{3}^{\epsilon, i_{1}} + u_{4}^{\epsilon, i_{1}}) - (u_{3}^{\epsilon, i_{1}} + \bar{u}_{4}^{\epsilon, i_{1}})(b_{3}^{\epsilon, j} + b_{4}^{\epsilon, j})]ds, \nonumber
\end{align}
\end{subequations} 
where we recall from Example \ref{Burgers' equation example} that $P_{t}^{\epsilon} = e^{t \Delta_{\epsilon}}$, and 
\begin{subequations}\label{[Equation (3.22aa)][ZZ17]}
\begin{align}
& u_{1}^{\epsilon, i} \diamond u_{1}^{\epsilon, j} \triangleq u_{1}^{\epsilon, i} u_{1}^{\epsilon, j} - C_{01}^{\epsilon, ij}, \hspace{3mm} b_{1}^{\epsilon, i} \diamond b_{1}^{\epsilon, j} \triangleq b_{1}^{\epsilon, i} b_{1}^{\epsilon, j} - C_{02}^{\epsilon, ij}, \\
& u_{1}^{\epsilon, i} \diamond b_{1}^{\epsilon, j} \triangleq u_{1}^{\epsilon, i} b_{1}^{\epsilon, j} - C_{03}^{\epsilon, ij}, \hspace{3mm} b_{1}^{\epsilon, i} \diamond u_{1}^{\epsilon, j} \triangleq b_{1}^{\epsilon, i} u_{1}^{\epsilon, j} - C_{03}^{\epsilon, ij}, 
\end{align}
\end{subequations} 
with $C_{01}^{\epsilon, ij}, C_{02}^{\epsilon, ij}, C_{03}^{\epsilon, ij}$ explicitly found in \eqref{estimate 40}, \eqref{estimate 49}, 
\begin{subequations}\label{[Equation (3.22ab)][ZZ17]}
\begin{align}
u_{2}^{\epsilon, i} \diamond u_{1}^{\epsilon, j} &\triangleq u_{1}^{\epsilon, j} \diamond u_{2}^{\epsilon, i} \nonumber \\
\triangleq& u_{1}^{\epsilon, j} u_{2}^{\epsilon, i} + \sum_{i_{1} =1}^{3} (C_{1,u}^{\epsilon, ii_{1} j}+ \tilde{C}_{1,u}^{\epsilon, i i_{1} j}) u_{1}^{\epsilon, i_{1}} + (C_{1,b}^{\epsilon, i i_{1} j} + \tilde{C}_{1,b}^{\epsilon, i i_{1} j}) b_{1}^{\epsilon, i_{1}}, \\
b_{2}^{\epsilon, i} \diamond b_{1}^{\epsilon, j} \triangleq& b_{1}^{\epsilon, j} \diamond b_{2}^{\epsilon, i} \nonumber \\
\triangleq & b_{1}^{\epsilon, j} b_{2}^{\epsilon, i} + \sum_{i_{1} =1}^{3} (C_{2,u}^{\epsilon, i i_{1} j}+ \tilde{C}_{2,u}^{\epsilon, i i_{1} j}) u_{1}^{\epsilon, i_{1}} + (C_{2,b}^{\epsilon, i i_{1} j} + \tilde{C}_{2,b}^{\epsilon, i i_{1} j}) b_{1}^{\epsilon, i_{1}}, \\
 u_{2}^{\epsilon, i} \diamond b_{1}^{\epsilon, j} \triangleq & b_{1}^{\epsilon, j} \diamond u_{2}^{\epsilon, i} \nonumber \\
 \triangleq & b_{1}^{\epsilon, j} u_{2}^{\epsilon, i} + \sum_{i_{1} =1}^{3} (C_{3,u}^{\epsilon, i i_{1} j}+ \tilde{C}_{3,u}^{\epsilon, i i_{1} j}) u_{1}^{\epsilon, i_{1}} + (C_{3,b}^{\epsilon, i i_{1} j} + \tilde{C}_{3,b}^{\epsilon, i i_{1} j}) b_{1}^{\epsilon, i_{1}}, \\
b_{2}^{\epsilon, i} \diamond u_{1}^{\epsilon, j} \triangleq & u_{1}^{\epsilon, j} \diamond b_{2}^{\epsilon, i} \nonumber \\
\triangleq & u_{1}^{\epsilon, j} b_{2}^{\epsilon, i} + \sum_{i_{1} =1}^{3} (C_{4,u}^{\epsilon, i i_{1} j}+ \tilde{C}_{4,u}^{\epsilon, i i_{1} j}) u_{1}^{\epsilon, i_{1}} + (C_{4,b}^{\epsilon, i i_{1} j} + \tilde{C}_{4,b}^{\epsilon, i i_{1} j}) b_{1}^{\epsilon, i_{1}}, 
\end{align}
\end{subequations} 
with $C_{k,u}^{\epsilon, i i_{1} j}, C_{k,b}^{\epsilon, i i_{1} j}, \tilde{C}_{k,u}^{\epsilon, i i_{1} j}$ and $\tilde{C}_{k,b}^{\epsilon, i i_{1} j}$ explicitly found in \eqref{C2uepsilonii1j}, \eqref{C2bepsilonii1j}, \eqref{C2uepsilonii2j},  \eqref{C2bepsilonii2j} and \eqref{estimate 124},
\begin{subequations}\label{[Equation (3.22ac)][ZZ17]} 
\begin{align}
& u_{2}^{\epsilon, i} \diamond u_{2}^{\epsilon, j} \triangleq u_{2}^{\epsilon, i} u_{2}^{\epsilon, j} - \phi_{21}^{\epsilon, ij} - C_{21}^{\epsilon, ij}, \hspace{1mm} b_{2}^{\epsilon, i} \diamond b_{2}^{\epsilon, j} \triangleq b_{2}^{\epsilon, i} b_{2}^{\epsilon, j} - \phi_{22}^{\epsilon, ij} - C_{22}^{\epsilon, ij}, \\
& b_{2}^{\epsilon, i} \diamond u_{2}^{\epsilon, j} \triangleq b_{2}^{\epsilon, i} u_{2}^{\epsilon, j} - \phi_{23}^{\epsilon, ij} - C_{23}^{\epsilon, ij}, \hspace{2mm} u_{2}^{\epsilon, i} \diamond b_{2}^{\epsilon, j} \triangleq u_{2}^{\epsilon, i} b_{2}^{\epsilon, j} - \phi_{24}^{\epsilon, ij} - C_{24}^{\epsilon, ij}, 
\end{align}
\end{subequations} 
with $\phi_{22}^{\epsilon, ij}(t)$ explicitly found in \eqref{estimate 45}, $C_{22}^{\epsilon, ij}$ in \eqref{estimate 43}, 
\begin{subequations}\label{[Equation (3.22ad)][ZZ17]}
\begin{align}
 u_{1}^{\epsilon, i} \diamond u_{3}^{\epsilon, j} \triangleq& \pi_{<} (u_{1}^{\epsilon, i}, u_{3}^{\epsilon, j}) + \pi_{>} (u_{1}^{\epsilon, i}, u_{3}^{\epsilon, j}) + \pi_{0, \diamond} (u_{3}^{\epsilon, j}, u_{1}^{\epsilon, i}), \\
 b_{1}^{\epsilon, i} \diamond b_{3}^{\epsilon, j} \triangleq& \pi_{<} ( b_{1}^{\epsilon, i}, b_{3}^{\epsilon, j}) + \pi_{>} ( b_{1}^{\epsilon, i}, b_{3}^{\epsilon, j}) + \pi_{0,\diamond} (b_{3}^{\epsilon, j}, b_{1}^{\epsilon, i}), \\
  u_{1}^{\epsilon, i} \diamond b_{3}^{\epsilon, j} \triangleq& \pi_{<} ( u_{1}^{\epsilon, i}, b_{3}^{\epsilon, j}) + \pi_{>} ( u_{1}^{\epsilon, i}, b_{3}^{\epsilon, j}) + \pi_{0,\diamond} (b_{3}^{\epsilon, j}, u_{1}^{\epsilon, i}), \\
 b_{1}^{\epsilon, i} \diamond u_{3}^{\epsilon, j} \triangleq& \pi_{<} ( b_{1}^{\epsilon, i}, u_{3}^{\epsilon, j}) + \pi_{>} ( b_{1}^{\epsilon, i}, u_{3}^{\epsilon, j}) + \pi_{0,\diamond} (u_{3}^{\epsilon, j}, b_{1}^{\epsilon, i}), 
\end{align}
\end{subequations} 
with 
\begin{subequations}\label{[Equation (3.22af)][ZZ17]}
\begin{align}
\pi_{0,\diamond } (u_{3}^{\epsilon, i}, u_{1}^{\epsilon, j}) \triangleq & \pi_{0}( u_{3}^{\epsilon, i}, u_{1}^{\epsilon, j}) - \phi_{11}^{\epsilon, ij}  - C_{11}^{\epsilon, ij} \nonumber \\
&+ \sum_{i_{1} =1}^{3} (C_{1,u}^{\epsilon, i i_{1} j}+ \tilde{C}_{1,u}^{\epsilon, i i_{1} j}) u_{2}^{\epsilon, i_{1}} + (C_{1,b}^{\epsilon, i i_{1} j} + \tilde{C}_{1,b}^{\epsilon, i i_{1} j}) b_{2}^{\epsilon, i_{1}}, \\
\pi_{0,\diamond } (b_{3}^{\epsilon, i}, b_{1}^{\epsilon, j}) \triangleq & \pi_{0}( b_{3}^{\epsilon, i}, b_{1}^{\epsilon, j}) - \phi_{12}^{\epsilon, ij}  - C_{12}^{\epsilon, ij}  \nonumber \\
&+ \sum_{i_{1} =1}^{3} (C_{2,u}^{\epsilon, ii_{1} j}+ \tilde{C}_{2,u}^{\epsilon, i i_{1} j}) u_{2}^{\epsilon, i_{1}} + (C_{2,b}^{\epsilon, i i_{1} j} + \tilde{C}_{2,b}^{\epsilon, i i_{1} j}) b_{2}^{\epsilon, i_{1}}, \\
\pi_{0,\diamond } (u_{3}^{\epsilon, i}, b_{1}^{\epsilon, j}) \triangleq& \pi_{0}( u_{3}^{\epsilon, i}, b_{1}^{\epsilon, j}) - \phi_{13}^{\epsilon, ij}  - C_{13}^{\epsilon, ij}  \nonumber \\
&+ \sum_{i_{1} =1}^{3} (C_{3,u}^{\epsilon, ii_{1} j}+ \tilde{C}_{3,u}^{\epsilon, i i_{1} j}) u_{2}^{\epsilon, i_{1}} + (C_{3,b}^{\epsilon, i i_{1} j} + \tilde{C}_{3,b}^{\epsilon, i i_{1} j}) b_{2}^{\epsilon, i_{1}}, \\
\pi_{0,\diamond } (b_{3}^{\epsilon, i}, u_{1}^{\epsilon, j}) \triangleq& \pi_{0}( b_{3}^{\epsilon, i}, u_{1}^{\epsilon, j}) - \phi_{14}^{\epsilon, ij}  - C_{14}^{\epsilon, ij} +  \nonumber  \\
&\sum_{i_{1} =1}^{3} (C_{4,u}^{\epsilon, i i_{1}j}+ \tilde{C}_{4,u}^{\epsilon, i i_{1} j}) u_{2}^{\epsilon, i_{1}} + (C_{4,b}^{\epsilon, i i_{1} j} + \tilde{C}_{4,b}^{\epsilon, i i_{1} j}) b_{2}^{\epsilon, i_{1}}, 
\end{align}
\end{subequations} 
and $\phi_{13}^{\epsilon, ij}(t)$ explicitly found in \eqref{estimate 48}, $C_{13}^{\epsilon, ij}$ in \eqref{estimate 47} 
\begin{subequations}\label{[Equation (3.22ae)][ZZ17]}
\begin{align}
 u_{1}^{\epsilon, i} \diamond u_{4}^{\epsilon, j} \triangleq& \pi_{<} (u_{1}^{\epsilon, i}, u_{4}^{\epsilon, j}) + \pi_{>} (u_{1}^{\epsilon, i}, u_{4}^{\epsilon, j}) + \pi_{0,\diamond} (u_{4}^{\epsilon, j}, u_{1}^{\epsilon, i}), \\
 b_{1}^{\epsilon, i} \diamond b_{4}^{\epsilon, j} \triangleq& \pi_{<} (b_{1}^{\epsilon, i}, b_{4}^{\epsilon, j}) + \pi_{>} (b_{1}^{\epsilon, i}, b_{4}^{\epsilon, j}) + \pi_{0,\diamond} (b_{4}^{\epsilon, j}, b_{1}^{\epsilon, i}), \\
u_{1}^{\epsilon, i} \diamond b_{4}^{\epsilon, j} \triangleq& \pi_{<} (u_{1}^{\epsilon, i}, b_{4}^{\epsilon, j}) + \pi_{>} (u_{1}^{\epsilon, i}, b_{4}^{\epsilon, j}) + \pi_{0,\diamond} (b_{4}^{\epsilon, j}, u_{1}^{\epsilon, i}), \\
 b_{1}^{\epsilon, i} \diamond u_{4}^{\epsilon, j} \triangleq& \pi_{<} (b_{1}^{\epsilon, i}, u_{4}^{\epsilon, j}) + \pi_{>} (b_{1}^{\epsilon, i}, u_{4}^{\epsilon, j}) + \pi_{0,\diamond} (u_{4}^{\epsilon, j}, b_{1}^{\epsilon, i}), 
\end{align}
\end{subequations}
with 
\begin{subequations}\label{[Equation (3.22ag)][ZZ17]}
\begin{align} 
\pi_{0,\diamond} (u_{4}^{\epsilon, i}, u_{1}^{\epsilon, j}) \triangleq \pi_{0}(u_{4}^{\epsilon, i}, u_{1}^{\epsilon, j}) + &\sum_{i_{1} =1}^{3} (C_{1,u}^{\epsilon, i i_{1} j}+ \tilde{C}_{1,u}^{\epsilon, ii_{1} j}) (u_{3}^{\epsilon, i_{1}} + u_{4}^{\epsilon, i_{1}}) \nonumber \\
& + (C_{1,b}^{\epsilon, i i_{1} j} + \tilde{C}_{1,b}^{\epsilon, i i_{1} j}) (b_{3}^{\epsilon, i_{1}} + b_{4}^{\epsilon, i_{1}}),\\
\pi_{0,\diamond} (b_{4}^{\epsilon, i}, b_{1}^{\epsilon, j})  \triangleq \pi_{0}(b_{4}^{\epsilon, i}, b_{1}^{\epsilon, j}) + & \sum_{i_{1} =1}^{3} (C_{2,u}^{\epsilon, i i_{1} j}+ \tilde{C}_{2,u}^{\epsilon, i i_{1} j}) (u_{3}^{\epsilon, i_{1}} + u_{4}^{\epsilon, i_{1}}) \nonumber  \\
&+ (C_{2,b}^{\epsilon, i i_{1} j} + \tilde{C}_{2,b}^{\epsilon, i i_{1} j}) (b_{3}^{\epsilon, i_{1}} + b_{4}^{\epsilon, i_{1}}),\\
\pi_{0,\diamond} (u_{4}^{\epsilon, i}, b_{1}^{\epsilon, j}) \triangleq \pi_{0}(u_{4}^{\epsilon, i}, b_{1}^{\epsilon, j}) & + \sum_{i_{1} =1}^{3} (C_{3,u}^{\epsilon, i i_{1} j}+ \tilde{C}_{3,u}^{\epsilon, i i_{1} j}) (u_{3}^{\epsilon, i_{1}} + u_{4}^{\epsilon, i_{1}}) \nonumber  \\
&+ (C_{3,b}^{\epsilon, i i_{1} j} + \tilde{C}_{3,b}^{\epsilon, i i_{1} j}) (b_{3}^{\epsilon, i_{1}} + b_{4}^{\epsilon, i_{1}}), \\
\pi_{0,\diamond} (b_{4}^{\epsilon, i}, u_{1}^{\epsilon, j}) \triangleq \pi_{0}(b_{4}^{\epsilon, i}, u_{1}^{\epsilon, j}) & + \sum_{i_{1} =1}^{3} (C_{4,u}^{\epsilon, i i_{1} j}+ \tilde{C}_{4,u}^{\epsilon, i i_{1} j}) (u_{3}^{\epsilon, i_{1}} + u_{4}^{\epsilon, i_{1}}) \nonumber \\
&+ (C_{4,b}^{\epsilon, i i_{1} j} + \tilde{C}_{4,b}^{\epsilon, i i_{1} j}) (b_{3}^{\epsilon, i_{1}} + b_{4}^{\epsilon, i_{1}}), 
\end{align}
\end{subequations}
and $\phi_{lk}^{\epsilon, ij}$ for $l \in \{1,2\}, k \in \{1,2,3,4\}, i, j \in \{1,2,3\}$, converges to some $\phi_{lk}^{ij}$ with respect to the norm of $\sup_{t\in [0,T]} t^{\rho} \lvert \cdot (t) \rvert$ for every $\rho > 0$ (see \eqref{estimate 127}, \eqref{estimate 54}-\eqref{estimate 55}).  Now we define $K_{u}^{\epsilon}$ and $K_{b}^{\epsilon}$ to satisfy for all $i \in \{1,2,3\}$, 
\begin{subequations}\label{estimate 276}
\begin{align}
&d K_{u}^{\epsilon, i}  = (\Delta_{\epsilon} K_{u}^{\epsilon, i} + u_{1}^{\epsilon, i}) dt, \hspace{3mm} K_{u}^{\epsilon, i} (0) = 0,  \\
&d K_{b}^{\epsilon, i}  = (\Delta_{\epsilon} K_{b}^{\epsilon, i} + b_{1}^{\epsilon, i}) dt,  \hspace{3mm} K_{b}^{\epsilon, i} (0) = 0. 
\end{align}
\end{subequations}  
For every $\epsilon > 0$, by a similar argument to the case of \eqref{[Equation (3.2)][ZZ17]}, we obtain solution $(u_{4}^{\epsilon}, b_{4}^{\epsilon})$ of \eqref{[Equation (3.22)][ZZ17]} and then the solution $(u^{\epsilon}, b^{\epsilon}) = (\sum_{k=1}^{4} u_{k}^{\epsilon}, \sum_{k=1}^{4} b_{k}^{\epsilon})$ of \eqref{MHD approximation}. We only mention that we need to define analogously to \eqref{[Equation (3.2e)][ZZ17]}-\eqref{[Equation (3.2f)][ZZ17]}, 
\begin{align}\label{[Equation (3.23a)][ZZ17]}
C_{W}^{\epsilon}(T) \triangleq& \sup_{t\in [0,T]} [ \sum_{i=1}^{3} \lVert (u_{1}^{\epsilon, i}, b_{1}^{\epsilon, i})(t) \rVert_{\mathcal{C}^{- \frac{1}{2} - \frac{\delta}{2}}} \\
& + \sum_{i,j=1}^{3} \lVert ( u_{1}^{\epsilon, i} \diamond u_{1}^{\epsilon, j}, b_{1}^{\epsilon, i} \diamond b_{1}^{\epsilon, j}, u_{1}^{\epsilon, i} \diamond b_{1}^{\epsilon, j}, b_{1}^{\epsilon, i} \diamond u_{1}^{\epsilon, j})(t) \rVert_{\mathcal{C}^{-1 - \frac{\delta}{2}}} \nonumber\\
&+  \sum_{i,j=1}^{3} \lVert ( u_{1}^{\epsilon, i} \diamond u_{2}^{\epsilon, j}, b_{1}^{\epsilon, i} \diamond b_{2}^{\epsilon, j}, u_{2}^{\epsilon, i} \diamond b_{1}^{\epsilon, j}, b_{2}^{\epsilon, i} \diamond u_{1}^{\epsilon, j})(t) \rVert_{\mathcal{C}^{- \frac{1}{2} - \frac{\delta}{2}}} \nonumber\\
&+  \sum_{i,j=1}^{3} \lVert ( u_{2}^{\epsilon, i} \diamond u_{2}^{\epsilon, j}, b_{2}^{\epsilon, i} \diamond b_{2}^{\epsilon, j}, b_{2}^{\epsilon, i} \diamond u_{2}^{\epsilon, j})(t) \rVert_{\mathcal{C}^{-\delta}} \nonumber\\
&+  \sum_{i,j=1}^{3} \lVert (\pi_{0,\diamond} (u_{3}^{\epsilon, j}, u_{1}^{\epsilon, i}), \pi_{0,\diamond} (b_{3}^{\epsilon, j}, b_{1}^{\epsilon, i}), \pi_{0,\diamond}( u_{3}^{\epsilon, j}, b_{1}^{\epsilon, i}), \pi_{0,\diamond} (b_{3}^{\epsilon, j}, u_{1}^{\epsilon, i})) \rVert_{\mathcal{C}^{-\delta}}\nonumber\\
&+  \sum_{i, i_{1}, j, j_{1} =1}^{3} \lVert (\pi_{0,\diamond} (\mathcal{P}^{ii_{1}} D_{j}^{\epsilon} K_{u}^{\epsilon, j}, u_{1}^{\epsilon, j_{1}}), \pi_{0,\diamond}(\mathcal{P}^{ii_{1}} D_{j}^{\epsilon} K_{u}^{\epsilon, i_{1}}, u_{1}^{\epsilon, j_{1}}), \nonumber\\
& \hspace{25mm} \pi_{0,\diamond} (\mathcal{P}^{ii_{1}} D_{j}^{\epsilon} K_{b}^{\epsilon, j}, u_{1}^{\epsilon, j_{1}}), \pi_{0,\diamond} (\mathcal{P}^{ii_{1}} D_{j}^{\epsilon} K_{b}^{\epsilon, i_{1}}, u_{1}^{\epsilon, j_{1}}), \nonumber\\
& \hspace{25mm} \pi_{0,\diamond} (\mathcal{P}^{ii_{1}} D_{j}^{\epsilon} K_{u}^{\epsilon, j}, b_{1}^{\epsilon, j_{1}}), \pi_{0,\diamond}(\mathcal{P}^{ii_{1}} D_{j}^{\epsilon} K_{u}^{\epsilon, i_{1}}, b_{1}^{\epsilon, j_{1}}), \nonumber\\
& \hspace{25mm} \pi_{0,\diamond} (\mathcal{P}^{ii_{1}} D_{j}^{\epsilon} K_{b}^{\epsilon, j}, b_{1}^{\epsilon, j_{1}}), \pi_{0,\diamond} (\mathcal{P}^{ii_{1}} D_{j}^{\epsilon} K_{b}^{\epsilon, i_{1}}, b_{1}^{\epsilon, j_{1}})) \rVert_{\mathcal{C}^{-\delta}} ]\nonumber
\end{align} 
where 
\begin{subequations}\label{[Equation (3.23b)][ZZ17]}
\begin{align}
 \pi_{0,\diamond} ( \mathcal{P}^{ii_{1}} D_{j}^{\epsilon} K_{u}^{\epsilon, j}, u_{1}^{\epsilon, j_{1}}) \triangleq& \pi_{0} ( \mathcal{P}^{ii_{1}} D_{j}^{\epsilon} K_{u}^{\epsilon, j}, u_{1}^{\epsilon, j_{1}}) - C_{31}^{\epsilon, ii_{1} j j_{1}}, \\
\pi_{0,\diamond} ( \mathcal{P}^{ii_{1}} D_{j}^{\epsilon} K_{u}^{\epsilon, i_{1}}, u_{1}^{\epsilon, j_{1}}) \triangleq& \pi_{0} ( \mathcal{P}^{ii_{1}} D_{j}^{\epsilon} K_{u}^{\epsilon, i_{1}}, u_{1}^{\epsilon, j_{1}}) - \tilde{C}_{31}^{\epsilon, i i_{1} j j_{1}}, \\
 \pi_{0,\diamond} ( \mathcal{P}^{ii_{1}} D_{j}^{\epsilon} K_{b}^{\epsilon, j}, u_{1}^{\epsilon, j_{1}}) \triangleq& \pi_{0} ( \mathcal{P}^{ii_{1}} D_{j}^{\epsilon} K_{b}^{\epsilon, j}, u_{1}^{\epsilon, j_{1}}) -C_{32}^{\epsilon, i i_{1} j j_{1}}, \\
\pi_{0,\diamond} ( \mathcal{P}^{ii_{1}} D_{j}^{\epsilon} K_{b}^{\epsilon, i_{1}}, u_{1}^{\epsilon, j_{1}}) \triangleq& \pi_{0} ( \mathcal{P}^{ii_{1}} D_{j}^{\epsilon} K_{b}^{\epsilon, i_{1}}, u_{1}^{\epsilon, j_{1}}) - \tilde{C}_{32}^{\epsilon, i i_{1} j j_{1}}, \\
 \pi_{0,\diamond} ( \mathcal{P}^{ii_{1}} D_{j}^{\epsilon} K_{u}^{\epsilon, j}, b_{1}^{\epsilon, j_{1}}) \triangleq& \pi_{0} ( \mathcal{P}^{ii_{1}} D_{j}^{\epsilon} K_{u}^{\epsilon, j}, b_{1}^{\epsilon, j_{1}}) - C_{33}^{\epsilon, i i_{1} j j_{1}}, \\
\pi_{0,\diamond} ( \mathcal{P}^{ii_{1}} D_{j}^{\epsilon} K_{u}^{\epsilon, i_{1}}, b_{1}^{\epsilon, j_{1}}) \triangleq& \pi_{0} ( \mathcal{P}^{ii_{1}} D_{j}^{\epsilon} K_{u}^{\epsilon, i_{1}}, b_{1}^{\epsilon, j_{1}}) - \tilde{C}_{33}^{\epsilon, i i_{1} j j_{1}}, \\
 \pi_{0,\diamond} ( \mathcal{P}^{ii_{1}} D_{j}^{\epsilon} K_{b}^{\epsilon, j}, b_{1}^{\epsilon, j_{1}}) \triangleq& \pi_{0} ( \mathcal{P}^{ii_{1}} D_{j}^{\epsilon} K_{b}^{\epsilon, j}, b_{1}^{\epsilon, j_{1}}) - C_{34}^{\epsilon, i i_{1} j j_{1}}, \\
\pi_{0,\diamond} ( \mathcal{P}^{ii_{1}} D_{j}^{\epsilon} K_{b}^{\epsilon, i_{1}}, b_{1}^{\epsilon, j_{1}}) \triangleq& \pi_{0} ( \mathcal{P}^{ii_{1}} D_{j}^{\epsilon} K_{b}^{\epsilon, i_{1}}, b_{1}^{\epsilon, j_{1}}) - \tilde{C}_{34}^{\epsilon, i i_{1} j j_{1}}, 
\end{align}
\end{subequations}
with $C_{34}^{\epsilon, i i_{1} j j_{1}}$ explicitly found in \eqref{estimate 50},  and 
\begin{equation}\label{estimate 290}
\tau_{L}^{\epsilon} \triangleq \inf\{t: \hspace{1mm} \lVert (u^{\epsilon} ,b^{\epsilon}) (t) \rVert_{\mathcal{C}^{-z}} \geq L \} \wedge L \text{ and } \rho_{L}^{\epsilon} \triangleq \inf\{t: \hspace{1mm} C_{W}^{\epsilon}(t) \geq L \} 
\end{equation} 
for some $L \geq 0$ similarly to \eqref{estimate 104}. In the next Subsection \ref{Subsection 2.2}, the following convergence will be proven in detail: for any $\delta > 0$,  
\begin{subequations}
\begin{align}
& \bar{u}_{1}^{\epsilon, i} - u_{1}^{\epsilon, i}, \bar{b}_{1}^{\epsilon, i} - b_{1}^{\epsilon, i} \to 0 \text{ in } C([0, T]; \mathcal{C}^{- \frac{1}{2} - \frac{\delta}{2}}), \label{first convergence} \\
& \bar{u}_{1}^{\epsilon, i} \diamond \bar{u}_{1}^{\epsilon, j} - u_{1}^{\epsilon, i} \diamond u_{1}^{\epsilon, j}, \bar{b}_{1}^{\epsilon, i} \diamond \bar{b}_{1}^{\epsilon, j} - b_{1}^{\epsilon, i} \diamond b_{1}^{\epsilon, j},  \nonumber\\
& \hspace{10mm} \bar{u}_{1}^{\epsilon, i} \diamond \bar{b}_{1}^{\epsilon, j} - u_{1}^{\epsilon, i} \diamond b_{1}^{\epsilon, j}, \bar{b}_{1}^{\epsilon, i} \diamond \bar{u}_{1}^{\epsilon, j} - b_{1}^{\epsilon, i} \diamond u_{1}^{\epsilon, j} \to 0 \text{ in } C([0, T]; \mathcal{C}^{-1 - \frac{\delta}{2}}), \label{second convergence}\\
& \bar{u}_{1}^{\epsilon, i} \diamond \bar{u}_{2}^{\epsilon, j} - u_{1}^{\epsilon, i} \diamond u_{2}^{\epsilon, j}, \bar{b}_{1}^{\epsilon, i} \diamond \bar{b}_{2}^{\epsilon, j} - b_{1}^{\epsilon, i} \diamond b_{2}^{\epsilon, j}, \nonumber\\
& \hspace{10mm} \bar{u}_{2}^{\epsilon, i} \diamond  \bar{b}_{1}^{\epsilon, j} -  u_{2}^{\epsilon, i} \diamond b_{1}^{\epsilon, j}, \bar{u}_{1}^{\epsilon, i} \diamond \bar{b}_{2}^{\epsilon, j} - u_{1}^{\epsilon, i} \diamond b_{2}^{\epsilon, j} \to 0  \text{ in } C([0, T]; \mathcal{C}^{-\frac{1}{2} - \frac{\delta}{2}}),  \label{third convergence} \\
& \bar{u}_{2}^{\epsilon, i} \diamond \bar{u}_{2}^{\epsilon, j} - u_{2}^{\epsilon, i} \diamond u_{2}^{\epsilon, j}, \bar{b}_{2}^{\epsilon, i} \diamond \bar{b}_{2}^{\epsilon, j} - b_{2}^{\epsilon, i} \diamond b_{2}^{\epsilon, j},  \nonumber\\
& \hspace{10mm} \bar{b}_{2}^{\epsilon, i} \diamond \bar{u}_{2}^{\epsilon, j} - b_{2}^{\epsilon, i} \diamond u_{2}^{\epsilon, j}  \to 0 \text{ in } C([0, T]; \mathcal{C}^{-\delta}), \label{fourth convergence}\\
& \pi_{0,\diamond} (\bar{u}_{3}^{\epsilon, i}, \bar{u}_{1}^{\epsilon, j}) - \pi_{0,\diamond} (u_{3}^{\epsilon, i}, u_{1}^{\epsilon, j}), \pi_{0,\diamond} (\bar{b}_{3}^{\epsilon, i}, \bar{b}_{1}^{\epsilon, j}) - \pi_{0,\diamond} (b_{3}^{\epsilon, i}, b_{1}^{\epsilon, j}),  \nonumber\\
& \hspace{10mm} \pi_{0,\diamond} (\bar{u}_{3}^{\epsilon, i}, \bar{b}_{1}^{\epsilon, j}) - \pi_{0,\diamond} (u_{3}^{\epsilon, i}, b_{1}^{\epsilon, j}), \pi_{0,\diamond} (\bar{b}_{3}^{\epsilon, i}, \bar{u}_{1}^{\epsilon, j}) - \pi_{0,\diamond} (b_{3}^{\epsilon, i}, u_{1}^{\epsilon, j}) \nonumber\\
& \hspace{30mm}  \to 0 \text{ in } C([0,T]; \mathcal{C}^{-\delta}),  \label{fifth convergence}\\
& \pi_{0,\diamond} (\mathcal{P}^{ii_{1}} D_{j} \bar{K}_{u}^{\epsilon, j}, \bar{u}_{1}^{\epsilon, j_{1}}) - \pi_{0,\diamond} (\mathcal{P}^{ii_{1}} D_{j}^{\epsilon} K_{u}^{\epsilon, j}, u_{1}^{\epsilon, j_{1}}), \nonumber\\
&  \pi_{0,\diamond} (\mathcal{P}^{ii_{1}} D_{j} \bar{K}_{b}^{\epsilon, j}, \bar{b}_{1}^{\epsilon, j_{1}}) - \pi_{0,\diamond} (\mathcal{P}^{ii_{1}} D_{j}^{\epsilon} K_{b}^{\epsilon, j}, b_{1}^{\epsilon, j_{1}}),  \nonumber\\
&  \pi_{0,\diamond} (\mathcal{P}^{ii_{1}} D_{j} \bar{K}_{u}^{\epsilon, j}, \bar{b}_{1}^{\epsilon, j_{1}}) - \pi_{0,\diamond} (\mathcal{P}^{ii_{1}} D_{j}^{\epsilon} K_{u}^{\epsilon, j}, b_{1}^{\epsilon, j_{1}}),  \nonumber\\
&  \pi_{0,\diamond} (\mathcal{P}^{ii_{1}} D_{j} \bar{K}_{b}^{\epsilon, j}, \bar{u}_{1}^{\epsilon, j_{1}}) - \pi_{0,\diamond} (\mathcal{P}^{ii_{1}} D_{j}^{\epsilon} K_{b}^{\epsilon, j}, u_{1}^{\epsilon, j_{1}}) \to 0 \text{ in } C([0, T]; \mathcal{C}^{-\delta}) \label{sixth convergence}
\end{align}
\end{subequations} 
as $\epsilon \searrow 0$. 

\subsection{Convergence of renormalization terms}\label{Subsection 2.2}

The purpose of this subsection is to prove rigorously the convergence of \eqref{first convergence} - \eqref{sixth convergence}. Following \cite[pg. 39]{Y19a} we write $X_{t,u}^{\epsilon} \triangleq u_{1}^{\epsilon}(t), X_{t,b}^{\epsilon} \triangleq b_{1}^{\epsilon}(t)$, $\bar{X}_{t,u}^{\epsilon} \triangleq \bar{u}_{1}^{\epsilon}(t), \bar{X}_{t,b}^{\epsilon} \triangleq \bar{b}_{1}^{\epsilon}(t)$ and following \cite[Notation 4.1]{CC18}, we write $k_{1, \hdots, n} \triangleq \sum_{i=1}^{n} k_{i}$ for $k_{1},\hdots, k_{n} \in \mathbb{Z}^{3}$. We then see that 
\begin{align*}
&X_{t,u}^{\epsilon, i} = \sum_{k\neq 0} \hat{X}_{t,u}^{\epsilon, i}(k) e_{k}, \hspace{3mm} X_{t,b}^{\epsilon, i} = \sum_{k \neq 0} \hat{X}_{t,b}^{\epsilon, i}(k) e_{k}, \\
&\bar{X}_{t,u}^{\epsilon, i} = \sum_{k \neq 0} \hat{\bar{X}}_{t,u}^{\epsilon, i}(k) e_{k}, \hspace{3mm} \bar{X}_{t,b}^{\epsilon, i} = \sum_{k \neq 0} \hat{\bar{X}}_{t,b}^{\epsilon, i}(k) e_{k}, 
\end{align*} 
where $e_{k} \triangleq (2\pi)^{-\frac{3}{2}} e^{ix \cdot k}
$ and $\hat{X}_{t,u}^{\epsilon}(0) = 0, \hat{X}_{t,b}^{\epsilon}(0) = 0, \hat{\bar{X}}_{t,u}^{\epsilon}(0) = 0, \hat{\bar{X}}_{t,b}^{\epsilon}(0) = 0$ due to $\hat{W}_{u}(0) = 0, \hat{W}_{b}(0) = 0$. The important covariance relations are given as follows: 
\begin{subequations}\label{covariance}
\begin{align}
\mathbb{E}[ \hat{X}_{t,u}^{\epsilon, i} (k) \hat{X}_{s,u}^{\epsilon, j}(k') ] =& 1_{k+ k'= 0} \sum_{i_{1} =1}^{3} \frac{ e^{- \lvert k \rvert^{2} f(\epsilon k) \lvert t-s \rvert} h_{u}(\epsilon k)^{2}}{2 \lvert k \rvert^{2} f(\epsilon k)}  \hat{\mathcal{P}}^{ii_{1}} (k) \hat{\mathcal{P}}^{ji_{1}} (k), \label{covariance a}\\
\mathbb{E}[ \hat{X}_{t,u}^{\epsilon, i} (k) \hat{X}_{s,b}^{\epsilon, j}(k') ] =& 1_{k+ k'= 0} \sum_{i_{1} =1}^{3} \frac{ e^{- \lvert k \rvert^{2} f(\epsilon k) \lvert t-s \rvert} h_{u}(\epsilon k)h_{b}(\epsilon k)}{2 \lvert k \rvert^{2} f(\epsilon k)} \nonumber\\
& \hspace{40mm} \times  \hat{\mathcal{P}}^{ii_{1}} (k) \hat{\mathcal{P}}^{ji_{1}} (k), \label{covariance b}\\
\mathbb{E}[ \hat{X}_{t,b}^{\epsilon, i} (k) \hat{X}_{s,b}^{\epsilon, j}(k') ] =& 1_{k+ k'= 0} \sum_{i_{1} =1}^{3} \frac{ e^{- \lvert k \rvert^{2} f(\epsilon k) \lvert t-s \rvert} h_{b}(\epsilon k)^{2}}{2 \lvert k \rvert^{2} f(\epsilon k)}  \hat{\mathcal{P}}^{ii_{1}} (k) \hat{\mathcal{P}}^{ji_{1}} (k), \label{covariance c}\\
\mathbb{E}[ \hat{\bar{X}}_{t,u}^{\epsilon, i} (k) \hat{\bar{X}}_{s,u}^{\epsilon, j}(k') ] =& 1_{k+ k'= 0} \sum_{i_{1} =1}^{3} \frac{ e^{- \lvert k \rvert^{2} \lvert t-s \rvert} h_{u}(\epsilon k)^{2}}{2 \lvert k \rvert^{2}}  \hat{\mathcal{P}}^{ii_{1}} (k) \hat{\mathcal{P}}^{ji_{1}} (k), \label{covariance d}\\
\mathbb{E}[ \hat{\bar{X}}_{t,u}^{\epsilon, i} (k) \hat{\bar{X}}_{s,b}^{\epsilon, j}(k') ] =& 1_{k+ k'= 0} \sum_{i_{1} =1}^{3} \frac{ e^{- \lvert k \rvert^{2} \lvert t-s \rvert} h_{u}(\epsilon k)h_{b}(\epsilon k)}{2 \lvert k \rvert^{2}}  \hat{\mathcal{P}}^{ii_{1}} (k) \hat{\mathcal{P}}^{ji_{1}} (k), \label{covariance e}\\
\mathbb{E}[ \hat{\bar{X}}_{t,b}^{\epsilon, i} (k) \hat{\bar{X}}_{s,b}^{\epsilon, j}(k') ] =& 1_{k+ k'= 0} \sum_{i_{1} =1}^{3} \frac{ e^{- \lvert k \rvert^{2} \lvert t-s \rvert} h_{b}(\epsilon k)^{2}}{2 \lvert k \rvert^{2}}  \hat{\mathcal{P}}^{ii_{1}} (k) \hat{\mathcal{P}}^{ji_{1}} (k).\label{covariance f}
\end{align}
\end{subequations}
In particular, these relations are derived by crucially relying on $\delta_{ij}$ in \eqref{STWN for MHD}. For $t \leq s$,  
\begin{align}\label{estimate 145}
 \mathbb{E} [ \hat{X}_{t,u}^{\epsilon, i} (k) \hat{\bar{X}}_{s,u}^{\epsilon, j} (k')] = 1_{k+ k'=0} \sum_{i_{1} =1}^{3}  \frac{ e^{- \lvert k \rvert^{2} (s-t)} h_{u}(\epsilon k)^{2}}{\lvert k \rvert^{2} (f(\epsilon k) + 1)} \hat{\mathcal{P}}^{ii_{1}}(k) \hat{\mathcal{P}}^{ji_{1}}(k') 
\end{align} 
by \eqref{[Equation (3.1aa)][ZZ17]}, \eqref{[Equation (3.21c)][ZZ17]} and \eqref{STWN for MHD}. For $t > s$, similarly 
\begin{align}
 \mathbb{E}[ \hat{X}_{t,u}^{\epsilon,i} (k) \hat{\bar{X}}_{s,u}^{\epsilon, j} (k')] = 1_{k+ k'= 0} \sum_{i_{1} =1}^{3} \frac{ e^{- \lvert k \rvert^{2} f( \epsilon k)(t-s)}  h_{u}(\epsilon k)^{2}}{ \lvert k \rvert^{2} (f( \epsilon k) + 1)} \hat{\mathcal{P}}^{ii_{1}}(k) \hat{\mathcal{P}}^{ji_{1}}(k). 
\end{align} 
Identically, for $t \leq s$, 
\begin{subequations}
\begin{align}
 \mathbb{E} [ \hat{X}_{t,b}^{\epsilon, i} (k) \hat{\bar{X}}_{s,b}^{\epsilon, j} (k')] =& 1_{k+ k'=0} \sum_{i_{1} =1}^{3}  \frac{ e^{- \lvert k \rvert^{2} (s-t)} h_{u}(\epsilon k)^{2}}{\lvert k \rvert^{2} (f(\epsilon k) + 1)} \hat{\mathcal{P}}^{ii_{1}}(k) \hat{\mathcal{P}}^{ji_{1}}(k'), \\
  \mathbb{E} [ \hat{X}_{t,u}^{\epsilon, i} (k) \hat{\bar{X}}_{s,b}^{\epsilon, j} (k')] = & \mathbb{E} [ \hat{X}_{t,b}^{\epsilon, i} (k) \hat{\bar{X}}_{s,u}^{\epsilon, j} (k')] \nonumber\\
  =& 1_{k+ k'=0} \sum_{i_{1} =1}^{3}  \frac{ e^{- \lvert k \rvert^{2} (s-t)} h_{u}(\epsilon k) h_{b}(\epsilon k)}{\lvert k \rvert^{2} (f(\epsilon k) + 1)} \hat{\mathcal{P}}^{ii_{1}}(k) \hat{\mathcal{P}}^{ji_{1}}(k'), 
\end{align}
\end{subequations} 
and for $t > s$, 
\begin{subequations}\label{estimate 174}
\begin{align}
\mathbb{E}[ \hat{X}_{t,b}^{\epsilon,i} (k) \hat{\bar{X}}_{s,b}^{\epsilon, j} (k')] =& 1_{k+ k'= 0} \sum_{i_{1} =1}^{3} \frac{ e^{- \lvert k \rvert^{2} f( \epsilon k)(t-s)}  h_{b}(\epsilon k)^{2}}{ \lvert k \rvert^{2} (f( \epsilon k) + 1)} \hat{\mathcal{P}}^{ii_{1}}(k) \hat{\mathcal{P}}^{ji_{1}}(k), \\
\mathbb{E}[ \hat{X}_{t,u}^{\epsilon,i} (k) \hat{\bar{X}}_{s,b}^{\epsilon, j} (k')] =& \mathbb{E}[ \hat{X}_{t,b}^{\epsilon,i} (k) \hat{\bar{X}}_{s,u}^{\epsilon, j} (k')] \\
=& 1_{k+ k'= 0} \sum_{i_{1} =1}^{3} \frac{ e^{- \lvert k \rvert^{2} f( \epsilon k)(t-s)}  h_{u}(\epsilon k)h_{b}(\epsilon k)}{ \lvert k \rvert^{2} (f( \epsilon k) + 1)} \hat{\mathcal{P}}^{ii_{1}}(k) \hat{\mathcal{P}}^{ji_{1}}(k).   \nonumber
\end{align}
\end{subequations} 

\subsubsection{Convergence of \eqref{first convergence}}
W.l.o.g. we fix $t_{1} < t_{2}$ and work on $b_{1}^{\epsilon} - \overline{b}_{1}^{\epsilon}$. We compute 
\begin{align}\label{estimate 105}
& \mathbb{E} [ \lvert \Delta_{q} [ ( b_{1}^{\epsilon, i} (t_{2}) - \bar{b}_{1}^{\epsilon, i} (t_{2})) - (b_{1}^{\epsilon, i} (t_{1}) - \bar{b}_{1}^{\epsilon, i} (t_{1}))] \rvert^{2}] \nonumber\\
\lesssim& \sum_{k \neq 0} \theta(2^{-q} k)^{2} \mathbb{E} [ \lvert \hat{X}_{t_{2}, b}^{\epsilon, i} (k) - \hat{\bar{X}}_{t_{2}, b}^{\epsilon, i} (k) - \hat{X}_{t_{1}, b}^{\epsilon, i} (k) + \hat{\bar{X}}_{t_{1}, b}^{\epsilon, i} (k) \rvert^{2}] e_{k} 
\end{align}
by orthogonality of $\{e_{k}\}_{k}$,  in which we can estimate 
\begin{align}
& \mathbb{E} [ \lvert \hat{X}_{t_{2}, b}^{\epsilon, i} (k) - \hat{\bar{X}}_{t_{2}, b}^{\epsilon, i} (k) - \hat{X}_{t_{1}, b}^{\epsilon, i} (k) + \hat{\bar{X}}_{t_{1}, b}^{\epsilon, i} (k) \rvert^{2}] \nonumber\\
=& h_{b}(\epsilon k)^{2}  \sum_{i_{1} =1}^{3} \hat{\mathcal{P}}^{ii_{1}}(k)^{2} [ \frac{1}{\lvert k \rvert^{2} f( \epsilon k)} - \frac{4}{\lvert k \rvert^{2} (f( \epsilon k) + 1)} - \frac{ e^{- \lvert k \rvert^{2}  f( \epsilon k)(t_{2} - t_{1})}}{\lvert k \rvert^{2} f( \epsilon k)} \nonumber\\
& \hspace{2mm}+ \frac{ 2 e^{- \lvert k \rvert^{2}  f(\epsilon k)(t_{2} - t_{1})}}{\lvert k \rvert^{2} (f( \epsilon k) + 1)}   + \frac{1}{\lvert k \rvert^{2}}+ \frac{ 2e^{- \lvert k \rvert^{2} (t_{2} - t_{1})}}{\lvert k \rvert^{2} (f( \epsilon k) + 1)} - \frac{ e^{- \lvert k \rvert^{2} (t_{2} - t_{1})}}{\lvert k \rvert^{2}}] \lesssim  A_{k, t_{2}t_{1}} 
\end{align}
where 
\begin{align}
A_{k, t_{2}t_{1}} \triangleq& \lvert  \frac{1}{\lvert k \rvert^{2} f( \epsilon k)} - \frac{4}{\lvert k \rvert^{2} (f( \epsilon k) + 1)} - \frac{ e^{- \lvert k \rvert^{2}  f( \epsilon k)(t_{2} - t_{1})}}{\lvert k \rvert^{2} f( \epsilon k)} + \frac{ 2 e^{- \lvert k \rvert^{2}  f(\epsilon k)(t_{2} - t_{1})}}{\lvert k \rvert^{2} (f( \epsilon k) + 1)} \nonumber\\
& \hspace{20mm} + \frac{1}{\lvert k \rvert^{2}} + \frac{ 2e^{- \lvert k \rvert^{2} (t_{2} - t_{1})}}{\lvert k \rvert^{2} (f( \epsilon k) + 1)} - \frac{ e^{- \lvert k \rvert^{2} (t_{2} - t_{1})}}{\lvert k \rvert^{2}} \rvert. 
\end{align} 
On one hand, 
\begin{align}\label{estimate 304}
A_{k, t_{2}t_{1}} \lesssim& \lvert \frac{ 1 - e^{- \lvert k \rvert^{2}  f(\epsilon k)(t_{2} - t_{1})}}{\lvert k \rvert^{2} f( \epsilon k)} \rvert + \lvert \frac{ -2 + 2 e^{- \lvert k \rvert^{2} f( \epsilon k) (t_{2} - t_{1})}}{\lvert k \rvert^{2} (f( \epsilon k) + 1) } \rvert \nonumber \\
&+ \lvert \frac{ -2 + 2 e^{- \lvert k \rvert^{2} (t_{2} - t_{1}) }}{\lvert k \rvert^{2} (f( \epsilon k) + 1)} \rvert + \lvert \frac{1- e^{- \lvert k \rvert^{2} (t_{2} - t_{1}) }}{\lvert k \rvert^{2}} \rvert \lesssim \frac{ (t_{2} - t_{1})^{\eta}}{\lvert k \rvert^{2-2\eta}} 
\end{align} 
for any $\eta \in [0,1]$ and $\epsilon > 0$ sufficiently small, by mean value theorem and \eqref{lower bound}. On the other hand,
\begin{align}\label{estimate 305}
A_{k, t_{2}t_{1}} 
\lesssim \frac{1}{\lvert k \rvert^{2}} \lvert \lvert f( \epsilon k) -1 \rvert^{2} + \lvert f( \epsilon k) -1 \rvert \rvert \lesssim \frac{ \min \{1, \lvert \epsilon k \rvert \} }{\lvert k \rvert^{2}} \lesssim \frac{ \epsilon^{2\eta}}{\lvert k \rvert^{2-2\eta}} 
\end{align} 
for any $\eta \in [0, \frac{1}{2}]$ where we used that $f(0) = 1$ by Definition \ref{Definition of approximation} and mean value theorem. Hence, 
\begin{align}\label{special}
& \mathbb{E} [ \lvert \hat{X}_{t_{2}, b}^{\epsilon, i} (k) - \hat{\bar{X}}_{t_{2}, b}^{\epsilon, i} (k) - \hat{X}_{t_{1}, b}^{\epsilon, i} (k) + \hat{\bar{X}}_{t_{1}, b}^{\epsilon, i} (k) \rvert^{2}] \nonumber\\
& \hspace{20mm} \lesssim \frac{ (t_{2} - t_{1})^{\eta}}{\lvert k \rvert^{2-2\eta}}  \wedge \frac{ \epsilon^{2\eta}}{\lvert k \rvert^{2-2\eta}} \lesssim  \frac{ (t_{2} - t_{1})^{\frac{\eta}{2}} \epsilon^{\eta}}{\lvert k \rvert^{2-2\eta}}. 
\end{align} 
Therefore, applying \eqref{special} to \eqref{estimate 105} gives 
\begin{align}\label{estimate 1061}
\mathbb{E} [ \lvert \Delta_{q}   [ ( b_{1}^{\epsilon, i} (t_{2}) - \bar{b}_{1}^{\epsilon, i} (t_{2})) - (b_{1}^{\epsilon, i} (t_{1}) - \bar{b}_{1}^{\epsilon, i} (t_{1}))] \rvert^{2}]  
\lesssim \epsilon^{\eta} (t_{2} - t_{1})^{\frac{\eta}{2}} 2^{q(1+ 2\eta)}. 
\end{align}
This leads us to for any $p > 1$, 
\begin{align*}
\mathbb{E} [  \lVert \Delta_{q} [ ( b_{1}^{\epsilon, i} (t_{2}) - \bar{b}_{1}^{\epsilon, i} (t_{2}))  - ( b_{1}^{\epsilon, i} (t_{1}) - \bar{b}_{1}^{\epsilon, i} (t_{1})) ] \rVert_{L^{p}}^{p} ] \lesssim \epsilon^{\frac{\eta p}{2}} \lvert t_{2} - t_{1} \rvert^{\frac{\eta p}{4}} 2^{\frac{qp}{2} (1+ 2\eta)}
\end{align*}
by Gaussian hypercontractivity theorem (e.g., \cite[Theorem 3.50]{J97}, an application on \cite[pg. 37]{GIP15}) and \eqref{estimate 1061}.  In turn, this leads to for $p > 1$ sufficiently large and $\eta, \epsilon_{0} > 0$ sufficiently small so that $\eta + \epsilon_{0} + \frac{3}{p} < \frac{\delta}{2}$, 
\begin{align}\label{estimate 36}
\mathbb{E} [ \lVert (b_{1}^{\epsilon, i} (t_{2}) - \bar{b}_{1}^{\epsilon, i} (t_{2})) - (b_{1}^{\epsilon, i} (t_{1}) - \bar{b}_{1}^{\epsilon, i} (t_{1})) \rVert_{\mathcal{C}^{-\frac{1}{2} - \frac{\delta}{2}}}^{p} ] \lesssim \epsilon^{\frac{\eta p}{2}} (t_{2} - t_{1})^{\frac{\eta p}{4}} 
\end{align}  
by \eqref{regularity 0} and Besov embedding Lemma \ref{Lemma 3.4}. This implies that for all $p > 1$, all $\delta > 0$, $b_{1}^{\epsilon, i} - \bar{b}_{1}^{\epsilon, i} \searrow 0$ in $L^{p} ( \Omega; C([0, T]; \mathcal{C}^{- \frac{1}{2} - \frac{\delta}{2}}))$ as $\epsilon \searrow 0$. 

\subsubsection{Convergence of \eqref{second convergence}}

W.l.o.g. we consider the convergence of $u_{1}^{\epsilon, i} \diamond b_{1}^{\epsilon, j} - \bar{u}_{1}^{\epsilon, i} \diamond \bar{b}_{1}^{\epsilon, j}$ and assume that $t_{2} > t_{1}$. We define 
\begin{align}\label{estimate 49}
C_{03}^{\epsilon, ij} \triangleq \mathbb{E} [ X_{t,u}^{\epsilon, i} X_{t,b}^{\epsilon, j}]  = (2\pi)^{-3} \sum_{k_{1} \neq 0} \sum_{i_{1} =1}^{3}  \frac{ h_{u}(\epsilon k_{1}) h_{b}(\epsilon k_{1})}{2 \lvert k_{1} \rvert^{2} f( \epsilon k_{1})} \hat{\mathcal{P}}^{ii_{1}}(k_{1}) \hat{\mathcal{P}}^{ji_{1}}(k_{1}) 
\end{align} 
by \eqref{covariance b}. Similarly, we define 
\begin{align}
\bar{C}_{03}^{\epsilon, ij} \triangleq \mathbb{E} [ \bar{X}_{t,u}^{\epsilon, i} \bar{X}_{t,b}^{\epsilon, j} ] = (2\pi)^{-3} \sum_{k_{1} \neq 0}  \sum_{i_{1} =1}^{3} \frac{ h_{u}(\epsilon k_{1}) h_{b}(\epsilon k_{1})}{2 \lvert k_{1} \rvert^{2}} \hat{\mathcal{P}}^{ii_{1}}(k_{1}) \hat{\mathcal{P}}^{ji_{1}}(k_{1}) \label{estimate 42}
\end{align} 
by \eqref{covariance e} and 
\begin{subequations}
\begin{align}
&C_{01}^{\epsilon, ij} \triangleq \mathbb{E} [X_{t,u}^{\epsilon, i} X_{t,u}^{\epsilon, j} ], C_{02}^{\epsilon, ij} \triangleq \mathbb{E} [X_{t,b}^{\epsilon, i} X_{t,b}^{\epsilon, j} ], \label{estimate 40} \\
& \bar{C}_{01}^{\epsilon, ij} \triangleq \mathbb{E} [ \bar{X}_{t,u}^{\epsilon, i} \bar{X}_{t,u}^{\epsilon, j} ], \bar{C}_{02}^{\epsilon, ij} \triangleq \mathbb{E} [ \bar{X}_{t,b}^{\epsilon, i} \bar{X}_{t,b}^{\epsilon, j}] \label{estimate 41},
\end{align} 
\end{subequations} 
which are all constants independent of $t$. Thus, we may now compute 
\begin{align}\label{estimate 131}
& \mathbb{E} [\lvert \Delta_{q} [ ( u_{1}^{\epsilon, i} \diamond b_{1}^{\epsilon, j} (t_{2}) - \bar{u}_{1}^{\epsilon, i} \diamond  \bar{b}_{1}^{\epsilon, j} (t_{2}))  - (u_{1}^{\epsilon, i} \diamond b_{1}^{\epsilon, j} (t_{1}) - \bar{u}_{1}^{\epsilon, i} \diamond \bar{b}_{1}^{\epsilon, j} (t_{1}))] \rvert^{2}] \nonumber  \\
=&  \mathbb{E} [\lvert \Delta_{q} [ : u_{1}^{\epsilon, i} b_{1}^{\epsilon, j} (t_{2}): - :\bar{u}_{1}^{\epsilon, i}  \bar{b}_{1}^{\epsilon, j} (t_{2}):  - :u_{1}^{\epsilon, i} b_{1}^{\epsilon, j} (t_{1}): + :\bar{u}_{1}^{\epsilon, i} \bar{b}_{1}^{\epsilon, j} (t_{1}): ] \rvert^{2}] \nonumber \\
\lesssim& \sum_{k} \theta(2^{-q} k)^{2} \sum_{k_{1}, k_{2} \neq 0: k_{12} =k} B_{k_{1}k_{2}, t_{1}t_{2}}
\end{align} 
by Example \ref{Example 3.1}, \eqref{[Equation (3.2aa)][ZZ17]}, \eqref{[Equation (3.22aa)][ZZ17]}, \eqref{estimate 49}, \eqref{estimate 42}, where 
\begin{align}\label{estimate 128}
B_{k_{1}k_{2}, t_{1}t_{2}} \triangleq& \lvert \frac{1}{  \prod_{i=1}^{2} \lvert k_{i} \rvert^{2} f(\epsilon k_{i})} - \frac{8}{ \prod_{i=1}^{2} \lvert k_{i} \rvert^{2} (f(\epsilon k_{i}) + 1)} \nonumber\\
& - \frac{ e^{- [\lvert k_{1} \rvert^{2} f(\epsilon k_{1}) + \lvert k_{2} \rvert^{2} f(\epsilon k_{2})] \lvert t_{2} - t_{1} \rvert}}{\prod_{i=1}^{2} \lvert k_{i} \rvert^{2} f(\epsilon k_{i})} + \frac{ 4 e^{- [ \lvert k_{1} \rvert^{2} f(\epsilon k_{1}) + \lvert k_{2} \rvert^{2} f(\epsilon k_{2})] \lvert t_{2} - t_{1} \rvert}}{\prod_{i=1}^{2} \lvert k_{i} \rvert^{2} (f(\epsilon k_{i}) + 1)} \nonumber\\
&+ \frac{1}{\prod_{i=1}^{2} \lvert k_{i} \rvert^{2}} + \frac{ 4 e^{- (\lvert k_{1} \rvert^{2} + \lvert k_{2} \rvert^{2}) \lvert t_{2} - t_{1} \rvert}}{\prod_{i=1}^{2} \lvert k_{i} \rvert^{2} (f(\epsilon k_{i}) + 1)} - \frac{ e^{- [ \lvert k_{1} \rvert^{2} + \lvert k_{2} \rvert^{2} ] \lvert t_{2} - t_{1} \rvert}}{\prod_{i=1}^{2} \lvert k_{i} \rvert^{2}} \rvert. 
\end{align} 
Now on one hand, we can estimate for any $\eta \in [0,1]$, 
\begin{align}\label{estimate 129}
&B_{k_{1}k_{2}, t_{1}t_{2}} \\
\leq& \lvert \frac{1 - e^{- ( \lvert k_{1} \rvert^{2} f( \epsilon k_{1}) + \lvert k_{2} \rvert^{2} f( \epsilon k_{2})) \lvert t_{2} - t_{1} \rvert}}{ \prod_{i=1}^{2} \lvert k_{i} \rvert^{2} f( \epsilon k_{i})} \rvert + 4 \lvert \frac{ e^{- ( \lvert k_{1} \rvert^{2} f(\epsilon k_{1}) + \lvert k_{2} \rvert^{2} f(\epsilon k_{2})) \lvert t_{2} -t_{1} \rvert} - 1}{\prod _{i=1}^{2} \lvert k_{i} \rvert^{2} (f(\epsilon k_{i}) + 1)} \rvert \nonumber\\  
&+ 4 \lvert \frac{ e^{- ( \lvert k_{1} \rvert^{2} + \lvert k_{2} \rvert^{2} ) \lvert t_{2} - t_{1} \rvert} -1}{ \prod_{i=1}^{2} \lvert k_{i} \rvert^{2} (f(\epsilon k_{i} ) +1)} \rvert + \frac{ \lvert 1- e^{- (\lvert k_{1} \rvert^{2} + \lvert k_{2} \rvert^{2}) \lvert t_{2} - t_{1} \rvert}}{\prod_{i=1}^{2} \lvert k_{i} \rvert^{2}} \nonumber  \\
\lesssim& \frac{1}{\prod_{i=1}^{2} \lvert k_{i} \rvert^{2}} [ 1 \wedge [ \lvert k_{1} \rvert^{2} + \lvert k_{2} \rvert^{2} ] \lvert t_{2} - t_{1} \rvert ] \lesssim \frac{ ( \lvert k_{1} \rvert^{2\eta} + \lvert k_{2} \rvert^{2\eta} ) \lvert t_{2} - t_{1} \rvert^{\eta}}{\prod_{i=1}^{2} \lvert k_{i} \rvert^{2}} \nonumber
\end{align} 
by \eqref{estimate 128}, mean value theorem and \eqref{lower bound}. On the other hand, for any $\eta \in [0,\frac{1}{2}]$, 
\begin{align}\label{estimate 130}
B_{k_{1}k_{2}, t_{1}t_{2}} \lesssim  \frac{1}{\prod_{i=1}^{2} \lvert k_{i} \rvert^{2}} [ (1\wedge \lvert \epsilon k_{2} \rvert) + (1 \wedge \lvert \epsilon k_{1} \rvert)] 
\lesssim \frac{ \lvert \epsilon k_{1} \rvert^{2\eta} + \lvert \epsilon k_{2} \rvert^{2\eta}}{\prod_{i=1}^{2} \lvert k_{i} \rvert^{2}}
\end{align} 
by \eqref{estimate 128}, mean value theorem as $f(0) = 1$. Applying \eqref{estimate 129}-\eqref{estimate 130} to \eqref{estimate 131} gives 
\begin{align}\label{estimate 132}
& \mathbb{E} [ \lvert \Delta_{q} [ \lvert (u_{1}^{\epsilon, i} \diamond b_{1}^{\epsilon, j} (t_{2}) - \bar{u}_{1}^{\epsilon, i} \diamond \bar{b}_{1}^{\epsilon, j} (t_{2})) \\
& \hspace{10mm} - (u_{1}^{\epsilon, i} \diamond b_{1}^{\epsilon, j} (t_{1}) - \bar{u}_{1}^{\epsilon, i} \diamond \bar{b}_{1}^{\epsilon, j} (t_{1}))] \rvert^{2} ] \lesssim  \epsilon^{\eta} \lvert t_{1}- t_{2} \rvert^{\frac{\eta}{2}} 2^{q(2+ 2\eta)} \nonumber 
\end{align}
by Lemma \ref{Lemma 3.13}. Therefore, by relying on Gaussian hypercontractivity theorem, \eqref{estimate 132} and  Besov embedding Lemma \ref{Lemma 3.4} similarly to \eqref{estimate 36}, we are able to deduce that for all $p > 1$, all $\delta > 0$, $u_{1}^{\epsilon, i} \diamond b_{1}^{\epsilon, j} - \bar{u}_{1}^{\epsilon, i} \diamond \bar{b}_{1}^{\epsilon, j} \to 0$ in  $L^{p} (\Omega; C([0, T]; \mathcal{C}^{-1 - \frac{\delta}{2}}))$ as $\epsilon \searrow 0$. 

\subsubsection{Convergence of \eqref{third convergence}}\label{convergence of third convergence}

W.l.o.g. we work on $b_{1}^{\epsilon} \diamond b_{2}^{\epsilon} - \bar{b}_{1}^{\epsilon} \diamond \bar{b}_{2}^{\epsilon}$. We have the following identity for $t \in [0,T]$ and $i, j \in \{1,2,3\}$ from \eqref{[Equation (3.1aa)][ZZ17]}-\eqref{[Equation (3.1ab)][ZZ17]}, \eqref{[Equation (3.2aa)][ZZ17]}, \eqref{[Equation (3.21c)][ZZ17]}-\eqref{[Equation (3.21d)][ZZ17]}, \eqref{[Equation (3.22aa)][ZZ17]}:  
\begin{align}
& \bar{b}_{1}^{\epsilon, j} \bar{b}_{2}^{\epsilon, i} (t) - b_{1}^{\epsilon, j} b_{2}^{\epsilon, i} (t) \\
=& \frac{ (2\pi)^{-3}}{2} \sum_{i_{1}, i_{2} =1}^{3} \sum_{k} \sum_{k_{1}, k_{2}, k_{3} \neq 0: k_{123} = k} \nonumber\\
& \times[ \int_{0}^{t} e^{- \lvert k_{12} \rvert^{2} f( \epsilon k_{12}) (t-s)} k_{12}^{i_{2}} g(\epsilon k_{12}^{i_{2}}) \nonumber\\
& \hspace{6mm}  \times [ : \hat{X}_{s,b}^{\epsilon, i_{1}}(k_{1}) \hat{X}_{s,u}^{\epsilon, i_{2}}(k_{2}) \hat{X}_{t,b}^{\epsilon, j} (k_{3}): - : \hat{X}_{s,u}^{\epsilon, i_{1}}(k_{1}) \hat{X}_{s,b}^{\epsilon, i_{2}}(k_{2}) \hat{X}_{t,b}^{\epsilon, j} (k_{3}):] ds \hat{\mathcal{P}}^{ii_{1}}(k_{12}) \nonumber \\
& \hspace{1mm} - \int_{0}^{t} e^{- \lvert k_{12} \rvert^{2} (t-s)} k_{12}^{i_{2}} i \nonumber\\
& \hspace{6mm} \times [ : \hat{\bar{X}}_{s,b}^{\epsilon, i_{1}}(k_{1}) \hat{\bar{X}}_{s,u}^{\epsilon, i_{2}}(k_{2}) \hat{\bar{X}}_{t,b}^{\epsilon, j} (k_{3}): - : \hat{\bar{X}}_{s,u}^{\epsilon, i_{1}}(k_{1}) \hat{\bar{X}}_{s,b}^{\epsilon, i_{2}}(k_{2}) \hat{\bar{X}}_{t,b}^{\epsilon, j} (k_{3}):] ds \hat{\mathcal{P}}^{ii_{1}}(k_{12})] e_{k} \nonumber \\
&+ \frac{ (2\pi)^{-3}}{2} \sum_{i_{1}, i_{2} =1}^{3} \sum_{k} \sum_{k_{1}, k_{2}, k_{3} \neq 0: k_{123} = k} \nonumber\\
& \times [ \int_{0}^{t} e^{- \lvert k_{12} \rvert^{2} f(\epsilon k_{12}) (t-s)} k_{12}^{i_{2}} g( \epsilon k_{12}^{i_{2}}) \nonumber\\
& \hspace{1mm} \times [ 1_{k_{23} = 0} \sum_{i_{3} =1}^{3} \frac{ e^{- \lvert k_{2} \rvert^{2} f( \epsilon k_{2}) (t-s)} h_{u}(\epsilon k_{2}) h_{b}(\epsilon k_{2})}{2 \lvert k_{2} \rvert^{2} f( \epsilon k_{2})} \hat{\mathcal{P}}^{i_{2}i_{3}} (k_{2}) \hat{\mathcal{P}}^{ji_{3}}(k_{2}) \hat{X}_{s,b}^{\epsilon, i_{1}}(k_{1}) \nonumber  \\
&\hspace{1mm} + 1_{k_{13} = 0} \sum_{i_{3} =1}^{3} \frac{ e^{- \lvert k_{1} \rvert^{2} f( \epsilon k_{1}) (t-s)} h_{b}(\epsilon k_{1})^{2}}{2 \lvert k_{1} \rvert^{2} f( \epsilon k_{1})} \hat{\mathcal{P}}^{i_{1}i_{3}} (k_{1}) \hat{\mathcal{P}}^{ji_{3}}(k_{1}) \hat{X}_{s,u}^{\epsilon, i_{2}}(k_{2}) \nonumber \\ 
&\hspace{1mm} - 1_{k_{23} = 0} \sum_{i_{3} =1}^{3} \frac{ e^{- \lvert k_{2} \rvert^{2} f( \epsilon k_{2}) (t-s)} h_{b}(\epsilon k_{2})^{2}}{2 \lvert k_{2} \rvert^{2} f( \epsilon k_{2})} \hat{\mathcal{P}}^{i_{2}i_{3}} (k_{2}) \hat{\mathcal{P}}^{ji_{3}}(k_{2}) \hat{X}_{s,u}^{\epsilon, i_{1}}(k_{1}) \nonumber \\ 
&\hspace{1mm} - 1_{k_{13} = 0} \sum_{i_{3} =1}^{3} \frac{ e^{- \lvert k_{1} \rvert^{2} f( \epsilon k_{1}) (t-s)} h_{u}(\epsilon k_{1})h_{b}(\epsilon k_{1})}{2 \lvert k_{1} \rvert^{2} f( \epsilon k_{1})} \hat{\mathcal{P}}^{i_{1}i_{3}} (k_{1}) \hat{\mathcal{P}}^{ji_{3}}(k_{1}) \hat{X}_{s,b}^{\epsilon, i_{2}}(k_{2}) \nonumber \\
& - \int_{0}^{t} e^{- \lvert k_{12} \rvert^{2} (t-s)} k_{12}^{i_{2}}i \nonumber\\
& \hspace{1mm} \times [ 1_{k_{23} = 0} \sum_{i_{3} =1}^{3} \frac{ e^{- \lvert k_{2} \rvert^{2} (t-s)} h_{u}(\epsilon k_{2}) h_{b}(\epsilon k_{2})}{2 \lvert k_{2} \rvert^{2} } \hat{\mathcal{P}}^{i_{2}i_{3}} (k_{2}) \hat{\mathcal{P}}^{ji_{3}}(k_{2})  \hat{\bar{X}}_{s,b}^{\epsilon, i_{1}}(k_{1}) \nonumber  \\
&\hspace{1mm} + 1_{k_{13} = 0} \sum_{i_{3} =1}^{3} \frac{ e^{- \lvert k_{1} \rvert^{2} (t-s)} h_{b}(\epsilon k_{1})^{2}}{2 \lvert k_{1} \rvert^{2}} \hat{\mathcal{P}}^{i_{1}i_{3}} (k_{1}) \hat{\mathcal{P}}^{ji_{3}}(k_{1}) \hat{\bar{X}}_{s,u}^{\epsilon, i_{2}}(k_{2}) \nonumber \\ 
&\hspace{1mm} - 1_{k_{23} = 0} \sum_{i_{3} =1}^{3} \frac{ e^{- \lvert k_{2} \rvert^{2} (t-s)} h_{b}(\epsilon k_{2})^{2}}{2 \lvert k_{2} \rvert^{2} } \hat{\mathcal{P}}^{i_{2}i_{3}} (k_{2}) \hat{\mathcal{P}}^{ji_{3}}(k_{2}) \hat{\bar{X}}_{s,u}^{\epsilon, i_{1}}(k_{1}) \nonumber \\ 
&\hspace{1mm} - 1_{k_{13} = 0} \sum_{i_{3} =1}^{3} \frac{ e^{- \lvert k_{1} \rvert^{2} (t-s)} h_{u}(\epsilon k_{1})h_{b}(\epsilon k_{1})}{2 \lvert k_{1} \rvert^{2}} \hat{\mathcal{P}}^{i_{1}i_{3}} (k_{1}) \hat{\mathcal{P}}^{ji_{3}}(k_{1}) \hat{\bar{X}}_{s,b}^{\epsilon, i_{2}}(k_{2})] ds] \hat{\mathcal{P}}^{ii_{1}}(k_{12}) e_{k}  \nonumber
\end{align} 
by Example \ref{Example 3.1} where all the terms corresponding to $1_{k_{12} = 0}$ vanished due to $k_{12}$. This is similar but slightly different from the cancellations used to derive \cite[Equation (119)]{Y19a}. Thus, we obtain 
\begin{equation}\label{estimate 133}
\bar{b}_{1}^{\epsilon, j} \bar{b}_{2}^{\epsilon, i} (t) - b_{1}^{\epsilon, j} b_{2}^{\epsilon, i} (t) = \sum_{l=1}^{3} V_{t, ii_{1}}^{l}
\end{equation} 
where 
\begin{subequations}\label{estimate 135}
\begin{align}
V_{t, ii_{1}}^{1} &\triangleq \frac{ (2\pi)^{-3}}{2} \sum_{i_{1}, i_{2} =1}^{3} \sum_{k} \sum_{k_{1}, k_{2}, k_{3} \neq 0: k_{123} = k} \\
& \times[ \int_{0}^{t} e^{- \lvert k_{12} \rvert^{2} f(\epsilon k_{12} ) (t-s)} k_{12}^{i_{2}} g(\epsilon k_{12}^{i_{2}}) \nonumber\\
& \hspace{2mm} \times [: \hat{X}_{s,b}^{\epsilon, i_{1}}(k_{1}) \hat{X}_{s,u}^{\epsilon, i_{2}}(k_{2}) \hat{X}_{t,b}^{\epsilon, j} (k_{3}): -: \hat{X}_{s,u}^{\epsilon, i_{1}}(k_{1}) \hat{X}_{s,b}^{\epsilon, i_{2}}(k_{2}) \hat{X}_{t,b}^{\epsilon, j} (k_{3}):] ds \hat{\mathcal{P}}^{ii_{1}}(k_{12}) \nonumber\\
& -  \int_{0}^{t} e^{- \lvert k_{12} \rvert^{2} (t-s)} k_{12}^{i_{2}}i \nonumber\\
&\hspace{2mm} \times [: \hat{\bar{X}}_{s,b}^{\epsilon, i_{1}}(k_{1}) \hat{\bar{X}}_{s,u}^{\epsilon, i_{2}}(k_{2}) \hat{\bar{X}}_{t,b}^{\epsilon, j} (k_{3}): - :\hat{\bar{X}}_{s,u}^{\epsilon, i_{1}}(k_{1}) \hat{\bar{X}}_{s,b}^{\epsilon, i_{2}}(k_{2}) \hat{\bar{X}}_{t,b}^{\epsilon, j} (k_{3}):] ds \hat{\mathcal{P}}^{ii_{1}}(k_{12})]e_{k}, \nonumber\\
V_{t, ii_{1}}^{2} &\triangleq \frac{ (2\pi)^{-3}}{2} \sum_{i_{1}, i_{2}, i_{3} =1}^{3} \sum_{k_{1}, k_{2} \neq 0} \\
& \times [ \int_{0}^{t} e^{- \lvert k_{12} \rvert^{2} f(\epsilon k_{12}) (t-s)} k_{12}^{i_{2}} g(\epsilon k_{12}^{i_{2}}) \hat{X}_{s,b}^{\epsilon, i_{1}}(k_{1}) \frac{ e^{- \lvert k_{2} \rvert^{2} f(\epsilon k_{2}) (t-s)} h_{u}(\epsilon k_{2}) h_{b}(\epsilon k_{2})}{2 \lvert k_{2} \rvert^{2} f(\epsilon k_{2})} ds \nonumber \\
& - \int_{0}^{t} e^{- \lvert k_{12} \rvert^{2} f(\epsilon k_{12}) (t-s)} k_{12}^{i_{2}} g(\epsilon k_{12}^{i_{2}}) \hat{X}_{s,u}^{\epsilon, i_{1}}(k_{1}) \frac{ e^{- \lvert k_{2} \rvert^{2} f(\epsilon k_{2}) (t-s)} h_{b}(\epsilon k_{2})^{2}}{2 \lvert k_{2} \rvert^{2} f(\epsilon k_{2})} ds \nonumber \\
& - \int_{0}^{t} e^{- \lvert k_{12} \rvert^{2} (t-s)} k_{12}^{i_{2}} i \hat{\bar{X}}_{s,b}^{\epsilon, i_{1}}(k_{1}) \frac{ e^{- \lvert k_{2} \rvert^{2}  (t-s)} h_{u}(\epsilon k_{2}) h_{b}(\epsilon k_{2})}{2 \lvert k_{2} \rvert^{2}} ds \nonumber \\
& + \int_{0}^{t} e^{- \lvert k_{12} \rvert^{2} (t-s)} k_{12}^{i_{2}} i \hat{\bar{X}}_{s,u}^{\epsilon, i_{1}}(k_{1}) \frac{ e^{- \lvert k_{2} \rvert^{2}  (t-s)} h_{b}(\epsilon k_{2})^{2}}{2 \lvert k_{2} \rvert^{2}} ds] \nonumber\\
& \hspace{40mm} \times  \hat{\mathcal{P}}^{ii_{1}}(k_{12}) \hat{\mathcal{P}}^{i_{2} i_{3}} (k_{2}) \hat{\mathcal{P}}^{ji_{3}}(k_{2}) e_{k_{1}}, \nonumber \\
V_{t, ii_{1}}^{3} &\triangleq \frac{ (2\pi)^{-3}}{2} \sum_{i_{1}, i_{2}, i_{3} =1}^{3} \sum_{k_{1}, k_{2} \neq 0} \\
& \hspace{2mm} \times [ \int_{0}^{t} e^{- \lvert k_{12} \rvert^{2} f(\epsilon k_{12}) (t-s)} k_{12}^{i_{2}} g(\epsilon k_{12}^{i_{2}}) \hat{X}_{s,u}^{\epsilon, i_{2}}(k_{2}) \frac{ e^{- \lvert k_{1}\rvert^{2} f(\epsilon k_{1}) (t-s)} h_{b}(\epsilon k_{1})^{2}}{2 \lvert k_{1} \rvert^{2} f( \epsilon k_{1})} ds \nonumber \\
& \hspace{2mm} -  \int_{0}^{t} e^{- \lvert k_{12} \rvert^{2} f(\epsilon k_{12}) (t-s)} k_{12}^{i_{2}} g(\epsilon k_{12}^{i_{2}}) \hat{X}_{s,b}^{\epsilon, i_{2}}(k_{2}) \frac{ e^{- \lvert k_{1}\rvert^{2} f(\epsilon k_{1}) (t-s)} h_{u}(\epsilon k_{1})h_{b}(\epsilon k_{1})}{2 \lvert k_{1} \rvert^{2} f( \epsilon k_{1})} ds \nonumber \\
& \hspace{2mm} -  \int_{0}^{t} e^{- \lvert k_{12} \rvert^{2}  (t-s)} k_{12}^{i_{2}} i \hat{\bar{X}}_{s,u}^{\epsilon, i_{2}}(k_{2}) \frac{ e^{- \lvert k_{1}\rvert^{2} (t-s)} h_{b}(\epsilon k_{1})^{2}}{2 \lvert k_{1} \rvert^{2}} ds \nonumber \\
&\hspace{2mm} +  \int_{0}^{t} e^{- \lvert k_{12} \rvert^{2}  (t-s)} k_{12}^{i_{2}} i \hat{\bar{X}}_{s,b}^{\epsilon, i_{2}}(k_{2}) \frac{ e^{- \lvert k_{1}\rvert^{2} (t-s)} h_{u}(\epsilon k_{1})h_{b}(\epsilon k_{1})}{2 \lvert k_{1} \rvert^{2}} ds]  \nonumber\\
& \hspace{2mm} \times \hat{\mathcal{P}}^{ii_{1}}(k_{12}) \hat{\mathcal{P}}^{i_{1}i_{3}}(k_{1}) \hat{\mathcal{P}}^{ji_{3}}(k_{1}) e_{k_{2}}. \nonumber 
\end{align} 
\end{subequations}
Here, $V_{t, ii_{1}}^{1}$ consists of a Wiener chaos of order three while $V_{t, ii_{1}}^{2}$ and $V_{t, ii_{1}}^{3}$ of order one. Now \eqref{[Equation (3.22ab)][ZZ17]} and \eqref{[Equation (3.2ab)][ZZ17]} inform us that \eqref{estimate 133} allows us to write 
\begin{align}\label{estimate 134}
\bar{b}_{1}^{\epsilon, j} \diamond \bar{b}_{2}^{\epsilon, i} (t) - b_{1}^{\epsilon, j} \diamond b_{2}^{\epsilon, i} (t) =&  V_{1, ii_{1}}^{l} + V_{2, ii_{1}}^{l} - \sum_{i_{1} =1}^{3} ( C_{2,u}^{\epsilon, ii_{1} j} X_{t,u}^{\epsilon, i_{1}} + C_{2,b}^{\epsilon, ii_{1} j} X_{t,b}^{\epsilon, i_{1}}) \nonumber\\
& \hspace{4mm} +V_{t, ii_{1}}^{3} - \sum_{i_{1} =1}^{3} (\tilde{C}_{2,u}^{\epsilon, ii_{1} j} X_{t,u}^{\epsilon, i_{1}} + \tilde{C}_{2,b}^{\epsilon, ii_{1} j} X_{t,b}^{\epsilon, i_{1}}). 
\end{align} 

\emph{Terms in the first chaos:} $V_{t, ii_{1}}^{2}, V_{t, ii_{1}}^{3}$ in \eqref{estimate 135}\\
Within \eqref{estimate 134} we write 
\begin{align}\label{estimate 137}
&V_{t, ii_{1}}^{2} - \sum_{i_{1} =1}^{3} ( C_{2,u}^{\epsilon, i i_{1} j} X_{t,u}^{\epsilon, i_{1}} + C_{2,b}^{\epsilon, i i_{1} j} X_{t,b}^{\epsilon, i_{1}}) = V_{t, ii_{1}}^{2} - \tilde{V}_{t, ii_{1}}^{2}  \\
& \hspace{3mm}  + \tilde{V}_{t, ii_{1}}^{2} - \sum_{i_{1} =1}^{3} (C_{2,u}^{\epsilon, i i_{1} j} X_{t,u}^{\epsilon, i_{1}} + C_{2,b}^{\epsilon, i i_{1} j} X_{t,b}^{\epsilon, i_{1}}) + \sum_{i_{1} =1}^{3} ( \bar{C}_{2,u}^{\epsilon, i i_{1} j} \bar{X}_{t,u}^{\epsilon, i_{1}} + \bar{C}_{2,b}^{\epsilon, i i_{1} j} \bar{X}_{t, b}^{\epsilon, i_{1}})  \nonumber 
\end{align} 
where 
\begin{align}\label{estimate 136}
\tilde{V}_{t, ii_{1}}^{2} &\triangleq \frac{(2\pi)^{-3}}{2} \sum_{i_{1}, i_{2}, i_{3} =1}^{3} \sum_{k_{1}, k_{2} \neq 0} \\
&\times [\hat{X}_{t,b}^{\epsilon, i_{1}}(k_{1}) \int_{0}^{t} e^{- \lvert k_{12} \rvert^{2} f( \epsilon k_{12})(t-s)} k_{12}^{i_{2}} g(\epsilon k_{12}^{i_{2}}) \frac{ e^{- \lvert k_{2} \rvert^{2} f(\epsilon k_{2}) (t-s)} h_{u}(\epsilon k_{2}) h_{b}(\epsilon k_{2})}{2 \lvert k_{2} \rvert^{2} f(\epsilon k_{2})} ds \nonumber \\
& - \hat{X}_{t,u}^{\epsilon, i_{1}}(k_{1}) \int_{0}^{t} e^{- \lvert k_{12} \rvert^{2} f( \epsilon k_{12})(t-s)} k_{12}^{i_{2}} g(\epsilon k_{12}^{i_{2}}) \frac{ e^{- \lvert k_{2} \rvert^{2} f(\epsilon k_{2}) (t-s)} h_{b}(\epsilon k_{2})^{2}}{2 \lvert k_{2} \rvert^{2} f(\epsilon k_{2})} ds \nonumber\\
& - \hat{\bar{X}}_{t,b}^{\epsilon, i_{1}}(k_{1}) \int_{0}^{t} e^{- \lvert k_{12} \rvert^{2} (t-s)} k_{12}^{i_{2}} i \frac{ e^{- \lvert k_{2} \rvert^{2}  (t-s)} h_{u}(\epsilon k_{2})h_{b}(\epsilon k_{2})}{2 \lvert k_{2} \rvert^{2} } ds \nonumber\\
&+  \hat{\bar{X}}_{t,u}^{\epsilon, i_{1}}(k_{1}) \int_{0}^{t} e^{- \lvert k_{12} \rvert^{2} (t-s)} k_{12}^{i_{2}} i \frac{ e^{- \lvert k_{2} \rvert^{2}  (t-s)} h_{b}(\epsilon k_{2})^{2}}{2 \lvert k_{2} \rvert^{2} } ds]  \nonumber\\
& \times \hat{\mathcal{P}}^{ii_{1}}(k_{12}) \hat{\mathcal{P}}^{i_{2} i_{3}} (k_{2}) \hat{\mathcal{P}}^{ji_{3}} (k_{2}) e_{k_{1}}, \nonumber 
\end{align} 
\begin{align}\label{C2uepsilonii1j}
C_{2,u}^{\epsilon, i i_{1} j} (t) \triangleq& - \frac{ (2\pi)^{-3}}{2} \sum_{i_{2}, i_{3} =1}^{3} \sum_{k_{2} \neq 0} \int_{0}^{t} e^{-2 \lvert k_{2} \rvert^{2} f( \epsilon k_{2}) (t-s)} k_{2}^{i_{2}} g(\epsilon k_{2}^{i_{2}}) \frac{ h_{b}(\epsilon k_{2})^{2}}{2 \lvert k_{2} \rvert^{2} f( \epsilon k_{2})} \nonumber\\
& \times \hat{\mathcal{P}}^{ii_{1}}(k_{2}) \hat{\mathcal{P}}^{i_{2}i_{3}}(k_{2}) \hat{\mathcal{P}}^{ji_{3}}(k_{2}) ds 
\end{align} 
so that we can readily compute using \eqref{g} and Definition \ref{Definition of approximation}  
\begin{align}\label{C2uepsilonii1jlimit}
\lim_{\epsilon \searrow 0} C_{2,u}^{\epsilon, i i_{1} j}(t) =& - \frac{ (2\pi)^{-3}}{8(a+b)} \sum_{i_{2}, i_{3} =1}^{3}  \\
& \times  \int_{\mathbb{R}^{3}} \frac{ [\cos(ax^{i_{2}}) - \cos(bx^{i_{2}})]}{\lvert x \rvert^{4} f(x)^{2}} h_{b}(x)^{2} \hat{\mathcal{P}}^{ii_{1}}(x) \hat{\mathcal{P}}^{i_{2}i_{3}}(x) \hat{\mathcal{P}}^{ji_{3}}(x)  dx; \nonumber
\end{align} 
similarly,  
\begin{align}\label{C2bepsilonii1j}
C_{2,b}^{\epsilon, i i_{1} j} (t) \triangleq& \frac{ (2\pi)^{-3}}{2} \sum_{i_{2}, i_{3} =1}^{3} \sum_{k_{2} \neq 0} \int_{0}^{t} e^{-2 \lvert k_{2} \rvert^{2} f( \epsilon k_{2}) (t-s)} k_{2}^{i_{2}} g(\epsilon k_{2}^{i_{2}}) \frac{h_{u}(\epsilon k_{2}) h_{b}(\epsilon k_{2})}{2 \lvert k_{2} \rvert^{2} f( \epsilon k_{2})} \nonumber\\
& \times \hat{\mathcal{P}}^{ii_{1}}(k_{2}) \hat{\mathcal{P}}^{i_{2}i_{3}}(k_{2}) \hat{\mathcal{P}}^{ji_{3}}(k_{2}) ds 
\end{align} 
so that 
\begin{align}\label{limitC2bepsilonii1j}
\lim_{\epsilon \searrow 0} C_{2,b}^{\epsilon, i i_{1} j}(t) =&  \frac{ (2\pi)^{-3}}{8(a+b)} \sum_{i_{2}, i_{3} =1}^{3} \\
& \times \int_{\mathbb{R}^{3}} \frac{ [\cos(ax^{i_{2}}) - \cos(bx^{i_{2}})]}{\lvert x \rvert^{4} f(x)^{2}} h_{u}(x) h_{b}(x) \hat{\mathcal{P}}^{ii_{1}}(x) \hat{\mathcal{P}}^{i_{2}i_{3}}(x) \hat{\mathcal{P}}^{ji_{3}}(x)dx. \nonumber 
\end{align} 
Additionally, we define 
\begin{subequations} 
\begin{align}
\bar{C}_{2,u}^{\epsilon, i i_{1} j}(t) \triangleq& - \frac {(2\pi)^{-3}}{2} \sum_{i_{2}, i_{3} =1}^{3} \sum_{k_{2} \neq 0} \int_{0}^{t} e^{-2 \lvert k_{2} \rvert^{2} (t-s)} ds  \nonumber \\ 
& \times k_{2}^{i_{2}} i \frac{ h_{b}(\epsilon k_{2})^{2}}{2 \lvert k_{2} \rvert^{2}} \hat{\mathcal{P}}^{ii_{1}}(k_{2}) \hat{\mathcal{P}}^{i_{2}i_{3}} (k_{2}) \hat{\mathcal{P}}^{ji_{3}}(k_{2}) ,  \label{barC2uepsilonii1j}\\
\bar{C}_{2,b}^{\epsilon, i i_{1} j}(t) \triangleq& \frac {(2\pi)^{-3}}{2} \sum_{i_{2}, i_{3} =1}^{3} \sum_{k_{2} \neq 0}  \int_{0}^{t} e^{-2 \lvert k_{2} \rvert^{2} (t-s)} ds   \nonumber \\ 
& \times k_{2}^{i_{2}} i \frac{h_{u} (\epsilon k_{2}) h_{b}(\epsilon k_{2})}{2 \lvert k_{2} \rvert^{2}} \hat{\mathcal{P}}^{ii_{1}}(k_{2}) \hat{\mathcal{P}}^{i_{2}i_{3}} (k_{2}) \hat{\mathcal{P}}^{ji_{3}}(k_{2}), \label{barC2bepsilonii1j}  
\end{align} 
\end{subequations} 
which are both zero. As we will see in \eqref{estimate 1062}, such $C_{2,u}^{\epsilon, ii_{1} j}, C_{2,b}^{\epsilon, ii_{1} j}, \bar{C}_{2,u}^{\epsilon, ii_{1} j}, \bar{C}_{2,b}^{\epsilon, ii_{1} j}$ had to be carefully selected to make the necessary estimates work. Now within \eqref{estimate 137} we compute 
\begin{align}\label{estimate 143}
\mathbb{E} [ \lvert \Delta_{q} (V_{t, ii_{1}}^{2} - \tilde{V}_{t, ii_{1}}^{2}) \rvert^{2}] \lesssim \sum_{l=1}^{2} D_{q,t, ij}^{l} 
\end{align} 
by \eqref{estimate 135}, \eqref{estimate 136} where 
\begin{subequations}\label{estimate 138}
\begin{align}
D_{q,t,ij}^{1} \triangleq& 
\mathbb{E} [ \lvert \sum_{i_{1}, i_{2}, i_{3} =1}^{3} \sum_{k_{1}, k_{2} \neq 0} \theta(2^{-q} k_{1}) \\
& \times [ \int_{0}^{t} e^{- \lvert k_{12} \rvert^{2} f( \epsilon k_{12}) (t-s)} k_{12}^{i_{2}} g(\epsilon k_{12}^{i_{2}}) \frac{ e^{- \lvert k_{2} \rvert^{2} f(\epsilon k_{2}) (t-s)} h_{u}(\epsilon k_{2}) h_{b}(\epsilon k_{2})}{ \lvert k_{2} \rvert^{2} f(\epsilon k_{2})} \nonumber\\
& \hspace{10mm} \times [\hat{X}_{s,b}^{\epsilon, i_{1}}(k_{1}) - \hat{X}_{t,b}^{\epsilon, i_{1}}(k_{1})] ds \nonumber \\
& - \int_{0}^{t} e^{- \lvert k_{12} \rvert^{2} (t-s)} k_{12}^{i_{2}} i \frac{ e^{- \lvert k_{2} \rvert^{2} (t-s)} h_{u}(\epsilon k_{2}) h_{b}(\epsilon k_{2})}{ \lvert k_{2} \rvert^{2} } \nonumber\\
& \hspace{10mm}  \times [\hat{\bar{X}}_{s,b}^{\epsilon, i_{1}}(k_{1}) - \hat{\bar{X}}_{t,b}^{\epsilon, i_{1}}(k_{1})] ds] \hat{\mathcal{P}}^{ii_{1}}(k_{12}) \hat{\mathcal{P}}^{i_{2}i_{3}}(k_{2}) \hat{\mathcal{P}}^{ji_{3}}(k_{2}) e_{k_{1}}\rvert^{2} ], \nonumber\\
D_{q,t,ij}^{2} \triangleq&
\mathbb{E} [ \lvert \sum_{i_{1}, i_{2}, i_{3} =1}^{3} \sum_{k_{1}, k_{2} \neq 0} \theta(2^{-q} k_{1}) \\
& \times [ \int_{0}^{t} e^{- \lvert k_{12} \rvert^{2} f( \epsilon k_{12}) (t-s)} k_{12}^{i_{2}} g(\epsilon k_{12}^{i_{2}}) \frac{ e^{- \lvert k_{2} \rvert^{2} f(\epsilon k_{2}) (t-s)} h_{b}(\epsilon k_{2})^{2}}{ \lvert k_{2} \rvert^{2} f(\epsilon k_{2})} \nonumber\\
& \hspace{10mm} \times  [\hat{X}_{s,u}^{\epsilon, i_{1}}(k_{1}) - \hat{X}_{t,u}^{\epsilon, i_{1}}(k_{1})] ds \nonumber \\
& - \int_{0}^{t} e^{- \lvert k_{12} \rvert^{2} (t-s)} k_{12}^{i_{2}} i \frac{ e^{- \lvert k_{2} \rvert^{2} (t-s)} h_{b}(\epsilon k_{2})^{2}}{ \lvert k_{2} \rvert^{2} } \nonumber\\
& \hspace{10mm} \times [\hat{\bar{X}}_{s,u}^{\epsilon, i_{1}}(k_{1}) - \hat{\bar{X}}_{t,u}^{\epsilon, i_{1}}(k_{1})] ds] \hat{\mathcal{P}}^{ii_{1}}(k_{12}) \hat{\mathcal{P}}^{i_{2}i_{3}}(k_{2}) \hat{\mathcal{P}}^{ji_{3}}(k_{2}) e_{k_{1}}\rvert^{2} ].  \nonumber 
\end{align}
\end{subequations} 

\begin{remark}\label{another remark}
Here we strategically paired up the two terms with $h_{u}(\epsilon k_{2}) h_{b}(\epsilon k_{2})$ in $D_{q,t, ij}^{1}$ and the other two terms with $h_{b}(\epsilon k_{2})^{2}$ in $D_{q,t, ij}^{2}$. Such an issue does not arise in the case of the NS equations which only has $h(\epsilon k_{2})^{2}$. 
\end{remark}

W.l.o.g. we work on $D_{q,t,ij}^{1}$ as the estimates on $D_{q,t,ij}^{2}$ are similar. We define for $k_{1} \neq 0$, 
\begin{subequations}
\begin{align}
d_{k_{1}, t-s, ii_{1}i_{2}i_{3}j}& \triangleq \sum_{k_{2} \neq 0} e^{- \lvert k_{12} \rvert^{2} f( \epsilon k_{12}) (t-s)} k_{12}^{i_{2}} g(\epsilon k_{12}^{i_{2}}) \label{[Equation (4.1na)][ZZ17]}  \\
& \times \frac{ e^{- \lvert k_{2} \rvert^{2} f(\epsilon k_{2}) (t-s)} h_{u}(\epsilon k_{2}) h_{b}(\epsilon k_{2})}{ \lvert k_{2} \rvert^{2} f(\epsilon k_{2})} \hat{\mathcal{P}}^{ii_{1}}(k_{12}) \hat{\mathcal{P}}^{i_{2}i_{3}}(k_{2}) \hat{\mathcal{P}}^{ji_{3}}(k_{2}), \nonumber\\
\bar{d}_{k_{1}, t-s, ii_{1}i_{2}i_{3}j}& (t-s) \triangleq \sum_{k_{2} \neq 0} e^{- \lvert k_{12} \rvert^{2} (t-s)} k_{12}^{i_{2}} i \label{[Equation (4.1nb)][ZZ17]} \\
& \times \frac{ e^{- \lvert k_{2} \rvert^{2} (t-s)} h_{u}(\epsilon k_{2}) h_{b}(\epsilon k_{2})}{ \lvert k_{2} \rvert^{2}}\hat{\mathcal{P}}^{ii_{1}}(k_{12}) \hat{\mathcal{P}}^{i_{2}i_{3}}(k_{2}) \hat{\mathcal{P}}^{ji_{3}}(k_{2}), \nonumber
\end{align}
\end{subequations} 
so that we may compute from \eqref{estimate 138}
\begin{align}\label{estimate 142}
D_{q, t, ij}^{1} \lesssim D_{q,t, ij}^{11} + D_{q,t, ij}^{12} 
\end{align}
by \eqref{[Equation (4.1na)][ZZ17]} -\eqref{[Equation (4.1nb)][ZZ17]} where 
\begin{subequations}\label{estimate 139}
\begin{align}
D_{q, t, ij}^{11} \triangleq& \mathbb{E} [ \lvert \sum_{i_{1}, i_{2}, i_{3} =1}^{3} \sum_{k_{1} \neq 0} \theta(2^{-q} k_{1}) e_{k_{1}} \int_{0}^{t} [ d_{k_{1}, t-s, ii_{1}i_{2}i_{3} j}- \bar{d}_{k_{1}, t-s, ii_{1}i_{2} i_{3} j}]\nonumber \\
& \times [ \hat{X}_{s,b}^{\epsilon, i_{1}}(k_{1}) - \hat{X}_{t,b}^{\epsilon, i_{1}}(k_{1})] ds \rvert^{2}],  \\
D_{q,t,ij}^{12} \triangleq& \mathbb{E} [ \lvert \sum_{i_{1}, i_{2}, i_{3} =1}^{3} \sum_{k_{1} \neq 0} \theta(2^{-q} k_{1}) e_{k_{1}}  \int_{0}^{t} \bar{d}_{k_{1}, t-s, ii_{1}i_{2} i_{3} j} \nonumber\\
& \times [ \hat{X}_{s,b}^{\epsilon, i_{1}}(k_{1}) - \hat{X}_{t,b}^{\epsilon, i_{1}}(k_{1}) - \hat{\bar{X}}_{s,b}^{\epsilon, i_{1}} (k_{1}) + \hat{\bar{X}}_{t,b}^{\epsilon, i_{1}}(k_{1}) ] ds \rvert^{2} ]. 
\end{align}
\end{subequations} 
In order to compute $D_{q,t,ij}^{11}$, we first notice that for any $\eta \in [0,1]$, 
\begin{align}\label{[Equation (4.3a)][ZZ17]}
& \lvert d_{k_{1}, t-s, ii_{1} i_{2} i_{3} j} - \bar{d}_{k_{1}, t-s, ii_{1} i_{2} i_{3} j}\rvert \nonumber\\
\lesssim& \epsilon^{\frac{\eta}{2}} \sum_{k_{2} \neq 0} \frac{ ( \lvert k_{12} \rvert^{\frac{\eta}{2}} + \lvert k_{2} \rvert^{\frac{\eta}{2}}) \lvert k_{12} \rvert}{\lvert k_{2} \rvert^{2}} e^{- \lvert k_{12}\rvert^{2} \bar{c}_{f} (t-s)} e^{- \lvert k_{2} \rvert^{2} \bar{c}_{f} (t-s)} 
\end{align} 
by \eqref{[Equation (4.2)][ZZ17]}-\eqref{[Equation (4.3)][ZZ17]} and \eqref{lower bound}, and second we see that
\begin{align}\label{[Equation (4.3b)][ZZ17]} 
& \mathbb{E} [ ( \hat{X}_{s,b}^{\epsilon, i_{1}}(k_{1}) - \hat{X}_{t,b}^{\epsilon, i_{1}}(k_{1})) \overline{ (\hat{X}_{\bar{s}, b}^{\epsilon, i_{1}'} (k_{1}) - \hat{X}_{t,b}^{\epsilon, i_{1}'}(k_{1} ) ) } ] \\
\leq&
\left( \sum_{j_{1} =1}^{3} \frac{ 1- e^{- \lvert k_{1} \rvert^{2} f(\epsilon k_{1}) \lvert t-s \rvert}}{\lvert k_{1} \rvert^{2} f(\epsilon k_{1})} h_{b}(\epsilon k_{1})^{2} \hat{\mathcal{P}}^{i_{1} j_{1}}(k_{1})^{2} \right)^{\frac{1}{2}} \nonumber\\
& \times \left( \sum_{j_{1} =1}^{3} \frac{ 1- e^{- \lvert k_{1} \rvert^{2} f(\epsilon k_{1}) \lvert t-\bar{s} \rvert}}{\lvert k_{1} \rvert^{2} f(\epsilon k_{1})} h_{b}(\epsilon k_{1})^{2} \hat{\mathcal{P}}^{i_{1}' j_{1}}(k_{1})^{2} \right)^{\frac{1}{2}}  \lesssim \frac{ \lvert k_{1} \rvert^{2\eta} \lvert t-s \rvert^{\frac{\eta}{2}} \lvert t- \bar{s} \rvert^{\frac{\eta}{2}}}{\lvert k_{1} \rvert^{2}} \nonumber
\end{align} 
by H$\ddot{\mathrm{o}}$lder's inequality, \eqref{covariance a}, \eqref{lower bound} and mean value theorem.  Applying \eqref{[Equation (4.3a)][ZZ17]}-\eqref{[Equation (4.3b)][ZZ17]} to \eqref{estimate 139} shows that for $\epsilon \in (0,\eta)$, 
\begin{align}\label{estimate 141}
D_{q,t,ij}^{11}\lesssim& \sum_{i_{1}, i_{2}, i_{3}, i_{1}', i_{2}', i_{3}' = 1}^{3} \sum_{k_{1} \neq 0} \theta(2^{-q} k_{1})^{2}   \\
& \times \int_{[0,t]^{2}} [ d_{k_{1}, t-s, ii_{1}i_{2}i_{3} j} - \bar{d}_{k_{1}, t-s, ii_{1}i_{2}i_{3} j} ][ d_{k_{1}, t-\bar{s}, ii_{1}'i_{2}'i_{3}' j} - \bar{d}_{k_{1}, t-\bar{s}, ii_{1}'i_{2}'i_{3}' j}  ] \nonumber\\
& \times \mathbb{E} [ ( \hat{X}_{s,b}^{\epsilon, i_{1}}(k_{1}) - \hat{X}_{t,b}^{\epsilon, i_{1}}(k_{1}) ) ( \hat{X}_{\bar{s}, b}^{\epsilon, i_{1}'}(k_{1}) - \hat{X}_{t,b}^{\epsilon, i_{1}'}(k_{1}) ) ] ds d \bar{s} \lesssim \epsilon^{\eta} t^{\frac{\eta - \epsilon}{2}} 2^{q(1+ 2\eta)} \nonumber 
\end{align}
by \eqref{key estimate}.  Next, we compute $D_{q, t, ij}^{12}$ as follows: 
\begin{align}\label{estimate 140}
D_{q, t, ij}^{12} &\lesssim  \sum_{i_{1}, i_{2}, i_{3}, i_{1}', i_{2}', i_{3}' = 1}^{3} \sum_{k_{1} \neq 0} \theta(2^{-q} k_{1})^{2} \int_{[0,t]^{2}} \\
&\times [ \sum_{k_{2} \neq 0} e^{-\lvert k_{12} \rvert^{2} (t-s)} k_{12}^{i_{2}} i \frac{ e^{- \lvert k_{2} \rvert^{2} (t-s)} h_{u}(\epsilon k_{2}) h_{b}(\epsilon k_{2})}{\lvert k_{2} \rvert^{2}} \hat{\mathcal{P}}^{ii_{1}}(k_{12}) \hat{\mathcal{P}}^{i_{2}i_{3}}(k_{2}) \hat{\mathcal{P}}^{ji_{3}}(k_{2})] \nonumber\\
& \times [ \sum_{k_{3} \neq 0} e^{-\lvert k_{13} \rvert^{2} (t-\bar{s})} k_{13}^{i_{2}'} i \frac{ e^{- \lvert k_{3} \rvert^{2} (t-\bar{s})} h_{u}(\epsilon k_{3}) h_{b}(\epsilon k_{3})}{\lvert k_{3} \rvert^{2}} \hat{\mathcal{P}}^{ii_{1}'}(k_{13}) \hat{\mathcal{P}}^{i_{2}'i_{3}'}(k_{3}) \hat{\mathcal{P}}^{ji_{3}'}(k_{3})] \nonumber\\
& \times ( \mathbb{E} [\lvert \hat{X}_{s,b}^{\epsilon, i_{1}}(k_{1}) - \hat{X}_{t,b}^{\epsilon, i_{1}}(k_{1}) - \hat{\bar{X}}_{s,b}^{\epsilon, i_{1}}(k_{1}) + \hat{\bar{X}}_{t,b}^{\epsilon, i_{1}}(k_{1}) \rvert^{2} ])^{\frac{1}{2}} \nonumber\\
& \times (\mathbb{E} [ \lvert \hat{X}_{\bar{s},b}^{\epsilon, i_{1}'}(k_{1}) - \hat{X}_{t,b}^{\epsilon, i_{1}'}(k_{1}) - \hat{\bar{X}}_{\bar{s},b}^{\epsilon, i_{1}'}(k_{1}) + \hat{\bar{X}}_{t,b}^{\epsilon, i_{1}'}(k_{1}) \rvert^{2}])^{\frac{1}{2}} ds d \bar{s}\lesssim \epsilon^{\eta} t^{\frac{\eta - \epsilon}{2}} 2^{q(1+ 2\eta)} \nonumber
\end{align} 
by \eqref{estimate 139}, \eqref{[Equation (4.1nb)][ZZ17]}, H$\ddot{\mathrm{o}}$lder's inequality and \eqref{key estimate}. Applying \eqref{estimate 141}-\eqref{estimate 140} to \eqref{estimate 142} and \eqref{estimate 143} leads to 
\begin{equation}\label{estimate 144}
\mathbb{E} [ \lvert \Delta_{q} (V_{t, ii_{1}}^{2} - \tilde{V}_{t, ii_{1}}^{2} )\rvert^{2}] \lesssim \epsilon^{\eta} t^{\frac{\eta - \epsilon}{2}} 2^{q(1+ 2 \eta)}.
\end{equation} 
Next, within \eqref{estimate 137} we compute 
\begin{align}\label{estimate 1062}
& \mathbb{E} [ \lvert \Delta_{q} [ \tilde{V}_{t, ii_{1}}^{2} - \sum_{i_{1} =1}^{3} (C_{2,u}^{\epsilon, ii_{1} j} X_{t,u}^{\epsilon, i_{1}} + C_{2,b}^{\epsilon, ii_{1} j} X_{t,b}^{\epsilon, i_{1}}) + \sum_{i_{1} =1}^{3} ( \bar{C}_{2,u}^{\epsilon, ii_{1} j} \bar{X}_{t,u}^{\epsilon, i_{1}} + \bar{C}_{2,b}^{\epsilon, i i_{1} j} \bar{X}_{t, b}^{\epsilon, i_{1}}) ] \rvert^{2}] \nonumber\\
\lesssim& \sum_{k_{1} \neq 0} \theta(2^{-q} k_{1})^{2} \nonumber\\
& \times ( \sum_{i, j = 1}^{3} \mathbb{E} [ \lvert \hat{X}_{t,b}^{\epsilon, i} (k_{1}) \overline{\hat{X}_{t, b}^{\epsilon, j}(k_{1})} \rvert] h_{u}(\epsilon k_{1})^{2} h_{b}(\epsilon k_{1})^{2} \nonumber\\
& \hspace{2mm} \times \lvert \sum_{i_{1}, i_{2} = 1}^{3} \sum_{k_{2} \neq 0} \int_{0}^{t} \frac{e^{- \lvert k_{2} \rvert^{2} f(\epsilon k_{2}) (t-s)}}{\lvert k_{2} \rvert^{2} f(\epsilon k_{2})} \nonumber\\
& \hspace{7mm} \times  [ e^{- \lvert k_{12} \rvert^{2} f(\epsilon k_{12}) (t-s)} k_{12}^{i_{2}} g(\epsilon k_{12}^{i_{2}}) \hat{\mathcal{P}}^{ii_{1}}(k_{12}) - e^{- \lvert k_{2} \rvert^{2} f(\epsilon k_{2}) (t-s) } k_{2}^{i_{2}} g(\epsilon k_{2}^{i_{2}}) \hat{\mathcal{P}}^{ii_{1}}(k_{2})] \nonumber\\
& \hspace{5mm} - \frac{e^{- \lvert k_{2} \rvert^{2} (t-s)}}{ \lvert k_{2} \rvert^{2}} [ e^{- \lvert k_{12} \rvert^{2}(t-s)} k_{12}^{i_{2}} i \hat{\mathcal{P}}^{ii_{1}}(k_{12}) - e^{- \lvert k_{2} \rvert^{2}(t-s)} k_{2}^{i_{2}} i \hat{\mathcal{P}}^{ii_{1}}(k_{2})] ds \rvert^{2} \nonumber \\
& + \sum_{i=1}^{3} \mathbb{E} [ \lvert \hat{X}_{t,b}^{\epsilon, i} (k_{1}) - \hat{\bar{X}}_{t,b}^{\epsilon, i} (k_{1}) \rvert^{2}] h_{u}(\epsilon k_{1})^{2} h_{b}(\epsilon k_{1})^{2} \nonumber\\
& \hspace{2mm} \times \lvert \sum_{i_{1}, i_{2} =1}^{3} \sum_{k_{2} \neq 0} \int_{0}^{t} \frac{e^{- \lvert k_{2} \rvert^{2} (t-s)}}{\lvert k_{2} \rvert^{2}} [ e^{- \lvert k_{12} \rvert^{2} (t-s)} k_{12}^{i_{2}} i \hat{\mathcal{P}}^{ii_{1}}(k_{12}) \nonumber\\
& \hspace{10mm}  - e^{- \lvert k_{2} \rvert^{2} (t-s)} k_{2}^{i_{2}} i \hat{\mathcal{P}}^{ii_{1}}(k_{2})] ds \rvert^{2} \nonumber\\
& + \sum_{i, j = 1}^{3} \mathbb{E} [ \lvert \hat{X}_{t,u}^{\epsilon, i} (k_{1}) \overline{\hat{X}_{t, u}^{\epsilon, j}(k_{1})} \rvert] h_{b}(\epsilon k_{1})^{4}  \nonumber\\
& \hspace{2mm} \times \lvert \sum_{i_{1}, i_{2} = 1}^{3} \sum_{k_{2} \neq 0} \int_{0}^{t}\frac{e^{- \lvert k_{2} \rvert^{2} f(\epsilon k_{2}) (t-s)}}{ \lvert k_{2} \rvert^{2} f(\epsilon k_{2})}  \nonumber\\
& \hspace{7mm} \times [ e^{- \lvert k_{12} \rvert^{2} f(\epsilon k_{12})(t-s) } k_{12}^{i_{2}} g(\epsilon k_{12}^{i_{2}}) \hat{\mathcal{P}}^{ii_{1}}(k_{12}) - e^{- \lvert k_{2} \rvert^{2} f(\epsilon k_{2}) (t-s) } k_{2}^{i_{2}} g(\epsilon k_{2}^{i_{2}}) \hat{\mathcal{P}}^{ii_{1}}(k_{2})] \nonumber\\
& \hspace{5mm} - \frac{e^{- \lvert k_{2} \rvert^{2} (t-s)}}{ \lvert k_{2} \rvert^{2}} [ e^{- \lvert k_{12} \rvert^{2}(t-s)} k_{12}^{i_{2}} i \hat{\mathcal{P}}^{ii_{1}}(k_{12}) - e^{- \lvert k_{2} \rvert^{2}(t-s)} k_{2}^{i_{2}} i \hat{\mathcal{P}}^{ii_{1}}(k_{2})] ds \rvert^{2} \nonumber \\
& + \sum_{i=1}^{3} \mathbb{E} [ \lvert \hat{X}_{t,u}^{\epsilon, i} (k_{1}) - \hat{\bar{X}}_{t,u}^{\epsilon, i} (k_{1}) \rvert^{2}] h_{b}(\epsilon k_{1})^{4}  \nonumber\\
& \hspace{2mm} \lvert \sum_{i_{1}, i_{2} =1}^{3} \sum_{k_{2} \neq 0} \int_{0}^{t} \frac{e^{- \lvert k_{2} \rvert^{2} (t-s)}}{\lvert k_{2} \rvert^{2}}  \times [ e^{- \lvert k_{12} \rvert^{2} (t-s)} k_{12}^{i_{2}} i \hat{\mathcal{P}}^{ii_{1}}(k_{12})  \nonumber\\
& \hspace{60mm} - e^{- \lvert k_{2} \rvert^{2} (t-s)} k_{2}^{i_{2}} i \hat{\mathcal{P}}^{ii_{1}}(k_{2})] ds \rvert^{2}) 
\end{align} 
by \eqref{estimate 136}, \eqref{C2uepsilonii1j}, \eqref{C2bepsilonii1j}, \eqref{barC2uepsilonii1j} and \eqref{barC2bepsilonii1j}.  Now we can immediately estimate from \eqref{covariance a} that for any $i, j \in \{1,2,3\}$, 
\begin{equation}\label{estimate 6}
\mathbb{E} [ \lvert\hat{X}_{t,b}^{\epsilon, i} (k_{1}) \overline{\hat{X}_{t,b}^{\epsilon, j} (k_{1})} \rvert ] \lesssim \frac{ h_{b}(\epsilon k_{1})^{2}}{\lvert k_{1} \rvert^{2}} \text{ and } \mathbb{E} [ \lvert\hat{X}_{t,u}^{\epsilon, i} (k_{1}) \overline{\hat{X}_{t,u}^{\epsilon, j} (k_{1})} \rvert ] \lesssim \frac{ h_{u}(\epsilon k_{1})^{2}}{\lvert k_{1} \rvert^{2}}
\end{equation} 
while for any $\eta \in [0,1]$, 
\begin{align}\label{[Equation (4.4a)][ZZ17]}
& \mathbb{E} [ \lvert \hat{X}_{t,b}^{\epsilon, i} (k_{1}) - \hat{\bar{X}}_{t,b}^{\epsilon, i}(k_{1}) \rvert^{2}] \nonumber\\
=&  \sum_{i_{1} =1}^{3}  \frac{ h_{b}(\epsilon k_{1})^{2}}{\lvert k_{1} \rvert^{2}} \hat{\mathcal{P}}^{ii_{1}}(k_{1})^{2} \frac{1}{2 f(\epsilon k_{1}) (f(\epsilon k_{1}) + 1)} (f(\epsilon k_{1}) - 1)^{2}  \lesssim \frac{ \lvert \epsilon k_{1} \rvert^{\eta}}{\lvert k_{1} \rvert^{2}} 
\end{align} 
by \eqref{covariance c}, \eqref{covariance f} and \eqref{estimate 145} and mean value theorem, and similarly 
\begin{equation}\label{estimate 146} 
\mathbb{E} [ \lvert \hat{X}_{t,u}^{\epsilon, i} (k_{1}) - \hat{\bar{X}}_{t,u}^{\epsilon, i}(k_{1}) \rvert^{2}] \lesssim \frac{ \lvert \epsilon k_{1} \rvert^{\eta}}{\lvert k_{1} \rvert^{2}}. 
\end{equation} 
Applying \eqref{estimate 6} - \eqref{estimate 146} to \eqref{estimate 1062} leads to  
\begin{align}\label{estimate 152}
& \mathbb{E} [ \lvert \Delta_{q} [ \tilde{V}_{t, ii_{1}}^{2} - \sum_{i_{1} =1}^{3} (C_{2,u}^{\epsilon, ii_{1} j} X_{t,u}^{\epsilon, i_{1}} + C_{2,b}^{\epsilon, ii_{1}j} X_{t,b}^{\epsilon, i_{1}}) \\
& \hspace{10mm}  + \sum_{i_{1} =1}^{3} ( \bar{C}_{2,u}^{\epsilon, i i_{1} j} \bar{X}_{t,u}^{\epsilon, i_{1}} + \bar{C}_{2,b}^{\epsilon, i i_{1} j} \bar{X}_{t, b}^{\epsilon, i_{1}}) ] \rvert^{2}]  \lesssim \sum_{l=1}^{2} E_{q, t, i}^{l} \nonumber
\end{align} 
where 
\begin{subequations}\label{estimate 149}
\begin{align}
E_{q,t, i}^{1} \triangleq&  \sum_{k_{1} \neq 0} \frac{\theta(2^{-q} k_{1})^{2}}{ \lvert k_{1} \rvert^{2}} \lvert \sum_{i_{1}, i_{2} = 1}^{3} \sum_{k_{2} \neq 0}  \\
&  \times \int_{0}^{t} \frac{e^{- \lvert k_{2} \rvert^{2} f(\epsilon k_{2}) (t-s)}}{ \lvert k_{2} \rvert^{2} f(\epsilon k_{2})} [ e^{- \lvert k_{12} \rvert^{2} f(\epsilon k_{12}) (t-s)} k_{12}^{i_{2}} g(\epsilon k_{12}^{i_{2}}) \hat{\mathcal{P}}^{ii_{1}}(k_{12}) \nonumber\\
& \hspace{30mm} - e^{- \lvert k_{2} \rvert^{2} f(\epsilon k_{2}) (t-s) } k_{2}^{i_{2}} g(\epsilon k_{2}^{i_{2}}) \hat{\mathcal{P}}^{ii_{1}}(k_{2})] \nonumber\\
& - \frac{e^{- \lvert k_{2} \rvert^{2} (t-s)}}{ \lvert k_{2} \rvert^{2}} [ e^{- \lvert k_{12} \rvert^{2}(t-s)} k_{12}^{i_{2}} i \hat{\mathcal{P}}^{ii_{1}}(k_{12}) - e^{- \lvert k_{2} \rvert^{2}(t-s)} k_{2}^{i_{2}} i \hat{\mathcal{P}}^{ii_{1}}(k_{2})] ds \rvert^{2}, \nonumber \\
E_{q,t,i}^{2}\triangleq&  \epsilon^{\eta} \sum_{k_{1} \neq 0} \frac{ \theta(2^{-q} k_{1})^{2}}{\lvert k_{1} \rvert^{2-\eta}}  \lvert \sum_{i_{1}, i_{2} =1}^{3} \sum_{k_{2} \neq 0}  \\
& \times \int_{0}^{t} \frac{e^{- \lvert k_{2} \rvert^{2} (t-s)}}{\lvert k_{2} \rvert^{2}} [ e^{- \lvert k_{12} \rvert^{2} (t-s)} k_{12}^{i_{2}} i \hat{\mathcal{P}}^{ii_{1}}(k_{12}) - e^{- \lvert k_{2} \rvert^{2} (t-s)} k_{2}^{i_{2}} i \hat{\mathcal{P}}^{ii_{1}}(k_{2})] ds \rvert^{2}.  \nonumber
\end{align}
\end{subequations} 
In order to estimate $E_{q,t, i}^{1}$, we see that applications of Lemmas \ref{Lemma 3.8} and \ref{Lemma 3.9} lead to 
\begin{align}\label{estimate 147}
& \lvert \frac{e^{- \lvert k_{2} \rvert^{2} f(\epsilon k_{2}) (t-s)}}{ \lvert k_{2} \rvert^{2} f(\epsilon k_{2})} [ e^{- \lvert k_{12} \rvert^{2} f(\epsilon k_{12}) (t-s)} k_{12}^{i_{2}} g(\epsilon k_{12}^{i_{2}}) \hat{\mathcal{P}}^{ii_{1}}(k_{12}) \nonumber\\
& \hspace{30mm} - e^{- \lvert k_{2} \rvert^{2} f(\epsilon k_{2}) (t-s) } k_{2}^{i_{2}} g(\epsilon k_{2}^{i_{2}}) \hat{\mathcal{P}}^{ii_{1}}(k_{2})] \nonumber\\
& - \frac{e^{- \lvert k_{2} \rvert^{2} (t-s)}}{ \lvert k_{2} \rvert^{2}} [ e^{- \lvert k_{12} \rvert^{2}(t-s)} k_{12}^{i_{2}} i \hat{\mathcal{P}}^{ii_{1}}(k_{12}) - e^{- \lvert k_{2} \rvert^{2}(t-s)} k_{2}^{i_{2}} i \hat{\mathcal{P}}^{ii_{1}}(k_{2})] \rvert \nonumber\\
\lesssim& \frac{ e^{- \lvert k_{2} \rvert^{2} \bar{c}_{f} ( t-s )}}{\lvert k_{2} \rvert^{2}} \lvert k_{1} \rvert^{\eta} \lvert t-s \rvert^{- \frac{1-\eta}{2}} 
\end{align} 
for any $\eta \in (0, 1)$, while \eqref{[Equation (4.2)][ZZ17]}, \eqref{[Equation (4.3)][ZZ17]} and \eqref{key estimate} lead to 
\begin{align}\label{estimate 148}
& \lvert \frac{e^{- \lvert k_{2} \rvert^{2} f(\epsilon k_{2}) (t-s)}}{ \lvert k_{2} \rvert^{2} f(\epsilon k_{2})} [ e^{- \lvert k_{12} \rvert^{2} f(\epsilon k_{12}) (t-s)} k_{12}^{i_{2}} g(\epsilon k_{12}^{i_{2}}) \hat{\mathcal{P}}^{ii_{1}}(k_{12}) \nonumber\\
& \hspace{30mm} - e^{- \lvert k_{2} \rvert^{2} f(\epsilon k_{2}) (t-s) } k_{2}^{i_{2}} g(\epsilon k_{2}^{i_{2}}) \hat{\mathcal{P}}^{ii_{1}}(k_{2})] \nonumber\\
& - \frac{e^{- \lvert k_{2} \rvert^{2} (t-s)}}{ \lvert k_{2} \rvert^{2}} [ e^{- \lvert k_{12} \rvert^{2}(t-s)} k_{12}^{i_{2}} i \hat{\mathcal{P}}^{ii_{1}}(k_{12}) - e^{- \lvert k_{2} \rvert^{2}(t-s)} k_{2}^{i_{2}} i \hat{\mathcal{P}}^{ii_{1}}(k_{2})] \rvert \nonumber\\
\lesssim& \epsilon^{\frac{\eta}{2}} \frac{ e^{- \lvert k_{2} \rvert^{2} \bar{c}_{f} (t-s)}}{\lvert k_{2} \rvert^{2}} ( \lvert t-s \rvert^{-\frac{1}{2}} \lvert k_{2} \rvert^{\frac{\eta}{2}} + \lvert t-s \rvert^{-\frac{1}{2} - \frac{\eta}{4}}). 
\end{align} 
Therefore, applying \eqref{estimate 147}-\eqref{estimate 148} to \eqref{estimate 149} gives for any $\epsilon \in (0, \eta)$,  
\begin{align}\label{estimate 150} 
E_{q,t,i}^{1} \lesssim& \sum_{k_{1} \neq 0} \frac{ \theta(2^{-q} k_{1})^{2}}{\lvert k_{1} \rvert^{2}} \lvert \sum_{k_{2} \neq 0} \int_{0}^{t} \frac{ e^{- \lvert k_{2} \rvert^{2} \bar{c}_{f} (t-s)}}{\lvert k_{2} \rvert^{2}} \lvert k_{1} \rvert^{\frac{\eta}{2}} \lvert t-s \rvert^{-\frac{1}{4} + \frac{\eta}{4} - \frac{1}{4}} \epsilon^{\frac{\eta}{4}} \lvert k_{2} \rvert^{\frac{\eta}{4}} ds \\
& \hspace{10mm} + \int_{0}^{t} \frac{ e^{- \lvert k_{2} \rvert^{2} \bar{c}_{f} (t-s)}}{\lvert k_{2} \rvert^{2}} \lvert k_{1} \rvert^{\frac{\eta}{2}} \lvert t-s \rvert^{-\frac{1}{4} + \frac{\eta}{4} - \frac{1}{4} - \frac{\eta}{8}} \epsilon^{\frac{\eta}{4}} ds \rvert^{2} \lesssim \epsilon^{\frac{\eta}{2}} t^{\frac{\eta - \epsilon}{4}} 2^{q(1+ \eta)} \nonumber
\end{align} 
by \eqref{key estimate}. Next, we can estimate from \eqref{estimate 149} for any $\eta \in (0,1)$ and $\epsilon \in (0, \eta)$ 
\begin{align}\label{estimate 151}
E_{q,t,i}^{2} \lesssim \epsilon^{\eta} \sum_{k_{1} \neq 0} \frac{ \theta(2^{-q} k_{1})^{2}}{\lvert k_{1} \rvert^{2- 3 \eta}} \lvert \sum_{k_{2} \neq 0} \frac{ t^{\frac{\eta - \epsilon}{2}}}{\lvert k_{2} \rvert^{3+ \epsilon}} \rvert^{2}  \lesssim \epsilon^{\eta} t^{\eta - \epsilon} 2^{q(1+ 3\eta)} 
\end{align} 
due to Lemma \ref{Lemma 3.8} and \eqref{key estimate}. Applying \eqref{estimate 150}-\eqref{estimate 151} to \eqref{estimate 152} gives 
\begin{align}\label{estimate 153}
& \mathbb{E} [ \lvert \Delta_{q} [ \tilde{V}_{t, ii_{1}}^{2} - \sum_{i_{1} =1}^{3} (C_{2,u}^{\epsilon, ii_{1} j} X_{t,u}^{\epsilon, i_{1}} + C_{2,b}^{\epsilon, ii_{1}j} X_{t,b}^{\epsilon, i_{1}}) \nonumber\\
& \hspace{10mm}  + \sum_{i_{1} =1}^{3} ( \bar{C}_{2,u}^{\epsilon, i i_{1} j} \bar{X}_{t,u}^{\epsilon, i_{1}} + \bar{C}_{2,b}^{\epsilon, i i_{1} j} \bar{X}_{t, b}^{\epsilon, i_{1}}) ] \rvert^{2}]  \lesssim \epsilon^{\frac{\eta}{2}} t^{\frac{\eta - \epsilon}{4}} 2^{q (1+ 3\eta)}. 
\end{align} 
Concerning $V_{t, ii_{1}}^{3}$ from \eqref{estimate 135} within \eqref{estimate 134}, we write 
\begin{align}\label{estimate 154}
&V_{t, ii_{1}}^{3} - \sum_{i_{1} =1}^{3} ( \tilde{C}_{2,u}^{\epsilon, i i_{1} j} X_{t,u}^{\epsilon, i_{1}} + \tilde{C}_{2,b}^{\epsilon, i i_{1} j} X_{t,b}^{\epsilon, i_{1}}) = V_{t, ii_{1}}^{3} - \tilde{V}_{t, ii_{1}}^{3}  \\
& \hspace{3mm} + \tilde{V}_{t, ii_{1}}^{3} - \sum_{i_{2} =1}^{3} (\tilde{C}_{2,u}^{\epsilon, i i_{2} j} X_{t,u}^{\epsilon, i_{2}} + \tilde{C}_{2,b}^{\epsilon, i i_{2} j} X_{t,b}^{\epsilon, i_{2}}) + \sum_{i_{2} =1}^{3} ( \bar{\tilde{C}}_{2,u}^{\epsilon, i i_{2} j} \bar{X}_{t,u}^{\epsilon, i_{2}} + \bar{\tilde{C}}_{2,b}^{\epsilon, i i_{2} j} \bar{X}_{t, b}^{\epsilon, i_{2}})  \nonumber 
\end{align} 
where similarly to \eqref{estimate 136} - \eqref{barC2bepsilonii1j}, 
\begin{align}\label{estimate 155}
\tilde{V}_{t, ii_{1}}^{3} &\triangleq \frac{(2\pi)^{-3}}{2} \sum_{i_{1}, i_{2}, i_{3} =1}^{3} \sum_{k_{1}, k_{2} \neq 0} \\
&\times [\hat{X}_{t,u}^{\epsilon, i_{2}}(k_{2}) \int_{0}^{t} e^{- \lvert k_{12} \rvert^{2} f( \epsilon k_{12})(t-s)} k_{12}^{i_{2}} g(\epsilon k_{12}^{i_{2}}) \frac{ e^{- \lvert k_{1} \rvert^{2} f(\epsilon k_{1}) (t-s)} h_{b}(\epsilon k_{1})^{2}}{2 \lvert k_{1} \rvert^{2} f(\epsilon k_{1})} ds \nonumber \\
& - \hat{X}_{t,b}^{\epsilon, i_{2}}(k_{2}) \int_{0}^{t} e^{- \lvert k_{12} \rvert^{2} f( \epsilon k_{12})(t-s)} k_{12}^{i_{2}} g(\epsilon k_{12}^{i_{2}}) \frac{ e^{- \lvert k_{1} \rvert^{2} f(\epsilon k_{1}) (t-s)} h_{u}(\epsilon k_{1})h_{b}(\epsilon k_{1})}{2 \lvert k_{1} \rvert^{2} f(\epsilon k_{1})} ds \nonumber\\
& - \hat{\bar{X}}_{t,u}^{\epsilon, i_{2}}(k_{2}) \int_{0}^{t} e^{- \lvert k_{12} \rvert^{2} (t-s)} k_{12}^{i_{2}} i \frac{ e^{- \lvert k_{1} \rvert^{2}  (t-s)} h_{b}(\epsilon k_{1})^{2}}{2 \lvert k_{1} \rvert^{2} } ds \nonumber\\
&+  \hat{\bar{X}}_{t,b}^{\epsilon, i_{2}}(k_{2}) \int_{0}^{t} e^{- \lvert k_{12} \rvert^{2} (t-s)} k_{12}^{i_{2}} i \frac{ e^{- \lvert k_{1} \rvert^{2}  (t-s)} h_{u}(\epsilon k_{1})h_{b}(\epsilon k_{1})}{2 \lvert k_{1} \rvert^{2} } ds]  \nonumber\\
& \times \hat{\mathcal{P}}^{ii_{1}}(k_{12}) \hat{\mathcal{P}}^{i_{1} i_{3}} (k_{1}) \hat{\mathcal{P}}^{ji_{3}} (k_{1}) e_{k_{2}}, \nonumber 
\end{align}  
\begin{align}\label{C2uepsilonii2j} 
\tilde{C}_{2,u}^{\epsilon, i i_{2} j} (t) \triangleq& \frac{ (2\pi)^{-3}}{2} \sum_{i_{1}, i_{3} =1}^{3} \sum_{k_{1} \neq 0} \int_{0}^{t} e^{-2 \lvert k_{1} \rvert^{2} f( \epsilon k_{1}) (t-s)} k_{1}^{i_{2}} g(\epsilon k_{1}^{i_{2}}) \frac{ h_{b}(\epsilon k_{1})^{2}}{2 \lvert k_{1} \rvert^{2} f( \epsilon k_{1})} \\
& \times \hat{\mathcal{P}}^{ii_{1}}(k_{1}) \hat{\mathcal{P}}^{i_{1}i_{3}}(k_{1}) \hat{\mathcal{P}}^{ji_{3}}(k_{1}) ds \nonumber 
\end{align} 
so that 
\begin{align}
\lim_{\epsilon \searrow 0} \tilde{C}_{2,u}^{\epsilon, i i_{2} j}(t) =& \frac{ (2\pi)^{-3}}{8(a+b)} \sum_{i_{1}, i_{3} =1}^{3}  \\
& \times \int_{\mathbb{R}^{3}} \frac{ [\cos(ax^{i_{2}}) - \cos(bx^{i_{2}})]}{\lvert x \rvert^{4} f(x)^{2}} h_{b}(x)^{2} \hat{\mathcal{P}}^{ii_{1}}(x) \hat{\mathcal{P}}^{i_{1}i_{3}}(x) \hat{\mathcal{P}}^{ji_{3}}(x)  dx\nonumber
\end{align} 
by \eqref{g} while  
\begin{align}\label{C2bepsilonii2j}
\tilde{C}_{2,b}^{\epsilon, i i_{2} j} (t) \triangleq& -\frac{ (2\pi)^{-3}}{2} \sum_{i_{1}, i_{3} =1}^{3} \sum_{k_{1} \neq 0} \int_{0}^{t} e^{-2 \lvert k_{1} \rvert^{2} f( \epsilon k_{1}) (t-s)} k_{1}^{i_{2}} g(\epsilon k_{1}^{i_{2}}) \frac{ h_{u}(\epsilon k_{1})h_{b}(\epsilon k_{1})}{2 \lvert k_{1} \rvert^{2} f( \epsilon k_{1})} \nonumber\\
& \times \hat{\mathcal{P}}^{ii_{1}}(k_{1}) \hat{\mathcal{P}}^{i_{1}i_{3}}(k_{1}) \hat{\mathcal{P}}^{ji_{3}}(k_{1}) ds
\end{align} 
so that  
\begin{align}
\lim_{\epsilon \searrow 0} \tilde{C}_{2,b}^{\epsilon, i i_{2} j}(t) =& -\frac{ (2\pi)^{-3}}{8(a+b)} \sum_{i_{1}, i_{3} =1}^{3} \int_{\mathbb{R}^{3}} \frac{ [\cos(ax^{i_{2}}) - \cos(bx^{i_{2}})]}{\lvert x \rvert^{4} f(x)^{2}} \nonumber\\
& \times h_{u}(x)h_{b}(x) \hat{\mathcal{P}}^{ii_{1}}(x) \hat{\mathcal{P}}^{i_{1}i_{3}}(x) \hat{\mathcal{P}}^{ji_{3}}(x)  dx.
\end{align}
Additionally, we define 
\begin{subequations}\label{estimate 164}
\begin{align}
\bar{\tilde{C}}_{2,u}^{\epsilon, i i_{2} j}(t) \triangleq& \frac {(2\pi)^{-3}}{2} \sum_{i_{1}, i_{3} =1}^{3} \sum_{k_{1} \neq 0} \int_{0}^{t} e^{-2 \lvert k_{1} \rvert^{2} (t-s)} ds  \nonumber \\ 
& \times k_{1}^{i_{2}} i \frac{ h_{b}(\epsilon k_{1})^{2}}{2 \lvert k_{1} \rvert^{2}} \hat{\mathcal{P}}^{ii_{1}}(k_{1}) \hat{\mathcal{P}}^{i_{1}i_{3}} (k_{1}) \hat{\mathcal{P}}^{ji_{3}}(k_{1}) ,  \\
\bar{\tilde{C}}_{2,b}^{\epsilon, i i_{2} j}(t) \triangleq&- \frac {(2\pi)^{-3}}{2} \sum_{i_{1}, i_{3} =1}^{3} \sum_{k_{1} \neq 0}  \int_{0}^{t} e^{-2 \lvert k_{1} \rvert^{2} (t-s)} ds   \nonumber \\ 
& \times k_{1}^{i_{2}} i \frac{h_{u} (\epsilon k_{1}) h_{b}(\epsilon k_{1})}{2 \lvert k_{1} \rvert^{2}} \hat{\mathcal{P}}^{ii_{1}}(k_{1}) \hat{\mathcal{P}}^{i_{1}i_{3}} (k_{1}) \hat{\mathcal{P}}^{ji_{3}}(k_{1}),  
\end{align} 
\end{subequations} 
which are both zero. Before we proceed, for completeness, we record all of $C_{k,u}^{\epsilon, ii_{1} j}$, $C_{k,b}^{\epsilon, ii_{1} j}$, $\tilde{C}_{k,u}^{\epsilon, ii_{1} j},$ $\tilde{C}_{k,b}^{\epsilon, ii_{1} j}$ for $k \in \{1,3,4\}$: 
\begin{subequations}\label{estimate 124}
\begin{align}
& C_{1,u}^{\epsilon, ii_{1} j} = C_{4,b}^{\epsilon, ii_{1} j} = \frac{ (2\pi)^{-3}}{2} \sum_{i_{2}, i_{3} =1}^{3} \sum_{k_{2} \neq 0} \int_{0}^{t} e^{-2 \lvert k_{2} \rvert^{2} f( \epsilon k_{2}) (t-s)} k_{2}^{i_{2}} g(\epsilon k_{2}^{i_{2}}) \nonumber\\
& \hspace{40mm} \times \frac{ h_{u}(\epsilon k_{2})^{2}}{2 \lvert k_{2} \rvert^{2} f( \epsilon k_{2})} \hat{\mathcal{P}}^{ii_{1}}(k_{2}) \hat{\mathcal{P}}^{i_{2}i_{3}}(k_{2}) \hat{\mathcal{P}}^{ji_{3}}(k_{2}) ds,  \\
& C_{1,b}^{\epsilon, ii_{1} j} = C_{4,u}^{\epsilon, ii_{1} j} = -\frac{ (2\pi)^{-3}}{2} \sum_{i_{2}, i_{3} =1}^{3} \sum_{k_{2} \neq 0} \int_{0}^{t} e^{-2 \lvert k_{2} \rvert^{2} f( \epsilon k_{2}) (t-s)} k_{2}^{i_{2}} g(\epsilon k_{2}^{i_{2}}) \nonumber\\
& \hspace{40mm} \times \frac{ h_{u}(\epsilon k_{2})h_{b}(\epsilon k_{2})}{2 \lvert k_{2} \rvert^{2} f( \epsilon k_{2})} \hat{\mathcal{P}}^{ii_{1}}(k_{2}) \hat{\mathcal{P}}^{i_{2}i_{3}}(k_{2}) \hat{\mathcal{P}}^{ji_{3}}(k_{2}) ds, \\
& C_{2,u}^{\epsilon, ii_{1} j} = C_{3,b}^{\epsilon, ii_{1} j}, \hspace{1mm} C_{2,b}^{\epsilon, ii_{1} j} = C_{3,u}^{\epsilon, ii_{1} j},  \\
&\tilde{C}_{1,u}^{\epsilon, i i_{2} j} (t) = - \tilde{C}_{4,b}^{\epsilon, i i_{2} j} (t) =  \frac{ (2\pi)^{-3}}{2} \sum_{i_{1}, i_{3} =1}^{3} \sum_{k_{1} \neq 0} \int_{0}^{t} e^{-2 \lvert k_{1} \rvert^{2} f( \epsilon k_{1}) (t-s)} k_{1}^{i_{2}} g(\epsilon k_{1}^{i_{2}}) \nonumber\\
& \hspace{40mm} \times \frac{ h_{u}(\epsilon k_{1})^{2}}{2 \lvert k_{1} \rvert^{2} f( \epsilon k_{1})} \hat{\mathcal{P}}^{ii_{1}}(k_{1}) \hat{\mathcal{P}}^{i_{1}i_{3}}(k_{1}) \hat{\mathcal{P}}^{ji_{3}}(k_{1}) ds \nonumber\\
&\tilde{C}_{1,b}^{\epsilon, i i_{2} j} (t) = \tilde{C}_{2,b}^{\epsilon, i i_{2} j} (t)   = - \tilde{C}_{3,u}^{\epsilon, i i_{2} j} (t) = - \tilde{C}_{4,u}^{\epsilon, i i_{2} j} (t), \hspace{1mm} \tilde{C}_{2,u}^{\epsilon, i i_{2} j} (t) = - \tilde{C}_{3,b}^{\epsilon, i i_{2} j} (t).   
\end{align}
\end{subequations} 
For a subsequent purpose, we point out that the appropriate $\bar{C}_{3,u}^{\epsilon, ii_{1} j}(t)$ would be identical to $\bar{C}_{2,b}^{\epsilon, ii_{1} j}(t)$ in \eqref{barC2bepsilonii1j} as $C_{3,u}^{\epsilon, ii_{1} j} = C_{2,b}^{\epsilon, ii_{1} j}$. Now we compute within \eqref{estimate 154}, 
\begin{align}\label{estimate 162}
 \mathbb{E} [ \lvert \Delta_{q} (V_{t, ii_{1}}^{3} - \tilde{V}_{t, ii_{1}}^{3}) \rvert^{2} ] \lesssim \sum_{l=1}^{2} F_{q, t, ij}^{l} 
\end{align} 
by \eqref{estimate 135} and \eqref{estimate 155} where 
\begin{subequations}\label{estimate 156}
\begin{align}
F_{q,t, ij}^{1} \triangleq&  \mathbb{E} [ \lvert  \sum_{i_{1}, i_{2}, i_{3} =1}^{3} \sum_{k_{1}, k_{2} \neq 0} \theta(2^{-q} k_{2})\nonumber\\
& \times [ \int_{0}^{t} e^{- \lvert k_{12} \rvert^{2} f( \epsilon k_{12})(t-s)} k_{12}^{i_{2}} g(\epsilon k_{12}^{i_{2}}) \frac{ e^{- \lvert k_{1} \rvert^{2} f(\epsilon k_{1}) (t-s)} h_{b}(\epsilon k_{1})^{2}}{ \lvert k_{1} \rvert^{2} f(\epsilon k_{1})} \nonumber\\
& \hspace{55mm} \times [ \hat{X}_{s,u}^{\epsilon, i_{2}}(k_{2}) - \hat{X}_{t,u}^{\epsilon, i_{2}}(k_{2})] ds \nonumber\\
&-  \int_{0}^{t} e^{- \lvert k_{12} \rvert^{2} (t-s)} k_{12}^{i_{2}} i \frac{ e^{- \lvert k_{1} \rvert^{2} (t-s)} h_{b}(\epsilon k_{1})^{2}}{ \lvert k_{1} \rvert^{2}} [ \hat{\bar{X}}_{s,u}^{\epsilon, i_{2}}(k_{2}) - \hat{\bar{X}}_{t,u}^{\epsilon, i_{2}}(k_{2})] ds ]  \nonumber\\
& \times \hat{\mathcal{P}}^{ii_{1}}(k_{12}) \hat{\mathcal{P}}^{i_{1}i_{3}}(k_{1}) \hat{\mathcal{P}}^{ji_{3}}(k_{1}) e_{k_{2}}  \rvert^{2}], \\
F_{q, t, ij}^{2} \triangleq&  \mathbb{E} [ \lvert \sum_{i_{1}, i_{2}, i_{3} =1}^{3} \sum_{k_{1}, k_{2} \neq 0}\theta(2^{-q}k_{2}) \nonumber\\
& \times [ \int_{0}^{t} e^{- \lvert k_{12} \rvert^{2} f( \epsilon k_{12})(t-s)} k_{12}^{i_{2}} g(\epsilon k_{12}^{i_{2}}) \frac{ e^{- \lvert k_{1} \rvert^{2} f(\epsilon k_{1}) (t-s)} h_{u}(\epsilon k_{1}) h_{b}(\epsilon k_{1})}{ \lvert k_{1} \rvert^{2} f(\epsilon k_{1})} \nonumber\\
& \hspace{55mm} \times [ \hat{X}_{s,b}^{\epsilon, i_{2}}(k_{2}) - \hat{X}_{t,b}^{\epsilon, i_{2}}(k_{2})] ds \nonumber\\
& - \int_{0}^{t} e^{- \lvert k_{12} \rvert^{2} (t-s)} k_{12}^{i_{2}} i \frac{ e^{- \lvert k_{1} \rvert^{2}  (t-s)} h_{u}(\epsilon k_{1}) h_{b}(\epsilon k_{1})}{ \lvert k_{1} \rvert^{2}} [ \hat{\bar{X}}_{s,b}^{\epsilon, i_{2}}(k_{2}) - \hat{\bar{X}}_{t,b}^{\epsilon, i_{2}}(k_{2})] ds ]  \nonumber\\
& \times \hat{\mathcal{P}}^{ii_{1}}(k_{12}) \hat{\mathcal{P}}^{i_{1}i_{3}}(k_{1}) \hat{\mathcal{P}}^{ji_{3}}(k_{1}) e_{k_{2}}  \rvert^{2}].  
\end{align}
\end{subequations} 
W.l.o.g. we work on $F_{q, t, ij}^{1}$ as the estimates on $F_{q,t, ij}^{2}$ are similar. We define for $k_{2} \neq 0$, 
\begin{subequations}\label{estimate 157}
\begin{align}
f_{k_{2}, t-s, ii_{1}i_{2}i_{3}j} \triangleq& \sum_{k_{1} \neq 0} e^{- \lvert k_{12} \rvert^{2} f( \epsilon k_{12})(t-s)} k_{12}^{i_{2}} g(\epsilon k_{12}^{i_{2}}) \frac{ e^{- \lvert k_{1} \rvert^{2} f(\epsilon k_{1}) (t-s)} h_{b}(\epsilon k_{1})^{2}}{ \lvert k_{1} \rvert^{2} f(\epsilon k_{1})} \nonumber\\
& \times \hat{\mathcal{P}}^{ii_{1}}(k_{12}) \hat{\mathcal{P}}^{i_{1}i_{3}}(k_{1}) \hat{\mathcal{P}}^{ji_{3}}(k_{1}), \\
\bar{f}_{k_{2}, t-s, ii_{1}i_{2}i_{3}j} \triangleq& \sum_{k_{1} \neq 0} e^{- \lvert k_{12} \rvert^{2} (t-s)} k_{12}^{i_{2}} i \frac{ e^{- \lvert k_{1} \rvert^{2} (t-s)} h_{b}(\epsilon k_{1})^{2}}{ \lvert k_{1} \rvert^{2}} \nonumber\\
& \times \hat{\mathcal{P}}^{ii_{1}}(k_{12}) \hat{\mathcal{P}}^{i_{1}i_{3}}(k_{1}) \hat{\mathcal{P}}^{ji_{3}}(k_{1}), 
\end{align}
\end{subequations} 
so that we see from \eqref{estimate 156} that 
\begin{align}\label{estimate 161}
F_{q, t, ij}^{1} \lesssim \sum_{l=1}^{2} F_{q, t, ij}^{1l}
\end{align} 
where 
\begin{subequations}\label{estimate 158} 
\begin{align}
F_{q, t, ij}^{11} \triangleq& \mathbb{E} [ \lvert \sum_{i_{1}, i_{2}, i_{3} =1}^{3} \sum_{k_{2} \neq 0} \theta(2^{-q} k_{2}) e_{k_{2}} \int_{0}^{t} [ f_{k_{2}, t-s, ii_{1}i_{2}i_{3} j} - \bar{f}_{k_{2}, t-s, ii_{1}i_{2}i_{3}j} ] \nonumber\\
& \times [ \hat{X}_{s,u}^{\epsilon, i_{2}}(k_{2}) - \hat{X}_{t,u}^{\epsilon, i_{2}}(k_{2})] ds \rvert^{2}], \\
F_{q, t, ij}^{12} \triangleq& \mathbb{E} [ \lvert \sum_{i_{1}, i_{2}, i_{3} =1}^{3} \sum_{k_{2} \neq 0} \theta(2^{-q} k_{2}) e_{k_{2}} \int_{0}^{t} \bar{f}_{k_{2}, t-s, ii_{1}i_{2}i_{3}j} \nonumber\\
& \times [\hat{X}_{s,u}^{\epsilon, i_{2}}(k_{2}) - \hat{X}_{t,u}^{\epsilon, i_{2}}(k_{2}) - \hat{\bar{X}}_{s,u}^{\epsilon, i_{2}}(k_{2}) + \hat{\bar{X}}_{t,u}^{\epsilon, i_{2}}(k_{2})] ds \rvert^{2}].
\end{align} 
\end{subequations} 
In order to compute $F_{q, t, ij}^{11}$, first we see that for any $\eta \in [0,1]$, 
\begin{align}
& \lvert f_{k_{2}, t-s, ii_{1}i_{2}i_{3}j} - \bar{f}_{k_{2}, t-s, ii_{1}i_{2}i_{3}j} \rvert \nonumber\\
\lesssim& \epsilon^{\frac{\eta}{2}} \sum_{k_{1} \neq 0} e^{- (\lvert k_{12} \rvert^{2} + \lvert k_{1} \rvert^{2} ) \bar{c}_{f} (t-s)} (\lvert k_{12} \rvert^{\frac{\eta}{2}} + \lvert k_{1} \rvert^{\frac{\eta}{2}}) \frac{ \lvert k_{12} \rvert}{\lvert k_{1} \rvert^{2}} 
\end{align} 
by \eqref{estimate 157}, \eqref{[Equation (4.2)][ZZ17]}, \eqref{[Equation (4.3)][ZZ17]} and mean value theorem. Thus, relying also on \eqref{[Equation (4.3b)][ZZ17]} and \eqref{key estimate}, we deduce from \eqref{estimate 158}
\begin{align}\label{estimate 159}
F_{q, t, ij}^{11} \lesssim& \sum_{i_{1}, i_{2}, i_{3}, i_{1}', i_{2}', i_{3}' = 1}^{3} \sum_{k_{2} \neq 0} \theta(2^{-q} k_{2})^{2} \\
& \times \int_{[0,t]^{2}}  \lvert f_{k_{2}, t-s, ii_{1}i_{2}i_{3}j} - \bar{f}_{k_{2}, t-s, ii_{1}i_{2}i_{3}j} \rvert \overline{ \lvert f_{k_{2}, t- \bar{s}, ii_{1}' i_{2}' i_{3}' j} - \bar{f}_{k_{2}, t-\bar{s}, ii_{1}'i_{2}'i_{3}'j}\rvert} \nonumber\\
& \times \mathbb{E} [\lvert \hat{X}_{s,u}^{\epsilon, i_{2}}(k_{2}) - \hat{X}_{t,u}^{\epsilon, i_{2}}(k_{2}) \rvert \lvert \hat{X}_{\bar{s}, u}^{\epsilon, i_{2}'} (k_{2}) - \hat{X}_{t,u}^{\epsilon, i_{2}'}(k_{2}) \rvert ds d \bar{s} ] \lesssim \epsilon^{\eta} t^{\frac{\eta - \epsilon}{2}} 2^{q(1+2\eta)}. \nonumber
\end{align} 
In order to compute $F_{q, t, ij}^{12}$, we rely on \eqref{special} to deduce that for any $\epsilon \in (0, \eta)$, 
\begin{align}\label{estimate 160}
F_{q, t, ij}^{12} &\lesssim \sum_{i_{1}, i_{2}, i_{3}, i_{1}', i_{2}', i_{3}' = 1}^{3} \sum_{k_{2} \neq 0} \theta(2^{-q} k_{2})^{2} \int_{[0,t]^{2}} \lvert \bar{f}_{k_{2}, t-s, ii_{1}i_{2}i_{3}j} \rvert \lvert \bar{f}_{k_{2}, t - \bar{s}, ii_{1}'i_{2}'i_{3}'j} \rvert  \\
& \times \mathbb{E} [ \lvert \hat{X}_{s,u}^{\epsilon, i_{2}}(k_{2}) - \hat{X}_{t,u}^{\epsilon, i_{2}}(k_{2}) - \hat{\bar{X}}_{s,u}^{\epsilon, i_{2}}(k_{2}) + \hat{\bar{X}}_{t,u}^{\epsilon, i_{2}}(k_{2}) \rvert \nonumber\\
& \hspace{3mm} \times  \lvert \hat{X}_{\bar{s},u}^{\epsilon, i_{2}'}(k_{2}) - \hat{X}_{t,u}^{\epsilon, i_{2}'}(k_{2}) - \hat{\bar{X}}_{\bar{s},u}^{\epsilon, i_{2}'}(k_{2}) + \hat{\bar{X}}_{t,u}^{\epsilon, i_{2}'}(k_{2}) \rvert] ds d \bar{s}  \lesssim \epsilon^{\eta} t^{\frac{\eta - \epsilon}{2}} 2^{q(1+ 2\eta)} \nonumber
\end{align} 
by \eqref{estimate 157}, \eqref{estimate 158}, H$\ddot{\mathrm{o}}$lder's inequality and \eqref{key estimate}. Applying \eqref{estimate 159}-\eqref{estimate 160} to \eqref{estimate 161} and \eqref{estimate 162} leads us to 
\begin{equation}\label{estimate 163}
\mathbb{E} [ \lvert \Delta_{q} (V_{t, ii_{1}}^{3} - \tilde{V}_{t, ii_{1}}^{3}) \rvert^{2}] \lesssim \epsilon^{\eta} t^{\frac{\eta - \epsilon}{2}} 2^{q(1+ 2\eta)}. 
\end{equation} 
Next, we estimate within \eqref{estimate 154}
\begin{align}\label{estimate 165}
& \mathbb{E} [ \lvert \Delta_{q} [ \tilde{V}_{t, ii_{1}}^{3} - \sum_{i_{2} =1}^{3} (\tilde{C}_{2,u}^{\epsilon, i i_{2} j} X_{t,u}^{\epsilon, i_{2}} + \tilde{C}_{2,b}^{\epsilon, i i_{2} j} X_{t,b}^{\epsilon, i_{2}}) \nonumber\\
& \hspace{50mm} + \sum_{i_{2} =1}^{3} ( \bar{\tilde{C}}_{2,u}^{\epsilon, i i_{2} j} \bar{X}_{t,u}^{\epsilon, i_{2}} + \bar{\tilde{C}}_{2,b}^{\epsilon, i i_{2} j} \bar{X}_{t, b}^{\epsilon, i_{2}})   ] \rvert^{2}] \nonumber \\
\lesssim& \sum_{k_{2} \neq 0} \theta(2^{-q} k_{2})^{2} \nonumber \\
& \times [( \sum_{i,j=1}^{3} \mathbb{E} [ \hat{X}_{t,u}^{\epsilon, i} (k_{2}) \overline{\hat{X}_{t,u}^{\epsilon, j} (k_{2})}])\nonumber\\
& \hspace{1mm} \times  \lvert \sum_{i_{1}, i_{2} =1}^{3} \sum_{k_{1} \neq 0} h_{b}(\epsilon k_{1})^{2}  \int_{0}^{t} [ \frac{ e^{-\lvert k_{1} \rvert^{2} f(\epsilon k_{1})(t-s)}}{\lvert k_{1} \rvert^{2} f(\epsilon k_{1})} [ e^{- \lvert k_{12} \rvert^{2} f(\epsilon k_{12})(t-s)} k_{12}^{i_{2}} g(\epsilon k_{12}^{i_{2}}) \hat{\mathcal{P}}^{ii_{1}}(k_{12}) \nonumber\\
& \hspace{65mm}  - e^{- \lvert k_{1} \rvert^{2} f(\epsilon k_{1}) (t-s)} k_{1}^{i_{2}} g(\epsilon k_{1}^{i_{2}}) \hat{\mathcal{P}}^{ii_{1}}(k_{1})] \nonumber\\
& \hspace{10mm} - \frac{e^{-\lvert k_{1} \rvert^{2} (t-s)}}{\lvert k_{1} \rvert^{2}}[ e^{- \lvert k_{12} \rvert^{2} (t-s)} k_{12}^{i_{2}} i \hat{\mathcal{P}}^{ii_{1}}(k_{12}) - e^{- \lvert k_{1}\rvert^{2} (t-s)} k_{1}^{i_{2}} i \hat{\mathcal{P}}^{ii_{1}}(k_{1})]] ds \rvert^{2} \nonumber \\
&+ (\sum_{i=1}^{3} \mathbb{E} [ \lvert \hat{X}_{t,u}^{\epsilon, i} (k_{2}) - \hat{\bar{X}}_{t,u}^{\epsilon, i} (k_{2}) \rvert^{2} ]) \nonumber\\
& \hspace{5mm} \times \lvert \sum_{i_{1}, i_{2} =1}^{3} \sum_{k_{1} \neq 0} h_{b}(\epsilon k_{1})^{2} \int_{0}^{t} \frac{ e^{-\lvert k_{1} \rvert^{2} (t-s)}}{\lvert k_{1} \rvert^{2}} \nonumber\\
& \hspace{20mm} \times [ e^{-\lvert k_{12} \rvert^{2} (t-s)} k_{12}^{i_{2}} i\hat{\mathcal{P}}^{ii_{1}}(k_{12}) - e^{-\lvert k_{1} \rvert^{2} (t-s)} k_{1}^{i_{2}} i \hat{\mathcal{P}}^{ii_{1}}(k_{1})] ds \rvert^{2} \nonumber \\
&+ ( \sum_{i,j=1}^{3} \mathbb{E} [ \hat{X}_{t,b}^{\epsilon, i} (k_{2}) \overline{\hat{X}_{t,b}^{\epsilon, j} (k_{2})}])\nonumber\\
& \hspace{5mm} \times  \lvert \sum_{i_{1}, i_{2} =1}^{3} \sum_{k_{1} \neq 0} h_{u}(\epsilon k_{1}) h_{b}(\epsilon k_{1}) \nonumber\\
& \hspace{10mm} \times \int_{0}^{t} [ \frac{ e^{-\lvert k_{1} \rvert^{2} f(\epsilon k_{1})(t-s)}}{\lvert k_{1} \rvert^{2} f(\epsilon k_{1})} [ e^{- \lvert k_{12} \rvert^{2} f(\epsilon k_{12})(t-s)} k_{12}^{i_{2}} g(\epsilon k_{12}^{i_{2}}) \hat{\mathcal{P}}^{ii_{1}}(k_{12}) \nonumber\\
& \hspace{65mm} - e^{- \lvert k_{1} \rvert^{2} f(\epsilon k_{1}) (t-s)} k_{1}^{i_{2}} g(\epsilon k_{1}^{i_{2}}) \hat{\mathcal{P}}^{ii_{1}}(k_{1})] \nonumber\\
& \hspace{10mm} - \frac{e^{-\lvert k_{1} \rvert^{2} (t-s)}}{\lvert k_{1} \rvert^{2}}[ e^{- \lvert k_{12} \rvert^{2} (t-s)} k_{12}^{i_{2}} i \hat{\mathcal{P}}^{ii_{1}}(k_{12}) - e^{- \lvert k_{1}\rvert^{2} (t-s)} k_{1}^{i_{2}} i \hat{\mathcal{P}}^{ii_{1}}(k_{1})]] ds \rvert^{2} \nonumber \\
&+ (\sum_{i=1}^{3} \mathbb{E} [ \lvert \hat{X}_{t,b}^{\epsilon, i} (k_{2}) - \hat{\bar{X}}_{t,b}^{\epsilon, i} (k_{2}) \rvert^{2} ]) \nonumber\\
& \hspace{5mm} \times \lvert \sum_{i_{1}, i_{2} =1}^{3} \sum_{k_{1} \neq 0} h_{u}(\epsilon k_{1}) h_{b}(\epsilon k_{1}) \int_{0}^{t} \frac{ e^{-\lvert k_{1} \rvert^{2} (t-s)}}{\lvert k_{1} \rvert^{2}} \nonumber\\
& \hspace{2mm} \times [ e^{-\lvert k_{12} \rvert^{2} (t-s)} k_{12}^{i_{2}} i\hat{\mathcal{P}}^{ii_{1}}(k_{12}) - e^{-\lvert k_{1} \rvert^{2} (t-s)} k_{1}^{i_{2}} i \hat{\mathcal{P}}^{ii_{1}}(k_{1})] ds \rvert^{2}] 
\end{align} 
by \eqref{estimate 155}, \eqref{C2uepsilonii2j}, \eqref{C2bepsilonii2j} and \eqref{estimate 164}. We use the previous estimates from \eqref{estimate 6} and \eqref{[Equation (4.4a)][ZZ17]} in \eqref{estimate 165}  so that 
\begin{align}\label{estimate 169}
& \mathbb{E} [ \lvert \Delta_{q} [ \tilde{V}_{t, ii_{1}}^{3} - \sum_{i_{2} =1}^{3} (\tilde{C}_{2,u}^{\epsilon, i i_{2} j} X_{t,u}^{\epsilon, i_{2}} + \tilde{C}_{2,b}^{\epsilon, i i_{2} j} X_{t,b}^{\epsilon, i_{2}}) \nonumber\\
&\hspace{10mm} + \sum_{i_{2} =1}^{3} ( \bar{\tilde{C}}_{2,u}^{\epsilon, i i_{2} j} \bar{X}_{t,u}^{\epsilon, i_{2}} + \bar{\tilde{C}}_{2,b}^{\epsilon, i i_{2} j} \bar{X}_{t, b}^{\epsilon, i_{2}})   ] \rvert^{2}] \lesssim \sum_{l=1}^{2} H_{q,t,i}^{l} 
\end{align} 
where 
\begin{subequations}\label{estimate 166}
\begin{align}
H_{q,t,i}^{1} \triangleq& \sum_{k_{2} \neq 0} \frac{\theta(2^{-q} k_{2})^{2}}{\lvert k_{2} \rvert^{2}} \lvert \sum_{i_{1}, i_{2} =1}^{3} \sum_{k_{1} \neq 0} \\
&  \times \int_{0}^{t} [ \frac{ e^{-\lvert k_{1} \rvert^{2} f(\epsilon k_{1})(t-s)}}{\lvert k_{1} \rvert^{2} f(\epsilon k_{1})} [ e^{- \lvert k_{12} \rvert^{2} f(\epsilon k_{12})(t-s)} k_{12}^{i_{2}} g(\epsilon k_{12}^{i_{2}}) \hat{\mathcal{P}}^{ii_{1}}(k_{12}) \nonumber\\
& \hspace{40mm} - e^{- \lvert k_{1} \rvert^{2} f(\epsilon k_{1}) (t-s)} k_{1}^{i_{2}} g(\epsilon k_{1}^{i_{2}}) \hat{\mathcal{P}}^{ii_{1}}(k_{1})] \nonumber\\
& \hspace{5mm}  - \frac{e^{-\lvert k_{1} \rvert^{2} (t-s)}}{\lvert k_{1} \rvert^{2}}[ e^{- \lvert k_{12} \rvert^{2} (t-s)} k_{12}^{i_{2}} i \hat{\mathcal{P}}^{ii_{1}}(k_{12}) - e^{- \lvert k_{1}\rvert^{2} (t-s)} k_{1}^{i_{2}} i \hat{\mathcal{P}}^{ii_{1}}(k_{1})]] ds \rvert^{2}, \nonumber \\
H_{q,t,i}^{2} \triangleq& \epsilon^{\eta} \sum_{k_{2} \neq 0} \frac{ \theta(2^{-q}k_{2})^{2}}{\lvert k_{2} \rvert^{2-\eta}} \lvert \sum_{i_{1}, i_{2} =1}^{3} \sum_{k_{1} \neq 0}  \\
& \times  \int_{0}^{t} \frac{ e^{-\lvert k_{1} \rvert^{2} (t-s)}}{\lvert k_{1} \rvert^{2}} [ e^{-\lvert k_{12} \rvert^{2} (t-s)} k_{12}^{i_{2}} i\hat{\mathcal{P}}^{ii_{1}}(k_{12}) - e^{-\lvert k_{1} \rvert^{2} (t-s)} k_{1}^{i_{2}} i \hat{\mathcal{P}}^{ii_{1}}(k_{1})] ds \rvert^{2}. \nonumber 
\end{align}
\end{subequations} 
In order to estimate $H_{q,t,i}^{1}$, we see that for any $\eta \in (0,1)$
\begin{align*}
& \lvert  \frac{ e^{-\lvert k_{1} \rvert^{2} f(\epsilon k_{1})(t-s)}}{\lvert k_{1} \rvert^{2} f(\epsilon k_{1})} [ e^{- \lvert k_{12} \rvert^{2} f(\epsilon k_{12})(t-s)} k_{12}^{i_{2}} g(\epsilon k_{12}^{i_{2}}) \hat{\mathcal{P}}^{ii_{1}}(k_{12}) \nonumber\\
& \hspace{40mm} - e^{- \lvert k_{1} \rvert^{2} f(\epsilon k_{1}) (t-s)} k_{1}^{i_{2}} g(\epsilon k_{1}^{i_{2}}) \hat{\mathcal{P}}^{ii_{1}}(k_{1})] \nonumber\\
& - \frac{e^{-\lvert k_{1} \rvert^{2} (t-s)}}{\lvert k_{1} \rvert^{2}}[ e^{- \lvert k_{12} \rvert^{2} (t-s)} k_{12}^{i_{2}} i \hat{\mathcal{P}}^{ii_{1}}(k_{12}) - e^{- \lvert k_{1}\rvert^{2} (t-s)} k_{1}^{i_{2}} i \hat{\mathcal{P}}^{ii_{1}}(k_{1})] \rvert  \nonumber \\
\lesssim& \frac{ e^{- \lvert k_{1} \rvert^{2} \bar{c}_{f} (t-s)}}{\lvert k_{1} \rvert^{2}} \lvert k_{2} \rvert^{\eta} \lvert t-s \rvert^{- \frac{1-\eta}{2}} \wedge  \epsilon^{\frac{\eta}{2}} \frac{ e^{- \lvert k_{1} \rvert^{2} \bar{c}_{f} (t-s)}}{\lvert k_{1} \rvert^{2}} ( \lvert t-s \rvert^{-\frac{1}{2}} \lvert k_{1} \rvert^{\frac{\eta}{2}} + \lvert t-s \rvert^{-\frac{1}{2} - \frac{\eta}{4}})
\end{align*} 
where the first estimate is due to Lemmas \ref{Lemma 3.8} and \ref{Lemma 3.9} similarly to \eqref{estimate 147}-\eqref{estimate 148}, \eqref{[Equation (4.2)][ZZ17]}, \eqref{[Equation (4.3)][ZZ17]} and \eqref{key estimate} while the second by \eqref{[Equation (4.2)][ZZ17]}, \eqref{[Equation (4.3)][ZZ17]} and \eqref{key estimate}. Therefore, we deduce from \eqref{estimate 166} that for $\epsilon \in (0, \eta)$, 
\begin{align}\label{estimate 167}
H_{t, q, i}^{1} \lesssim& \sum_{k_{2} \neq 0} \frac{ \theta(2^{-q} k_{2})^{2}}{\lvert k_{2} \rvert^{2}} \lvert \sum_{k_{1} \neq 0} \int_{0}^{t}  \frac{ e^{- \lvert k_{1} \rvert^{2} \bar{c}_{f} (t-s)}}{\lvert k_{1} \rvert^{2}} \lvert k_{2} \rvert^{\frac{\eta}{2}} (t-s)^{-\frac{1-\eta}{4}} \epsilon^{\frac{\eta}{4}} \lvert t-s \rvert^{-\frac{1}{4}} \lvert k_{1} \rvert^{\frac{\eta}{4}} \nonumber\\
& \hspace{5mm}  + \frac{ e^{- \lvert k_{1} \rvert^{2} \bar{c}_{f} (t-s)}}{\lvert k_{1} \rvert^{2}} \lvert k_{2} \rvert^{\frac{\eta}{2}} \lvert t-s \rvert^{-\frac{1}{4} + \frac{\eta}{4} - \frac{1}{4} - \frac{\eta}{8}} \epsilon^{\frac{\eta}{4}} ds \rvert^{2}\lesssim  \epsilon^{\frac{\eta}{2}} t^{ \frac{\eta-\epsilon}{4}} 2^{q(1+ \eta)}.   
\end{align} 
On the other hand, we can estimate from \eqref{estimate 166}, for $\epsilon \in (0, \eta)$, 
\begin{align}\label{estimate 168}
H_{q, t, i}^{2}\lesssim \epsilon^{\eta} t^{\eta -\epsilon} \sum_{k_{2} \neq 0} \frac{ \theta(2^{-q} k_{2})^{2}}{\lvert k_{2} \rvert^{2-3\eta}} \lvert \sum_{k_{1} \neq 0} \frac{1}{\lvert k_{1} \rvert^{3+ \epsilon}} \rvert^{2}  \lesssim \epsilon^{\eta} t^{\eta - \epsilon} 2^{q(1+ 3\eta)}
\end{align} 
by Lemma \ref{Lemma 3.8} and \eqref{key estimate}. Applying \eqref{estimate 167}-\eqref{estimate 168} to \eqref{estimate 169} gives 
\begin{align}\label{estimate 170}
& \mathbb{E} [ \lvert \Delta_{q} [ \tilde{V}_{t, ii_{1}}^{3} - \sum_{i_{2} =1}^{3} (\tilde{C}_{2,u}^{\epsilon, i i_{2} j} X_{t,u}^{\epsilon, i_{2}} + \tilde{C}_{2,b}^{\epsilon, i i_{2} j} X_{t,b}^{\epsilon, i_{2}}) \nonumber\\
&\hspace{10mm} + \sum_{i_{2} =1}^{3} ( \bar{\tilde{C}}_{2,u}^{\epsilon, i i_{2} j} \bar{X}_{t,u}^{\epsilon, i_{2}} + \bar{\tilde{C}}_{2,b}^{\epsilon, i i_{2} j} \bar{X}_{t, b}^{\epsilon, i_{2}})   ] \rvert^{2}] \lesssim \epsilon^{\frac{\eta}{2}} t^{\frac{\eta - \epsilon}{4}} 2^{q(1+ 3\eta)}. 
\end{align} 

\emph{Terms in the first chaos}: $V_{t, ii_{1}}^{1}$ in \eqref{estimate 135}\\
We compute 
\begin{align} 
& \mathbb{E} [ \lvert \Delta_{q} V_{t, ii_{1}}^{1} \rvert^{2}] \\
\lesssim& \mathbb{E} [ \lvert \Delta_{q} \sum_{i_{1}, i_{2} =1}^{3} \sum_{k} \sum_{k_{1}, k_{2}, k_{3} \neq 0: k_{123} = k} \nonumber\\
& \times \int_{0}^{t} [ e^{- \lvert k_{12} \rvert^{2} f(\epsilon k_{12}) (t-s)} k_{12}^{i_{2}} g(\epsilon k_{12}^{i_{2}}) - e^{- \lvert k_{12} \rvert^{2} (t-s)} k_{12}^{i_{2}}i] \nonumber\\
& \hspace{40mm} \times : \hat{X}_{s,b}^{\epsilon, i_{1}}(k_{1}) \hat{X}_{s,u}^{\epsilon, i_{2}}(k_{2}) \hat{X}_{t,b}^{\epsilon, j} (k_{3}): ds \hat{\mathcal{P}}^{ii_{1}}(k_{12}) e_{k} \rvert^{2} ] \nonumber \\
&+ \mathbb{E} [ \lvert \Delta_{q} \sum_{i_{1}, i_{2} =1}^{3} \sum_{k} \sum_{k_{1}, k_{2}, k_{3} \neq 0: k_{123} = k} \int_{0}^{t} e^{- \lvert k_{12} \rvert^{2} (t-s)} k_{12}^{i_{2}} i \nonumber\\
& \times[ : \hat{X}_{s,b}^{\epsilon, i_{1}}(k_{1}) \hat{X}_{s,u}^{\epsilon, i_{2}}(k_{2}) \hat{X}_{t,b}^{\epsilon, j} (k_{3}): - : \hat{\bar{X}}_{s,b}^{\epsilon, i_{1}}(k_{1}) \hat{\bar{X}}_{s,u}^{\epsilon, i_{2}}(k_{2}) \hat{\bar{X}}_{t,b}^{\epsilon, j} (k_{3}): ] ds \hat{\mathcal{P}}^{ii_{1}}(k_{12}) e_{k} \rvert^{2}] \nonumber \\
&+  \mathbb{E} [ \lvert \Delta_{q} \sum_{i_{1}, i_{2} =1}^{3} \sum_{k} \sum_{k_{1}, k_{2}, k_{3} \neq 0: k_{123} = k} \nonumber\\
& \times \int_{0}^{t} [ e^{- \lvert k_{12} \rvert^{2} f(\epsilon k_{12}) (t-s)} k_{12}^{i_{2}} g(\epsilon k_{12}^{i_{2}}) - e^{- \lvert k_{12} \rvert^{2} (t-s)} k_{12}^{i_{2}}i] \nonumber\\
& \hspace{40mm} \times : \hat{X}_{s,u}^{\epsilon, i_{1}}(k_{1}) \hat{X}_{s,b}^{\epsilon, i_{2}}(k_{2}) \hat{X}_{t,b}^{\epsilon, j} (k_{3}): ds \hat{\mathcal{P}}^{ii_{1}}(k_{12}) e_{k} \rvert^{2} ] \nonumber \\
&+ \mathbb{E} [ \lvert \Delta_{q} \sum_{i_{1}, i_{2} =1}^{3} \sum_{k} \sum_{k_{1}, k_{2}, k_{3} \neq 0: k_{123} = k} \int_{0}^{t} e^{- \lvert k_{12} \rvert^{2} (t-s)} k_{12}^{i_{2}} i \nonumber\\
& \times[ : \hat{X}_{s,u}^{\epsilon, i_{1}}(k_{1}) \hat{X}_{s,b}^{\epsilon, i_{2}}(k_{2}) \hat{X}_{t,b}^{\epsilon, j} (k_{3}): - : \hat{\bar{X}}_{s,u}^{\epsilon, i_{1}}(k_{1}) \hat{\bar{X}}_{s,b}^{\epsilon, i_{2}}(k_{2}) \hat{\bar{X}}_{t,b}^{\epsilon, j} (k_{3}): ] ds \hat{\mathcal{P}}^{ii_{1}}(k_{12}) e_{k} \rvert^{2}]. \nonumber
\end{align} 
Relying on Example \ref{Example 3.1} and \eqref{covariance a}-\eqref{covariance c}, we may further bound by 
\begin{equation}\label{estimate 181}
\mathbb{E} [ \lvert \Delta_{q} V_{t, ii_{1}}^{1} \rvert^{2}] \lesssim \sum_{l=1}^{3} I_{q,t}^{l} 
\end{equation} 
where 
\begin{subequations}\label{estimate 173} 
\begin{align}
I_{q,t}^{1} \triangleq&  \sum_{i_{2}, i_{2}' = 1}^{3} \sum_{k} \theta(2^{-q} k)^{2} \sum_{k_{1}, k_{2}, k_{3}, k_{1}', k_{2}', k_{3}' \neq 0: k_{123} = k_{123}' = k} J\\
&  \times  \int_{[0,t]^{2}} \lvert e^{- \lvert k_{12} \rvert^{2} f(\epsilon k_{12}) (t-s)} k_{12}^{i_{2}} g(\epsilon k_{12}^{i_{2}}) - e^{- \lvert k_{12} \rvert^{2} (t-s)} k_{12}^{i_{2}} i \rvert \nonumber\\
& \hspace{3mm} \times \lvert e^{- \lvert k_{12}' \rvert^{2} f(\epsilon k_{12}') (t- \bar{s})} (k_{12}')^{i_{2}'} g(\epsilon (k_{12}')^{i_{2}'}) - e^{- \lvert k_{12}' \rvert^{2} (t- \bar{s})} (k_{12}')^{i_{2}'} i \rvert  \prod_{i=1}^{3} \frac{1}{\lvert k_{i} \rvert^{2}} ds d \bar{s}, \nonumber\\
I_{q,t}^{2} \triangleq&  \sum_{i_{2}, i_{2}' = 1}^{3} \sum_{k} \theta(2^{-q} k)^{2} \sum_{k_{1}, k_{2}, k_{3}, k_{1}', k_{2}', k_{3}' \neq 0: k_{123} = k_{123'} = k} J  \\
&  \times \int_{[0,t]^{2}}\lvert e^{- \lvert k_{12} \rvert^{2} (t-s)} k_{12}^{i_{2}} \rvert \lvert e^{- \lvert k_{12}' \rvert^{2} (t- \bar{s})} (k_{12}')^{i_{2}'} \rvert  \nonumber\\
& \hspace{3mm} \times \mathbb{E} [ ( :\hat{X}_{s,b}^{\epsilon} (k_{1}) \hat{X}_{s,u}^{\epsilon}(k_{2}) \hat{X}_{t,b}^{\epsilon} (k_{3}): - : \hat{\bar{X}}_{s,b}^{\epsilon}(k_{1}) \hat{\bar{X}}_{s,u}^{\epsilon} (k_{2}) \hat{\bar{X}}_{t,b}^{\epsilon} (k_{3}):) \nonumber\\
& \hspace{3mm} \times \overline{ ( :\hat{X}_{\bar{s},b}^{\epsilon} (k_{1}') \hat{X}_{\bar{s},u}^{\epsilon}(k_{2}') \hat{X}_{t,b}^{\epsilon} (k_{3}'): - : \hat{\bar{X}}_{\bar{s},b}^{\epsilon}(k_{1}') \hat{\bar{X}}_{\bar{s},u}^{\epsilon} (k_{2}') \hat{\bar{X}}_{t,b}^{\epsilon} (k_{3}'):) }] ds d \bar{s}, \nonumber\\
I_{q,t}^{3} \triangleq&  \sum_{i_{2},  i_{2}' = 1}^{3} \sum_{k} \theta(2^{-q} k)^{2} \sum_{k_{1}, k_{2}, k_{3}, k_{1}', k_{2}', k_{3}' \neq 0: k_{123} = k_{123'} = k} J  \\
&  \times \int_{[0,t]^{2}}\lvert e^{- \lvert k_{12} \rvert^{2} (t-s)} k_{12}^{i_{2}} \rvert \lvert e^{- \lvert k_{12}' \rvert^{2} (t- \bar{s})} (k_{12}')^{i_{2}'} \rvert \nonumber\\
& \hspace{3mm} \times \mathbb{E} [ ( :\hat{X}_{s,u}^{\epsilon} (k_{1}) \hat{X}_{s,b}^{\epsilon}(k_{2}) \hat{X}_{t,b}^{\epsilon} (k_{3}): - : \hat{\bar{X}}_{s,u}^{\epsilon}(k_{1}) \hat{\bar{X}}_{s,b}^{\epsilon} (k_{2}) \hat{\bar{X}}_{t,b}^{\epsilon} (k_{3}):) \nonumber\\
& \hspace{3mm} \times \overline{ ( :\hat{X}_{\bar{s},u}^{\epsilon} (k_{1}') \hat{X}_{\bar{s},b}^{\epsilon}(k_{2}') \hat{X}_{t,b}^{\epsilon} (k_{3}'): - : \hat{\bar{X}}_{\bar{s},u}^{\epsilon}(k_{1}') \hat{\bar{X}}_{\bar{s},b}^{\epsilon} (k_{2}') \hat{\bar{X}}_{t,b}^{\epsilon} (k_{3}'):) }] ds d \bar{s} \nonumber
\end{align}
\end{subequations} 
and
\begin{align}\label{estimate 171}
J \triangleq& 1_{k_{1} = k_{1}', k_{2} = k_{2}', k_{3} = k_{3}'} + 1_{k_{1} = k_{1}', k_{2} = k_{3}', k_{3} = k_{2}'} + 1_{k_{1} = k_{2}', k_{2} = k_{1}', k_{3} = k_{3}'}  \nonumber\\
&+ 1_{k_{1} = k_{2}', k_{2} = k_{3}', k_{3} = k_{1}'} + 1_{k_{1} = k_{3}', k_{2} = k_{1}', k_{3} = k_{2}'} + 1_{k_{1} = k_{3}', k_{2} = k_{2}', k_{3} = k_{1}'}. 
\end{align} 

\begin{remark}\label{Remark 2.2}
In the case of the NS equations, an analogous bound would be the same with all $b$ replaced by $u$; i.e., 
\begin{align*}
&\sum_{i_{2}, i_{2}' = 1}^{3} \sum_{k} \theta(2^{-q} k)^{2} \sum_{k_{1}, k_{2}, k_{3}, k_{1}', k_{2}', k_{3}' \neq 0: k_{123} = k_{123}' = k} J \nonumber\\
&  \times \int_{[0,t]^{2}} \lvert e^{- \lvert k_{12} \rvert^{2} f(\epsilon k_{12}) (t-s)} k_{12}^{i_{2}} g(\epsilon k_{12}^{i_{2}}) - e^{- \lvert k_{12} \rvert^{2} (t-s)} k_{12}^{i_{2}} i \rvert \nonumber\\
& \hspace{3mm} \times \lvert e^{- \lvert k_{12}' \rvert^{2} f(\epsilon k_{12}') (t- \bar{s})} (k_{12}')^{i_{2}'} g(\epsilon (k_{12}')^{i_{2}'}) - e^{- \lvert k_{12}' \rvert^{2} (t- \bar{s})} (k_{12}')^{i_{2}'} i \rvert  \prod_{i=1}^{3} \frac{1}{\lvert k_{i} \rvert^{2}} ds d \bar{s}, \nonumber\\
&+  \sum_{i_{1}, i_{2}, i_{1}', i_{2}' = 1}^{3} \sum_{k} \theta(2^{-q} k)^{2} \sum_{k_{1}, k_{2}, k_{3}, k_{1}', k_{2}', k_{3}' \neq 0: k_{123} = k_{123'} = k} J  \nonumber\\
&  \times  \int_{[0,t]^{2}}\lvert e^{- \lvert k_{12} \rvert^{2} (t-s)} k_{12}^{i_{2}} \rvert \lvert e^{- \lvert k_{12}' \rvert^{2} (t- \bar{s})} (k_{12}')^{i_{2}'} \rvert \nonumber\\
& \hspace{3mm} \times \mathbb{E} [ ( :\hat{X}_{s,u}^{\epsilon} (k_{1}) \hat{X}_{s,u}^{\epsilon}(k_{2}) \hat{X}_{t,u}^{\epsilon} (k_{3}): - : \hat{\bar{X}}_{s,u}^{\epsilon}(k_{1}) \hat{\bar{X}}_{s,u}^{\epsilon} (k_{2}) \hat{\bar{X}}_{t,u}^{\epsilon} (k_{3}):) \nonumber\\
& \hspace{3mm} \times \overline{ ( :\hat{X}_{\bar{s},u}^{\epsilon} (k_{1}') \hat{X}_{\bar{s},u}^{\epsilon}(k_{2}') \hat{X}_{t,u}^{\epsilon} (k_{3}'): - : \hat{\bar{X}}_{\bar{s},u}^{\epsilon}(k_{1}') \hat{\bar{X}}_{\bar{s},u}^{\epsilon} (k_{2}') \hat{\bar{X}}_{t,u}^{\epsilon} (k_{3}'):) }] ds d \bar{s}. \nonumber
\end{align*} 
Then we notice the symmetry in $k_{1}$ with $k_{2}$, $k_{1}'$ with $k_{2}'$. Using this symmetry, Zhu and Zhu \cite{ZZ17} reduced the work load of six different cases in $J$ of \eqref{estimate 171} to two cases. In the case of the MHD system, because e.g., $\hat{X}_{s,b}^{\epsilon}(k_{1}) \hat{X}_{s,u}^{\epsilon}(k_{2}) \neq \hat{X}_{s,b}^{\epsilon}(k_{2}) \hat{X}_{s,u}^{\epsilon}(k_{1})$, there is no symmetry in $k_{1}$ with $k_{2}$ within $I_{q,t}^{2}$ or within $I_{q,t}^{1} + I_{q,t}^{2}$ alone. However, we see that there is a remarkable symmetry in $I_{q,t}^{1} + I_{q,t}^{2} + I_{q,t}^{3}$ as a whole. Due to this symmetry, we can reduce the six cases in $J$ of \eqref{estimate 171} to two cases of 
\begin{equation}\label{estimate 172}
1_{k_{1} = k_{1}', k_{2} = k_{2}', k_{3} = k_{3}'} \text{ and } 1_{k_{1} = k_{1}', k_{2} = k_{3}', k_{3} = k_{2}'}. 
\end{equation} 
For example, in the sixth case of $1_{k_{1} = k_{3}', k_{2} = k_{2}', k_{3} = k_{1}'}$, we can swap $k_{1}$ with $k_{2}$  and $k_{1}'$ with $k_{2}'$ to deduce the case $(k_{1},k_{2},k_{3}) = (k_{1}', k_{3}', k_{2}')$. Similarly, other cases may be readily reduced to the one of two cases in \eqref{estimate 172} (see \cite[Equation (250)]{Y19a} for similar computations). 
\end{remark} 

W.l.o.g. we work hereafter on $I_{q,t}^{1} +I_{q,t}^{2}$ as the estimates on $I_{q,t}^{3}$ is similar to that on $I_{q,t}^{2}$; however, we emphasize again that we needed $I_{q,t}^{3}$ for the reduction of six cases to two taking advantage of the symmetry from the structure of the MHD system. We estimate for $\eta \in [0,1]$, 
\begin{align}\label{estimate 179}
I_{q,t}^{1}  \lesssim&  \sum_{k} \theta(2^{-q} k)^{2} \sum_{k_{1}, k_{2}, k_{3}, k_{1}', k_{2}', k_{3}' \neq 0: k_{123} = k_{123}' = k} \\
& \times (1_{k_{1} = k_{1}', k_{2} = k_{2}', k_{3} = k_{3}'} +1_{k_{1} = k_{1}', k_{2} = k_{3}', k_{3} = k_{2}'})  \prod_{i=1}^{3} \frac{1}{\lvert k_{i} \rvert^{2}}    \nonumber\\
& \times \int_{[0,t]^{2}} e^{- \lvert k_{12} \rvert^{2} \bar{c}_{f} (t-s)} \lvert \epsilon k_{12} \rvert^{\frac{\eta}{2}} \lvert k_{12} \rvert e^{- \lvert k_{12} ' \rvert^{2} (t- \bar{s})} \lvert \epsilon k_{12}' \rvert^{\frac{\eta}{2}} \lvert k_{12} ' \rvert ds d \bar{s}  \nonumber\\
\lesssim& \epsilon^{\eta} t^{\eta} \sum_{k} \theta(2^{-q} k)^{2} [ \sum_{k_{3}\neq 0 } \frac{1}{\lvert k_{3} \rvert^{2}}  \frac{1}{\lvert k - k_{3} \rvert^{3-3\eta}} \nonumber\\
& \hspace{10mm} + \left( \sum_{k_{3} \neq 0} \frac{1}{\lvert k_{3} \rvert^{2}} \sum_{k_{1}, k_{2} \neq 0: k_{12} = k - k_{3}} \frac{1}{\lvert k_{1} \rvert^{2} \lvert k_{2} \rvert^{2} \lvert k_{12} \rvert^{2-3\eta}} \right)^{\frac{1}{2}} \nonumber\\
& \hspace{10mm} \times \left( \sum_{k_{2} \neq 0} \frac{1}{\lvert k_{2} \rvert^{2}} \sum_{k_{1}, k_{3} \neq 0: k_{13} = k - k_{2}} \frac{1}{\lvert k_{1} \rvert^{2} \lvert k_{3} \rvert^{2} \lvert k_{13} \rvert^{2-3\eta}} \right)^{\frac{1}{2}} ] \lesssim \epsilon^{\eta} t^{\eta} 2^{q(1+ 3\eta)} \nonumber
\end{align} 
by \eqref{[Equation (4.2)][ZZ17]}, \eqref{[Equation (4.3)][ZZ17]}, \eqref{key estimate} and H$\ddot{\mathrm{o}}$lder's inequality. Next, for $I_{q,t}^{2}$ in \eqref{estimate 173}, we see that for any $i_{1}, i_{2}, j, i_{1}', i_{2}', j \in \{1,2,3\}$, 
\begin{align}
& \mathbb{E} [ ( : \hat{X}_{s,b}^{\epsilon, i_{1}}(k_{1}) \hat{X}_{s,u}^{\epsilon, i_{2}}(k_{2}) \hat{X}_{t,b}^{\epsilon, j}(k_{3}): - : \hat{\bar{X}}_{s,b}^{\epsilon, i_{1}}(k_{1}) \hat{\bar{X}}_{s,u}^{\epsilon, i_{2}}(k_{2}) \hat{\bar{X}}_{t,b}^{\epsilon, j} (k_{3}):) \nonumber\\
& \times  \overline{ ( : \hat{X}_{\bar{s},b}^{\epsilon, i_{1}'}(k_{1}') \hat{X}_{\bar{s},u}^{\epsilon, i_{2}'}(k_{2}') \hat{X}_{t,b}^{\epsilon, j}(k_{3}'): - : \hat{\bar{X}}_{\bar{s},b}^{\epsilon, i_{1}'}(k_{1}') \hat{\bar{X}}_{\bar{s},u}^{\epsilon, i_{2}'}(k_{2}') \hat{\bar{X}}_{t,b}^{\epsilon, j'} (k_{3}'):)   }] \nonumber\\
\leq& ( \mathbb{E} [ \lvert  : \hat{X}_{s,b}^{\epsilon, i_{1}}(k_{1}) \hat{X}_{s,u}^{\epsilon, i_{2}}(k_{2}) \hat{X}_{t,b}^{\epsilon, j}(k_{3}): - : \hat{\bar{X}}_{s,b}^{\epsilon, i_{1}}(k_{1}) \hat{\bar{X}}_{s,u}^{\epsilon, i_{2}}(k_{2}) \hat{\bar{X}}_{t,b}^{\epsilon, j} (k_{3}): \rvert^{2} ])^{\frac{1}{2}} \nonumber\\
& \times (\mathbb{E} [ \lvert : \hat{X}_{\bar{s},b}^{\epsilon, i_{1}'}(k_{1}') \hat{X}_{\bar{s},u}^{\epsilon, i_{2}'}(k_{2}') \hat{X}_{t,b}^{\epsilon, j}(k_{3}'): - : \hat{\bar{X}}_{\bar{s},b}^{\epsilon, i_{1}'}(k_{1}') \hat{\bar{X}}_{\bar{s},u}^{\epsilon, i_{2}'}(k_{2}') \hat{\bar{X}}_{t,b}^{\epsilon, j'} (k_{3}'): \rvert^{2} ])^{\frac{1}{2}} \nonumber 
\end{align}
by H$\ddot{\mathrm{o}}$lder's inequality. 
\begin{remark}\label{yet another remark}
Only $\mathbb{E} [ \lvert  : \hat{X}_{s,b}^{\epsilon, i_{1}}(k_{1}) \hat{X}_{s,u}^{\epsilon, i_{2}}(k_{2}) \hat{X}_{t,b}^{\epsilon, j}(k_{3}): - : \hat{\bar{X}}_{s,b}^{\epsilon, i_{1}}(k_{1}) \hat{\bar{X}}_{s,u}^{\epsilon, i_{2}}(k_{2}) \hat{\bar{X}}_{t,b}^{\epsilon, j} (k_{3}): \rvert^{2} ]$ will be computed in detail below; it is more complex than the analogous computations in case of the NS equations and require careful couplings of certain terms as we will see in \eqref{estimate 126}.
\end{remark}
We compute  
\begin{align}\label{estimate 177}
&\mathbb{E} [ \lvert  : \hat{X}_{s,b}^{\epsilon, i_{1}}(k_{1}) \hat{X}_{s,u}^{\epsilon, i_{2}}(k_{2}) \hat{X}_{t,b}^{\epsilon, j}(k_{3}): \nonumber\\
& \hspace{20mm} - : \hat{\bar{X}}_{s,b}^{\epsilon, i_{1}}(k_{1}) \hat{\bar{X}}_{s,u}^{\epsilon, i_{2}}(k_{2}) \hat{\bar{X}}_{t,b}^{\epsilon, j} (k_{3}): \rvert^{2} ]= \sum_{i=1}^{24} K_{k_{1}k_{2}k_{3}, i_{1}i_{2} j}^{i} 
\end{align} 
where 
\begin{subequations}
\begin{align}
K_{k_{1}k_{2}k_{3}, i_{1}i_{2}j}^{1} &\triangleq \sum_{i_{3}, i_{4}, i_{5} = 1}^{3} \frac{ h_{b}(\epsilon k_{1})^{2}}{2 \lvert k_{1} \rvert^{2} f(\epsilon k_{1})} \hat{\mathcal{P}}^{i_{1}i_{3}} (k_{1})^{2} \frac{ h_{u}(\epsilon k_{2})^{2}}{2 \lvert k_{2} \rvert^{2} f(\epsilon k_{2})}\nonumber\\
& \times \hat{\mathcal{P}}^{i_{2}i_{4}}(k_{2})^{2} \frac{ h_{b}(\epsilon k_{3})^{2}}{2 \lvert k_{3} \rvert^{2} f(\epsilon k_{3})} \hat{\mathcal{P}}^{ji_{5}}(k_{3})^{2},\\
K_{k_{1}k_{2}k_{3}, i_{1}i_{2}j}^{2} &\triangleq 1_{k_{2} =k_{3}} \sum_{i_{3}, i_{4}, i_{5} = 1}^{3} \frac{ h_{b}(\epsilon k_{1})^{2}}{2 \lvert k_{1} \rvert^{2} f(\epsilon k_{1})} \hat{\mathcal{P}}^{i_{1}i_{3}}(k_{1})^{2} \nonumber\\
& \times  \frac{ e^{- \lvert k_{2} \rvert^{2} f(\epsilon k_{2}) (t-s)} h_{u}(\epsilon k_{2}) h_{b}(\epsilon k_{2})}{2 \lvert k_{2} \rvert^{2} f(\epsilon k_{2})}  \hat{\mathcal{P}}^{i_{2} i_{4}}(k_{2}) \hat{\mathcal{P}}^{ji_{4}}(k_{2})  \nonumber\\
& \times \frac{ e^{- \lvert k_{2} \rvert^{2} f(\epsilon k_{2}) (t-s)} h_{u}(\epsilon k_{2}) h_{b}(\epsilon k_{2})}{2 \lvert k_{2} \rvert^{2} f(\epsilon k_{2})} \hat{\mathcal{P}}^{i_{2}i_{5}}(k_{2}) \hat{\mathcal{P}}^{ji_{5}}(k_{2}),   \\
K_{k_{1}k_{2}k_{3}, i_{1}i_{2}j}^{3} &\triangleq 1_{k_{1} =k_{2}} \sum_{i_{3}, i_{4}, i_{5} =1}^{3} \frac{ h_{b}(\epsilon k_{1}) h_{u}(\epsilon k_{1})}{2 \lvert k_{1} \rvert^{2} f(\epsilon k_{1})} \hat{\mathcal{P}}^{i_{1}i_{3}}(k_{1}) \hat{\mathcal{P}}^{i_{2}i_{3}} (k_{1}) \frac{ h_{u}(\epsilon k_{1}) h_{b}(\epsilon k_{1})}{2 \lvert k_{1} \rvert^{2} f(\epsilon k_{1})} \nonumber\\
& \times  \hat{\mathcal{P}}^{i_{2}i_{4}}(k_{1}) \hat{\mathcal{P}}^{i_{1}i_{4}}(k_{1}) \frac{h_{b}(\epsilon k_{3})^{2}}{2\lvert k_{3} \rvert^{2} f(\epsilon k_{3})} \hat{\mathcal{P}}^{ji_{5}}(k_{3})^{2}, \\
K_{k_{1}k_{2}k_{3}, i_{1}i_{2}j}^{4} &\triangleq 1_{k_{1} = k_{2} = k_{3}} \sum_{i_{3}, i_{4}, i_{5} =1}^{3} \frac{ h_{b}(\epsilon k_{1}) h_{u}(\epsilon k_{1})}{2 \lvert k_{1} \rvert^{2} f(\epsilon k_{1})} \hat{\mathcal{P}}^{i_{1}i_{3}}(k_{1}) \hat{\mathcal{P}}^{i_{2}i_{3}}(k_{1}) \nonumber\\
& \times  \frac{ e^{-\lvert k_{1} \rvert^{2} f(\epsilon k_{1}) (t-s)} h_{u}(\epsilon k_{1}) h_{b}(\epsilon k_{1})}{2\lvert k_{1} \rvert^{2} f(\epsilon k_{1})}\hat{\mathcal{P}}^{i_{2}i_{4}}(k_{1}) \hat{\mathcal{P}}^{ji_{4}}(k_{1})  \nonumber\\
& \times \frac{ e^{-\lvert k_{1} \rvert^{2} f(\epsilon k_{1}) (t-s)} h_{b}(\epsilon k_{1})^{2}}{2\lvert k_{1} \rvert^{2} f(\epsilon k_{1})} \hat{\mathcal{P}}^{ji_{5}}(k_{1}) \hat{\mathcal{P}}^{i_{1}i_{5}}(k_{1}),  \\
K_{k_{1}k_{2}k_{3}, i_{1}i_{2}j}^{5}& \triangleq 1_{k_{1}=k_{2}=k_{3}}  \sum_{i_{3}, i_{4}, i_{5} =1}^{3} \frac{ e^{- \lvert k_{1} \rvert^{2} f(\epsilon k_{1})(t-s)} h_{b}(\epsilon k_{1})^{2}}{2\lvert k_{1} \rvert^{2} f(\epsilon k_{1})} \hat{\mathcal{P}}^{i_{1}i_{3}}(k_{1}) \hat{\mathcal{P}}^{ji_{3}}(k_{1}) \nonumber\\
& \times \frac{ h_{u}(\epsilon k_{1}) h_{b}(\epsilon k_{1})}{2 \lvert k_{1} \rvert^{2} f(\epsilon k_{1})} \hat{\mathcal{P}}^{i_{2}i_{4}} (k_{1}) \hat{\mathcal{P}}^{i_{1}i_{4}}(k_{1}) \nonumber\\
& \times  \frac{ e^{- \lvert k_{1} \rvert^{2} f(\epsilon k_{1}) (t-s)} h_{b}(\epsilon k_{1}) h_{u}(\epsilon k_{1})}{2 \lvert k_{1} \rvert^{2} f(\epsilon k_{1})} \hat{\mathcal{P}}^{ji_{5}}(k_{1}) \hat{\mathcal{P}}^{i_{2}i_{5}}(k_{1}), \\
K_{k_{1}k_{2}k_{3}, i_{1}i_{2}j}^{6}&  \triangleq 1_{k_{1} = k_{3}} \sum_{i_{3}, i_{4}, i_{5} =1}^{3} \frac{ e^{- \lvert k_{1} \rvert^{2} f(\epsilon k_{1}) (t-s)} h_{b}(\epsilon k_{1})^{2}}{2 \lvert k_{1} \rvert^{2} f(\epsilon k_{1})} \hat{\mathcal{P}}^{i_{1}i_{3}}(k_{1}) \hat{\mathcal{P}}^{ji_{3}}(k_{1}) \nonumber\\
& \hspace{5mm} \times  \frac{ h_{u}(\epsilon k_{2})^{2}}{2 \lvert k_{2} \rvert^{2} f(\epsilon k_{2})}  \hat{\mathcal{P}}^{i_{2}i_{4}}(k_{2})^{2}  \frac{ e^{-\lvert k_{1} \rvert^{2} f(\epsilon k_{1}) (t-s)} h_{b}(\epsilon k_{1})^{2}}{2\lvert k_{1} \rvert^{2} f(\epsilon k_{1})} \nonumber\\
& \hspace{10mm} \times  \hat{\mathcal{P}}^{ji_{5}}(k_{1}) \hat{\mathcal{P}}^{i_{1}i_{5}}(k_{1}), \\
K_{k_{1}k_{2}k_{3}, i_{1}i_{2}j}^{7} &\triangleq -\sum_{i_{3}, i_{4}, i_{5} = 1}^{3} \frac{ h_{b}(\epsilon k_{1})^{2}}{ \lvert k_{1} \rvert^{2} (f(\epsilon k_{1}) + 1)} \hat{\mathcal{P}}^{i_{1}i_{3}} (k_{1})^{2} \frac{ h_{u}(\epsilon k_{2})^{2}}{\lvert k_{2} \rvert^{2} (f(\epsilon k_{2}) + 1)}\nonumber\\
& \times \hat{\mathcal{P}}^{i_{2}i_{4}}(k_{2})^{2} \frac{ h_{b}(\epsilon k_{3})^{2}}{\lvert k_{3} \rvert^{2} (f(\epsilon k_{3}) + 1)} \hat{\mathcal{P}}^{ji_{5}}(k_{3})^{2},\\
K_{k_{1}k_{2}k_{3}, i_{1}i_{2}j}^{8} &\triangleq -1_{k_{2} =k_{3}} \sum_{i_{3}, i_{4}, i_{5} = 1}^{3} \frac{ h_{b}(\epsilon k_{1})^{2}}{ \lvert k_{1} \rvert^{2} (f(\epsilon k_{1}) + 1)} \hat{\mathcal{P}}^{i_{1}i_{3}}(k_{1})^{2} \nonumber\\
& \times  \frac{ e^{- \lvert k_{2} \rvert^{2} (t-s)} h_{u}(\epsilon k_{2}) h_{b}(\epsilon k_{2})}{\lvert k_{2} \rvert^{2} (f(\epsilon k_{2}) + 1)}  \hat{\mathcal{P}}^{i_{2} i_{4}}(k_{2}) \hat{\mathcal{P}}^{ji_{4}}(k_{2}) \nonumber\\
& \times  \frac{ e^{- \lvert k_{2} \rvert^{2} f(\epsilon k_{2}) (t-s)} h_{u}(\epsilon k_{2}) h_{b}(\epsilon k_{2})}{ \lvert k_{2} \rvert^{2}( f(\epsilon k_{2}) + 1)} \hat{\mathcal{P}}^{i_{2}i_{5}}(k_{2}) \hat{\mathcal{P}}^{ji_{5}}(k_{2}),  \\ 
K_{k_{1}k_{2}k_{3}, i_{1}i_{2}j}^{9} &\triangleq -1_{k_{1} =k_{2}} \sum_{i_{3}, i_{4}, i_{5} =1}^{3} \frac{ h_{b}(\epsilon k_{1}) h_{u}(\epsilon k_{1})}{ \lvert k_{1} \rvert^{2} (f(\epsilon k_{1}) + 1)} \hat{\mathcal{P}}^{i_{1}i_{3}}(k_{1}) \hat{\mathcal{P}}^{i_{2}i_{3}} (k_{1})  \\
& \times \frac{ h_{u}(\epsilon k_{1}) h_{b}(\epsilon k_{1})}{ \lvert k_{1} \rvert^{2} (f(\epsilon k_{1}) + 1)} \hat{\mathcal{P}}^{i_{2}i_{4}}(k_{1}) \hat{\mathcal{P}}^{i_{1}i_{4}}(k_{1}) \frac{h_{b}(\epsilon k_{3})^{2}}{\lvert k_{3} \rvert^{2} (f(\epsilon k_{3}) +  1)} \hat{\mathcal{P}}^{ji_{5}}(k_{3})^{2}, \nonumber\\
K_{k_{1}k_{2}k_{3}, i_{1}i_{2}j}^{10} &\triangleq -1_{k_{1} = k_{2} = k_{3}} \sum_{i_{3}, i_{4}, i_{5} =1}^{3} \frac{ h_{b}(\epsilon k_{1}) h_{u}(\epsilon k_{1})}{\lvert k_{1} \rvert^{2} (f(\epsilon k_{1}) + 1)} \hat{\mathcal{P}}^{i_{1}i_{3}}(k_{1}) \hat{\mathcal{P}}^{i_{2}i_{3}}(k_{1}) \nonumber\\
& \times  \frac{ e^{-\lvert k_{1} \rvert^{2} (t-s)} h_{u}(\epsilon k_{1}) h_{b}(\epsilon k_{1})}{\lvert k_{1} \rvert^{2} (f(\epsilon k_{1}) + 1)}\hat{\mathcal{P}}^{i_{2}i_{4}}(k_{1}) \hat{\mathcal{P}}^{ji_{4}}(k_{1}) \nonumber\\
& \times  \frac{ e^{-\lvert k_{1} \rvert^{2} f(\epsilon k_{1}) (t-s)} h_{b}(\epsilon k_{1})^{2}}{\lvert k_{1} \rvert^{2} (f(\epsilon k_{1}) + 1)} \hat{\mathcal{P}}^{ji_{5}}(k_{1}) \hat{\mathcal{P}}^{i_{1}i_{5}}(k_{1}),  \\
K_{k_{1}k_{2}k_{3}, i_{1}i_{2}j}^{11}& \triangleq -1_{k_{1}=k_{2}=k_{3}}  \sum_{i_{3}, i_{4}, i_{5} =1}^{3} \frac{ e^{- \lvert k_{1} \rvert^{2} (t-s)} h_{b}(\epsilon k_{1})^{2}}{\lvert k_{1} \rvert^{2} (f(\epsilon k_{1})+ 1)} \hat{\mathcal{P}}^{i_{1}i_{3}}(k_{1}) \hat{\mathcal{P}}^{ji_{3}}(k_{1})  \nonumber\\
& \times  \frac{ h_{u}(\epsilon k_{1}) h_{b}(\epsilon k_{1})}{\lvert k_{1} \rvert^{2} (f(\epsilon k_{1}) + 1)} \hat{\mathcal{P}}^{i_{2}i_{4}} (k_{1}) \hat{\mathcal{P}}^{i_{1}i_{4}}(k_{1}) \nonumber\\
& \times  \frac{ e^{- \lvert k_{1} \rvert^{2} f(\epsilon k_{1}) (t-s)} h_{b}(\epsilon k_{1}) h_{u}(\epsilon k_{1})}{\lvert k_{1} \rvert^{2} (f(\epsilon k_{1}) + 1)} \hat{\mathcal{P}}^{ji_{5}}(k_{1}) \hat{\mathcal{P}}^{i_{2}i_{5}}(k_{1}), \\
K_{k_{1}k_{2}k_{3}, i_{1}i_{2}j}^{12}&  \triangleq -1_{k_{1} = k_{3}} \sum_{i_{3}, i_{4}, i_{5} =1}^{3} \frac{ e^{- \lvert k_{1} \rvert^{2}(t-s)} h_{b}(\epsilon k_{1})^{2}}{\lvert k_{1} \rvert^{2} (f(\epsilon k_{1}) + 1)} \hat{\mathcal{P}}^{i_{1}i_{3}}(k_{1}) \hat{\mathcal{P}}^{ji_{3}}(k_{1})  \nonumber\\
& \times \frac{ h_{u}(\epsilon k_{2})^{2}}{ \lvert k_{2} \rvert^{2} (f(\epsilon k_{2}) + 1)}  \hat{\mathcal{P}}^{i_{2}i_{4}}(k_{2})^{2} \nonumber\\
& \times  \frac{ e^{-\lvert k_{1} \rvert^{2} f(\epsilon k_{1}) (t-s)} h_{b}(\epsilon k_{1})^{2}}{\lvert k_{1} \rvert^{2} (f(\epsilon k_{1}) + 1)} \hat{\mathcal{P}}^{ji_{5}}(k_{1}) \hat{\mathcal{P}}^{i_{1}i_{5}}(k_{1}), \\
K_{k_{1}k_{2}k_{3}, i_{1}i_{2}j}^{13} &\triangleq -\sum_{i_{3}, i_{4}, i_{5} = 1}^{3} \frac{ h_{b}(\epsilon k_{1})^{2}}{\lvert k_{1} \rvert^{2} (f(\epsilon k_{1})+ 1)} \hat{\mathcal{P}}^{i_{1}i_{3}} (k_{1})^{2} \frac{ h_{u}(\epsilon k_{2})^{2}}{\lvert k_{2} \rvert^{2} (f(\epsilon k_{2}) + 1)}\nonumber\\
& \times \hat{\mathcal{P}}^{i_{2}i_{4}}(k_{2})^{2} \frac{ h_{b}(\epsilon k_{3})^{2}}{\lvert k_{3} \rvert^{2} (f(\epsilon k_{3}) + 1)} \hat{\mathcal{P}}^{ji_{5}}(k_{3})^{2},\\
K_{k_{1}k_{2}k_{3}, i_{1}i_{2}j}^{14} &\triangleq -1_{k_{2} =k_{3}} \sum_{i_{3}, i_{4}, i_{5} = 1}^{3} \frac{ h_{b}(\epsilon k_{1})^{2}}{\lvert k_{1} \rvert^{2} (f(\epsilon k_{1}) + 1)} \hat{\mathcal{P}}^{i_{1}i_{3}}(k_{1})^{2} \nonumber\\
& \times  \frac{ e^{- \lvert k_{2} \rvert^{2} f(\epsilon k_{2}) (t-s)} h_{u}(\epsilon k_{2}) h_{b}(\epsilon k_{2})}{\lvert k_{2} \rvert^{2} (f(\epsilon k_{2})+ 1)} \hat{\mathcal{P}}^{i_{2} i_{4}}(k_{2}) \hat{\mathcal{P}}^{ji_{4}}(k_{2}) \nonumber\\
& \times \frac{ e^{- \lvert k_{2} \rvert^{2} (t-s)} h_{u}(\epsilon k_{2}) h_{b}(\epsilon k_{2})}{\lvert k_{2} \rvert^{2} (f(\epsilon k_{2}) + 1)} \hat{\mathcal{P}}^{i_{2}i_{5}}(k_{2}) \hat{\mathcal{P}}^{ji_{5}}(k_{2}),  \\
K_{k_{1}k_{2}k_{3}, i_{1}i_{2}j}^{15} &\triangleq -1_{k_{1} =k_{2}} \sum_{i_{3}, i_{4}, i_{5} =1}^{3} \frac{ h_{b}(\epsilon k_{1}) h_{u}(\epsilon k_{1})}{ \lvert k_{1} \rvert^{2} (f(\epsilon k_{1})+ 1)} \hat{\mathcal{P}}^{i_{1}i_{3}}(k_{1}) \hat{\mathcal{P}}^{i_{2}i_{3}} (k_{1}) \nonumber\\
& \times \frac{ h_{u}(\epsilon k_{1}) h_{b}(\epsilon k_{1})}{\lvert k_{1} \rvert^{2} (f(\epsilon k_{1}) + 1)} \hat{\mathcal{P}}^{i_{2}i_{4}}(k_{1}) \hat{\mathcal{P}}^{i_{1}i_{4}}(k_{1}) \nonumber\\
& \times  \frac{h_{b}(\epsilon k_{3})^{2}}{\lvert k_{3} \rvert^{2} (f(\epsilon k_{3})+ 1)} \hat{\mathcal{P}}^{ji_{5}}(k_{3})^{2}, \\
K_{k_{1}k_{2}k_{3}, i_{1}i_{2}j}^{16}&\triangleq -1_{k_{1} = k_{2} = k_{3}} \sum_{i_{3}, i_{4}, i_{5} =1}^{3} \frac{ h_{b}(\epsilon k_{1}) h_{u}(\epsilon k_{1})}{ \lvert k_{1} \rvert^{2} (f(\epsilon k_{1}) + 1)} \hat{\mathcal{P}}^{i_{1}i_{3}}(k_{1}) \hat{\mathcal{P}}^{i_{2}i_{3}}(k_{1}) \nonumber\\
& \times  \frac{ e^{-\lvert k_{1} \rvert^{2} f(\epsilon k_{1}) (t-s)} h_{u}(\epsilon k_{1}) h_{b}(\epsilon k_{1})}{\lvert k_{1} \rvert^{2} (f(\epsilon k_{1}) + 1)} \nonumber\\
& \times \hat{\mathcal{P}}^{i_{2}i_{4}}(k_{1}) \hat{\mathcal{P}}^{ji_{4}}(k_{1}) \frac{ e^{-\lvert k_{1} \rvert^{2} (t-s)} h_{b}(\epsilon k_{1})^{2}}{\lvert k_{1} \rvert^{2} (f(\epsilon k_{1}) + 1)} \hat{\mathcal{P}}^{ji_{5}}(k_{1}) \hat{\mathcal{P}}^{i_{1}i_{5}}(k_{1}),  \\
K_{k_{1}k_{2}k_{3}, i_{1}i_{2}j}^{17}& \triangleq -1_{k_{1}=k_{2}=k_{3}}  \sum_{i_{3}, i_{4}, i_{5} =1}^{3} \frac{ e^{- \lvert k_{1} \rvert^{2} f(\epsilon k_{1})(t-s)} h_{b}(\epsilon k_{1})^{2}}{\lvert k_{1} \rvert^{2} (f(\epsilon k_{1}) + 1)} \hat{\mathcal{P}}^{i_{1}i_{3}}(k_{1}) \hat{\mathcal{P}}^{ji_{3}}(k_{1}) \nonumber\\
& \times \frac{ h_{u}(\epsilon k_{1}) h_{b}(\epsilon k_{1})}{\lvert k_{1} \rvert^{2} (f(\epsilon k_{1}) + 1)} \hat{\mathcal{P}}^{i_{2}i_{4}} (k_{1}) \hat{\mathcal{P}}^{i_{1}i_{4}}(k_{1}) \nonumber\\
& \times  \frac{ e^{- \lvert k_{1} \rvert^{2} (t-s)} h_{b}(\epsilon k_{1}) h_{u}(\epsilon k_{1})}{\lvert k_{1} \rvert^{2} (f(\epsilon k_{1}) + 1)} \hat{\mathcal{P}}^{ji_{5}}(k_{1}) \hat{\mathcal{P}}^{i_{2}i_{5}}(k_{1}), \\
K_{k_{1}k_{2}k_{3}, i_{1}i_{2}j}^{18}&  \triangleq -1_{k_{1} = k_{3}} \sum_{i_{3}, i_{4}, i_{5} =1}^{3} \frac{ e^{- \lvert k_{1} \rvert^{2} f(\epsilon k_{1}) (t-s)} h_{b}(\epsilon k_{1})^{2}}{\lvert k_{1} \rvert^{2} (f(\epsilon k_{1}) + 1)} \hat{\mathcal{P}}^{i_{1}i_{3}}(k_{1}) \hat{\mathcal{P}}^{ji_{3}}(k_{1}) \nonumber\\
& \times  \frac{ h_{u}(\epsilon k_{2})^{2}}{\lvert k_{2} \rvert^{2} (f(\epsilon k_{2}) + 1)}  \hat{\mathcal{P}}^{i_{2}i_{4}}(k_{2})^{2} \nonumber\\
& \times  \frac{ e^{-\lvert k_{1} \rvert^{2} (t-s)} h_{b}(\epsilon k_{1})^{2}}{\lvert k_{1} \rvert^{2} (f(\epsilon k_{1})+ 1)} \hat{\mathcal{P}}^{ji_{5}}(k_{1}) \hat{\mathcal{P}}^{i_{1}i_{5}}(k_{1}), \\
K_{k_{1}k_{2}k_{3}, i_{1}i_{2}j}^{19} &\triangleq \sum_{i_{3}, i_{4}, i_{5} = 1}^{3} \frac{ h_{b}(\epsilon k_{1})^{2}}{2 \lvert k_{1} \rvert^{2}} \hat{\mathcal{P}}^{i_{1}i_{3}} (k_{1})^{2} \frac{ h_{u}(\epsilon k_{2})^{2}}{2 \lvert k_{2} \rvert^{2}}\nonumber\\
& \times \hat{\mathcal{P}}^{i_{2}i_{4}}(k_{2})^{2} \frac{ h_{b}(\epsilon k_{3})^{2}}{2 \lvert k_{3} \rvert^{2} } \hat{\mathcal{P}}^{ji_{5}}(k_{3})^{2},\\ 
K_{k_{1}k_{2}k_{3}, i_{1}i_{2}j}^{20} &\triangleq 1_{k_{2} =k_{3}} \sum_{i_{3}, i_{4}, i_{5} = 1}^{3} \frac{ h_{b}(\epsilon k_{1})^{2}}{2 \lvert k_{1} \rvert^{2}} \hat{\mathcal{P}}^{i_{1}i_{3}}(k_{1})^{2}  \nonumber\\
& \times \frac{ e^{- \lvert k_{2} \rvert^{2} (t-s)} h_{u}(\epsilon k_{2}) h_{b}(\epsilon k_{2})}{2 \lvert k_{2} \rvert^{2}} \hat{\mathcal{P}}^{i_{2} i_{4}}(k_{2}) \hat{\mathcal{P}}^{ji_{4}}(k_{2}) \nonumber\\
& \times  \frac{ e^{- \lvert k_{2} \rvert^{2} (t-s)} h_{u}(\epsilon k_{2}) h_{b}(\epsilon k_{2})}{2 \lvert k_{2} \rvert^{2}} \hat{\mathcal{P}}^{i_{2}i_{5}}(k_{2}) \hat{\mathcal{P}}^{ji_{5}}(k_{2}),  \\
K_{k_{1}k_{2}k_{3}, i_{1}i_{2}j}^{21} &\triangleq 1_{k_{1} =k_{2}} \sum_{i_{3}, i_{4}, i_{5} =1}^{3} \frac{ h_{b}(\epsilon k_{1}) h_{u}(\epsilon k_{1})}{2 \lvert k_{1} \rvert^{2} } \hat{\mathcal{P}}^{i_{1}i_{3}}(k_{1}) \hat{\mathcal{P}}^{i_{2}i_{3}} (k_{1}) \frac{ h_{u}(\epsilon k_{1}) h_{b}(\epsilon k_{1})}{2 \lvert k_{1} \rvert^{2} } \nonumber\\
& \times  \hat{\mathcal{P}}^{i_{2}i_{4}}(k_{1}) \hat{\mathcal{P}}^{i_{1}i_{4}}(k_{1}) \frac{h_{b}(\epsilon k_{3})^{2}}{2\lvert k_{3} \rvert^{2}} \hat{\mathcal{P}}^{ji_{5}}(k_{3})^{2}, \\
K_{k_{1}k_{2}k_{3}, i_{1}i_{2}j}^{22} &\triangleq 1_{k_{1} = k_{2} = k_{3}} \sum_{i_{3}, i_{4}, i_{5} =1}^{3} \frac{ h_{b}(\epsilon k_{1}) h_{u}(\epsilon k_{1})}{2 \lvert k_{1} \rvert^{2}} \hat{\mathcal{P}}^{i_{1}i_{3}}(k_{1}) \hat{\mathcal{P}}^{i_{2}i_{3}}(k_{1}) \nonumber\\
& \times  \frac{ e^{-\lvert k_{1} \rvert^{2}  (t-s)} h_{u}(\epsilon k_{1}) h_{b}(\epsilon k_{1})}{2\lvert k_{1} \rvert^{2} } \nonumber\\
& \times \hat{\mathcal{P}}^{i_{2}i_{4}}(k_{1}) \hat{\mathcal{P}}^{ji_{4}}(k_{1}) \frac{ e^{-\lvert k_{1} \rvert^{2} (t-s)} h_{b}(\epsilon k_{1})^{2}}{2\lvert k_{1} \rvert^{2}} \hat{\mathcal{P}}^{ji_{5}}(k_{1}) \hat{\mathcal{P}}^{i_{1}i_{5}}(k_{1}),  \\
K_{k_{1}k_{2}k_{3}, i_{1}i_{2}j}^{23}& \triangleq 1_{k_{1}=k_{2}=k_{3}}  \sum_{i_{3}, i_{4}, i_{5} =1}^{3} \frac{ e^{- \lvert k_{1} \rvert^{2} (t-s)} h_{b}(\epsilon k_{1})^{2}}{2\lvert k_{1} \rvert^{2}} \hat{\mathcal{P}}^{i_{1}i_{3}}(k_{1}) \hat{\mathcal{P}}^{ji_{3}}(k_{1}) \nonumber\\
& \times  \frac{ h_{u}(\epsilon k_{1}) h_{b}(\epsilon k_{1})}{2 \lvert k_{1} \rvert^{2} }  \hat{\mathcal{P}}^{i_{2}i_{4}} (k_{1}) \hat{\mathcal{P}}^{i_{1}i_{4}}(k_{1}) \nonumber\\
& \times  \frac{ e^{- \lvert k_{1} \rvert^{2} (t-s)} h_{b}(\epsilon k_{1}) h_{u}(\epsilon k_{1})}{2 \lvert k_{1} \rvert^{2} } \hat{\mathcal{P}}^{ji_{5}}(k_{1}) \hat{\mathcal{P}}^{i_{2}i_{5}}(k_{1}), \\
K_{k_{1}k_{2}k_{3}, i_{1}i_{2}j}^{24}&  \triangleq 1_{k_{1} = k_{3}} \sum_{i_{3}, i_{4}, i_{5} =1}^{3} \frac{ e^{- \lvert k_{1} \rvert^{2} (t-s)} h_{b}(\epsilon k_{1})^{2}}{2 \lvert k_{1} \rvert^{2}} \hat{\mathcal{P}}^{i_{1}i_{3}}(k_{1}) \hat{\mathcal{P}}^{ji_{3}}(k_{1}) \frac{ h_{u}(\epsilon k_{2})^{2}}{2 \lvert k_{2} \rvert^{2}}  \nonumber\\
& \times \hat{\mathcal{P}}^{i_{2}i_{4}}(k_{2})^{2}  \frac{ e^{-\lvert k_{1} \rvert^{2} (t-s)} h_{b}(\epsilon k_{1})^{2}}{2\lvert k_{1} \rvert^{2}} \hat{\mathcal{P}}^{ji_{5}}(k_{1}) \hat{\mathcal{P}}^{i_{1}i_{5}}(k_{1})
\end{align}
\end{subequations} 
by Example \ref{Example 3.1} and \eqref{covariance a} - \eqref{estimate 174}. Now we strategically group together 
\begin{subequations}\label{estimate 126}
\begin{align}
& \sum_{i\in \{1,7,13,19\}} K_{k_{1}k_{2}k_{3}, i_{1}i_{2} j}^{i} = \nonumber\\
& \times \sum_{i_{3}, i_{4}, i_{5} =1}^{3} \frac{ h_{b}(\epsilon k_{1})^{2} h_{u}(\epsilon k_{2})^{2} h_{b}(\epsilon k_{3})^{2} \hat{\mathcal{P}}^{i_{1}i_{3}} (k_{1})^{2} \hat{\mathcal{P}}^{i_{2}i_{4}}(k_{2})^{2} \hat{\mathcal{P}}^{ji_{5}}(k_{3})^{2}}{\lvert k_{1} \rvert^{2} \lvert k_{2} \rvert^{2} \lvert k_{3} \rvert^{2}} \nonumber\\
& \times [ \frac{1}{8 \prod_{i=1}^{3} f(\epsilon k_{i})} - \frac{2}{\prod_{i=1}^{3} (f(\epsilon k_{i} + 1)} + \frac{1}{8}],  \\
&  \sum_{i\in \{2,8,14,20\}} K_{k_{1}k_{2}k_{3}, i_{1}i_{2} j}^{i} = 1_{k_{2} = k_{3}} \nonumber\\
& \times \sum_{i_{3}, i_{4}, i_{5} = 1}^{3} \frac{ h_{b}(\epsilon k_{1})^{2} h_{u}(\epsilon k_{2})^{2} h_{b}(\epsilon k_{2})^{2} \hat{\mathcal{P}}^{i_{1}i_{3}}(k_{1})^{2} \hat{\mathcal{P}}^{i_{2}i_{4}}(k_{2}) \hat{\mathcal{P}}^{ji_{4}}(k_{2}) \hat{\mathcal{P}}^{ji_{5}}(k_{2}) \hat{\mathcal{P}}^{i_{2}i_{5}}(k_{2})}{\lvert k_{1}\rvert^{2} \lvert k_{2} \rvert^{2} \lvert k_{3} \rvert^{2}} \nonumber\\
& \times [ \frac{ e^{-2 \lvert k_{2} \rvert^{2} f(\epsilon k_{2}) (t-s)}}{8 \prod_{i=1}^{2} f(\epsilon k_{i})} - \frac{ 2 e^{- \lvert k_{2} \rvert^{2} (t-s) (f(\epsilon k_{2}) + 1)}}{\prod_{i=1}^{3} (f(\epsilon k_{i}) + 1)} + \frac{ e^{-2 \lvert k_{2} \rvert^{2} (t-s)}}{8} ], \\
&\sum_{i\in \{3,9,15,21\}} K_{k_{1}k_{2}k_{3}, i_{1}i_{2} j}^{i} = 1_{k_{1} = k_{2}} \nonumber\\
& \times\sum_{i_{3}, i_{4}, i_{5} =1}^{3} \frac{ h_{b}^{2}(\epsilon k_{1}) h_{u}(\epsilon k_{1})^{2} h_{b}(\epsilon k_{3})^{2} \hat{\mathcal{P}}^{i_{1}i_{3}}(k_{1}) \hat{\mathcal{P}}^{i_{2}i_{3}}(k_{1}) \hat{\mathcal{P}}^{i_{2}i_{4}}(k_{1}) \hat{\mathcal{P}}^{i_{1}i_{4}}(k_{1}) \hat{\mathcal{P}}^{ji_{5}}(k_{3})^{2}}{\lvert k_{1} \rvert^{2} \lvert k_{2} \rvert^{2} \lvert k_{3} \rvert^{2}} \nonumber\\
& \times [ \frac{1}{8 \prod_{i=1}^{3} f(\epsilon k_{i})} - \frac{2}{\prod_{i=1}^{2} (f(\epsilon k_{i}) + 1)} + \frac{1}{8}], \\
& \sum_{i\in \{4,10,16,22\}} K_{k_{1}k_{2}k_{3}, i_{1}i_{2} j}^{i} = 1_{k_{1} = k_{2} = k_{3}} \nonumber\\
& \times \sum_{i_{3}, i_{4}, i_{5} = 1}^{3} \frac{ h_{b}(\epsilon k_{1})^{4} h_{u}(\epsilon k_{1})^{2}  \hat{\mathcal{P}}^{i_{1}i_{3}}(k_{1}) \hat{\mathcal{P}}^{i_{2}i_{3}}(k_{1}) \hat{\mathcal{P}}^{i_{2}i_{4}}(k_{1}) \hat{\mathcal{P}}^{ji_{4}}(k_{1}) \hat{\mathcal{P}}^{ji_{5}}(k_{1}) \hat{\mathcal{P}}^{i_{1} i_{5}}(k_{1})}{\lvert k_{1}\rvert^{2} \lvert k_{2} \rvert^{2} \lvert k_{3} \rvert^{2}} \nonumber\\
& \times [ \frac{ e^{-2 \lvert k_{1} \rvert^{2} f(\epsilon k_{1}) (t-s)}}{8 \prod_{i=1}^{2} f(\epsilon k_{i})} - \frac{ 2 e^{- \lvert k_{1} \rvert^{2} (t-s) (f(\epsilon k_{1}) + 1)}}{\prod_{i=1}^{3} (f(\epsilon k_{i}) + 1)} + \frac{ e^{-2 \lvert k_{1} \rvert^{2} (t-s)}}{8} ], \\
& \sum_{i\in \{5,11,17,23\}} K_{k_{1}k_{2}k_{3}, i_{1}i_{2} j}^{i} = 1_{k_{1} = k_{2} = k_{3}} \nonumber\\
& \times \sum_{i_{3}, i_{4}, i_{5} = 1}^{3} \frac{ h_{b}(\epsilon k_{1})^{4} h_{u}(\epsilon k_{1})^{2}  \hat{\mathcal{P}}^{i_{1}i_{3}}(k_{1}) \hat{\mathcal{P}}^{ji_{3}}(k_{1}) \hat{\mathcal{P}}^{i_{2}i_{4}}(k_{1}) \hat{\mathcal{P}}^{i_{1}i_{4}}(k_{1}) \hat{\mathcal{P}}^{ji_{5}}(k_{1}) \hat{\mathcal{P}}^{i_{2} i_{5}}(k_{1})}{\lvert k_{1}\rvert^{2} \lvert k_{2} \rvert^{2} \lvert k_{3} \rvert^{2}} \nonumber\\
& \times [ \frac{ e^{-2 \lvert k_{1} \rvert^{2} f(\epsilon k_{1}) (t-s)}}{8 \prod_{i=1}^{2} f(\epsilon k_{i})} - \frac{ 2 e^{- \lvert k_{1} \rvert^{2} (t-s) (f(\epsilon k_{1}) + 1)}}{\prod_{i=1}^{3} (f(\epsilon k_{i}) + 1)} + \frac{ e^{-2 \lvert k_{1} \rvert^{2} (t-s)}}{8} ], \\
&  \sum_{i\in \{6,12,18,24\}} K_{k_{1}k_{2}k_{3}, i_{1}i_{2} j}^{i} = 1_{k_{1} = k_{3}} \nonumber\\
& \times \sum_{i_{3}, i_{4}, i_{5} = 1}^{3} \frac{ h_{b}(\epsilon k_{1})^{4} h_{u}(\epsilon k_{2})^{2}  \hat{\mathcal{P}}^{i_{1}i_{3}}(k_{1}) \hat{\mathcal{P}}^{ji_{3}}(k_{1}) \hat{\mathcal{P}}^{i_{2}i_{4}}(k_{2})^{2} \hat{\mathcal{P}}^{ji_{5}}(k_{1}) \hat{\mathcal{P}}^{i_{1}i_{5}}(k_{1})}{\lvert k_{1}\rvert^{2} \lvert k_{2} \rvert^{2} \lvert k_{3} \rvert^{2}} \nonumber\\
& \times [ \frac{ e^{-2 \lvert k_{1} \rvert^{2} f(\epsilon k_{1}) (t-s)}}{8 \prod_{i=1}^{2} f(\epsilon k_{i})} - \frac{ 2 e^{- \lvert k_{1} \rvert^{2} (t-s) (f(\epsilon k_{1}) + 1)}}{\prod_{i=1}^{3} (f(\epsilon k_{i}) + 1)} + \frac{ e^{-2 \lvert k_{1} \rvert^{2} (t-s)}}{8} ]. 
\end{align} 
\end{subequations}
We can compute for any $\eta \in [0,1]$ by mean value theorem that 
\begin{align}\label{estimate 175}
\lvert \frac{ e^{-2 \lvert k_{1} \rvert^{2} f(\epsilon k_{1}) (t-s)}}{8 \prod_{i=1}^{2} f(\epsilon k_{i})} - \frac{ 2 e^{- \lvert k_{1} \rvert^{2} (t-s) (f(\epsilon k_{1}) + 1)}}{\prod_{i=1}^{3} (f(\epsilon k_{i}) + 1)} + \frac{ e^{-2 \lvert k_{1} \rvert^{2} (t-s)}}{8} \rvert  \lesssim \sum_{i=1}^{3} \lvert \epsilon k_{i} \rvert^{\eta}, 
\end{align} 
from which it is also clear that 
 \begin{align}\label{estimate 176}
 \lvert \frac{1}{8 \prod_{i=1}^{3} f(\epsilon k_{i})} - \frac{2}{\prod_{i=1}^{2} (f(\epsilon k_{i}) + 1)} + \frac{1}{8} \rvert \lesssim \sum_{i=1}^{3} \lvert \epsilon k_{i} \rvert^{\eta}. 
 \end{align} 
Applying \eqref{estimate 175}-\eqref{estimate 176} to \eqref{estimate 126} and \eqref{estimate 177} gives 
\begin{align}\label{estimate 178}
&\mathbb{E} [ \lvert  : \hat{X}_{s,b}^{\epsilon, i_{1}}(k_{1}) \hat{X}_{s,u}^{\epsilon, i_{2}}(k_{2}) \hat{X}_{t,b}^{\epsilon, j}(k_{3}): \nonumber\\
& \hspace{20mm} - : \hat{\bar{X}}_{s,b}^{\epsilon, i_{1}}(k_{1}) \hat{\bar{X}}_{s,u}^{\epsilon, i_{2}}(k_{2}) \hat{\bar{X}}_{t,b}^{\epsilon, j} (k_{3}): \rvert^{2} ] \lesssim \frac{ \sum_{i=1}^{3} \lvert \epsilon k_{i} \rvert^{\eta}}{\prod_{i=1}^{3} \lvert k_{i} \rvert^{2}}. 
\end{align} 
Similar computations can show that an analogous upper bound applies also to $\mathbb{E} [ \lvert : \hat{X}_{\bar{s},b}^{\epsilon, i_{1}'}(k_{1}') \hat{X}_{\bar{s},u}^{\epsilon, i_{2}'}(k_{2}') \hat{X}_{t,b}^{\epsilon, j}(k_{3}'): - : \hat{\bar{X}}_{\bar{s},b}^{\epsilon, i_{1}'}(k_{1}') \hat{\bar{X}}_{\bar{s},u}^{\epsilon, i_{2}'}(k_{2}') \hat{\bar{X}}_{t,b}^{\epsilon, j'} (k_{3}'): \rvert^{2} ]$. Thus, we finally estimate from \eqref{estimate 173}, 
\begin{align}\label{estimate 180}
I_{q,t}^{2}\lesssim&  \sum_{k} \theta(2^{-q} k)^{2} \sum_{k_{1}, k_{2}, k_{3}, k_{1}', k_{2}', k_{3}' \neq 0: k_{123} = k_{123}'= k} \\
& \times (1_{k_{1} = k_{1}', k_{2} = k_{2}', k_{3} = k_{3}'} +1_{k_{1} = k_{1}', k_{2} = k_{3}', k_{3} = k_{2}'})    \int_{[0,t]^{2}} e^{- \lvert k_{12} \rvert^{2} (t-s) - \lvert k_{12} ' \rvert^{2} (t- \bar{s})} \nonumber\\
& \hspace{5mm} \times \lvert k_{12} \rvert \lvert k_{12}' \rvert \left( \frac{ \sum_{i=1}^{3} \lvert \epsilon k_{i} \rvert^{\eta}}{\prod_{i=1}^{3} \lvert k_{i} \rvert^{2}}  +  \frac{ \sum_{i=1}^{3} \lvert \epsilon k_{i} ' \rvert^{\eta}}{\prod_{i=1}^{3} \lvert k_{i} ' \rvert^{2}} \right) ds d \bar{s} \nonumber \\
\lesssim&  \epsilon^{\eta} t^{\eta} \sum_{k} \theta(2^{-q} k)^{2} \sum_{k_{1}, k_{2}, k_{3} \neq 0: k_{123} = k}  \nonumber\\
& \hspace{5mm} \times [ \lvert k_{12} \rvert^{-2 + 2\eta} \frac{ \sum_{i=1}^{3} \lvert k_{i} \rvert^{\eta}}{\prod_{i=1}^{3} \lvert k_{i} \rvert^{2}} + \lvert k_{12} \rvert^{-1 + \eta} \lvert k_{13} \rvert^{-1+ \eta} \frac{ \sum_{i=1}^{3} \lvert k_{i} \rvert^{\eta}}{\prod_{i=1}^{3} \lvert k_{i} \rvert^{2}} ] \lesssim \epsilon^{\eta} t^{\eta} 2^{q(1+ 3\eta)} \nonumber
\end{align}
by \eqref{estimate 172}, H$\ddot{\mathrm{o}}$lder's inequality, \eqref{estimate 178}, \eqref{key estimate}, Lemma \ref{Lemma 3.13}. Considering \eqref{estimate 179} and \eqref{estimate 180} in \eqref{estimate 181} leads to 
\begin{equation}\label{estimate 182}
\mathbb{E} [ \lvert \Delta_{q} V_{t, ii_{1}}^{1} \rvert^{2}] \lesssim \epsilon^{\eta} t^{\eta} 2^{q(1+ 3\eta)}.
\end{equation} 
Considering \eqref{estimate 144}, \eqref{estimate 153}, \eqref{estimate 163}, \eqref{estimate 170} and \eqref{estimate 182} in \eqref{estimate 134} leads to
\begin{equation}
\mathbb{E}[ \lvert \Delta_{q} ( \bar{b}_{1}^{\epsilon, j} \diamond \bar{b}_{2}^{\epsilon, i} (t) - b_{1}^{\epsilon, j} \diamond b_{2}^{\epsilon, i} (t) \rvert^{2}] \lesssim \epsilon^{\frac{\eta}{2}} t^{\frac{\eta-\epsilon}{4}} 2^{q(1+ 3\eta)}.  
\end{equation} 
Similar calculations show that there exist $\eta, \epsilon, \gamma > 0$ sufficiently small so that 
\begin{align}
& \mathbb{E} [ \lvert \Delta_{q} ( \bar{b}_{1}^{\epsilon, j} \diamond \bar{b}_{2}^{\epsilon, i} (t_{1}) - \bar{b}_{1}^{\epsilon, j} \diamond \bar{b}_{2}^{\epsilon, i} (t_{2}) \nonumber\\
& \hspace{10mm}  - b_{1}^{\epsilon, j} \diamond b_{2}^{\epsilon, i} (t_{1}) + b_{1}^{\epsilon, j} \diamond b_{2}^{\epsilon, i} (t_{2})) \rvert^{2}] \lesssim \epsilon^{\gamma} \lvert t_{2} - t_{1} \rvert^{\eta} 2^{q(1+ \epsilon)}. 
\end{align} 
Therefore, by relying on Gaussian hypercontractivity theorem and Besov embedding Lemma \ref{Lemma 3.4}, similarly to \eqref{estimate 36}, we deduce that for all $\delta > 0$, all $p > 1$, $b_{2}^{\epsilon, i} \diamond b_{1}^{\epsilon, j} - \bar{b}_{2}^{\epsilon, i} \diamond \bar{b}_{1}^{\epsilon, j} \to 0$ in $L^{p}(\Omega; C([0,T]; \mathcal{C}^{-\frac{1}{2} - \frac{\delta}{2}}))$ as $\epsilon \searrow 0$. 

\subsubsection{Convergence of \eqref{fifth convergence}}\label{Subsubsection 2.2.4}

W.l.o.g. we work on $\pi_{0,\diamond} (u_{3}^{\epsilon, i_{0}}, b_{1}^{\epsilon, j_{0}}) - \pi_{0,\diamond}(\bar{u}_{3}^{\epsilon, i_{0}}, \bar{b}_{1}^{\epsilon, j_{0}})$. By \eqref{[Equation (3.2b)][ZZ17]} and \eqref{[Equation (3.22af)][ZZ17]}, we have 
\begin{align}\label{estimate 194}
&\pi_{0,\diamond} (u_{3}^{\epsilon, i_{0}}, b_{1}^{\epsilon, j_{0}}) - \pi_{0,\diamond} (\bar{u}_{3}^{\epsilon, i_{0}}, \bar{b}_{1}^{\epsilon, j_{0}}) \nonumber\\
=& \pi_{0} (u_{3}^{\epsilon, i_{0}}, b_{1}^{\epsilon, j_{0}}) - \pi_{0}(\bar{u}_{3}^{\epsilon, i_{0}}, \bar{b}_{1}^{\epsilon, j_{0}}) - \phi_{13}^{\epsilon, i_{0}j_{0}}(t) + \bar{\phi}_{13}^{\epsilon, i_{0}j_{0}}(t) - C_{13}^{\epsilon, i_{0}j_{0}} + \bar{C}_{13}^{\epsilon, i_{0}j_{0}} \nonumber\\
&+ \sum_{i_{1}=1}^{3} (C_{3,u}^{\epsilon, i_{0}i_{1}j_{0}} + \tilde{C}_{3,u}^{\epsilon, i_{0}i_{1}j_{0}})u_{2}^{\epsilon, i_{1}}+ (C_{3,b}^{\epsilon, i_{0}i_{1}j_{0}}+ \tilde{C}_{3,b}^{\epsilon, i_{0}i_{1}j_{0}}) b_{2}^{\epsilon, i_{1}}. 
\end{align}
Considering \eqref{[Equation (3.1ac)][ZZ17]} and \eqref{[Equation (3.21e)][ZZ17]}, for simplicity we define $\bar{u}_{31}, \bar{u}_{32}, u_{31}, u_{32}$ to satisfy for $i_{0} \in \{1,2,3\}$, $\bar{u}_{31}^{\epsilon} (0, \cdot) \equiv  \bar{u}_{32}^{\epsilon} (0, \cdot) \equiv u_{31}^{\epsilon} (0, \cdot) \equiv u_{32}^{\epsilon} (0, \cdot) \equiv 0$ and 
\begin{subequations}
\begin{align}
& d \bar{u}_{31}^{\epsilon, i_{0}} = \Delta \bar{u}_{31}^{\epsilon, i_{0}} dt - \frac{1}{2} \sum_{i_{1}, j=1}^{3} \mathcal{P}^{i_{0} i_{1}} D_{j} ( \bar{u}_{1}^{\epsilon, i_{1}} \diamond \bar{u}_{2}^{\epsilon, j} - \bar{b}_{1}^{\epsilon, i_{1}} \diamond \bar{b}_{2}^{\epsilon, j}) dt, \label{baru31}\\
& d \bar{u}_{32}^{\epsilon, i_{0}} = \Delta \bar{u}_{32}^{\epsilon, i_{0}} dt - \frac{1}{2} \sum_{i_{1}, j=1}^{3} \mathcal{P}^{i_{0} i_{1}} D_{j} ( \bar{u}_{2}^{\epsilon, i_{1}} \diamond \bar{u}_{1}^{\epsilon, j} - \bar{b}_{2}^{\epsilon, i_{1}} \diamond \bar{b}_{1}^{\epsilon, j}) dt, \label{baru32}\\
& d u_{31}^{\epsilon, i_{0}} = \Delta_{\epsilon} u_{31}^{\epsilon, i_{0}} dt - \frac{1}{2} \sum_{i_{1}, j=1}^{3} \mathcal{P}^{i_{0} i_{1}} D_{j}^{\epsilon} (u_{1}^{\epsilon, i_{1}} \diamond u_{2}^{\epsilon, j} - b_{1}^{\epsilon, i_{1}} \diamond b_{2}^{\epsilon, j} ) dt, \label{u31}\\
& d u_{32}^{\epsilon, i_{0}} = \Delta_{\epsilon} u_{32}^{\epsilon, i_{0}} dt - \frac{1}{2} \sum_{i_{1}, j=1}^{3} \mathcal{P}^{i_{0} i_{1}} D_{j}^{\epsilon} (u_{2}^{\epsilon, i_{1}} \diamond u_{1}^{\epsilon, j} - b_{2}^{\epsilon, i_{1}} \diamond b_{1}^{\epsilon, j} ) dt. \label{u32}
\end{align} 
\end{subequations} 
W.l.o.g. we show the estimates of $\pi_{0,\diamond} (u_{32}^{\epsilon, i_{0}}, b_{1}^{\epsilon, j_{0}}) - \pi_{0,\diamond} (\bar{u}_{32}^{\epsilon, i_{0}}, \bar{b}_{1}^{\epsilon, j_{0}})$; the estimate of $\pi_{0,\diamond} (u_{31}^{\epsilon, i_{0}}, b_{1}^{\epsilon, j_{0}}) - \pi_{0,\diamond} (\bar{u}_{31}^{\epsilon, i_{0}}, \bar{b}_{1}^{\epsilon, j_{0}})$ follows similarly. From \eqref{u32},  \eqref{baru32}, \eqref{[Equation (3.21d)][ZZ17]}, \eqref{[Equation (3.22aa)][ZZ17]}, \eqref{[Equation (3.22ab)][ZZ17]}, \eqref{[Equation (3.2aa)][ZZ17]}, \eqref{[Equation (3.2ab)][ZZ17]} and \eqref{[Equation (3.1ab)][ZZ17]} to deduce
\begin{align}\label{estimate 183}
& \pi_{0} (u_{32}^{\epsilon, i_{0}}, b_{1}^{\epsilon, j_{0}})(t) - \pi_{0}(\bar{u}_{32}^{\epsilon, i_{0}}, \bar{b}_{1}^{\epsilon, j_{0}})(t) \nonumber\\
=&  \frac{1}{4} [ \pi_{0} ( \sum_{i_{1}, i_{3}, i_{4}, j_{1} =1}^{3} \int_{0}^{t}  e^{-(s-t) \Delta_{\epsilon}} \mathcal{P}^{i_{0} i_{1}} D_{j_{1}}^{\epsilon} (u_{1}^{\epsilon, j_{1}}(s) \mathcal{P}^{i_{1} i_{4}} \nonumber\\
& \hspace{25mm} \times \int_{0}^{s} e^{-(r-s) \Delta_{\epsilon}} D_{i_{3}}^{\epsilon} (u_{1}^{\epsilon, i_{4}} u_{1}^{\epsilon, i_{3}} - b_{1}^{\epsilon, i_{4}} b_{1}^{\epsilon, i_{3}}) dr ) ds \nonumber\\
&\hspace{5mm} - \sum_{i_{1}, i_{3}, i_{4}, j_{1} =1}^{3} \int_{0}^{t} e^{-(s-t) \Delta_{\epsilon}} \mathcal{P}^{i_{0}i_{1}} D_{j_{1}}^{\epsilon} (b_{1}^{\epsilon, j_{1}}(s) \mathcal{P}^{i_{1} i_{4}} \nonumber\\
& \hspace{25mm} \times  \int_{0}^{s} e^{-(r-s) \Delta_{\epsilon}} D_{i_{3}}^{\epsilon} (b_{1}^{\epsilon, i_{4}} u_{1}^{\epsilon, i_{3}} - u_{1}^{\epsilon, i_{4}} b_{1}^{\epsilon, i_{3}}) dr ) ds, b_{1}^{\epsilon, j_{0}}(t)) \nonumber \\
& - \bar{L}_{t, i_{0}j_{0}}^{5} - \bar{L}_{t, i_{0}j_{0}}^{6} \nonumber\\
& - \pi_{0} (\sum_{i_{1}, i_{3}, i_{4}, j_{1} =1}^{3} \int_{0}^{t} e^{-(s-t) \Delta} \mathcal{P}^{i_{0}i_{1}} D_{j_{1}} (\bar{u}_{1}^{\epsilon, j_{1}}(s) \mathcal{P}^{i_{1}i_{4}} \nonumber\\
& \hspace{25mm} \times \int_{0}^{s} e^{-(r-s)\Delta} D_{i_{3}} ( \bar{u}_{1}^{\epsilon, i_{4}} \bar{u}_{1}^{\epsilon, i_{3}} - \bar{b}_{1}^{\epsilon, i_{4}} \bar{b}_{1}^{\epsilon, i_{3}}) dr) ds \nonumber\\
& \hspace{5mm} - \sum_{i_{1}, i_{3}, i_{4}, j_{1} =1}^{3} \int_{0}^{t} e^{-(s-t) \Delta} \mathcal{P}^{i_{0} i_{1}} D_{j_{1}} (\bar{b}_{1}^{\epsilon, j_{1}}(s) \mathcal{P}^{i_{1}i_{4}} \nonumber\\
& \hspace{25mm} \times \int_{0}^{s} e^{-(r-s) \Delta} D_{i_{3}} (\bar{b}_{1}^{\epsilon, i_{4}} \bar{u}_{1}^{\epsilon, i_{3}} - \bar{u}_{1}^{\epsilon, i_{4}} \bar{b}_{1}^{\epsilon, i_{3}}) dr) ds, \bar{b}_{1}^{\epsilon, j_{0}}(t))]  
\end{align} 
where 
\begin{subequations}
\begin{align}
\bar{L}_{t, i_{0}j_{0}}^{5} \triangleq&  2\pi_{0} (\sum_{i_{1}, i_{2}, j_{1} =1}^{3} \int_{0}^{t} e^{-(s-t) \Delta_{\epsilon}} \mathcal{P}^{i_{0} i_{1}} D_{j_{1}}^{\epsilon}(C_{1,u}^{\epsilon, i_{1} i_{2} j_{1}} u_{1}^{\epsilon, i_{2}} + C_{1,b}^{\epsilon, i_{1}i_{2} j_{1}} b_{1}^{\epsilon, i_{2}}) ds \nonumber \\ 
& \hspace{1mm}  - \int_{0}^{t} e^{-(s-t) \Delta_{\epsilon} } \mathcal{P}^{i_{0} i_{1}} D_{j_{1}}^{\epsilon} (C_{2,u}^{\epsilon, i_{1} i_{2} j_{1}} u_{1}^{\epsilon, i_{2}} + C_{2,b}^{\epsilon, i_{1} i_{2} j_{1}}b_{1}^{\epsilon, i_{2}}) ds, b_{1}^{\epsilon, j_{0}}(t)), \label{[Equation (4.6k)][ZZ17]}\\
\bar{L}_{t, i_{0}j_{0}}^{6} \triangleq&  2\pi_{0} (\sum_{i_{1}, i_{2}, j_{1} =1}^{3} \int_{0}^{t} e^{-(s-t) \Delta_{\epsilon}} \mathcal{P}^{i_{0} i_{1}} D_{j_{1}}^{\epsilon}(\tilde{C}_{1,u}^{\epsilon, i_{1} i_{2} j_{1}} u_{1}^{\epsilon, i_{2}} +  \tilde{C}_{1,b}^{\epsilon, i_{1} i_{2} j_{1}} b_{1}^{\epsilon, i_{2}}) ds\nonumber\\ 
& \hspace{1mm}  - \int_{0}^{t} e^{-(s-t) \Delta_{\epsilon} } \mathcal{P}^{i_{0} i_{1}} D_{j_{1}}^{\epsilon} (\tilde{C}_{2,u}^{\epsilon, i_{1} i_{2} j_{1}} u_{1}^{\epsilon, i_{2}} + \tilde{C}_{2,b}^{\epsilon, i_{1} i_{2} j_{1}} b_{1}^{\epsilon, i_{2}}) ds, b_{1}^{\epsilon, j_{0}}(t)). \label{[Equation (4.6m)][ZZ17]}
\end{align} 
\end{subequations} 
Within the right hand side of \eqref{estimate 183}, besides $\bar{L}_{t, i_{0}j_{0}}^{5} - \bar{L}_{t, i_{0}j_{0}}^{6}$, we see that 
\begin{align}\label{estimate 184}
&\pi_{0} ( \sum_{i_{1}, i_{2}, i_{3}, j_{1} =1}^{3} \int_{0}^{t}  e^{-(s-t) \Delta_{\epsilon}} \mathcal{P}^{i_{0} i_{1}} D_{j_{1}}^{\epsilon} (u_{1}^{\epsilon, j_{1}}(s) \mathcal{P}^{i_{1} i_{2}} \\
& \hspace{25mm} \times \int_{0}^{s} e^{-(\sigma-s) \Delta_{\epsilon}} D_{i_{3}}^{\epsilon} (u_{1}^{\epsilon, i_{2}} u_{1}^{\epsilon, i_{3}} - b_{1}^{\epsilon, i_{2}} b_{1}^{\epsilon, i_{3}}) d\sigma ) ds, b_{1}^{\epsilon, j_{0}}(t)) \nonumber\\
&- \pi_{0}(\sum_{i_{1}, i_{2}, i_{3}, j_{1} =1}^{3} \int_{0}^{t} e^{-(s-t) \Delta_{\epsilon}} \mathcal{P}^{i_{0}i_{1}} D_{j_{1}}^{\epsilon} (b_{1}^{\epsilon, j_{1}}(s) \mathcal{P}^{i_{1} i_{2}} \nonumber \\
& \hspace{25mm} \times \int_{0}^{s} e^{-(\sigma-s) \Delta_{\epsilon}} D_{i_{3}}^{\epsilon} (b_{1}^{\epsilon, i_{2}} u_{1}^{\epsilon, i_{3}} - u_{1}^{\epsilon, i_{2}} b_{1}^{\epsilon, i_{3}}) d\sigma ) ds, b_{1}^{\epsilon, j_{0}}(t)) \nonumber \\
& - \pi_{0} (\sum_{i_{1}, i_{2}, i_{3}, j_{1} =1}^{3} \int_{0}^{t} e^{-(s-t) \Delta} \mathcal{P}^{i_{0}i_{1}} D_{j_{1}} (\bar{u}_{1}^{\epsilon, j_{1}}(s) \mathcal{P}^{i_{1}i_{2}} \nonumber\\
& \hspace{25mm} \times \int_{0}^{s} e^{-(\sigma-s)\Delta} D_{i_{3}} ( \bar{u}_{1}^{\epsilon, i_{2}} \bar{u}_{1}^{\epsilon, i_{3}} - \bar{b}_{1}^{\epsilon, i_{2}} \bar{b}_{1}^{\epsilon, i_{3}}) d\sigma) ds, \bar{b}_{1}^{\epsilon, j_{0}}(t)) \nonumber\\
& + \pi_{0}(\sum_{i_{1}, i_{2}, i_{3}, j_{1} =1}^{3} \int_{0}^{t} e^{-(s-t) \Delta} \mathcal{P}^{i_{0} i_{1}} D_{j_{1}} (\bar{b}_{1}^{\epsilon, j_{1}}(s) \mathcal{P}^{i_{1}i_{2}} \nonumber\\
& \hspace{25mm} \times \int_{0}^{s} e^{-(\sigma-s) \Delta} D_{i_{3}} (\bar{b}_{1}^{\epsilon, i_{2}} \bar{u}_{1}^{\epsilon, i_{3}} - \bar{u}_{1}^{\epsilon, i_{2}} \bar{b}_{1}^{\epsilon, i_{3}}) d\sigma) ds, \bar{b}_{1}^{\epsilon, j_{0}}(t))  \nonumber \\
=&  (2\pi)^{-\frac{9}{2}} \sum_{\lvert i-j\rvert \leq 1} \sum_{i_{1}, i_{2}, i_{3}, j_{1} =1}^{3} \sum_{k} \sum_{k_{1}, k_{2}, k_{3}, k_{4} \neq 0: k_{1234} = k} \theta(2^{-i} k_{123}) \theta(2^{-j} k_{4})  \nonumber\\ 
& \times [( \int_{0}^{t} e^{- \lvert k_{123} \rvert^{2}  f( \epsilon k_{123})(t-s)} k_{123}^{j_{1}} g(\epsilon k_{123}^{j_{1}}) \hat{X}_{s,u}^{\epsilon, j_{1}}(k_{3}) \nonumber\\
& \hspace{1mm} \times \int_{0}^{s} e^{- \lvert k_{12} \rvert^{2}  f(\epsilon k_{12})(s-\sigma)} k_{12}^{i_{3}} g(\epsilon k_{12}^{i_{3}}) \hat{X}_{\sigma, u}^{\epsilon, i_{2}}(k_{1}) \hat{X}_{\sigma, u}^{\epsilon, i_{3}}(k_{2}) d \sigma ds \hat{X}_{t,b}^{\epsilon, j_{0}}(k_{4})  \nonumber \\
& - \int_{0}^{t} e^{- \lvert k_{123} \rvert^{2} (t-s)} k_{123}^{j_{1}} i \hat{\bar{X}}_{s,u}^{\epsilon, j_{1}}(k_{3}) \nonumber \\
& \hspace{1mm} \times \int_{0}^{s} e^{- \lvert k_{12} \rvert^{2} (s-\sigma)} k_{12}^{i_{3}} i \hat{ \bar{X}}_{\sigma, u}^{\epsilon, i_{2}}(k_{1}) \hat{\bar{X}}_{\sigma, u}^{\epsilon, i_{3}}(k_{2}) d\sigma ds \hat{\bar{X}}_{t,b}^{\epsilon, j_{0}}(k_{4})) \hat{\mathcal{P}}^{i_{0}i_{1}} (k_{123}) \hat{\mathcal{P}}^{i_{1}i_{2}}(k_{12}) e_{k}  \nonumber\\   
& - ( \int_{0}^{t} e^{- \lvert k_{123} \rvert^{2}  f( \epsilon k_{123})(t-s)} k_{123}^{j_{1}} g(\epsilon k_{123}^{j_{1}}) \hat{X}_{s,u}^{\epsilon, j_{1}}(k_{3}) \nonumber\\
& \hspace{1mm} \times \int_{0}^{s} e^{- \lvert k_{12} \rvert^{2}  f(\epsilon k_{12})(s-\sigma)} k_{12}^{i_{3}} g(\epsilon k_{12}^{i_{3}}) \hat{X}_{\sigma, b}^{\epsilon, i_{2}}(k_{1}) \hat{X}_{\sigma, b}^{\epsilon, i_{3}}(k_{2}) d \sigma ds \hat{X}_{t,b}^{\epsilon, j_{0}}(k_{4})   \nonumber\\
& + \int_{0}^{t} e^{- \lvert k_{123} \rvert^{2} (t-s)} k_{123}^{j_{1}} i \hat{\bar{X}}_{s,u}^{\epsilon, j_{1}}(k_{3})  \nonumber\\
& \hspace{1mm} \times \int_{0}^{s} e^{- \lvert k_{12} \rvert^{2} (s-\sigma)} k_{12}^{i_{3}} i \hat{ \bar{X}}_{\sigma, b}^{\epsilon, i_{2}}(k_{1}) \hat{\bar{X}}_{\sigma, b}^{\epsilon, i_{3}}(k_{2}) d\sigma ds \hat{\bar{X}}_{t,b}^{\epsilon, j_{0}}(k_{4})) \hat{\mathcal{P}}^{i_{0}i_{1}} (k_{123}) \hat{\mathcal{P}}^{i_{1}i_{2}}(k_{12}) e_{k}  \nonumber\\
& - ( \int_{0}^{t} e^{- \lvert k_{123} \rvert^{2} f( \epsilon k_{123}) (t-s)} k_{123}^{j_{1}} g(\epsilon k_{123}^{j_{1}}) \hat{X}_{s,b}^{\epsilon, j_{1}}(k_{3}) \nonumber\\
& \hspace{1mm} \times \int_{0}^{s} e^{- \lvert k_{12} \rvert^{2}  f(\epsilon k_{12})(s-\sigma)} k_{12}^{i_{3}} g(\epsilon k_{12}^{i_{3}}) \hat{X}_{\sigma, b}^{\epsilon, i_{2}}(k_{1}) \hat{X}_{\sigma, u}^{\epsilon, i_{3}}(k_{2}) d \sigma ds \hat{X}_{t,b}^{\epsilon, j_{0}}(k_{4})   \nonumber\\
& + \int_{0}^{t} e^{- \lvert k_{123} \rvert^{2} (t-s)} k_{123}^{j_{1}} i \hat{\bar{X}}_{s,b}^{\epsilon, j_{1}}(k_{3})  \nonumber\\
& \hspace{1mm} \times \int_{0}^{s} e^{- \lvert k_{12} \rvert^{2} (s-\sigma)} k_{12}^{i_{3}} i \hat{ \bar{X}}_{\sigma, b}^{\epsilon, i_{2}}(k_{1}) \hat{\bar{X}}_{\sigma, u}^{\epsilon, i_{3}}(k_{2}) d\sigma ds \hat{\bar{X}}_{t,b}^{\epsilon, j_{0}}(k_{4})) \hat{\mathcal{P}}^{i_{0}i_{1}} (k_{123}) \hat{\mathcal{P}}^{i_{1}i_{2}}(k_{12}) e_{k}  \nonumber\\   
&+ ( \int_{0}^{t} e^{- \lvert k_{123} \rvert^{2}  f( \epsilon k_{123})(t-s)} k_{123}^{j_{1}} g(\epsilon k_{123}^{j_{1}}) \hat{X}_{s,b}^{\epsilon, j_{1}}(k_{3}) \nonumber\\
& \hspace{1mm} \times \int_{0}^{s} e^{- \lvert k_{12} \rvert^{2}  f(\epsilon k_{12})(s-\sigma)} k_{12}^{i_{3}} g(\epsilon k_{12}^{i_{3}}) \hat{X}_{\sigma, u}^{\epsilon, i_{2}}(k_{1}) \hat{X}_{\sigma, b}^{\epsilon, i_{3}}(k_{2}) d \sigma ds \hat{X}_{t,b}^{\epsilon, j_{0}}(k_{4})   \nonumber\\
& - \int_{0}^{t} e^{- \lvert k_{123} \rvert^{2} (t-s)} k_{123}^{j_{1}} i \hat{\bar{X}}_{s,b}^{\epsilon, j_{1}}(k_{3})  \nonumber\\
& \hspace{1mm} \times \int_{0}^{s} e^{- \lvert k_{12} \rvert^{2} (s-\sigma)} k_{12}^{i_{3}} i \hat{ \bar{X}}_{\sigma, u}^{\epsilon, i_{2}}(k_{1}) \hat{\bar{X}}_{\sigma, b}^{\epsilon, i_{3}}(k_{2}) d\sigma ds \hat{\bar{X}}_{t,b}^{\epsilon, j_{0}}(k_{4})) \hat{\mathcal{P}}^{i_{0}i_{1}} (k_{123}) \hat{\mathcal{P}}^{i_{1}i_{2}}(k_{12}) e_{k}].   \nonumber
\end{align}
Now we can write first within \eqref{estimate 184}  by Example \ref{Example 3.1} 
\begin{align}\label{estimate 185}
\hat{X}_{\sigma, u}^{\epsilon, i_{2}}(k_{1}) \hat{X}_{\sigma, u}^{\epsilon, i_{3}}(k_{2}) \hat{X}_{s,u}^{\epsilon, j_{1}}(k_{3}) \hat{X}_{t,b}^{\epsilon, j_{0}}(k_{4}) = \sum_{i=1}^{10} l_{k_{1}k_{2}k_{3}k_{4}, i_{2}i_{3}j_{0}j_{1}, s\sigma t}^{1i}
\end{align}
where due to \eqref{covariance}, 
\begin{subequations}\label{estimate 186}
\begin{align}
l_{k_{1}k_{2}k_{3}k_{4}, i_{2}i_{3}j_{0}j_{1}, s\sigma t}^{11} \triangleq& : \hat{X}_{\sigma, u}^{\epsilon, i_{2}}(k_{1}) \hat{X}_{\sigma, u}^{\epsilon, i_{3}}(k_{2}) \hat{X}_{s,u}^{\epsilon, j_{1}}(k_{3}) \hat{X}_{t,b}^{\epsilon, j_{0}}(k_{4}):, \\
l_{k_{1}k_{2}k_{3}k_{4}, i_{2}i_{3}j_{0}j_{1}, s\sigma t}^{12} \triangleq& 1_{k_{12} = 0} \sum_{i_{4} =1}^{3} \frac{ h_{u}(\epsilon k_{1})^{2}}{2 \lvert k_{1} \rvert^{2} f(\epsilon k_{1})} \nonumber\\
& \times  \hat{\mathcal{P}}^{i_{2} i_{4}}(k_{1}) \hat{\mathcal{P}}^{i_{3}i_{4}} (k_{1}): \hat{X}_{s,u}^{\epsilon, j_{1}}(k_{3}) \hat{X}_{t,b}^{\epsilon, j_{0}}(k_{4}):,\\
l_{k_{1}k_{2}k_{3}k_{4}, i_{2}i_{3}j_{0}j_{1}, s\sigma t}^{13} \triangleq& 1_{k_{12} = 0, k_{34} = 0} \sum_{i_{4}, i_{5} =1}^{3} \frac{ h_{u}(\epsilon k_{1})^{2}}{2\lvert k_{1} \rvert^{2} f(\epsilon k_{1})} \hat{\mathcal{P}}^{i_{2} i_{4} }(k_{1}) \hat{\mathcal{P}}^{i_{3}i_{4}}(k_{1}) \nonumber\\
& \times  \frac{ e^{ - \lvert k_{3} \rvert^{2} f(\epsilon k_{3}) (t-s)} h_{u}(\epsilon k_{3}) h_{b}(\epsilon k_{3})}{2\lvert k_{3} \rvert^{2} f(\epsilon k_{3})} \hat{\mathcal{P}}^{j_{1}i_{5}} (k_{3}) \hat{\mathcal{P}}^{j_{0} i_{5}}(k_{3}), \\
l_{k_{1}k_{2}k_{3}k_{4}, i_{2}i_{3}j_{0}j_{1}, s\sigma t}^{14} \triangleq& 1_{k_{13} = 0} \sum_{i_{4} =1}^{3} \frac{ e^{- \lvert k_{1}\rvert^{2} f(\epsilon k_{1}) (s-\sigma)} h_{u}(\epsilon k_{1})^{2}}{2 \lvert k_{1} \rvert^{2} f(\epsilon k_{1})} \nonumber\\
& \times \hat{\mathcal{P}}^{i_{2}i_{4}}(k_{1}) \hat{\mathcal{P}}^{j_{1}i_{4}}(k_{1}) : \hat{X}_{\sigma, u}^{\epsilon, i_{3}}(k_{2}) \hat{X}_{t,b}^{\epsilon, j_{0}}(k_{4}):, \\
l_{k_{1}k_{2}k_{3}k_{4}, i_{2}i_{3}j_{0}j_{1}, s\sigma t}^{15} \triangleq& 1_{k_{13} = 0, k_{24} = 0} \sum_{i_{4}, i_{5} = 1}^{3} \frac{ e^{- \lvert k_{1}  \rvert^{2} f(\epsilon k_{1}) (s-\sigma)} h_{u}(\epsilon k_{1})^{2}}{2 \lvert k_{1} \rvert^{2} f(\epsilon k_{1})} \nonumber\\
&  \times \hat{\mathcal{P}}^{i_{2}i_{4}}(k_{1}) \hat{\mathcal{P}}^{j_{1}i_{4}}(k_{1})  \frac{ e^{- \lvert k_{2} \rvert^{2} f(\epsilon k_{2}) (t-\sigma)} h_{u}(\epsilon k_{2}) h_{b}(\epsilon k_{2})}{2 \lvert k_{2} \rvert^{2} f(\epsilon k_{2})} \nonumber\\
& \hspace{5mm} \times  \hat{\mathcal{P}}^{i_{3}i_{5}}(k_{2}) \hat{\mathcal{P}}^{j_{0} i_{5}}(k_{2}), \\
l_{k_{1}k_{2}k_{3}k_{4}, i_{2}i_{3}j_{0}j_{1}, s\sigma t}^{16} \triangleq& 1_{k_{14} = 0} \sum_{i_{4} =1}^{3} \frac{ e^{- \lvert k_{1} \rvert^{2} f(\epsilon k_{1}) (t-\sigma)} h_{u}(\epsilon k_{1})h_{b}(\epsilon k_{1})}{2 \lvert k_{1} \rvert^{2} f(\epsilon k_{1})} \nonumber\\
& \times  \hat{\mathcal{P}}^{i_{2}i_{4}}(k_{1}) \hat{\mathcal{P}}^{j_{0}i_{4}}(k_{1}) : \hat{X}_{\sigma, u}^{\epsilon, i_{3}}(k_{2}) \hat{X}_{s,u}^{\epsilon, j_{1}}(k_{3}):, \\
l_{k_{1}k_{2}k_{3}k_{4}, i_{2}i_{3}j_{0}j_{1}, s\sigma t}^{17} \triangleq& 1_{k_{14} = 0, k_{23} = 0} \sum_{i_{4}, i_{5} = 1}^{3} \frac{ e^{- \lvert k_{1} \rvert^{2} f(\epsilon k_{1}) (t-\sigma)} h_{u}(\epsilon k_{1})h_{b}(\epsilon k_{1})}{2 \lvert k_{1} \rvert^{2} f(\epsilon k_{1})}  \nonumber\\
& \times \hat{\mathcal{P}}^{i_{2}i_{4}}(k_{1}) \hat{\mathcal{P}}^{j_{0}i_{4}}(k_{1}) \frac{ e^{- \lvert k_{2} \rvert^{2} f(\epsilon k_{2}) (s-\sigma)} h_{u}(\epsilon k_{2})^{2} }{2 \lvert k_{2} \rvert^{2} f(\epsilon k_{2})} \nonumber\\
& \hspace{5mm}   \times \hat{\mathcal{P}}^{i_{3}i_{5}}(k_{2}) \hat{\mathcal{P}}^{j_{1} i_{5}}(k_{2}), \\
l_{k_{1}k_{2}k_{3}k_{4}, i_{2}i_{3}j_{0}j_{1}, s\sigma t}^{18} \triangleq& 1_{k_{23} = 0} \sum_{i_{4} =1}^{3} \frac{e^{- \lvert k_{2} \rvert^{2} f(\epsilon k_{2}) (s-\sigma)} h_{u}(\epsilon k_{2})^{2}}{2 \lvert k_{2} \rvert^{2} f(\epsilon k_{2})} \nonumber\\
& \times  \hat{\mathcal{P}}^{i_{3} i_{4}}(k_{2}) \hat{\mathcal{P}}^{j_{1}i_{4}} (k_{2}): \hat{X}_{\sigma,u}^{\epsilon, i_{2}}(k_{1}) \hat{X}_{t,b}^{\epsilon, j_{0}}(k_{4}):, \\
l_{k_{1}k_{2}k_{3}k_{4}, i_{2}i_{3}j_{0}j_{1}, s\sigma t}^{19} \triangleq& 1_{k_{24} = 0} \sum_{i_{4} =1}^{3} \frac{e^{- \lvert k_{2} \rvert^{2} f(\epsilon k_{2}) (t-\sigma)} h_{u}(\epsilon k_{2})h_{b}(\epsilon k_{2})}{2 \lvert k_{2} \rvert^{2} f(\epsilon k_{2})} \nonumber\\
& \times \hat{\mathcal{P}}^{i_{3} i_{4}}(k_{2}) \hat{\mathcal{P}}^{j_{0}i_{4}} (k_{2}): \hat{X}_{\sigma,u}^{\epsilon, i_{2}}(k_{1}) \hat{X}_{s,u}^{\epsilon, j_{1}}(k_{3}):, \\
l_{k_{1}k_{2}k_{3}k_{4}, i_{2}i_{3}j_{0}j_{1}, s\sigma t}^{110} \triangleq& 1_{k_{34} = 0} \sum_{i_{4} =1}^{3} \frac{e^{- \lvert k_{3} \rvert^{2} f(\epsilon k_{3}) (t- s)} h_{u}(\epsilon k_{3})h_{b}(\epsilon k_{3})}{2 \lvert k_{3} \rvert^{2} f(\epsilon k_{3})} \nonumber\\
& \times  \hat{\mathcal{P}}^{j_{1} i_{4}}(k_{3}) \hat{\mathcal{P}}^{j_{0}i_{4}} (k_{3}): \hat{X}_{\sigma,u}^{\epsilon, i_{2}}(k_{1}) \hat{X}_{\sigma,u}^{\epsilon, i_{3}}(k_{2}):,     
\end{align}
\end{subequations} 
where the terms corresponding to $l_{k_{1}k_{2}k_{3}k_{4}, i_{2}i_{3}j_{0}j_{1}, s\sigma t}^{1i}$ for $i\in\{2, 3\}$ vanish due to $k_{12}$ in \eqref{estimate 184}. Similarly, second within \eqref{estimate 184}, we see that
\begin{equation}
\hat{\bar{X}}_{\sigma, u}^{\epsilon, i_{2}}(k_{1}) \hat{\bar{X}}_{\sigma, u}^{\epsilon, i_{3}}(k_{2}) \hat{\bar{X}}_{s,u}^{\epsilon, j_{1}}(k_{3}) \hat{\bar{X}}_{t,b}^{\epsilon, j_{0}}(k_{4}) = \sum_{i=1}^{10}  l_{k_{1}k_{2}k_{3}k_{4}, i_{2}i_{3}j_{0}j_{1}, s\sigma t}^{2i}
\end{equation} 
where 
\begin{subequations}
\begin{align}
l_{k_{1}k_{2}k_{3}k_{4}, i_{2}i_{3}j_{0}j_{1}, s\sigma t}^{21} \triangleq& : \hat{\bar{X}}_{\sigma, u}^{\epsilon, i_{2}}(k_{1}) \hat{\bar{X}}_{\sigma, u}^{\epsilon, i_{3}}(k_{2}) \hat{\bar{X}}_{s,u}^{\epsilon, j_{1}}(k_{3}) \hat{\bar{X}}_{t,b}^{\epsilon, j_{0}}(k_{4}):, \\
l_{k_{1}k_{2}k_{3}k_{4}, i_{2}i_{3}j_{0}j_{1}, s\sigma t}^{22} \triangleq& 1_{k_{12} = 0} \sum_{i_{4} =1}^{3} \frac{ h_{u}(\epsilon k_{1})^{2}}{2 \lvert k_{1} \rvert^{2}} \nonumber\\
& \hspace{5mm} \times  \hat{\mathcal{P}}^{i_{2} i_{4}}(k_{1}) \hat{\mathcal{P}}^{i_{3}i_{4}} (k_{1}): \hat{\bar{X}}_{s,u}^{\epsilon, j_{1}}(k_{3}) \hat{\bar{X}}_{t,b}^{\epsilon, j_{0}}(k_{4}):,\\
l_{k_{1}k_{2}k_{3}k_{4}, i_{2}i_{3}j_{0}j_{1}, s\sigma t}^{23} \triangleq& 1_{k_{12} = 0, k_{34} = 0} \sum_{i_{4}, i_{5} =1}^{3} \frac{ h_{u}(\epsilon k_{1})^{2}}{2\lvert k_{1} \rvert^{2}} \hat{\mathcal{P}}^{i_{2} i_{4} }(k_{1}) \hat{\mathcal{P}}^{i_{3}i_{4}}(k_{1}) \nonumber\\
& \hspace{5mm} \times  \frac{ e^{ - \lvert k_{3} \rvert^{2}  (t-s)} h_{u}(\epsilon k_{3}) h_{b}(\epsilon k_{3})}{2\lvert k_{3} \rvert^{2}} \hat{\mathcal{P}}^{j_{1}i_{5}} (k_{3}) \hat{\mathcal{P}}^{j_{0} i_{5}}(k_{3}), \\
l_{k_{1}k_{2}k_{3}k_{4}, i_{2}i_{3}j_{0}j_{1}, s\sigma t}^{24} \triangleq& 1_{k_{13} = 0} \sum_{i_{4} =1}^{3} \frac{ e^{- \lvert k_{1}\rvert^{2}  (s-\sigma)} h_{u}(\epsilon k_{1})^{2}}{2 \lvert k_{1} \rvert^{2}} \nonumber\\
& \hspace{5mm} \times  \hat{\mathcal{P}}^{i_{2}i_{4}}(k_{1}) \hat{\mathcal{P}}^{j_{1}i_{4}}(k_{1}) : \hat{\bar{X}}_{\sigma, u}^{\epsilon, i_{3}}(k_{2}) \hat{\bar{X}}_{t,b}^{\epsilon, j_{0}}(k_{4}):, \\
l_{k_{1}k_{2}k_{3}k_{4}, i_{2}i_{3}j_{0}j_{1}, s\sigma t}^{25} \triangleq& 1_{k_{13} = 0, k_{24} = 0} \sum_{i_{4}, i_{5} = 1}^{3} \frac{ e^{- \lvert k_{1}  \rvert^{2} (s-\sigma)} h_{u}(\epsilon k_{1})^{2}}{2 \lvert k_{1} \rvert^{2}} \hat{\mathcal{P}}^{i_{2}i_{4}}(k_{1}) \hat{\mathcal{P}}^{j_{1}i_{4}}(k_{1}) \nonumber\\
& \hspace{5mm} \times \frac{ e^{- \lvert k_{2} \rvert^{2}  (t-\sigma)} h_{u}(\epsilon k_{2}) h_{b}(\epsilon k_{2})}{2 \lvert k_{2} \rvert^{2}} \hat{\mathcal{P}}^{i_{3}i_{5}}(k_{2}) \hat{\mathcal{P}}^{j_{0} i_{5}}(k_{2}), \\
l_{k_{1}k_{2}k_{3}k_{4}, i_{2}i_{3}j_{0}j_{1}, s\sigma t}^{26} \triangleq& 1_{k_{14} = 0} \sum_{i_{4} =1}^{3} \frac{ e^{- \lvert k_{1} \rvert^{2} (t-\sigma)} h_{u}(\epsilon k_{1})h_{b}(\epsilon k_{1})}{2 \lvert k_{1} \rvert^{2} } \nonumber\\
& \times \hspace{5mm} \hat{\mathcal{P}}^{i_{2}i_{4}}(k_{1}) \hat{\mathcal{P}}^{j_{0}i_{4}}(k_{1}) : \hat{\bar{X}}_{\sigma, u}^{\epsilon, i_{3}}(k_{2}) \hat{\bar{X}}_{s,u}^{\epsilon, j_{1}}(k_{3}):, \\
l_{k_{1}k_{2}k_{3}k_{4}, i_{2}i_{3}j_{0}j_{1}, s\sigma t}^{27} \triangleq& 1_{k_{14} = 0, k_{23} = 0} \sum_{i_{4}, i_{5} = 1}^{3} \frac{ e^{- \lvert k_{1} \rvert^{2} (t-\sigma)} h_{u}(\epsilon k_{1})h_{b}(\epsilon k_{1})}{2 \lvert k_{1} \rvert^{2}}  \\
&  \times  \hat{\mathcal{P}}^{i_{2}i_{4}}(k_{1}) \hat{\mathcal{P}}^{j_{0}i_{4}}(k_{1})\frac{ e^{- \lvert k_{2} \rvert^{2} (s-\sigma)} h_{u}(\epsilon k_{2})^{2} }{2 \lvert k_{2} \rvert^{2}} \hat{\mathcal{P}}^{i_{3}i_{5}}(k_{2}) \hat{\mathcal{P}}^{j_{1} i_{5}}(k_{2}), \nonumber\\
l_{k_{1}k_{2}k_{3}k_{4}, i_{2}i_{3}j_{0}j_{1}, s\sigma t}^{28} \triangleq& 1_{k_{23} = 0} \sum_{i_{4} =1}^{3} \frac{e^{- \lvert k_{2} \rvert^{2} (s-\sigma)} h_{u}(\epsilon k_{2})^{2}}{2 \lvert k_{2} \rvert^{2}} \nonumber\\
& \hspace{5mm} \times \hat{\mathcal{P}}^{i_{3} i_{4}}(k_{2}) \hat{\mathcal{P}}^{j_{1}i_{4}} (k_{2}): \hat{\bar{X}}_{\sigma,u}^{\epsilon, i_{2}}(k_{1}) \hat{\bar{X}}_{t,b}^{\epsilon, j_{0}}(k_{4}):, \\
l_{k_{1}k_{2}k_{3}k_{4}, i_{2}i_{3}j_{0}j_{1}, s\sigma t}^{29} \triangleq& 1_{k_{24} = 0} \sum_{i_{4} =1}^{3} \frac{e^{- \lvert k_{2} \rvert^{2} (t-\sigma)} h_{u}(\epsilon k_{2})h_{b}(\epsilon k_{2})}{2 \lvert k_{2} \rvert^{2}} \nonumber\\
& \hspace{5mm} \times  \hat{\mathcal{P}}^{i_{3} i_{4}}(k_{2}) \hat{\mathcal{P}}^{j_{0}i_{4}} (k_{2}): \hat{\bar{X}}_{\sigma,u}^{\epsilon, i_{2}}(k_{1}) \hat{\bar{X}}_{s,u}^{\epsilon, j_{1}}(k_{3}):, \\
l_{k_{1}k_{2}k_{3}k_{4}, i_{2}i_{3}j_{0}j_{1}, s\sigma t}^{210} \triangleq& 1_{k_{34} = 0} \sum_{i_{4} =1}^{3} \frac{e^{- \lvert k_{3} \rvert^{2} (t- s)} h_{u}(\epsilon k_{3})h_{b}(\epsilon k_{3})}{2 \lvert k_{3} \rvert^{2}} \nonumber\\
& \hspace{5mm} \times \hat{\mathcal{P}}^{j_{1} i_{4}}(k_{3}) \hat{\mathcal{P}}^{j_{0}i_{4}} (k_{3}): \hat{\bar{X}}_{\sigma,u}^{\epsilon, i_{2}}(k_{1}) \hat{\bar{X}}_{\sigma,u}^{\epsilon, i_{3}}(k_{2}):     
\end{align}
\end{subequations}
where terms corresponding to $l_{k_{1}k_{2}k_{3}k_{4}, i_{2}i_{3}j_{0}j_{1}, s\sigma t}^{2i}$ for $i\in\{2,3\}$ vanish due to $k_{12}$ in \eqref{estimate 184}. Third, within \eqref{estimate 184} 
\begin{align}
\hat{X}_{\sigma, b}^{\epsilon, i_{2}}(k_{1}) \hat{X}_{\sigma, b}^{\epsilon, i_{3}}(k_{2}) \hat{X}_{s,u}^{\epsilon, j_{1}}(k_{3}) \hat{X}_{t,b}^{\epsilon, j_{0}}(k_{4}) = \sum_{i=1}^{10} l_{k_{1}k_{2}k_{3}k_{4}, i_{2}i_{3}j_{0}j_{1}, s\sigma t}^{3i}
\end{align}
where 
\begin{subequations}
\begin{align}
l_{k_{1}k_{2}k_{3}k_{4}, i_{2}i_{3}j_{0}j_{1}, s\sigma t}^{31} &\triangleq : \hat{X}_{\sigma, b}^{\epsilon, i_{2}}(k_{1}) \hat{X}_{\sigma, b}^{\epsilon, i_{3}}(k_{2}) \hat{X}_{s,u}^{\epsilon, j_{1}}(k_{3}) \hat{X}_{t,b}^{\epsilon, j_{0}}(k_{4}):, \\
l_{k_{1}k_{2}k_{3}k_{4}, i_{2}i_{3}j_{0}j_{1}, s\sigma t}^{32} &\triangleq 1_{k_{12} = 0} \sum_{i_{4} =1}^{3} \frac{ h_{b}(\epsilon k_{1})^{2}}{2 \lvert k_{1} \rvert^{2} f(\epsilon k_{1})} \nonumber\\
& \hspace{5mm} \times \hat{\mathcal{P}}^{i_{2} i_{4}}(k_{1}) \hat{\mathcal{P}}^{i_{3}i_{4}} (k_{1}): \hat{X}_{s,u}^{\epsilon, j_{1}}(k_{3}) \hat{X}_{t,b}^{\epsilon, j_{0}}(k_{4}):,\\
l_{k_{1}k_{2}k_{3}k_{4}, i_{2}i_{3}j_{0}j_{1}, s\sigma t}^{33} &\triangleq 1_{k_{12} = 0, k_{34} = 0} \sum_{i_{4}, i_{5} =1}^{3} \frac{ h_{b}(\epsilon k_{1})^{2}}{2\lvert k_{1} \rvert^{2} f(\epsilon k_{1})} \hat{\mathcal{P}}^{i_{2} i_{4} }(k_{1}) \hat{\mathcal{P}}^{i_{3}i_{4}}(k_{1}) \nonumber\\
&  \times  \frac{ e^{ - \lvert k_{3} \rvert^{2} f(\epsilon k_{3}) (t-s)} h_{u}(\epsilon k_{3}) h_{b}(\epsilon k_{3})}{2\lvert k_{3} \rvert^{2} f(\epsilon k_{3})} \hat{\mathcal{P}}^{j_{1}i_{5}} (k_{3}) \hat{\mathcal{P}}^{j_{0} i_{5}}(k_{3}), \\ 
l_{k_{1}k_{2}k_{3}k_{4}, i_{2}i_{3}j_{0}j_{1}, s\sigma t}^{34} &\triangleq 1_{k_{13} = 0} \sum_{i_{4} =1}^{3} \frac{ e^{- \lvert k_{1}\rvert^{2} f(\epsilon k_{1}) (s-\sigma)} h_{b}(\epsilon k_{1}) h_{u}(\epsilon k_{1})}{2 \lvert k_{1} \rvert^{2} f(\epsilon k_{1})} \nonumber\\
& \hspace{5mm} \times \hat{\mathcal{P}}^{i_{2}i_{4}}(k_{1}) \hat{\mathcal{P}}^{j_{1}i_{4}}(k_{1}) : \hat{X}_{\sigma, b}^{\epsilon, i_{3}}(k_{2}) \hat{X}_{t,b}^{\epsilon, j_{0}}(k_{4}):, \\
l_{k_{1}k_{2}k_{3}k_{4}, i_{2}i_{3}j_{0}j_{1}, s\sigma t}^{35} &\triangleq 1_{k_{13} = 0, k_{24} = 0} \sum_{i_{4}, i_{5} = 1}^{3} \frac{ e^{- \lvert k_{1}  \rvert^{2} f(\epsilon k_{1}) (s-\sigma)} h_{b}(\epsilon k_{1}) h_{u}(\epsilon k_{1})}{2 \lvert k_{1} \rvert^{2} f(\epsilon k_{1})} \nonumber\\
& \times  \hat{\mathcal{P}}^{i_{2}i_{4}}(k_{1}) \hat{\mathcal{P}}^{j_{1}i_{4}}(k_{1})  \frac{ e^{- \lvert k_{2} \rvert^{2} f(\epsilon k_{2}) (t-\sigma)} h_{b}(\epsilon k_{2})^{2}}{2 \lvert k_{2} \rvert^{2} f(\epsilon k_{2})} \nonumber\\
& \hspace{5mm} \times \hat{\mathcal{P}}^{i_{3}i_{5}}(k_{2}) \hat{\mathcal{P}}^{j_{0} i_{5}}(k_{2}), \\
l_{k_{1}k_{2}k_{3}k_{4}, i_{2}i_{3}j_{0}j_{1}, s\sigma t}^{36} &\triangleq 1_{k_{14} = 0} \sum_{i_{4} =1}^{3} \frac{ e^{- \lvert k_{1} \rvert^{2} f(\epsilon k_{1}) (t-\sigma)} h_{b}(\epsilon k_{1})^{2}}{2 \lvert k_{1} \rvert^{2} f(\epsilon k_{1})} \nonumber\\
& \times  \hat{\mathcal{P}}^{i_{2}i_{4}}(k_{1}) \hat{\mathcal{P}}^{j_{0}i_{4}}(k_{1}) : \hat{X}_{\sigma, b}^{\epsilon, i_{3}}(k_{2}) \hat{X}_{s,u}^{\epsilon, j_{1}}(k_{3}):, \\
l_{k_{1}k_{2}k_{3}k_{4}, i_{2}i_{3}j_{0}j_{1}, s\sigma t}^{37} &\triangleq 1_{k_{14} = 0, k_{23} = 0} \sum_{i_{4}, i_{5} = 1}^{3} \frac{ e^{- \lvert k_{1} \rvert^{2} f(\epsilon k_{1}) (t-\sigma)} h_{b}(\epsilon k_{1})^{2}}{2 \lvert k_{1} \rvert^{2} f(\epsilon k_{1})} \nonumber\\
& \hspace{5mm} \times  \hat{\mathcal{P}}^{i_{2}i_{4}}(k_{1}) \hat{\mathcal{P}}^{j_{0}i_{4}}(k_{1})  \frac{ e^{- \lvert k_{2} \rvert^{2} f(\epsilon k_{2}) (s-\sigma)} h_{b}(\epsilon k_{2}) h_{u}(\epsilon k_{2}) }{2 \lvert k_{2} \rvert^{2} f(\epsilon k_{2})}  \nonumber\\
& \hspace{10mm} \times \hat{\mathcal{P}}^{i_{3}i_{5}}(k_{2}) \hat{\mathcal{P}}^{j_{1} i_{5}}(k_{2}), \\
l_{k_{1}k_{2}k_{3}k_{4}, i_{2}i_{3}j_{0}j_{1}, s\sigma t}^{38} &\triangleq 1_{k_{23} = 0} \sum_{i_{4} =1}^{3} \frac{e^{- \lvert k_{2} \rvert^{2} f(\epsilon k_{2}) (s-\sigma)} h_{b}(\epsilon k_{2}) h_{u}(\epsilon k_{2})}{2 \lvert k_{2} \rvert^{2} f(\epsilon k_{2})} \nonumber\\
& \times \hat{\mathcal{P}}^{i_{3} i_{4}}(k_{2}) \hat{\mathcal{P}}^{j_{1}i_{4}} (k_{2}): \hat{X}_{\sigma,b}^{\epsilon, i_{2}}(k_{1}) \hat{X}_{t,b}^{\epsilon, j_{0}}(k_{4}):, \\
l_{k_{1}k_{2}k_{3}k_{4}, i_{2}i_{3}j_{0}j_{1}, s\sigma t}^{39} &\triangleq 1_{k_{24} = 0} \sum_{i_{4} =1}^{3} \frac{e^{- \lvert k_{2} \rvert^{2} f(\epsilon k_{2}) (t-\sigma)} h_{b}(\epsilon k_{2})^{2}}{2 \lvert k_{2} \rvert^{2} f(\epsilon k_{2})} \nonumber\\
& \times  \hat{\mathcal{P}}^{i_{3} i_{4}}(k_{2}) \hat{\mathcal{P}}^{j_{0}i_{4}} (k_{2}): \hat{X}_{\sigma,b}^{\epsilon, i_{2}}(k_{1}) \hat{X}_{s,u}^{\epsilon, j_{1}}(k_{3}):, \\
l_{k_{1}k_{2}k_{3}k_{4}, i_{2}i_{3}j_{0}j_{1}, s\sigma t}^{310} &\triangleq 1_{k_{34} = 0} \sum_{i_{4} =1}^{3} \frac{e^{- \lvert k_{3} \rvert^{2} f(\epsilon k_{3}) (t- s)} h_{u}(\epsilon k_{3})h_{b}(\epsilon k_{3})}{2 \lvert k_{3} \rvert^{2} f(\epsilon k_{3})} \nonumber\\
& \times \hat{\mathcal{P}}^{j_{1} i_{4}}(k_{3}) \hat{\mathcal{P}}^{j_{0}i_{4}} (k_{3}): \hat{X}_{\sigma,b}^{\epsilon, i_{2}}(k_{1}) \hat{X}_{\sigma,b}^{\epsilon, i_{3}}(k_{2}):     
\end{align}
\end{subequations} 
where the terms corresponding to $l_{k_{1}k_{2}k_{3}k_{4}, i_{2}i_{3}j_{0}j_{1}, s\sigma t}^{3i}$ for $i \in \{2,3\}$ vanish due to $k_{12}$ in \eqref{estimate 184}. Similarly, fourth, within \eqref{estimate 184},  we see that 
\begin{equation}
\hat{\bar{X}}_{\sigma, b}^{\epsilon, i_{2}}(k_{1}) \hat{\bar{X}}_{\sigma, b}^{\epsilon, i_{3}}(k_{2}) \hat{\bar{X}}_{s,u}^{\epsilon, j_{1}}(k_{3}) \hat{\bar{X}}_{t,b}^{\epsilon, j_{0}}(k_{4}) = \sum_{i=1}^{10} l_{k_{1}k_{2}k_{3}k_{4}, i_{2}i_{3}j_{0}j_{1}, s\sigma t}^{4i}
\end{equation}
where 
\begin{subequations} 
\begin{align}
l_{k_{1}k_{2}k_{3}k_{4}, i_{2}i_{3}j_{0}j_{1}, s\sigma t}^{41} \triangleq& : \hat{\bar{X}}_{\sigma, b}^{\epsilon, i_{2}}(k_{1}) \hat{\bar{X}}_{\sigma, b}^{\epsilon, i_{3}}(k_{2}) \hat{\bar{X}}_{s,u}^{\epsilon, j_{1}}(k_{3}) \hat{\bar{X}}_{t,b}^{\epsilon, j_{0}}(k_{4}):, \\
l_{k_{1}k_{2}k_{3}k_{4}, i_{2}i_{3}j_{0}j_{1}, s\sigma t}^{42} \triangleq& 1_{k_{12} = 0} \sum_{i_{4} =1}^{3} \frac{ h_{b}(\epsilon k_{1})^{2}}{2 \lvert k_{1} \rvert^{2}} \nonumber\\
& \hspace{5mm} \times \hat{\mathcal{P}}^{i_{2} i_{4}}(k_{1}) \hat{\mathcal{P}}^{i_{3}i_{4}} (k_{1}): \hat{\bar{X}}_{s,u}^{\epsilon, j_{1}}(k_{3}) \hat{\bar{X}}_{t,b}^{\epsilon, j_{0}}(k_{4}):,\\
l_{k_{1}k_{2}k_{3}k_{4}, i_{2}i_{3}j_{0}j_{1}, s\sigma t}^{43} \triangleq& 1_{k_{12} = 0, k_{34} = 0} \sum_{i_{4}, i_{5} =1}^{3} \frac{ h_{b}(\epsilon k_{1})^{2}}{2\lvert k_{1} \rvert^{2}} \hat{\mathcal{P}}^{i_{2} i_{4} }(k_{1}) \hat{\mathcal{P}}^{i_{3}i_{4}}(k_{1}) \nonumber\\
& \hspace{5mm} \times  \frac{ e^{ - \lvert k_{3} \rvert^{2} (t-s)} h_{u}(\epsilon k_{3}) h_{b}(\epsilon k_{3})}{2\lvert k_{3} \rvert^{2}} \hat{\mathcal{P}}^{j_{1}i_{5}} (k_{3}) \hat{\mathcal{P}}^{j_{0} i_{5}}(k_{3}), \\ 
l_{k_{1}k_{2}k_{3}k_{4}, i_{2}i_{3}j_{0}j_{1}, s\sigma t}^{44} \triangleq& 1_{k_{13} = 0} \sum_{i_{4} =1}^{3} \frac{ e^{- \lvert k_{1}\rvert^{2}  (s-\sigma)} h_{b}(\epsilon k_{1}) h_{u}(\epsilon k_{1})}{2 \lvert k_{1} \rvert^{2} } \nonumber\\
& \hspace{5mm} \times \hat{\mathcal{P}}^{i_{2}i_{4}}(k_{1}) \hat{\mathcal{P}}^{j_{1}i_{4}}(k_{1}) : \hat{\bar{X}}_{\sigma, b}^{\epsilon, i_{3}}(k_{2}) \hat{\bar{X}}_{t,b}^{\epsilon, j_{0}}(k_{4}):, \\
l_{k_{1}k_{2}k_{3}k_{4}, i_{2}i_{3}j_{0}j_{1}, s\sigma t}^{45} \triangleq& 1_{k_{13} = 0, k_{24} = 0} \sum_{i_{4}, i_{5} = 1}^{3} \frac{ e^{- \lvert k_{1}  \rvert^{2} (s-\sigma)} h_{b}(\epsilon k_{1}) h_{u}(\epsilon k_{1})}{2 \lvert k_{1} \rvert^{2}} \hat{\mathcal{P}}^{i_{2}i_{4}}(k_{1}) \nonumber\\
& \times  \hat{\mathcal{P}}^{j_{1}i_{4}}(k_{1})  \frac{ e^{- \lvert k_{2} \rvert^{2} (t-\sigma)} h_{b}(\epsilon k_{2})^{2}}{2 \lvert k_{2} \rvert^{2}} \hat{\mathcal{P}}^{i_{3}i_{5}}(k_{2}) \hat{\mathcal{P}}^{j_{0} i_{5}}(k_{2}), \\
l_{k_{1}k_{2}k_{3}k_{4}, i_{2}i_{3}j_{0}j_{1}, s\sigma t}^{46} \triangleq& 1_{k_{14} = 0} \sum_{i_{4} =1}^{3} \frac{ e^{- \lvert k_{1} \rvert^{2} (t-\sigma)} h_{b}(\epsilon k_{1})^{2}}{2 \lvert k_{1} \rvert^{2}} \nonumber\\
& \hspace{5mm} \times \hat{\mathcal{P}}^{i_{2}i_{4}}(k_{1}) \hat{\mathcal{P}}^{j_{0}i_{4}}(k_{1}) : \hat{\bar{X}}_{\sigma, b}^{\epsilon, i_{3}}(k_{2}) \hat{\bar{X}}_{s,u}^{\epsilon, j_{1}}(k_{3}):, \\
l_{k_{1}k_{2}k_{3}k_{4}, i_{2}i_{3}j_{0}j_{1}, s\sigma t}^{47} \triangleq& 1_{k_{14} = 0, k_{23} = 0} \sum_{i_{4}, i_{5} = 1}^{3} \frac{ e^{- \lvert k_{1} \rvert^{2} (t-\sigma)} h_{b}(\epsilon k_{1})^{2}}{2 \lvert k_{1} \rvert^{2}} \hat{\mathcal{P}}^{i_{2}i_{4}}(k_{1}) \hat{\mathcal{P}}^{j_{0}i_{4}}(k_{1}) \nonumber\\
& \hspace{5mm} \times \frac{ e^{- \lvert k_{2} \rvert^{2} (s-\sigma)} h_{b}(\epsilon k_{2}) h_{u}(\epsilon k_{2}) }{2 \lvert k_{2} \rvert^{2}} \hat{\mathcal{P}}^{i_{3}i_{5}}(k_{2}) \hat{\mathcal{P}}^{j_{1} i_{5}}(k_{2}), \\
l_{k_{1}k_{2}k_{3}k_{4}, i_{2}i_{3}j_{0}j_{1}, s\sigma t}^{48} \triangleq& 1_{k_{23} = 0} \sum_{i_{4} =1}^{3} \frac{e^{- \lvert k_{2} \rvert^{2} (s-\sigma)} h_{b}(\epsilon k_{2}) h_{u}(\epsilon k_{2})}{2 \lvert k_{2} \rvert^{2}} \nonumber\\
& \hspace{5mm} \times \hat{\mathcal{P}}^{i_{3} i_{4}}(k_{2}) \hat{\mathcal{P}}^{j_{1}i_{4}} (k_{2}): \hat{\bar{X}}_{\sigma,b}^{\epsilon, i_{2}}(k_{1}) \hat{\bar{X}}_{t,b}^{\epsilon, j_{0}}(k_{4}):, \\
l_{k_{1}k_{2}k_{3}k_{4}, i_{2}i_{3}j_{0}j_{1}, s\sigma t}^{49} \triangleq& 1_{k_{24} = 0} \sum_{i_{4} =1}^{3} \frac{e^{- \lvert k_{2} \rvert^{2} (t-\sigma)} h_{b}(\epsilon k_{2})^{2}}{2 \lvert k_{2} \rvert^{2}} \nonumber\\
& \hspace{5mm} \times \hat{\mathcal{P}}^{i_{3} i_{4}}(k_{2}) \hat{\mathcal{P}}^{j_{0}i_{4}} (k_{2}): \hat{\bar{X}}_{\sigma,b}^{\epsilon, i_{2}}(k_{1}) \hat{\bar{X}}_{s,u}^{\epsilon, j_{1}}(k_{3}):, \\
l_{k_{1}k_{2}k_{3}k_{4}, i_{2}i_{3}j_{0}j_{1}, s\sigma t}^{410} \triangleq& 1_{k_{34} = 0} \sum_{i_{4} =1}^{3} \frac{e^{- \lvert k_{3} \rvert^{2} (t- s)} h_{u}(\epsilon k_{3})h_{b}(\epsilon k_{3})}{2 \lvert k_{3} \rvert^{2}}  \nonumber\\
& \times \hat{\mathcal{P}}^{j_{1} i_{4}}(k_{3}) \hat{\mathcal{P}}^{j_{0}i_{4}} (k_{3}): \hat{\bar{X}}_{\sigma,b}^{\epsilon, i_{2}}(k_{1}) \hat{\bar{X}}_{\sigma,b}^{\epsilon, i_{3}}(k_{2}):     
\end{align}
\end{subequations} 
where the terms corresponding to $l_{k_{1}k_{2}k_{3}k_{4}, i_{2}i_{3}j_{0}j_{1}, s\sigma t}^{4i}$ for $i\in \{2,3\}$ vanish due to $k_{12}$ in \eqref{estimate 184}. Fifth, within \eqref{estimate 184} 
\begin{align}
\hat{X}_{\sigma, b}^{\epsilon, i_{2}}(k_{1}) \hat{X}_{\sigma, u}^{\epsilon, i_{3}}(k_{2}) \hat{X}_{s,b}^{\epsilon, j_{1}}(k_{3}) \hat{X}_{t,b}^{\epsilon, j_{0}}(k_{4}) = \sum_{i=1}^{10} l_{k_{1}k_{2}k_{3}k_{4}, i_{2}i_{3}j_{0}j_{1}, s\sigma t}^{5i}
\end{align}
where  
\begin{subequations}
\begin{align}
l_{k_{1}k_{2}k_{3}k_{4}, i_{2}i_{3}j_{0}j_{1}, s\sigma t}^{51} &\triangleq : \hat{X}_{\sigma, b}^{\epsilon, i_{2}}(k_{1}) \hat{X}_{\sigma, u}^{\epsilon, i_{3}}(k_{2}) \hat{X}_{s,b}^{\epsilon, j_{1}}(k_{3}) \hat{X}_{t,b}^{\epsilon, j_{0}}(k_{4}):, \\
l_{k_{1}k_{2}k_{3}k_{4}, i_{2}i_{3}j_{0}j_{1}, s\sigma t}^{52} &\triangleq 1_{k_{12} = 0} \sum_{i_{4} =1}^{3} \frac{ h_{b}(\epsilon k_{1})h_{u}(\epsilon k_{1})}{2 \lvert k_{1} \rvert^{2} f(\epsilon k_{1})} \nonumber\\
& \hspace{5mm} \times \hat{\mathcal{P}}^{i_{2} i_{4}}(k_{1}) \hat{\mathcal{P}}^{i_{3}i_{4}} (k_{1}): \hat{X}_{s,b}^{\epsilon, j_{1}}(k_{3}) \hat{X}_{t,b}^{\epsilon, j_{0}}(k_{4}):,\\
l_{k_{1}k_{2}k_{3}k_{4}, i_{2}i_{3}j_{0}j_{1}, s\sigma t}^{53} &\triangleq 1_{k_{12} = 0, k_{34} = 0} \sum_{i_{4}, i_{5} =1}^{3} \frac{ h_{b}(\epsilon k_{1})h_{u}(\epsilon k_{1})}{2\lvert k_{1} \rvert^{2} f(\epsilon k_{1})} \hat{\mathcal{P}}^{i_{2} i_{4} }(k_{1}) \hat{\mathcal{P}}^{i_{3}i_{4}}(k_{1}) \nonumber\\
& \hspace{5mm} \times  \frac{ e^{ - \lvert k_{3} \rvert^{2} f(\epsilon k_{3}) (t-s)} h_{b}(\epsilon k_{3})^{2}}{2\lvert k_{3} \rvert^{2} f(\epsilon k_{3})} \hat{\mathcal{P}}^{j_{1}i_{5}} (k_{3}) \hat{\mathcal{P}}^{j_{0} i_{5}}(k_{3}), \\ 
l_{k_{1}k_{2}k_{3}k_{4}, i_{2}i_{3}j_{0}j_{1}, s\sigma t}^{54} &\triangleq 1_{k_{13} = 0} \sum_{i_{4} =1}^{3} \frac{ e^{- \lvert k_{1}\rvert^{2} f(\epsilon k_{1}) (s-\sigma)} h_{b}(\epsilon k_{1})^{2}}{2 \lvert k_{1} \rvert^{2} f(\epsilon k_{1})} \nonumber\\
& \hspace{5mm}  \hat{\mathcal{P}}^{i_{2}i_{4}}(k_{1}) \hat{\mathcal{P}}^{j_{1}i_{4}}(k_{1}) : \hat{X}_{\sigma, u}^{\epsilon, i_{3}}(k_{2}) \hat{X}_{t,b}^{\epsilon, j_{0}}(k_{4}):, \\
l_{k_{1}k_{2}k_{3}k_{4}, i_{2}i_{3}j_{0}j_{1}, s\sigma t}^{55} &\triangleq 1_{k_{13} = 0, k_{24} = 0} \sum_{i_{4}, i_{5} = 1}^{3} \frac{ e^{- \lvert k_{1}  \rvert^{2} f(\epsilon k_{1}) (s-\sigma)} h_{b}(\epsilon k_{1})^{2}}{2 \lvert k_{1} \rvert^{2} f(\epsilon k_{1})} \nonumber\\
& \hspace{5mm} \times  \hat{\mathcal{P}}^{i_{2}i_{4}}(k_{1})  \hat{\mathcal{P}}^{j_{1}i_{4}}(k_{1}) \frac{ e^{- \lvert k_{2} \rvert^{2} f(\epsilon k_{2}) (t-\sigma)} h_{u}(\epsilon k_{2}) h_{b}(\epsilon k_{2})}{2 \lvert k_{2} \rvert^{2} f(\epsilon k_{2})} \nonumber\\
& \hspace{10mm} \times  \hat{\mathcal{P}}^{i_{3}i_{5}}(k_{2}) \hat{\mathcal{P}}^{j_{0} i_{5}}(k_{2}),  \\
l_{k_{1}k_{2}k_{3}k_{4}, i_{2}i_{3}j_{0}j_{1}, s\sigma t}^{56} &\triangleq 1_{k_{14} = 0} \sum_{i_{4} =1}^{3} \frac{ e^{- \lvert k_{1} \rvert^{2} f(\epsilon k_{1}) (t-\sigma)} h_{b}(\epsilon k_{1})^{2}}{2 \lvert k_{1} \rvert^{2} f(\epsilon k_{1})} \nonumber\\
& \hspace{5mm}  \times \hat{\mathcal{P}}^{i_{2}i_{4}}(k_{1}) \hat{\mathcal{P}}^{j_{0}i_{4}}(k_{1}) : \hat{X}_{\sigma, u}^{\epsilon, i_{3}}(k_{2}) \hat{X}_{s,b}^{\epsilon, j_{1}}(k_{3}):, \\
l_{k_{1}k_{2}k_{3}k_{4}, i_{2}i_{3}j_{0}j_{1}, s\sigma t}^{57} &\triangleq 1_{k_{14} = 0, k_{23} = 0} \sum_{i_{4}, i_{5} = 1}^{3} \frac{ e^{- \lvert k_{1} \rvert^{2} f(\epsilon k_{1}) (t-\sigma)} h_{b}(\epsilon k_{1})^{2}}{2 \lvert k_{1} \rvert^{2} f(\epsilon k_{1})} \nonumber\\
& \times  \hat{\mathcal{P}}^{i_{2}i_{4}}(k_{1}) \hat{\mathcal{P}}^{j_{0}i_{4}}(k_{1})  \frac{ e^{- \lvert k_{2} \rvert^{2} f(\epsilon k_{2}) (s-\sigma)} h_{b}(\epsilon k_{2}) h_{u}(\epsilon k_{2}) }{2 \lvert k_{2} \rvert^{2} f(\epsilon k_{2})} \nonumber\\
&  \hspace{5mm} \times \hat{\mathcal{P}}^{i_{3}i_{5}}(k_{2}) \hat{\mathcal{P}}^{j_{1} i_{5}}(k_{2}), \\
l_{k_{1}k_{2}k_{3}k_{4}, i_{2}i_{3}j_{0}j_{1}, s\sigma t}^{58} &\triangleq 1_{k_{23} = 0} \sum_{i_{4} =1}^{3} \frac{e^{- \lvert k_{2} \rvert^{2} f(\epsilon k_{2}) (s-\sigma)} h_{b}(\epsilon k_{2}) h_{u}(\epsilon k_{2})}{2 \lvert k_{2} \rvert^{2} f(\epsilon k_{2})}  \nonumber\\
& \hspace{5mm} \times \hat{\mathcal{P}}^{i_{3} i_{4}}(k_{2}) \hat{\mathcal{P}}^{j_{1}i_{4}} (k_{2}): \hat{X}_{\sigma,b}^{\epsilon, i_{2}}(k_{1}) \hat{X}_{t,b}^{\epsilon, j_{0}}(k_{4}):, \\
l_{k_{1}k_{2}k_{3}k_{4}, i_{2}i_{3}j_{0}j_{1}, s\sigma t}^{59} &\triangleq 1_{k_{24} = 0} \sum_{i_{4} =1}^{3} \frac{e^{- \lvert k_{2} \rvert^{2} f(\epsilon k_{2}) (t-\sigma)}h_{u}(\epsilon k_{2}) h_{b}(\epsilon k_{2})}{2 \lvert k_{2} \rvert^{2} f(\epsilon k_{2})} \nonumber\\
& \hspace{5mm} \times \hat{\mathcal{P}}^{i_{3} i_{4}}(k_{2}) \hat{\mathcal{P}}^{j_{0}i_{4}} (k_{2}): \hat{X}_{\sigma,b}^{\epsilon, i_{2}}(k_{1}) \hat{X}_{s,b}^{\epsilon, j_{1}}(k_{3}):, \\
l_{k_{1}k_{2}k_{3}k_{4}, i_{2}i_{3}j_{0}j_{1}, s\sigma t}^{510} &\triangleq 1_{k_{34} = 0} \sum_{i_{4} =1}^{3} \frac{e^{- \lvert k_{3} \rvert^{2} f(\epsilon k_{3}) (t- s)} h_{b}(\epsilon k_{3})^{2}}{2 \lvert k_{3} \rvert^{2} f(\epsilon k_{3})} \nonumber\\
& \hspace{5mm} \times \hat{\mathcal{P}}^{j_{1} i_{4}}(k_{3}) \hat{\mathcal{P}}^{j_{0}i_{4}} (k_{3}): \hat{X}_{\sigma,b}^{\epsilon, i_{2}}(k_{1}) \hat{X}_{\sigma,u}^{\epsilon, i_{3}}(k_{2}):     
\end{align}
\end{subequations} 
where the terms corresponding to $l_{k_{1}k_{2}k_{3}k_{4}, i_{2}i_{3}j_{0}j_{1}, s\sigma t}^{5i}$ for $i\in \{2,3\}$ vanish due to $k_{12}$ within \eqref{estimate 184}. Similarly, sixth, within \eqref{estimate 184}, we see that 
\begin{equation}
 \hat{\bar{X}}_{\sigma, b}^{\epsilon, i_{2}}(k_{1}) \hat{\bar{X}}_{\sigma, u}^{\epsilon, i_{3}}(k_{2}) \hat{\bar{X}}_{s,b}^{\epsilon, j_{1}}(k_{3}) \hat{\bar{X}}_{t,b}^{\epsilon, j_{0}}(k_{4})  = \sum_{i=1}^{10}l_{k_{1}k_{2}k_{3}k_{4}, i_{2}i_{3}j_{0}j_{1}, s\sigma t}^{6i}
\end{equation}
where  
\begin{subequations}
\begin{align} 
l_{k_{1}k_{2}k_{3}k_{4}, i_{2}i_{3}j_{0}j_{1}, s\sigma t}^{61} \triangleq& : \hat{\bar{X}}_{\sigma, b}^{\epsilon, i_{2}}(k_{1}) \hat{\bar{X}}_{\sigma, u}^{\epsilon, i_{3}}(k_{2}) \hat{\bar{X}}_{s,b}^{\epsilon, j_{1}}(k_{3}) \hat{\bar{X}}_{t,b}^{\epsilon, j_{0}}(k_{4}):, \\
l_{k_{1}k_{2}k_{3}k_{4}, i_{2}i_{3}j_{0}j_{1}, s\sigma t}^{62} \triangleq& 1_{k_{12} = 0} \sum_{i_{4} =1}^{3} \frac{ h_{b}(\epsilon k_{1})h_{u}(\epsilon k_{1})}{2 \lvert k_{1} \rvert^{2}}  \nonumber\\
& \hspace{5mm} \times \hat{\mathcal{P}}^{i_{2} i_{4}}(k_{1}) \hat{\mathcal{P}}^{i_{3}i_{4}} (k_{1}): \hat{\bar{X}}_{s,b}^{\epsilon, j_{1}}(k_{3}) \hat{\bar{X}}_{t,b}^{\epsilon, j_{0}}(k_{4}):,\\
l_{k_{1}k_{2}k_{3}k_{4}, i_{2}i_{3}j_{0}j_{1}, s\sigma t}^{63} \triangleq& 1_{k_{12} = 0, k_{34} = 0} \sum_{i_{4}, i_{5} =1}^{3} \frac{ h_{b}(\epsilon k_{1})h_{u}(\epsilon k_{1})}{2\lvert k_{1} \rvert^{2}} \hat{\mathcal{P}}^{i_{2} i_{4} }(k_{1}) \hat{\mathcal{P}}^{i_{3}i_{4}}(k_{1}) \nonumber\\
& \hspace{5mm} \times  \frac{ e^{ - \lvert k_{3} \rvert^{2} (t-s)} h_{b}(\epsilon k_{3})^{2}}{2\lvert k_{3} \rvert^{2}} \hat{\mathcal{P}}^{j_{1}i_{5}} (k_{3}) \hat{\mathcal{P}}^{j_{0} i_{5}}(k_{3}), \\ 
l_{k_{1}k_{2}k_{3}k_{4}, i_{2}i_{3}j_{0}j_{1}, s\sigma t}^{64} \triangleq& 1_{k_{13} = 0} \sum_{i_{4} =1}^{3} \frac{ e^{- \lvert k_{1}\rvert^{2} (s-\sigma)} h_{b}(\epsilon k_{1})^{2}}{2 \lvert k_{1} \rvert^{2}} \nonumber \\
& \hspace{5mm} \times \hat{\mathcal{P}}^{i_{2}i_{4}}(k_{1}) \hat{\mathcal{P}}^{j_{1}i_{4}}(k_{1}) : \hat{\bar{X}}_{\sigma, u}^{\epsilon, i_{3}}(k_{2}) \hat{\bar{X}}_{t,b}^{\epsilon, j_{0}}(k_{4}):, \\
l_{k_{1}k_{2}k_{3}k_{4}, i_{2}i_{3}j_{0}j_{1}, s\sigma t}^{65} \triangleq& 1_{k_{13} = 0, k_{24} = 0} \sum_{i_{4}, i_{5} = 1}^{3} \frac{ e^{- \lvert k_{1}  \rvert^{2} (s-\sigma)} h_{b}(\epsilon k_{1})^{2}}{2 \lvert k_{1} \rvert^{2} } \hat{\mathcal{P}}^{i_{2}i_{4}}(k_{1}) \hat{\mathcal{P}}^{j_{1}i_{4}}(k_{1}) \nonumber\\
& \hspace{5mm} \times \frac{ e^{- \lvert k_{2} \rvert^{2} (t-\sigma)} h_{u}(\epsilon k_{2}) h_{b}(\epsilon k_{2})}{2 \lvert k_{2} \rvert^{2}} \hat{\mathcal{P}}^{i_{3}i_{5}}(k_{2}) \hat{\mathcal{P}}^{j_{0} i_{5}}(k_{2}), \\
l_{k_{1}k_{2}k_{3}k_{4}, i_{2}i_{3}j_{0}j_{1}, s\sigma t}^{66} \triangleq& 1_{k_{14} = 0} \sum_{i_{4} =1}^{3} \frac{ e^{- \lvert k_{1} \rvert^{2} (t-\sigma)} h_{b}(\epsilon k_{1})^{2}}{2 \lvert k_{1} \rvert^{2}} \nonumber\\
& \hspace{5mm} \times \hat{\mathcal{P}}^{i_{2}i_{4}}(k_{1}) \hat{\mathcal{P}}^{j_{0}i_{4}}(k_{1}) : \hat{\bar{X}}_{\sigma, u}^{\epsilon, i_{3}}(k_{2}) \hat{\bar{X}}_{s,b}^{\epsilon, j_{1}}(k_{3}):, \\
l_{k_{1}k_{2}k_{3}k_{4}, i_{2}i_{3}j_{0}j_{1}, s\sigma t}^{67} \triangleq& 1_{k_{14} = 0, k_{23} = 0} \sum_{i_{4}, i_{5} = 1}^{3} \frac{ e^{- \lvert k_{1} \rvert^{2} (t-\sigma)} h_{b}(\epsilon k_{1})^{2}}{2 \lvert k_{1} \rvert^{2}} \hat{\mathcal{P}}^{i_{2}i_{4}}(k_{1}) \hat{\mathcal{P}}^{j_{0}i_{4}}(k_{1}) \nonumber\\
& \hspace{5mm} \times \frac{ e^{- \lvert k_{2} \rvert^{2} (s-\sigma)} h_{b}(\epsilon k_{2}) h_{u}(\epsilon k_{2}) }{2 \lvert k_{2} \rvert^{2}} \hat{\mathcal{P}}^{i_{3}i_{5}}(k_{2}) \hat{\mathcal{P}}^{j_{1} i_{5}}(k_{2}), \\
l_{k_{1}k_{2}k_{3}k_{4}, i_{2}i_{3}j_{0}j_{1}, s\sigma t}^{68} \triangleq& 1_{k_{23} = 0} \sum_{i_{4} =1}^{3} \frac{e^{- \lvert k_{2} \rvert^{2} (s-\sigma)} h_{b}(\epsilon k_{2}) h_{u}(\epsilon k_{2})}{2 \lvert k_{2} \rvert^{2}} \nonumber\\
& \hspace{5mm} \times  \hat{\mathcal{P}}^{i_{3} i_{4}}(k_{2}) \hat{\mathcal{P}}^{j_{1}i_{4}} (k_{2}): \hat{\bar{X}}_{\sigma,b}^{\epsilon, i_{2}}(k_{1}) \hat{\bar{X}}_{t,b}^{\epsilon, j_{0}}(k_{4}):, \\
l_{k_{1}k_{2}k_{3}k_{4}, i_{2}i_{3}j_{0}j_{1}, s\sigma t}^{69} \triangleq& 1_{k_{24} = 0} \sum_{i_{4} =1}^{3} \frac{e^{- \lvert k_{2} \rvert^{2} (t-\sigma)}h_{u}(\epsilon k_{2}) h_{b}(\epsilon k_{2})}{2 \lvert k_{2} \rvert^{2}} \nonumber\\
& \hspace{5mm} \times \hat{\mathcal{P}}^{i_{3} i_{4}}(k_{2}) \hat{\mathcal{P}}^{j_{0}i_{4}} (k_{2}): \hat{\bar{X}}_{\sigma,b}^{\epsilon, i_{2}}(k_{1}) \hat{\bar{X}}_{s,b}^{\epsilon, j_{1}}(k_{3}):, \\
l_{k_{1}k_{2}k_{3}k_{4}, i_{2}i_{3}j_{0}j_{1}, s\sigma t}^{610} \triangleq& 1_{k_{34} = 0} \sum_{i_{4} =1}^{3} \frac{e^{- \lvert k_{3} \rvert^{2} (t- s)} h_{b}(\epsilon k_{3})^{2}}{2 \lvert k_{3} \rvert^{2}} \nonumber\\
& \hspace{5mm} \times \hat{\mathcal{P}}^{j_{1} i_{4}}(k_{3})  \hat{\mathcal{P}}^{j_{0}i_{4}} (k_{3}): \hat{\bar{X}}_{\sigma,b}^{\epsilon, i_{2}}(k_{1}) \hat{\bar{X}}_{\sigma,u}^{\epsilon, i_{3}}(k_{2}):     
\end{align}
\end{subequations} 
where the terms corresponding to $l_{k_{1}k_{2}k_{3}k_{4}, i_{2}i_{3}j_{0}j_{1}, s\sigma t}^{6i}$ for $i\in\{2,3\}$ vanish due to $k_{12}$ in \eqref{estimate 184}. Seventh, within \eqref{estimate 184} 
\begin{align}
\hat{X}_{\sigma, u}^{\epsilon, i_{2}}(k_{1}) \hat{X}_{\sigma, b}^{\epsilon, i_{3}}(k_{2}) \hat{X}_{s,b}^{\epsilon, j_{1}}(k_{3}) \hat{X}_{t,b}^{\epsilon, j_{0}}(k_{4}) = \sum_{i=1}^{10}l_{k_{1}k_{2}k_{3}k_{4}, i_{2}i_{3}j_{0}j_{1}, s\sigma t}^{7i}
\end{align}
where  
\begin{subequations}
\begin{align}
l_{k_{1}k_{2}k_{3}k_{4}, i_{2}i_{3}j_{0}j_{1}, s\sigma t}^{71} \triangleq& : \hat{X}_{\sigma, u}^{\epsilon, i_{2}}(k_{1}) \hat{X}_{\sigma, b}^{\epsilon, i_{3}}(k_{2}) \hat{X}_{s,b}^{\epsilon, j_{1}}(k_{3}) \hat{X}_{t,b}^{\epsilon, j_{0}}(k_{4}):, \\
l_{k_{1}k_{2}k_{3}k_{4}, i_{2}i_{3}j_{0}j_{1}, s\sigma t}^{72} \triangleq& 1_{k_{12} = 0} \sum_{i_{4} =1}^{3} \frac{ h_{b}(\epsilon k_{1})h_{u}(\epsilon k_{1})}{2 \lvert k_{1} \rvert^{2} f(\epsilon k_{1})} \nonumber \\
& \hspace{5mm} \times \hat{\mathcal{P}}^{i_{2} i_{4}}(k_{1}) \hat{\mathcal{P}}^{i_{3}i_{4}} (k_{1}): \hat{X}_{s,b}^{\epsilon, j_{1}}(k_{3}) \hat{X}_{t,b}^{\epsilon, j_{0}}(k_{4}):,\\
l_{k_{1}k_{2}k_{3}k_{4}, i_{2}i_{3}j_{0}j_{1}, s\sigma t}^{73} \triangleq& 1_{k_{12} = 0, k_{34} = 0} \sum_{i_{4}, i_{5} =1}^{3} \frac{ h_{b}(\epsilon k_{1})h_{u}(\epsilon k_{1})}{2\lvert k_{1} \rvert^{2} f(\epsilon k_{1})} \hat{\mathcal{P}}^{i_{2} i_{4} }(k_{1}) \hat{\mathcal{P}}^{i_{3}i_{4}}(k_{1}) \nonumber\\
& \hspace{5mm} \times  \frac{ e^{ - \lvert k_{3} \rvert^{2} f(\epsilon k_{3}) (t-s)} h_{b}(\epsilon k_{3})^{2}}{2\lvert k_{3} \rvert^{2} f(\epsilon k_{3})} \hat{\mathcal{P}}^{j_{1}i_{5}} (k_{3}) \hat{\mathcal{P}}^{j_{0} i_{5}}(k_{3}), \\ 
l_{k_{1}k_{2}k_{3}k_{4}, i_{2}i_{3}j_{0}j_{1}, s\sigma t}^{74} \triangleq& 1_{k_{13} = 0} \sum_{i_{4} =1}^{3} \frac{ e^{- \lvert k_{1}\rvert^{2} f(\epsilon k_{1}) (s-\sigma)} h_{b}(\epsilon k_{1})h_{u}(\epsilon k_{1})}{2 \lvert k_{1} \rvert^{2} f(\epsilon k_{1})} \nonumber\\
& \hspace{5mm} \times \hat{\mathcal{P}}^{i_{2}i_{4}}(k_{1}) \hat{\mathcal{P}}^{j_{1}i_{4}}(k_{1}) : \hat{X}_{\sigma, b}^{\epsilon, i_{3}}(k_{2}) \hat{X}_{t,b}^{\epsilon, j_{0}}(k_{4}):, \\
l_{k_{1}k_{2}k_{3}k_{4}, i_{2}i_{3}j_{0}j_{1}, s\sigma t}^{75} \triangleq& 1_{k_{13} = 0, k_{24} = 0} \sum_{i_{4}, i_{5} = 1}^{3} \frac{ e^{- \lvert k_{1}  \rvert^{2} f(\epsilon k_{1}) (s-\sigma)} h_{u} (\epsilon k_{1}) h_{b}(\epsilon k_{1})}{2 \lvert k_{1} \rvert^{2} f(\epsilon k_{1})} \nonumber\\
& \times \hat{\mathcal{P}}^{i_{2}i_{4}}(k_{1}) \hat{\mathcal{P}}^{j_{1}i_{4}}(k_{1}) \frac{ e^{- \lvert k_{2} \rvert^{2} f(\epsilon k_{2}) (t-\sigma)} h_{b}(\epsilon k_{2})^{2}}{2 \lvert k_{2} \rvert^{2} f(\epsilon k_{2})} \nonumber\\
& \hspace{5mm} \times  \hat{\mathcal{P}}^{i_{3}i_{5}}(k_{2}) \hat{\mathcal{P}}^{j_{0} i_{5}}(k_{2}), \\
l_{k_{1}k_{2}k_{3}k_{4}, i_{2}i_{3}j_{0}j_{1}, s\sigma t}^{76} \triangleq& 1_{k_{14} = 0} \sum_{i_{4} =1}^{3} \frac{ e^{- \lvert k_{1} \rvert^{2} f(\epsilon k_{1}) (t-\sigma)} h_{u}(\epsilon k_{1})  h_{b}(\epsilon k_{1})}{2 \lvert k_{1} \rvert^{2} f(\epsilon k_{1})} \nonumber\\
& \hspace{5mm} \times \hat{\mathcal{P}}^{i_{2}i_{4}}(k_{1}) \hat{\mathcal{P}}^{j_{0}i_{4}}(k_{1}) : \hat{X}_{\sigma, b}^{\epsilon, i_{3}}(k_{2}) \hat{X}_{s,b}^{\epsilon, j_{1}}(k_{3}):, \\
l_{k_{1}k_{2}k_{3}k_{4}, i_{2}i_{3}j_{0}j_{1}, s\sigma t}^{77} \triangleq& 1_{k_{14} = 0, k_{23} = 0} \sum_{i_{4}, i_{5} = 1}^{3} \frac{ e^{- \lvert k_{1} \rvert^{2} f(\epsilon k_{1}) (t-\sigma)} h_{u}(\epsilon k_{1}) h_{b}(\epsilon k_{1})}{2 \lvert k_{1} \rvert^{2} f(\epsilon k_{1})} \nonumber\\
& \times \hat{\mathcal{P}}^{i_{2}i_{4}}(k_{1}) \hat{\mathcal{P}}^{j_{0}i_{4}}(k_{1}) \frac{ e^{- \lvert k_{2} \rvert^{2} f(\epsilon k_{2}) (s-\sigma)} h_{b}(\epsilon k_{2})^{2} }{2 \lvert k_{2} \rvert^{2} f(\epsilon k_{2})} \nonumber\\
& \hspace{5mm} \times  \hat{\mathcal{P}}^{i_{3}i_{5}}(k_{2}) \hat{\mathcal{P}}^{j_{1} i_{5}}(k_{2}), \\
l_{k_{1}k_{2}k_{3}k_{4}, i_{2}i_{3}j_{0}j_{1}, s\sigma t}^{78} \triangleq& 1_{k_{23} = 0} \sum_{i_{4} =1}^{3} \frac{e^{- \lvert k_{2} \rvert^{2} f(\epsilon k_{2}) (s-\sigma)} h_{b}(\epsilon k_{2})^{2}}{2 \lvert k_{2} \rvert^{2} f(\epsilon k_{2})} \nonumber\\
& \hspace{5mm} \times  \hat{\mathcal{P}}^{i_{3} i_{4}}(k_{2}) \hat{\mathcal{P}}^{j_{1}i_{4}} (k_{2}): \hat{X}_{\sigma,u}^{\epsilon, i_{2}}(k_{1}) \hat{X}_{t,b}^{\epsilon, j_{0}}(k_{4}):, \\
l_{k_{1}k_{2}k_{3}k_{4}, i_{2}i_{3}j_{0}j_{1}, s\sigma t}^{79} \triangleq& 1_{k_{24} = 0} \sum_{i_{4} =1}^{3} \frac{e^{- \lvert k_{2} \rvert^{2} f(\epsilon k_{2}) (t-\sigma)}h_{b}(\epsilon k_{2})^{2}}{2 \lvert k_{2} \rvert^{2} f(\epsilon k_{2})} \nonumber\\
& \hspace{5mm} \times  \hat{\mathcal{P}}^{i_{3} i_{4}}(k_{2}) \hat{\mathcal{P}}^{j_{0}i_{4}} (k_{2}): \hat{X}_{\sigma,u}^{\epsilon, i_{2}}(k_{1}) \hat{X}_{s,b}^{\epsilon, j_{1}}(k_{3}):, \\
l_{k_{1}k_{2}k_{3}k_{4}, i_{2}i_{3}j_{0}j_{1}, s\sigma t}^{710} \triangleq& 1_{k_{34} = 0} \sum_{i_{4} =1}^{3} \frac{e^{- \lvert k_{3} \rvert^{2} f(\epsilon k_{3}) (t- s)} h_{b}(\epsilon k_{3})^{2}}{2 \lvert k_{3} \rvert^{2} f(\epsilon k_{3})} \nonumber\\
& \hspace{5mm} \times \hat{\mathcal{P}}^{j_{1} i_{4}}(k_{3}) \hat{\mathcal{P}}^{j_{0}i_{4}} (k_{3}): \hat{X}_{\sigma,u}^{\epsilon, i_{2}}(k_{1}) \hat{X}_{\sigma,b}^{\epsilon, i_{3}}(k_{2}):     
\end{align}
\end{subequations} 
where the terms corresponding to $l_{k_{1}k_{2}k_{3}k_{4}, i_{2}i_{3}j_{0}j_{1}, s\sigma t}^{7i}$ for $i\in \{2,3\}$ vanish due to $k_{12}$ in \eqref{estimate 184}. Similarly, eighth, within \eqref{estimate 184}, we see that 
\begin{equation}
 \hat{\bar{X}}_{\sigma, u}^{\epsilon, i_{2}}(k_{1}) \hat{\bar{X}}_{\sigma, b}^{\epsilon, i_{3}}(k_{2}) \hat{\bar{X}}_{s,b}^{\epsilon, j_{1}}(k_{3}) \hat{\bar{X}}_{t,b}^{\epsilon, j_{0}}(k_{4}) = \sum_{i=1}^{10} l_{k_{1}k_{2}k_{3}k_{4}, i_{2}i_{3}j_{0}j_{1}, s\sigma t}^{8i}
\end{equation}
where  
\begin{subequations}\label{estimate 187}
\begin{align} 
l_{k_{1}k_{2}k_{3}k_{4}, i_{2}i_{3}j_{0}j_{1}, s\sigma t}^{81} \triangleq& : \hat{\bar{X}}_{\sigma, u}^{\epsilon, i_{2}}(k_{1}) \hat{\bar{X}}_{\sigma, b}^{\epsilon, i_{3}}(k_{2}) \hat{\bar{X}}_{s,b}^{\epsilon, j_{1}}(k_{3}) \hat{\bar{X}}_{t,b}^{\epsilon, j_{0}}(k_{4}):, \\
 l_{k_{1}k_{2}k_{3}k_{4}, i_{2}i_{3}j_{0}j_{1}, s\sigma t}^{82} \triangleq& 1_{k_{12} = 0} \sum_{i_{4} =1}^{3} \frac{ h_{b}(\epsilon k_{1})h_{u}(\epsilon k_{1})}{2 \lvert k_{1} \rvert^{2} } \nonumber\\
 & \hspace{5mm} \times \hat{\mathcal{P}}^{i_{2} i_{4}}(k_{1}) \hat{\mathcal{P}}^{i_{3}i_{4}} (k_{1}): \hat{\bar{X}}_{s,b}^{\epsilon, j_{1}}(k_{3}) \hat{\bar{X}}_{t,b}^{\epsilon, j_{0}}(k_{4}):,\\
 l_{k_{1}k_{2}k_{3}k_{4}, i_{2}i_{3}j_{0}j_{1}, s\sigma t}^{83} \triangleq& 1_{k_{12} = 0, k_{34} = 0} \sum_{i_{4}, i_{5} =1}^{3} \frac{ h_{b}(\epsilon k_{1})h_{u}(\epsilon k_{1})}{2\lvert k_{1} \rvert^{2} } \hat{\mathcal{P}}^{i_{2} i_{4} }(k_{1}) \hat{\mathcal{P}}^{i_{3}i_{4}}(k_{1}) \nonumber\\
& \hspace{5mm} \times  \frac{ e^{ - \lvert k_{3} \rvert^{2} (t-s)} h_{b}(\epsilon k_{3})^{2}}{2\lvert k_{3} \rvert^{2}} \hat{\mathcal{P}}^{j_{1}i_{5}} (k_{3}) \hat{\mathcal{P}}^{j_{0} i_{5}}(k_{3}), \\ 
 l_{k_{1}k_{2}k_{3}k_{4}, i_{2}i_{3}j_{0}j_{1}, s\sigma t}^{84} \triangleq& 1_{k_{13} = 0} \sum_{i_{4} =1}^{3} \frac{ e^{- \lvert k_{1}\rvert^{2} (s-\sigma)} h_{b}(\epsilon k_{1})h_{u}(\epsilon k_{1})}{2 \lvert k_{1} \rvert^{2}} \nonumber\\
 & \hspace{5mm}  \hat{\mathcal{P}}^{i_{2}i_{4}}(k_{1}) \hat{\mathcal{P}}^{j_{1}i_{4}}(k_{1}) : \hat{\bar{X}}_{\sigma, b}^{\epsilon, i_{3}}(k_{2}) \hat{\bar{X}}_{t,b}^{\epsilon, j_{0}}(k_{4}):, \\
 l_{k_{1}k_{2}k_{3}k_{4}, i_{2}i_{3}j_{0}j_{1}, s\sigma t}^{85} \triangleq& 1_{k_{13} = 0, k_{24} = 0} \sum_{i_{4}, i_{5} = 1}^{3} \frac{ e^{- \lvert k_{1}  \rvert^{2} (s-\sigma)} h_{u} (\epsilon k_{1}) h_{b}(\epsilon k_{1})}{2 \lvert k_{1} \rvert^{2}} \hat{\mathcal{P}}^{i_{2}i_{4}}(k_{1}) \hat{\mathcal{P}}^{j_{1}i_{4}}(k_{1}) \nonumber\\
& \hspace{5mm} \times \frac{ e^{- \lvert k_{2} \rvert^{2} (t-\sigma)} h_{b}(\epsilon k_{2})^{2}}{2 \lvert k_{2} \rvert^{2}} \hat{\mathcal{P}}^{i_{3}i_{5}}(k_{2}) \hat{\mathcal{P}}^{j_{0} i_{5}}(k_{2}), \\
 l_{k_{1}k_{2}k_{3}k_{4}, i_{2}i_{3}j_{0}j_{1}, s\sigma t}^{86} \triangleq& 1_{k_{14} = 0} \sum_{i_{4} =1}^{3} \frac{ e^{- \lvert k_{1} \rvert^{2} (t-\sigma)} h_{u}(\epsilon k_{1})  h_{b}(\epsilon k_{1})}{2 \lvert k_{1} \rvert^{2}} \nonumber\\
 & \hspace{5mm} \times  \hat{\mathcal{P}}^{i_{2}i_{4}}(k_{1}) \hat{\mathcal{P}}^{j_{0}i_{4}}(k_{1}) : \hat{\bar{X}}_{\sigma, b}^{\epsilon, i_{3}}(k_{2}) \hat{\bar{X}}_{s,b}^{\epsilon, j_{1}}(k_{3}):, \\
 l_{k_{1}k_{2}k_{3}k_{4}, i_{2}i_{3}j_{0}j_{1}, s\sigma t}^{87} \triangleq& 1_{k_{14} = 0, k_{23} = 0} \sum_{i_{4}, i_{5} = 1}^{3} \frac{ e^{- \lvert k_{1} \rvert^{2} (t-\sigma)} h_{u}(\epsilon k_{1}) h_{b}(\epsilon k_{1})}{2 \lvert k_{1} \rvert^{2}} \hat{\mathcal{P}}^{i_{2}i_{4}}(k_{1}) \hat{\mathcal{P}}^{j_{0}i_{4}}(k_{1}) \nonumber\\
& \hspace{5mm} \times \frac{ e^{- \lvert k_{2} \rvert^{2} (s-\sigma)} h_{b}(\epsilon k_{2})^{2} }{2 \lvert k_{2} \rvert^{2}} \hat{\mathcal{P}}^{i_{3}i_{5}}(k_{2}) \hat{\mathcal{P}}^{j_{1} i_{5}}(k_{2}), \\
 l_{k_{1}k_{2}k_{3}k_{4}, i_{2}i_{3}j_{0}j_{1}, s\sigma t}^{88} \triangleq& 1_{k_{23} = 0} \sum_{i_{4} =1}^{3} \frac{e^{- \lvert k_{2} \rvert^{2} (s-\sigma)} h_{b}(\epsilon k_{2})^{2}}{2 \lvert k_{2} \rvert^{2}}  \nonumber\\
 & \hspace{5mm} \times \hat{\mathcal{P}}^{i_{3} i_{4}}(k_{2}) \hat{\mathcal{P}}^{j_{1}i_{4}} (k_{2}): \hat{\bar{X}}_{\sigma,u}^{\epsilon, i_{2}}(k_{1}) \hat{\bar{X}}_{t,b}^{\epsilon, j_{0}}(k_{4}):, \\
 l_{k_{1}k_{2}k_{3}k_{4}, i_{2}i_{3}j_{0}j_{1}, s\sigma t}^{89} \triangleq& 1_{k_{24} = 0} \sum_{i_{4} =1}^{3} \frac{e^{- \lvert k_{2} \rvert^{2} (t-\sigma)}h_{b}(\epsilon k_{2})^{2}}{2 \lvert k_{2} \rvert^{2}} \nonumber\\
 & \hspace{5mm} \times \hat{\mathcal{P}}^{i_{3} i_{4}}(k_{2}) \hat{\mathcal{P}}^{j_{0}i_{4}} (k_{2}): \hat{\bar{X}}_{\sigma,u}^{\epsilon, i_{2}}(k_{1}) \hat{\bar{X}}_{s,b}^{\epsilon, j_{1}}(k_{3}):, \\
 l_{k_{1}k_{2}k_{3}k_{4}, i_{2}i_{3}j_{0}j_{1}, s\sigma t}^{810} \triangleq& 1_{k_{34} = 0} \sum_{i_{4} =1}^{3} \frac{e^{- \lvert k_{3} \rvert^{2} (t- s)} h_{b}(\epsilon k_{3})^{2}}{2 \lvert k_{3} \rvert^{2}}  \nonumber\\
 & \hspace{5mm} \times \hat{\mathcal{P}}^{j_{1} i_{4}}(k_{3}) \hat{\mathcal{P}}^{j_{0}i_{4}} (k_{3}): \hat{\bar{X}}_{\sigma,u}^{\epsilon, i_{2}}(k_{1}) \hat{\bar{X}}_{\sigma,b}^{\epsilon, i_{3}}(k_{2}):     
\end{align}
\end{subequations} 
where the terms corresponding to $ l_{k_{1}k_{2}k_{3}k_{4}, i_{2}i_{3}j_{0}j_{1}, s\sigma t}^{8i}$ for $i\in \{2,3\}$ vanish due to $k_{12}$ in \eqref{estimate 184}. Applying \eqref{estimate 185} - \eqref{estimate 187} to \eqref{estimate 184} and \eqref{estimate 183} gives us 
\begin{equation}\label{[Equation (4.6e)][ZZ17]} 
\pi_{0} (u_{32}^{\epsilon, i_{0}}, b_{1}^{\epsilon, j_{0}})(t) - \pi_{0} ( \bar{u}_{32}^{\epsilon, i_{0}}, \bar{b}_{1}^{\epsilon, j_{0}})(t) = \frac{1}{4} (\sum_{k=1}^{7} L_{t, i_{0}j_{0}}^{k} - \bar{L}_{t_{0}, i_{0}j_{0}}^{5} - \bar{L}_{t_{0}, i_{0}j_{0}}^{6}) 
\end{equation} 
where $\bar{L}_{t, i_{0}j_{0}}^{5}$ and $\bar{L}_{t, i_{0}j_{0}}^{6}$ were defined in \eqref{[Equation (4.6k)][ZZ17]}-\eqref{[Equation (4.6m)][ZZ17]}, $L_{t, i_{0}j_{0}}^{1}$ corresponds to the terms from $l_{k_{1}k_{2}k_{3}k_{4}, i_{1}i_{2}i_{3}j_{0}j_{1}, s \sigma t}^{m1}$ for $m\in \{1, \hdots, 8\}$ so that 
\begin{align}\label{[Equation (4.6f)][ZZ17]}
L_{t, i_{0}j_{0}}^{1} \triangleq& (2\pi)^{-\frac{9}{2}} \sum_{\lvert i-j \rvert \leq 1} \sum_{i_{1}, i_{2}, i_{3}, j_{1} =1}^{3} \sum_{k} \sum_{k_{1}, k_{2}, k_{3}, k_{4} \neq 0: k_{1234} = k} \theta(2^{-i} k_{123}) \theta(2^{-j} k_{4}) \nonumber\\
& \times \int_{0}^{t} [ e^{- \lvert k_{123} \rvert^{2}  f(\epsilon k_{123})(t-s)} \int_{0}^{s} : \hat{X}_{\sigma, u}^{\epsilon, i_{2}}(k_{1}) \hat{X}_{\sigma, u}^{\epsilon, i_{3}}(k_{2}) \hat{X}_{s,u}^{\epsilon, j_{1}}(k_{3}) \hat{X}_{t,b}^{\epsilon, j_{0}}(k_{4}): \nonumber \\
& \hspace{10mm} \times e^{- \lvert k_{12} \rvert^{2} f(\epsilon k_{12})(s-\sigma) } d\sigma k_{123}^{j_{1}} g(\epsilon k_{123}^{j_{1}}) k_{12}^{i_{3}} g(\epsilon k_{12}^{i_{3}}) \nonumber \\
& - e^{- \lvert k_{123} \rvert^{2} (t-s) } \int_{0}^{s}  : \hat{\bar{X}}_{\sigma, u}^{\epsilon, i_{2}}(k_{1}) \hat{\bar{X}}_{\sigma, u}^{\epsilon, i_{3}}(k_{2}) \hat{\bar{X}}_{s,u}^{\epsilon, j_{1}}(k_{3}) \hat{\bar{X}}_{t,b}^{\epsilon, j_{0}}(k_{4}): \nonumber \\
& \hspace{10mm} \times e^{- \lvert k_{12} \rvert^{2} (s-\sigma)} d\sigma k_{123}^{j_{1}} i k_{12}^{i_{3}} i \nonumber \\
& - e^{- \lvert k_{123} \rvert^{2}  f(\epsilon k_{123})(t-s)} \int_{0}^{s} : \hat{X}_{\sigma, b}^{\epsilon, i_{2}}(k_{1}) \hat{X}_{\sigma, b}^{\epsilon, i_{3}}(k_{2}) \hat{X}_{s,u}^{\epsilon, j_{1}}(k_{3}) \hat{X}_{t,b}^{\epsilon, j_{0}}(k_{4}): \nonumber \\
& \hspace{10mm} \times e^{- \lvert k_{12} \rvert^{2}  f(\epsilon k_{12})(s-\sigma)}  d\sigma k_{123}^{j_{1}} g(\epsilon k_{123}^{j_{1}}) k_{12}^{i_{3}} g(\epsilon k_{12}^{i_{3}}) \nonumber \\
& + e^{- \lvert k_{123} \rvert^{2} (t-s) } \int_{0}^{s}  : \hat{\bar{X}}_{\sigma, b}^{\epsilon, i_{2}}(k_{1}) \hat{\bar{X}}_{\sigma, b}^{\epsilon, i_{3}}(k_{2}) \hat{\bar{X}}_{s,u}^{\epsilon, j_{1}}(k_{3}) \hat{\bar{X}}_{t,b}^{\epsilon, j_{0}}(k_{4}): \nonumber \\
& \hspace{10mm} \times e^{- \lvert k_{12} \rvert^{2} (s-\sigma)} d\sigma k_{123}^{j_{1}} i k_{12}^{i_{3}} i \nonumber \\
&-  e^{- \lvert k_{123} \rvert^{2}  f(\epsilon k_{123})(t-s)} \int_{0}^{s} : \hat{X}_{\sigma, b}^{\epsilon, i_{2}}(k_{1}) \hat{X}_{\sigma, u}^{\epsilon, i_{3}}(k_{2}) \hat{X}_{s,b}^{\epsilon, j_{1}}(k_{3}) \hat{X}_{t,b}^{\epsilon, j_{0}}(k_{4}): \nonumber \\
& \hspace{10mm} \times e^{- \lvert k_{12} \rvert^{2}  f(\epsilon k_{12})(s-\sigma)} d\sigma k_{123}^{j_{1}} g(\epsilon k_{123}^{j_{1}}) k_{12}^{i_{3}} g(\epsilon k_{12}^{i_{3}}) \nonumber \\
& + e^{- \lvert k_{123} \rvert^{2} (t-s) } \int_{0}^{s}  : \hat{\bar{X}}_{\sigma, b}^{\epsilon, i_{2}}(k_{1}) \hat{\bar{X}}_{\sigma, u}^{\epsilon, i_{3}}(k_{2}) \hat{\bar{X}}_{s,b}^{\epsilon, j_{1}}(k_{3}) \hat{\bar{X}}_{t,b}^{\epsilon, j_{0}}(k_{4}): \nonumber \\
& \hspace{10mm} \times e^{- \lvert k_{12} \rvert^{2} (s-\sigma)} d\sigma k_{123}^{j_{1}} i k_{12}^{i_{3}} i \nonumber \\
& + e^{- \lvert k_{123} \rvert^{2}  f(\epsilon k_{123})(t-s)} \int_{0}^{s} : \hat{X}_{\sigma, u}^{\epsilon, i_{2}}(k_{1}) \hat{X}_{\sigma, b}^{\epsilon, i_{3}}(k_{2}) \hat{X}_{s,b}^{\epsilon, j_{1}}(k_{3}) \hat{X}_{t,b}^{\epsilon, j_{0}}(k_{4}): \nonumber \\
& \hspace{10mm} \times e^{- \lvert k_{12} \rvert^{2}  f(\epsilon k_{12})(s-\sigma)} d\sigma k_{123}^{j_{1}} g(\epsilon k_{123}^{j_{1}}) k_{12}^{i_{3}} g(\epsilon k_{12}^{i_{3}}) \nonumber \\
& - e^{- \lvert k_{123} \rvert^{2} (t-s) } \int_{0}^{s}  : \hat{\bar{X}}_{\sigma, u}^{\epsilon, i_{2}}(k_{1}) \hat{\bar{X}}_{\sigma, b}^{\epsilon, i_{3}}(k_{2}) \hat{\bar{X}}_{s,b}^{\epsilon, j_{1}}(k_{3}) \hat{\bar{X}}_{t,b}^{\epsilon, j_{0}}(k_{4}): \nonumber \\
& \hspace{10mm} \times e^{- \lvert k_{12} \rvert^{2} (s-\sigma)} d\sigma k_{123}^{j_{1}} i k_{12}^{i_{3}} i ] ds \hat{\mathcal{P}}^{i_{0} i_{1}}(k_{123}) \hat{\mathcal{P}}^{i_{1} i_{2}}(k_{12}) e_{k}, 
\end{align} 
$L_{t, i_{0}j_{0}}^{2}$ corresponds to the terms from $l_{k_{1}k_{2}k_{3}k_{4}, i_{1}i_{2}i_{3}j_{0}j_{1}, s \sigma t}^{m6}$ for $m \in \{1, \hdots, 8 \}$ so that 
\begin{align}\label{estimate 195}
L_{t, i_{0}j_{0}}^{2} \triangleq& (2\pi)^{-\frac{9}{2}} \sum_{\lvert i-j \rvert \leq 1} \sum_{i_{1}, i_{2}, i_{3}, j_{1} =1}^{3} \sum_{k} \sum_{k_{1}, k_{2}, k_{3} \neq 0: k_{23} = k} \theta(2^{-i} k_{123}) \theta(2^{-j} k_{1}) \\
& \times [ \int_{0}^{t} e^{- \lvert k_{123} \rvert^{2}  f( \epsilon k_{123})(t-s) } \int_{0}^{s} : \hat{X}_{\sigma, u}^{\epsilon, i_{3}}(k_{2}) \hat{X}_{s,u}^{\epsilon, j_{1}}(k_{3}): k_{12}^{i_{3}} g(\epsilon k_{12}^{i_{3}}) k_{123}^{j_{1}} g(\epsilon k_{123}^{j_{1}})  \nonumber\\
& \hspace{5mm} \times\frac{ e^{- \lvert k_{1} \rvert^{2} f(\epsilon k_{1})(t-\sigma)} h_{u}(\epsilon k_{1})h_{b}(\epsilon k_{1})}{2 \lvert k_{1} \rvert^{2} f(\epsilon k_{1})} e^{- \lvert k_{12} \rvert^{2} f(\epsilon k_{12}) (s-\sigma)} d \sigma ds \nonumber \\
& - \int_{0}^{t} e^{- \lvert k_{123} \rvert^{2} (t-s) } \int_{0}^{s} : \hat{\bar{X}}_{\sigma, u}^{\epsilon, i_{3}}(k_{2}) \hat{\bar{X}}_{s,u}^{\epsilon, j_{1}}(k_{3}): k_{12}^{i_{3}} i  k_{123}^{j_{1}} i  \nonumber\\
& \hspace{5mm} \times \frac{ e^{- \lvert k_{1} \rvert^{2} (t-\sigma)} h_{u}(\epsilon k_{1})h_{b}(\epsilon k_{1})}{2 \lvert k_{1} \rvert^{2}} e^{- \lvert k_{12} \rvert^{2} (s-\sigma) } d \sigma ds \nonumber \\
& - \int_{0}^{t} e^{- \lvert k_{123} \rvert^{2} f( \epsilon k_{123}) (t-s) } \int_{0}^{s} : \hat{X}_{\sigma, b}^{\epsilon, i_{3}}(k_{2}) \hat{X}_{s,u}^{\epsilon, j_{1}}(k_{3}): k_{12}^{i_{3}} g(\epsilon k_{12}^{i_{3}}) k_{123}^{j_{1}} g(\epsilon k_{123}^{j_{1}})  \nonumber\\
& \hspace{5mm} \times\frac{ e^{- \lvert k_{1} \rvert^{2} f(\epsilon k_{1})(t-\sigma)}h_{b}(\epsilon k_{1})^{2}}{2 \lvert k_{1} \rvert^{2} f(\epsilon k_{1})} e^{- \lvert k_{12} \rvert^{2}  f(\epsilon k_{12})(s-\sigma)} d \sigma ds \nonumber \\
& + \int_{0}^{t} e^{- \lvert k_{123} \rvert^{2} (t-s) } \int_{0}^{s} : \hat{\bar{X}}_{\sigma, b}^{\epsilon, i_{3}}(k_{2}) \hat{\bar{X}}_{s,u}^{\epsilon, j_{1}}(k_{3}): k_{12}^{i_{3}} i  k_{123}^{j_{1}} i  \nonumber\\
& \hspace{5mm} \times \frac{ e^{- \lvert k_{1} \rvert^{2} (t-\sigma)} h_{b}(\epsilon k_{1})^{2}}{2 \lvert k_{1} \rvert^{2}} e^{- \lvert k_{12} \rvert^{2} (s-\sigma) } d \sigma ds \nonumber \\
& - \int_{0}^{t} e^{- \lvert k_{123} \rvert^{2}  f( \epsilon k_{123})(t-s) } \int_{0}^{s} : \hat{X}_{\sigma, u}^{\epsilon, i_{3}}(k_{2}) \hat{X}_{s,b}^{\epsilon, j_{1}}(k_{3}): k_{12}^{i_{3}} g(\epsilon k_{12}^{i_{3}}) k_{123}^{j_{1}} g(\epsilon k_{123}^{j_{1}})  \nonumber\\
& \hspace{5mm} \times\frac{ e^{- \lvert k_{1} \rvert^{2} f(\epsilon k_{1})(t-\sigma)} h_{b}(\epsilon k_{1})^{2}}{2 \lvert k_{1} \rvert^{2} f(\epsilon k_{1})} e^{- \lvert k_{12} \rvert^{2} f(\epsilon k_{12})(s-\sigma) } d \sigma ds \nonumber \\
& + \int_{0}^{t} e^{- \lvert k_{123} \rvert^{2} (t-s) } \int_{0}^{s} : \hat{\bar{X}}_{\sigma, u}^{\epsilon, i_{3}}(k_{2}) \hat{\bar{X}}_{s,b}^{\epsilon, j_{1}}(k_{3}): k_{12}^{i_{3}} i  k_{123}^{j_{1}} i  \nonumber\\
& \hspace{5mm} \times \frac{ e^{- \lvert k_{1} \rvert^{2} (t-\sigma)} h_{b}(\epsilon k_{1})^{2}}{2 \lvert k_{1} \rvert^{2}} e^{- \lvert k_{12} \rvert^{2} (s-\sigma) } d \sigma ds \nonumber \\
&+ \int_{0}^{t} e^{- \lvert k_{123} \rvert^{2}  f( \epsilon k_{123})(t-s) } \int_{0}^{s} : \hat{X}_{\sigma, b}^{\epsilon, i_{3}}(k_{2}) \hat{X}_{s,b}^{\epsilon, j_{1}}(k_{3}): k_{12}^{i_{3}} g(\epsilon k_{12}^{i_{3}}) k_{123}^{j_{1}} g(\epsilon k_{123}^{j_{1}})  \nonumber\\
& \hspace{5mm} \times\frac{ e^{- \lvert k_{1} \rvert^{2} f(\epsilon k_{1})(t-\sigma)} h_{u}(\epsilon k_{1})h_{b}(\epsilon k_{1})}{2 \lvert k_{1} \rvert^{2} f(\epsilon k_{1})} e^{- \lvert k_{12} \rvert^{2}  f(\epsilon k_{12})(s-\sigma)} d \sigma ds \nonumber \\
& - \int_{0}^{t} e^{- \lvert k_{123} \rvert^{2} (t-s) } \int_{0}^{s} : \hat{\bar{X}}_{\sigma, b}^{\epsilon, i_{3}}(k_{2}) \hat{\bar{X}}_{s,b}^{\epsilon, j_{1}}(k_{3}): k_{12}^{i_{3}} i  k_{123}^{j_{1}} i  \nonumber\\
& \hspace{5mm} \times \frac{ e^{- \lvert k_{1} \rvert^{2} (t-\sigma)} h_{u}(\epsilon k_{1})h_{b}(\epsilon k_{1})}{2 \lvert k_{1} \rvert^{2}} e^{- \lvert k_{12} \rvert^{2} (s-\sigma) } d \sigma ds] \nonumber \\
& \times \sum_{i_{4} = 1}^{3} \hat{\mathcal{P}}^{i_{2} i_{4}} (k_{1}) \hat{\mathcal{P}}^{j_{0} i_{4}}(k_{1}) \hat{\mathcal{P}}^{i_{1} i_{2}} (k_{12}) \hat{\mathcal{P}}^{i_{0} i_{1}} (k_{123}) e_{k}, \nonumber
\end{align} 
$L_{t, i_{0}j_{0}}^{3}$ corresponds to the terms from $l_{k_{1}k_{2}k_{3}k_{4}, i_{1}i_{2}i_{3}j_{0}j_{1}, s \sigma t}^{m9}$ for $m\in \{1, \hdots, 8\}$ so that 
\begin{align}\label{estimate 196}
L_{t, i_{0}j_{0}}^{3} \triangleq& (2\pi)^{-\frac{9}{2}} \sum_{\lvert i-j \rvert \leq 1} \sum_{i_{1}, i_{2}, i_{3}, j_{1} =1}^{3} \sum_{k} \sum_{k_{1}, k_{2}, k_{3} \neq 0: k_{13} = k} \theta(2^{-i} k_{123}) \theta(2^{-j} k_{2}) \\
& \times [ \int_{0}^{t} e^{- \lvert k_{123} \rvert^{2}  f(\epsilon k_{123})(t-s)} \int_{0}^{s} : \hat{X}_{\sigma, u}^{\epsilon, i_{2}}(k_{1}) \hat{X}_{s,u}^{\epsilon, j_{1}}(k_{3}): k_{12}^{i_{3}} g(\epsilon k_{12}^{i_{3}}) k_{123}^{j_{1}} g(\epsilon k_{123}^{j_{1}}) \nonumber\\
& \hspace{5mm} \times \frac{ e^{ - \lvert k_{2} \rvert^{2} f(\epsilon k_{2}) (t-\sigma)} h_{u}(\epsilon k_{2}) h_{b}(\epsilon k_{2})}{2\lvert k_{2} \rvert^{2} f(\epsilon k_{2})} e^{- \lvert k_{12} \rvert^{2} f(\epsilon k_{12})(s-\sigma) } d\sigma ds   \nonumber  \\
& -  \int_{0}^{t} e^{- \lvert k_{123} \rvert^{2} (t-s) } \int_{0}^{s} : \hat{\bar{X}}_{\sigma, u}^{\epsilon, i_{2}}(k_{1}) \hat{\bar{X}}_{s,u}^{\epsilon, j_{1}}(k_{3}): k_{12}^{i_{3}} i k_{123}^{j_{1}} i \nonumber\\
& \hspace{5mm} \times \frac{ e^{ - \lvert k_{2} \rvert^{2}  (t-\sigma)} h_{u}(\epsilon k_{2}) h_{b}(\epsilon k_{2})}{2\lvert k_{2} \rvert^{2}} e^{- \lvert k_{12} \rvert^{2} (s-\sigma)} d\sigma ds   \nonumber \\
& - \int_{0}^{t} e^{- \lvert k_{123} \rvert^{2}  f(\epsilon k_{123})(t-s)} \int_{0}^{s} : \hat{X}_{\sigma, b}^{\epsilon, i_{2}}(k_{1}) \hat{X}_{s,u}^{\epsilon, j_{1}}(k_{3}): k_{12}^{i_{3}} g(\epsilon k_{12}^{i_{3}}) k_{123}^{j_{1}} g(\epsilon k_{123}^{j_{1}}) \nonumber\\
& \hspace{5mm} \times \frac{ e^{ - \lvert k_{2} \rvert^{2} f(\epsilon k_{2}) (t-\sigma)} h_{b}(\epsilon k_{2})^{2}}{2\lvert k_{2} \rvert^{2} f(\epsilon k_{2})} e^{- \lvert k_{12} \rvert^{2} f(\epsilon k_{12})(s-\sigma) } d\sigma ds   \nonumber  \\
& +  \int_{0}^{t} e^{- \lvert k_{123} \rvert^{2} (t-s) } \int_{0}^{s} : \hat{\bar{X}}_{\sigma, b}^{\epsilon, i_{2}}(k_{1}) \hat{\bar{X}}_{s,u}^{\epsilon, j_{1}}(k_{3}): k_{12}^{i_{3}} i k_{123}^{j_{1}} i \nonumber\\
& \hspace{5mm} \times \frac{ e^{ - \lvert k_{2} \rvert^{2}  (t-\sigma)} h_{b}(\epsilon k_{2})^{2}}{2\lvert k_{2} \rvert^{2}} e^{- \lvert k_{12} \rvert^{2} (s-\sigma)} d\sigma ds   \nonumber \\
& - \int_{0}^{t} e^{- \lvert k_{123} \rvert^{2}  f(\epsilon k_{123})(t-s)} \int_{0}^{s} : \hat{X}_{\sigma, b}^{\epsilon, i_{2}}(k_{1}) \hat{X}_{s,b}^{\epsilon, j_{1}}(k_{3}): k_{12}^{i_{3}} g(\epsilon k_{12}^{i_{3}}) k_{123}^{j_{1}} g(\epsilon k_{123}^{j_{1}}) \nonumber\\
& \hspace{5mm} \times \frac{ e^{ - \lvert k_{2} \rvert^{2} f(\epsilon k_{2}) (t-\sigma)} h_{u}(\epsilon k_{2}) h_{b}(\epsilon k_{2})}{2\lvert k_{2} \rvert^{2} f(\epsilon k_{2})} e^{- \lvert k_{12} \rvert^{2}  f(\epsilon k_{12})(s-\sigma)} d\sigma ds   \nonumber  \\
& +  \int_{0}^{t} e^{- \lvert k_{123} \rvert^{2} (t-s) } \int_{0}^{s} : \hat{\bar{X}}_{\sigma, b}^{\epsilon, i_{2}}(k_{1}) \hat{\bar{X}}_{s,b}^{\epsilon, j_{1}}(k_{3}): k_{12}^{i_{3}} i k_{123}^{j_{1}} i \nonumber\\
& \hspace{5mm} \times \frac{ e^{ - \lvert k_{2} \rvert^{2}  (t-\sigma)} h_{u}(\epsilon k_{2}) h_{b}(\epsilon k_{2})}{2\lvert k_{2} \rvert^{2}} e^{- \lvert k_{12} \rvert^{2} (s-\sigma)} d\sigma ds   \nonumber \\
&+ \int_{0}^{t} e^{- \lvert k_{123} \rvert^{2}  f(\epsilon k_{123})(t-s)} \int_{0}^{s} : \hat{X}_{\sigma, u}^{\epsilon, i_{2}}(k_{1}) \hat{X}_{s,b}^{\epsilon, j_{1}}(k_{3}): k_{12}^{i_{3}} g(\epsilon k_{12}^{i_{3}}) k_{123}^{j_{1}} g(\epsilon k_{123}^{j_{1}}) \nonumber\\
& \hspace{5mm} \times \frac{ e^{ - \lvert k_{2} \rvert^{2} f(\epsilon k_{2}) (t-\sigma)} h_{b}(\epsilon k_{2})^{2}}{2\lvert k_{2} \rvert^{2} f(\epsilon k_{2})} e^{- \lvert k_{12} \rvert^{2}f(\epsilon k_{12}) (s-\sigma) } d\sigma ds   \nonumber  \\
& -  \int_{0}^{t} e^{- \lvert k_{123} \rvert^{2} (t-s) } \int_{0}^{s} : \hat{\bar{X}}_{\sigma, u}^{\epsilon, i_{2}}(k_{1}) \hat{\bar{X}}_{s,b}^{\epsilon, j_{1}}(k_{3}): k_{12}^{i_{3}} i k_{123}^{j_{1}} i \nonumber\\
& \hspace{5mm} \times \frac{ e^{ - \lvert k_{2} \rvert^{2}  (t-\sigma)} h_{b}(\epsilon k_{2})^{2}}{2\lvert k_{2} \rvert^{2}} e^{- \lvert k_{12} \rvert^{2} (s-\sigma)} d\sigma ds]   \nonumber \\
& \times \sum_{i_{4} =1}^{3} \hat{\mathcal{P}}^{i_{3} i_{4}}(k_{2}) \hat{\mathcal{P}}^{j_{0} i_{4}}(k_{2}) \hat{\mathcal{P}}^{i_{1} i_{2}}(k_{12}) \hat{\mathcal{P}}^{i_{0} i_{1}}(k_{123}) e_{k},\nonumber
\end{align} 
$L_{t, i_{0}j_{0}}^{4}$ corresponds to terms from $l_{k_{1}k_{2}k_{3}k_{4}, i_{2}i_{3}j_{0}j_{1}, s \sigma t}^{m 10}$ for $m \in \{1, \hdots, 8\}$ so that  
\begin{align}\label{estimate 197}
L_{t, i_{0}j_{0}}^{4} \triangleq&(2\pi)^{-\frac{9}{2}} \sum_{\lvert i -j \rvert \leq 1} \sum_{i_{1}, i_{2}, i_{3}, j_{1} =1}^{3} \sum_{k} \sum_{k_{1}, k_{2}, k_{3} \neq 0: k_{12} = k} \theta(2^{-i} k_{123}) \theta(2^{-j} k_{3}) \\
& \times [ \int_{0}^{t} e^{- \lvert k_{123} \rvert^{2} f(\epsilon k_{123}) (t-s)} \int_{0}^{s}: \hat{X}_{\sigma, u}^{\epsilon, i_{2}}(k_{1}) \hat{X}_{\sigma, u}^{\epsilon, i_{3}}(k_{2}): k_{12}^{i_{3}} g(\epsilon k_{12}^{i_{3}}) k_{123}^{j_{1}} g(\epsilon k_{123}^{j_{1}}) \nonumber\\
& \hspace{5mm} \times \frac{ e^{- \lvert k_{3} \rvert^{2} f(\epsilon k_{3}) (t-s)} h_{u}(\epsilon k_{3}) h_{b}(\epsilon k_{3})}{2\lvert k_{3} \rvert^{2} f(\epsilon k_{3})} e^{- \lvert k_{12} \rvert^{2} f(\epsilon k_{12})(s-\sigma)} d\sigma ds \nonumber\\
& - \int_{0}^{t} e^{- \lvert k_{123} \rvert^{2} (t-s) } \int_{0}^{s}: \hat{\bar{X}}_{\sigma, u}^{\epsilon, i_{2}}(k_{1}) \hat{\bar{X}}_{\sigma, u}^{\epsilon, i_{3}}(k_{2}): k_{12}^{i_{3}} i k_{123}^{j_{1}} i \nonumber\\
& \hspace{5mm} \times \frac{ e^{- \lvert k_{3} \rvert^{2} (t-s)} h_{u}(\epsilon k_{3}) h_{b}(\epsilon k_{3})}{2\lvert k_{3} \rvert^{2}} e^{- \lvert k_{12} \rvert^{2} (s-\sigma)} d\sigma ds \nonumber\\
& - \int_{0}^{t} e^{- \lvert k_{123} \rvert^{2}  f(\epsilon k_{123})(t-s)} \int_{0}^{s}: \hat{X}_{\sigma, b}^{\epsilon, i_{2}}(k_{1}) \hat{X}_{\sigma, b}^{\epsilon, i_{3}}(k_{2}): k_{12}^{i_{3}} g(\epsilon k_{12}^{i_{3}}) k_{123}^{j_{1}} g(\epsilon k_{123}^{j_{1}}) \nonumber\\
& \hspace{5mm} \times \frac{ e^{- \lvert k_{3} \rvert^{2} f(\epsilon k_{3}) (t-s)} h_{u}(\epsilon k_{3}) h_{b}(\epsilon k_{3})}{2\lvert k_{3} \rvert^{2} f(\epsilon k_{3})} e^{- \lvert k_{12} \rvert^{2} f(\epsilon k_{12})(s-\sigma)} d\sigma ds \nonumber\\
& + \int_{0}^{t} e^{- \lvert k_{123} \rvert^{2} (t-s) } \int_{0}^{s}: \hat{\bar{X}}_{\sigma, b}^{\epsilon, i_{2}}(k_{1}) \hat{\bar{X}}_{\sigma, b}^{\epsilon, i_{3}}(k_{2}): k_{12}^{i_{3}} i k_{123}^{j_{1}} i \nonumber\\
& \hspace{5mm} \times \frac{ e^{- \lvert k_{3} \rvert^{2} (t-s)} h_{u}(\epsilon k_{3}) h_{b}(\epsilon k_{3})}{2\lvert k_{3} \rvert^{2}} e^{- \lvert k_{12} \rvert^{2} (s-\sigma)} d\sigma ds \nonumber\\
& - \int_{0}^{t} e^{- \lvert k_{123} \rvert^{2}  f(\epsilon k_{123})(t-s)} \int_{0}^{s}: \hat{X}_{\sigma, b}^{\epsilon, i_{2}}(k_{1}) \hat{X}_{\sigma, u}^{\epsilon, i_{3}}(k_{2}): k_{12}^{i_{3}} g(\epsilon k_{12}^{i_{3}}) k_{123}^{j_{1}} g(\epsilon k_{123}^{j_{1}}) \nonumber\\
& \hspace{5mm} \times \frac{ e^{- \lvert k_{3} \rvert^{2} f(\epsilon k_{3}) (t-s)}  h_{b}(\epsilon k_{3})^{2}}{2\lvert k_{3} \rvert^{2} f(\epsilon k_{3})} e^{- \lvert k_{12} \rvert^{2}  f(\epsilon k_{12})(s-\sigma)} d\sigma ds \nonumber\\
& + \int_{0}^{t} e^{- \lvert k_{123} \rvert^{2} (t-s) } \int_{0}^{s}: \hat{\bar{X}}_{\sigma, b}^{\epsilon, i_{2}}(k_{1}) \hat{\bar{X}}_{\sigma, u}^{\epsilon, i_{3}}(k_{2}): k_{12}^{i_{3}} i k_{123}^{j_{1}} i \nonumber\\
& \hspace{5mm} \times \frac{ e^{- \lvert k_{3} \rvert^{2} (t-s)} h_{b}(\epsilon k_{3})^{2}}{2\lvert k_{3} \rvert^{2}} e^{- \lvert k_{12} \rvert^{2} (s-\sigma)} d\sigma ds \nonumber\\
& + \int_{0}^{t} e^{- \lvert k_{123} \rvert^{2}  f(\epsilon k_{123})(t-s)} \int_{0}^{s}: \hat{X}_{\sigma, u}^{\epsilon, i_{2}}(k_{1}) \hat{X}_{\sigma, b}^{\epsilon, i_{3}}(k_{2}): k_{12}^{i_{3}} g(\epsilon k_{12}^{i_{3}}) k_{123}^{j_{1}} g(\epsilon k_{123}^{j_{1}}) \nonumber\\
& \hspace{5mm} \times \frac{ e^{- \lvert k_{3} \rvert^{2} f(\epsilon k_{3}) (t-s)} h_{b}(\epsilon k_{3})^{2}}{2\lvert k_{3} \rvert^{2} f(\epsilon k_{3})} e^{- \lvert k_{12} \rvert^{2}  f(\epsilon k_{12})(s-\sigma)} d\sigma ds \nonumber\\
& - \int_{0}^{t} e^{- \lvert k_{123} \rvert^{2} (t-s) } \int_{0}^{s}: \hat{\bar{X}}_{\sigma, u}^{\epsilon, i_{2}}(k_{1}) \hat{\bar{X}}_{\sigma, b}^{\epsilon, i_{3}}(k_{2}): k_{12}^{i_{3}} i k_{123}^{j_{1}} i \nonumber\\
& \hspace{5mm} \times \frac{ e^{- \lvert k_{3} \rvert^{2} (t-s)} h_{b}(\epsilon k_{3})^{2}}{2\lvert k_{3} \rvert^{2}} e^{- \lvert k_{12} \rvert^{2} (s-\sigma)} d\sigma ds] \nonumber\\
& \times \sum_{i_{4} =1}^{3} \hat{\mathcal{P}}^{j_{1}i_{4}} (k_{3}) \hat{\mathcal{P}}^{j_{0}i_{4}}(k_{3}) \hat{\mathcal{P}}^{i_{1} i_{2}}(k_{12}) \hat{\mathcal{P}}^{i_{0}i_{1}}(k_{123}) e_{k}, \nonumber
\end{align} 
$L_{t, i_{0}j_{0}}^{5}$ corresponds to the terms from $l_{k_{1}k_{2}k_{3}k_{4}, i_{2}i_{3}j_{0}j_{1}, s \sigma t}^{m8}$ for $m \in \{1, \hdots, 8 \}$ so that 
\begin{align}\label{[Equation (4.6j)][ZZ17]}
L_{t, i_{0}j_{0}}^{5} \triangleq& (2\pi)^{-\frac{9}{2}} \sum_{\lvert i -j \rvert \leq 1} \sum_{i_{1}, i_{2}, i_{3}, j_{1} =1}^{3} \sum_{k} \sum_{k_{1}, k_{2}, k_{4} \neq 0: k_{14} = k} \theta(2^{-i} k_{1}) \theta(2^{-j} k_{4})\\
& \times [ \int_{0}^{t} e^{- \lvert k_{1} \rvert^{2}  f(\epsilon k_{1})(t-s)} \int_{0}^{s} : \hat{X}_{\sigma, u}^{\epsilon, i_{2}}(k_{1}) \hat{X}_{t,b}^{\epsilon, j_{0}}(k_{4}): k_{12}^{i_{3}} g(\epsilon k_{12}^{i_{3}}) k_{1}^{j_{1}} g(\epsilon k_{1}^{j_{1}}) \nonumber\\
& \hspace{5mm} \times \frac{ e^{- \lvert k_{2} \rvert^{2} f(\epsilon k_{2}) (s-\sigma)} h_{u}(\epsilon k_{2})^{2}}{2 \lvert k_{2} \rvert^{2} f(\epsilon k_{2})} e^{- \lvert k_{12} \rvert^{2} f(\epsilon k_{12})(s-\sigma)}d\sigma ds \nonumber  \\
& -  \int_{0}^{t} e^{- \lvert k_{1} \rvert^{2} (t-s) } \int_{0}^{s} : \hat{\bar{X}}_{\sigma, u}^{\epsilon, i_{2}}(k_{1}) \hat{\bar{X}}_{t,b}^{\epsilon, j_{0}}(k_{4}): k_{12}^{i_{3}} i k_{1}^{j_{1}} i \nonumber\\
& \hspace{5mm} \times \frac{ e^{- \lvert k_{2} \rvert^{2} (s-\sigma)} h_{u}(\epsilon k_{2})^{2}}{2 \lvert k_{2} \rvert^{2}} e^{- \lvert k_{12} \rvert^{2}(s-\sigma)}d\sigma ds \nonumber  \\
& - \int_{0}^{t} e^{- \lvert k_{1} \rvert^{2}  f(\epsilon k_{1})(t-s)} \int_{0}^{s} : \hat{X}_{\sigma, b}^{\epsilon, i_{2}}(k_{1}) \hat{X}_{t,b}^{\epsilon, j_{0}}(k_{4}): k_{12}^{i_{3}} g(\epsilon k_{12}^{i_{3}}) k_{1}^{j_{1}} g(\epsilon k_{1}^{j_{1}}) \nonumber\\
& \hspace{5mm} \times \frac{ e^{- \lvert k_{2} \rvert^{2} f(\epsilon k_{2}) (s-\sigma)} h_{b}(\epsilon k_{2}) h_{u}(\epsilon k_{2})}{2 \lvert k_{2} \rvert^{2} f(\epsilon k_{2})} e^{- \lvert k_{12} \rvert^{2} f(\epsilon k_{12})(s-\sigma)}d\sigma ds \nonumber  \\
& +  \int_{0}^{t} e^{- \lvert k_{1} \rvert^{2} (t-s) } \int_{0}^{s} : \hat{\bar{X}}_{\sigma, b}^{\epsilon, i_{2}}(k_{1}) \hat{\bar{X}}_{t,b}^{\epsilon, j_{0}}(k_{4}): k_{12}^{i_{3}} i k_{1}^{j_{1}} i \nonumber\\
& \hspace{5mm} \times \frac{ e^{- \lvert k_{2} \rvert^{2} (s-\sigma)} h_{b}(\epsilon k_{2}) h_{u}(\epsilon k_{2})}{2 \lvert k_{2} \rvert^{2}} e^{- \lvert k_{12} \rvert^{2}(s-\sigma)}d\sigma ds \nonumber  \\
& - \int_{0}^{t} e^{- \lvert k_{1} \rvert^{2} f(\epsilon k_{1})(t-s)  } \int_{0}^{s} : \hat{X}_{\sigma, b}^{\epsilon, i_{2}}(k_{1}) \hat{X}_{t,b}^{\epsilon, j_{0}}(k_{4}): k_{12}^{i_{3}} g(\epsilon k_{12}^{i_{3}}) k_{1}^{j_{1}} g(\epsilon k_{1}^{j_{1}}) \nonumber\\
& \hspace{5mm} \times \frac{ e^{- \lvert k_{2} \rvert^{2} f(\epsilon k_{2}) (s-\sigma)} h_{u}(\epsilon k_{2})h_{b}(\epsilon k_{2})}{2 \lvert k_{2} \rvert^{2} f(\epsilon k_{2})} e^{- \lvert k_{12} \rvert^{2} f(\epsilon k_{12})(s-\sigma)}d\sigma ds \nonumber  \\
& +  \int_{0}^{t} e^{- \lvert k_{1} \rvert^{2} (t-s) } \int_{0}^{s} : \hat{\bar{X}}_{\sigma, b}^{\epsilon, i_{2}}(k_{1}) \hat{\bar{X}}_{t,b}^{\epsilon, j_{0}}(k_{4}): k_{12}^{i_{3}} i k_{1}^{j_{1}} i \nonumber\\
& \hspace{5mm} \times \frac{ e^{- \lvert k_{2} \rvert^{2} (s-\sigma)} h_{u}(\epsilon k_{2})h_{b}(\epsilon k_{2})}{2 \lvert k_{2} \rvert^{2}} e^{- \lvert k_{12} \rvert^{2}(s-\sigma)}d\sigma ds \nonumber  \\
&+ \int_{0}^{t} e^{- \lvert k_{1} \rvert^{2}  f(\epsilon k_{1})(t-s)} \int_{0}^{s} : \hat{X}_{\sigma, u}^{\epsilon, i_{2}}(k_{1}) \hat{X}_{t,b}^{\epsilon, j_{0}}(k_{4}): k_{12}^{i_{3}} g(\epsilon k_{12}^{i_{3}}) k_{1}^{j_{1}} g(\epsilon k_{1}^{j_{1}}) \nonumber\\
& \hspace{5mm} \times \frac{ e^{- \lvert k_{2} \rvert^{2} f(\epsilon k_{2}) (s-\sigma)} h_{b}(\epsilon k_{2})^{2}}{2 \lvert k_{2} \rvert^{2} f(\epsilon k_{2})} e^{- \lvert k_{12} \rvert^{2}f(\epsilon k_{12})(s-\sigma) }d\sigma ds \nonumber  \\
& -  \int_{0}^{t} e^{- \lvert k_{1} \rvert^{2} (t-s) } \int_{0}^{s} : \hat{\bar{X}}_{\sigma, u}^{\epsilon, i_{2}}(k_{1}) \hat{\bar{X}}_{t,b}^{\epsilon, j_{0}}(k_{4}): k_{12}^{i_{3}} i k_{1}^{j_{1}} i \nonumber\\
& \hspace{5mm} \times \frac{ e^{- \lvert k_{2} \rvert^{2} (s-\sigma)} h_{b}(\epsilon k_{2})^{2}}{2 \lvert k_{2} \rvert^{2}} e^{- \lvert k_{12} \rvert^{2}(s-\sigma)}d\sigma ds] \nonumber  \\
& \times \sum_{i_{4} =1}^{3} \hat{\mathcal{P}}^{i_{3} i_{4}}(k_{2}) \hat{\mathcal{P}}^{j_{1} i_{4}}(k_{2}) \hat{\mathcal{P}}^{i_{1} i_{2}}(k_{12}) \hat{\mathcal{P}}^{i_{0} i_{1}}(k_{1}) e_{k}, \nonumber
\end{align} 
and $L_{t, i_{0}j_{0}}^{6}$ corresponds to the terms from $l_{k_{1}k_{2}k_{3}k_{4}, i_{2}i_{3}j_{0}j_{1}, s \sigma t}^{m4}$ for $m\in \{1, \hdots, 8 \}$ so that 
\begin{align}\label{estimate 199}
L_{t, i_{0}j_{0}}^{6} \triangleq&  (2\pi)^{-\frac{9}{2}} \sum_{\lvert i -j \rvert \leq 1} \sum_{i_{1}, i_{2}, i_{3}, j_{1} =1}^{3} \sum_{k} \sum_{k_{1}, k_{2}, k_{4} \neq 0: k_{24} = k} \theta(2^{-i} k_{2}) \theta(2^{-j} k_{4}) \\
& \times [\int_{0}^{t} e^{- \lvert k_{2} \rvert^{2}  f(\epsilon k_{2})(t-s)} \int_{0}^{s} : \hat{X}_{\sigma, u}^{\epsilon, i_{3}}(k_{2}) \hat{X}_{t,b}^{\epsilon, j_{0}} (k_{4}): k_{12}^{i_{3}} g(\epsilon k_{12}^{i_{3}}) k_{2}^{j_{1}} g(\epsilon k_{2}^{j_{1}}) \nonumber\\
& \hspace{5mm} \times \frac{ e^{ - \lvert k_{1} \rvert^{2} f(\epsilon k_{1}) (s-\sigma)} h_{u}(\epsilon k_{1})^{2}}{2 \lvert k_{1} \rvert^{2} f(\epsilon k_{1})} e^{- \lvert k_{12} \rvert^{2}  f(\epsilon k_{12})(s-\sigma)} d\sigma ds \nonumber \\
& - \int_{0}^{t} e^{- \lvert k_{2} \rvert^{2} (t-s)} \int_{0}^{s} : \hat{\bar{X}}_{\sigma, u}^{\epsilon, i_{3}}(k_{2}) \hat{\bar{X}}_{t,b}^{\epsilon, j_{0}} (k_{4}): k_{12}^{i_{3}} i k_{2}^{j_{1}} i \nonumber\\
& \hspace{5mm} \times \frac{ e^{ - \lvert k_{1} \rvert^{2} (s-\sigma)} h_{u}(\epsilon k_{1})^{2}}{2 \lvert k_{1} \rvert^{2} } e^{- \lvert k_{12} \rvert^{2} (s-\sigma)} d\sigma ds \nonumber \\
& - \int_{0}^{t} e^{- \lvert k_{2} \rvert^{2}  f(\epsilon k_{2})(t-s)} \int_{0}^{s} : \hat{X}_{\sigma, b}^{\epsilon, i_{3}}(k_{2}) \hat{X}_{t,b}^{\epsilon, j_{0}} (k_{4}): k_{12}^{i_{3}} g(\epsilon k_{12}^{i_{3}}) k_{2}^{j_{1}} g(\epsilon k_{2}^{j_{1}}) \nonumber\\
& \hspace{5mm} \times \frac{ e^{ - \lvert k_{1} \rvert^{2} f(\epsilon k_{1}) (s-\sigma)} h_{b}(\epsilon k_{1}) h_{u}(\epsilon k_{1})}{2 \lvert k_{1} \rvert^{2} f(\epsilon k_{1})} e^{- \lvert k_{12} \rvert^{2} f(\epsilon k_{12})(s-\sigma)} d\sigma ds \nonumber \\
& + \int_{0}^{t} e^{- \lvert k_{2} \rvert^{2} (t-s)} \int_{0}^{s} : \hat{\bar{X}}_{\sigma, b}^{\epsilon, i_{3}}(k_{2}) \hat{\bar{X}}_{t,b}^{\epsilon, j_{0}} (k_{4}): k_{12}^{i_{3}} i k_{2}^{j_{1}} i \nonumber\\
& \hspace{5mm} \times \frac{ e^{ - \lvert k_{1} \rvert^{2} (s-\sigma)} h_{b}(\epsilon k_{1}) h_{u}(\epsilon k_{1})}{2 \lvert k_{1} \rvert^{2} } e^{- \lvert k_{12} \rvert^{2} (s-\sigma)} d\sigma ds \nonumber \\
&- \int_{0}^{t} e^{- \lvert k_{2} \rvert^{2}  f(\epsilon k_{2})(t-s)} \int_{0}^{s} : \hat{X}_{\sigma, u}^{\epsilon, i_{3}}(k_{2}) \hat{X}_{t,b}^{\epsilon, j_{0}} (k_{4}): k_{12}^{i_{3}} g(\epsilon k_{12}^{i_{3}}) k_{2}^{j_{1}} g(\epsilon k_{2}^{j_{1}}) \nonumber\\
& \hspace{5mm} \times\frac{ e^{ - \lvert k_{1} \rvert^{2} f(\epsilon k_{1}) (s-\sigma)} h_{b}(\epsilon k_{1})^{2}}{2 \lvert k_{1} \rvert^{2} f(\epsilon k_{1})} e^{- \lvert k_{12} \rvert^{2} f(\epsilon k_{12})(s-\sigma) } d\sigma ds \nonumber \\
& + \int_{0}^{t} e^{- \lvert k_{2} \rvert^{2} (t-s)} \int_{0}^{s} : \hat{\bar{X}}_{\sigma, u}^{\epsilon, i_{3}}(k_{2}) \hat{\bar{X}}_{t,b}^{\epsilon, j_{0}} (k_{4}): k_{12}^{i_{3}} i k_{2}^{j_{1}} i \nonumber\\
& \hspace{5mm} \times\frac{ e^{ - \lvert k_{1} \rvert^{2} (s-\sigma)} h_{b}(\epsilon k_{1})^{2}}{2 \lvert k_{1} \rvert^{2} } e^{- \lvert k_{12} \rvert^{2} (s-\sigma)} d\sigma ds \nonumber \\
&+ \int_{0}^{t} e^{- \lvert k_{2} \rvert^{2}  f(\epsilon k_{2})(t-s)} \int_{0}^{s} : \hat{X}_{\sigma, b}^{\epsilon, i_{3}}(k_{2}) \hat{X}_{t,b}^{\epsilon, j_{0}} (k_{4}): k_{12}^{i_{3}} g(\epsilon k_{12}^{i_{3}}) k_{2}^{j_{1}} g(\epsilon k_{2}^{j_{1}}) \nonumber\\
& \hspace{5mm} \times \frac{ e^{ - \lvert k_{1} \rvert^{2} f(\epsilon k_{1}) (s-\sigma)} h_{u}(\epsilon k_{1})h_{b}(\epsilon k_{1})}{2 \lvert k_{1} \rvert^{2} f(\epsilon k_{1})} e^{- \lvert k_{12} \rvert^{2}  f(\epsilon k_{12})(s-\sigma)} d\sigma ds \nonumber \\
& - \int_{0}^{t} e^{- \lvert k_{2} \rvert^{2} (t-s)} \int_{0}^{s} : \hat{\bar{X}}_{\sigma, b}^{\epsilon, i_{3}}(k_{2}) \hat{\bar{X}}_{t,b}^{\epsilon, j_{0}} (k_{4}): k_{12}^{i_{3}} i k_{2}^{j_{1}} i \nonumber\\
& \hspace{5mm} \times \frac{ e^{ - \lvert k_{1} \rvert^{2} (s-\sigma)} h_{u}(\epsilon k_{1})h_{b}(\epsilon k_{1})}{2 \lvert k_{1} \rvert^{2} } e^{- \lvert k_{12} \rvert^{2} (s-\sigma)} d\sigma ds] \nonumber \\
& \times \sum_{i_{4} =1}^{3} \hat{\mathcal{P}}^{i_{2} i_{4}} (k_{1}) \hat{\mathcal{P}}^{j_{1} i_{4}}(k_{1}) \hat{\mathcal{P}}^{i_{1} i_{2}}(k_{12}) \hat{\mathcal{P}}^{i_{0} i_{1}} (k_{2}) e_{k}. \nonumber
\end{align} 
Finally, we take the terms corresponding to $l_{k_{1}k_{2}k_{3}k_{4}, i_{2}i_{3}j_{0}j_{1}, s \sigma t}^{m7}$ for $m\in \{1, \hdots, 8 \}$ and define it as 
\begin{align}\label{estimate 188}
L_{t, i_{0}j_{0}}^{71} \triangleq& (2\pi)^{-6}\sum_{\lvert i-j \rvert \leq 1} \sum_{i_{1}, i_{2}, i_{3}, j_{1} =1}^{3} \sum_{k_{1}, k_{2} \neq 0} \theta(2^{-i} k_{2}) \theta(2^{-j} k_{2})  \\
& \times [  \int_{0}^{t} e^{- \lvert k_{2} \rvert^{2}  f(\epsilon k_{2})(t-s)} \int_{0}^{s} \frac{ h_{u}(\epsilon k_{2}) h_{b}(\epsilon k_{2}) h_{u}(\epsilon k_{1})^{2}}{4 \lvert k_{1} \rvert^{2} \lvert k_{2} \rvert^{2} f(\epsilon k_{1}) f(\epsilon k_{2})} k_{12}^{i_{3}} g(\epsilon k_{12}^{i_{3}}) k_{2}^{j_{1}} g(\epsilon k_{2}^{j_{1}})  \nonumber\\
& \hspace{5mm}  \times e^{- \lvert k_{12} \rvert^{2}  f(\epsilon k_{12})(s-\sigma) - \lvert k_{1} \rvert^{2} f(\epsilon k_{1}) (s-\sigma) - \lvert k_{2} \rvert^{2} f(\epsilon k_{2}) (t-\sigma)}  d\sigma ds \nonumber \\
& - \int_{0}^{t} e^{- \lvert k_{2} \rvert^{2} (t-s)} \int_{0}^{s} \frac{ h_{u}(\epsilon k_{2}) h_{b}(\epsilon k_{2}) h_{u}(\epsilon k_{1})^{2}}{4 \lvert k_{1} \rvert^{2} \lvert k_{2} \rvert^{2}} k_{12}^{i_{3}} i k_{2}^{j_{1}} i  \nonumber\\
& \hspace{5mm}  \times e^{- \lvert k_{12} \rvert^{2} (s-\sigma) - \lvert k_{1} \rvert^{2} (s-\sigma) - \lvert k_{2} \rvert^{2} (t-\sigma)}  d\sigma ds \nonumber \\
& -  \int_{0}^{t} e^{- \lvert k_{2} \rvert^{2}  f(\epsilon k_{2})(t-s)} \int_{0}^{s} \frac{ h_{b}(\epsilon k_{2})^{2} h_{b}(\epsilon k_{1}) h_{u}(\epsilon k_{1})}{4 \lvert k_{1} \rvert^{2} \lvert k_{2} \rvert^{2} f(\epsilon k_{1}) f(\epsilon k_{2})} k_{12}^{i_{3}} g(\epsilon k_{12}^{i_{3}}) k_{2}^{j_{1}} g(\epsilon k_{2}^{j_{1}})  \nonumber\\
& \hspace{5mm}  \times e^{- \lvert k_{12} \rvert^{2}  f(\epsilon k_{12})(s-\sigma) - \lvert k_{1} \rvert^{2} f(\epsilon k_{1}) (s-\sigma) - \lvert k_{2} \rvert^{2} f(\epsilon k_{2}) (t-\sigma)} d\sigma ds \nonumber \\
& + \int_{0}^{t} e^{- \lvert k_{2} \rvert^{2} (t-s)} \int_{0}^{s} \frac{ h_{b}(\epsilon k_{2})^{2} h_{b}(\epsilon k_{1}) h_{u}(\epsilon k_{1})}{4 \lvert k_{1} \rvert^{2} \lvert k_{2} \rvert^{2}} k_{12}^{i_{3}} i k_{2}^{j_{1}} i  \nonumber\\
& \hspace{5mm}  \times e^{- \lvert k_{12} \rvert^{2} (s-\sigma) - \lvert k_{1} \rvert^{2} (s-\sigma) - \lvert k_{2} \rvert^{2} (t-\sigma)}  d\sigma ds \nonumber \\
& -  \int_{0}^{t} e^{- \lvert k_{2} \rvert^{2}  f(\epsilon k_{2})(t-s)} \int_{0}^{s} \frac{ h_{b}(\epsilon k_{2})^{2} h_{u}(\epsilon k_{1}) h_{b}(\epsilon k_{1})}{4 \lvert k_{1} \rvert^{2} \lvert k_{2} \rvert^{2} f(\epsilon k_{1}) f(\epsilon k_{2})} k_{12}^{i_{3}} g(\epsilon k_{12}^{i_{3}}) k_{2}^{j_{1}} g(\epsilon k_{2}^{j_{1}})  \nonumber\\
& \hspace{5mm}  \times e^{- \lvert k_{12} \rvert^{2}  f(\epsilon k_{12})(s-\sigma) - \lvert k_{1} \rvert^{2} f(\epsilon k_{1}) (s-\sigma) - \lvert k_{2} \rvert^{2} f(\epsilon k_{2}) (t-\sigma)}  d\sigma ds \nonumber \\
& + \int_{0}^{t} e^{- \lvert k_{2} \rvert^{2} (t-s)} \int_{0}^{s} \frac{ h_{b}(\epsilon k_{2})^{2} h_{u}(\epsilon k_{1}) h_{b}(\epsilon k_{1})}{4 \lvert k_{1} \rvert^{2} \lvert k_{2} \rvert^{2}} k_{12}^{i_{3}} i k_{2}^{j_{1}} i  \nonumber\\
& \hspace{5mm}  \times e^{- \lvert k_{12} \rvert^{2} (s-\sigma) - \lvert k_{1} \rvert^{2} (s-\sigma) - \lvert k_{2} \rvert^{2} (t-\sigma)}  d\sigma ds \nonumber \\
&+  \int_{0}^{t} e^{- \lvert k_{2} \rvert^{2}  f(\epsilon k_{2})(t-s)} \int_{0}^{s} \frac{ h_{u}(\epsilon k_{2}) h_{b}(\epsilon k_{2}) h_{b}(\epsilon k_{1})^{2}}{4 \lvert k_{1} \rvert^{2} \lvert k_{2} \rvert^{2} f(\epsilon k_{1}) f(\epsilon k_{2})} k_{12}^{i_{3}} g(\epsilon k_{12}^{i_{3}}) k_{2}^{j_{1}} g(\epsilon k_{2}^{j_{1}})  \nonumber\\
& \hspace{5mm}  \times e^{- \lvert k_{12} \rvert^{2}  f(\epsilon k_{12})(s-\sigma) - \lvert k_{1} \rvert^{2} f(\epsilon k_{1}) (s-\sigma) - \lvert k_{2} \rvert^{2} f(\epsilon k_{2}) (t-\sigma)}  d\sigma ds \nonumber \\
& - \int_{0}^{t} e^{- \lvert k_{2} \rvert^{2} (t-s)} \int_{0}^{s} \frac{ h_{u}(\epsilon k_{2}) h_{b}(\epsilon k_{2}) h_{b}(\epsilon k_{1})^{2}}{4 \lvert k_{1} \rvert^{2} \lvert k_{2} \rvert^{2}} k_{12}^{i_{3}} i k_{2}^{j_{1}} i  \nonumber\\
& \hspace{5mm}  \times e^{- \lvert k_{12} \rvert^{2} (s-\sigma) - \lvert k_{1} \rvert^{2} (s-\sigma) - \lvert k_{2} \rvert^{2} (t-\sigma)}  d\sigma ds] \nonumber \\
& \times \sum_{i_{4}, i_{5} =1}^{3} \hat{\mathcal{P}}^{i_{3}i_{4}}(k_{1}) \hat{\mathcal{P}}^{j_{1}i_{4}}(k_{1}) \hat{\mathcal{P}}^{i_{2}i_{5}}(k_{2}) \hat{\mathcal{P}}^{j_{0} i_{5}}(k_{2}) \hat{\mathcal{P}}^{i_{1}i_{2}}(k_{12}) \hat{\mathcal{P}}^{i_{0}i_{1}}(k_{2}) \nonumber 
\end{align}
where we strategically swapped $k_{1}$ with $k_{2}$ and $i_{4}$ with $i_{5}$. The motivation for this swap will be clear in Remark \ref{Remark 2.3}. Similarly, we take terms corresponding to $l_{k_{1}k_{2}k_{3}k_{4}, i_{2}i_{3}j_{0}j_{1}, s \sigma t}^{m5}$ for $m\in \{1, \hdots, 8 \}$ and define 
\begin{align}\label{It72}
L_{t, i_{0}j_{0}}^{72} \triangleq &(2\pi)^{-6} \sum_{\lvert i-j \rvert \leq 1} \sum_{k_{1}, k_{2} \neq 0} \sum_{i_{1}, i_{2}, i_{3}, j_{1} =1}^{3} \theta(2^{-i} k_{2}) \theta(2^{-j} k_{2})  \\
& \times [\int_{0}^{t} e^{- \lvert k_{2} \rvert^{2} f(\epsilon k_{2})(t-s)} \int_{0}^{s} \frac{ h_{u}(\epsilon k_{1})^{2} h_{u}(\epsilon k_{2}) h_{b}(\epsilon k_{2})}{4 \lvert k_{1} \rvert^{2} \lvert k_{2} \rvert^{2} f(\epsilon k_{1}) f(\epsilon k_{2})} k_{12}^{i_{3}} g(\epsilon k_{12}^{i_{3}}) k_{2}^{j_{1}} g(\epsilon k_{2}^{j_{1}})  \nonumber\\
& \hspace{5mm} \times e^{- \lvert k_{12} \rvert^{2} f(\epsilon k_{12})(s-\sigma) - \lvert k_{1} \rvert^{2} f(\epsilon k_{1}) (s-\sigma) - \lvert k_{2} \rvert^{2} f(\epsilon k_{2}) (t-\sigma)}  d\sigma d s  \nonumber \\
& - \int_{0}^{t} e^{- \lvert k_{2} \rvert^{2}(t-s) } \int_{0}^{s} \frac{ h_{u}(\epsilon k_{1})^{2} h_{u}(\epsilon k_{2}) h_{b}(\epsilon k_{2})}{4 \lvert k_{1} \rvert^{2} \lvert k_{2} \rvert^{2}} k_{12}^{i_{3}} i k_{2}^{j_{1}} i  \nonumber\\
& \hspace{5mm} \times e^{- \lvert k_{12} \rvert^{2} (s-\sigma)  - \lvert k_{1} \rvert^{2} (s-\sigma) - \lvert k_{2} \rvert^{2} (t-\sigma)}  d\sigma d s  \nonumber \\
& - \int_{0}^{t} e^{- \lvert k_{2} \rvert^{2} f(\epsilon k_{2})(t-s)} \int_{0}^{s} \frac{ h_{b}(\epsilon k_{1}) h_{u}(\epsilon k_{1}) h_{b}(\epsilon k_{2})^{2}}{4 \lvert k_{1} \rvert^{2} \lvert k_{2} \rvert^{2} f(\epsilon k_{1}) f(\epsilon k_{2})} k_{12}^{i_{3}} g(\epsilon k_{12}^{i_{3}}) k_{2}^{j_{1}} g(\epsilon k_{2}^{j_{1}})  \nonumber\\
& \hspace{5mm} \times e^{- \lvert k_{12} \rvert^{2}  f(\epsilon k_{12})(s-\sigma) - \lvert k_{1} \rvert^{2} f(\epsilon k_{1}) (s-\sigma) - \lvert k_{2} \rvert^{2} f(\epsilon k_{2}) (t-\sigma)}  d\sigma d s  \nonumber \\
& + \int_{0}^{t} e^{- \lvert k_{2} \rvert^{2}(t-s) } \int_{0}^{s} \frac{ h_{b}(\epsilon k_{1}) h_{u}(\epsilon k_{1}) h_{b}(\epsilon k_{2})^{2}}{4 \lvert k_{1} \rvert^{2} \lvert k_{2} \rvert^{2}} k_{12}^{i_{3}} i  k_{2}^{j_{1}} i  \nonumber\\
& \hspace{5mm} \times e^{- \lvert k_{12} \rvert^{2} (s-\sigma)  - \lvert k_{1} \rvert^{2} (s-\sigma) - \lvert k_{2} \rvert^{2} (t-\sigma)}  d\sigma d s  \nonumber \\
& - \int_{0}^{t} e^{- \lvert k_{2} \rvert^{2} f(\epsilon k_{2})(t-s)} \int_{0}^{s} \frac{ h_{b}(\epsilon k_{1})^{2} h_{u}(\epsilon k_{2}) h_{b}(\epsilon k_{2})}{4 \lvert k_{1} \rvert^{2} \lvert k_{2} \rvert^{2} f(\epsilon k_{1}) f(\epsilon k_{2})} k_{12}^{i_{3}} g(\epsilon k_{12}^{i_{3}}) k_{2}^{j_{1}} g(\epsilon k_{2}^{j_{1}})  \nonumber\\
& \hspace{5mm} \times e^{- \lvert k_{12} \rvert^{2}  f(\epsilon k_{12})(s-\sigma) - \lvert k_{1} \rvert^{2} f(\epsilon k_{1}) (s-\sigma) - \lvert k_{2} \rvert^{2} f(\epsilon k_{2}) (t-\sigma)}  d\sigma d s  \nonumber \\
& + \int_{0}^{t} e^{- \lvert k_{2} \rvert^{2}(t-s) } \int_{0}^{s} \frac{ h_{b}(\epsilon k_{1})^{2} h_{u}(\epsilon k_{2}) h_{b}(\epsilon k_{2})}{4 \lvert k_{1} \rvert^{2} \lvert k_{2} \rvert^{2}} k_{12}^{i_{3}} i k_{2}^{j_{1}} i  \nonumber\\
& \hspace{5mm} \times e^{- \lvert k_{12} \rvert^{2} (s-\sigma)  - \lvert k_{1} \rvert^{2} (s-\sigma) - \lvert k_{2} \rvert^{2} (t-\sigma)}  d\sigma d s  \nonumber \\
& + \int_{0}^{t} e^{- \lvert k_{2} \rvert^{2} f(\epsilon k_{2})(t-s)} \int_{0}^{s} \frac{ h_{u}(\epsilon k_{1}) h_{b}(\epsilon k_{1}) h_{b}(\epsilon k_{2})^{2}}{4 \lvert k_{1} \rvert^{2} \lvert k_{2} \rvert^{2} f(\epsilon k_{1}) f(\epsilon k_{2})} k_{12}^{i_{3}} g(\epsilon k_{12}^{i_{3}}) k_{2}^{j_{1}} g(\epsilon k_{2}^{j_{1}})  \nonumber\\
& \hspace{5mm} \times e^{- \lvert k_{12} \rvert^{2}  f(\epsilon k_{12})(s-\sigma) - \lvert k_{1} \rvert^{2} f(\epsilon k_{1}) (s-\sigma) - \lvert k_{2} \rvert^{2} f(\epsilon k_{2}) (t-\sigma)}  d\sigma d s  \nonumber \\
& - \int_{0}^{t} e^{- \lvert k_{2} \rvert^{2}(t-s) } \int_{0}^{s} \frac{ h_{u}(\epsilon k_{1}) h_{b}(\epsilon k_{1}) h_{b}(\epsilon k_{2})^{2}}{4 \lvert k_{1} \rvert^{2} \lvert k_{2} \rvert^{2}} k_{12}^{i_{3}} i  k_{2}^{j_{1}} i  \nonumber\\
& \hspace{5mm} \times e^{- \lvert k_{12} \rvert^{2} (s-\sigma)  - \lvert k_{1} \rvert^{2} (s-\sigma) - \lvert k_{2} \rvert^{2} (t-\sigma)}  d\sigma d s]  \nonumber \\
& \times \sum_{i_{4}, i_{5} =1}^{3} \hat{\mathcal{P}}^{i_{2} i_{4}} (k_{1}) \hat{\mathcal{P}}^{j_{1}i_{4}} (k_{1}) \hat{\mathcal{P}}^{i_{3} i_{5}} (k_{2}) \hat{\mathcal{P}}^{j_{0} i_{5}} (k_{2}) \hat{\mathcal{P}}^{i_{1} i_{2}} (k_{12}) \hat{\mathcal{P}}^{i_{0} i_{1}} (k_{2}). \nonumber
\end{align}

\begin{remark}\label{Remark 2.3}
We observe the cancellations between the third and seventh, and fourth and eighth lines of \eqref{It72}; however, due to the lack of such cancellations in $L_{t, i_{0}j_{0}}^{71}$ in \eqref{estimate 188}, we intentionally choose not to take advantage of these cancellations and instead work as follows. We add the first four terms of $L_{t, i_{0}j_{0}}^{71}$ in \eqref{estimate 188} to the and first four terms of $L_{t, i_{0}j_{0}}^{72}$ in \eqref{It72} to deduce 
\begin{align}\label{estimate 190}
L_{t, i_{0}j_{0}}^{7,\star} \triangleq& (2\pi)^{-6} \sum_{ \lvert i-j \rvert \leq 1} \sum_{k_{1}, k_{2} \neq 0} \sum_{i_{1}, i_{2}, i_{3}, j_{1} =1}^{3} \theta(2^{-i} k_{2}) \theta(2^{-j} k_{2}) \\
& \times [ \int_{0}^{t} e^{- \lvert k_{2} \rvert^{2}  f(\epsilon k_{2})(t-s)} \int_{0}^{s} \frac{ h_{u}(\epsilon k_{2}) h_{b}(\epsilon k_{2}) h_{u}(\epsilon k_{1})^{2}}{4 \lvert k_{1} \rvert^{2} \lvert k_{2} \rvert^{2} f(\epsilon k_{1}) f(\epsilon k_{2})}  k_{12}^{i_{3}} g(\epsilon k_{12}^{i_{3}}) k_{2}^{j_{1}} g(\epsilon k_{2}^{j_{1}}) \nonumber\\
& \hspace{5mm} \times e^{- \lvert k_{12} \rvert^{2} f(\epsilon k_{12})(s-\sigma) - \lvert k_{1} \rvert^{2} f(\epsilon k_{1}) (s-\sigma) - \lvert k_{2} \rvert^{2} f(\epsilon k_{2}) (t-\sigma)}d\sigma ds  \nonumber \\
& - \int_{0}^{t} e^{- \lvert k_{2} \rvert^{2} (t-s) } \int_{0}^{s} \frac{ h_{u}(\epsilon k_{2}) h_{b}(\epsilon k_{2}) h_{u}(\epsilon k_{1})^{2}}{4 \lvert k_{1} \rvert^{2} \lvert k_{2} \rvert^{2} }  k_{12}^{i_{3}} i k_{2}^{j_{1}} i \nonumber\\
& \hspace{5mm} \times e^{- \lvert k_{12} \rvert^{2} (s-\sigma) - \lvert k_{1} \rvert^{2} (s-\sigma) - \lvert k_{2} \rvert^{2} (t-\sigma)}d\sigma ds  \nonumber \\
& - \int_{0}^{t} e^{- \lvert k_{2} \rvert^{2} f(\epsilon k_{2})(t-s) } \int_{0}^{s} \frac{ h_{b}(\epsilon k_{2})^{2} h_{b}(\epsilon k_{1}) h_{u}(\epsilon k_{1})}{4 \lvert k_{1} \rvert^{2} \lvert k_{2} \rvert^{2} f(\epsilon k_{1}) f(\epsilon k_{2})}  k_{12}^{i_{3}} g(\epsilon k_{12}^{i_{3}}) k_{2}^{j_{1}} g(\epsilon k_{2}^{j_{1}}) \nonumber\\
& \hspace{5mm} \times e^{- \lvert k_{12} \rvert^{2} f(\epsilon k_{12})(s-\sigma) - \lvert k_{1} \rvert^{2} f(\epsilon k_{1}) (s-\sigma) - \lvert k_{2} \rvert^{2} f(\epsilon k_{2}) (t-\sigma)}d\sigma ds  \nonumber \\
& + \int_{0}^{t} e^{- \lvert k_{2} \rvert^{2} (t-s) } \int_{0}^{s} \frac{ h_{b}(\epsilon k_{2})^{2} h_{b}(\epsilon k_{1}) h_{u}(\epsilon k_{1})}{4 \lvert k_{1} \rvert^{2} \lvert k_{2} \rvert^{2} }  k_{12}^{i_{3}} i k_{2}^{j_{1}} i \nonumber\\
& \hspace{5mm} \times e^{- \lvert k_{12} \rvert^{2} (s-\sigma) - \lvert k_{1} \rvert^{2} (s-\sigma) - \lvert k_{2} \rvert^{2} (t-\sigma)}d\sigma ds]  \nonumber \\
& \times \sum_{i_{4}, i_{5} =1}^{3} \hat{\mathcal{P}}^{j_{1} i_{4}}(k_{1}) \hat{\mathcal{P}}^{j_{0} i_{5}}(k_{2}) \hat{\mathcal{P}}^{i_{1} i_{2}}(k_{12}) \hat{\mathcal{P}}^{i_{0} i_{1}} (k_{2}) \nonumber\\
& \hspace{5mm} \times [ \hat{\mathcal{P}}^{i_{2} i_{4}}(k_{1}) \hat{\mathcal{P}}^{i_{3} i_{5}} (k_{2}) + \hat{\mathcal{P}}^{i_{3} i_{4}}(k_{1}) \hat{\mathcal{P}}^{i_{2} i_{5}} (k_{2})] \nonumber. 
\end{align} 
For the fifth to eighth terms in $L_{t, i_{0}j_{0}}^{71}$ of \eqref{estimate 188} and $L_{t, i_{0}j_{0}}^{72}$ of \eqref{It72}, we see that adding them similarly won't work well as e.g., the fifth terms of $L_{t, i_{0}j_{0}}^{71}$  and the $L_{t, i_{0}j_{0}}^{72}$ have more differences than just Leray projection terms, specifically $h_{b}(\epsilon k_{2}) h_{u}(\epsilon k_{1}) \neq h_{b}(\epsilon k_{1}) h_{u}(\epsilon k_{2})$. This is a difference from the case of the NS equations as a result of the complex structure of the MHD system. Remarkably, we still realize that the following terms are identical, modulus signs: the fifth term of $L_{t, i_{0}j_{0}}^{71}$ and the seventh term of $L_{t, i_{0}j_{0}}^{72}$, the sixth term of $L_{t, i_{0}j_{0}}^{71}$ and the eighth term of $L_{t, i_{0}j_{0}}^{72}$, the seventh term of $L_{t, i_{0}j_{0}}^{71}$ and the fifth term of $L_{t, i_{0}j_{0}}^{72}$, the eighth term of $L_{t, i_{0}j_{0}}^{71}$ and the sixth term of $L_{t, i_{0}j_{0}}^{72}$. Thus, we can still combine the last four terms of  $L_{t, i_{0}j_{0}}^{71}$ and the last four terms of $L_{t, i_{0}j_{0}}^{72}$ to deduce 
\begin{align}\label{estimate 189}
L_{t, i_{0}j_{0}}^{7,\star\star} \triangleq& (2\pi)^{-6} \sum_{ \lvert i-j \rvert \leq 1} \sum_{k_{1}, k_{2} \neq 0} \sum_{i_{1}, i_{2}, i_{3}, j_{1} =1}^{3} \theta(2^{-i} k_{2}) \theta(2^{-j} k_{2}) \\
& \times [ \int_{0}^{t} e^{- \lvert k_{2} \rvert^{2} f(\epsilon k_{2}) (t-s)} \int_{0}^{s} \frac{ h_{b}(\epsilon k_{2})^{2} h_{u}(\epsilon k_{1}) h_{b}(\epsilon k_{1})}{4 \lvert k_{1} \rvert^{2} \lvert k_{2} \rvert^{2} f(\epsilon k_{1}) f(\epsilon k_{2})}  k_{12}^{i_{3}} g(\epsilon k_{12}^{i_{3}}) k_{2}^{j_{1}} g(\epsilon k_{2}^{j_{1}}) \nonumber\\
& \hspace{5mm} \times e^{- \lvert k_{12} \rvert^{2}f(\epsilon k_{12}) (s-\sigma) - \lvert k_{1} \rvert^{2} f(\epsilon k_{1}) (s-\sigma) - \lvert k_{2} \rvert^{2} f(\epsilon k_{2}) (t-\sigma)}d\sigma ds  \nonumber \\
& - \int_{0}^{t} e^{- \lvert k_{2} \rvert^{2} (t-s) } \int_{0}^{s} \frac{ h_{b}(\epsilon k_{2})^{2} h_{u}(\epsilon k_{1}) h_{b}(\epsilon k_{1})}{4 \lvert k_{1} \rvert^{2} \lvert k_{2} \rvert^{2} }  k_{12}^{i_{3}} i k_{2}^{j_{1}} i \nonumber\\
& \hspace{5mm} \times e^{- \lvert k_{12} \rvert^{2} (s-\sigma) - \lvert k_{1} \rvert^{2} (s-\sigma) - \lvert k_{2} \rvert^{2} (t-\sigma)}d\sigma ds  \nonumber \\
& - \int_{0}^{t} e^{- \lvert k_{2} \rvert^{2} f(\epsilon k_{2}) (t-s)} \int_{0}^{s} \frac{ h_{u}(\epsilon k_{2}) h_{b}(\epsilon k_{2}) h_{b}(\epsilon k_{1})^{2}}{4 \lvert k_{1} \rvert^{2} \lvert k_{2} \rvert^{2} f(\epsilon k_{1}) f(\epsilon k_{2})}  k_{12}^{i_{3}} g(\epsilon k_{12}^{i_{3}}) k_{2}^{j_{1}} g(\epsilon k_{2}^{j_{1}}) \nonumber\\
& \hspace{5mm} \times e^{- \lvert k_{12} \rvert^{2} f(\epsilon k_{12})(s-\sigma) - \lvert k_{1} \rvert^{2} f(\epsilon k_{1}) (s-\sigma) - \lvert k_{2} \rvert^{2} f(\epsilon k_{2}) (t-\sigma)}d\sigma ds  \nonumber \\
& + \int_{0}^{t} e^{- \lvert k_{2} \rvert^{2} (t-s) } \int_{0}^{s} \frac{ h_{u}(\epsilon k_{2}) h_{b}(\epsilon k_{2}) h_{b}(\epsilon k_{1})^{2}}{4 \lvert k_{1} \rvert^{2} \lvert k_{2} \rvert^{2} }  k_{12}^{i_{3}} i k_{2}^{j_{1}} i \nonumber\\
& \hspace{5mm} \times e^{- \lvert k_{12} \rvert^{2} (s-\sigma) - \lvert k_{1} \rvert^{2} (s-\sigma) - \lvert k_{2} \rvert^{2} (t-\sigma)}d\sigma ds]  \nonumber \\
& \times \sum_{i_{4}, i_{5} =1}^{3} \hat{\mathcal{P}}^{j_{1} i_{4}}(k_{1}) \hat{\mathcal{P}}^{j_{0} i_{5}}(k_{2}) \hat{\mathcal{P}}^{i_{1} i_{2}}(k_{12}) \hat{\mathcal{P}}^{i_{0} i_{1}} (k_{2}) \nonumber\\
& \hspace{5mm} \times [ \hat{\mathcal{P}}^{i_{2} i_{4}}(k_{1}) \hat{\mathcal{P}}^{i_{3} i_{5}} (k_{2}) - \hat{\mathcal{P}}^{i_{3} i_{4}}(k_{1}) \hat{\mathcal{P}}^{i_{2} i_{5}} (k_{2})] \nonumber. 
\end{align} 
We observe $[ \hat{\mathcal{P}}^{i_{2} i_{4}}(k_{1}) \hat{\mathcal{P}}^{i_{3} i_{5}} (k_{2}) - \hat{\mathcal{P}}^{i_{3} i_{4}}(k_{1}) \hat{\mathcal{P}}^{i_{2} i_{5}} (k_{2})]$ in \eqref{estimate 189} instead of $[ \hat{\mathcal{P}}^{i_{2} i_{4}}(k_{1})$  $\hat{\mathcal{P}}^{i_{3} i_{5}} (k_{2}) + \hat{\mathcal{P}}^{i_{3} i_{4}}(k_{1}) \hat{\mathcal{P}}^{i_{2} i_{5}} (k_{2})]$ in \eqref{estimate 190} due to the difference in signs that we mentioned. We define 
\begin{equation}\label{estimate 222}
L_{t, i_{0}j_{0}}^{7} \triangleq L_{t, i_{0}j_{0}}^{7,\star} + L_{t, i_{0}j_{0}}^{7,\star\star}.
\end{equation} 
\end{remark} 

\emph{Terms in the zeroth chaos: $L_{t, i_{0}j_{0}}^{7} = L_{t, i_{0}j_{0}}^{7,\star} + L_{t, i_{0}j_{0}}^{7,\star\star}$ from \eqref{estimate 190}-\eqref{estimate 189}}\\

Within $L_{t, i_{0}j_{0}}^{7,\star}$, we compute 
\begin{align}
&  \int_{0}^{t} e^{- \lvert k_{2} \rvert^{2}  f(\epsilon k_{2})(t-s)} \int_{0}^{s} \frac{ h_{u}(\epsilon k_{2}) h_{b}(\epsilon k_{2}) h_{u}(\epsilon k_{1})^{2}}{4 \lvert k_{1} \rvert^{2} \lvert k_{2} \rvert^{2} f(\epsilon k_{1}) f(\epsilon k_{2})}  k_{12}^{i_{3}} g(\epsilon k_{12}^{i_{3}}) k_{2}^{j_{1}} g(\epsilon k_{2}^{j_{1}}) \nonumber\\
& \hspace{5mm} \times e^{- \lvert k_{12} \rvert^{2} f(\epsilon k_{12})(s-\sigma) - \lvert k_{1} \rvert^{2} f(\epsilon k_{1}) (s-\sigma) - \lvert k_{2} \rvert^{2} f(\epsilon k_{2}) (t-\sigma)}d\sigma ds  \nonumber \\
& - \int_{0}^{t} e^{- \lvert k_{2} \rvert^{2} (t-s) } \int_{0}^{s} \frac{ h_{u}(\epsilon k_{2}) h_{b}(\epsilon k_{2}) h_{u}(\epsilon k_{1})^{2}}{4 \lvert k_{1} \rvert^{2} \lvert k_{2} \rvert^{2} }  k_{12}^{i_{3}} i k_{2}^{j_{1}} i \nonumber\\
& \hspace{5mm} \times e^{- \lvert k_{12} \rvert^{2} (s-\sigma) - \lvert k_{1} \rvert^{2} (s-\sigma) - \lvert k_{2} \rvert^{2} (t-\sigma)}d\sigma ds  \nonumber \\
& - \int_{0}^{t} e^{- \lvert k_{2} \rvert^{2} f(\epsilon k_{2})(t-s) } \int_{0}^{s} \frac{ h_{b}(\epsilon k_{2})^{2} h_{b}(\epsilon k_{1}) h_{u}(\epsilon k_{1})}{4 \lvert k_{1} \rvert^{2} \lvert k_{2} \rvert^{2} f(\epsilon k_{1}) f(\epsilon k_{2})}  k_{12}^{i_{3}} g(\epsilon k_{12}^{i_{3}}) k_{2}^{j_{1}} g(\epsilon k_{2}^{j_{1}}) \nonumber\\
& \hspace{5mm} \times e^{- \lvert k_{12} \rvert^{2} f(\epsilon k_{12})(s-\sigma) - \lvert k_{1} \rvert^{2} f(\epsilon k_{1}) (s-\sigma) - \lvert k_{2} \rvert^{2} f(\epsilon k_{2}) (t-\sigma)}d\sigma ds  \nonumber \\
& + \int_{0}^{t} e^{- \lvert k_{2} \rvert^{2} (t-s) } \int_{0}^{s} \frac{ h_{b}(\epsilon k_{2})^{2} h_{b}(\epsilon k_{1}) h_{u}(\epsilon k_{1})}{4 \lvert k_{1} \rvert^{2} \lvert k_{2} \rvert^{2} }  k_{12}^{i_{3}} i k_{2}^{j_{1}} i \nonumber\\
& \hspace{5mm} \times e^{- \lvert k_{12} \rvert^{2} (s-\sigma) - \lvert k_{1} \rvert^{2} (s-\sigma) - \lvert k_{2} \rvert^{2} (t-\sigma)}d\sigma ds\nonumber\\
=& \frac{ h_{u}(\epsilon k_{2}) h_{b}(\epsilon k_{2}) h_{u}(\epsilon k_{1})^{2}}{ (\prod_{i=1}^{2} 2 \lvert k_{i} \rvert^{2} f(\epsilon k_{i})) [ \lvert k_{12} \rvert^{2} f(\epsilon k_{12}) + \lvert k_{1} \rvert^{2} f(\epsilon k_{1}) + \lvert k_{2} \rvert^{2} f(\epsilon k_{2})]} k_{12}^{i_{3}} g(\epsilon k_{12}^{i_{3}}) k_{2}^{j_{1}} g(\epsilon k_{2}^{j_{1}}) \nonumber \\
& \hspace{5mm} \times [ \frac{1- e^{-2 \lvert k_{2} \rvert^{2}  f(\epsilon k_{2})t}}{2 \lvert k_{2} \rvert^{2} f(\epsilon k_{2})} - \int_{0}^{t} e^{ - 2 \lvert k_{2} \rvert^{2} f(\epsilon k_{2}) (t-s) - ( \lvert k_{12} \rvert^{2} f(\epsilon k_{12}) + \sum_{i=1}^{2} \lvert k_{i} \rvert^{2} f(\epsilon k_{i}) ) s } ds]  \nonumber\\
 -& \frac{ h_{u}(\epsilon k_{2}) h_{b}(\epsilon k_{2}) h_{u}(\epsilon k_{1})^{2}}{ (\prod_{i=1}^{2} 2 \lvert k_{i} \rvert^{2} ) [ \lvert k_{12} \rvert^{2} + \lvert k_{1} \rvert^{2} + \lvert k_{2} \rvert^{2} ]} k_{12}^{i_{3}} i k_{2}^{j_{1}} i \nonumber \\
& \hspace{5mm} \times [ \frac{1- e^{-2 \lvert k_{2} \rvert^{2} t}}{2 \lvert k_{2} \rvert^{2}} - \int_{0}^{t} e^{ - 2 \lvert k_{2} \rvert^{2} (t-s) - ( \lvert k_{12} \rvert^{2} + \sum_{i=1}^{2} \lvert k_{i} \rvert^{2} ) s } ds]  \nonumber\\
 -&  \frac{ h_{b}(\epsilon k_{2})^{2} h_{b}(\epsilon k_{1}) h_{u}(\epsilon k_{1})}{ (\prod_{i=1}^{2} 2 \lvert k_{i} \rvert^{2} f(\epsilon k_{i})) [ \lvert k_{12} \rvert^{2} f(\epsilon k_{12}) + \lvert k_{1} \rvert^{2} f(\epsilon k_{1}) + \lvert k_{2} \rvert^{2} f(\epsilon k_{2})]} k_{12}^{i_{3}} g(\epsilon k_{12}^{i_{3}}) k_{2}^{j_{1}} g(\epsilon k_{2}^{j_{1}}) \nonumber \\
& \hspace{5mm} \times [ \frac{1- e^{-2 \lvert k_{2} \rvert^{2} f(\epsilon k_{2})t}}{2 \lvert k_{2} \rvert^{2} f(\epsilon k_{2})} - \int_{0}^{t} e^{ - 2 \lvert k_{2} \rvert^{2} f(\epsilon k_{2}) (t-s) - ( \lvert k_{12} \rvert^{2} f(\epsilon k_{12}) + \sum_{i=1}^{2} \lvert k_{i} \rvert^{2} f(\epsilon k_{i}) ) s } ds]  \nonumber\\
 +& \frac{ h_{b}(\epsilon k_{2})^{2} h_{b}(\epsilon k_{1}) h_{u}(\epsilon k_{1})}{ (\prod_{i=1}^{2} 2 \lvert k_{i} \rvert^{2} ) [ \lvert k_{12} \rvert^{2} + \lvert k_{1} \rvert^{2} + \lvert k_{2} \rvert^{2} ]} k_{12}^{i_{3}} i k_{2}^{j_{1}} i \nonumber \\
& \hspace{5mm} \times [ \frac{1- e^{-2 \lvert k_{2} \rvert^{2} t}}{2 \lvert k_{2} \rvert^{2}} - \int_{0}^{t} e^{ - 2 \lvert k_{2} \rvert^{2} (t-s) - ( \lvert k_{12} \rvert^{2} + \sum_{i=1}^{2} \lvert k_{i} \rvert^{2} ) s } ds].  \nonumber
\end{align} 
An identical computation gives an analogous result within $L_{t, i_{0}j_{0}}^{7, \star\star}$ as well. Thus, \eqref{estimate 190}-\eqref{estimate 189} leads us to 
\begin{equation}\label{estimate 273}
L_{t, i_{0}j_{0}}^{7,\star}= L_{t, i_{0}j_{0}}^{7, \star, 1} + L_{t, i_{0}j_{0}}^{7, \star, 2}  
\end{equation} 
where 
\begin{subequations}
\begin{align}
&L_{t, i_{0}j_{0}}^{7, \star, 1} \triangleq (2\pi)^{-6} \sum_{\lvert i-j\rvert \leq 1} \sum_{k_{1}, k_{2} \neq 0} \sum_{i_{1}, i_{2}, i_{3}, j_{1} =1}^{3} \theta(2^{-i} k_{2}) \theta(2^{-j} k_{2}) \hat{\mathcal{P}}^{i_{1} i_{2}}(k_{12}) \hat{\mathcal{P}}^{i_{0} i_{1}} (k_{2}) \nonumber\\
& \hspace{3mm} \times \sum_{i_{4}, i_{5} =1}^{3} \hat{\mathcal{P}}^{j_{1} i_{4}}(k_{1}) \hat{\mathcal{P}}^{j_{0} i_{5}}(k_{2})   [ \hat{\mathcal{P}}^{i_{2} i_{4}}(k_{1}) \hat{\mathcal{P}}^{i_{3} i_{5}} (k_{2}) + \hat{\mathcal{P}}^{i_{3} i_{4}}(k_{1}) \hat{\mathcal{P}}^{i_{2} i_{5}} (k_{2})]  \nonumber\\
& \hspace{3mm} \times [ \frac{ h_{u}(\epsilon k_{2}) h_{b}(\epsilon k_{2}) h_{u}(\epsilon k_{1})^{2}}{ (\prod_{i=1}^{2} 2 \lvert k_{i} \rvert^{2} f(\epsilon k_{i})) [ \lvert k_{12} \rvert^{2} f(\epsilon k_{12}) + \sum_{i=1}^{2} \lvert k_{i} \rvert^{2} f(\epsilon k_{i} ) ]} k_{12}^{i_{3}} g(\epsilon k_{12}^{i_{3}}) k_{2}^{j_{1}} g(\epsilon k_{2}^{j_{1}}) \nonumber \\
& \hspace{8mm} \times ( \frac{ 1 - e^{-2 \lvert k_{2} \rvert^{2}  f(\epsilon k_{2})t}}{2 \lvert k_{2} \rvert^{2} f(\epsilon k_{2})} - \int_{0}^{t} e^{-2 \lvert k_{2} \rvert^{2} f(\epsilon k_{2}) (t-s) - ( \lvert k_{12} \rvert^{2} f(\epsilon k_{12}) + \sum_{i=1}^{2} \lvert k_{i} \rvert^{2} f(\epsilon k_{i} ) ) s} ds) \nonumber\\
& \hspace{3mm} - \frac{ h_{u}(\epsilon k_{2}) h_{b}(\epsilon k_{2}) h_{u}(\epsilon k_{1})^{2}}{ (\prod_{i=1}^{2}2 \lvert k_{i} \rvert^{2} ) [ \lvert k_{12} \rvert^{2} + \sum_{i=1}^{2} \lvert k_{i} \rvert^{2} ]} k_{12}^{i_{3}} i k_{2}^{j_{1}} i \nonumber \\
& \hspace{8mm} \times ( \frac{ 1- e^{-2 \lvert k_{2} \rvert^{2} t}}{2 \lvert k_{2} \rvert^{2} } - \int_{0}^{t} e^{-2 \lvert k_{2} \rvert^{2} (t-s) - ( \lvert k_{12} \rvert^{2} + \sum_{i=1}^{2} \lvert k_{i} \rvert^{2} ) s} ds)], \\
&L_{t, i_{0}j_{0}}^{7, \star, 2} \triangleq (2\pi)^{-6} \sum_{\lvert i-j\rvert \leq 1} \sum_{k_{1}, k_{2} \neq 0} \sum_{i_{1}, i_{2}, i_{3}, j_{1} =1}^{3} \theta(2^{-i} k_{2}) \theta(2^{-j} k_{2}) \hat{\mathcal{P}}^{i_{1} i_{2}}(k_{12}) \hat{\mathcal{P}}^{i_{0} i_{1}} (k_{2}) \nonumber\\
& \hspace{3mm} \times \sum_{i_{4}, i_{5} =1}^{3} \hat{\mathcal{P}}^{j_{1} i_{4}}(k_{1}) \hat{\mathcal{P}}^{j_{0} i_{5}}(k_{2})   [ \hat{\mathcal{P}}^{i_{2} i_{4}}(k_{1}) \hat{\mathcal{P}}^{i_{3} i_{5}} (k_{2}) + \hat{\mathcal{P}}^{i_{3} i_{4}}(k_{1}) \hat{\mathcal{P}}^{i_{2} i_{5}} (k_{2})]  \nonumber\\
& \hspace{3mm} \times [ - \frac{ h_{b}(\epsilon k_{2})^{2} h_{b}(\epsilon k_{1}) h_{u}(\epsilon k_{1})}{ (\prod_{i=1}^{2} 2 \lvert k_{i} \rvert^{2} f(\epsilon k_{i})) [ \lvert k_{12} \rvert^{2} f(\epsilon k_{12}) + \sum_{i=1}^{2} \lvert k_{i} \rvert^{2} f(\epsilon k_{i} ) ]} k_{12}^{i_{3}} g(\epsilon k_{12}^{i_{3}}) k_{2}^{j_{1}} g(\epsilon k_{2}^{j_{1}}) \nonumber \\
& \hspace{8mm} \times ( \frac{ 1 - e^{-2 \lvert k_{2} \rvert^{2} f(\epsilon k_{2})t}}{2 \lvert k_{2} \rvert^{2} f(\epsilon k_{2})} - \int_{0}^{t} e^{-2 \lvert k_{2} \rvert^{2} f(\epsilon k_{2}) (t-s) - ( \lvert k_{12} \rvert^{2} f(\epsilon k_{12}) + \sum_{i=1}^{2} \lvert k_{i} \rvert^{2} f(\epsilon k_{i} ) ) s} ds) \nonumber\\
& \hspace{3mm}+ \frac{ h_{b}(\epsilon k_{2})^{2} h_{b}(\epsilon k_{1}) h_{u}(\epsilon k_{1})}{ (\prod_{i=1}^{2}2 \lvert k_{i} \rvert^{2} ) [ \lvert k_{12} \rvert^{2} + \sum_{i=1}^{2} \lvert k_{i} \rvert^{2} ]} k_{12}^{i_{3}} i k_{2}^{j_{1}} i \nonumber \\
& \hspace{8mm} \times ( \frac{ 1- e^{-2 \lvert k_{2} \rvert^{2} t}}{2 \lvert k_{2} \rvert^{2} } - \int_{0}^{t} e^{-2 \lvert k_{2} \rvert^{2} (t-s) - ( \lvert k_{12} \rvert^{2} + \sum_{i=1}^{2} \lvert k_{i} \rvert^{2} ) s} ds)], 
\end{align}
\end{subequations} 
while 
\begin{equation}\label{estimate 274}
L_{t, i_{0}j_{0}}^{7,\star\star}= L_{t, i_{0}j_{0}}^{7, \star\star, 1} + L_{t, i_{0}j_{0}}^{7, \star\star, 2}  
\end{equation} 
where 
\begin{subequations}
\begin{align}
&L_{t, i_{0}j_{0}}^{7, \star\star, 1} \triangleq (2\pi)^{-6} \sum_{\lvert i-j\rvert \leq 1} \sum_{k_{1}, k_{2} \neq 0} \sum_{i_{1}, i_{2}, i_{3}, j_{1} =1}^{3} \theta(2^{-i} k_{2}) \theta(2^{-j} k_{2}) \hat{\mathcal{P}}^{i_{1} i_{2}}(k_{12}) \hat{\mathcal{P}}^{i_{0} i_{1}} (k_{2}) \nonumber\\
&\hspace{3mm} \times \sum_{i_{4}, i_{5} =1}^{3} \hat{\mathcal{P}}^{j_{1} i_{4}}(k_{1}) \hat{\mathcal{P}}^{j_{0} i_{5}}(k_{2})   [ \hat{\mathcal{P}}^{i_{2} i_{4}}(k_{1}) \hat{\mathcal{P}}^{i_{3} i_{5}} (k_{2}) - \hat{\mathcal{P}}^{i_{3} i_{4}}(k_{1}) \hat{\mathcal{P}}^{i_{2} i_{5}} (k_{2})]  \nonumber\\
&\hspace{3mm} \times [ \frac{ h_{b}(\epsilon k_{2})^{2} h_{u}(\epsilon k_{1}) h_{b}(\epsilon k_{1})}{ (\prod_{i=1}^{2} 2 \lvert k_{i} \rvert^{2} f(\epsilon k_{i})) [ \lvert k_{12} \rvert^{2} f(\epsilon k_{12}) + \sum_{i=1}^{2} \lvert k_{i} \rvert^{2} f(\epsilon k_{i} ) ]} k_{12}^{i_{3}} g(\epsilon k_{12}^{i_{3}}) k_{2}^{j_{1}} g(\epsilon k_{2}^{j_{1}}) \nonumber \\
& \hspace{8mm} \times ( \frac{ 1 - e^{-2 \lvert k_{2} \rvert^{2} f(\epsilon k_{2})t}}{2 \lvert k_{2} \rvert^{2} f(\epsilon k_{2})} - \int_{0}^{t} e^{-2 \lvert k_{2} \rvert^{2} f(\epsilon k_{2}) (t-s) - ( \lvert k_{12} \rvert^{2} f(\epsilon k_{12}) + \sum_{i=1}^{2} \lvert k_{i} \rvert^{2} f(\epsilon k_{i} ) ) s} ds) \nonumber\\
& \hspace{3mm}- \frac{ h_{b}(\epsilon k_{2})^{2} h_{u}(\epsilon k_{1}) h_{b}(\epsilon k_{1})}{ (\prod_{i=1}^{2}2 \lvert k_{i} \rvert^{2} ) [ \lvert k_{12} \rvert^{2} + \sum_{i=1}^{2} \lvert k_{i} \rvert^{2} ]} k_{12}^{i_{3}} i k_{2}^{j_{1}} i \nonumber \\
& \hspace{8mm} \times ( \frac{ 1- e^{-2 \lvert k_{2} \rvert^{2} t}}{2 \lvert k_{2} \rvert^{2} } - \int_{0}^{t} e^{-2 \lvert k_{2} \rvert^{2} (t-s) - ( \lvert k_{12} \rvert^{2} + \sum_{i=1}^{2} \lvert k_{i} \rvert^{2} ) s} ds)], \\
&L_{t, i_{0}j_{0}}^{7, \star\star, 2} \triangleq (2\pi)^{-6} \sum_{\lvert i-j\rvert \leq 1} \sum_{k_{1}, k_{2} \neq 0} \sum_{i_{1}, i_{2}, i_{3}, j_{1} =1}^{3} \theta(2^{-i} k_{2}) \theta(2^{-j} k_{2}) \hat{\mathcal{P}}^{i_{1} i_{2}}(k_{12}) \hat{\mathcal{P}}^{i_{0} i_{1}} (k_{2}) \nonumber\\
&\hspace{3mm} \times \sum_{i_{4}, i_{5} =1}^{3} \hat{\mathcal{P}}^{j_{1} i_{4}}(k_{1}) \hat{\mathcal{P}}^{j_{0} i_{5}}(k_{2})   [ \hat{\mathcal{P}}^{i_{2} i_{4}}(k_{1}) \hat{\mathcal{P}}^{i_{3} i_{5}} (k_{2}) - \hat{\mathcal{P}}^{i_{3} i_{4}}(k_{1}) \hat{\mathcal{P}}^{i_{2} i_{5}} (k_{2})]  \nonumber\\
& \hspace{3mm}\times [ - \frac{ h_{u}(\epsilon k_{2}) h_{b}(\epsilon k_{2}) h_{b}(\epsilon k_{1})^{2}}{ (\prod_{i=1}^{2} 2 \lvert k_{i} \rvert^{2} f(\epsilon k_{i})) [ \lvert k_{12} \rvert^{2} f(\epsilon k_{12}) + \sum_{i=1}^{2} \lvert k_{i} \rvert^{2} f(\epsilon k_{i} ) ]} k_{12}^{i_{3}} g(\epsilon k_{12}^{i_{3}}) k_{2}^{j_{1}} g(\epsilon k_{2}^{j_{1}}) \nonumber \\
& \hspace{8mm} \times ( \frac{ 1 - e^{-2 \lvert k_{2} \rvert^{2}  f(\epsilon k_{2})t}}{2 \lvert k_{2} \rvert^{2} f(\epsilon k_{2})} - \int_{0}^{t} e^{-2 \lvert k_{2} \rvert^{2} f(\epsilon k_{2}) (t-s) - ( \lvert k_{12} \rvert^{2} f(\epsilon k_{12}) + \sum_{i=1}^{2} \lvert k_{i} \rvert^{2} f(\epsilon k_{i} ) ) s} ds) \nonumber\\
&\hspace{3mm} + \frac{ h_{u}(\epsilon k_{2}) h_{b}(\epsilon k_{2}) h_{b}(\epsilon k_{1})^{2}}{ (\prod_{i=1}^{2}2 \lvert k_{i} \rvert^{2} ) [ \lvert k_{12} \rvert^{2} + \sum_{i=1}^{2} \lvert k_{i} \rvert^{2} ]} k_{12}^{i_{3}} i k_{2}^{j_{1}} i \nonumber \\
& \hspace{8mm} \times ( \frac{ 1- e^{-2 \lvert k_{2} \rvert^{2} t}}{2 \lvert k_{2} \rvert^{2} } - \int_{0}^{t} e^{-2 \lvert k_{2} \rvert^{2} (t-s) - ( \lvert k_{12} \rvert^{2} + \sum_{i=1}^{2} \lvert k_{i} \rvert^{2} ) s} ds)]. 
\end{align}
\end{subequations} 
For $L_{t, i_{0}j_{0}}^{7, \star, 1}$ we define 
\begin{subequations} 
\begin{align}
C_{131}^{\epsilon, i_{0} j_{0}} (t) \triangleq& (2\pi)^{-6} \sum_{\lvert i-j \rvert \leq 1} \sum_{k_{1}, k_{2} \neq 0} \sum_{i_{1}, i_{2}, i_{3}, j_{1} =1}^{3} \theta(2^{-i} k_{2}) \theta(2^{-j} k_{2}) \hat{\mathcal{P}}^{i_{1} i_{2}} (k_{12}) \hat{\mathcal{P}}^{i_{0} i_{1}} (k_{2}) \nonumber \\
& \times \sum_{i_{4}, i_{5} =1}^{3} \hat{\mathcal{P}}^{j_{1} i_{4}} (k_{1}) \hat{\mathcal{P}}^{j_{0} i_{5}}(k_{2})[\hat{\mathcal{P}}^{i_{2} i_{4}}(k_{1}) \hat{\mathcal{P}}^{i_{3} i_{5}}(k_{2}) + \hat{\mathcal{P}}^{i_{3} i_{4}}(k_{1}) \hat{\mathcal{P}}^{i_{2} i_{5}} (k_{2})] \nonumber\\
& \times \frac{ h_{u}(\epsilon k_{2}) h_{b}(\epsilon k_{2}) h_{u}(\epsilon k_{1})^{2}}{ (\prod_{i=1}^{2} 2 \lvert k_{i} \rvert^{2} f(\epsilon k_{i})) [ \lvert k_{12} \rvert^{2} f(\epsilon k_{12}) + \sum_{i=1}^{2} \lvert k_{i} \rvert^{2} f(\epsilon k_{i}) ] }  \nonumber\\
& \times k_{12}^{i_{3}} g(\epsilon k_{12}^{i_{3}}) k_{2}^{j_{1}} g(\epsilon k_{2}^{j_{1}}) ( \frac{ 1 - e^{-2 \lvert k_{2} \rvert^{2} f(\epsilon k_{2})t}}{2 \lvert k_{2} \rvert^{2} f(\epsilon k_{2})} ),  \\
\bar{C}_{131}^{\epsilon, i_{0} j_{0}} (t)  \triangleq& (2\pi)^{-6} \sum_{\lvert i-j \rvert \leq 1} \sum_{k_{1}, k_{2} \neq 0} \sum_{i_{1}, i_{2}, i_{3}, j_{1} =1}^{3} \theta(2^{-i} k_{2}) \theta(2^{-j} k_{2}) \hat{\mathcal{P}}^{i_{1} i_{2}} (k_{12}) \hat{\mathcal{P}}^{i_{0} i_{1}} (k_{2}) \nonumber \\
& \times \sum_{i_{4}, i_{5} =1}^{3} \hat{\mathcal{P}}^{j_{1} i_{4}} (k_{1}) \hat{\mathcal{P}}^{j_{0} i_{5}}(k_{2})[\hat{\mathcal{P}}^{i_{2} i_{4}}(k_{1}) \hat{\mathcal{P}}^{i_{3} i_{5}}(k_{2}) +  \hat{\mathcal{P}}^{i_{3} i_{4}}(k_{1}) \hat{\mathcal{P}}^{i_{2} i_{5}} (k_{2})] \nonumber\\
& \times \frac{ h_{u}(\epsilon k_{2}) h_{b}(\epsilon k_{2}) h_{u}(\epsilon k_{1})^{2}}{ (\prod_{i=1}^{2} 2 \lvert k_{i} \rvert^{2} ) [ \lvert k_{12} \rvert^{2} + \sum_{i=1}^{2} \lvert k_{i} \rvert^{2} ] } k_{12}^{i_{3}} i k_{2}^{j_{1}} i ( \frac{ 1 - e^{-2 \lvert k_{2} \rvert^{2} t}}{2 \lvert k_{2} \rvert^{2}} ),  
\end{align} 
\end{subequations} 
\begin{subequations}\label{estimate 191}
\begin{align}
\phi_{131}^{ \epsilon, i_{0}j_{0}}(t) \triangleq& (2\pi)^{-6} \sum_{\lvert i-j \rvert \leq 1} \sum_{k_{1}, k_{2} \neq 0} \sum_{i_{1}, i_{2}, i_{3}, j_{1} =1}^{3} \theta(2^{-i} k_{2}) \theta(2^{-j} k_{2}) \hat{\mathcal{P}}^{i_{1} i_{2} } (k_{12}) \hat{\mathcal{P}}^{i_{0} i_{1}} (k_{2})\nonumber \\
& \times \sum_{i_{4}, i_{5} =1}^{3} \hat{\mathcal{P}}^{j_{1} i_{4}}(k_{1}) \hat{\mathcal{P}}^{j_{0} i_{5}} (k_{2}) [\hat{\mathcal{P}}^{i_{2} i_{4}}(k_{1}) \hat{\mathcal{P}}^{i_{3} i_{5}} (k_{2}) +  \hat{\mathcal{P}}^{i_{3} i_{4}}(k_{1}) \hat{\mathcal{P}}^{i_{2} i_{5}} (k_{2})] \nonumber \\
& \times [ - \frac{ h_{u} (\epsilon k_{2}) h_{b}(\epsilon k_{2}) h_{u}(\epsilon k_{1})^{2} k_{12}^{i_{3}} g(\epsilon k_{12}^{i_{3}}) k_{2}^{j_{1}} g(\epsilon k_{2}^{j_{1}})  }{  (\prod_{i=1}^{2} 2 \lvert k_{i} \rvert^{2} f(\epsilon k_{i})) [ \lvert k_{12} \rvert^{2} f(\epsilon k_{12}) + \sum_{i=1}^{2} \lvert k_{i} \rvert^{2} f(\epsilon k_{i})]} \nonumber \\
& \hspace{5mm} \times \int_{0}^{t} e^{-2 \lvert k_{2} \rvert^{2} f(\epsilon k_{2}) (t-s) - ( \lvert k_{12} \rvert^{2} f(\epsilon k_{12}) + \sum_{i=1}^{2} \lvert k_{i} \rvert^{2} f(\epsilon k_{i}) ) s} ds], \\
\bar{\phi}_{131}^{ \epsilon, i_{0}j_{0}}(t) \triangleq& (2\pi)^{-6} \sum_{\lvert i-j \rvert \leq 1} \sum_{k_{1}, k_{2} \neq 0} \sum_{i_{1}, i_{2}, i_{3}, j_{1} =1}^{3} \theta(2^{-i} k_{2}) \theta(2^{-j} k_{2}) \hat{\mathcal{P}}^{i_{1} i_{2} } (k_{12}) \hat{\mathcal{P}}^{i_{0} i_{1}} (k_{2})\nonumber \\
& \times \sum_{i_{4}, i_{5} =1}^{3} \hat{\mathcal{P}}^{j_{1} i_{4}}(k_{1}) \hat{\mathcal{P}}^{j_{0} i_{5}} (k_{2}) [\hat{\mathcal{P}}^{i_{2} i_{4}}(k_{1}) \hat{\mathcal{P}}^{i_{3} i_{5}} (k_{2}) + \hat{\mathcal{P}}^{i_{3} i_{4}}(k_{1}) \hat{\mathcal{P}}^{i_{2} i_{5}} (k_{2}) ] \nonumber \\
& \times [ - \frac{ h_{u} (\epsilon k_{2}) h_{b}(\epsilon k_{2}) h_{u}(\epsilon k_{1})^{2} k_{12}^{i_{3}} i k_{2}^{j_{1}} i  }{  (\prod_{i=1}^{2} 2 \lvert k_{i} \rvert^{2}) [ \lvert k_{12} \rvert^{2}  + \sum_{i=1}^{2} \lvert k_{i} \rvert^{2} ]} \nonumber \\
& \hspace{5mm} \times \int_{0}^{t} e^{-2 \lvert k_{2} \rvert^{2}  (t-s) - ( \lvert k_{12} \rvert^{2} + \sum_{i=1}^{2} \lvert k_{i} \rvert^{2} ) s} ds], 
\end{align} 
\end{subequations} 
so that 
\begin{equation}\label{estimate 275}
L_{t, i_{0}j_{0}}^{7,\star, 1} - \phi_{131}^{\epsilon, i_{0}j_{0}} + \bar{\phi}_{131}^{\epsilon, i_{0} j_{0}} - C_{131}^{\epsilon, i_{0} j_{0}} + \bar{C}_{131}^{\epsilon, i_{0} j_{0}} = 0.
\end{equation} 
We now estimate 
\begin{align}\label{estimate 193}
& \lvert \phi_{131}^{\epsilon, i_{0} j_{0}} (t) - \bar{\phi}_{131}^{\epsilon, i_{0} j_{0}}(t) \rvert \\
\lesssim& \sum_{k_{1}, k_{2} \neq 0} \sum_{i_{3}, j_{1} =1}^{3} \frac{ \lvert k_{12} \rvert \lvert k_{2} \rvert}{\lvert k_{1} \rvert^{2} \lvert k_{2} \rvert^{2}}  \lvert \int_{0}^{t} \frac{g(\epsilon k_{12}^{i_{3}}) g(\epsilon k_{2}^{j_{1}})}{ \prod_{i=1}^{2} f(\epsilon k_{i}) [ \lvert k_{12} \rvert^{2} f(\epsilon k_{12}) + \sum_{i=1}^{2} \lvert k_{i} \rvert^{2} f(\epsilon k_{i})]} \nonumber\\
& \hspace{30mm} \times  e^{-2 \lvert k_{2} \rvert^{2} f(\epsilon k_{2}) (t-s) - (\lvert k_{12} \rvert^{2} f(\epsilon k_{12}) + \sum_{i=1}^{2} \lvert k_{i} \rvert^{2} f(\epsilon k_{i}) ) s}  \nonumber\\
& \hspace{30mm}  - \frac{ ii}{[ \lvert k_{12} \rvert^{2} + \sum_{i=1}^{2} \lvert k_{i} \rvert^{2}]}  e^{-2 \lvert k_{2} \rvert^{2} (t-s) - (\lvert k_{12} \rvert^{2} + \sum_{i=1}^{2} \lvert k_{i} \rvert^{2} ) s} ds \rvert \nonumber 
\end{align} 
by \eqref{estimate 191} where
\begin{align}\label{estimate 192}
&\lvert \frac{g(\epsilon k_{12}^{i_{3}}) g(\epsilon k_{2}^{j_{1}})}{ \prod_{i=1}^{2} f(\epsilon k_{i}) [ \lvert k_{12} \rvert^{2} f(\epsilon k_{12}) + \sum_{i=1}^{2} \lvert k_{i} \rvert^{2} f(\epsilon k_{i})]} \nonumber\\
&  \times  e^{-2 \lvert k_{2} \rvert^{2} f(\epsilon k_{2}) (t-s) - (\lvert k_{12} \rvert^{2} f(\epsilon k_{12}) + \sum_{i=1}^{2} \lvert k_{i} \rvert^{2} f(\epsilon k_{i}) ) s}  \nonumber\\
&   - \frac{ ii}{[ \lvert k_{12} \rvert^{2} + \sum_{i=1}^{2} \lvert k_{i} \rvert^{2}]}  e^{-2 \lvert k_{2} \rvert^{2} (t-s) - (\lvert k_{12} \rvert^{2} + \sum_{i=1}^{2} \lvert k_{i} \rvert^{2} ) s}  \rvert \nonumber\\
\lesssim& \left( \frac{ \lvert \epsilon k_{1} \rvert^{\eta} + \lvert \epsilon k_{2} \rvert^{\eta} + \lvert \epsilon k_{12} \rvert^{\eta}}{\lvert k_{1} \rvert^{2} + \lvert k_{2} \rvert^{2} + \lvert k_{12} \rvert^{2}} \right) e^{-  \lvert k_{2} \rvert^{2}\bar{c}_{f} (t-s) - ( \lvert k_{12} \rvert^{2} + \sum_{i=1}^{2}\lvert k_{i} \rvert^{2} ) \bar{c}_{f} s} 
\end{align} 
by \eqref{[Equation (4.2)][ZZ17]}, \eqref{[Equation (4.3)][ZZ17]} and an inequality of 
\begin{align*}
& \lvert  \frac{1}{ \prod_{i=1}^{2} f(\epsilon k_{i}) [ \lvert k_{12} \rvert^{2} f(\epsilon k_{12}) + \sum_{i=1}^{2} \lvert k_{i} \rvert^{2} f(\epsilon k_{i})]} - \frac{1}{ \lvert k_{12} \rvert^{2} + \sum_{i=1}^{2} \lvert k_{i} \rvert^{2}} \rvert  \nonumber\\
\lesssim& \left( \frac{ \lvert \epsilon k_{1} \rvert^{\eta} + \lvert \epsilon k_{2} \rvert^{\eta} + \lvert \epsilon k_{12} \rvert^{\eta}}{\lvert k_{1} \rvert^{2} + \lvert k_{2} \rvert^{2} + \lvert k_{12} \rvert^{2}} \right) 
\end{align*} 
for any $\eta \in [0,1]$ which is due to mean value theorem. Applying  \eqref{estimate 192} to \eqref{estimate 193} leads to for any $\epsilon_{0} \in (0, 2\rho)$ for $\rho \in (0, \frac{1}{2})$ sufficiently small so that $1 - 2 \rho - \eta + \epsilon_{0} > 0$ 
\begin{align}\label{estimate 127}
\lvert \phi_{131}^{\epsilon, i_{0} j_{0}} (t) - \bar{\phi}_{131}^{\epsilon, i_{0} j_{0}}(t) \rvert \lesssim \lesssim \epsilon^{\eta}t^{-\rho - \frac{\eta}{2}}
\end{align} 
where we used \eqref{key estimate}. In summary, we worked on $\pi_{0,\diamond} (u_{3,2}^{\epsilon, i_{0}}, b_{1}^{\epsilon, j_{0}}) - \pi_{0,\diamond} (\bar{u}_{3,2}^{\epsilon, i_{0}}, \bar{b}_{1}^{\epsilon, j_{0}})$ and  defined $L_{t, i_{0}j_{0}}^{7} = L_{t, i_{0}j_{0}}^{7, \star} + L_{t, i_{0}j_{0}}^{7, \star\star}$ where $L_{t, i_{0}j_{0}}^{7,\star} = L_{t, i_{0}j_{0}}^{7,\star, 1} + L_{t, i_{0}j_{0}}^{7, \star, 2}$, $L_{t, i_{0}j_{0}}^{7, \star\star} = L_{t, i_{0}j_{0}}^{7, \star\star, 1} + L_{t, i_{0}j_{0}}^{7, \star\star, 2}$ and found such $C_{131}^{\epsilon,i_{0}j_{0}}(t), \bar{C}_{131}^{\epsilon, i_{0}j_{0}}(t)$, $\phi_{131}^{\epsilon, i_{0} j_{0}}, \bar{\phi}_{131}^{\epsilon, i_{0} j_{0}}$ for $L_{t, i_{0}j_{0}}^{7,\star, 1}$. We can similarly define $C_{132}^{\epsilon, i_{0}j_{0}}(t)$, $\bar{C}_{132}^{\epsilon, i_{0}j_{0}}(t)$, $\phi_{132}^{\epsilon, i_{0}j_{0}}$,  $\bar{\phi}_{132}^{\epsilon, i_{0} j_{0}}$ for $L_{t, i_{0}j_{0}}^{7,\star, 2}$, $C_{133}^{\epsilon, i_{0}j_{0}}(t)$, $\bar{C}_{133}^{\epsilon, i_{0}j_{0}}(t)$, $\phi_{133}^{\epsilon, i_{0}j_{0}}$,  $\bar{\phi}_{133}^{\epsilon, i_{0} j_{0}}$ for $L_{t, i_{0}j_{0}}^{7,\star\star, 1}$, and $C_{134}^{\epsilon, i_{0}j_{0}}(t)$, $\bar{C}_{134}^{\epsilon, i_{0}j_{0}}(t)$, $\phi_{134}^{\epsilon, i_{0}j_{0}}$,  $\bar{\phi}_{134}^{\epsilon, i_{0} j_{0}}$ for $L_{t, i_{0}j_{0}}^{7,\star\star, 2}$. Upon working on $\pi_{0,\diamond} (u_{31}^{\epsilon, i_{0}}, b_{1}^{\epsilon, j_{0}}) - \pi_{0,\diamond} (\bar{u}_{31}^{\epsilon, i_{0}}, \bar{b}_{1}^{\epsilon, j_{0}})$, we similarly obtain $C_{13k}^{\epsilon, i_{0} j_{0}}(t)$, $\bar{C}_{13k}^{\epsilon, i_{0}j_{0}}(t)$, $\phi_{13k}^{\epsilon, i_{0} j_{0}},$ $\bar{\phi}_{13k}^{\epsilon, i_{0} j_{0}}$ for $k \in \{5,6,7,8\}$, and then define 
\begin{subequations}
\begin{align} 
&C_{13}^{\epsilon, i_{0} j_{0}} \triangleq \sum_{k=1}^{8} C_{13k}^{\epsilon, i_{0} j_{0}}, \hspace{3mm}  \bar{C}_{13}^{\epsilon, i_{0} j_{0}} \triangleq \sum_{k=1}^{8} \bar{C}_{13k}^{8, i_{0} j_{0}}, \label{estimate 47}\\
&\phi_{13}^{\epsilon, i_{0} j_{0}} \triangleq \sum_{k=1}^{8} \phi_{13k}^{\epsilon, i_{0} j_{0}}, \hspace{3mm}  \bar{\phi}_{13}^{\epsilon, i_{0} j_{0}} \triangleq \sum_{k=1}^{8} \bar{\phi}_{13k}^{\epsilon, i_{0} j_{0}}, \label{estimate 48}
\end{align} 
\end{subequations}
which was needed in \eqref{estimate 194}. Finally, upon working on $\pi_{0,\diamond} (u_{3}^{\epsilon, i_{0}}, u_{1}^{\epsilon, j_{0}}) - \pi_{0,\diamond} (\bar{u}_{3}^{\epsilon, i_{0}}, \bar{u}_{1}^{\epsilon, j_{0}})$, $\pi_{0,\diamond} (b_{3}^{\epsilon, i_{0}}, b_{1}^{\epsilon, j_{0}}) - \pi_{0,\diamond} (\bar{b}_{3}^{\epsilon, i_{0}}, \bar{b}_{1}^{\epsilon, j_{0}})$ and $\pi_{0,\diamond} (b_{3}^{\epsilon, i_{0}}, u_{1}^{\epsilon, j_{0}}) - \pi_{0,\diamond} (\bar{b}_{3}^{\epsilon, i_{0}}, \bar{u}_{1}^{\epsilon, j_{0}})$, we define similarly the quadruples of $(C_{1l}^{\epsilon, i_{0}j_{0}}, \bar{C}_{1l}^{\epsilon, i_{0}j_{0}}, \phi_{1l}^{\epsilon, i_{0} j_{0}}, \bar{\phi}_{1l}^{\epsilon, i_{0} j_{0}})$ for $l \in \{1,2,4\}$. \\
 
\emph{Terms in the second chaos: $L_{t, i_{0}j_{0}}^{k}$ for $k \in \{2,3,4,5, 6\}$ in \eqref{estimate 195}-\eqref{estimate 199}}\\

W.l.o.g. we work on $L_{t, i_{0} j_{0}}^{2}$, within which there are eight components. We define $L_{t, i_{0}j_{0}}^{2k}$ for $k \in \{1,2,3, 4\}$ to be respectively the first two, second two, third two and then fourth two components; i.e., $L_{t, i_{0}j_{0}}^{2} = \sum_{i=1}^{4} L_{t, i_{0}j_{0}}^{2i}$ where 
\begin{subequations} 
\begin{align}
L_{t, i_{0}j_{0}}^{21} \triangleq& (2\pi)^{-\frac{9}{2}} \sum_{\lvert i-j \rvert \leq 1} \sum_{i_{1}, i_{2}, i_{3}, j_{1} =1}^{3} \sum_{k} \sum_{k_{1}, k_{2}, k_{3} \neq 0: k_{23} = k} \theta(2^{-i} k_{123}) \theta(2^{-j} k_{1}) \label{estimate 200} \\
& \times [ \int_{0}^{t} e^{- \lvert k_{123} \rvert^{2}  f( \epsilon k_{123})(t-s) } \int_{0}^{s} : \hat{X}_{\sigma, u}^{\epsilon, i_{3}}(k_{2}) \hat{X}_{s,u}^{\epsilon, j_{1}}(k_{3}): k_{12}^{i_{3}} g(\epsilon k_{12}^{i_{3}}) k_{123}^{j_{1}} g(\epsilon k_{123}^{j_{1}})  \nonumber\\
& \hspace{5mm} \times\frac{ e^{- \lvert k_{1} \rvert^{2} f(\epsilon k_{1})(t-\sigma)} h_{u}(\epsilon k_{1})h_{b}(\epsilon k_{1})}{2 \lvert k_{1} \rvert^{2} f(\epsilon k_{1})} e^{- \lvert k_{12} \rvert^{2}  f(\epsilon k_{12})(s-\sigma)} d \sigma ds \nonumber \\
& - \int_{0}^{t} e^{- \lvert k_{123} \rvert^{2} (t-s) } \int_{0}^{s} : \hat{\bar{X}}_{\sigma, u}^{\epsilon, i_{3}}(k_{2}) \hat{\bar{X}}_{s,u}^{\epsilon, j_{1}}(k_{3}): k_{12}^{i_{3}} i  k_{123}^{j_{1}} i  \nonumber\\
& \hspace{5mm} \times \frac{ e^{- \lvert k_{1} \rvert^{2} (t-\sigma)} h_{u}(\epsilon k_{1})h_{b}(\epsilon k_{1})}{2 \lvert k_{1} \rvert^{2}} e^{- \lvert k_{12} \rvert^{2} (s-\sigma) } d \sigma ds] \nonumber \\
& \times \sum_{i_{4} = 1}^{3} \hat{\mathcal{P}}^{i_{2} i_{4}} (k_{1}) \hat{\mathcal{P}}^{j_{0} i_{4}}(k_{1}) \hat{\mathcal{P}}^{i_{1} i_{2}} (k_{12}) \hat{\mathcal{P}}^{i_{0} i_{1}} (k_{123}) e_{k}, \nonumber\\
L_{t, i_{0}j_{0}}^{22} \triangleq& (2\pi)^{-\frac{9}{2}} \sum_{\lvert i-j \rvert \leq 1} \sum_{i_{1}, i_{2}, i_{3}, j_{1} =1}^{3} \sum_{k} \sum_{k_{1}, k_{2}, k_{3} \neq 0: k_{23} = k} \theta(2^{-i} k_{123}) \theta(2^{-j} k_{1}) \label{estimate 201}\\
& \times [- \int_{0}^{t} e^{- \lvert k_{123} \rvert^{2}  f( \epsilon k_{123})(t-s) } \int_{0}^{s} : \hat{X}_{\sigma, b}^{\epsilon, i_{3}}(k_{2}) \hat{X}_{s,u}^{\epsilon, j_{1}}(k_{3}): k_{12}^{i_{3}} g(\epsilon k_{12}^{i_{3}}) k_{123}^{j_{1}} g(\epsilon k_{123}^{j_{1}})  \nonumber\\
& \hspace{5mm} \times\frac{ e^{- \lvert k_{1} \rvert^{2} f(\epsilon k_{1})(t-\sigma)} h_{b}(\epsilon k_{1})^{2}}{2 \lvert k_{1} \rvert^{2} f(\epsilon k_{1})} e^{- \lvert k_{12} \rvert^{2}  f(\epsilon k_{12})(s-\sigma)} d \sigma ds \nonumber \\
& + \int_{0}^{t} e^{- \lvert k_{123} \rvert^{2} (t-s) } \int_{0}^{s} : \hat{\bar{X}}_{\sigma, b}^{\epsilon, i_{3}}(k_{2}) \hat{\bar{X}}_{s,u}^{\epsilon, j_{1}}(k_{3}): k_{12}^{i_{3}} i  k_{123}^{j_{1}} i  \nonumber\\
& \hspace{5mm} \times \frac{ e^{- \lvert k_{1} \rvert^{2} (t-\sigma)} h_{b}(\epsilon k_{1})^{2}}{2 \lvert k_{1} \rvert^{2}} e^{- \lvert k_{12} \rvert^{2} (s-\sigma) } d \sigma ds] \nonumber \\
& \times \sum_{i_{4} = 1}^{3} \hat{\mathcal{P}}^{i_{2} i_{4}} (k_{1}) \hat{\mathcal{P}}^{j_{0} i_{4}}(k_{1}) \hat{\mathcal{P}}^{i_{1} i_{2}} (k_{12}) \hat{\mathcal{P}}^{i_{0} i_{1}} (k_{123}) e_{k}, \nonumber\\
L_{t, i_{0}j_{0}}^{23} \triangleq& (2\pi)^{-\frac{9}{2}} \sum_{\lvert i-j \rvert \leq 1} \sum_{i_{1}, i_{2}, i_{3}, j_{1} =1}^{3} \sum_{k} \sum_{k_{1}, k_{2}, k_{3} \neq 0: k_{23} = k} \theta(2^{-i} k_{123}) \theta(2^{-j} k_{1}) \label{estimate 202}\\
& \times [- \int_{0}^{t} e^{- \lvert k_{123} \rvert^{2}  f( \epsilon k_{123})(t-s) } \int_{0}^{s} : \hat{X}_{\sigma, u}^{\epsilon, i_{3}}(k_{2}) \hat{X}_{s,b}^{\epsilon, j_{1}}(k_{3}): k_{12}^{i_{3}} g(\epsilon k_{12}^{i_{3}}) k_{123}^{j_{1}} g(\epsilon k_{123}^{j_{1}})  \nonumber\\
& \hspace{5mm} \times\frac{ e^{- \lvert k_{1} \rvert^{2} f(\epsilon k_{1})(t-\sigma)} h_{b}(\epsilon k_{1})^{2}}{2 \lvert k_{1} \rvert^{2} f(\epsilon k_{1})} e^{- \lvert k_{12} \rvert^{2} f(\epsilon k_{12})(s-\sigma)} d \sigma ds \nonumber \\
& + \int_{0}^{t} e^{- \lvert k_{123} \rvert^{2} (t-s) } \int_{0}^{s} : \hat{\bar{X}}_{\sigma, u}^{\epsilon, i_{3}}(k_{2}) \hat{\bar{X}}_{s,b}^{\epsilon, j_{1}}(k_{3}): k_{12}^{i_{3}} i  k_{123}^{j_{1}} i  \nonumber\\
& \hspace{5mm} \times \frac{ e^{- \lvert k_{1} \rvert^{2} (t-\sigma)} h_{b}(\epsilon k_{1})^{2}}{2 \lvert k_{1} \rvert^{2}} e^{- \lvert k_{12} \rvert^{2} (s-\sigma) } d \sigma ds] \nonumber \\
& \times \sum_{i_{4} = 1}^{3} \hat{\mathcal{P}}^{i_{2} i_{4}} (k_{1}) \hat{\mathcal{P}}^{j_{0} i_{4}}(k_{1}) \hat{\mathcal{P}}^{i_{1} i_{2}} (k_{12}) \hat{\mathcal{P}}^{i_{0} i_{1}} (k_{123}) e_{k}, \nonumber\\
L_{t, i_{0}j_{0}}^{24} \triangleq& (2\pi)^{-\frac{9}{2}} \sum_{\lvert i-j \rvert \leq 1} \sum_{i_{1}, i_{2}, i_{3}, j_{1} =1}^{3} \sum_{k} \sum_{k_{1}, k_{2}, k_{3} \neq 0: k_{23} = k} \theta(2^{-i} k_{123}) \theta(2^{-j} k_{1}) \label{estimate 203}\\
& \times [\int_{0}^{t} e^{- \lvert k_{123} \rvert^{2}  f( \epsilon k_{123})(t-s) } \int_{0}^{s} : \hat{X}_{\sigma, b}^{\epsilon, i_{3}}(k_{2}) \hat{X}_{s,b}^{\epsilon, j_{1}}(k_{3}): k_{12}^{i_{3}} g(\epsilon k_{12}^{i_{3}}) k_{123}^{j_{1}} g(\epsilon k_{123}^{j_{1}})  \nonumber\\
& \hspace{5mm} \times\frac{ e^{- \lvert k_{1} \rvert^{2} f(\epsilon k_{1})(t-\sigma)} h_{u}(\epsilon k_{1}) h_{b}(\epsilon k_{1})}{2 \lvert k_{1} \rvert^{2} f(\epsilon k_{1})} e^{- \lvert k_{12} \rvert^{2} f(\epsilon k_{12})(s-\sigma)} d \sigma ds \nonumber \\
& -\int_{0}^{t} e^{- \lvert k_{123} \rvert^{2} (t-s) } \int_{0}^{s} : \hat{\bar{X}}_{\sigma, b}^{\epsilon, i_{3}}(k_{2}) \hat{\bar{X}}_{s,b}^{\epsilon, j_{1}}(k_{3}): k_{12}^{i_{3}} i  k_{123}^{j_{1}} i  \nonumber\\
& \hspace{5mm} \times \frac{ e^{- \lvert k_{1} \rvert^{2} (t-\sigma)} h_{u}(\epsilon k_{1}) h_{b}(\epsilon k_{1})}{2 \lvert k_{1} \rvert^{2}} e^{- \lvert k_{12} \rvert^{2} (s-\sigma) } d \sigma ds] \nonumber \\
& \times \sum_{i_{4} = 1}^{3} \hat{\mathcal{P}}^{i_{2} i_{4}} (k_{1}) \hat{\mathcal{P}}^{j_{0} i_{4}}(k_{1}) \hat{\mathcal{P}}^{i_{1} i_{2}} (k_{12}) \hat{\mathcal{P}}^{i_{0} i_{1}} (k_{123}) e_{k}. \nonumber
\end{align}
\end{subequations} 
W.l.o.g. we show the estimate on $L_{t, i_{0} j_{0}}^{24}$ in \eqref{estimate 203} as other cases are similar. We write 
\begin{align*}
& \mathbb{E} [ \lvert \Delta_{q} L_{t, i_{0}j_{0}}^{24} \rvert^{2} ] \nonumber \\
\lesssim& \mathbb{E} [ \lvert \sum_{k} \theta(2^{-q} k)^{2} \sum_{\lvert i-j \rvert \leq 1} \sum_{k_{1}, k_{2}, k_{3} \neq 0: k_{23} = k} \sum_{i_{3}, j_{1} =1}^{3} \theta(2^{-i} k_{123}) \theta(2^{-j} k_{1}) \nonumber\\
& \times  \int_{0}^{t} (e^{- \lvert k_{123} \rvert^{2}  f(\epsilon k_{123})(t-s) } k_{123}^{j_{1}} g(\epsilon k_{123}^{j_{1}}) - e^{- \lvert k_{123} \rvert^{2} (t-s)} k_{123}^{j_{1}} i) \nonumber\\
& \hspace{5mm} \times \int_{0}^{s} : \hat{X}_{\sigma, b}^{\epsilon, i_{3}}(k_{2}) \hat{X}_{s,b}^{\epsilon, j_{1}}(k_{3}): \frac{ e^{- \lvert k_{1} \rvert^{2} f(\epsilon k_{1} ) (t-\sigma)} h_{u}(\epsilon k_{1}) h_{b}(\epsilon k_{1})}{2 \lvert k_{1} \rvert^{2} f(\epsilon k_{1})} \nonumber\\
& \hspace{5mm} \times k_{12}^{i_{3}} g(\epsilon k_{12}^{i_{3}}) e^{- \lvert k_{12} \rvert^{2}  f(\epsilon k_{12})(s-\sigma)} d\sigma ds e_{k} \rvert^{2} ] \nonumber \\
&+ \mathbb{E} [ \lvert \sum_{k} \theta(2^{-q} k)^{2} \sum_{\lvert i-j \rvert \leq 1} \sum_{k_{1}, k_{2}, k_{3} \neq 0: k_{23} = k} \sum_{i_{3}, j_{1} =1}^{3} \theta(2^{-i} k_{123}) \theta(2^{-j} k_{1}) \nonumber\\
& \times  \int_{0}^{t} e^{- \lvert k_{123} \rvert^{2} (t-s)} k_{123}^{j_{1}} i\int_{0}^{s} [ : \hat{X}_{\sigma, b}^{\epsilon, i_{3}} (k_{2}) \hat{X}_{s,b}^{\epsilon, j_{1}}(k_{3}): - : \hat{\bar{X}}_{\sigma, b}^{\epsilon, i_{3}}(k_{2}) \hat{\bar{X}}_{s,b}^{\epsilon, j_{1}}(k_{3}):] \nonumber \\
& \hspace{5mm} \times \frac{ e^{- \lvert k_{1} \rvert^{2} f(\epsilon k_{1}) (t-\sigma)} h_{u}(\epsilon k_{1}) h_{b}(\epsilon k_{1})}{2\lvert k_{1} \rvert^{2} f(\epsilon k_{1})} k_{12}^{i_{3}} g(\epsilon k_{12}^{i_{3}}) e^{- \lvert k_{12} \rvert^{2} f(\epsilon k_{12})(s-\sigma)} d\sigma ds  e_{k}\rvert^{2} ] \nonumber \\
&+ \mathbb{E} [ \lvert \sum_{k} \theta(2^{-q} k)^{2} \sum_{\lvert i-j \rvert \leq 1} \sum_{k_{1}, k_{2}, k_{3} \neq 0: k_{23} = k} \sum_{i_{3}, j_{1} =1}^{3} \theta(2^{-i} k_{123}) \theta(2^{-j} k_{1}) \nonumber\\
& \times  \int_{0}^{t} e^{- \lvert k_{123} \rvert^{2} (t-s)} k_{123}^{j_{1}} i \int_{0}^{s} : \hat{\bar{X}}_{\sigma, b}^{\epsilon, i_{3}} (k_{2}) \hat{\bar{X}}_{s,b}^{\epsilon, j_{1}}(k_{3}): [ \frac{ e^{- \lvert k_{1} \rvert^{2} f(\epsilon k_{1}) (t-\sigma)}}{2 \lvert k_{1} \rvert^{2} f(\epsilon k_{1})} - \frac{ e^{- \lvert k_{1} \rvert^{2} (t-\sigma)}}{2 \lvert k_{1}\rvert^{2}}]  \nonumber\\
& \hspace{5mm} \times h_{u}(\epsilon k_{1}) h_{b}(\epsilon k_{1}) k_{12}^{i_{3}} g(\epsilon k_{12}^{i_{3}}) e^{- \lvert k_{12} \rvert^{2}  f(\epsilon k_{12})(s-\sigma)} d\sigma ds e_{k}\rvert^{2} ] \nonumber \\
&+ \mathbb{E} [ \lvert \sum_{k} \theta(2^{-q} k)^{2} \sum_{\lvert i-j \rvert \leq 1} \sum_{k_{1}, k_{2}, k_{3} \neq 0: k_{23} = k} \sum_{i_{3}, j_{1} =1}^{3} \theta(2^{-i} k_{123}) \theta(2^{-j} k_{1}) \nonumber\\
& \times  \int_{0}^{t} e^{- \lvert k_{123} \rvert^{2} (t-s)} k_{123}^{j_{1}} i \int_{0}^{s} : \hat{\bar{X}}_{\sigma, b}^{\epsilon, i_{3}} (k_{2}) \hat{\bar{X}}_{s,b}^{\epsilon, j_{1}}(k_{3}): \frac{ e^{- \lvert k_{1} \rvert^{2} (t-\sigma)}}{2 \lvert k_{1} \rvert^{2}} h_{u}(\epsilon k_{1}) h_{b}(\epsilon k_{1}) \nonumber \\
& \hspace{5mm} \times [ e^{- \lvert k_{12} \rvert^{2}  f(\epsilon k_{12})(s-\sigma)} k_{12}^{i_{3}} g(\epsilon k_{12}^{i_{3}}) - e^{-\lvert k_{12} \rvert^{2} (s-\sigma)} k_{12}^{i_{3}} i] d\sigma ds e_{k}  \rvert^{2} ]. \nonumber 
\end{align*} 
We can rely on \eqref{covariance} and continue to bound by 
\begin{equation}\label{estimate 215}
\mathbb{E} [ \lvert \Delta_{q} L_{t, i_{0} j_{0}}^{24} \rvert^{2} ]  \lesssim  \sum_{i=1}^{4} L_{t, q}^{24i} 
\end{equation}  
where 
\begin{subequations}
\begin{align}
L_{t, q}^{241} &\triangleq \sum_{k} \sum_{\lvert i-j \rvert \leq 1, \lvert i'- j' \rvert \leq 1} \sum_{k_{1}, k_{2}, k_{3}, k_{1}', k_{2}', k_{3}' \neq 0: k_{23} = k_{23}' = k} \sum_{j_{1}, j_{1}' = 1}^{3}  \label{estimate 204} \\
& \times ( 1_{k_{2} = k_{2}', k_{3} = k_{3}'} + 1_{k_{2} = k_{3}', k_{3} = k_{2}'} ) \theta(2^{-q} k)^{2} \theta(2^{-i} k_{123}) \theta(2^{-i'} k_{123}') \theta(2^{-j} k_{1}) \theta(2^{-j'} k_{1}') \nonumber\\
& \times  \int_{[0,t]^{2}} ( e^{- \lvert k_{123} \rvert^{2}  f(\epsilon k_{123})(t-s) } k_{123}^{j_{1}} g(\epsilon k_{123}^{j_{1}}) - e^{- \lvert k_{123} \rvert^{2} (t-s)}k_{123}^{j_{1}} i) \nonumber\\
& \hspace{2mm} \times  ( e^{- \lvert k_{123}' \rvert^{2}  f(\epsilon k_{123}')(t-\bar{s}) } (k_{123}')^{j_{1}'} g(\epsilon (k_{123}')^{j_{1}'}) - e^{- \lvert k_{123}' \rvert^{2} (t-\bar{s})}(k_{123}')^{j_{1}'} i) \nonumber \\
& \hspace{2mm} \times \int_{0}^{s} \int_{0}^{\bar{s}} \prod_{i=2}^{3} \frac{1}{ \lvert k_{i} \rvert^{2}} \frac{ e^{- \lvert k_{1} \rvert^{2} f(\epsilon k_{1}) (t-\sigma) - \lvert k_{1}' \rvert^{2} f(\epsilon k_{1}') (t-\bar{\sigma})} }{\lvert k_{1} \rvert^{2} f(\epsilon k_{1}) \lvert k_{1}' \rvert^{2} f(\epsilon k_{1}')} \lvert k_{12} \rvert \lvert k_{12} ' \rvert \nonumber\\
& \hspace{2mm} \times e^{- \lvert k_{12} \rvert^{2} f(\epsilon k_{12}) (s-\sigma) - \lvert k_{12}' \rvert^{2} f(\epsilon k_{12}') (\bar{s}-\bar{\sigma})}d \bar{\sigma} d \sigma  ds d \bar{s}, \nonumber\\
L_{t, q}^{242} &\triangleq  \sum_{k} \sum_{\lvert i-j \rvert \leq 1, \lvert i'- j' \rvert \leq 1} \sum_{k_{1}, k_{2}, k_{3}, k_{1}', k_{2}', k_{3}' \neq 0: k_{23} = k_{23}' = k} \sum_{i_{3}, j_{1}, i_{3}', j_{1}' = 1}^{3} \label{estimate 205}  \\
& \times ( 1_{k_{2} = k_{2}', k_{3} = k_{3}'} + 1_{k_{2} = k_{3}', k_{3} = k_{2}'} ) \theta(2^{-q} k)^{2} \theta(2^{-i} k_{123}) \theta(2^{-i'} k_{123}') \theta(2^{-j} k_{1}) \theta(2^{-j'} k_{1}') \nonumber\\
& \times \int_{[0,t]^{2}} e^{- \lvert k_{123} \rvert^{2} (t-s)} e^{- \lvert k_{123} ' \rvert^{2} (t- \bar{s})} \lvert k_{123} \rvert \lvert k_{123}' \rvert \lvert k_{12} \rvert \lvert k_{12} ' \rvert \nonumber \\
& \hspace{2mm} \times \int_{0}^{s} \int_{0}^{\bar{s}} \frac{ e^{ - \lvert k_{1} \rvert^{2} f(\epsilon k_{1}) (t-\sigma) - \lvert k_{1}' \rvert^{2} f(\epsilon k_{1}') (t- \bar{\sigma})}}{ \lvert k_{1} \rvert^{2} \lvert k_{1}' \rvert^{2} f(\epsilon k_{1}) f(\epsilon k_{1}')} e^{- \lvert k_{12} \rvert^{2}  f(\epsilon k_{12})(s-\sigma) - \lvert k_{12}' \rvert^{2} f(\epsilon k_{12} ') (\bar{s} - \bar{\sigma})} \nonumber \\
& \hspace{2mm} \times \mathbb{E} [ ( : \hat{X}_{\sigma, b}^{\epsilon, i_{3}}(k_{2}) \hat{X}_{s,b}^{\epsilon, j_{1}}(k_{3}): - : \hat{\bar{X}}_{\sigma, b}^{\epsilon, i_{3}}(k_{2}) \hat{\bar{X}}_{s,b}^{\epsilon, j_{1}}(k_{3}): ) \nonumber\\
& \hspace{5mm} \times  \overline{( : \hat{X}_{\bar{\sigma}, b}^{\epsilon, i_{3}'}(k_{2}') \hat{X}_{\bar{s},b}^{\epsilon, j_{1}'}(k_{3}'): - : \hat{\bar{X}}_{\bar{\sigma}, b}^{\epsilon, i_{3}'}(k_{2}') \hat{\bar{X}}_{\bar{s},b}^{\epsilon, j_{1}'}(k_{3}'): )} ] d \bar{\sigma } d\sigma  ds d \bar{s},  \nonumber \\
L_{t,q}^{243} &\triangleq   \sum_{k} \sum_{\lvert i-j \rvert \leq 1, \lvert i'- j' \rvert \leq 1} \sum_{k_{1}, k_{2}, k_{3}, k_{1}', k_{2}', k_{3}' \neq 0: k_{23} = k_{23}' = k}   \label{estimate 206}\\
& \times ( 1_{k_{2} = k_{2}', k_{3} = k_{3}'} + 1_{k_{2} = k_{3}', k_{3} = k_{2}'} ) \theta(2^{-q} k)^{2} \theta(2^{-i} k_{123}) \theta(2^{-i'} k_{123}') \theta(2^{-j} k_{1}) \theta(2^{-j'} k_{1}') \nonumber\\
& \times \int_{[0,t]^{2}} e^{- \lvert k_{123} \rvert^{2} (t-s)} \lvert k_{123} \rvert e^{- \lvert k_{123}' \rvert^{2} (t- \bar{s})} \lvert k_{123}' \rvert \int_{0}^{s} \int_{0}^{\bar{s}} \frac{1}{ \lvert k_{2} \rvert^{2} \lvert k_{3} \rvert^{2}} \nonumber \\
& \hspace{2mm} \times \lvert \frac{ e^{- \lvert k_{1} \rvert^{2} f(\epsilon k_{1}) (t-\sigma)}}{\lvert k_{1} \rvert^{2} f(\epsilon k_{1})} - \frac{ e^{- \lvert k_{1} \rvert^{2} (t-\sigma)}}{\lvert k_{1} \rvert^{2}} \rvert \lvert \frac{ e^{- \lvert k_{1}' \rvert^{2} f(\epsilon k_{1}') (t- \bar{\sigma})}}{\lvert k_{1}' \rvert^{2} f(\epsilon k_{1}')} - \frac{ e^{- \lvert k_{1}' \rvert^{2} (t- \bar{\sigma})}}{\lvert k_{1}' \rvert^{2}} \rvert \nonumber \\
& \hspace{2mm} \times e^{- \lvert k_{12} \rvert^{2} f(\epsilon k_{12})  (s-\sigma)- \lvert k_{12} ' \rvert^{2} f(\epsilon k_{12}') ( \bar{s}- \bar{\sigma})} \lvert k_{12} \rvert \lvert k_{12}' \rvert d \bar{\sigma} d\sigma ds d \bar{s}, \nonumber\\
L_{t,q}^{244} &\triangleq   \sum_{k} \sum_{\lvert i-j \rvert \leq 1, \lvert i'- j' \rvert \leq 1} \sum_{k_{1}, k_{2}, k_{3}, k_{1}', k_{2}', k_{3}' \neq 0: k_{23} = k_{23}' = k} \sum_{i_{3}, i_{3}'  = 1}^{3}  \label{estimate 207}\\
& \times ( 1_{k_{2} = k_{2}', k_{3} = k_{3}'} + 1_{k_{2} = k_{3}', k_{3} = k_{2}'} ) \theta(2^{-q} k)^{2} \theta(2^{-i} k_{123}) \theta(2^{-i'} k_{123}') \theta(2^{-j} k_{1}) \theta(2^{-j'} k_{1}') \nonumber\\
& \times  \int_{[0,t]^{2}} \frac{e^{- \lvert k_{123} \rvert^{2} (t-s)} \lvert k_{123} \rvert e^{- \lvert k_{123}' \rvert^{2} (t- \bar{s})} \lvert k_{123} ' \rvert}{\lvert k_{2} \rvert^{2} \lvert k_{3} \rvert^{2}}  \int_{0}^{s} \int_{0}^{\bar{s}}   \frac{ e^{- \lvert k_{1} \rvert^{2} (t-\sigma)}}{\lvert k_{1} \rvert^{2}} \frac{ e^{- \lvert k_{1}' \rvert^{2} (t- \bar{\sigma})}}{\lvert k_{1}' \rvert^{2}} \nonumber   \\
& \hspace{2mm} \times \lvert e^{- \lvert k_{12} \rvert^{2}  f(\epsilon k_{12})(s-\sigma) } k_{12}^{i_{3}} g(\epsilon k_{12}^{i_{3}}) - e^{- \lvert k_{12} \rvert^{2} (s-\sigma)} k_{12}^{i_{3}} i \rvert \nonumber \\
&\hspace{2mm} \times \lvert e^{- \lvert k_{12}' \rvert^{2}  f(\epsilon k_{12}')(\bar{s}-\bar{\sigma}) } (k_{12}')^{i_{3}'} g(\epsilon (k_{12}')^{i_{3}'}) - e^{- \lvert k_{12}' \rvert^{2} (\bar{s}-\bar{\sigma})} (k_{12}')^{i_{3}'} i \rvert d \bar{\sigma}d \sigma  d s d \bar{s}. \nonumber
\end{align}
\end{subequations} 
First, we compute from \eqref{estimate 204} for any $\eta \in [0,1]$, 
\begin{align}\label{estimate 208}
L_{t,q}^{241} \lesssim& \epsilon^{\eta} \sum_{k} \sum_{\lvert i-j \rvert \leq 1, \lvert i' - j' \rvert \leq 1} \sum_{k_{1}, k_{2}, k_{3}, k_{4} \neq 0: k_{23} = k} \theta(2^{-q} k)^{2} \theta(2^{-i} k_{123}) \theta(2^{-i'} k_{234}) \nonumber \\
& \times  \theta(2^{-j} k_{1}) \theta(2^{-j'} k_{4}) \prod_{i=1}^{4} \frac{1}{\lvert k_{i} \rvert^{2}}  \lvert k_{123} \rvert^{1+ \frac{\eta}{2}} \lvert k_{234} \rvert^{1+ \frac{\eta}{2}} \nonumber \\ 
& \times e^{ - \lvert k_{123} \rvert^{2} \bar{c}_{f} t - \lvert k_{234} \rvert^{2} \bar{c}_{f} t - \lvert k_{1} \rvert^{2} f(\epsilon k_{1}) t - \lvert k_{4} \rvert^{2} f(\epsilon k_{4}) t} \frac{ \lvert k_{12}\rvert}{\lvert k_{1} \rvert^{2} + \lvert k_{12} \rvert^{2}} \nonumber \\
& \times \int_{[0,t]^{2}} e^{ \lvert k_{123} \rvert^{2} \bar{c}_{f} s + \lvert k_{234} \rvert^{2} \bar{c}_{f} \bar{s} + \lvert k_{1} \rvert^{2} f(\epsilon k_{1}) s + \lvert k_{4} \rvert^{2} f(\epsilon k_{4}) \bar{s}} \nonumber\\
& \times [ \frac{\lvert k_{24} \rvert}{\lvert k_{4} \rvert^{2} + \lvert k_{24} \rvert^{2}} + \frac{\lvert k_{34} \rvert}{\lvert k_{4} \rvert^{2} + \lvert k_{34} \rvert^{2}} ] ds d \bar{s} \nonumber \\
\lesssim& \epsilon^{\eta}t^{\eta}  \sum_{k} \sum_{\lvert i-j \rvert \leq 1, \lvert i' - j' \rvert \leq 1} \sum_{k_{1}, k_{2}, k_{3}, k_{4} \neq 0: k_{23} = k} \theta(2^{-q} k)^{2} \theta(2^{-i} k_{123}) \theta(2^{-i'} k_{234}) \nonumber\\
& \times  \theta(2^{-j} k_{1}) \theta(2^{-j'} k_{4}) (\prod_{i=1}^{4} \frac{1}{\lvert k_{i} \rvert^{2}}) \frac{1}{ \lvert k_{1} \rvert^{2- \frac{3\eta}{2}} \lvert k_{4} \rvert^{2- \frac{3\eta}{2}}} 
\end{align} 
by \eqref{[Equation (4.2)][ZZ17]}-\eqref{[Equation (4.3)][ZZ17]} and mean value theorem. Next, in order to work on $L_{t, q}^{242}$ in \eqref{estimate 205}, we first compute 
\begin{align}\label{estimate 211}
& \mathbb{E} [ \lvert : \hat{X}_{\sigma, b}^{\epsilon, i_{3}}(k_{2}) \hat{X}_{s,b}^{\epsilon, j_{1}}(k_{3}): - : \hat{\bar{X}}_{\sigma, b}^{\epsilon, i_{3}}(k_{2}) \hat{\bar{X}}_{s,b}^{\epsilon, j_{1}}(k_{3}): \rvert^{2} ]  \\
\lesssim& \frac{1}{ \lvert k_{2} \rvert^{2} \lvert k_{3} \rvert^{2}}  \lvert \frac{1}{4 f(\epsilon k_{2}) f(\epsilon k_{3})} - \frac{2}{(f(\epsilon k_{2}) + 1)(f(\epsilon k_{3}) + 1)} + \frac{1}{4} \rvert \nonumber\\
&+ \frac{1}{\lvert k_{2} \rvert^{4}} \lvert \frac{ e^{-2 \lvert k_{2} \rvert^{2} f(\epsilon k_{2})(s-\sigma)}}{4 f(\epsilon k_{2})^{2}} - \frac{ 2e^{- \lvert k_{2} \rvert^{2} (f(\epsilon k_{2}) + 1)(s-\sigma)}}{(f(\epsilon k_{2}) + 1)^{2}} + \frac{ e^{-2 \lvert k_{2} \rvert^{2} (s-\sigma)}}{4} \rvert 1_{k_{2} = k_{3}} \nonumber 
\end{align}
by Example \ref{Example 3.1} and \eqref{covariance} - \eqref{estimate 174} where for any $\eta \in [0,1]$, 
\begin{align}\label{estimate 209}
\lvert \frac{1}{4 f(\epsilon k_{2}) f(\epsilon k_{3})} - \frac{2}{(f(\epsilon k_{2}) + 1)(f(\epsilon k_{3}) + 1)} + \frac{1}{4} \rvert \lesssim \lvert \epsilon k_{2} \rvert^{\eta} + \lvert \epsilon k_{3} \rvert^{\eta}  
\end{align} 
by mean value theorem and 
\begin{align}\label{estimate 210}
\lvert \frac{ e^{-2 \lvert k_{2} \rvert^{2} f(\epsilon k_{2})(s-\sigma)}}{4 f(\epsilon k_{2})^{2}} - \frac{ 2e^{- \lvert k_{2} \rvert^{2} (f(\epsilon k_{2}) + 1)(s-\sigma)}}{(f(\epsilon k_{2}) + 1)^{2}} + \frac{ e^{-2 \lvert k_{2} \rvert^{2} (s-\sigma)}}{4} \rvert  \lesssim \lvert \epsilon k_{2} \rvert^{\eta}
\end{align} 
by \eqref{[Equation (4.2)][ZZ17]}. Applying H$\ddot{\mathrm{o}}$lder's inequality and \eqref{estimate 209}-\eqref{estimate 210} to \eqref{estimate 211} shows that 
\begin{align}\label{[Equation (4.7)][ZZ17]}
& \mathbb{E} [ ( : \hat{X}_{\sigma, b}^{\epsilon, i_{3}}(k_{2}) \hat{X}_{s,b}^{\epsilon, j_{1}}(k_{3}): - : \hat{\bar{X}}_{\sigma, b}^{\epsilon, i_{3}}(k_{2}) \hat{\bar{X}}_{s,b}^{\epsilon, j_{1}}(k_{3}): ) \nonumber\\
& \hspace{5mm} \times  \overline{( : \hat{X}_{\bar{\sigma}, b}^{\epsilon, i_{3}'}(k_{2}') \hat{X}_{\bar{s},b}^{\epsilon, j_{1}'}(k_{3}'): - : \hat{\bar{X}}_{\bar{\sigma}, b}^{\epsilon, i_{3}'}(k_{2}') \hat{\bar{X}}_{\bar{s},b}^{\epsilon, j_{1}'}(k_{3}'): )} \nonumber\\
& \hspace{10mm} \lesssim \frac{ ( \lvert \epsilon k_{2} \rvert^{\frac{\eta}{2}} + \lvert \epsilon k_{3} \rvert^{\frac{\eta}{2}})}{\lvert k_{2} \rvert \lvert k_{3} \rvert}   \frac{ ( \lvert \epsilon k_{2}' \rvert^{\frac{\eta}{2}} + \lvert \epsilon k_{3}' \rvert^{\frac{\eta}{2}})}{\lvert k_{2}' \rvert \lvert k_{3}' \rvert}. 
\end{align} 
Applying \eqref{[Equation (4.7)][ZZ17]} to \eqref{estimate 205}  gives us  
\begin{align}\label{estimate 212}
L_{t, q}^{242} \lesssim&  \epsilon^{\eta}  \sum_{k} \sum_{\lvert i-j \rvert \leq 1, \lvert i' - j' \rvert \leq 1} \sum_{k_{1}, k_{2}, k_{3}, k_{4} \neq 0: k_{23} =k} \nonumber\\
& \times \theta(2^{-q} k)^{2} \theta(2^{-i} k_{123}) \theta(2^{-i'} k_{234}) \theta(2^{-j} k_{1}) \theta(2^{-j'} k_{4}) \nonumber\\
& \times \prod_{i=1}^{4} \frac{1}{\lvert k_{i} \rvert^{2}} (\lvert k_{2} \rvert^{\eta} + \lvert k_{3} \rvert^{\eta}) \lvert k_{123} \rvert \lvert k_{234} \rvert \frac{1}{ ( \lvert k_{1} \rvert + \lvert k_{12} \rvert )} [ \frac{1}{ \lvert k_{4} \rvert + \lvert k_{24} \rvert } + \frac{1}{ \lvert k_{4} \rvert + \lvert k_{34} \rvert} ] \nonumber \\
& \times \int_{[0,t]^{2}} e^{- \lvert k_{123} \rvert^{2} (t-s) - \lvert k_{234} \rvert^{2} (t- \bar{s}) - \lvert k_{1} \rvert^{2} f(\epsilon k_{1}) (t-s) - \lvert k_{4} \rvert^{2} f(\epsilon k_{4}) (t- \bar{s})} ds d \bar{s} \nonumber \\
\lesssim& \epsilon^{\eta} t^{\eta} \sum_{k} \sum_{\lvert i-j \rvert \leq 1, \lvert i' - j' \rvert \leq 1} \sum_{k_{1}, k_{2}, k_{3}, k_{4} \neq 0: k_{23} =k} \nonumber\\
& \times  \theta(2^{-q} k)^{2} \theta(2^{-i} k_{123}) \theta(2^{-i'} k_{234}) \theta(2^{-j} k_{1}) \theta(2^{-j'} k_{4}) \nonumber\\ 
& \times [ \frac{1}{ \lvert k_{1} \rvert^{4-\eta}} \frac{1}{\lvert k_{2} \rvert^{2-\eta}} \frac{1}{\lvert k_{3} \rvert^{2}}  \frac{1}{\lvert k_{4} \rvert^{4-\eta}} + \frac{1}{ \lvert k_{1} \rvert^{4-\eta}} \frac{1}{\lvert k_{2} \rvert^{2}} \frac{1}{\lvert k_{3} \rvert^{2 - \eta}}  \frac{1}{\lvert k_{4} \rvert^{4-\eta}}
\end{align} 
by \eqref{key estimate}. Next, we estimate from \eqref{estimate 206}
\begin{align}\label{estimate 213}
L_{t,q}^{243} \lesssim&  \sum_{k} \sum_{\lvert i-j \rvert \leq 1, \lvert i'- j' \rvert \leq 1} \sum_{k_{1}, k_{2}, k_{3}, k_{1}', k_{2}', k_{3}' \neq 0: k_{23} = k_{23}' = k} ( 1_{k_{2} = k_{2}', k_{3} = k_{3}'} + 1_{k_{2} = k_{3}', k_{3} = k_{2}'} )  \nonumber\\
& \times  \theta(2^{-q} k)^{2} \theta(2^{-i} k_{123}) \theta(2^{-i'} k_{123}') \theta(2^{-j} k_{1}) \theta(2^{-j'} k_{1}') \nonumber\\
& \times \int_{[0,t]^{2}} e^{- \lvert k_{123} \rvert^{2} (t-s)} \lvert k_{123} \rvert e^{- \lvert k_{123}' \rvert^{2} (t- \bar{s})} \lvert k_{123}' \rvert \int_{0}^{s} \int_{0}^{\bar{s}} \frac{1}{ \lvert k_{2} \rvert^{2} \lvert k_{3} \rvert^{2}} \frac{1}{\lvert k_{1} \rvert^{2}} \frac{1}{\lvert k_{1}' \rvert^{2}} \nonumber \\
&  \hspace{5mm} \times [ e^{ - \lvert k_{1} \rvert^{2} \bar{c}_{f} (t-\sigma)} \lvert \epsilon k_{1} \rvert^{\frac{\eta}{2}} e^{- \lvert k_{1}' \rvert^{2} \bar{c}_{f} (t- \bar{\sigma})} \lvert \epsilon k_{1}' \rvert^{\frac{\eta}{2}} ] \nonumber \\
& \hspace{5mm} \times e^{- \lvert k_{12} \rvert^{2}  f(\epsilon k_{12})(s-\sigma) - \lvert k_{12} ' \rvert^{2} f(\epsilon k_{12}')(\bar{s} - \bar{\sigma})} \lvert k_{12} \rvert \lvert k_{12} ' \rvert d \bar{\sigma} d\sigma  ds d \bar{s} \nonumber \\
\lesssim& \epsilon^{\eta} t^{\eta} \sum_{k} \sum_{\lvert i-j \rvert \leq 1, \lvert i'-j'\rvert \leq 1} \sum_{k_{1}, k_{2}, k_{3}, k_{4} \neq 0: k_{23} = k}  \theta(2^{-q} k)^{2} \theta(2^{-i} k_{123}) \theta(2^{-i'} k_{234})\nonumber\\
& \times  \theta(2^{-j} k_{1}) \theta(2^{-j'} k_{4}) (\prod_{i=1}^{4} \frac{1}{\lvert k_{i} \rvert^{2}}) \frac{1}{\lvert k_{1} \rvert^{2- \frac{3\eta}{2}}} \frac{1}{\lvert k_{4} \rvert^{2- \frac{3\eta}{2}}} 
\end{align} 
by \eqref{[Equation (4.2)][ZZ17]} and mean value theorem. Finally, we estimate from \eqref{estimate 207} for any $\eta \in [0,1]$, 
\begin{align}\label{estimate 214}
L_{t,q}^{244} \lesssim& \epsilon^{\eta}  \sum_{k} \sum_{\lvert i-j \rvert \leq 1, \lvert i'- j' \rvert \leq 1} \sum_{k_{1}, k_{2}, k_{3}, k_{1}', k_{2}', k_{3}' \neq 0: k_{23} = k_{23}' = k} ( 1_{k_{2} = k_{2}', k_{3} = k_{3}'} + 1_{k_{2} = k_{3}', k_{3} = k_{2}'} ) \nonumber\\
& \times \theta(2^{-q} k)^{2} \theta(2^{-i} k_{123}) \theta(2^{-i'} k_{123}') \theta(2^{-j} k_{1}) \theta(2^{-j'} k_{1}') \nonumber\\
& \times  \lvert k_{123} \rvert \lvert k_{123} ' \rvert \frac{ \lvert k_{12} \rvert^{\frac{\eta}{2}} \lvert k_{12} ' \rvert^{\frac{\eta}{2}}}{ \lvert k_{2} \rvert^{2} \lvert k_{3} \rvert^{2} \lvert k_{1} \rvert^{2} \lvert k_{1} ' \rvert^{2}} \frac{1}{ \lvert k_{1} \rvert + \lvert k_{12} \rvert} \frac{1}{\lvert k_{1}' \rvert + \lvert k_{12} ' \rvert} \nonumber\\
& \times \int_{[0,t]^{2}} e^{- ( \lvert k_{123} \rvert^{2} + \lvert k_{1} \rvert^{2} )(t-s)} e^{- ( \lvert k_{123} ' \rvert^{2} + \lvert k_{1}' \rvert^{2} ) (t- \bar{s})} ds d \bar{s} \nonumber \\
\lesssim& \epsilon^{\eta} t^{\eta}  \sum_{k} \sum_{\lvert i-j \rvert \leq 1, \lvert i'- j' \rvert \leq 1} \sum_{k_{1}, k_{2}, k_{3}, k_{4} \neq 0: k_{23} = k} \theta(2^{-q}k)^{2} \theta(2^{-i} k_{123}) \theta(2^{-i'} k_{234}) \nonumber\\
& \times \theta(2^{-j} k_{1}) \theta(2^{-j'} k_{4}) (\prod_{i=1}^{4} \frac{1}{\lvert k_{i} \rvert^{2}})  \frac{1}{ \lvert k_{1} \rvert^{2- \frac{3\eta}{2}}} \frac{1}{ \lvert k_{4} \rvert^{2- \frac{3\eta}{2}}} 
\end{align} 
by \eqref{[Equation (4.2)][ZZ17]}-\eqref{[Equation (4.3)][ZZ17]} and mean value theorem. Therefore, we can now conclude by applying \eqref{estimate 208}, \eqref{estimate 212}-\eqref{estimate 214} to \eqref{estimate 215}, using the observation that $2^{q} \lesssim 2^{j}$ so that $q \lesssim j$ and similarly $q \lesssim j'$, that for any $\eta \in [0, \frac{2}{3})$, 
\begin{align}\label{estimate 14}
& \mathbb{E} [ \lvert \Delta_{1} L_{t, i_{0}j_{0}}^{24} \rvert^{2} ]\\
\lesssim& \epsilon^{\eta} t^{\eta} \sum_{k} \sum_{k_{2}, k_{3}\neq 0: k_{23} = k} \theta(2^{-q} k)^{2} \nonumber\\
\times & [ \frac{1}{ \lvert k_{2} \rvert^{2} \lvert k_{3} \rvert^{2}} \frac{1}{2^{2q(1- \frac{3\eta}{2})}} + \frac{1}{\lvert k_{2} \rvert^{2-\eta} \lvert k_{3} \rvert^{2}} \frac{1}{2^{2q (1- \eta)}} + \frac{1}{\lvert k_{2} \rvert^{2} \lvert k_{3} \rvert^{2 - \eta}} \frac{1}{2^{2q(1-\eta)}}]  \lesssim  \epsilon^{\eta} t^{\eta} 2^{q3\eta}. \nonumber
\end{align}   
The estimate of $L_{t, i_{0}j_{0}}^{3}$ in \eqref{estimate 196} is very similar to that of $L_{t, i_{0} j_{0}}^{2}$. Next, we look at $L_{t, i_{0} j_{0}}^{4}$ in \eqref{estimate 197}, which has eight terms. We label the first four terms as $L_{t, i_{0} j_{0}}^{41}$ and the last four terms as $L_{t, i_{0} j_{0}}^{42}$ due to necessity (see the coupled renormalization explained in \cite[Remark 3.4]{Y19a}); i.e., 
\begin{subequations}\label{estimate 125}
\begin{align}
L_{t, i_{0} j_{0}}^{41} \triangleq& (2\pi)^{-\frac{9}{2}} \sum_{\lvert i -j \rvert \leq 1} \sum_{i_{1}, i_{2}, i_{3}, j_{1} =1}^{3} \sum_{k} \sum_{k_{1}, k_{2}, k_{3} \neq 0: k_{12} = k} \theta(2^{-i} k_{123}) \theta(2^{-j} k_{3}) \\
& \times [ \int_{0}^{t} e^{- \lvert k_{123} \rvert^{2} f(\epsilon k_{123})(t-s) } \int_{0}^{s} \nonumber\\
& \hspace{5mm} \times (: \hat{X}_{\sigma, u}^{\epsilon, i_{2}}(k_{1}) \hat{X}_{\sigma, u}^{\epsilon, i_{3}}(k_{2}): - : \hat{X}_{\sigma, b}^{\epsilon, i_{2}}(k_{1}) \hat{X}_{\sigma, b}^{\epsilon, i_{3}}(k_{2}): ) k_{12}^{i_{3}} g(\epsilon k_{12}^{i_{3}}) k_{123}^{j_{1}} g(\epsilon k_{123}^{j_{1}}) \nonumber\\
& \hspace{5mm} \times \frac{ e^{- \lvert k_{3} \rvert^{2} f(\epsilon k_{3}) (t-s)} h_{u}(\epsilon k_{3}) h_{b}(\epsilon k_{3})}{2\lvert k_{3} \rvert^{2} f(\epsilon k_{3})} e^{- \lvert k_{12} \rvert^{2}f(\epsilon k_{12}) (s-\sigma) } d\sigma ds \nonumber\\
& - \int_{0}^{t} e^{- \lvert k_{123} \rvert^{2} (t-s) } \int_{0}^{s} \nonumber\\
& \hspace{5mm} \times (: \hat{\bar{X}}_{\sigma, u}^{\epsilon, i_{2}}(k_{1}) \hat{\bar{X}}_{\sigma, u}^{\epsilon, i_{3}}(k_{2}): - : \hat{\bar{X}}_{\sigma, b}^{\epsilon, i_{2}}(k_{1}) \hat{\bar{X}}_{\sigma, b}^{\epsilon, i_{3}}(k_{2}):) k_{12}^{i_{3}} i k_{123}^{j_{1}} i \nonumber\\
& \hspace{5mm} \times \frac{ e^{- \lvert k_{3} \rvert^{2} (t-s)} h_{u}(\epsilon k_{3}) h_{b}(\epsilon k_{3})}{2\lvert k_{3} \rvert^{2}} e^{- \lvert k_{12} \rvert^{2} (s-\sigma)} d\sigma ds] \nonumber\\
& \times \sum_{i_{4} =1}^{3} \hat{\mathcal{P}}^{j_{1}i_{4}} (k_{3}) \hat{\mathcal{P}}^{j_{0}i_{4}}(k_{3}) \hat{\mathcal{P}}^{i_{1} i_{2}}(k_{12}) \hat{\mathcal{P}}^{i_{0}i_{1}}(k_{123}) e_{k}, \nonumber\\
L_{t, i_{0} j_{0}}^{42} \triangleq& (2\pi)^{-\frac{9}{2}} \sum_{\lvert i -j \rvert \leq 1} \sum_{i_{1}, i_{2}, i_{3}, j_{1} =1}^{3} \sum_{k} \sum_{k_{1}, k_{2}, k_{3}\neq 0: k_{12} = k} \theta(2^{-i} k_{123}) \theta(2^{-j} k_{3}) \\
& \times [ - \int_{0}^{t} e^{- \lvert k_{123} \rvert^{2}  f(\epsilon k_{123})(t-s)} \int_{0}^{s}  \nonumber\\
& \hspace{5mm} \times (: \hat{X}_{\sigma, b}^{\epsilon, i_{2}}(k_{1}) \hat{X}_{\sigma, u}^{\epsilon, i_{3}}(k_{2}): - : \hat{X}_{\sigma, u}^{\epsilon, i_{2}}(k_{1}) \hat{X}_{\sigma, b}^{\epsilon, i_{3}}(k_{2}):) k_{12}^{i_{3}} g(\epsilon k_{12}^{i_{3}}) k_{123}^{j_{1}} g(\epsilon k_{123}^{j_{1}}) \nonumber\\
& \hspace{5mm} \times \frac{ e^{- \lvert k_{3} \rvert^{2} f(\epsilon k_{3}) (t-s)}  h_{b}(\epsilon k_{3})^{2}}{2\lvert k_{3} \rvert^{2} f(\epsilon k_{3})} e^{- \lvert k_{12} \rvert^{2} f(\epsilon k_{12})(s-\sigma)} d\sigma ds \nonumber\\
& + \int_{0}^{t} e^{- \lvert k_{123} \rvert^{2} (t-s) } \int_{0}^{s} \nonumber\\
& \hspace{5mm} \times (: \hat{\bar{X}}_{\sigma, b}^{\epsilon, i_{2}}(k_{1}) \hat{\bar{X}}_{\sigma, u}^{\epsilon, i_{3}}(k_{2}): - : \hat{\bar{X}}_{\sigma, u}^{\epsilon, i_{2}}(k_{1}) \hat{\bar{X}}_{\sigma, b}^{\epsilon, i_{3}}(k_{2}):) k_{12}^{i_{3}} i k_{123}^{j_{1}} i \nonumber\\
& \hspace{5mm} \times \frac{ e^{- \lvert k_{3} \rvert^{2} (t-s)} h_{b}(\epsilon k_{3})^{2}}{2\lvert k_{3} \rvert^{2}} e^{- \lvert k_{12} \rvert^{2} (s-\sigma)} d\sigma ds] \nonumber\\
& \times \sum_{i_{4} =1}^{3} \hat{\mathcal{P}}^{j_{1}i_{4}} (k_{3}) \hat{\mathcal{P}}^{j_{0}i_{4}}(k_{3}) \hat{\mathcal{P}}^{i_{1} i_{2}}(k_{12}) \hat{\mathcal{P}}^{i_{0}i_{1}}(k_{123}) e_{k}. \nonumber
\end{align} 
\end{subequations} 
Due to similarity, we show estimate s for only $L_{t, i_{0} j_{0}}^{41}$ of \eqref{estimate 125}. We define 
\begin{align}\label{estimate 216}
\tilde{L}_{t, i_{0}j_{0}}^{41} \triangleq& (2\pi)^{-\frac{9}{2}} \sum_{\lvert i -j \rvert \leq 1} \sum_{i_{1}, i_{2}, i_{3}, j_{1} =1}^{3} \sum_{k} \sum_{k_{1}, k_{2}, k_{3}\neq 0: k_{12} = k} \theta(2^{-i} k_{123}) \theta(2^{-j} k_{3}) \\
& \times [ \int_{0}^{t} e^{- \lvert k_{123} \rvert^{2}  f(\epsilon k_{123})(t-s)} \int_{0}^{t} \nonumber\\
& \hspace{5mm} \times (: \hat{X}_{\sigma, u}^{\epsilon, i_{2}}(k_{1}) \hat{X}_{\sigma, u}^{\epsilon, i_{3}}(k_{2}): - : \hat{X}_{\sigma, b}^{\epsilon, i_{2}}(k_{1}) \hat{X}_{\sigma, b}^{\epsilon, i_{3}}(k_{2}): ) k_{12}^{i_{3}} g(\epsilon k_{12}^{i_{3}}) k_{123}^{j_{1}} g(\epsilon k_{123}^{j_{1}}) \nonumber\\
& \hspace{5mm} \times \frac{ e^{- \lvert k_{3} \rvert^{2} f(\epsilon k_{3}) (t-s)} h_{u}(\epsilon k_{3}) h_{b}(\epsilon k_{3})}{2\lvert k_{3} \rvert^{2} f(\epsilon k_{3})} e^{- \lvert k_{12} \rvert^{2}  f(\epsilon k_{12})(t-\sigma)} d\sigma ds \nonumber\\
& - \int_{0}^{t} e^{- \lvert k_{123} \rvert^{2} (t-s) } \int_{0}^{t} \nonumber\\
& \hspace{5mm} \times (: \hat{\bar{X}}_{\sigma, u}^{\epsilon, i_{2}}(k_{1}) \hat{\bar{X}}_{\sigma, u}^{\epsilon, i_{3}}(k_{2}): - : \hat{\bar{X}}_{\sigma, b}^{\epsilon, i_{2}}(k_{1}) \hat{\bar{X}}_{\sigma, b}^{\epsilon, i_{3}}(k_{2}):) k_{12}^{i_{3}} i k_{123}^{j_{1}} i \nonumber\\
& \hspace{5mm} \times \frac{ e^{- \lvert k_{3} \rvert^{2} (t-s)} h_{u}(\epsilon k_{3}) h_{b}(\epsilon k_{3})}{2\lvert k_{3} \rvert^{2}} e^{- \lvert k_{12} \rvert^{2} (t-\sigma)} d\sigma ds] \nonumber\\
& \times \sum_{i_{4} =1}^{3} \hat{\mathcal{P}}^{j_{1}i_{4}} (k_{3}) \hat{\mathcal{P}}^{j_{0}i_{4}}(k_{3}) \hat{\mathcal{P}}^{i_{1} i_{2}}(k_{12}) \hat{\mathcal{P}}^{i_{0}i_{1}}(k_{123}) e_{k}, \nonumber
\end{align} 
while we rely on \eqref{estimate 124}, \eqref{C2bepsilonii1j} and the fact that $\sum_{\lvert i-j \rvert \leq 1} \theta(2^{-i} k_{3}) \theta(2^{-j} k_{3}) = 1$ to write 
\begin{align}\label{estimate 225}
C_{3,u}^{\epsilon, i_{0} i_{1} j_{0}} &=\frac{ (2\pi)^{-3}}{2} \sum_{\lvert i-j \rvert \leq 1} \sum_{k_{3} \neq 0} \sum_{j_{1} =1}^{3} \theta(2^{-i} k_{3}) \theta(2^{-j} k_{3}) k_{3}^{j_{1}} g(\epsilon k_{3}^{j_{1}})  \\
& \times \int_{0}^{t} \frac{ e^{-2 \lvert k_{3} \rvert^{2} f(\epsilon k_{3}) (t-s)} h_{u}(\epsilon k_{3}) h_{b}(\epsilon k_{3})}{2 \lvert k_{3}\rvert^{2} f(\epsilon k_{3})} \sum_{i_{4} =1}^{3} \hat{\mathcal{P}}^{j_{1} i_{4}}(k_{3}) \hat{\mathcal{P}}^{j_{0} i_{4}}(k_{3}) \hat{\mathcal{P}}^{i_{0} i_{1}} (k_{3}) ds. \nonumber
\end{align} 
Similarly, identically from \eqref{barC2bepsilonii1j} because $\bar{C}_{3,u}^{\epsilon, ii_{1} j}(t) = \bar{C}_{2,b}^{\epsilon, ii_{1} j}(t)$ as mentioned after \eqref{estimate 124}, we may write 
\begin{align}\label{estimate 228}
\bar{C}_{3,u}^{\epsilon, i_{0} i_{1} j_{0}}(t) \triangleq& \frac {(2\pi)^{-3}}{2} \sum_{j_{1}, i_{4} =1}^{3}\sum_{\lvert i-j \rvert \leq 1}   \sum_{k_{3} \neq 0}\theta(2^{-i} k_{3}) \theta(2^{-j} k_{3})  \int_{0}^{t} e^{-2 \lvert k_{3} \rvert^{2} (t-s)} ds   \nonumber \\ 
& \times k_{3}^{j_{1}} i \frac{h_{u} (\epsilon k_{3}) h_{b}(\epsilon k_{3})}{2 \lvert k_{3} \rvert^{2}} \hat{\mathcal{P}}^{i_{0}i_{1}}(k_{3}) \hat{\mathcal{P}}^{j_{1}i_{4}} (k_{3}) \hat{\mathcal{P}}^{j_{0}i_{4}}(k_{3}), 
\end{align} 
which has already been explained to vanish after \eqref{barC2bepsilonii1j} so that we may write 
\begin{align}\label{estimate 246}
&L_{t, i_{0} j_{0}}^{41} + 4 \sum_{i_{1} =1}^{3} u_{2}^{\epsilon, i_{1}}(t) C_{3,u}^{\epsilon, i_{0} i_{1} j_{0}}(t) \\
=& L_{t, i_{0} j_{0}}^{41} - \tilde{L}_{t, i_{0} j_{0}}^{41} + \tilde{L}_{t, i_{0} j_{0}}^{41} + 4 \sum_{i_{1} =1}^{3} u_{2}^{\epsilon, i_{1}}(t) C_{3,u}^{\epsilon, i_{0} i_{1} j_{0}}(t) - 4 \sum_{i_{1} =1}^{3} \bar{u}_{2}^{\epsilon, i_{1}}(t) \bar{C}_{3,u}^{\epsilon, i_{0} i_{1} j_{0}}(t), \nonumber
\end{align} 
where the multiplication by 4 is due to the multiplication by $\frac{1}{4}$ in \eqref{[Equation (4.6e)][ZZ17]}. We will see in \eqref{estimate 223} - \eqref{estimate 224} that our choice of $C_{3, u}^{\epsilon, i_{0} i_{1} j_{0}}$ and $\bar{C}_{3,u}^{\epsilon, i_{0}i_{1}j_{0}}$ in \eqref{estimate 225}- \eqref{estimate 228} which were originally derived from \eqref{estimate 124},  \eqref{C2bepsilonii1j} and \eqref{barC2uepsilonii1j} by necessity for the convergence of $b_{1}^{\epsilon} \diamond b_{2}^{\epsilon} - \bar{b}_{1}^{\epsilon}  \diamond \bar{b}_{2}^{\epsilon}$ makes perfect cancellations with $\tilde{L}_{t,i_{0}j_{0}}^{41}$ defined in \eqref{estimate 216}.  Now we define 
\begin{subequations}
\begin{align}
l_{k_{123}k_{3}, t-s, i_{0} j_{1}} \triangleq& \sum_{i_{1} =1}^{3} e^{- \lvert k_{123} \rvert^{2} f(\epsilon k_{123}) (t-s)} \frac{ e^{- \lvert k_{3} \rvert^{2} f(\epsilon k_{3})(t-s)} h_{u}(\epsilon k_{3}) h_{b}(\epsilon k_{3})}{\lvert k_{3} \rvert^{2} f(\epsilon k_{3})} \nonumber\\
& \times k_{123}^{j_{1}} g(\epsilon k_{123}^{j_{1}}) \hat{\mathcal{P}}^{i_{0} i_{1}} (k_{123}),  \label{[Equation (4.7d)][ZZ17]}\\
\bar{l}_{k_{123}k_{3}, t-s, i_{0} j_{1}}  \triangleq& \sum_{i_{1} =1}^{3} e^{- \lvert k_{123} \rvert^{2} (t-s)} \frac{ e^{- \lvert k_{3} \rvert^{2} (t-s) } h_{u}(\epsilon k_{3}) h_{b}(\epsilon k_{3})}{\lvert k_{3} \rvert^{2}} \nonumber\\
& \times k_{123}^{j_{1}} i \hat{\mathcal{P}}^{i_{0} i_{1}} (k_{123}), \label{[Equation (4.7e)][ZZ17]}
\end{align}
\end{subequations} 
and compute from \eqref{estimate 125} and \eqref{estimate 216} 
\begin{align}
& \mathbb{E} [ \lvert \Delta_{q} (L_{t, i_{0} j_{0}}^{41} - \tilde{L}_{t, i_{0} j_{0}}^{41} ) \rvert^{2}] \\
&\lesssim \mathbb{E} [ \lvert \sum_{k} \theta(2^{-q}k) \sum_{\lvert i-j \rvert \leq 1} \sum_{k_{1}, k_{2}, k_{3} \neq 0: k_{12} = k} \sum_{i_{2}, i_{3}, i_{4}, j_{1} =1}^{3} \theta(2^{-i} k_{123}) \theta(2^{-j} k_{3}) \nonumber\\ 
& \times  \hat{\mathcal{P}}^{j_{1} i_{4}}(k_{3}) \hat{\mathcal{P}}^{j_{0} i_{4}}(k_{3}) e_{k} \nonumber \\
& \times [\int_{0}^{t} (l_{k_{123}k_{3}, t-s, i_{0} j_{1}} - \bar{l}_{k_{123}k_{3}, t-s, i_{0} j_{1}})  \int_{0}^{s} [ : \hat{X}_{\sigma, u}^{\epsilon, i_{2}}(k_{1}) \hat{X}_{\sigma, u}^{\epsilon, i_{3}}(k_{2}):  \nonumber\\
& \hspace{5mm} - : \hat{X}_{\sigma, b}^{\epsilon, i_{2}}(k_{1}) \hat{X}_{\sigma, b}^{\epsilon, i_{3}}(k_{2}):]  [e^{- \lvert k_{12} \rvert^{2}  f(\epsilon k_{12})(s-\sigma)} - e^{- \lvert k_{12} \rvert^{2} f(\epsilon k_{12}) (t-\sigma)}] d\sigma k_{12}^{i_{3}} g(\epsilon k_{12}^{i_{3}}) ds \nonumber \\
&+ \int_{0}^{t} \bar{l}_{k_{123}k_{3}, t-s, i_{0} j_{1}}\int_{0}^{s} [ : \hat{X}_{\sigma, u}^{\epsilon, i_{2}}(k_{1}) \hat{X}_{\sigma, u}^{\epsilon, i_{3}}(k_{2}): - : \hat{X}_{\sigma, b}^{\epsilon, i_{2}}(k_{1}) \hat{X}_{\sigma, b}^{\epsilon, i_{3}}(k_{2}):] \nonumber\\
& \hspace{10mm} \times [e^{- \lvert k_{12} \rvert^{2}  f(\epsilon k_{12})(s-\sigma)} - e^{- \lvert k_{12} \rvert^{2}  f(\epsilon k_{12})(t-\sigma)}] d\sigma k_{12}^{i_{3}} g(\epsilon k_{12}^{i_{3}}) ds\nonumber\\
& - \int_{0}^{t} (l_{k_{123}k_{3}, t-s, i_{0} j_{1}} - \bar{l}_{k_{123}k_{3}, t-s, i_{0} j_{1}})  \int_{s}^{t} [ : \hat{X}_{\sigma, u}^{\epsilon, i_{2}}(k_{1}) \hat{X}_{\sigma, u}^{\epsilon, i_{3}}(k_{2}):  \nonumber\\
& \hspace{38mm}  - : \hat{X}_{\sigma, b}^{\epsilon, i_{2}}(k_{1}) \hat{X}_{\sigma, b}^{\epsilon, i_{3}}(k_{2}):]   e^{- \lvert k_{12} \rvert^{2}  f(\epsilon k_{12})(t-\sigma)} d\sigma k_{12}^{i_{3}} g(\epsilon k_{12}^{i_{3}}) ds \nonumber \\
&- \int_{0}^{t}  \bar{l}_{k_{123}k_{3}, t-s, i_{0} j_{1}} \int_{s}^{t} [ : \hat{X}_{\sigma, u}^{\epsilon, i_{2}}(k_{1}) \hat{X}_{\sigma, u}^{\epsilon, i_{3}}(k_{2}): - : \hat{X}_{\sigma, b}^{\epsilon, i_{2}}(k_{1}) \hat{X}_{\sigma, b}^{\epsilon, i_{3}}(k_{2}):]  \nonumber\\
& \hspace{45mm} \times e^{- \lvert k_{12} \rvert^{2}  f(\epsilon k_{12})(t-\sigma)} d\sigma k_{12}^{i_{3}} g(\epsilon k_{12}^{i_{3}}) ds\nonumber\\
& -  \int_{0}^{t}  \bar{l}_{k_{123}k_{3}, t-s, i_{0} j_{1}}\int_{0}^{s} [ : \hat{\bar{X}}_{\sigma, u}^{\epsilon, i_{2}}(k_{1}) \hat{\bar{X}}_{\sigma, u}^{\epsilon, i_{3}}(k_{2}): - : \hat{\bar{X}}_{\sigma, b}^{\epsilon, i_{2}}(k_{1}) \hat{\bar{X}}_{\sigma, b}^{\epsilon, i_{3}}(k_{2}):]   \nonumber\\
& \hspace{45mm} \times  [ e^{- \lvert k_{12} \rvert^{2} (s-\sigma) } - e^{-\lvert k_{12} \rvert^{2}(t-\sigma)}] d\sigma k_{12}^{i_{3}} i ds \nonumber \\
& +  \int_{0}^{t}  \bar{l}_{k_{123}k_{3}, t-s, i_{0} j_{1}} \int_{s}^{t} [ : \hat{\bar{X}}_{\sigma, u}^{\epsilon, i_{2}}(k_{1}) \hat{\bar{X}}_{\sigma, u}^{\epsilon, i_{3}}(k_{2}):  \nonumber\\
& \hspace{45mm} - : \hat{\bar{X}}_{\sigma, b}^{\epsilon, i_{2}}(k_{1}) \hat{\bar{X}}_{\sigma, b}^{\epsilon, i_{3}}(k_{2}):]  e^{-\lvert k_{12} \rvert^{2}(t-\sigma) } d\sigma k_{12}^{i_{3}} i  ds ] \rvert^{2}] \nonumber
\end{align} 
by \eqref{[Equation (4.7d)][ZZ17]} and \eqref{[Equation (4.7e)][ZZ17]}. Finally, we take the difference between the second and fifth terms and then fourth and sixth terms to obtain 
\begin{align}\label{estimate 237}
\mathbb{E} [ \lvert \Delta_{q} ( L_{t, i_{0} j_{0}}^{41} - \tilde{L}_{t, i_{0} j_{0}}^{41}) \rvert^{2} ]  \lesssim  \sum_{i=1}^{6} M_{q,t, i_{0}}^{i} 
\end{align} 
where 
\begin{subequations}\label{estimate 226}
\begin{align}
M_{q,t, i_{0}}^{1} &\triangleq \mathbb{E} [ \lvert \sum_{k} \theta(2^{-q}k) \sum_{\lvert i-j \rvert \leq 1} \sum_{k_{1}, k_{2}, k_{3} \neq 0: k_{12} = k} \sum_{i_{2}, i_{3}, j_{1} =1}^{3} \theta(2^{-i} k_{123}) \theta(2^{-j} k_{3})   \\
& \times e_{k} \int_{0}^{t}( l_{k_{123}k_{3}, t-s, i_{0} j_{1}} - \bar{l}_{k_{123} k_{3}, t-s, i_{0} j_{1}}) \nonumber\\
& \hspace{5mm} \times  \int_{0}^{s}  [ : \hat{X}_{\sigma, u}^{\epsilon, i_{2}}(k_{1}) \hat{X}_{\sigma, u}^{\epsilon, i_{3}}(k_{2}): - : \hat{X}_{\sigma, b}^{\epsilon, i_{2}}(k_{1}) \hat{X}_{\sigma, b}^{\epsilon, i_{3}}(k_{2}):] \nonumber\\
& \hspace{10mm} \times [ e^{- \lvert k_{12} \rvert^{2} f(\epsilon k_{12}) (s-\sigma)} - e^{- \lvert k_{12} \rvert^{2} f(\epsilon k_{12}) (t-\sigma)}] d\sigma k_{12}^{i_{3}} g(\epsilon k_{12}^{i_{3}}) ds  \rvert^{2} ], \nonumber\\
M_{q, t, i_{0}}^{2} &\triangleq \mathbb{E} [ \lvert \sum_{k} \theta(2^{-q}k) \sum_{\lvert i-j \rvert \leq 1} \sum_{k_{1}, k_{2}, k_{3}\neq 0: k_{12} = k} \sum_{i_{2}, i_{3}, j_{1} =1}^{3} \theta(2^{-i} k_{123}) \theta(2^{-j} k_{3}) \\
& \times e_{k} \int_{0}^{t} (l_{k_{123}k_{3}, t-s, i_{0} j_{1}} - \bar{l}_{k_{123}k_{3}, t-s, i_{0} j_{1}}) \nonumber\\
& \times \int_{s}^{t}  [ : \hat{X}_{\sigma, u}^{\epsilon, i_{2}}(k_{1}) \hat{X}_{\sigma, u}^{\epsilon, i_{3}}(k_{2}): - : \hat{X}_{\sigma, b}^{\epsilon, i_{2}}(k_{1}) \hat{X}_{\sigma, b}^{\epsilon, i_{3}}(k_{2}):] \nonumber\\
& \hspace{5mm} \times  e^{- \lvert k_{12} \rvert^{2} f(\epsilon k_{12}) (t-\sigma)} d\sigma k_{12}^{i_{3}} g(\epsilon k_{12}^{i_{3}}) ds \rvert^{2} ], \nonumber\\
M_{q,t, i_{0}}^{3} &\triangleq \mathbb{E} [ \lvert \sum_{k} \theta(2^{-q}k) \sum_{\lvert i-j \rvert \leq 1} \sum_{k_{1}, k_{2}, k_{3} \neq 0: k_{12} = k} \sum_{i_{2}, i_{3}, j_{1} =1}^{3} \theta(2^{-i} k_{123}) \theta(2^{-j} k_{3}) \\
& \times e_{k} \int_{0}^{t} \bar{l}_{k_{123}k_{3}, t-s, i_{0} j_{1}}  \int_{0}^{s} [: \hat{X}_{\sigma, u}^{\epsilon, i_{2}} (k_{1}) \hat{X}_{\sigma, u}^{\epsilon, i_{3}}(k_{2}): - :\hat{X}_{\sigma, b}^{\epsilon, i_{2}}(k_{1}) \hat{X}_{\sigma, b}^{\epsilon, i_{3}}(k_{2}):]  \nonumber \\
& \hspace{5mm} \times [ e^{- \lvert k_{12} \rvert^{2}  f(\epsilon k_{12})(s-\sigma)} k_{12}^{i_{3}} g(\epsilon k_{12}^{i_{3}}) - e^{- \lvert k_{12} \rvert^{2}  f(\epsilon k_{12})(t-\sigma)} k_{12}^{i_{3}} g(\epsilon k_{12}^{i_{3}}) \nonumber\\
& \hspace{20mm}  - e^{- \lvert k_{12} \rvert^{2} (s-\sigma) } k_{12}^{i_{3}} i + e^{- \lvert k_{12} \rvert^{2} (t-\sigma) } k_{12}^{i_{3}} i] d\sigma ds \rvert^{2} ], \nonumber\\
M_{q,t,i_{0}}^{4} &\triangleq \mathbb{E} [ \lvert \sum_{k} \theta(2^{-q}k) \sum_{\lvert i-j \rvert \leq 1} \sum_{k_{1}, k_{2}, k_{3} \neq 0: k_{12} = k} \sum_{i_{2}, i_{3}, j_{1} =1}^{3} \theta(2^{-i} k_{123}) \theta(2^{-j} k_{3}) \\
& \times e_{k} \int_{0}^{t}  \bar{l}_{k_{123}k_{3}, t-s, i_{0} j_{1}} \int_{0}^{s} [: \hat{X}_{\sigma, u}^{\epsilon, i_{2}} (k_{1}) \hat{X}_{\sigma, u}^{\epsilon, i_{3}}(k_{2}): - :\hat{X}_{\sigma, b}^{\epsilon, i_{2}}(k_{1}) \hat{X}_{\sigma, b}^{\epsilon, i_{3}}(k_{2}):  \nonumber\\
& \hspace{35mm} - : \hat{\bar{X}}_{\sigma, u}^{\epsilon, i_{2}} (k_{1}) \hat{\bar{X}}_{\sigma, u}^{\epsilon, i_{3}} (k_{2}): + :\hat{\bar{X}}_{\sigma, b}^{\epsilon, i_{2}}(k_{1}) \hat{\bar{X}}_{\sigma, b}^{\epsilon, i_{3}} (k_{2}):]  \nonumber\\
& \hspace{32mm} \times  [e^{- \lvert k_{12} \rvert^{2} (s-\sigma) } - e^{- \lvert k_{12} \rvert^{2} (t-\sigma) }]  d\sigma  k_{12}^{i_{3}} i ds \rvert^{2} ], \nonumber \\
M_{q, t, i_{0}}^{5} &\triangleq \mathbb{E} [ \lvert \sum_{k} \theta(2^{-q}k) \sum_{\lvert i-j \rvert \leq 1} \sum_{k_{1}, k_{2}, k_{3} \neq 0: k_{12} = k} \sum_{i_{2}, i_{3}, j_{1} =1}^{3} \theta(2^{-i} k_{123}) \theta(2^{-j} k_{3})\\
& \times e_{k}  \int_{0}^{t} \bar{l}_{k_{123}k_{3}, t-s, i_{0} j_{1}}  \int_{s}^{t} [: \hat{X}_{\sigma, u}^{\epsilon, i_{2}} (k_{1}) \hat{X}_{\sigma, u}^{\epsilon, i_{3}}(k_{2}): - :\hat{X}_{\sigma, b}^{\epsilon, i_{2}}(k_{1}) \hat{X}_{\sigma, b}^{\epsilon, i_{3}}(k_{2}):]  \nonumber \\
& \hspace{10mm} \times [ e^{- \lvert k_{12} \rvert^{2}  f(\epsilon k_{12})(t-\sigma)} g(\epsilon k_{12}^{i_{3}}) - e^{- \lvert k_{12} \rvert^{2} (t-\sigma) } i] k_{12}^{i_{3}} d\sigma ds \rvert^{2} ], \nonumber\\
M_{q,t, i_{0}}^{6} &\triangleq \mathbb{E} [ \lvert \sum_{k} \theta(2^{-q}k) \sum_{\lvert i-j \rvert \leq 1} \sum_{k_{1}, k_{2}, k_{3}\neq 0: k_{12} = k} \sum_{i_{2}, i_{3}, j_{1} =1}^{3} \theta(2^{-i} k_{123}) \theta(2^{-j} k_{3}) \\
& \times e_{k} \int_{0}^{t} \bar{l}_{k_{123}k_{3}, t-s, i_{0} j_{1}}  \int_{s}^{t} [: \hat{X}_{\sigma, u}^{\epsilon, i_{2}} (k_{1}) \hat{X}_{\sigma, u}^{\epsilon, i_{3}}(k_{2}): - :\hat{X}_{\sigma, b}^{\epsilon, i_{2}}(k_{1}) \hat{X}_{\sigma, b}^{\epsilon, i_{3}}(k_{2}):  \nonumber\\
& \hspace{1mm}- : \hat{\bar{X}}_{\sigma, u}^{\epsilon, i_{2}} (k_{1}) \hat{\bar{X}}_{\sigma, u}^{\epsilon, i_{3}} (k_{2}):   + :\hat{\bar{X}}_{\sigma, b}^{\epsilon, i_{2}}(k_{1}) \hat{\bar{X}}_{\sigma, b}^{\epsilon, i_{3}} (k_{2}):]  e^{- \lvert k_{12} \rvert^{2} (t-\sigma) }  d\sigma  k_{12}^{i_{3}} i ds \rvert^{2} ]. \nonumber
\end{align}
\end{subequations} 
In order to compute $M_{q, t, i_{0}}^{1}$ of \eqref{estimate 226}, we estimate
\begin{align}\label{estimate 227}
& \mathbb{E} [ ( : \hat{X}_{\sigma, u}^{\epsilon, i_{2}} (k_{1}) \hat{X}_{\sigma, u}^{\epsilon, i_{3}}(k_{2}): - : \hat{X}_{\sigma, b}^{\epsilon, i_{2}} (k_{1}) \hat{X}_{\sigma, b}^{\epsilon, i_{3}}(k_{2}): ) \nonumber\\
& \hspace{5mm} \times \overline{( : \hat{X}_{\bar{\sigma}, u}^{\epsilon, i_{2}'} (k_{1}') \hat{X}_{\bar{\sigma}, u}^{\epsilon, i_{3}'}(k_{2}'): - : \hat{X}_{\bar{\sigma}, b}^{\epsilon, i_{2}'} (k_{1}') \hat{X}_{\bar{\sigma}, b}^{\epsilon, i_{3}'}(k_{2}'): ) } ] \nonumber\\
& \hspace{10mm} \lesssim ( 1_{k_{1} = k_{1}', k_{2} = k_{2}'} + 1_{k_{1} = k_{2}', k_{2} = k_{1}'}) \frac{1}{ \prod_{i=1}^{2} \lvert k_{i} \rvert^{2}}               
\end{align} 
by Example \ref{Example 3.1} and \eqref{covariance}. Applying \eqref{estimate 227} to \eqref{estimate 226} gives 
\begin{align}\label{estimate 8}
M_{q,t, i_{0}}^{1} &\lesssim \sum_{k} \theta(2^{-q} k)^{2} \sum_{\lvert i-j \rvert \leq 1, \lvert i'-j' \rvert \leq 1} \sum_{k_{1}, k_{2}, k_{3}, k_{4} \neq 0: k_{12} = k} \sum_{j_{1}, j_{1}' =1}^{3} \\
& \times  \theta (2^{-i} k_{123}) \theta(2^{-i'} k_{124}) \theta(2^{-j} k_{3}) \theta(2^{-j'} k_{4}) \nonumber\\
& \times  \int_{[0,t]^{2}} \lvert l_{k_{123} k_{3}, t-s, i_{0} j_{1}} - \bar{l}_{k_{123}k_{3}, t-s, i_{0} j_{1}} \rvert \lvert l_{k_{124} k_{4}, t-\bar{s}, i_{0} j_{1}'} - \bar{l}_{k_{124}k_{4}, t-\bar{s}, i_{0} j_{1}'} \rvert \nonumber \\
& \hspace{5mm} \times \int_{0}^{s} \int_{0}^{\bar{s}} \lvert e^{-\lvert k_{12} \rvert^{2} f(\epsilon k_{12})(s-\sigma)} - e^{- \lvert k_{12} \rvert^{2} f(\epsilon k_{12} )(t - \sigma)} \rvert  \nonumber\\ 
& \hspace{15mm} \times \lvert e^{-\lvert k_{12} \rvert^{2} f(\epsilon k_{12})( \bar{s}-\bar{\sigma})} - e^{- \lvert k_{12} \rvert^{2} f(\epsilon k_{12} )(t - \bar{\sigma})} \rvert \frac{ \lvert k_{12} \rvert^{2}}{\lvert k_{1} \rvert^{2} \lvert k_{2} \rvert^{2}} d \bar{\sigma}  d \sigma ds d \bar{s}.   \nonumber
\end{align} 
Similarly, we obtain from \eqref{estimate 226}
\begin{align}\label{estimate 9}
M_{q,t, i_{0}}^{2} &\lesssim \sum_{k} \theta(2^{-q} k)^{2} \sum_{\lvert i-j \rvert \leq 1, \lvert i' - j' \rvert \leq 1} \sum_{k_{1}, k_{2}, k_{3}, k_{4}\neq 0: k_{12} = k} \sum_{j_{1}, j_{1}' =1}^{3} \\
& \times \theta(2^{-i} k_{123}) \theta(2^{-i'} k_{124})  \theta(2^{-j} k_{3}) \theta(2^{-j'} k_{4}) \nonumber\\
& \times  \int_{[0,t]^{2}}  \lvert l_{k_{123} k_{3}, t-s, i_{0} j_{1}} - \bar{l}_{k_{123}k_{3}, t-s, i_{0} j_{1}} \rvert \lvert l_{k_{124} k_{4}, t-\bar{s}, i_{0} j_{1}'} - \bar{l}_{k_{124}k_{4}, t-\bar{s}, i_{0} j_{1}'} \rvert \nonumber\\
& \times \int_{s}^{t} \int_{\bar{s}}^{t} e^{- \lvert k_{12} \rvert^{2} f(\epsilon k_{12} )(t-\sigma)} e^{- \lvert k_{12} \rvert^{2} f(\epsilon k_{12}) (t- \bar{\sigma})} d \bar{\sigma} d\sigma \frac{ \lvert k_{12} \rvert^{2}}{\lvert k_{1} \rvert^{2} \lvert k_{2} \rvert^{2}} dsd \bar{s}, \nonumber
\end{align} 
\begin{align}\label{estimate 10}
&M_{q,t, i_{0}}^{3} \lesssim \sum_{k} \theta(2^{-q} k)^{2} \sum_{\lvert i-j \rvert \leq 1, \lvert i' -j' \rvert \leq 1} \sum_{k_{1}, k_{2}, k_{3}, k_{4}\neq 0: k_{12} = k} \sum_{i_{3}, j_{1}, i_{3}', j_{1}' = 1}^{3}  \\
& \hspace{5mm}  \times \theta(2^{-i} k_{123}) \theta(2^{-i'} k_{124})  \theta(2^{-j} k_{3}) \theta(2^{-j'} k_{4}) \int_{[0,t]^{2}}  \bar{l}_{k_{123}k_{3}, t-s, i_{0} j_{1}} \bar{l}_{k_{124}k_{4}, t-\bar{s}, i_{0} j_{1}'} \nonumber\\
& \hspace{5mm}  \times  \int_{0}^{s} \int_{0}^{\bar{s}} \frac{1}{\lvert k_{1} \rvert^{2} \lvert k_{2} \rvert^{2}}  \lvert e^{- \lvert k_{12} \rvert^{2}  f(\epsilon k_{12})(s-\sigma)} k_{12}^{i_{3}} g(\epsilon k_{12}^{i_{3}}) - e^{- \lvert k_{12} \rvert^{2}  f(\epsilon k_{12})(t-\sigma) } k_{12}^{i_{3}} g(\epsilon k_{12}^{i_{3}}) \nonumber \\
& \hspace{35mm} - e^{- \lvert k_{12}\rvert^{2} (s-\sigma)} k_{12}^{i_{3}} i + e^{- \lvert k_{12} \rvert^{2} (t-\sigma) } k_{12}^{i_{3}} i \rvert \nonumber \\
& \hspace{30mm} \times  \lvert e^{- \lvert k_{12} \rvert^{2}  f(\epsilon k_{12} )(\bar{s}-\bar{\sigma})} k_{12}^{i_{3}'} g(\epsilon k_{12}^{i_{3}'}) - e^{- \lvert k_{12} \rvert^{2}  f(\epsilon k_{12})(t-\bar{\sigma}) } k_{12}^{i_{3}'} g(\epsilon k_{12}^{i_{3}'}) \nonumber \\
& \hspace{35mm} - e^{- \lvert k_{12}\rvert^{2} (\bar{s}-\bar{\sigma})} k_{12}^{i_{3}'} i + e^{- \lvert k_{12} \rvert^{2} (t- \bar{\sigma}) } k_{12}^{i_{3}'} i \rvert  d \bar{\sigma} d \sigma ds d \bar{s}. \nonumber
\end{align} 
In order to estimate $M_{q, t, i_{0}}^{4}$, we observe that for any $\eta \in [0,1]$, 
\begin{align}\label{estimate 7}
& \mathbb{E} [ ( : \hat{X}_{\sigma, u}^{\epsilon, i_{2}}(k_{1}) \hat{X}_{\sigma, u}^{\epsilon, i_{3}} (k_{2}): - : \hat{X}_{\sigma, b}^{\epsilon, i_{2}} (k_{1}) \hat{X}_{\sigma, b}^{\epsilon, i_{3}} (k_{2}): \nonumber\\
& \hspace{10mm} - : \hat{\bar{X}}_{\sigma, u}^{\epsilon, i_{2}} (k_{1}) \hat{\bar{X}}_{\sigma, u}^{\epsilon, i_{3}} (k_{2}): + : \hat{\bar{X}}_{\sigma, b}^{\epsilon, i_{2}} (k_{1}) \hat{\bar{X}}_{\sigma, b}^{\epsilon, i_{3}} (k_{2}): )  \nonumber \\
& \times \overline{( : \hat{X}_{\bar{\sigma}, u}^{\epsilon, i_{2}'}(k_{1}') \hat{X}_{\bar{\sigma}, u}^{\epsilon, i_{3}'} (k_{2}'): - : \hat{X}_{\bar{\sigma}, b}^{\epsilon, i_{2}'} (k_{1}') \hat{X}_{\bar{\sigma}, b}^{\epsilon, i_{3}'} (k_{2}'): }\nonumber\\
& \hspace{10mm} \overline{- : \hat{\bar{X}}_{\bar{\sigma}, u}^{\epsilon, i_{2}'} (k_{1}') \hat{\bar{X}}_{\bar{\sigma}, u}^{\epsilon, i_{3}'} (k_{2}'): + : \hat{\bar{X}}_{\bar{\sigma}, b}^{\epsilon, i_{2}'} (k_{1}') \hat{\bar{X}}_{\bar{\sigma}, b}^{\epsilon, i_{3}'} (k_{2}'): ) } ] \nonumber\\
\lesssim& (1_{k_{1} = k_{1}', k_{2} = k_{2}'} + 1_{k_{1} = k_{2}', k_{2} = k_{1}'}) (\frac{ \lvert \epsilon k_{1} \rvert^{\eta} + \lvert \epsilon k_{2} \rvert^{\eta}}{\lvert k_{1} \rvert^{2} \lvert k_{2} \rvert^{2}})
\end{align} 
by H$\ddot{\mathrm{o}}$lder's inequality, and the proof of \eqref{[Equation (4.7)][ZZ17]}. Applying \eqref{estimate 7} to \eqref{estimate 226} gives 
\begin{align}\label{estimate 11}
&M_{q,t, i_{0}}^{4} \\
\lesssim& \sum_{k} \theta(2^{-q} k)^{2} \sum_{\lvert i-j \rvert \leq 1, \lvert i' - j' \rvert \leq 1} \sum_{k_{1}, k_{2}, k_{3}, k_{4}: k_{12} = k} \sum_{j_{1}, j_{1}' = 1}^{3} \theta(2^{-i} k_{123}) \theta(2^{-i'} k_{124}) \nonumber\\
& \times \theta(2^{-j} k_{3}) \theta(2^{-j'} k_{4}) \int_{[0,t]^{2}} \bar{l}_{k_{123}k_{3}, t-s, i_{0} j_{1}} \bar{l}_{k_{124}k_{4}, t-\bar{s}, i_{0} j_{1}'}  \int_{0}^{s} \int_{0}^{\bar{s}} ( \frac{ \lvert \epsilon k_{1} \rvert^{\eta} + \lvert \epsilon k_{2} \rvert^{\eta}}{\lvert k_{1} \rvert^{2} \lvert k_{2} \rvert^{2}}) \nonumber \\
& \hspace{5mm} \times \lvert e^{- \lvert k_{12} \rvert^{2} (s-\sigma)} - e^{- \lvert k_{12} \rvert^{2} (t-\sigma)} \rvert  \lvert e^{- \lvert k_{12} \rvert^{2} (\bar{s}-\bar{\sigma})} - e^{- \lvert k_{12} \rvert^{2} (t-\bar{\sigma})} \rvert d \bar{\sigma} d \sigma \lvert k_{12} \rvert^{2} ds d \bar{s}.   \nonumber
\end{align} 
Due to \eqref{estimate 227} and \eqref{estimate 7}, by replacing $k_{3}'$ with $k_{4}$, we obtain from \eqref{estimate 226}  
\begin{align}\label{estimate 12}
M_{q,t, i_{0}}^{5}& \lesssim \sum_{k} \theta(2^{-q} k)^{2} \sum_{\lvert i-j \rvert \leq 1, \lvert i'-j' \rvert \leq 1} \sum_{k_{1}, k_{2}, k_{3}, k_{4} \neq 0: k_{12} = k} \sum_{i_{3}, i_{3}', j_{1}, j_{1}'  =1}^{3}  \nonumber\\
& \times \theta(2^{-i} k_{123}) \theta(2^{-i'} k_{124}) \theta(2^{-j} k_{3}) \theta(2^{-j'} k_{4}) \int_{[0,t]^{2}} \bar{l}_{k_{123}k_{3}, t-s, i_{0}j_{1}} \bar{l}_{k_{124}k_{4}, t-\bar{s}, i_{0} j_{1}'}  \nonumber\\
& \hspace{5mm} \times \int_{s}^{t} \int_{\bar{s}}^{t} \frac{ \lvert k_{12} \rvert^{2}}{\lvert k_{1} \rvert^{2} \lvert k_{2} \rvert^{2}} \lvert e^{-\lvert k_{12} \rvert^{2} f(\epsilon k_{12})(t-\sigma)} g(\epsilon k_{12}^{i_{3}}) - e^{- \lvert k_{12} \rvert^{2} (t-\sigma)} i \rvert  \nonumber\\
& \hspace{15mm} \times \lvert e^{-\lvert k_{12} \rvert^{2} f(\epsilon k_{12})(t-\bar{\sigma})} g(\epsilon k_{12}^{i_{3}'}) - e^{- \lvert k_{12} \rvert^{2} (t-\bar{\sigma})} i \rvert  d \bar{\sigma} d \sigma ds d \bar{s}, 
\end{align} 
\begin{align}\label{estimate 13} 
M_{q,t, i_{0}}^{6} \lesssim& \sum_{k} \theta(2^{-q} k)^{2} \sum_{\lvert i-j \rvert \leq 1, \lvert i'-j'\rvert \leq 1} \sum_{k_{1}, k_{2}, k_{3}, k_{4}\neq 0: k_{12} = k} \sum_{j_{1}, j_{1}' =1}^{3} \theta(2^{-i} k_{123}) \theta(2^{-i'} k_{124}) \nonumber\\
& \times \theta(2^{-j} k_{3}) \theta(2^{-j'} k_{4}) \int_{[0,t]^{2}} \bar{l}_{k_{123}k_{3}, t-s, i_{0} j_{1}} \bar{l}_{k_{124}k_{4}, t-\bar{s}, i_{0} j_{1}}  \nonumber\\
& \times \int_{s}^{t} \int_{\bar{s}}^{t} ( \frac{  \lvert \epsilon k_{1} \rvert^{\eta} + \lvert \epsilon k_{2} \rvert^{\eta}}{ \lvert k_{1} \rvert^{2} \lvert k_{2} \rvert^{2}}) e^{- \lvert k_{12} \rvert^{2} (t-\sigma) } e^{- \lvert k_{12} \rvert^{2} (t- \bar{\sigma})} d \bar{\sigma} d \sigma \lvert k_{12} \rvert^{2} ds d \bar{s}. 
\end{align}  
Next, we compute for $\lambda \in \{3,4\}$, $l \in \{j_{1}, j_{1}'\}$, any $\eta \in [0,1]$, from \eqref{[Equation (4.7d)][ZZ17]}-\eqref{[Equation (4.7e)][ZZ17]} 
\begin{align}\label{estimate 229}
& \lvert l_{k_{12\lambda}k_{\lambda}, t-s, i_{0} l} - \bar{l}_{k_{12\lambda} k_{\lambda}, t-s, i_{0} l} \rvert \nonumber\\
\lesssim& \frac{\lvert k_{12\lambda} \rvert}{\lvert k_{\lambda} \rvert^{2}} e^{- (\lvert k_{12\lambda} \rvert^{2} + \lvert k_{\lambda} \rvert^{2} ) \bar{c}_{f} (t-s)} [ \lvert \epsilon k_{12\lambda} \rvert^{\frac{\eta}{2}} + \lvert \epsilon k_{\lambda}\rvert^{\frac{\eta}{2}}] 
\end{align} 
by \eqref{[Equation (4.2)][ZZ17]}-\eqref{[Equation (4.3)][ZZ17]}. Thus, applying \eqref{estimate 229} to \eqref{estimate 8} leads to 
\begin{align}\label{estimate 231}
M_{q, t, i_{0}}^{1} \lesssim&   \epsilon^{\eta} \sum_{k} \theta(2^{-q} k)^{2} \sum_{\lvert i-j \rvert \leq 1, \lvert i'-j'\rvert \leq 1} \sum_{k_{1}, k_{2}, k_{3}, k_{4}\neq 0: k_{12} =k} \theta(2^{-i} k_{123}) \theta(2^{-i'} k_{124}) \nonumber\\
& \times \theta(2^{-j} k_{3}) \theta(2^{-j'} k_{4}) (\prod_{i=1}^{4} \frac{1}{\lvert k_{i} \rvert^{2}}) \frac{ \lvert k_{123} \rvert \lvert k_{124} \rvert}{ \lvert k_{12} \rvert } [ \lvert k_{123} \rvert + \lvert k_{3} \rvert]^{\frac{\eta}{2}} [ \lvert k_{124} \rvert + \lvert k_{4} \rvert]^{\frac{\eta}{2}} \nonumber\\
& \times \int_{[0,t]^{2}} e^{- (\lvert k_{123} \rvert^{2} + \lvert k_{3} \rvert^{2} ) \bar{c}_{f} (t-s)} e^{- ( \lvert k_{124} \rvert^{2} + \lvert k_{4} \rvert^{2} ) \bar{c}_{f} (t- \bar{s})} \lvert t-s \rvert^{\frac{1}{4}} \lvert t- \bar{s} \rvert^{\frac{1}{4}} ds d \bar{s} \nonumber \\
\lesssim&  \epsilon^{\eta} t^{\epsilon} \sum_{k} \theta(2^{-q} k)^{2} \sum_{\lvert i-j \rvert \leq 1, \lvert i' -j' \rvert \leq 1} \sum_{k_{1}, k_{2}, k_{3}, k_{4}\neq 0: k_{12} = k} \theta(2^{-i} k_{123}) \theta(2^{-i'} k_{124}) \nonumber\\
& \times \theta(2^{-j} k_{3}) \theta(2^{-j'} k_{4}) (\prod_{i=1}^{4} \frac{1}{\lvert k_{i} \rvert^{2}})  \frac{1}{\lvert k_{12} \rvert } \frac{1}{ \lvert k_{3} \rvert ^{\frac{3}{2} - \frac{\eta}{2} -\epsilon}} \frac{1}{\lvert k_{4} \rvert^{\frac{3}{2} - \frac{\eta}{2} - \epsilon}}  
\end{align} 
by mean value theorem and \eqref{key estimate}. Next, continuing the bound on $M_{q, t, i_{0}}^{2}$ from \eqref{estimate 9} 
\begin{align}\label{estimate 232}
M_{q,t, i_{0}}^{2} \lesssim& \epsilon^{\eta} \sum_{k} \theta(2^{-q} k)^{2} \sum_{\lvert i-j \rvert \leq 1, \lvert i'-j'\rvert \leq 1} \sum_{k_{1}, k_{2}, k_{3}, k_{4}\neq 0: k_{12} = k} \theta(2^{-i} k_{123}) \theta(2^{-i'} k_{124}) \nonumber\\
& \times \theta(2^{-j} k_{3}) \theta(2^{-j'} k_{4}) \frac{ \lvert k_{123} \rvert \lvert k_{124} \rvert}{ ( \prod_{i=1}^{4} \lvert k_{i} \rvert^{2} ) \lvert k_{12} \rvert^{2}} \nonumber\\
& \times[\lvert k_{123} \rvert + \lvert k_{3} \rvert]^{\frac{\eta}{2}} [\lvert k_{124} \rvert + \lvert k_{4} \rvert]^{\frac{\eta}{2}} \int_{[0,t]^{2}} e^{- (\lvert k_{123} \rvert^{2} + \lvert k_{3} \rvert^{2} ) \bar{c}_{f} (t-s)} \nonumber\\
& \hspace{5mm} \times e^{- (\lvert k_{124} \rvert^{2} + \lvert k_{4} \rvert^{2})\bar{c}_{f} (t- \bar{s})} \lvert \lvert k_{12} \rvert^{2} (t-s) \rvert^{\frac{1}{4}} \lvert \lvert k_{12} \rvert^{2} (t- \bar{s}) \rvert^{\frac{1}{4}} ds d \bar{s} \nonumber \\
\lesssim& \epsilon^{\eta} t^{\epsilon} \sum_{k} \theta(2^{-q} k)^{2} \sum_{\lvert i-j \rvert \leq 1, \lvert i'-j'\rvert \leq 1} \sum_{k_{1}, k_{2}, k_{3}, k_{4}\neq 0: k_{12} = k} \theta(2^{-i} k_{123}) \theta(2^{-i'} k_{124}) \nonumber\\
& \times \theta(2^{-j} k_{3}) \theta(2^{-j'} k_{4}) ( \prod_{i=1}^{4} \frac{1}{\lvert k_{i} \rvert^{2}}) \frac{1}{\lvert k_{12} \rvert} \frac{1}{ \lvert k_{3} \rvert^{\frac{3}{2} - \frac{\eta}{2} -\epsilon}} \frac{1}{ \lvert k_{4} \rvert^{\frac{3}{2} - \frac{\eta}{2} -\epsilon}} 
\end{align} 
by \eqref{estimate 229}, mean value theorem and \eqref{key estimate}. Next, in order to continue the bound on $M_{q,t,i_{0}}^{3}$ from \eqref{estimate 10}, we first compute 
\begin{align}\label{estimate 230}
& \lvert e^{- \lvert k_{12} \rvert^{2}  f(\epsilon k_{12})(s-\sigma)} k_{12}^{i_{3}} g(\epsilon k_{12}^{i_{3}}) - e^{- \lvert k_{12} \rvert^{2}  f(\epsilon k_{12})(t-\sigma)} k_{12}^{i_{3}} g(\epsilon k_{12}^{i_{3}}) \\
& \hspace{5mm} - e^{- \lvert k_{12} \rvert^{2}(s-\sigma) } k_{12}^{i_{3}} i + e^{- \lvert k_{12}\rvert^{2}(t-\sigma)} k_{12}^{i_{3}} i \rvert \nonumber\\
& \hspace{13mm} \lesssim (e^{- \lvert k_{12} \rvert^{2}  \bar{c}_{f}(s-\sigma)} \lvert \epsilon k_{12} \rvert^{\eta} \lvert k_{12} \rvert 
) \wedge (\lvert k_{12} \rvert e^{- \lvert k_{12} \rvert^{2}  \bar{c}_{f}(s-\sigma)} \lvert k_{12} \rvert \lvert t-s \rvert^{\frac{1}{2}}) \nonumber\\
&\hspace{13mm} \lesssim e^{- \lvert k_{12} \rvert^{2}  \bar{c}_{f}(s-\sigma)} \epsilon^{\frac{\eta}{2}} \lvert k_{12} \rvert^{\frac{3}{2} + \frac{\eta}{2}} \lvert t-s \rvert^{\frac{1}{4}} \nonumber
\end{align} 
where the first estimate is due to \eqref{[Equation (4.2)][ZZ17]}-\eqref{[Equation (4.3)][ZZ17]} and that $0  < \sigma \leq s \leq t$ while the second estimate due to mean value theorem. Now we can deduce from \eqref{estimate 10}, 
\begin{align}\label{estimate 233}
M_{q,t, i_{0}}^{3} \lesssim& \epsilon^{\eta} \sum_{k} \theta(2^{-q} k)^{2} \sum_{\lvert i-j\rvert \leq 1, \lvert i' - j' \rvert \leq 1} \sum_{k_{1}, k_{2}, k_{3}, k_{4}\neq 0: k_{12} = k} \theta(2^{-i} k_{123}) \theta(2^{-i'} k_{124})  \nonumber \\
&\times \theta(2^{-j} k_{3}) \theta(2^{-j '} k_{4}) \int_{[0,t]^{2}} e^{- \lvert k_{123} \rvert^{2} (t-s)} \frac{ e^{- \lvert k_{3} \rvert^{2} (t-s)}}{\lvert k_{3} \rvert^{2}} \lvert k_{123} \rvert  \nonumber\\
& \times e^{- \lvert k_{124} \rvert^{2} (t- \bar{s})} \frac{ e^{- \lvert k_{4} \rvert^{2} (t-\bar{s})}}{\lvert k_{4} \rvert^{2}} \lvert k_{124} \rvert \int_{0}^{s} \int_{0}^{\bar{s}} \frac{1}{\lvert k_{1} \rvert^{2} \lvert k_{2} \rvert^{2}} \nonumber\\
& \times e^{- \lvert k_{12} \rvert^{2} \bar{c}_{f}(s-\sigma) } \lvert k_{12} \rvert^{\frac{3}{2} + \frac{\eta}{2}} \lvert t-s \rvert^{\frac{1}{4}}  e^{- \lvert k_{12} \rvert^{2} \bar{c}_{f}(\bar{s}-\bar{\sigma})}  \lvert k_{12} \rvert^{\frac{3}{2} + \frac{\eta}{2}} \lvert t-\bar{s} \rvert^{\frac{1}{4}} d \bar{\sigma} d \sigma ds d \bar{s} \nonumber\\
\lesssim& \epsilon^{\eta} t^{\eta} \sum_{k} \theta(2^{-q} k)^{2}\sum_{\lvert i-j \rvert \leq 1, \lvert i'-j'\rvert \leq 1} \sum_{k_{1}, k_{2}, k_{3}, k_{4}\neq 0: k_{12} =k}   \theta(2^{-i} k_{123}) \theta(2^{-i'} k_{124})\nonumber\\
& \times  \theta(2^{-j} k_{3}) \theta(2^{-j'} k_{4}) [ \frac{1}{ \lvert k_{1} \rvert^{2-\eta} \lvert k_{2} \rvert^{2} \lvert k_{3} \rvert^{\frac{7}{2} - \epsilon} \lvert k_{4} \rvert^{\frac{7}{2} - \epsilon} \lvert k_{12} \rvert} \nonumber\\
& \hspace{40mm} + \frac{1}{ \lvert k_{1} \rvert^{2} \lvert k_{2} \rvert^{2- \eta} \lvert k_{3} \rvert^{\frac{7}{2} - \epsilon} \lvert k_{4} \rvert^{\frac{7}{2} - \epsilon} \lvert k_{12} \rvert} ]
\end{align} 
by \eqref{[Equation (4.7d)][ZZ17]}-\eqref{[Equation (4.7e)][ZZ17]}, \eqref{estimate 230} and \eqref{key estimate}.  Next, we compute from \eqref{estimate 11}, 
\begin{align}\label{estimate 234}
&M_{q,t, i_{0}}^{4} \\
\lesssim& \epsilon^{\eta}  \sum_{k} \theta(2^{-q} k)^{2} \sum_{\lvert i-j\rvert \leq 1, \lvert i' - j' \rvert \leq 1} \sum_{k_{1}, k_{2}, k_{3}, k_{4}\neq 0: k_{12} = k} \theta(2^{-i} k_{123}) \theta(2^{-i'} k_{124}) \nonumber\\
& \times \theta(2^{-j} k_{3}) \theta(2^{-j'} k_{4}) \frac{1}{ \prod_{i=1}^{4} \lvert k_{i} \rvert^{2}} \frac{1}{\lvert k_{12} \rvert} (\lvert k_{1} \rvert^{\eta} + \lvert k_{2} \rvert^{\eta}) \lvert k_{123} \rvert \lvert k_{124} \rvert \nonumber\\
& \times \int_{[0,t]^{2}} e^{- (\lvert k_{123} \rvert^{2} + \lvert k_{3} \rvert^{2} )(t-s)} e^{- (\lvert k_{124} \rvert^{2} + \lvert k_{4} \rvert^{2} )(t- \bar{s})} (t-s)^{\frac{1}{4}} (t- \bar{s})^{\frac{1}{4}} ds d \bar{s} \nonumber\\
\lesssim& \epsilon^{\eta} t^{\epsilon} \sum_{k} \theta(2^{-q} k)^{2} \sum_{\lvert i-j \rvert \leq 1, \lvert i'-j'\rvert \leq 1} \sum_{k_{1}, k_{2}, k_{3}, k_{4}\neq 0: k_{12} = k} \theta(2^{-i} k_{123})  \theta(2^{-i'} k_{124})\theta(2^{-j} k_{3})  \nonumber\\
& \times \theta(2^{-j'} k_{4})  [ \frac{1}{ \lvert k_{1} \rvert^{2-\eta} \lvert k_{2} \rvert^{2} \lvert k_{3} \rvert^{\frac{7}{2} - \epsilon}\lvert k_{4} \rvert^{\frac{7}{2} - \epsilon}  \lvert k_{12} \rvert} + \frac{1}{ \lvert k_{1} \rvert^{2} \lvert k_{2} \rvert^{2 - \eta} \lvert k_{3} \rvert^{\frac{7}{2} - \epsilon}\lvert k_{4} \rvert^{\frac{7}{2} - \epsilon}   \lvert k_{12} \rvert}  ]  \nonumber
\end{align} 
by \eqref{[Equation (4.7e)][ZZ17]}, mean value theorem and \eqref{key estimate}. Next, we continue from \eqref{estimate 12} for any $\eta \in [0,1]$, 
\begin{align}\label{estimate 235}
M_{q,t, i_{0}}^{5} \lesssim& \sum_{k} \theta(2^{-q} k)^{2} \sum_{\lvert i-j \rvert \leq 1, \lvert i' - j' \rvert \leq 1} \sum_{k_{1}, k_{2}, k_{3}, k_{4}\neq 0: k_{12} = k} \\
& \times \theta(2^{-i} k_{123}) \theta(2^{-i'} k_{124}) \theta(2^{-j} k_{3}) \theta(2^{-j'} k_{4}) \int_{[0,t]^{2}} ( \prod_{i=1}^{4} \frac{1}{\lvert k_{i} \rvert^{2}}) \lvert k_{123} \rvert \lvert k_{124} \rvert \lvert k_{12} \rvert^{2}  \nonumber\\
& \times e^{- [ \lvert k_{123} \rvert^{2} + \lvert k_{3} \rvert^{2} ](t-s)} e^{- [\lvert k_{124} \rvert^{2} + \lvert k_{4} \rvert^{2} (t- \bar{s})} \int_{s}^{t} \int_{\bar{s}}^{t} \nonumber\\ 
&\times e^{- \lvert k_{12} \rvert^{2} \bar{c}_{f} (t-\sigma)} \lvert \epsilon k_{12} \rvert^{\frac{\eta}{2}} e^{- \lvert k_{12}\rvert^{2} \bar{c}_{f} (t-\bar{\sigma} )} \lvert \epsilon k_{12} \rvert^{\frac{\eta}{2}} d \bar{\sigma} d \sigma  ds d \bar{s} \nonumber\\
\lesssim&  \epsilon^{\eta} t^{\epsilon} \sum_{k} \theta(2^{-q} k)^{2} \sum_{\lvert i-j \rvert \leq 1, \lvert i' - j' \rvert \leq 1} \sum_{k_{1}, k_{2}, k_{3}, k_{4}\neq 0: k_{12} = k} \nonumber\\ 
& \times  \theta(2^{-i} k_{123}) \theta(2^{-i'} k_{124}) \theta(2^{-j} k_{3}) \theta(2^{-j'} k_{4}) \nonumber\\
& \times ( \frac{1}{ \lvert k_{1} \rvert^{2-\eta} \lvert k_{2} \rvert^{2} \lvert k_{3} \rvert^{\frac{7}{2} - \epsilon} \lvert k_{4} \rvert^{\frac{7}{2} - \epsilon} \lvert k_{12} \rvert} + \frac{1}{\lvert k_{1} \rvert^{2} \lvert k_{2} \rvert^{2-\eta} \lvert k_{3} \rvert^{\frac{7}{2} - \epsilon} \lvert k_{4} \rvert^{\frac{7}{2} - \epsilon} \lvert k_{12} \rvert}) \nonumber
\end{align} 
by \eqref{[Equation (4.7e)][ZZ17]}, \eqref{[Equation (4.2)][ZZ17]}, \eqref{[Equation (4.3)][ZZ17]} and \eqref{key estimate}. Finally continuing our estimate on $M_{q,t, i_{0}}^{6}$ from \eqref{estimate 13}, 
\begin{align}\label{estimate 236}
M_{q, t, i_{0}}^{6} \lesssim&   \epsilon^{\eta} t^{\epsilon}  \sum_{k} \theta(2^{-q} k)^{2} \sum_{\lvert i-j \rvert \leq 1, \lvert i' - j' \rvert \leq 1} \sum_{k_{1}, k_{2}, k_{3}, k_{4}\neq 0: k_{12} =k} \theta(2^{-i} k_{123}) \theta(2^{-i'} k_{124})  \nonumber\\
& \times \theta(2^{-j} k_{3}) \theta(2^{-j'} k_{4}) [\frac{1}{ \lvert k_{1} \rvert^{2-\eta} \lvert k_{2} \rvert^{2} \lvert k_{3} \rvert^{\frac{7}{2} - \epsilon} \lvert k_{4} \rvert^{\frac{7}{2} - \epsilon} \lvert k_{12} \rvert}  \nonumber\\
& \hspace{30mm} + \frac{1}{ \lvert k_{1} \rvert^{2} \lvert k_{2} \rvert^{2 - \eta} \lvert k_{3} \rvert^{\frac{7}{2} - \epsilon} \lvert k_{4} \rvert^{\frac{7}{2} - \epsilon} \lvert k_{12} \rvert} ]
\end{align} 
by \eqref{[Equation (4.7e)][ZZ17]} and \eqref{key estimate}. Applying \eqref{estimate 231}, \eqref{estimate 232},\eqref{estimate 233}-\eqref{estimate 236} to \eqref{estimate 237} leads to 
\begin{align}\label{estimate 245}
& \mathbb{E} [ \lvert \Delta_{q} (L_{t, i_{0}j_{0}}^{41} - \tilde{L}_{t, i_{0}j_{0}}^{41}) \rvert^{2} ] \nonumber\\
\lesssim& \epsilon^{\eta}  t^{\epsilon} \sum_{k} \theta(2^{-q} k)^{2} \sum_{q \lesssim j, q \lesssim j'} \sum_{k_{1}, k_{2}\neq 0: k_{12} = k} \frac{1}{\lvert k \rvert}  \nonumber\\
& \times [ \frac{1}{ \lvert k_{1} \rvert^{2} \lvert k_{2} \rvert^{2}} \frac{1}{2^{j [ \frac{1}{2} - \frac{\eta}{2} - \epsilon]} } \frac{1}{2^{j' [ \frac{1}{2} - \frac{\eta}{2} -\epsilon]}} + \frac{1}{\lvert k_{1} \rvert^{2-\eta} \lvert k_{2} \rvert^{2}} \frac{1}{2^{j [ \frac{1}{2} - \epsilon]}} \frac{1}{2^{j' [ \frac{1}{2} - \epsilon]}} \nonumber\\
& \hspace{20mm} + \frac{1}{ \lvert k_{1} \rvert^{2} \lvert k_{2} \rvert^{2- \eta}} \frac{1}{2^{j[ \frac{1}{2} - \epsilon]}} \frac{1}{2^{j' [ \frac{1}{2} - \epsilon]}}]  \lesssim \epsilon^{\eta} t^{\epsilon} 2^{q(\eta + 2 \epsilon)} 
\end{align} 
where we used that $2^{q} \lesssim 2^{j}$ so that $q \lesssim j$ and similarly $q \lesssim j'$ and Lemma \ref{Lemma 3.13}. Next, considering \eqref{estimate 246}, we rely on \eqref{[Equation (3.21d)][ZZ17]}, \eqref{[Equation (3.1ab)][ZZ17]}, \eqref{[Equation (3.22aa)][ZZ17]}, \eqref{[Equation (1.3i)][ZZ17]}, \eqref{estimate 216}-  \eqref{estimate 228} to compute 
\begin{align}\label{estimate 223}
& \mathbb{E} [ \lvert \Delta_{q} (\tilde{L}_{t, i_{0}j_{0}}^{41} + 4 \sum_{i_{1} =1}^{3} u_{2}^{\epsilon, i_{1}}(t) C_{3,u}^{\epsilon, i_{0} i_{1} j_{0}}(t) - 4 \sum_{i_{1} =1}^{3} \bar{u}_{2}^{\epsilon, i_{1}}(t) \bar{C}_{3,u}^{\epsilon, i_{0}i_{1} j_{0}} (t) ) \rvert^{2} ]  \\
=&  \mathbb{E} [ \lvert \sum_{k} \theta(2^{-q} k) (2\pi)^{-\frac{9}{2}} \sum_{k_{1}, k_{2}, k_{3} \neq 0: k_{12} =k} \sum_{i_{1}, i_{2}, i_{3}, i_{4}, j_{1} =1}^{3} \sum_{\lvert i-j \rvert \leq 1} \nonumber\\
& \times \hat{\mathcal{P}}^{j_{1} i_{4}}(k_{3}) \hat{\mathcal{P}}^{j_{0}i_{4}}(k_{3}) \hat{\mathcal{P}}^{i_{1}i_{2}}(k_{12}) \nonumber \\
& \times [ \theta(2^{-j} k_{3}) \int_{0}^{t} [ : \hat{X}_{\sigma, u}^{\epsilon, i_{2}} (k_{1}) \hat{X}_{\sigma, u}^{\epsilon, i_{3}}(k_{2}): - : \hat{X}_{\sigma, b}^{\epsilon, i_{2}}(k_{1}) \hat{X}_{\sigma, b}^{\epsilon, i_{3}}(k_{2}):] \nonumber\\
& \hspace{40mm} \times e^{-\lvert k_{12} \rvert^{2} f(\epsilon k_{12}) (t-\sigma)} k_{12}^{i_{3}} g(\epsilon k_{12}^{i_{3}}) d \sigma \nonumber\\
& \hspace{4mm} \times \int_{0}^{t} \frac{ e^{- \lvert k_{3} \rvert^{2}  f(\epsilon k_{3})(t-s)} h_{u}(\epsilon k_{3}) h_{b}(\epsilon k_{3})}{2\lvert k_{3} \rvert^{2} f(\epsilon k_{3})} \nonumber\\
& \hspace{10mm} \times  [ \theta(2^{-i} k_{123}) e^{- \lvert k_{123}\rvert^{2} f(\epsilon k_{123}) (t-s)} k_{123}^{j_{1}} g(\epsilon k_{123}^{j_{1}})\hat{\mathcal{P}}^{i_{0} i_{1}} (k_{123}) \nonumber\\
& \hspace{40mm} - \theta(2^{-i} k_{3}) e^{- \lvert k_{3} \rvert^{2}  f(\epsilon k_{3})(t-s)} k_{3}^{j_{1}} g(\epsilon k_{3}^{j_{1}}) \hat{\mathcal{P}}^{i_{0} i_{1}} (k_{3})] ds \nonumber\\
&+ \theta(2^{-j} k_{3}) \int_{0}^{t} [ : \hat{X}_{\sigma, u}^{\epsilon, i_{2}} (k_{1}) \hat{X}_{\sigma, u}^{\epsilon, i_{3}}(k_{2}): - : \hat{X}_{\sigma, b}^{\epsilon, i_{2}}(k_{1}) \hat{X}_{\sigma, b}^{\epsilon, i_{3}}(k_{2}):  \nonumber\\
& \hspace{40mm} - \hat{X}_{\sigma, u}^{\epsilon, i_{2}}(k_{1}) \hat{X}_{\sigma, u}^{\epsilon, i_{3}}(k_{2}) + \hat{X}_{\sigma, b}^{\epsilon, i_{2}}(k_{1}) \hat{X}_{\sigma, b}^{\epsilon, i_{3}}(k_{2})] \nonumber\\
& \hspace{4mm} \times e^{- \lvert k_{12} \rvert^{2} f(\epsilon k_{12} )(t-\sigma)} k_{12}^{i_{3}} g(\epsilon k_{12}^{i_{3}}) d\sigma \nonumber\\
& \hspace{4mm} \times  \int_{0}^{t} \theta(2^{-i} k_{3}) \frac{ e^{-2 \lvert k_{3} \rvert^{2} f(\epsilon k_{3})(t-s)} h_{u}(\epsilon k_{2}) h_{b}(\epsilon k_{3})}{2\lvert k_{3} \rvert^{2} f(\epsilon k_{3})} k_{3}^{j_{1}} g(\epsilon k_{3}^{j_{1}}) \hat{\mathcal{P}}^{i_{0} i_{1}}(k_{3}) ds \nonumber\\
& -  \theta(2^{-j} k_{3}) \int_{0}^{t} [: \hat{\bar{X}}_{\sigma, u}^{\epsilon, i_{2}} (k_{1}) \hat{\bar{X}}_{\sigma, u}^{\epsilon, i_{3}}(k_{2}): - : \hat{\bar{X}}_{\sigma, b}^{\epsilon, i_{2}}(k_{1}) \hat{\bar{X}}_{\sigma, b}^{\epsilon, i_{3}}(k_{2}):] e^{-\lvert k_{12} \rvert^{2} (t-\sigma)} k_{12}^{i_{3}} i d \sigma \nonumber\\
& \hspace{4mm} \times \int_{0}^{t} \frac{ e^{- \lvert k_{3} \rvert^{2} (t-s) } h_{u}(\epsilon k_{3}) h_{b}(\epsilon k_{3})}{2\lvert k_{3} \rvert^{2}} [ \theta(2^{-i} k_{123}) e^{- \lvert k_{123}\rvert^{2}  (t-s)} k_{123}^{j_{1}} i\hat{\mathcal{P}}^{i_{0} i_{1}} (k_{123}) \nonumber\\
& \hspace{10mm} - \theta(2^{-i} k_{3}) e^{- \lvert k_{3} \rvert^{2} (t-s)} k_{3}^{j_{1}} i \hat{\mathcal{P}}^{i_{0} i_{1}} (k_{3})] ds \nonumber\\
&- \theta(2^{-j} k_{3}) \int_{0}^{t} [ : \hat{\bar{X}}_{\sigma, u}^{\epsilon, i_{2}} (k_{1}) \hat{\bar{X}}_{\sigma, u}^{\epsilon, i_{3}}(k_{2}): - : \hat{\bar{X}}_{\sigma, b}^{\epsilon, i_{2}}(k_{1}) \hat{\bar{X}}_{\sigma, b}^{\epsilon, i_{3}}(k_{2}): \nonumber\\
& \hspace{40mm} - \hat{\bar{X}}_{\sigma, u}^{\epsilon, i_{2}}(k_{1}) \hat{\bar{X}}_{\sigma, u}^{\epsilon, i_{3}}(k_{2}) + \hat{\bar{X}}_{\sigma, b}^{\epsilon, i_{2}}(k_{1}) \hat{\bar{X}}_{\sigma, b}^{\epsilon, i_{3}}(k_{2})] \nonumber\\
& \hspace{3mm} \times e^{- \lvert k_{12} \rvert^{2} (t-\sigma)} k_{12}^{i_{3}} i d\sigma \int_{0}^{t} \theta(2^{-i} k_{3}) \frac{ e^{-2 \lvert k_{3} \rvert^{2} (t-s)} h_{u}(\epsilon k_{2}) h_{b}(\epsilon k_{3})}{2\lvert k_{3} \rvert^{2} } k_{3}^{j_{1}} i \hat{\mathcal{P}}^{i_{0} i_{1}}(k_{3}) ds] e_{k} \rvert^{2}]. \nonumber
\end{align} 
Now the second and fourth terms in \eqref{estimate 223} vanish because e.g., 
\begin{align*}
& : \hat{X}_{\sigma, u}^{\epsilon, i_{2}} (k_{1}) \hat{X}_{\sigma, u}^{\epsilon, i_{3}}(k_{2}): - : \hat{X}_{\sigma, b}^{\epsilon, i_{2}}(k_{1}) \hat{X}_{\sigma, b}^{\epsilon, i_{3}}(k_{2}): - \hat{X}_{\sigma, u}^{\epsilon, i_{2}}(k_{1}) \hat{X}_{\sigma, u}^{\epsilon, i_{3}}(k_{2}) + \hat{X}_{\sigma, b}^{\epsilon, i_{2}}(k_{1}) \hat{X}_{\sigma, b}^{\epsilon, i_{3}}(k_{2}) \\
=& - \mathbb{E} [ \hat{X}_{\sigma, u}^{\epsilon, i_{2}}(k_{1}) \hat{X}_{\sigma, u}^{\epsilon, i_{3}}(k_{2})] + \mathbb{E} [ \hat{X}_{\sigma, b}^{\epsilon, i_{2}}(k_{1}) \hat{X}_{\sigma, b}^{\epsilon, i_{3}}(k_{2})] 
\end{align*} 
by Example \ref{Example 3.1} in which both mathematical expectations give $1_{k_{12} =0}$ due to \eqref{covariance}. Thus, we deduce 
\begin{align}\label{estimate 224}
& \mathbb{E} [ \lvert \Delta_{q} (\tilde{L}_{t, i_{0}j_{0}}^{41} + 4 \sum_{i_{1} =1}^{3} u_{2}^{\epsilon, i_{1}}(t) C_{3,u}^{\epsilon, i_{0} i_{1} j_{0}}(t) \nonumber\\
& \hspace{25mm} - 4 \sum_{i_{1} =1}^{3} \bar{u}_{2}^{\epsilon, i_{1}}(t) \bar{C}_{3,u}^{\epsilon, i_{0} i_{1} j_{0}} (t) ) \rvert^{2} ] \lesssim \sum_{i=1}^{3} N_{q, t, i_{0} j_{0}}^{i}
\end{align} 
where 
\begin{subequations}\label{estimate 238}
\begin{align}
N_{q, t, i_{0} j_{0}}^{1} &\approx  \mathbb{E} [ \lvert \sum_{k} \theta(2^{-q} k) \sum_{k_{1}, k_{2}, k_{3}\neq 0: k_{12} = k} \sum_{i_{1}, i_{2}, i_{3}, j_{1} =1}^{3} \sum_{\lvert i-j \rvert \leq 1}  \theta(2^{-j} k_{3}) \nonumber\\
& \times  \int_{0}^{t} [: \hat{X}_{\sigma, u}^{\epsilon, i_{2}}(k_{1}) \hat{X}_{\sigma, u}^{\epsilon, i_{3}}(k_{2}): - : \hat{X}_{\sigma, b}^{\epsilon, i_{2}}(k_{1}) \hat{X}_{\sigma, b}^{\epsilon, i_{3}}(k_{2}):]  \nonumber\\
& \hspace{5mm} \times e^{- \lvert k_{12} \rvert^{2} f(\epsilon k_{12}) (t-\sigma) } k_{12}^{i_{3}} g(\epsilon k_{12}^{i_{3}}) d \sigma \nonumber\\
& \hspace{5mm} \times [ \int_{0}^{t} \frac{ e^{- \lvert k_{3} \rvert^{2} f(\epsilon k_{3})(t-s)} h_{u}(\epsilon k_{3}) h_{b}(\epsilon k_{3})}{2 \lvert k_{3} \rvert^{2} f(\epsilon k_{3})}  \nonumber\\
& \hspace{15mm} \times [ \theta(2^{-i}k_{123}) e^{- \lvert k_{123} \rvert^{2} f(\epsilon k_{123} )(t-s)} k_{123}^{j_{1}} g(\epsilon k_{123}^{j_{1}}) \hat{\mathcal{P}}^{i_{0} i_{1}}(k_{123}) \nonumber\\
& \hspace{15mm} - \theta(2^{-i}k_{3}) e^{- \lvert k_{3} \rvert^{2} f(\epsilon k_{3} )(t-s)} k_{3}^{j_{1}} g(\epsilon k_{3}^{j_{1}}) \hat{\mathcal{P}}^{i_{0} i_{1}}(k_{3})] ds  \nonumber\\
& \hspace{5mm} - \int_{0}^{t} \frac{ e^{- \lvert k_{3} \rvert^{2} (t-s)} h_{u}(\epsilon k_{3}) h_{b}(\epsilon k_{3})}{2 \lvert k_{3} \rvert^{2} } [ \theta(2^{-i}k_{123}) e^{- \lvert k_{123} \rvert^{2} (t-s)} k_{123}^{j_{1}} i \hat{\mathcal{P}}^{i_{0} i_{1}}(k_{123}) \nonumber\\
& \hspace{15mm} - \theta(2^{-i}k_{3}) e^{- \lvert k_{3} \rvert^{2} (t-s)} k_{3}^{j_{1}} i \hat{\mathcal{P}}^{i_{0} i_{1}}(k_{3})] ds ]    e_{k} \rvert^{2}] , \\
N_{q, t, i_{0} j_{0}}^{2}  &\approx  \mathbb{E} [ \lvert \sum_{k} \theta(2^{-q} k)  \sum_{k_{1}, k_{2}, k_{3} \neq 0: k_{12} = k} \sum_{i_{1}, i_{2}, i_{3}, j_{1} =1}^{3} \sum_{\lvert i-j \rvert \leq 1} \theta(2^{-j} k_{3}) \nonumber\\
& \times  \int_{0}^{t} [: \hat{X}_{\sigma, u}^{\epsilon, i_{2}}(k_{1}) \hat{X}_{\sigma, u}^{\epsilon, i_{3}}(k_{2}): - : \hat{X}_{\sigma, b}^{\epsilon, i_{2}}(k_{1}) \hat{X}_{\sigma, b}^{\epsilon, i_{3}}(k_{2}):] \nonumber\\
& \hspace{10mm} \times [e^{- \lvert k_{12} \rvert^{2} f(\epsilon k_{12}) (t-\sigma) } k_{12}^{i_{3}} g(\epsilon k_{12}^{i_{3}}) - e^{- \lvert k_{12} \rvert^{2} (t-\sigma)} k_{12}^{i_{3}}i] d \sigma \nonumber\\
& \hspace{5mm} \times  \int_{0}^{t} \frac{ e^{- \lvert k_{3} \rvert^{2} (t-s)} h_{u}(\epsilon k_{3}) h_{b}(\epsilon k_{3})}{2 \lvert k_{3} \rvert^{2} } [ \theta(2^{-i}k_{123}) e^{- \lvert k_{123} \rvert^{2} (t-s)} k_{123}^{j_{1}} i \hat{\mathcal{P}}^{i_{0} i_{1}}(k_{123}) \nonumber\\
& \hspace{15mm} - \theta(2^{-i}k_{3}) e^{- \lvert k_{3} \rvert^{2} (t-s)} k_{3}^{j_{1}} i \hat{\mathcal{P}}^{i_{0} i_{1}}(k_{3})] ds e_{k} \rvert^{2}],  \\
N_{q, t, i_{0} j_{0}}^{3}  &\approx  \mathbb{E} [ \lvert \sum_{k} \theta(2^{-q} k)  \sum_{k_{1}, k_{2}, k_{3}\neq 0: k_{12} = k} \sum_{i_{1}, i_{2}, i_{3}, j_{1} =1}^{3} \sum_{\lvert i-j \rvert \leq 1}\theta(2^{-j} k_{3})  \nonumber\\
& \times   \int_{0}^{t} [: \hat{X}_{\sigma, u}^{\epsilon, i_{2}}(k_{1}) \hat{X}_{\sigma, u}^{\epsilon, i_{3}}(k_{2}): - : \hat{X}_{\sigma, b}^{\epsilon, i_{2}}(k_{1}) \hat{X}_{\sigma, b}^{\epsilon, i_{3}}(k_{2}): \nonumber\\
& \hspace{10mm} - : \hat{\bar{X}}_{\sigma, u}^{\epsilon, i_{2}}(k_{1}) \hat{\bar{X}}_{\sigma, u}^{\epsilon, i_{3}}(k_{2}): + \hat{\bar{X}}_{\sigma, b}^{\epsilon, i_{2}}(k_{1}) \hat{\bar{X}}_{\sigma, b}^{\epsilon, i_{3}}(k_{2}): ] e^{- \lvert k_{12} \rvert^{2} (t-\sigma)} k_{12}^{i_{3}}i d \sigma \nonumber\\ 
& \hspace{5mm} \times  \int_{0}^{t} \frac{ e^{- \lvert k_{3} \rvert^{2} (t-s)} h_{u}(\epsilon k_{3}) h_{b}(\epsilon k_{3})}{2 \lvert k_{3} \rvert^{2} } [ \theta(2^{-i}k_{123}) e^{- \lvert k_{123} \rvert^{2} (t-s)} k_{123}^{j_{1}} i \hat{\mathcal{P}}^{i_{0} i_{1}}(k_{123}) \nonumber\\
& \hspace{15mm} - \theta(2^{-i}k_{3}) e^{- \lvert k_{3} \rvert^{2} (t-s)} k_{3}^{j_{1}} i \hat{\mathcal{P}}^{i_{0} i_{1}}(k_{3})] ds e_{k} \rvert^{2} ]. 
\end{align} 
\end{subequations} 
Now in order to compute $N_{q, t, i_{0} j_{0}}^{1}$, we see that 
\begin{align}\label{estimate 16}
& \int_{[0,t]^{2}} \mathbb{E} [ [ : \hat{X}_{\sigma, u}^{\epsilon, i_{2}}(k_{1}) \hat{X}_{\sigma, u}^{\epsilon, i_{3}}(k_{2}): - : \hat{X}_{\sigma, b}^{\epsilon, i_{2}}(k_{1}) \hat{X}_{\sigma, b}^{\epsilon, i_{3}}(k_{2}): ] \nonumber\\
& \hspace{5mm}  \times  \overline{[ : \hat{X}_{\bar{\sigma}, u}^{\epsilon, i_{2}'}(k_{1}') \hat{X}_{\bar{\sigma}, u}^{\epsilon, i_{3}'}(k_{2}'): - : \hat{X}_{\bar{\sigma}, b}^{\epsilon, i_{2}'}(k_{1}') \hat{X}_{\bar{\sigma}, b}^{\epsilon, i_{3}'}(k_{2}'): ] }]\nonumber\\
& \hspace{5mm} \times e^{- \lvert k_{12} \rvert^{2} f(\epsilon k_{12}) (t-\sigma)} e^{- \lvert k_{12} ' \rvert^{2} f(\epsilon k_{12} ') (t- \bar{\sigma}) } k_{12}^{i_{3}} (k_{12}')^{i_{3}'} g(\epsilon k_{12}^{i_{3}}) g(\epsilon (k_{12}')^{i_{3}'} ) d \sigma d \bar{\sigma} \nonumber\\
\lesssim& (1_{k_{1} = k_{1}', k_{2} = k_{2}'} + 1_{k_{1} = k_{2}', k_{2} = k_{1}'}) \frac{1}{\lvert k_{1} \rvert^{2} \lvert k_{2} \rvert^{2} \lvert k_{12} \rvert^{2}} 
\end{align} 
by Example \ref{Example 3.1} and \eqref{covariance}. Thus, we obtain from applying \eqref{estimate 16} to \eqref{estimate 238} 
\begin{align}\label{estimate 18}
N_{q, t, i_{0} j_{0}}^{1} &\lesssim \sum_{k} \theta(2^{-q} k)^{2} \sum_{k_{1}, k_{2} \neq 0: k_{12} = k} \frac{1}{ \lvert k_{1} \rvert^{2} \lvert k_{2} \rvert^{2} \lvert k_{12} \rvert^{2}} \nonumber\\
& \lvert \sum_{i_{1}, j_{1} =1}^{3} \sum_{\lvert i-j \rvert \leq 1} \sum_{k_{3} \neq 0} \theta(2^{-j} k_{3}) \nonumber\\
& \hspace{5mm} \times [\int_{0}^{t} \frac{ e^{- \lvert k_{3} \rvert^{2} f(\epsilon k_{3})(t-s)} h_{u}(\epsilon k_{3}) h_{b}(\epsilon k_{3})}{ \lvert k_{3} \rvert^{2} f(\epsilon k_{3})} \nonumber\\
& \hspace{15mm} [ \theta(2^{-i}k_{123}) e^{- \lvert k_{123} \rvert^{2} f(\epsilon k_{123} )(t-s)} k_{123}^{j_{1}} g(\epsilon k_{123}^{j_{1}}) \hat{\mathcal{P}}^{i_{0} i_{1}}(k_{123}) \nonumber\\
& \hspace{20mm} - \theta(2^{-i}k_{3}) e^{- \lvert k_{3} \rvert^{2} f(\epsilon k_{3} )(t-s)} k_{3}^{j_{1}} g(\epsilon k_{3}^{j_{1}}) \hat{\mathcal{P}}^{i_{0} i_{1}}(k_{3})] ds  \nonumber\\
& \hspace{5mm} - \int_{0}^{t} \frac{ e^{- \lvert k_{3} \rvert^{2} (t-s)} h_{u}(\epsilon k_{3}) h_{b}(\epsilon k_{3})}{ \lvert k_{3} \rvert^{2} } [ \theta(2^{-i}k_{123}) e^{- \lvert k_{123} \rvert^{2} (t-s)} k_{123}^{j_{1}} i \hat{\mathcal{P}}^{i_{0} i_{1}}(k_{123}) \nonumber\\
& \hspace{15mm} - \theta(2^{-i}k_{3}) e^{- \lvert k_{3} \rvert^{2} (t-s)} k_{3}^{j_{1}} i \hat{\mathcal{P}}^{i_{0} i_{1}}(k_{3})] ds ] \rvert^{2}. 
\end{align} 
Next, in order to bound $N_{q, t, i_{0} j_{0}}^{2}$, similarly to \eqref{estimate 16}, we can estimate
\begin{align}\label{estimate 239}
& \int_{[0,t]^{2}} \mathbb{E} [  (: \hat{X}_{\sigma, u}^{\epsilon, i_{2}}(k_{1}) \hat{X}_{\sigma, u}^{\epsilon, i_{3}}(k_{2}): - : \hat{X}_{\sigma, b}^{\epsilon, i_{2}}(k_{1}) \hat{X}_{\sigma, b}^{\epsilon, i_{3}}(k_{2}):) \nonumber\\
& \hspace{15mm} \times \overline{(: \hat{X}_{\bar{\sigma}, u}^{\epsilon, i_{2}'}(k_{1}') \hat{X}_{\bar{\sigma}, u}^{\epsilon, i_{3}'}(k_{2}'): - : \hat{X}_{\bar{\sigma}, b}^{\epsilon, i_{2}'}(k_{1}') \hat{X}_{\bar{\sigma}, b}^{\epsilon, i_{3}'}(k_{2}'):) }] \nonumber\\
& \hspace{5mm}  \times \lvert [e^{- \lvert k_{12} \rvert^{2} f(\epsilon k_{12}) (t-\sigma) } k_{12}^{i_{3}} g(\epsilon k_{12}^{i_{3}}) - e^{- \lvert k_{12} \rvert^{2} (t-\sigma)} k_{12}^{i_{3}}i] \nonumber\\
& \hspace{15mm}  \times [e^{- \lvert k_{12}' \rvert^{2} f(\epsilon k_{12}') (t-\bar{\sigma}) } (k_{12}')^{i_{3}'} g(\epsilon (k_{12}')^{i_{3}'}) - e^{- \lvert k_{12}' \rvert^{2} (t-\bar{\sigma})} (k_{12}')^{i_{3}'}i] \rvert   \nonumber\\ 
\lesssim& (1_{k_{1} =k_{1}', k_{2} = k_{2}'} + 1_{k_{1} = k_{2}', k_{2} = k_{1}'}) \frac{\lvert \epsilon k_{12} \rvert^{\eta} }{\lvert k_{1} \rvert^{2} \lvert k_{2} \rvert^{2} \lvert k_{12} \rvert^{2}}.
\end{align} 
Thus, we obtain from applying \eqref{estimate 239} to \eqref{estimate 238}
\begin{align}\label{estimate 241}
N_{q, t, i_{0} j_{0}}^{2} \lesssim& \sum_{k} \sum_{k_{1}, k_{2} \neq 0: k_{12} = k} \frac{ \lvert \epsilon k_{12} \rvert^{\eta}}{\lvert k_{1}\rvert^{2} \lvert k_{2} \rvert^{2} \lvert k_{12} \rvert^{2}} \theta(2^{-q} k)^{2} \nonumber\\
& \times \lvert  \sum_{\lvert i-j\rvert \leq 1} \sum_{i_{1}, j_{1} =1}^{3} \sum_{k_{3} \neq 0} \theta(2^{-j} k_{3})   \int_{0}^{t} \frac{ e^{- \lvert k_{3} \rvert^{2} (t-s)} h_{u}(\epsilon k_{3}) h_{b}(\epsilon k_{3})}{\lvert k_{3} \rvert^{2} } \nonumber\\
& \hspace{5mm} \times[ \theta(2^{-i}k_{123}) e^{- \lvert k_{123} \rvert^{2} (t-s)} k_{123}^{j_{1}} i \hat{\mathcal{P}}^{i_{0} i_{1}}(k_{123}) \nonumber\\
& \hspace{10mm} - \theta(2^{-i}k_{3}) e^{- \lvert k_{3} \rvert^{2} (t-s)} k_{3}^{j_{1}} i \hat{\mathcal{P}}^{i_{0} i_{1}}(k_{3})] ds \rvert^{2}. 
\end{align} 
Next, in order to compute $N_{q, t, i_{0} j_{0}}^{3}$, we see that 
\begin{align}\label{estimate 240}
& \lvert \int_{[0,t]^{2}} \mathbb{E} [ (: \hat{X}_{\sigma, u}^{\epsilon, i_{2}}(k_{1}) \hat{X}_{\sigma, u}^{\epsilon, i_{3}}(k_{2}): - : \hat{X}_{\sigma, b}^{\epsilon, i_{2}}(k_{1}) \hat{X}_{\sigma, b}^{\epsilon, i_{3}}(k_{2}): \\
& \hspace{8mm} - : \hat{\bar{X}}_{\sigma, u}^{\epsilon, i_{2}}(k_{1}) \hat{\bar{X}}_{\sigma, u}^{\epsilon, i_{3}}(k_{2}): + \hat{\bar{X}}_{\sigma, b}^{\epsilon, i_{2}}(k_{1}) \hat{\bar{X}}_{\sigma, b}^{\epsilon, i_{3}}(k_{2}): ) \nonumber\\
& \hspace{8mm} \times \overline{: \hat{X}_{\bar{\sigma}, u}^{\epsilon, i_{2}'}(k_{1}') \hat{X}_{\bar{\sigma}, u}^{\epsilon, i_{3}'}(k_{2}'): - : \hat{X}_{\bar{\sigma}, b}^{\epsilon, i_{2}'}(k_{1}') \hat{X}_{\bar{\sigma}, b}^{\epsilon, i_{3}'}(k_{2}'):} \nonumber\\
& \hspace{8mm} \overline{- : \hat{\bar{X}}_{\bar{\sigma}, u}^{\epsilon, i_{2}'}(k_{1}') \hat{\bar{X}}_{\bar{\sigma}, u}^{\epsilon, i_{3}'}(k_{2}'): + \hat{\bar{X}}_{\bar{\sigma}, b}^{\epsilon, i_{2}'}(k_{1}') \hat{\bar{X}}_{\bar{\sigma}, b}^{\epsilon, i_{3}'}(k_{2}'): }]  e^{- \lvert k_{12} \rvert^{2} (t-\sigma)} e^{- \lvert k_{12} '\rvert^{2} (t- \bar{\sigma})}  \nonumber\\
&\hspace{18mm}  \times k_{12}^{i_{3}} i (k_{12}')^{i_{3}'} i d \sigma d \bar{\sigma} \rvert \lesssim (1_{k_{1} = k_{1}', k_{2} = k_{2}'} + 1_{k_{1} = k_{2}', k_{2} = k_{1}'}) \frac{ \lvert \epsilon k_{1} \rvert^{\eta} + \lvert \epsilon k_{2} \rvert^{\eta}}{\lvert k_{1} \rvert^{2} \lvert k_{2} \rvert^{2} \lvert k_{12} \rvert^{2}} \nonumber
\end{align} 
by H$\ddot{\mathrm{o}}$lder's inequality and the proof of \eqref{[Equation (4.7)][ZZ17]}. Thus, applying \eqref{estimate 240} to \eqref{estimate 238} gives 
\begin{align}\label{estimate 242}
&N_{q, t, i_{0} j_{0}}^{3} \lesssim \sum_{k} \sum_{k_{1}, k_{2} \neq 0: k_{12} = k} \frac{ \lvert \epsilon k_{1} \rvert^{\eta} + \lvert \epsilon k_{2} \rvert^{\eta}}{\lvert k_{1} \rvert^{2} \lvert k_{2} \rvert^{2} \lvert k_{12} \rvert^{2}} \theta(2^{-q} k)^{2} \\
& \hspace{1mm}   \times \lvert \sum_{\lvert i-j \rvert \leq 1} \sum_{i_{1}, j_{1} =1}^{3} \sum_{k_{3} \neq 0} \theta(2^{-j} k_{3}) \int_{0}^{t} \frac{ e^{- \lvert k_{3} \rvert^{2} (t-s)} h_{u}(\epsilon k_{3}) h_{b}(\epsilon k_{3})}{ \lvert k_{3} \rvert^{2}} \nonumber\\
& \hspace{1mm} \times [ \theta(2^{-i} k_{123}) e^{- \lvert k_{123} \rvert^{2} (t-s)} k_{123}^{j_{1}} i \hat{\mathcal{P}}^{i_{0} i_{1}}(k_{123})  - \theta(2^{-i} k_{3}) e^{- \lvert k_{3} \rvert^{2} (t-s)} k_{3}^{j_{1}} i \hat{\mathcal{P}}^{i_{0} i_{1}} (k_{3}) ] ds \rvert^{2}.  \nonumber
\end{align} 
Next, within the bound on $N_{q, t, i_{0} j_{0}}^{1}$ in \eqref{estimate 18}, we can bound for any $\eta \in (0,1)$, 
\begin{align}
& \lvert \sum_{i_{1}, j_{1} =1}^{3} [ \frac{ e^{- \lvert k_{3} \rvert^{2} f(\epsilon k_{3})(t-s)} h_{u}(\epsilon k_{3}) h_{b}(\epsilon k_{3})}{ \lvert k_{3} \rvert^{2} f(\epsilon k_{3})} \nonumber\\
& \hspace{15mm} \times [ \theta(2^{-i}k_{123}) e^{- \lvert k_{123} \rvert^{2} f(\epsilon k_{123} )(t-s)} k_{123}^{j_{1}} g(\epsilon k_{123}^{j_{1}}) \hat{\mathcal{P}}^{i_{0} i_{1}}(k_{123}) \nonumber\\
& \hspace{20mm} - \theta(2^{-i}k_{3}) e^{- \lvert k_{3} \rvert^{2} f(\epsilon k_{3} )(t-s)} k_{3}^{j_{1}} g(\epsilon k_{3}^{j_{1}}) \hat{\mathcal{P}}^{i_{0} i_{1}}(k_{3})]   \nonumber\\
& \hspace{5mm} - \frac{ e^{- \lvert k_{3} \rvert^{2} (t-s)} h_{u}(\epsilon k_{3}) h_{b}(\epsilon k_{3})}{ \lvert k_{3} \rvert^{2} } [ \theta(2^{-i}k_{123}) e^{- \lvert k_{123} \rvert^{2} (t-s)} k_{123}^{j_{1}} i \hat{\mathcal{P}}^{i_{0} i_{1}}(k_{123}) \nonumber\\
& \hspace{15mm} - \theta(2^{-i}k_{3}) e^{- \lvert k_{3} \rvert^{2} (t-s)} k_{3}^{j_{1}} i \hat{\mathcal{P}}^{i_{0} i_{1}}(k_{3})] ] \rvert \nonumber\\
\lesssim& \frac{ e^{- \lvert k_{3} \rvert^{2} \bar{c}_{f} (t-s)}}{\lvert k_{3} \rvert^{2}} ( \lvert k_{12} \rvert^{2\eta} (t-s)^{\frac{-1 + 2\eta}{2}} \wedge (t-s)^{-\frac{1}{2}} [ \lvert \epsilon k_{3} \rvert^{\eta} + \lvert \epsilon k_{123} \rvert^{\eta} ]
\end{align}
by Lemmas \ref{Lemma 3.8} and \ref{Lemma 3.9} for the first estimate, while \eqref{[Equation (4.2)][ZZ17]}, \eqref{[Equation (4.3)][ZZ17]} and \eqref{key estimate} for the second estimate. 

We will apply this to \eqref{estimate 18}. Now we see that due to the lack of $\theta(2^{-i} k_{123})$ in some terms of \eqref{estimate 18}, we cannot estimate $2^{q} \approx \lvert k \rvert = \lvert k_{12} \rvert \leq \lvert k_{123} \rvert + \lvert k_{3} \rvert \approx 2^{i} + 2^{j} \approx 2^{j}$ as $\lvert i-j \rvert \leq 1$ and conclude that $q \lesssim j$. Nevertheless, we estimate from \eqref{estimate 18} for $\epsilon \in (0,\frac{\eta}{2})$, 
\begin{align}\label{estimate 243}
N_{q, t, i_{0} j_{0}}^{1} \lesssim& \sum_{k} \theta( 2^{-q} k)^{2} \sum_{k_{1}, k_{2}\neq 0: k_{12} = k} \frac{1}{\lvert k_{1} \rvert^{2} \lvert k_{2} \rvert^{2} \lvert k_{12} \rvert^{2}} \lvert \sum_{\lvert i-j \rvert \leq 1} \sum_{k_{3} \neq 0} \theta(2^{-j} k_{3})  \\
& \times \int_{0}^{t} \frac{ e^{- \lvert k_{3}\rvert^{2} \bar{c}_{f}(t-s)}}{\lvert k_{3} \rvert^{2} }  \lvert k_{12} \rvert^{\eta} (t-s)^{-\frac{1}{2} + \frac{\eta}{2}} ( \lvert \epsilon k_{3} \rvert^{\frac{\eta}{2}} + \lvert \epsilon k_{123} \rvert^{\frac{\eta}{2}}) ds \rvert^{2} \lesssim \epsilon^{\eta} t^{\epsilon} 2^{q2\eta} \nonumber 
\end{align}  
where we used \eqref{key estimate} and Lemma \ref{Lemma 3.13}.  Next, again, due to the lack of $\theta(2^{-i} k_{123})$ in \eqref{estimate 241} and \eqref{estimate 242}, we cannot estimate $q\lesssim j$ and thus we estimate for $\epsilon \in (0, \frac{\eta}{2})$ from \eqref{estimate 241} and \eqref{estimate 242} 
\begin{align}\label{estimate 244} 
\sum_{i=1}^{2} N_{q, t, i_{0} j_{0}}^{i} &\lesssim \epsilon^{\eta} \sum_{k} \sum_{k_{1}, k_{2}\neq 0: k_{12} =k} \frac{ \lvert k_{1} \rvert^{\eta} + \lvert k_{2} \rvert^{\eta}}{ \lvert k_{1} \rvert^{2} \lvert k_{2} \rvert^{2} \lvert k_{12} \rvert^{2}} \theta(2^{-q} k)^{2} \\
& \times \lvert \sum_{\lvert i-j \rvert \leq 1} \sum_{k_{3} \neq 0} \theta(2^{-j} k_{3}) \int_{0}^{t} \frac{e^{- \lvert k_{3} \rvert^{2} (t-s)}}{\lvert k_{3} \rvert^{2}} \lvert k_{12} \rvert^{\eta} (t-s)^{\frac{-1 + \eta}{2}} ds \rvert^{2} \lesssim \epsilon^{\eta} t^{\epsilon} 2^{q3\eta} \nonumber
\end{align} 
by \eqref{key estimate} and Lemma \ref{Lemma 3.8}. Applying \eqref{estimate 243}- \eqref{estimate 244} to \eqref{estimate 224} and applying the resulting inequality together with \eqref{estimate 245} to \eqref{estimate 246} gives us 
\begin{align}\label{estimate 247}
\mathbb{E} [ \lvert \Delta_{q} (L_{t, i_{0}j_{0}}^{41} + 4 \sum_{i_{1} =1}^{3} u_{2}^{\epsilon, i_{1}}(t) C_{3,u}^{\epsilon, i_{0} i_{1} j_{0}}(t)) \rvert^{2} ] \lesssim \epsilon^{\eta} t^{\epsilon} 2^{q (3 \eta + 2 \epsilon)}.
\end{align}
Next, we work on $L_{t, i_{0}j_{0}}^{5}$ and $\bar{L}_{t, i_{0}j_{0}}^{5}$ respectively from \eqref{[Equation (4.6j)][ZZ17]} and \eqref{[Equation (4.6k)][ZZ17]}. Because $\bar{L}_{t, i_{0}j_{0}}^{5}$ has two parts, we split the first to be $\bar{L}_{t, i_{0}j_{0}}^{51}$ and the second to be $\bar{L}_{t, i_{0}j_{0}}^{52}$; i.e., $\bar{L}_{t, i_{0} j_{0}}^{5} = \sum_{l=1}^{2} \bar{L}_{t, i_{0} j_{0}}^{5l}$ where  
\begin{subequations}\label{estimate 248}
\begin{align}
\bar{L}_{t, i_{0}j_{0}}^{51} \triangleq&  2\pi_{0} (\sum_{i_{1}, i_{2}, j_{1} =1}^{3} \int_{0}^{t} e^{-(s-t) \Delta_{\epsilon}} \nonumber\\
& \times  \mathcal{P}^{i_{0} i_{1}} D_{j_{1}}^{\epsilon}(C_{1,u}^{\epsilon, i_{1} i_{2} j_{1}} u_{1}^{\epsilon, i_{2}} + C_{1,b}^{\epsilon, i_{1} i_{2} j_{1}} b_{1}^{\epsilon, i_{2}}) ds, b_{1}^{\epsilon, j_{0}}(t)),  \\
\bar{L}_{t, i_{0}j_{0}}^{52} \triangleq&  - 2\pi_{0} (\sum_{i_{1}, i_{2}, j_{1} =1}^{3} \int_{0}^{t} e^{-(s-t) \Delta_{\epsilon} } \nonumber\\
& \times \mathcal{P}^{i_{0} i_{1}} D_{j_{1}}^{\epsilon} (C_{2,u}^{\epsilon, i_{1} i_{2} j_{1}} u_{1}^{\epsilon, i_{2}} + C_{2,b}^{\epsilon, i_{1} i_{2} j_{1}}b_{1}^{\epsilon, i_{2}}) ds, b_{1}^{\epsilon, j_{0}}(t)), 
\end{align}
\end{subequations} 
where we recall $C_{2,u}^{\epsilon, i i_{1} j} (t) $ and $C_{2,b}^{\epsilon, i i_{1} j} (t)$ respectively from \eqref{C2uepsilonii1j}, \eqref{C2bepsilonii1j}, and we also split $L_{t, i_{0}j_{0}}^{5}$ between the first and second four terms: $L_{t, i_{0} j_{0}}^{5} = \sum_{l=1}^{2} L_{t, i_{0} j_{0}}^{5l}$ where  
\begin{subequations}\label{estimate 264}
\begin{align}
L_{t, i_{0}j_{0}}^{51} \triangleq &  (2\pi)^{-\frac{9}{2}} \sum_{\lvert i -j \rvert \leq 1} \sum_{i_{1}, i_{2}, i_{3}, j_{1} =1}^{3} \sum_{k} \sum_{k_{1}, k_{2}, k_{4} \neq 0: k_{14} = k} \theta(2^{-i} k_{1}) \theta(2^{-j} k_{4})\\
& \times [ \int_{0}^{t} e^{- \lvert k_{1} \rvert^{2}  f(\epsilon k_{1})(t-s)} \int_{0}^{s} : \hat{X}_{\sigma, u}^{\epsilon, i_{2}}(k_{1}) \hat{X}_{t,b}^{\epsilon, j_{0}}(k_{4}): k_{12}^{i_{3}} g(\epsilon k_{12}^{i_{3}}) k_{1}^{j_{1}} g(\epsilon k_{1}^{j_{1}}) \nonumber\\
& \hspace{5mm} \times \frac{ e^{- \lvert k_{2} \rvert^{2} f(\epsilon k_{2}) (s-\sigma)} h_{u}(\epsilon k_{2})^{2}}{2 \lvert k_{2} \rvert^{2} f(\epsilon k_{2})} e^{- \lvert k_{12} \rvert^{2} f(\epsilon k_{12})(s-\sigma)}d\sigma ds \nonumber  \\
& -  \int_{0}^{t} e^{- \lvert k_{1} \rvert^{2} (t-s) } \int_{0}^{s} : \hat{\bar{X}}_{\sigma, u}^{\epsilon, i_{2}}(k_{1}) \hat{\bar{X}}_{t,b}^{\epsilon, j_{0}}(k_{4}): k_{12}^{i_{3}} i k_{1}^{j_{1}} i \nonumber\\
& \hspace{5mm} \times \frac{ e^{- \lvert k_{2} \rvert^{2} (s-\sigma)} h_{u}(\epsilon k_{2})^{2}}{2 \lvert k_{2} \rvert^{2}} e^{- \lvert k_{12} \rvert^{2}(s-\sigma)}d\sigma ds \nonumber  \\
& - \int_{0}^{t} e^{- \lvert k_{1} \rvert^{2}  f(\epsilon k_{1})(t-s)} \int_{0}^{s} : \hat{X}_{\sigma, b}^{\epsilon, i_{2}}(k_{1}) \hat{X}_{t,b}^{\epsilon, j_{0}}(k_{4}): k_{12}^{i_{3}} g(\epsilon k_{12}^{i_{3}}) k_{1}^{j_{1}} g(\epsilon k_{1}^{j_{1}}) \nonumber\\
& \hspace{5mm} \times \frac{ e^{- \lvert k_{2} \rvert^{2} f(\epsilon k_{2}) (s-\sigma)} h_{b}(\epsilon k_{2}) h_{u}(\epsilon k_{2})}{2 \lvert k_{2} \rvert^{2} f(\epsilon k_{2})} e^{- \lvert k_{12} \rvert^{2} f(\epsilon k_{12})(s-\sigma)}d\sigma ds \nonumber  \\
& +  \int_{0}^{t} e^{- \lvert k_{1} \rvert^{2} (t-s) } \int_{0}^{s} : \hat{\bar{X}}_{\sigma, b}^{\epsilon, i_{2}}(k_{1}) \hat{\bar{X}}_{t,b}^{\epsilon, j_{0}}(k_{4}): k_{12}^{i_{3}} i k_{1}^{j_{1}} i \nonumber\\
& \hspace{5mm} \times \frac{ e^{- \lvert k_{2} \rvert^{2} (s-\sigma)} h_{b}(\epsilon k_{2}) h_{u}(\epsilon k_{2})}{2 \lvert k_{2} \rvert^{2}} e^{- \lvert k_{12} \rvert^{2}(s-\sigma)}d\sigma ds] \nonumber  \\
& \times \sum_{i_{4} =1}^{3} \hat{\mathcal{P}}^{i_{3} i_{4}}(k_{2}) \hat{\mathcal{P}}^{j_{1} i_{4}}(k_{2}) \hat{\mathcal{P}}^{i_{1} i_{2}}(k_{12}) \hat{\mathcal{P}}^{i_{0} i_{1}}(k_{1}) e_{k}, \nonumber\\
L_{t, i_{0}j_{0}}^{52} \triangleq&  (2\pi)^{-\frac{9}{2}} \sum_{\lvert i -j \rvert \leq 1} \sum_{i_{1}, i_{2}, i_{3}, j_{1} =1}^{3} \sum_{k} \sum_{k_{1}, k_{2}, k_{4}\neq 0: k_{14} = k} \theta(2^{-i} k_{1}) \theta(2^{-j} k_{4})\\
& \times [- \int_{0}^{t} e^{- \lvert k_{1} \rvert^{2}  f(\epsilon k_{1})(t-s)} \int_{0}^{s} : \hat{X}_{\sigma, b}^{\epsilon, i_{2}}(k_{1}) \hat{X}_{t,b}^{\epsilon, j_{0}}(k_{4}): k_{12}^{i_{3}} g(\epsilon k_{12}^{i_{3}}) k_{1}^{j_{1}} g(\epsilon k_{1}^{j_{1}}) \nonumber\\
& \hspace{5mm} \times \frac{ e^{- \lvert k_{2} \rvert^{2} f(\epsilon k_{2}) (s-\sigma)} h_{u}(\epsilon k_{2})h_{b}(\epsilon k_{2})}{2 \lvert k_{2} \rvert^{2} f(\epsilon k_{2})} e^{- \lvert k_{12} \rvert^{2}f(\epsilon k_{12})(s-\sigma) }d\sigma ds \nonumber  \\
& +  \int_{0}^{t} e^{- \lvert k_{1} \rvert^{2} (t-s) } \int_{0}^{s} : \hat{\bar{X}}_{\sigma, b}^{\epsilon, i_{2}}(k_{1}) \hat{\bar{X}}_{t,b}^{\epsilon, j_{0}}(k_{4}): k_{12}^{i_{3}} i k_{1}^{j_{1}} i \nonumber\\
& \hspace{5mm} \times \frac{ e^{- \lvert k_{2} \rvert^{2} (s-\sigma)} h_{u}(\epsilon k_{2})h_{b}(\epsilon k_{2})}{2 \lvert k_{2} \rvert^{2}} e^{- \lvert k_{12} \rvert^{2}(s-\sigma)}d\sigma ds \nonumber  \\
&+ \int_{0}^{t} e^{- \lvert k_{1} \rvert^{2} f(\epsilon k_{1})(t-s) } \int_{0}^{s} : \hat{X}_{\sigma, u}^{\epsilon, i_{2}}(k_{1}) \hat{X}_{t,b}^{\epsilon, j_{0}}(k_{4}): k_{12}^{i_{3}} g(\epsilon k_{12}^{i_{3}}) k_{1}^{j_{1}} g(\epsilon k_{1}^{j_{1}}) \nonumber\\
& \hspace{5mm} \times \frac{ e^{- \lvert k_{2} \rvert^{2} f(\epsilon k_{2}) (s-\sigma)} h_{b}(\epsilon k_{2})^{2}}{2 \lvert k_{2} \rvert^{2} f(\epsilon k_{2})} e^{- \lvert k_{12} \rvert^{2} f(\epsilon k_{12})(s-\sigma)}d\sigma ds \nonumber  \\
& -  \int_{0}^{t} e^{- \lvert k_{1} \rvert^{2} (t-s) } \int_{0}^{s} : \hat{\bar{X}}_{\sigma, u}^{\epsilon, i_{2}}(k_{1}) \hat{\bar{X}}_{t,b}^{\epsilon, j_{0}}(k_{4}): k_{12}^{i_{3}} i k_{1}^{j_{1}} i \nonumber\\
& \hspace{5mm} \times \frac{ e^{- \lvert k_{2} \rvert^{2} (s-\sigma)} h_{b}(\epsilon k_{2})^{2}}{2 \lvert k_{2} \rvert^{2}} e^{- \lvert k_{12} \rvert^{2}(s-\sigma)}d\sigma ds] \nonumber  \\
& \times \sum_{i_{4} =1}^{3} \hat{\mathcal{P}}^{i_{3} i_{4}}(k_{2}) \hat{\mathcal{P}}^{j_{1} i_{4}}(k_{2}) \hat{\mathcal{P}}^{i_{1} i_{2}}(k_{12}) \hat{\mathcal{P}}^{i_{0} i_{1}}(k_{1}) e_{k}. \nonumber
\end{align} 
\end{subequations} 
W.l.o.g. we work on $L_{t, i_{0} j_{0}}^{52}$ which we rewrite it as $L_{t, i_{0} j_{0}}^{52} = L_{t, i_{0}j_{0}}^{52} - \tilde{L}_{t, i_{0}j_{0}}^{52} +\tilde{L}_{t, i_{0}j_{0}}^{52} + \bar{L}_{t, i_{0}j_{0}}^{52}$ where 
\begin{align}\label{estimate 24}
\tilde{L}_{t, i_{0}j_{0}}^{52} \triangleq&  (2\pi)^{-\frac{9}{2}} \sum_{\lvert i -j \rvert \leq 1} \sum_{i_{1}, i_{2}, i_{3}, j_{1} =1}^{3} \sum_{k} \sum_{k_{1}, k_{2}, k_{4}\neq 0: k_{14} = k} \theta(2^{-i} k_{1}) \theta(2^{-j} k_{4}) \nonumber\\
& \times [- \int_{0}^{t} e^{- \lvert k_{1} \rvert^{2} f(\epsilon k_{1})(t-s)} \int_{0}^{s} : \hat{X}_{s, b}^{\epsilon, i_{2}}(k_{1}) \hat{X}_{t,b}^{\epsilon, j_{0}}(k_{4}): k_{12}^{i_{3}} g(\epsilon k_{12}^{i_{3}}) k_{1}^{j_{1}} g(\epsilon k_{1}^{j_{1}})  \nonumber\\
& \hspace{5mm} \times \frac{ e^{- \lvert k_{2} \rvert^{2} f(\epsilon k_{2}) (s-\sigma)} h_{u}(\epsilon k_{2})h_{b}(\epsilon k_{2})}{2 \lvert k_{2} \rvert^{2} f(\epsilon k_{2})} e^{- \lvert k_{12} \rvert^{2} f(\epsilon k_{12})(s-\sigma)} d\sigma  ds \nonumber  \\
& +  \int_{0}^{t} e^{- \lvert k_{1} \rvert^{2} (t-s) } \int_{0}^{s} : \hat{\bar{X}}_{s, b}^{\epsilon, i_{2}}(k_{1}) \hat{\bar{X}}_{t,b}^{\epsilon, j_{0}}(k_{4}): k_{12}^{i_{3}} i k_{1}^{j_{1}} i \nonumber\\
& \hspace{5mm} \times \frac{ e^{- \lvert k_{2} \rvert^{2} (s-\sigma)} h_{u}(\epsilon k_{2})h_{b}(\epsilon k_{2})}{2 \lvert k_{2} \rvert^{2}} e^{- \lvert k_{12} \rvert^{2}(s-\sigma)}d\sigma  ds \nonumber  \\ 
&+ \int_{0}^{t} e^{- \lvert k_{1} \rvert^{2} f(\epsilon k_{1})(t-s)} \int_{0}^{s} : \hat{X}_{s, u}^{\epsilon, i_{2}}(k_{1}) \hat{X}_{t,b}^{\epsilon, j_{0}}(k_{4}): k_{12}^{i_{3}} g(\epsilon k_{12}^{i_{3}}) k_{1}^{j_{1}} g(\epsilon k_{1}^{j_{1}})  \nonumber\\
& \hspace{5mm} \times \frac{ e^{- \lvert k_{2} \rvert^{2} f(\epsilon k_{2}) (s-\sigma)} h_{b}(\epsilon k_{2})^{2}}{2 \lvert k_{2} \rvert^{2} f(\epsilon k_{2})} e^{- \lvert k_{12} \rvert^{2}f(\epsilon k_{12})(s-\sigma)}d\sigma ds \nonumber  \\
& -  \int_{0}^{t} e^{- \lvert k_{1} \rvert^{2} (t-s) } \int_{0}^{s} : \hat{\bar{X}}_{s, u}^{\epsilon, i_{2}}(k_{1}) \hat{\bar{X}}_{t,b}^{\epsilon, j_{0}}(k_{4}): k_{12}^{i_{3}} i k_{1}^{j_{1}} i  \nonumber\\
& \hspace{5mm} \times \frac{ e^{- \lvert k_{2} \rvert^{2} (s-\sigma)} h_{b}(\epsilon k_{2})^{2}}{2 \lvert k_{2} \rvert^{2}} e^{- \lvert k_{12} \rvert^{2}(s-\sigma)}d\sigma  ds] \nonumber  \\
& \times \sum_{i_{4} =1}^{3} \hat{\mathcal{P}}^{i_{3} i_{4}}(k_{2}) \hat{\mathcal{P}}^{j_{1} i_{4}}(k_{2}) \hat{\mathcal{P}}^{i_{1} i_{2}}(k_{12}) \hat{\mathcal{P}}^{i_{0} i_{1}}(k_{1}) e_{k}.
\end{align} 
Now using $\bar{C}_{2,u}^{\epsilon, i_{1} i_{2} j_{1}}(t)$ and $\bar{C}_{2,b}^{\epsilon, i_{1} i_{2} j_{1}}(t) $ respectively from  \eqref{barC2uepsilonii1j}, \eqref{barC2bepsilonii1j} which are both zeroes, we can rewrite 
$\bar{L}_{t, i_{0}j_{0}}^{52}$ of \eqref{estimate 248} as 
\begin{align}\label{estimate 25}
\bar{L}_{t, i_{0}j_{0}}^{52} &= (2\pi)^{-\frac{9}{2}} \sum_{k\neq 0} \sum_{\lvert i-j \rvert \leq 1} \sum_{k_{1}, k_{2}, k_{4} \neq 0: k_{14} = k} \sum_{i_{1}, i_{2}, i_{3}, j_{1} =1}^{3} \theta(2^{-i} k_{1}) \theta(2^{-j} k_{4}) \nonumber\\
& \times [ \int_{0}^{t} :\hat{X}_{s,u}^{\epsilon, i_{2}}(k_{1}) \hat{X}_{t,b}^{\epsilon, j_{0}}(k_{4}): e^{- \lvert k_{1} \rvert^{2} f(\epsilon k_{1}) (t-s)} k_{1}^{j_{1}} g(\epsilon k_{1}^{j_{1}}) \nonumber\\
& \hspace{10mm} \times \int_{0}^{s} e^{- \lvert k_{2} \rvert^{2} f(\epsilon k_{2}) (s-\sigma)} \frac{ e^{- \lvert k_{2} \rvert^{2} f(\epsilon k_{2})(s-\sigma)} h_{b}(\epsilon k_{2})^{2}}{2\lvert k_{2} \rvert^{2} f(\epsilon k_{2})} k_{2}^{i_{3}} g(\epsilon k_{2}^{i_{3}}) d \sigma ds \nonumber\\
& \hspace{5mm} - \int_{0}^{t} : \hat{\bar{X}}_{s,u}^{\epsilon, i_{2}}(k_{1}) \hat{\bar{X}}_{t,b}^{\epsilon, j_{0}} (k_{4}): e^{- \lvert k_{1} \rvert^{2} (t-s)} k_{1}^{j_{1}} i \nonumber\\
& \hspace{10mm} \times \int_{0}^{s} e^{- \lvert k_{2} \rvert^{2} (s-\sigma)} \frac{ e^{- \lvert k_{2} \rvert^{2} (s-\sigma)} h_{b}(\epsilon k_{2})^{2}}{2\lvert k_{2} \rvert^{2}} k_{2}^{i_{3}} i d\sigma ds \nonumber\\
& \hspace{5mm} -  \int_{0}^{t} :\hat{X}_{s,b}^{\epsilon, i_{2}}(k_{1}) \hat{X}_{t,b}^{\epsilon, j_{0}}(k_{4}): e^{- \lvert k_{1} \rvert^{2} f(\epsilon k_{1}) (t-s)} k_{1}^{j_{1}} g(\epsilon k_{1}^{j_{1}}) \nonumber\\
& \hspace{10mm} \times \int_{0}^{s} e^{- \lvert k_{2} \rvert^{2} f(\epsilon k_{2}) (s-\sigma)} \frac{ e^{- \lvert k_{2} \rvert^{2} f(\epsilon k_{2})(s-\sigma)} h_{u}(\epsilon k_{2})h_{b}(\epsilon k_{2})}{2\lvert k_{2} \rvert^{2} f(\epsilon k_{2})} k_{2}^{i_{3}} g(\epsilon k_{2}^{i_{3}}) d \sigma ds\nonumber\\
&\hspace{5mm} + \int_{0}^{t} : \hat{\bar{X}}_{s,b}^{\epsilon, i_{2}}(k_{1}) \hat{\bar{X}}_{t,b}^{\epsilon, j_{0}} (k_{4}): e^{- \lvert k_{1} \rvert^{2} (t-s)} k_{1}^{j_{1}} i\nonumber\\
& \hspace{10mm} \times  \int_{0}^{s} e^{- \lvert k_{2} \rvert^{2} (s-\sigma)} \frac{ e^{- \lvert k_{2} \rvert^{2} (s-\sigma)} h_{u}(\epsilon k_{2})h_{b}(\epsilon k_{2})}{2\lvert k_{2} \rvert^{2}} k_{2}^{i_{3}} i d\sigma ds] \nonumber\\
&\times \hat{\mathcal{P}}^{i_{1}i_{2}}(k_{2}) \sum_{i_{4} =1}^{3} \hat{\mathcal{P}}^{i_{3}i_{4}}(k_{2}) \hat{\mathcal{P}}^{j_{1}i_{4}}(k_{2}) \hat{\mathcal{P}}^{i_{0}i_{1}}(k_{1}) e_{k}. 
\end{align}
Such $\bar{L}_{t, i_{0}j_{0}}^{52}$ works perfectly well with $\tilde{L}_{t, i_{0}j_{0}}^{52}$ from \eqref{estimate 24}, as we will see in \eqref{estimate 26}. Now 
\begin{align}\label{estimate 253}
\mathbb{E} [ \lvert \Delta_{q} (L_{t, i_{0}j_{0}}^{52} - \tilde{L}_{t, i_{0}j_{0}}^{52}) \rvert^{2} ] 
\lesssim \sum_{l=1}^{2} O_{q, t, i_{0} j_{0}}^{l} 
\end{align} 
where 
\begin{subequations}
\begin{align}
O_{q, t, i_{0} j_{0}}^{1} &\triangleq \mathbb{E} [ \lvert \sum_{k} \theta(2^{-q} k) \sum_{\lvert i-j \rvert \leq 1} \sum_{k_{1}, k_{2}, k_{4}\neq 0: k_{14} = k} \sum_{i_{1}, i_{2}, i_{3}, j_{1} =1}^{3} \theta(2^{-i} k_{1}) \theta(2^{-j} k_{4}) \nonumber\\
& \times [- \int_{0}^{t} e^{- \lvert k_{1} \rvert^{2}  f(\epsilon k_{1})(t-s)} \int_{0}^{s} [: \hat{X}_{\sigma, b}^{\epsilon, i_{2}}(k_{1}) \hat{X}_{t,b}^{\epsilon, j_{0}}(k_{4}): - : \hat{X}_{s,b}^{\epsilon, i_{2}}(k_{1}) \hat{X}_{t,b}^{\epsilon, j_{0}} (k_{4}):] \nonumber\\
& \hspace{5mm} \times \frac{ e^{- \lvert k_{2} \rvert^{2} f(\epsilon k_{2}) (s-\sigma)} h_{u}(\epsilon k_{2})h_{b}(\epsilon k_{2})}{2 \lvert k_{2} \rvert^{2} f(\epsilon k_{2})}  \nonumber\\
& \hspace{5mm} \times e^{- \lvert k_{12} \rvert^{2} f(\epsilon k_{12})(s-\sigma)} k_{12}^{i_{3}} g(\epsilon k_{12}^{i_{3}}) k_{1}^{j_{1}} g(\epsilon k_{1}^{j_{1}})d\sigma ds \nonumber\\
& +  \int_{0}^{t} e^{- \lvert k_{1} \rvert^{2} (t-s) } \int_{0}^{s} [: \hat{\bar{X}}_{\sigma, b}^{\epsilon, i_{2}}(k_{1}) \hat{\bar{X}}_{t,b}^{\epsilon, j_{0}}(k_{4}): - : \hat{\bar{X}}_{s,b}^{\epsilon, i_{2}}(k_{1}) \hat{\bar{X}}_{t,b}^{\epsilon, j_{0}} (k_{4}):]   \nonumber\\
& \hspace{5mm} \times \frac{ e^{- \lvert k_{2} \rvert^{2} (s-\sigma)} h_{u}(\epsilon k_{2})h_{b}(\epsilon k_{2})}{2 \lvert k_{2} \rvert^{2}} e^{- \lvert k_{12} \rvert^{2}(s-\sigma)} k_{12}^{i_{3}} i k_{1}^{j_{1}} i d\sigma ds] \nonumber  \\
& \times \sum_{i_{4} =1}^{3} \hat{\mathcal{P}}^{i_{3} i_{4}}(k_{2}) \hat{\mathcal{P}}^{j_{1} i_{4}}(k_{2}) \hat{\mathcal{P}}^{i_{1} i_{2}}(k_{12}) \hat{\mathcal{P}}^{i_{0} i_{1}}(k_{1}) e_{k} \rvert^{2} ], \\
O_{q, t, i_{0} j_{0}}^{2} &\triangleq \mathbb{E} [ \lvert \sum_{k} \theta(2^{-q} k) \sum_{\lvert i-j \rvert \leq 1} \sum_{k_{1}, k_{2}, k_{4} \neq 0: k_{14} = k} \sum_{i_{1}, i_{2}, i_{3}, j_{1} =1}^{3} \theta(2^{-i} k_{1}) \theta(2^{-j} k_{4}) \nonumber\\
& \times [\int_{0}^{t} e^{- \lvert k_{1} \rvert^{2}  f(\epsilon k_{1})(t-s)} \int_{0}^{s} [: \hat{X}_{\sigma, u}^{\epsilon, i_{2}}(k_{1}) \hat{X}_{t,b}^{\epsilon, j_{0}}(k_{4}): - : \hat{X}_{s,u}^{\epsilon, i_{2}} (k_{1}) \hat{X}_{t,b}^{\epsilon, j_{0}} (k_{4}):] \nonumber\\
& \hspace{5mm} \times \frac{ e^{- \lvert k_{2} \rvert^{2} f(\epsilon k_{2}) (s-\sigma)} h_{b}(\epsilon k_{2})^{2}}{2 \lvert k_{2} \rvert^{2} f(\epsilon k_{2})} e^{- \lvert k_{12} \rvert^{2}f(\epsilon k_{12})(s-\sigma)} k_{12}^{i_{3}} g(\epsilon k_{12}^{i_{3}}) k_{1}^{j_{1}} g(\epsilon k_{1}^{j_{1}}) d\sigma ds \nonumber  \\
& -  \int_{0}^{t} e^{- \lvert k_{1} \rvert^{2} (t-s) } \int_{0}^{s} [: \hat{\bar{X}}_{\sigma, u}^{\epsilon, i_{2}}(k_{1}) \hat{\bar{X}}_{t,b}^{\epsilon, j_{0}}(k_{4}): - : \hat{\bar{X}}_{s,u}^{\epsilon, i_{2}}(k_{1}) \hat{\bar{X}}_{t,b}^{\epsilon, j_{0}} (k_{4}):]  \nonumber\\
& \hspace{5mm} \times \frac{ e^{- \lvert k_{2} \rvert^{2} (s-\sigma)} h_{b}(\epsilon k_{2})^{2}}{2 \lvert k_{2} \rvert^{2}} e^{- \lvert k_{12} \rvert^{2}(s-\sigma)} k_{12}^{i_{3}} i k_{1}^{j_{1}} i  d\sigma ds] \nonumber  \\
& \times \sum_{i_{4} =1}^{3} \hat{\mathcal{P}}^{i_{3} i_{4}}(k_{2}) \hat{\mathcal{P}}^{j_{1} i_{4}}(k_{2}) \hat{\mathcal{P}}^{i_{1} i_{2}}(k_{12}) \hat{\mathcal{P}}^{i_{0} i_{1}}(k_{1}) e_{k} \rvert^{2} ]. 
\end{align}
\end{subequations} 
W.l.o.g. we show the estimates on $O_{q, t, i_{0} j_{0}}^{1}$ as those on $O_{q, t, i_{0} j_{0}}^{2}$ are similar. We estimate 
\begin{align}\label{estimate 251}
O_{q, t, i_{0} j_{0}}^{1}\lesssim \sum_{l=1}^{3} O_{q, t, i_{0} j_{0}}^{1l} 
\end{align} 
where 
\begin{subequations}\label{estimate 249}
\begin{align}
O_{q, t, i_{0} j_{0}}^{11} &\triangleq \sum_{k} \sum_{\lvert i-j \rvert \leq 1, \lvert i'-j'\rvert \leq 1} \sum_{k_{1}, k_{2}, k_{4}, k_{1}', k_{2}', k_{4}' \neq 0: k_{14} = k_{14}' = k} \sum_{i_{1}, i_{2}, i_{3}, j_{1}, i_{1}', i_{2}', i_{3}', j_{1}' = 1}^{3}\\
& \times  \theta(2^{-q} k)^{2}  \theta(2^{-i} k_{1}) \theta(2^{-i'} k_{1}') \theta(2^{-j} k_{4}) \theta(2^{-j'} k_{4}') \int_{[0,t]^{2}} \int_{0}^{s} \int_{0}^{\bar{s}} \nonumber \\
& \times\lvert e^{- \lvert k_{1} \rvert^{2} f(\epsilon k_{1}) (t-s)} k_{1}^{j_{1}} g(\epsilon k_{1}^{j_{1}}) - e^{- \lvert k_{1} \rvert^{2} (t-s)} k_{1}^{j_{1}} i \rvert \nonumber\\
& \hspace{5mm} \times \lvert e^{- \lvert k_{1}' \rvert^{2} f(\epsilon k_{1}') (t- \bar{s}) } (k_{1}')^{j_{1}'} g(\epsilon (k_{1}')^{j_{1}'} ) - e^{- \lvert k_{1}' \rvert (t- \bar{s})} (k_{1}')^{j_{1}' } i \rvert  \nonumber\\
& \times \mathbb{E} [ ( : \hat{X}_{\sigma, b}^{\epsilon, i_{2}}(k_{1}) \hat{X}_{t,b}^{\epsilon, j_{0}}(k_{4}): - : \hat{X}_{s,b}^{\epsilon, i_{2}}(k_{1}) \hat{X}_{t,b}^{\epsilon, j_{0}}(k_{4}):) \nonumber\\
& \hspace{5mm} \times \overline{(: \hat{X}_{\bar{\sigma}, b}^{\epsilon, i_{2}'} (k_{1}') \hat{X}_{t,b}^{\epsilon, j_{0}}(k_{4}') : - : \hat{X}_{\bar{s}, b}^{\epsilon, i_{2}'} (k_{1}') \hat{X}_{t,b}^{\epsilon, j_{0}}(k_{4}'):)}]\nonumber\\
& \times [ e^{- \lvert k_{12} \rvert^{2} f(\epsilon k_{12})(s-\sigma)} \frac{ e^{- \lvert k_{2} \rvert^{2} f(\epsilon k_{2}) (s-\sigma)} h_{u}(\epsilon k_{2}) h_{b}(\epsilon k_{2})}{\lvert k_{2} \rvert^{2} f(\epsilon k_{2})} k_{12}^{i_{3}} g(\epsilon k_{12}^{i_{3}}) \hat{\mathcal{P}}^{i_{1} i_{2}}(k_{12})] \nonumber\\
& \times [ e^{- \lvert k_{12}' \rvert^{2} f(\epsilon k_{12}')(\bar{s}-\bar{\sigma})} \frac{ e^{- \lvert k_{2}' \rvert^{2} f(\epsilon k_{2}') (\bar{s}-\bar{\sigma})} h_{u}(\epsilon k_{2}') h_{b}(\epsilon k_{2}')}{\lvert k_{2}' \rvert^{2} f(\epsilon k_{2}')}\nonumber\\
& \hspace{15mm} \times (k_{12}')^{i_{3}'} g(\epsilon (k_{12}')^{i_{3}'}) \hat{\mathcal{P}}^{i_{1}' i_{2}'}(k_{12}')]  d\bar{\sigma}  d\sigma ds d \bar{s}, \nonumber\\
O_{q, t, i_{0} j_{0}}^{12} &\triangleq \sum_{k} \sum_{\lvert i-j \rvert \leq 1, \lvert i'-j'\rvert \leq 1} \sum_{k_{1}, k_{2}, k_{4}, k_{1}', k_{2}', k_{4}' \neq 0: k_{14} = k_{14}' = k} \sum_{i_{1}, i_{2}, i_{3}, i_{1}', i_{2}', i_{3}' = 1}^{3}\\
& \times  \theta(2^{-q} k)^{2}  \theta(2^{-i} k_{1}) \theta(2^{-i'} k_{1}') \theta(2^{-j} k_{4}) \theta(2^{-j'} k_{4}') \int_{[0,t]^{2}} \int_{0}^{s} \int_{0}^{\bar{s}} \nonumber \\
& \times e^{- \lvert k_{1} \rvert^{2} (t-s) - \lvert k_{1}' \rvert^{2} (t-\bar{s})} \lvert k_{1} \rvert \lvert k_{1} ' \rvert \nonumber\\
&\times \mathbb{E} [  ( : \hat{X}_{\sigma, b}^{\epsilon, i_{2}} (k_{1}) \hat{X}_{t,b}^{\epsilon, j_{0}}(k_{4}): - : \hat{X}_{s,b}^{\epsilon, i_{2}}(k_{1}) \hat{X}_{t,b}^{\epsilon, j_{0}}(k_{4}): \nonumber\\
& \hspace{15mm}- : \hat{\bar{X}}_{\sigma, b}^{\epsilon, i_{2}}(k_{1}) \hat{\bar{X}}_{t,b}^{\epsilon, j_{0}}(k_{4}): + :\hat{\bar{X}}_{s,b}^{\epsilon, i_{2}}(k_{1}) \hat{\bar{X}}_{t,b}^{\epsilon, j_{0}}(k_{4}): )  \nonumber\\
& \times \overline{( : \hat{X}_{\bar{\sigma}, b}^{\epsilon, i_{2}'} (k_{1}') \hat{X}_{t,b}^{\epsilon, j_{0}}(k_{4}'): - : \hat{X}_{\bar{s},b}^{\epsilon, i_{2}'}(k_{1}') \hat{X}_{t,b}^{\epsilon, j_{0}}(k_{4}'):  } \nonumber\\
& \hspace{15mm} \overline{- : \hat{\bar{X}}_{\bar{\sigma}, b}^{\epsilon, i_{2}'}(k_{1}') \hat{\bar{X}}_{t,b}^{\epsilon, j_{0}}(k_{4}'): + :\hat{\bar{X}}_{\bar{s},b}^{\epsilon, i_{2}'}(k_{1}') \hat{\bar{X}}_{t,b}^{\epsilon, j_{0}}(k_{4}'): ) }] \nonumber\\
& \times e^{- \lvert k_{12} \rvert^{2} f(\epsilon k_{12})(s-\sigma)} \frac{ e^{- \lvert k_{2} \rvert^{2} f(\epsilon k_{2})(s-\sigma)} h_{u}(\epsilon k_{2}) h_{b}(\epsilon k_{2})}{\lvert k_{2} \rvert^{2} f(\epsilon k_{2})} k_{12}^{i_{3}} g(\epsilon k_{12}^{i_{3}}) \hat{\mathcal{P}}^{i_{1} i_{2}}(k_{12})  \nonumber \\  
& \times e^{- \lvert k_{12}' \rvert^{2} f(\epsilon k_{12}')(\bar{s}-\bar{\sigma})} \frac{ e^{- \lvert k_{2}' \rvert^{2} f(\epsilon k_{2}')(\bar{s}-\bar{\sigma})} h_{u}(\epsilon k_{2}') h_{b}(\epsilon k_{2}')}{\lvert k_{2}' \rvert^{2} f(\epsilon k_{2}')}  \nonumber \\  
&\times (k_{12}')^{i_{3}'} g(\epsilon (k_{12}')^{i_{3}'}) \hat{\mathcal{P}}^{i_{1}' i_{2}'}(k_{12}') d\bar{\sigma}  d\sigma ds d \bar{s}, \nonumber\\
O_{q, t, i_{0} j_{0}}^{13} &\triangleq \sum_{k} \sum_{\lvert i-j \rvert \leq 1, \lvert i'-j'\rvert \leq 1} \sum_{k_{1}, k_{2}, k_{4}, k_{1}', k_{2}', k_{4}' \neq 0: k_{14} = k_{14}' = k} \sum_{i_{1}, i_{2}, i_{3}, i_{1}', i_{2}', i_{3}' = 1}^{3}\\
& \times  \theta(2^{-q} k)^{2}  \theta(2^{-i} k_{1}) \theta(2^{-i'} k_{1}') \theta(2^{-j} k_{4}) \theta(2^{-j'} k_{4}') \int_{[0,t]^{2}} \int_{0}^{s} \int_{0}^{\bar{s}} \lvert k_{1} \rvert \lvert k_{1}' \rvert  \nonumber \\
& \times e^{- \lvert k_{1} \rvert^{2} (t-s) - \lvert k_{1}' \rvert^{2} (t-\bar{s})} \mathbb{E} [ ( : \hat{\bar{X}}_{\sigma, b}^{\epsilon, i_{2}}(k_{1}) \hat{\bar{X}}_{t,b}^{\epsilon, j_{0}}(k_{4}): - : \hat{\bar{X}}_{s,b}^{\epsilon, i_{2}}(k_{1}) \hat{\bar{X}}_{t,b}^{\epsilon, j_{0}}(k_{4}):) \nonumber\\
& \hspace{15mm} \times \overline{( : \hat{\bar{X}}_{\bar{\sigma}, b}^{\epsilon, i_{2}'}(k_{1}') \hat{\bar{X}}_{t,b}^{\epsilon, j_{0}}(k_{4}'): - : \hat{\bar{X}}_{\bar{s},b}^{\epsilon, i_{2}'}(k_{1}') \hat{\bar{X}}_{t,b}^{\epsilon, j_{0}}(k_{4}'):)}] \nonumber\\
& \times \lvert e^{- \lvert k_{12} \rvert^{2} f(\epsilon k_{12}) (s-\sigma)} \frac{ e^{- \lvert k_{2} \rvert^{2} f(\epsilon k_{2}) (s-\sigma)} h_{u}(\epsilon k_{2}) h_{b}(\epsilon k_{2})}{\lvert k_{2} \rvert^{2} f(\epsilon k_{2})} k_{12}^{i_{3}} g(\epsilon k_{12}^{i_{3}}) \hat{\mathcal{P}}^{i_{1} i_{2}} (k_{12}) \nonumber\\
& \hspace{5mm} - e^{- \lvert k_{12} \rvert^{2} (s-\sigma)} \frac{ e^{- \lvert k_{2} \rvert^{2} (s-\sigma)} h_{u}(\epsilon k_{2}) h_{b}(\epsilon k_{2})}{\lvert k_{2} \rvert^{2}} k_{12}^{i_{3}} i \hat{\mathcal{P}}^{i_{1} i_{2}}(k_{12}) \rvert \nonumber \\
& \times   \lvert e^{- \lvert k_{12}' \rvert^{2} f(\epsilon k_{12}') (\bar{s}-\bar{\sigma})} \frac{ e^{- \lvert k_{2}' \rvert^{2} f(\epsilon k_{2}') (\bar{s}-\bar{\sigma})} h_{u}(\epsilon k_{2}') h_{b}(\epsilon k_{2}')}{\lvert k_{2}' \rvert^{2} f(\epsilon k_{2}')} \nonumber\\
& \hspace{55mm} \times (k_{12}')^{i_{3}'} g(\epsilon (k_{12}')^{i_{3}'}) \hat{\mathcal{P}}^{i_{1}' i_{2}'} (k_{12}') \nonumber\\
& \hspace{5mm} - e^{- \lvert k_{12}' \rvert^{2} (\bar{s}-\bar{\sigma})} \frac{ e^{- \lvert k_{2}' \rvert^{2} (\bar{s}-\bar{\sigma})} h_{u}(\epsilon k_{2}') h_{b}(\epsilon k_{2}')}{\lvert k_{2}' \rvert^{2}} (k_{12}')^{i_{3}'} i \hat{\mathcal{P}}^{i_{1}' i_{2}'}(k_{12}') \rvert d\bar{\sigma} d\sigma  ds d \bar{s}.\nonumber
\end{align}
\end{subequations} 
First, we estimate 
\begin{align*} 
& \mathbb{E} [ ( : \hat{X}_{\sigma, b}^{\epsilon, i_{2}}(k_{1}) \hat{X}_{t,b}^{\epsilon, j_{0}}(k_{4}): - : \hat{X}_{s,b}^{\epsilon, i_{2}}(k_{1}) \hat{X}_{t,b}^{\epsilon, j_{0}}(k_{4}):)  \\
& \hspace{20mm} \times \overline{ (: \hat{X}_{\bar{\sigma}, b}^{\epsilon, i_{2}'}(k_{1}') \hat{X}_{t,b}^{\epsilon, j_{0}}(k_{4}'): - : \hat{X}_{\bar{s},b}^{\epsilon, i_{2}'}(k_{1}') \hat{X}_{t,b}^{\epsilon, j_{0}}(k_{4}'):)}] \\
\leq&  (1_{k_{1} = k_{1}', k_{4} = k_{4}'} + 1_{k_{1} = k_{4}', k_{4} = k_{1}'})  \\
\times& ( \sum_{i_{3}, i_{4} =1}^{3} \frac{ h_{b} (\epsilon k_{1})^{2}}{ 2 \lvert k_{1} \rvert^{2} f(\epsilon k_{1})} \hat{\mathcal{P}}^{i_{2} i_{3}}(k_{1})^{2} \frac{ h_{b}(\epsilon k_{4})^{2}}{2 \lvert k_{4} \rvert^{2} f(\epsilon k_{4})} \hat{\mathcal{P}}^{j_{0} i_{4}}(k_{4})^{2}  \\
& + 1_{k_{1} = k_{4}} \frac{ e^{- \lvert k_{1} \rvert^{2} f(\epsilon k_{1}) (t-\sigma)} h_{b}(\epsilon k_{1})^{2}}{2 \lvert k_{1} \rvert^{2} f(\epsilon k_{1})} \hat{\mathcal{P}}^{i_{2}i_{3}}(k_{1})\hat{\mathcal{P}}^{j_{0} i_{3}}(k_{1}) \\
& \hspace{20mm} \times  \frac{ e^{- \lvert k_{1} \rvert^{2} f(\epsilon k_{1}) (t-\sigma)} h_{b}(\epsilon k_{1})^{2}}{2\lvert k_{1} \rvert^{2} f(\epsilon k_{1})} \hat{\mathcal{P}}^{j_{0} i_{4}}(k_{1}) \hat{\mathcal{P}}^{i_{2} i_{4}}(k_{1}) \nonumber\\
& - \frac{ e^{- \lvert k_{1} \rvert^{2} f(\epsilon k_{1}) (s-\sigma)} h_{b}(\epsilon k_{1})^{2}}{2 \lvert k_{1} \rvert^{2} f(\epsilon k_{1})} \hat{\mathcal{P}}^{i_{2} i_{3}}(k_{1})^{2} \frac{h_{b}(\epsilon k_{4})^{2}}{2 \lvert k_{4} \rvert^{2} f(\epsilon k_{4})} \hat{\mathcal{P}}^{j_{0} i_{4}}(k_{4})^{2} \\
& - 1_{k_{1} = k_{4}} \frac{e^{- \lvert k_{1} \rvert^{2} f(\epsilon k_{1}) (t-\sigma)} h_{b}(\epsilon k_{1})^{2}}{2 \lvert k_{1} \rvert^{2} f(\epsilon k_{1})} \hat{\mathcal{P}}^{i_{2} i_{3}}(k_{1}) \hat{\mathcal{P}}^{j_{0} i_{3}}(k_{1}) \\
& \hspace{20mm} \times  \frac{ e^{- \lvert k_{1} \rvert^{2} f(\epsilon k_{1}) (t-s)} h_{b}(\epsilon k_{1})^{2}}{2 \lvert k_{1} \rvert^{2} f(\epsilon k_{1})} \hat{\mathcal{P}}^{j_{0} i_{4}}(k_{1}) \hat{\mathcal{P}}^{i_{2} i_{4}}(k_{1}) \\
& -  \frac{ e^{- \lvert k_{1} \rvert^{2} f(\epsilon k_{1}) (s-\sigma)} h_{b}(\epsilon k_{1})^{2}}{2 \lvert k_{1} \rvert^{2} f(\epsilon k_{1})} \hat{\mathcal{P}}^{i_{2} i_{3}}(k_{1})^{2} \frac{h_{b}(\epsilon k_{4})^{2}}{2 \lvert k_{4} \rvert^{2} f(\epsilon k_{4})} \hat{\mathcal{P}}^{j_{0} i_{4}}(k_{4})^{2} \\
& - 1_{k_{1} = k_{4}} \frac{e^{- \lvert k_{1} \rvert^{2} f(\epsilon k_{1}) (t-s)} h_{b}(\epsilon k_{1})^{2}}{2 \lvert k_{1} \rvert^{2} f(\epsilon k_{1})} \hat{\mathcal{P}}^{i_{2} i_{3}}(k_{1}) \hat{\mathcal{P}}^{j_{0} i_{3}}(k_{1}) \\
& \hspace{20mm} \times \frac{ e^{- \lvert k_{1} \rvert^{2} f(\epsilon k_{1}) (t-\sigma)} h_{b}(\epsilon k_{1})^{2}}{2 \lvert k_{1} \rvert^{2} f(\epsilon k_{1})} \hat{\mathcal{P}}^{j_{0} i_{4}}(k_{1}) \hat{\mathcal{P}}^{i_{2} i_{4}}(k_{1}) \\
&+ \frac{ h_{b} (\epsilon k_{1})^{2}}{2 \lvert k_{1} \rvert^{2} f(\epsilon k_{1})} \hat{\mathcal{P}}^{i_{2} i_{3}}(k_{1})^{2} \frac{h_{b}(\epsilon k_{4})^{2}}{2\lvert k_{4} \rvert^{2} f(\epsilon k_{4})} \hat{\mathcal{P}}^{j_{0} i_{4}}(k_{4})^{2} \\
&+ 1_{k_{1} = k_{4}} \frac{e^{- \lvert k_{1} \rvert^{2} f(\epsilon k_{1})(t-s)} h_{b}(\epsilon k_{1})^{2}}{2 \lvert k_{1} \rvert^{2} f(\epsilon k_{1})} \hat{\mathcal{P}}^{i_{2} i_{3}}(k_{1}) \hat{\mathcal{P}}^{j_{0} i_{3}}(k_{1})  \\
& \hspace{20mm} \times  \frac{e^{- \lvert k_{1} \rvert^{2} f(\epsilon k_{1}) (t-s)} h_{b}(\epsilon k_{1})^{2}}{2 \lvert k_{1} \rvert^{2} f(\epsilon k_{1})} \hat{\mathcal{P}}^{j_{0} i_{4}}(k_{1}) \hat{\mathcal{P}}^{i_{2} i_{4}}(k_{1}))^{\frac{1}{2}} \\
\times & ( \sum_{i_{3}, i_{4} =1}^{3} \frac{ h_{b} (\epsilon k_{1}')^{2}}{ 2 \lvert k_{1}' \rvert^{2} f(\epsilon k_{1}')} \hat{\mathcal{P}}^{i_{2}' i_{3}}(k_{1}')^{2} \frac{ h_{b}(\epsilon k_{4}')^{2}}{2 \lvert k_{4}' \rvert^{2} f(\epsilon k_{4}')} \hat{\mathcal{P}}^{j_{0} i_{4}}(k_{4}')^{2}  \\
& + 1_{k_{1}' = k_{4}'} \frac{ e^{- \lvert k_{1}' \rvert^{2} f(\epsilon k_{1}') (t-\bar{\sigma})} h_{b}(\epsilon k_{1}')^{2}}{2 \lvert k_{1}' \rvert^{2} f(\epsilon k_{1}')} \hat{\mathcal{P}}^{i_{2} i_{3}}(k_{1}')\hat{\mathcal{P}}^{j_{0} i_{3}}(k_{1}') \\
& \hspace{20mm} \frac{ e^{- \lvert k_{1}' \rvert^{2} f(\epsilon k_{1}') (t-\bar{\sigma})} h_{b}(\epsilon k_{1}')^{2}}{2\lvert k_{1}' \rvert^{2} f(\epsilon k_{1}')} \hat{\mathcal{P}}^{j_{0} i_{4}}(k_{1}') \hat{\mathcal{P}}^{i_{2}' i_{4}}(k_{1}') \nonumber\\
& - \frac{ e^{- \lvert k_{1}' \rvert^{2} f(\epsilon k_{1}') (\bar{s}-\bar{\sigma})} h_{b}(\epsilon k_{1}')^{2}}{2 \lvert k_{1}' \rvert^{2} f(\epsilon k_{1}')} \hat{\mathcal{P}}^{i_{2}' i_{3}}(k_{1}')^{2} \frac{h_{b}(\epsilon k_{4}')^{2}}{2 \lvert k_{4}' \rvert^{2} f(\epsilon k_{4}')} \hat{\mathcal{P}}^{j_{0} i_{4}}(k_{4}')^{2} \\
& - 1_{k_{1}' = k_{4}'} \frac{e^{- \lvert k_{1}' \rvert^{2} f(\epsilon k_{1}') (t-\bar{\sigma})} h_{b}(\epsilon k_{1}')^{2}}{2 \lvert k_{1}' \rvert^{2} f(\epsilon k_{1}')} \hat{\mathcal{P}}^{i_{2}' i_{3}}(k_{1}') \hat{\mathcal{P}}^{j_{0} i_{3}}(k_{1}')  \\
& \hspace{20mm} \times \frac{ e^{- \lvert k_{1}' \rvert^{2} f(\epsilon k_{1}') (t-\bar{s})} h_{b}(\epsilon k_{1}')^{2}}{2 \lvert k_{1}' \rvert^{2} f(\epsilon k_{1}')} \hat{\mathcal{P}}^{j_{0} i_{4}}(k_{1}') \hat{\mathcal{P}}^{i_{2}' i_{4}}(k_{1}') \\
& -  \frac{ e^{- \lvert k_{1}' \rvert^{2} f(\epsilon k_{1}') (\bar{s}-\bar{\sigma})} h_{b}(\epsilon k_{1}')^{2}}{2 \lvert k_{1}' \rvert^{2} f(\epsilon k_{1}')} \hat{\mathcal{P}}^{i_{2}' i_{3}}(k_{1}')^{2} \frac{h_{b}(\epsilon k_{4}')^{2}}{2 \lvert k_{4}' \rvert^{2} f(\epsilon k_{4}')} \hat{\mathcal{P}}^{j_{0} i_{4}}(k_{4}')^{2} \\
& - 1_{k_{1}' = k_{4}'} \frac{e^{- \lvert k_{1}' \rvert^{2} f(\epsilon k_{1}') (t-\bar{s})} h_{b}(\epsilon k_{1}')^{2}}{2 \lvert k_{1}' \rvert^{2} f(\epsilon k_{1}')} \hat{\mathcal{P}}^{i_{2}' i_{3}'}(k_{1}') \hat{\mathcal{P}}^{j_{0} i_{3}}(k_{1}')  \\
& \hspace{20mm} \times \frac{ e^{- \lvert k_{1}' \rvert^{2} f(\epsilon k_{1}') (t-\bar{\sigma})} h_{b}(\epsilon k_{1}')^{2}}{2 \lvert k_{1}' \rvert^{2} f(\epsilon k_{1}')} \hat{\mathcal{P}}^{j_{0} i_{4}}(k_{1}') \hat{\mathcal{P}}^{i_{2}' i_{4}}(k_{1}') \\
&+ \frac{ h_{b} (\epsilon k_{1}')^{2}}{2 \lvert k_{1}' \rvert^{2} f(\epsilon k_{1}')} \hat{\mathcal{P}}^{i_{2}' i_{3}}(k_{1}')^{2} \frac{h_{b}(\epsilon k_{4}')^{2}}{2\lvert k_{4}' \rvert^{2} f(\epsilon k_{4}')} \hat{\mathcal{P}}^{j_{0} i_{4}}(k_{4}')^{2} \\
&+ 1_{k_{1}' = k_{4}'} \frac{e^{- \lvert k_{1}' \rvert^{2} f(\epsilon k_{1}')(t-\bar{s})} h_{b}(\epsilon k_{1}')^{2}}{2 \lvert k_{1}' \rvert^{2} f(\epsilon k_{1}')} \hat{\mathcal{P}}^{i_{2}' i_{3}}(k_{1}') \hat{\mathcal{P}}^{j_{0} i_{3}}(k_{1}')   \\
& \hspace{20mm} \times \frac{e^{- \lvert k_{1}' \rvert^{2} f(\epsilon k_{1}') (t-\bar{s})} h_{b}(\epsilon k_{1}')^{2}}{2 \lvert k_{1}' \rvert^{2} f(\epsilon k_{1}')} \hat{\mathcal{P}}^{j_{0} i_{4}}(k_{1}') \hat{\mathcal{P}}^{i_{2'} i_{4}}(k_{1}'))^{\frac{1}{2}}
\end{align*}
by H$\ddot{\mathrm{o}}$lder's inequality, Example \ref{Example 3.1} and \eqref{covariance c}. Within each square root, there are eight terms. We match first with third, second with fourth, fifth with seventh, and sixth with eighth to bound for ay $\eta \in [0,1]$ by 
\begin{align}\label{[Equation (4.7o)][ZZ17]}
& \mathbb{E} [ ( : \hat{X}_{\sigma, b}^{\epsilon, i_{2}}(k_{1}) \hat{X}_{t,b}^{\epsilon, j_{0}}(k_{4}): - : \hat{X}_{s,b}^{\epsilon, i_{2}}(k_{1}) \hat{X}_{t,b}^{\epsilon, j_{0}}(k_{4}):) \nonumber\\
& \hspace{5mm} \times \overline{ (: \hat{X}_{\bar{\sigma}, b}^{\epsilon, i_{2}'}(k_{1}') \hat{X}_{t,b}^{\epsilon, j_{0}}(k_{4}'): - : \hat{X}_{\bar{s},b}^{\epsilon, i_{2}'}(k_{1}') \hat{X}_{t,b}^{\epsilon, j_{0}}(k_{4}'):)}] \nonumber\\
\leq& (1_{k_{1} = k_{1}, k_{4} = k_{4}'} + 1_{k_{1} = k_{4}', k_{4} = k_{1}'}) \nonumber\\ 
& \times ( \frac{1- e^{- \lvert k_{1} \rvert^{2} f(\epsilon k_{1}) (s-\sigma)}}{\lvert k_{1} \rvert^{2} \lvert k_{4} \rvert^{2}} + \frac{ \lvert e^{- \lvert k_{1} \rvert^{2} f(\epsilon k_{1}) (t-\sigma)} - e^{-\lvert k_{1} \rvert^{2} f(\epsilon k_{1})(t-s)}}{\lvert k_{1} \rvert^{2} \lvert k_{4} \rvert^{2}})^{\frac{1}{2}} \nonumber\\
& \times ( \frac{\lvert 1- e^{- \lvert k_{1}' \rvert^{2} f(\epsilon k_{1}') (\bar{s}-\bar{\sigma})} \rvert }{\lvert k_{1}' \rvert^{2} \lvert k_{4}' \rvert^{2}} + \frac{ \lvert e^{- \lvert k_{1}' \rvert^{2} f(\epsilon k_{1}') (t-\bar{\sigma})} - e^{-\lvert k_{1}' \rvert^{2} f(\epsilon k_{1}')(t-\bar{s})}\rvert }{\lvert k_{1}' \rvert^{2} \lvert k_{4}' \rvert^{2}})^{\frac{1}{2}} \nonumber\\
\lesssim& (1_{k_{1} = k_{1}, k_{4} = k_{4}'} + 1_{k_{1} = k_{4}', k_{4} = k_{1}'}) ( \frac{ (\lvert k_{1} \rvert^{2} \lvert s-\sigma \rvert)^{\eta}}{ \lvert k_{1} \rvert^{2} \lvert k_{4} \rvert^{2}})^{\frac{1}{2}} ( \frac{ (\lvert k_{1}' \rvert^{2} \lvert \bar{s}-\bar{\sigma} \rvert)^{\eta}}{ \lvert k_{1}' \rvert^{2} \lvert k_{4}' \rvert^{2}})^{\frac{1}{2}} 
\end{align} 
by mean value theorem. Moreover, we can now bound from \eqref{estimate 249} for $\epsilon \in (0, 2\eta)$ by 
\begin{align}\label{estimate 20} 
O_{q, t, i_{0} j_{0}}^{11} \lesssim&  \epsilon^{\eta} \sum_{k} \sum_{\lvert i-j\rvert \leq 1, \lvert i' - j'\rvert \leq 1} \sum_{k_{1}, k_{2}, k_{4}, k_{1}', k_{2}', k_{4}' \neq 0: k_{14} = k_{14}' = k}  \theta(2^{-q} k)^{2} \theta(2^{-i} k_{1})   \nonumber\\
& \times \theta(2^{-i'} k_{1}')  \theta(2^{-j} k_{4}) \theta(2^{-j'} k_{4} ') \int_{[0,t]^{2}}  (1_{k_{1} =k_{1}', k_{4} = k_{4}'} + 1_{k_{1} = k_{4}', k_{4} = k_{1}'})s^{\frac{\epsilon}{4}} \bar{s}^{\frac{\epsilon}{4}} \nonumber\\
& \times \frac{ \lvert k_{1} \rvert^{\frac{3\eta}{2}} \lvert k_{1}' \rvert^{\frac{3\eta}{2}}}{ \lvert k_{4} \rvert \lvert k_{4} ' \rvert} \frac{ \lvert k_{12} \rvert}{\lvert k_{2} \rvert^{2}} \frac{ \lvert k_{12}' \rvert}{\lvert k_{2}' \rvert^{2}} \frac{ e^{- \lvert k_{1} \rvert^{2} \bar{c}_{f} (t-s)}}{[ \lvert k_{12} \rvert^{2} + \lvert k_{2} \rvert^{2} ]^{1+ \frac{\eta}{2} - \frac{\epsilon}{4}}} \frac{ e^{- \lvert k_{1}' \rvert^{2} \bar{c}_{f} (t-\bar{s})}}{[ \lvert k_{12}' \rvert^{2} + \lvert k_{2}' \rvert^{2} ]^{1+ \frac{\eta}{2} - \frac{\epsilon}{4}}} ds d \bar{s} \nonumber  \\
\lesssim&  \epsilon^{\eta}t^{\epsilon} \sum_{k} \sum_{\lvert i-j\rvert \leq 1, \lvert i' - j'\rvert \leq 1} \sum_{k_{1}, k_{4}, k_{1}', k_{4}' \neq 0: k_{14} = k_{14}' = k}  \theta(2^{-q} k)^{2} \theta(2^{-i} k_{1})  \theta(2^{-i'} k_{1}')  \nonumber\\
& \times  \theta(2^{-j} k_{4}) \theta(2^{-j'} k_{4} ') (1_{k_{1} = k_{1}', k_{4} = k_{4}'} + 1_{k_{1} = k_{4}', k_{4} = k_{1}'})  \nonumber\\
& \times \frac{1}{\lvert k_{4} \rvert \lvert k_{4}' \rvert} \frac{1}{\lvert k_{1} \rvert^{2- \frac{3\eta}{2} - \frac{\epsilon}{2}}} \frac{1}{\lvert k_{1}' \rvert^{2- \frac{3\eta}{2} - \frac{\epsilon}{2}}}
\end{align}
by \eqref{[Equation (4.2)][ZZ17]}, \eqref{[Equation (4.3)][ZZ17]} and \eqref{key estimate}. Next, in order to estimate $O_{q, t, i_{0} j_{0}}^{12}$ from \eqref{estimate 249}, we first see that we may bound for any $\eta \in [0,\frac{1}{2}]$, 
\begin{align}\label{estimate 250}
&\mathbb{E} [  ( : \hat{X}_{\sigma, b}^{\epsilon, i_{2}} (k_{1}) \hat{X}_{t,b}^{\epsilon, j_{0}}(k_{4}): - : \hat{X}_{s,b}^{\epsilon, i_{2}}(k_{1}) \hat{X}_{t,b}^{\epsilon, j_{0}}(k_{4}): \\
& \hspace{15mm}- : \hat{\bar{X}}_{\sigma, b}^{\epsilon, i_{2}}(k_{1}) \hat{\bar{X}}_{t,b}^{\epsilon, j_{0}}(k_{4}): + :\hat{\bar{X}}_{s,b}^{\epsilon, i_{2}}(k_{1}) \hat{\bar{X}}_{t,b}^{\epsilon, j_{0}}(k_{4}): )  \nonumber\\
& \times \overline{( : \hat{X}_{\bar{\sigma}, b}^{\epsilon, i_{2}'} (k_{1}') \hat{X}_{t,b}^{\epsilon, j_{0}}(k_{4}'): - : \hat{X}_{\bar{s},b}^{\epsilon, i_{2}'}(k_{1}') \hat{X}_{t,b}^{\epsilon, j_{0}}(k_{4}'):  } \nonumber\\
& \hspace{15mm} \overline{- : \hat{\bar{X}}_{\bar{\sigma}, b}^{\epsilon, i_{2}'}(k_{1}') \hat{\bar{X}}_{t,b}^{\epsilon, j_{0}}(k_{4}'): + :\hat{\bar{X}}_{\bar{s},b}^{\epsilon, i_{2}'}(k_{1}') \hat{\bar{X}}_{t,b}^{\epsilon, j_{0}}(k_{4}'): ) }] \nonumber\\
\lesssim&[ \frac{ (\lvert \epsilon k_{1} \rvert^{\eta} + \lvert \epsilon k_{4} \rvert^{\eta} )}{\lvert k_{1} \rvert \lvert k_{4} \rvert} \frac{ ( \lvert \epsilon k_{1}'\rvert^{\eta} + \lvert \epsilon k_{4} ' \rvert^{\eta})}{\lvert k_{1}' \rvert \lvert k_{4} ' \rvert}  \wedge   \left( \frac{ ( \lvert k_{1} \rvert^{2} \lvert s- \sigma \rvert)^{\eta}}{\lvert k_{1} \rvert^{2} \lvert k_{4} \rvert^{2} } \right)^{\frac{1}{2}} \left(  \frac{ ( \lvert k_{1}' \rvert^{2} \lvert \bar{s} - \bar{\sigma} \rvert)^{\eta}}{\lvert k_{1} ' \rvert^{2} \lvert k_{4} ' \rvert^{2}} \right)^{\frac{1}{2}} ] \nonumber\\
& \times (1_{k_{1} = k_{1}', k_{4} = k_{4}'} + 1_{k_{1} = k_{4}', k_{4} = k_{1}'})   \nonumber\\
\lesssim& \frac{ \epsilon^{\eta} (\lvert k_{1} \rvert^{\eta} + \lvert k_{4} \rvert^{\eta}) \lvert k_{1} \rvert^{\frac{\eta}{2}} \lvert k_{1}' \rvert^{\frac{\eta}{2}} \lvert s-\sigma \rvert^{\frac{\eta}{4}} \lvert \bar{s} - \bar{\sigma} \rvert^{\frac{\eta}{4}}}{\lvert k_{1} \rvert \lvert k_{4} \rvert \lvert k_{1}' \rvert \lvert k_{4}' \rvert} (1_{k_{1} = k_{1}', k_{4} = k_{4}'} + 1_{k_{1} = k_{4}', k_{4} = k_{1}'})  \nonumber
\end{align} 
where the first estimate is due to H$\ddot{\mathrm{o}}$lder's inequality and \eqref{[Equation (4.7)][ZZ17]} while the second due to an estimate similar to \eqref{[Equation (4.7o)][ZZ17]}. We apply \eqref{estimate 250} to \eqref{estimate 249} to deduce for $\epsilon \in (0, \eta)$, 
\begin{align}\label{estimate 21}
O_{q, t, i_{0} j_{0}}^{12} \lesssim& \epsilon^{\eta} t^{\epsilon} \sum_{k} \sum_{\lvert i-j \rvert \leq 1, \lvert i'-j'\rvert \leq 1} \sum_{k_{1}, k_{4}, k_{1}', k_{4}' \neq 0: k_{14} = k_{14}' = k} \nonumber\\
& \times  \theta(2^{-q} k)^{2}  \theta(2^{-i} k_{1}) \theta(2^{-i'} k_{1}') \theta(2^{-j} k_{4}) \theta(2^{-j'} k_{4}')   \nonumber \\
& \times (1_{k_{1} = k_{1}', k_{4} = k_{4}'} + 1_{k_{1} = k_{4}', k_{4} = k_{1}'}) \frac{ ( \lvert k_{1} \rvert^{\eta} + \lvert k_{4} \rvert^{\eta})}{ \lvert k_{1} \rvert^{2- \frac{\epsilon}{2} - \frac{\eta}{2}} \lvert k_{1}' \rvert^{2- \frac{\epsilon}{2} - \frac{\eta}{2}} \lvert k_{4} \rvert \lvert k_{4} ' \rvert}
\end{align} 
by \eqref{key estimate}. Finally, applying \eqref{[Equation (4.7o)][ZZ17]} to \eqref{estimate 249} leads to for $\epsilon \in (0, \eta)$, 
\begin{align}\label{estimate 22}
O_{q, t, i_{0} j_{0}}^{13} 
\lesssim&  \sum_{k} \sum_{\lvert i-j \rvert \leq 1, \lvert i'-j'\rvert \leq 1} \sum_{k_{1}, k_{2}, k_{4}, k_{1}', k_{2}', k_{4}' \neq 0: k_{14} = k_{14}' = k}\theta(2^{-q} k)^{2}  \theta(2^{-i} k_{1}) \theta(2^{-i'} k_{1}')\nonumber\\
& \times \theta(2^{-j} k_{4}) \theta(2^{-j'} k_{4}') \int_{[0,t]^{2}} \int_{0}^{s} \int_{0}^{\bar{s}}  (1_{k_{1} = k_{1}', k_{4} = k_{4}'} + 1_{k_{1} = k_{4}', k_{4} = k_{1}'})  \nonumber\\
& \times  \frac{1}{ \lvert k_{1} \rvert^{2- \frac{\epsilon}{2} - \eta}} \frac{1}{\lvert k_{1}' \rvert^{2- \frac{\epsilon}{2} - \eta}} \frac{1}{(t-s)^{1- \frac{\epsilon}{4}} ( t- \bar{s})^{1- \frac{\epsilon}{4}}}\frac{1}{\lvert k_{4} \rvert \lvert k_{4}' \rvert} \lvert s-\sigma \rvert^{\frac{\eta}{2}} \lvert \bar{s} - \bar{\sigma} \rvert^{\frac{\eta}{2}}\nonumber\\
& \times  \frac{ \lvert k_{12} \rvert}{\lvert k_{2} \rvert^{2}} \frac{ \lvert k_{12}'\rvert}{\lvert k_{2}' \rvert^{2}}  e^{- \lvert k_{12} \rvert^{2} \bar{c}_{f}(s-\sigma)} e^{- \lvert k_{2} \rvert^{2} \bar{c}_{f}(s-\sigma)} (\lvert \epsilon k_{12} \rvert^{\frac{\eta}{2}} + \lvert \epsilon k_{2} \rvert^{\frac{\eta}{2}}) \nonumber\\
& \hspace{17mm} \times e^{- \lvert k_{12}' \rvert^{2} \bar{c}_{f} (\bar{s} - \bar{\sigma})} e^{- \lvert k_{2}' \rvert^{2} \bar{c}_{f} (\bar{s} - \bar{\sigma})} (\lvert \epsilon k_{12}' \rvert^{\frac{\eta}{2}} + \lvert \epsilon k_{2}' \rvert^{\frac{\eta}{2}}) d \bar{\sigma}d \sigma ds d \bar{s} \nonumber\\
\lesssim&  \epsilon^{\eta}  t^{\epsilon} \sum_{k} \sum_{\lvert i-j \rvert \leq 1, \lvert i'-j'\rvert \leq 1} \sum_{k_{1}, k_{4}, k_{1}', k_{4}' \neq 0: k_{14} = k_{14}' = k}\theta(2^{-q} k)^{2}  \theta(2^{-i} k_{1}) \theta(2^{-i'} k_{1}')\nonumber\\
& \times \theta(2^{-j} k_{4}) \theta(2^{-j'} k_{4}')   (1_{k_{1} = k_{1}', k_{4} = k_{4}'} + 1_{k_{1} = k_{4}', k_{4} = k_{1}'})  \nonumber\\
& \times \frac{1}{ \lvert k_{1} \rvert^{2- \frac{\epsilon}{2} - \eta}} \frac{1}{\lvert k_{1}' \rvert^{2- \frac{\epsilon}{2} - \eta}} \frac{1}{\lvert k_{4} \rvert \lvert k_{4}' \rvert}
\end{align} 
by \eqref{key estimate}, \eqref{[Equation (4.2)][ZZ17]} and \eqref{[Equation (4.3)][ZZ17]}. Applying \eqref{estimate 20}, \eqref{estimate 21} and \eqref{estimate 22} to \eqref{estimate 251} leads to  
\begin{align}\label{estimate 252}
&O_{q, t, i_{0} j_{0}}^{1} \lesssim \epsilon^{\eta} t^{\epsilon} \sum_{k} \sum_{\lvert i-j \rvert \leq 1, \lvert i' - j' \rvert \leq 1} \sum_{k_{1}, k_{4} \neq 0: k_{14} = k} \theta(2^{-q} k)^{2} \theta(2^{-i} k_{1}) \theta(2^{-j} k_{4}) \nonumber\\
& \hspace{5mm} \times [ \theta(2^{-i'} k_{1}) \theta(2^{-j'} k_{4}) [ \frac{1}{\lvert k_{1} \rvert^{4- 3\eta - \epsilon} \lvert k_{4} \rvert^{2}} + \frac{1}{\lvert k_{1} \rvert^{4-\epsilon - 2\eta} \lvert k_{4} \rvert^{2}} + \frac{1}{ \lvert k_{1} \rvert^{4- \epsilon - \eta} \lvert k_{4} \rvert^{2-\eta}}] \nonumber\\ 
& \hspace{5mm} + \theta(2^{-i'} k_{4}) \theta(2^{-j'} k_{1}) [ \frac{1}{ \lvert k_{1} \rvert^{3- \frac{3\eta}{2} - \frac{\epsilon}{2}} \lvert k_{4} \rvert^{3- \frac{3\eta}{2} - \frac{\epsilon}{2}}} + \frac{1}{\lvert k_{1} \rvert^{3- \frac{\epsilon}{2} - \frac{3\eta}{2}} \lvert k_{4} \rvert^{3- \frac{\epsilon}{2} - \frac{\eta}{2}}} \nonumber\\
& \hspace{10mm} + \frac{1}{\lvert k_{1} \rvert^{3- \frac{\epsilon}{2} - \frac{\eta}{2}} \lvert k_{4} \rvert^{3- \frac{\epsilon}{2} - \frac{3\eta}{2}}} + \frac{1}{ \lvert k_{1} \rvert^{3- \frac{\epsilon}{2} - \eta} \lvert k_{4} \rvert^{3- \frac{\epsilon}{2} -\eta} } ] ]  \lesssim \epsilon^{\eta} t^{\epsilon} 2^{q(3\eta + \epsilon)}
\end{align} 
by relying on Lemma \ref{Lemma 3.13}. Applying \eqref{estimate 252} to \eqref{estimate 253} gives us 
\begin{equation}\label{estimate 254}
 \mathbb{E} [ \lvert \Delta_{q} (L_{t, i_{0}j_{0}}^{52} - \tilde{L}_{t, i_{0}j_{0}}^{52}) \rvert^{2} ]  \lesssim \epsilon^{\eta} t^{\epsilon} 2^{q(3\eta + \epsilon)}.
\end{equation}
Next, we compute from \eqref{estimate 24} and \eqref{estimate 25} that 
\begin{align}\label{estimate 26}
\mathbb{E} [ \lvert \Delta_{q} ( \tilde{L}_{t, i_{0}j_{0}}^{52} - \bar{L}_{t, i_{0}j_{0}}^{52}) \rvert^{2}] \lesssim \sum_{l=1}^{2} Q_{q, t, j_{0}}^{l}
\end{align} 
where 
\begin{subequations}
\begin{align}
Q_{q, t, j_{0}}^{1} &\triangleq \mathbb{E} [ \lvert \sum_{k\neq 0} \sum_{\lvert i-j \rvert \leq 1} \sum_{k_{1}, k_{2}, k_{4} \neq 0: k_{14} = k} \sum_{i_{1}, i_{2}, i_{3}, j_{1} =1}^{3} \theta(2^{-q} k) \theta(2^{-i} k_{1}) \theta(2^{-j} k_{4}) e_{k} \\
& \times [ - \int_{0}^{t} : \hat{X}_{s,b}^{\epsilon, i_{2}}(k_{1}) \hat{X}_{t,b}^{\epsilon, j_{0}}(k_{4}): e^{- \lvert k_{1} \rvert^{2} f(\epsilon k_{1})(t-s)} k_{1}^{j_{1}} g(\epsilon k_{1}^{j_{1}}) \nonumber\\
& \times  \int_{0}^{s} \frac{ e^{- \lvert k_{2} \rvert^{2} f(\epsilon k_{2}) (s-\sigma) } h_{u}(\epsilon k_{2}) h_{b}(\epsilon k_{2})}{\lvert k_{2} \rvert^{2} f(\epsilon k_{2})}  [ e^{- \lvert k_{12} \rvert^{2} f(\epsilon k_{12}) (s-\sigma)} k_{12}^{i_{3}} g(\epsilon k_{12}^{i_{3}}) \hat{\mathcal{P}}^{i_{1} i_{2}}(k_{12}) \nonumber\\
& \hspace{15mm} - e^{- \lvert k_{2} \rvert^{2} f(\epsilon k_{2})(s-\sigma)} k_{2}^{i_{3}} g(\epsilon k_{2}^{i_{3}}) \hat{\mathcal{P}}^{i_{1} i_{2}}(k_{2})] d \sigma ds \nonumber \\
&+ \int_{0}^{t} : \hat{\bar{X}}_{s,b}^{\epsilon, i_{2}}(k_{1}) \hat{\bar{X}}_{t,b}^{\epsilon, j_{0}}(k_{4}): e^{- \lvert k_{1} \rvert^{2} (t-s)} k_{1}^{j_{1}} i \int_{0}^{s} \frac{ e^{- \lvert k_{2} \rvert^{2} (s-\sigma) } h_{u}(\epsilon k_{2}) h_{b}(\epsilon k_{2})}{\lvert k_{2} \rvert^{2}} \nonumber\\
& \hspace{5mm} \times [ e^{- \lvert k_{12} \rvert^{2} (s-\sigma)} k_{12}^{i_{3}} i \hat{\mathcal{P}}^{i_{1} i_{2}}(k_{12}) - e^{- \lvert k_{2} \rvert^{2}(s-\sigma)} k_{2}^{i_{3}} i \hat{\mathcal{P}}^{i_{1} i_{2}}(k_{2})] d \sigma ds] \rvert^{2} ], \nonumber\\
Q_{q, t, j_{0}}^{2} &\triangleq \mathbb{E} [ \lvert \sum_{k\neq 0} \sum_{\lvert i-j \rvert \leq 1} \sum_{k_{1}, k_{2}, k_{4} \neq 0: k_{14} = k} \sum_{i_{1}, i_{2}, i_{3}, j_{1} =1}^{3} \theta(2^{-q} k) \theta(2^{-i} k_{1}) \theta(2^{-j} k_{4}) e_{k} \\
& \times [\int_{0}^{t} : \hat{X}_{s,u}^{\epsilon, i_{2}}(k_{1}) \hat{X}_{t,b}^{\epsilon, j_{0}}(k_{4}): e^{- \lvert k_{1} \rvert^{2} f(\epsilon k_{1})(t-s)} k_{1}^{j_{1}} g(\epsilon k_{1}^{j_{1}}) \nonumber\\
& \times \int_{0}^{s} \frac{ e^{- \lvert k_{2} \rvert^{2} f(\epsilon k_{2}) (s-\sigma) } h_{b}(\epsilon k_{2})^{2}}{\lvert k_{2} \rvert^{2} f(\epsilon k_{2})} [ e^{- \lvert k_{12} \rvert^{2} f(\epsilon k_{12}) (s-\sigma)} k_{12}^{i_{3}} g(\epsilon k_{12}^{i_{3}}) \hat{\mathcal{P}}^{i_{1} i_{2}}(k_{12}) \nonumber\\
& \hspace{15mm}  - e^{- \lvert k_{2} \rvert^{2} f(\epsilon k_{2})(s-\sigma)} k_{2}^{i_{3}} g(\epsilon k_{2}^{i_{3}}) \hat{\mathcal{P}}^{i_{1} i_{2}}(k_{2})] d \sigma ds \nonumber \\
& - \int_{0}^{t} : \hat{\bar{X}}_{s,u}^{\epsilon, i_{2}}(k_{1}) \hat{\bar{X}}_{t,b}^{\epsilon, j_{0}}(k_{4}): e^{- \lvert k_{1} \rvert^{2} (t-s)} k_{1}^{j_{1}} i \int_{0}^{s} \frac{ e^{- \lvert k_{2} \rvert^{2} (s-\sigma) } h_{b}(\epsilon k_{2})^{2}}{\lvert k_{2} \rvert^{2}} \nonumber\\
& \hspace{5mm} \times [ e^{- \lvert k_{12} \rvert^{2} (s-\sigma)} k_{12}^{i_{3}} i \hat{\mathcal{P}}^{i_{1} i_{2}}(k_{12}) - e^{- \lvert k_{2} \rvert^{2}(s-\sigma)} k_{2}^{i_{3}} i \hat{\mathcal{P}}^{i_{1} i_{2}}(k_{2})] d \sigma ds ] \rvert^{2} ]. \nonumber
\end{align}
\end{subequations} 
W.l.o.g. we work on $Q_{q, t, j_{0}}^{2}$. We further bound it by 
\begin{align}\label{estimate 263}
Q_{q, t, j_{0}}^{2} \lesssim \sum_{l=1}^{3} Q_{q, t, j_{0}}^{2l} 
\end{align}
where 
\begin{subequations}
\begin{align} 
Q_{q, t, j_{0}}^{21} &\approx \sum_{k\neq 0} \sum_{\lvert i-j \rvert \leq 1, \lvert i'-j' \rvert \leq 1} \sum_{k_{1}, k_{2}, k_{4}, k_{1}', k_{2}', k_{4}' \neq 0: k_{14} = k_{14}' = k} \sum_{i_{1}, i_{2}, i_{3}, j_{1}, i_{1}', i_{2}', i_{3}', j_{1}' = 1}^{3} \label{estimate 27}\\
& \times \theta(2^{-q} k)^{2} \theta(2^{-i} k_{1}) \theta(2^{-i'} k_{1}') \theta(2^{-j} k_{4}) \theta(2^{-j'} k_{4}') \nonumber\\
& \times \int_{[0,t]^{2}} \mathbb{E} [ : \hat{X}_{s,u}^{\epsilon, i_{2}}(k_{1}) \hat{X}_{t,b}^{\epsilon, j_{0}}(k_{4}): \overline{: \hat{X}_{\bar{s}, u}^{\epsilon, i_{2}'}(k_{1}') \hat{X}_{t,b}^{\epsilon, j_{0}} (k_{4}'):}] \nonumber\\
& \hspace{15mm} \times e^{- \vert k_{1} \rvert^{2} f(\epsilon k_{1}) (t-s)} k_{1}^{j_{1}} g(\epsilon k_{1}^{j_{1}}) e^{- \lvert k_{1}' \rvert^{2} f(\epsilon k_{1}') (t- \bar{s})} (k_{1}')^{j_{1}'} g(\epsilon (k_{1}')^{j_{1}'}) \nonumber\\
& \times \int_{0}^{s} \int_{0}^{\bar{s}} \lvert  \frac{ e^{- \lvert k_{2} \rvert^{2} f(\epsilon k_{2}) (s-\sigma)} h_{b}(\epsilon k_{2})^{2}}{\lvert k_{2} \rvert^{2} f(\epsilon k_{2})} [ e^{- \lvert k_{12} \rvert^{2} f(\epsilon k_{12}) (s-\sigma)} k_{12}^{i_{3}} g(\epsilon k_{12}^{i_{3}}) \hat{\mathcal{P}}^{i_{1} i_{2}}(k_{12}) \nonumber\\
& \hspace{15mm} - e^{- \lvert k_{2} \rvert^{2} f(\epsilon k_{2}) (s-\sigma)} k_{2}^{i_{3}} g(\epsilon k_{2}^{i_{3}}) \hat{\mathcal{P}}^{i_{1} i_{2}}(k_{2})] \nonumber\\
& \hspace{5mm} - \frac{ e^{- \lvert k_{2} \rvert^{2} (s-\sigma)} h_{b}(\epsilon k_{2})^{2}}{\lvert k_{2} \rvert^{2}} \nonumber\\
& \hspace{15mm} \times [ e^{- \lvert k_{12} \rvert^{2} (s-\sigma)} k_{12}^{i_{3}} i \hat{\mathcal{P}}^{i_{1}i_{2}}(k_{12}) - e^{- \lvert k_{2} \rvert^{2} (s-\sigma)} k_{2}^{i_{3}} i \hat{\mathcal{P}}^{i_{1} i_{2}} (k_{2}) ] \rvert \nonumber\\
& \times \lvert  \frac{ e^{- \lvert k_{2}' \rvert^{2} f(\epsilon k_{2}') (\bar{s}-\bar{\sigma})} h_{b}(\epsilon k_{2}')^{2}}{\lvert k_{2}' \rvert^{2} f(\epsilon k_{2}')} [ e^{- \lvert k_{12}' \rvert^{2} f(\epsilon k_{12}') (\bar{s}-\bar{\sigma})} (k_{12}')^{i_{3}'} g(\epsilon (k_{12}')^{i_{3}'}) \hat{\mathcal{P}}^{i_{1}' i_{2}'}(k_{12}') \nonumber\\
& \hspace{15mm} - e^{- \lvert k_{2}' \rvert^{2} f(\epsilon k_{2}') (\bar{s}-\bar{\sigma})} (k_{2}')^{i_{3}'} g(\epsilon (k_{2}')^{i_{3}'}) \hat{\mathcal{P}}^{i_{1}' i_{2}'}(k_{2}')] \nonumber\\
& \hspace{5mm} - \frac{ e^{- \lvert k_{2}' \rvert^{2} (\bar{s}-\bar{\sigma})} h_{b}(\epsilon k_{2}')^{2}}{\lvert k_{2}' \rvert^{2}}[ e^{- \lvert k_{12}' \rvert^{2} (\bar{s}-\bar{\sigma})} (k_{12}')^{i_{3}'} i \hat{\mathcal{P}}^{i_{1}'i_{2}'}(k_{12}') \nonumber\\
& \hspace{15mm}  - e^{- \lvert k_{2}' \rvert^{2} (\bar{s}-\bar{\sigma})} (k_{2}')^{i_{3}'} i \hat{\mathcal{P}}^{i_{1}' i_{2}'} (k_{2}') ] \rvert  d\bar{\sigma}d \sigma ds d \bar{s}, \nonumber\\
Q_{q, t, j_{0}}^{22} &\approx \sum_{k\neq 0} \sum_{\lvert i-j \rvert \leq 1, \lvert i'-j' \rvert \leq 1} \sum_{k_{1}, k_{2}, k_{4}, k_{1}', k_{2}', k_{4}' \neq 0: k_{14} = k_{14}' = k} \sum_{i_{1}, i_{2}, i_{3}, j_{1}, i_{1}', i_{2}', i_{3}', j_{1}' = 1}^{3} \label{estimate 28}\\
& \times \theta(2^{-q} k)^{2} \theta(2^{-i} k_{1}) \theta(2^{-i'} k_{1}') \theta(2^{-j} k_{4}) \theta(2^{-j'} k_{4}') \nonumber\\
& \times \int_{[0,t]^{2}} \mathbb{E} [ : \hat{X}_{s,u}^{\epsilon, i_{2}}(k_{1}) \hat{X}_{t,b}^{\epsilon, j_{0}}(k_{4}): \overline{: \hat{X}_{\bar{s}, u}^{\epsilon, i_{2}'}(k_{1}') \hat{X}_{t,b}^{\epsilon, j_{0}} (k_{4}'):}] \nonumber\\
& \hspace{2mm} \times \lvert e^{- \lvert k_{1} \rvert^{2} f(\epsilon k_{1}) (t-s)} k_{1}^{j_{1}} g(\epsilon k_{1}^{j_{1}}) - e^{- \lvert k_{1} \rvert^{2} (t-s) } k_{1}^{j_{1}} i \rvert \nonumber\\
& \hspace{2mm} \times \lvert e^{- \lvert k_{1}' \rvert^{2} f(\epsilon k_{1}') (t-\bar{s})} (k_{1}')^{j_{1}'} g(\epsilon (k_{1}')^{j_{1}'} ) - e^{- \lvert k_{1}' \rvert^{2} (t- \bar{s})} (k_{1}')^{j_{1}'} i \rvert \nonumber\\
& \times \int_{0}^{s} \int_{0}^{\bar{s}} \frac{ e^{- \lvert k_{2} \rvert^{2} (s-\sigma)} h_{b}(\epsilon k_{2})^{2}}{\lvert k_{2} \rvert^{2}} \frac{ e^{- \lvert k_{2}' \rvert^{2} (\bar{s} - \bar{\sigma})}h_{b}(\epsilon k_{2}')^{2}}{\lvert k_{2}' \rvert^{2}} \nonumber\\
& \hspace{2mm} \times \lvert e^{- \lvert k_{12} \rvert^{2} (s-\sigma)} k_{12}^{i_{3}} i \hat{\mathcal{P}}^{i_{1} i_{2}} (k_{12}) - e^{- \lvert k_{2} \rvert^{2} (s-\sigma)} k_{2}^{i_{3}} i \hat{\mathcal{P}}^{i_{1}i_{2}} (k_{2}) \rvert \nonumber\\
& \hspace{2mm} \times \lvert e^{- \lvert k_{12}'\rvert^{2} (\bar{s} - \bar{\sigma})} (k_{12}')^{i_{3}'} i \hat{\mathcal{P}}^{i_{1}' i_{2}'} (k_{12}') - e^{- \lvert k_{2} ' \rvert^{2}(\bar{s} - \bar{\sigma})} (k_{2}')^{i_{3}'}  i \hat{\mathcal{P}}^{i_{1}' i_{2}'} (k_{2}') \rvert d\bar{\sigma}  d \sigma ds d \bar{s}, \nonumber\\
Q_{q, t, j_{0}}^{23} &\approx \sum_{k\neq 0} \sum_{\lvert i-j \rvert \leq 1, \lvert i'-j' \rvert \leq 1} \sum_{k_{1}, k_{2}, k_{4}, k_{1}', k_{2}', k_{4}' \neq 0: k_{14} = k_{14}' = k} \sum_{i_{1}, i_{2}, i_{3}, i_{1}', i_{2}', i_{3}' = 1}^{3} \label{estimate 29}\\
& \times \theta(2^{-q} k)^{2} \theta(2^{-i} k_{1}) \theta(2^{-i'} k_{1}') \theta(2^{-j} k_{4}) \theta(2^{-j'} k_{4}') \nonumber\\
& \times \int_{[0,t]^{2}} \mathbb{E} [ ( : \hat{X}_{s,u}^{\epsilon, i_{2}}(k_{1})\hat{X}_{t,b}^{\epsilon, j_{0}}(k_{4}): - : \hat{\bar{X}}_{s,u}^{\epsilon, i_{2}}(k_{1}) \hat{\bar{X}}_{t,b}^{\epsilon, j_{0}}(k_{4}): ) \nonumber\\
& \hspace{15mm} \times \overline{ (: \hat{X}_{\bar{s}, u}^{\epsilon, i_{2}'} (k_{1}') \hat{X}_{t,b}^{\epsilon, j_{0}}(k_{4}'): - : \hat{\bar{X}}_{\bar{s}, u}^{\epsilon, i_{2}'} (k_{1}') \hat{\bar{X}}_{t,b}^{\epsilon, j_{0}}(k_{4}') : )}]  \nonumber\\
& \hspace{15mm} \times e^{- \lvert k_{1} \rvert^{2} (t-s)} \lvert k_{1} \rvert e^{- \lvert k_{1}' \rvert^{2} (t-\bar{s})} \lvert k_{1}' \rvert \nonumber\\
& \times \int_{0}^{s} \int_{0}^{\bar{s}} \frac{e^{- \lvert k_{2} \rvert^{2} (s-\sigma)} h_{b}(\epsilon k_{2})^{2}}{\lvert k_{2} \rvert^{2}} \lvert e^{- \lvert k_{12} \rvert^{2} (s-\sigma)} k_{12}^{i_{3}}  \hat{\mathcal{P}}^{i_{1}i_{2}} (k_{12}) \nonumber\\
& \hspace{25mm} - e^{- \lvert k_{2} \rvert^{2} (s-\sigma)} k_{2}^{i_{3}} \hat{\mathcal{P}}^{i_{1}i_{2}}(k_{2}) \rvert \nonumber\\
& \hspace{15mm} \times \frac{ e^{- \lvert k_{2} ' \rvert^{2} (\bar{s} - \bar{\sigma})} h_{b}(\epsilon k_{2}')^{2}}{\lvert k_{2} ' \rvert^{2}} \lvert e^{- \lvert k_{12}'\rvert^{2} (\bar{s} - \bar{\sigma})} (k_{12}')^{i_{3}'}  \hat{\mathcal{P}}^{i_{1}' i_{2}'} (k_{12}')  \nonumber\\
& \hspace{25mm} - e^{- \lvert k_{2}' \rvert^{2} ( \bar{s} - \bar{\sigma})} (k_{2}')^{i_{3}'}  \hat{\mathcal{P}}^{i_{1}' i_{2}'} (k_{2}') \rvert d \bar{\sigma} d \sigma ds d \bar{s}. \nonumber
\end{align} 
\end{subequations} 
We compute
\begin{align}\label{estimate 257}
\mathbb{E} [ : \hat{X}_{s,u}^{\epsilon, i_{2}}(k_{1}) \hat{X}_{t,b}^{\epsilon, j_{0}}(k_{4}): \overline{: \hat{X}_{\bar{s}, u}^{\epsilon, i_{2}'}(k_{1}') \hat{X}_{t,b}^{\epsilon, j_{0}} (k_{4}'):}] \lesssim \frac{(1_{k_{1} = k_{1}', k_{4} = k_{4}'} + 1_{k_{1}  = k_{4}', k_{4} = k_{1}'})}{\lvert k_{1} \rvert \lvert k_{1}' \rvert \lvert k_{4} \rvert \lvert k_{4} ' \rvert}  
\end{align}
by Example \ref{Example 3.1} and \eqref{covariance} while for any $\eta \in [0,1]$, 
\begin{align}\label{estimate 259}
&\mathbb{E} [ ( : \hat{X}_{s,u}^{\epsilon, i_{2}}(k_{1})\hat{X}_{t,b}^{\epsilon, j_{0}}(k_{4}): - : \hat{\bar{X}}_{s,u}^{\epsilon, i_{2}}(k_{1}) \hat{\bar{X}}_{t,b}^{\epsilon, j_{0}}(k_{4}): ) \nonumber\\
& \hspace{15mm} \times \overline{ (: \hat{X}_{\bar{s}, u}^{\epsilon, i_{2}'} (k_{1}') \hat{X}_{t,b}^{\epsilon, j_{0}}(k_{4}'): - : \hat{\bar{X}}_{\bar{s}, u}^{\epsilon, i_{2}'} (k_{1}') \hat{\bar{X}}_{t,b}^{\epsilon, j_{0}}(k_{4}') : )}]  \nonumber \\
\lesssim& \left( \frac{ \lvert \epsilon k_{1} \rvert^{\frac{\eta}{2}} + \lvert \epsilon k_{4} \rvert^{\frac{\eta}{2}}}{\lvert k_{1} \rvert \lvert k_{4} \rvert} \right)   \left( \frac{ \lvert \epsilon k_{1}' \rvert^{\frac{\eta}{2}} + \lvert \epsilon k_{4}' \rvert^{\frac{\eta}{2}}}{\lvert k_{1}' \rvert \lvert k_{4}' \rvert} \right)  
\end{align}
by H$\ddot{\mathrm{o}}$lder's inequality and \eqref{[Equation (4.7)][ZZ17]}. Moreover, to estimate $Q_{q, t, j_{0}}^{21}$, we can estimate for $\eta \in (0, \frac{1}{2})$, 
\begin{align}\label{estimate 258}
& \lvert  \frac{ e^{- \lvert k_{2} \rvert^{2} f(\epsilon k_{2}) (s-\sigma)} h_{b}(\epsilon k_{2})^{2}}{\lvert k_{2} \rvert^{2} f(\epsilon k_{2})} [ e^{- \lvert k_{12} \rvert^{2} f(\epsilon k_{12}) (s-\sigma)} k_{12}^{i_{3}} g(\epsilon k_{12}^{i_{3}}) \hat{\mathcal{P}}^{i_{1} i_{2}}(k_{12}) \\
& \hspace{15mm} - e^{- \lvert k_{2} \rvert^{2} f(\epsilon k_{2}) (s-\sigma)} k_{2}^{i_{3}} g(\epsilon k_{2}^{i_{3}}) \hat{\mathcal{P}}^{i_{1} i_{2}}(k_{2})] \nonumber\\
& \hspace{5mm} - \frac{ e^{- \lvert k_{2} \rvert^{2} (s-\sigma)} h_{b}(\epsilon k_{2})^{2}}{\lvert k_{2} \rvert^{2}}[ e^{- \lvert k_{12} \rvert^{2} (s-\sigma)} k_{12}^{i_{3}} i \hat{\mathcal{P}}^{i_{1}i_{2}}(k_{12}) \nonumber\\
& \hspace{15mm} - e^{- \lvert k_{2} \rvert^{2} (s-\sigma)} k_{2}^{i_{3}} i \hat{\mathcal{P}}^{i_{1} i_{2}} (k_{2}) ] \rvert \nonumber \\ 
\lesssim& \frac{ e^{- \lvert k_{2} \rvert^{2} \bar{c}_{f} (s-\sigma)}}{\lvert k_{2} \rvert^{2}} ( \lvert k_{1} \rvert^{2\eta} \lvert s-\sigma \rvert^{-\frac{1}{2} + \eta} \wedge ( \lvert \epsilon k_{2} \rvert^{\eta} + \lvert \epsilon k_{12} \rvert^{\eta} e^{-\frac{1}{2} \lvert k_{12} \rvert^{2} \bar{c}_{f} (s-\sigma)} ) \lvert s-\sigma \rvert^{-\frac{1}{2}}) \nonumber
\end{align} 
where the first estimate is by Lemma \ref{Lemma 3.9} and the second estimate is by relying on \eqref{key estimate}, \eqref{[Equation (4.2)][ZZ17]} and \eqref{[Equation (4.3)][ZZ17]}; we note that the extra exponential decay term $e^{- \frac{1}{2} \lvert k_{12} \rvert^{2} \bar{c}_{f} (s-\sigma)}$ will be actually crucial in this estimate of $Q_{q, t, j_{0}}^{21}$. Thus, we continue our estimates on $Q_{q, t, j_{0}}^{21}$ for any $\epsilon \in (0, \eta)$ from \eqref{estimate 27} by 
\begin{align}\label{estimate 260}
Q_{q, t, j_{0}}^{21} 
\lesssim&  \epsilon^{\eta}    \sum_{k\neq 0} \sum_{\lvert i-j \rvert \leq 1, \lvert i' -j'\rvert \leq 1} \sum_{k_{1}, k_{4}, k_{1}', k_{4}' \neq 0: k_{14} = k_{14}' = k} \theta(2^{-q} k)^{2} \theta(2^{-i} k_{1}) \theta(2^{-i'} k_{1}')\nonumber\\
& \times \theta(2^{-j} k_{4}) \theta(2^{-j'} k_{4}') (1_{k_{1} = k_{1}', k_{4} = k_{4}'} + 1_{k_{1} = k_{4}', k_{4} = k_{1}'})  \frac{ \lvert k_{1} \rvert^{\eta} \lvert k_{1}' \rvert^{\eta}}{ \lvert k_{4} \rvert \lvert k_{4} ' \rvert} \nonumber\\
& \times \int_{[0,t]^{2}} \frac{1}{ [ \lvert k_{1} \rvert^{2} (t-s) ]^{1- \frac{\epsilon}{4}}} \frac{1}{ [ \lvert k_{1}' \rvert^{2} (t- \bar{s})]^{1- \frac{\epsilon}{4}}}  s^{\frac{\epsilon}{4}} \bar{s}^{\frac{\epsilon}{4}} ds d \bar{s}  \lesssim \epsilon^{\eta} t^{\epsilon} 2^{q(2\eta + \epsilon)}
\end{align}
where we used \eqref{estimate 257}, \eqref{estimate 258}, \eqref{key estimate}, that $2^{q}\lesssim 2^{i}$ so that $q \lesssim i$ and Lemma \ref{Lemma 3.13}. Next, continuing our estimate on $Q_{q, t, j_{0}}^{22}$ from \eqref{estimate 28}, for $\eta \in (0,1), \epsilon \in (0, 2 \eta)$, 
\begin{align}\label{estimate 261}
Q_{q, t, j_{0}}^{22} \lesssim& \epsilon^{\eta} t^{\epsilon} \sum_{k\neq 0} \sum_{\lvert i-j \rvert \leq 1, \lvert i'-j'\rvert \leq 1} \sum_{k_{1}, k_{4}, k_{1}', k_{4}' \neq 0: k_{14} = k_{14}' = k} \theta(2^{-q} k)^{2} \nonumber \\
& \times \theta(2^{-i} k_{1}) \theta(2^{-i'} k_{1}') \theta(2^{-j} k_{4}) \theta(2^{-j'} k_{4}')(1_{k_{1} = k_{1}', k_{4} = k_{4}'} + 1_{k_{1} = k_{4}', k_{4} = k_{1}'}) \nonumber\\
& \times \frac{ \lvert k_{1} \rvert^{\frac{3\eta}{2}} \lvert k_{1} '\rvert^{\frac{3\eta}{2}}}{\lvert k_{4} \rvert \lvert k_{4}'\rvert \lvert k_{1} \rvert^{2- \frac{\epsilon}{2}} \lvert k_{1}'\rvert^{2-\frac{\epsilon}{2}}}\lesssim \epsilon^{\eta} t^{\epsilon} 2^{q( 3 \eta + \epsilon)} 
\end{align} 
by \eqref{estimate 257}, \eqref{[Equation (4.2)][ZZ17]}, \eqref{[Equation (4.3)][ZZ17]}, \eqref{key estimate}, Lemmas \ref{Lemma 3.8} and \ref{Lemma 3.13}, that $2^{q} \lesssim 2^{i}$ so that $q \lesssim i$. Finally, continuing the estimate on $Q_{q, t, j_{0}}^{23}$ from \eqref{estimate 29}, we compute for $\eta \in (0,1)$, 
\begin{align}\label{estimate 262}
Q_{q, t, j_{0}}^{23} \lesssim& \epsilon^{\eta} t^{\epsilon} \sum_{k\neq 0}  \sum_{\lvert i-j \rvert \leq 1, \lvert i'-j' \rvert \leq 1} \sum_{k_{1}, k_{4}, k_{1}', k_{4}' \neq 0: k_{14} = k_{14}' = k}  \theta(2^{-q} k)^{2} \theta(2^{-i} k_{1}) \theta(2^{-i'} k_{1}')  \nonumber\\
& \times \theta(2^{-j} k_{4}) \theta(2^{-j'} k_{4}') (1_{k_{1} =k_{1}', k_{4} = k_{4}'} + 1_{k_{1} = k_{4}', k_{4} = k_{1}'})  \nonumber\\
& \hspace{25mm} \times \frac{ ( \lvert k_{1} \rvert^{\frac{\eta}{2}} + \lvert k_{4} \rvert^{\frac{\eta}{2}}) (\lvert k_{1}' \rvert^{\frac{\eta}{2}} + \lvert k_{4}' \rvert^{\frac{\eta}{2}})}{ \lvert k_{4} \rvert \lvert k_{4}' \rvert \lvert k_{1} \rvert^{2- \frac{\epsilon}{2} - \eta} \lvert k_{1}' \rvert^{2- \frac{\epsilon}{2} - \eta}} 
\lesssim \epsilon^{\eta} t^{\epsilon} 2^{q(3\eta + \epsilon)} 
\end{align}
by \eqref{estimate 259}, Lemmas \ref{Lemma 3.8} and \ref{Lemma 3.13}, \eqref{key estimate}, that  $2^{q} \lesssim 2^{i}$ so that $q \lesssim i$. Applying \eqref{estimate 260}-\eqref{estimate 262} to \eqref{estimate 263} and then \eqref{estimate 26}, and considering the resulting inequality together with \eqref{estimate 254}, along with  \eqref{estimate 248} and \eqref{estimate 264}, we obtain 
\begin{equation}\label{estimate 265}
 \mathbb{E} [ \lvert \Delta_{q} (L_{t, i_{0}j_{0}}^{5} - \bar{L}_{t, i_{0}j_{0}}^{5}) \rvert^{2}]   \lesssim \epsilon^{\eta} t^{\epsilon} 2^{q (3 \eta + \epsilon)}.
\end{equation} 
The estimate on $L_{t, i_{0} j_{0}}^{6} - \bar{L}_{t, i_{0} j_{0}}^{6}$ can be obtained similarly. \\

\emph{Terms in the fourth chaos: $L_{t, i_{0}j_{0}}^{1}$ in \eqref{[Equation (4.6f)][ZZ17]}}

We strategically split the eight terms within $L_{t, i_{0}j_{0}}^{1}$ on \eqref{[Equation (4.6f)][ZZ17]} to four so that $L_{t, i_{0}j_{0}}^{1} = \sum_{i=1}^{4} L_{t, i_{0}j_{0}}^{1i}$ where 
\begin{subequations} 
\begin{align}
L_{t, i_{0}j_{0}}^{11} \triangleq& (2\pi)^{-\frac{9}{2}} \sum_{\lvert i-j \rvert \leq 1} \sum_{i_{1}, i_{2}, i_{3}, j_{1} =1}^{3} \sum_{k} \sum_{k_{1}, k_{2}, k_{3}, k_{4} \neq 0: k_{1234} = k} \theta(2^{-i} k_{123}) \theta(2^{-j} k_{4}) \nonumber \\
& \times \int_{0}^{t} [ e^{- \lvert k_{123} \rvert^{2} f(\epsilon k_{123}) (t-s)} \int_{0}^{s} : \hat{X}_{\sigma, u}^{\epsilon, i_{2}}(k_{1}) \hat{X}_{\sigma, u}^{\epsilon, i_{3}}(k_{2}) \hat{X}_{s,u}^{\epsilon, j_{1}}(k_{3}) \hat{X}_{t,b}^{\epsilon, j_{0}}(k_{4}): \nonumber \\
& \hspace{10mm} \times e^{- \lvert k_{12} \rvert^{2}  f(\epsilon k_{12})(s-\sigma)} d\sigma k_{123}^{j_{1}} g(\epsilon k_{123}^{j_{1}}) k_{12}^{i_{3}} g(\epsilon k_{12}^{i_{3}}) \nonumber \\
& - e^{- \lvert k_{123} \rvert^{2} (t-s) } \int_{0}^{s}  : \hat{\bar{X}}_{\sigma, u}^{\epsilon, i_{2}}(k_{1}) \hat{\bar{X}}_{\sigma, u}^{\epsilon, i_{3}}(k_{2}) \hat{\bar{X}}_{s,u}^{\epsilon, j_{1}}(k_{3}) \hat{\bar{X}}_{t,b}^{\epsilon, j_{0}}(k_{4}): \nonumber \\
& \hspace{10mm} \times e^{- \lvert k_{12} \rvert^{2} (s-\sigma)} d\sigma k_{123}^{j_{1}} i k_{12}^{i_{3}} i ] ds \hat{\mathcal{P}}^{i_{0} i_{1}}(k_{123}) \hat{\mathcal{P}}^{i_{1} i_{2}}(k_{12}) e_{k},  \\
L_{t, i_{0}j_{0}}^{12} \triangleq& (2\pi)^{-\frac{9}{2}} \sum_{\lvert i-j \rvert \leq 1} \sum_{i_{1}, i_{2}, i_{3}, j_{1} =1}^{3} \sum_{k} \sum_{k_{1}, k_{2}, k_{3}, k_{4} \neq 0: k_{1234} = k} \theta(2^{-i} k_{123}) \theta(2^{-j} k_{4}) \nonumber\\
& \times \int_{0}^{t} [ - e^{- \lvert k_{123} \rvert^{2}  f(\epsilon k_{123})(t-s)} \int_{0}^{s} : \hat{X}_{\sigma, b}^{\epsilon, i_{2}}(k_{1}) \hat{X}_{\sigma, b}^{\epsilon, i_{3}}(k_{2}) \hat{X}_{s,u}^{\epsilon, j_{1}}(k_{3}) \hat{X}_{t,b}^{\epsilon, j_{0}}(k_{4}): \nonumber \\
& \hspace{10mm} \times e^{- \lvert k_{12} \rvert^{2}  f(\epsilon k_{12})(s-\sigma)}  d\sigma k_{123}^{j_{1}} g(\epsilon k_{123}^{j_{1}}) k_{12}^{i_{3}} g(\epsilon k_{12}^{i_{3}}) \nonumber \\
& + e^{- \lvert k_{123} \rvert^{2} (t-s) } \int_{0}^{s}  : \hat{\bar{X}}_{\sigma, b}^{\epsilon, i_{2}}(k_{1}) \hat{\bar{X}}_{\sigma, b}^{\epsilon, i_{3}}(k_{2}) \hat{\bar{X}}_{s,u}^{\epsilon, j_{1}}(k_{3}) \hat{\bar{X}}_{t,b}^{\epsilon, j_{0}}(k_{4}): \nonumber \\
& \hspace{10mm} \times e^{- \lvert k_{12} \rvert^{2} (s-\sigma)} d\sigma k_{123}^{j_{1}} i k_{12}^{i_{3}} i ] ds \hat{\mathcal{P}}^{i_{0} i_{1}}(k_{123}) \hat{\mathcal{P}}^{i_{1} i_{2}}(k_{12}) e_{k},  \\
L_{t, i_{0}j_{0}}^{13} \triangleq& (2\pi)^{-\frac{9}{2}} \sum_{\lvert i-j \rvert \leq 1} \sum_{i_{1}, i_{2}, i_{3}, j_{1} =1}^{3} \sum_{k} \sum_{k_{1}, k_{2}, k_{3}, k_{4} \neq 0: k_{1234} = k} \theta(2^{-i} k_{123}) \theta(2^{-j} k_{4}) \nonumber\\
& \times \int_{0}^{t} [ -  e^{- \lvert k_{123} \rvert^{2}  f(\epsilon k_{123})(t-s)} \int_{0}^{s} : \hat{X}_{\sigma, b}^{\epsilon, i_{2}}(k_{1}) \hat{X}_{\sigma, u}^{\epsilon, i_{3}}(k_{2}) \hat{X}_{s,b}^{\epsilon, j_{1}}(k_{3}) \hat{X}_{t,b}^{\epsilon, j_{0}}(k_{4}): \nonumber \\
& \hspace{10mm} \times e^{- \lvert k_{12} \rvert^{2}  f(\epsilon k_{12})(s-\sigma)} d\sigma k_{123}^{j_{1}} g(\epsilon k_{123}^{j_{1}}) k_{12}^{i_{3}} g(\epsilon k_{12}^{i_{3}}) \nonumber \\
& + e^{- \lvert k_{123} \rvert^{2} (t-s) } \int_{0}^{s}  : \hat{\bar{X}}_{\sigma, b}^{\epsilon, i_{2}}(k_{1}) \hat{\bar{X}}_{\sigma, u}^{\epsilon, i_{3}}(k_{2}) \hat{\bar{X}}_{s,b}^{\epsilon, j_{1}}(k_{3}) \hat{\bar{X}}_{t,b}^{\epsilon, j_{0}}(k_{4}): \nonumber \\
& \hspace{10mm} \times e^{- \lvert k_{12} \rvert^{2} (s-\sigma)} d\sigma k_{123}^{j_{1}} i k_{12}^{i_{3}} i ] ds \hat{\mathcal{P}}^{i_{0} i_{1}}(k_{123}) \hat{\mathcal{P}}^{i_{1} i_{2}}(k_{12}) e_{k},   \\
L_{t, i_{0}j_{0}}^{14} \triangleq& (2\pi)^{-\frac{9}{2}} \sum_{\lvert i-j \rvert \leq 1} \sum_{i_{1}, i_{2}, i_{3}, j_{1} =1}^{3} \sum_{k} \sum_{k_{1}, k_{2}, k_{3}, k_{4} \neq 0: k_{1234} = k} \theta(2^{-i} k_{123}) \theta(2^{-j} k_{4}) \nonumber\\
& \times \int_{0}^{t} [ e^{- \lvert k_{123} \rvert^{2} f(\epsilon k_{123})(t-s)} \int_{0}^{s} : \hat{X}_{\sigma, u}^{\epsilon, i_{2}}(k_{1}) \hat{X}_{\sigma, b}^{\epsilon, i_{3}}(k_{2}) \hat{X}_{s,b}^{\epsilon, j_{1}}(k_{3}) \hat{X}_{t,b}^{\epsilon, j_{0}}(k_{4}): \nonumber \\
& \hspace{10mm} \times e^{- \lvert k_{12} \rvert^{2}  f(\epsilon k_{12})(s-\sigma)} d\sigma k_{123}^{j_{1}} g(\epsilon k_{123}^{j_{1}}) k_{12}^{i_{3}} g(\epsilon k_{12}^{i_{3}}) \nonumber \\
& - e^{- \lvert k_{123} \rvert^{2} (t-s) } \int_{0}^{s}  : \hat{\bar{X}}_{\sigma, u}^{\epsilon, i_{2}}(k_{1}) \hat{\bar{X}}_{\sigma, b}^{\epsilon, i_{3}}(k_{2}) \hat{\bar{X}}_{s,b}^{\epsilon, j_{1}}(k_{3}) \hat{\bar{X}}_{t,b}^{\epsilon, j_{0}}(k_{4}): \nonumber \\
& \hspace{10mm} \times e^{- \lvert k_{12} \rvert^{2} (s-\sigma)} d\sigma k_{123}^{j_{1}} i k_{12}^{i_{3}} i ] ds \hat{\mathcal{P}}^{i_{0} i_{1}}(k_{123}) \hat{\mathcal{P}}^{i_{1} i_{2}}(k_{12}) e_{k}. 
\end{align} 
\end{subequations}
W.l.o.g. we show the necessary estimate on $L_{t, i_{0}j_{0}}^{14}$ as those of others are similar. We compute 
\begin{align*}
& \mathbb{E} [ \lvert \Delta_{q} L_{t, i_{0}j_{0}}^{14} \rvert^{2}] \nonumber\\
\lesssim& \mathbb{E} [ \lvert \sum_{k} \sum_{\lvert i-j \rvert \leq 1} \sum_{i_{1}, i_{2}, i_{3}, j_{1} =1}^{3} \sum_{k_{1}, k_{2}, k_{3}, k_{4} \neq 0: k_{1234} = k} \theta(2^{-q} k) \theta(2^{-i} k_{123}) \theta(2^{-j} k_{4})\nonumber\\
& \times \int_{0}^{t} [ e^{- \lvert k_{123} \rvert^{2}  f(\epsilon k_{123})(t-s)}  k_{123}^{j_{1}} g( \epsilon k_{123}^{j_{1}}) - e^{- \lvert k_{123} \rvert^{2} (t-s) } k_{123}^{j_{1}} i ] \nonumber\\
& \times \int_{0}^{s} : \hat{X}_{\sigma, u}^{\epsilon, i_{2}}(k_{1}) \hat{X}_{\sigma, b}^{\epsilon, i_{3}}(k_{2}) \hat{X}_{s,b}^{\epsilon, j_{1}}(k_{3}) \hat{X}_{t,b}^{\epsilon, j_{0}}(k_{4}): \nonumber \\
& \hspace{10mm} \times e^{- \lvert k_{12} \rvert^{2}  f(\epsilon k_{12})(s-\sigma)}   k_{12}^{i_{3}} g(\epsilon k_{12}^{i_{3}})d \sigma ds  \hat{\mathcal{P}}^{i_{0} i_{1}}(k_{123}) \hat{\mathcal{P}}^{i_{1} i_{2}}(k_{12}) e_{k} \rvert^{2} ] \nonumber\\
&+  \mathbb{E} [ \lvert \sum_{k} \sum_{\lvert i-j \rvert \leq 1} \sum_{i_{1}, i_{2}, i_{3}, j_{1} =1}^{3} \sum_{k_{1}, k_{2}, k_{3}, k_{4} \neq 0: k_{1234} = k} \theta(2^{-q} k) \theta(2^{-i} k_{123}) \theta(2^{-j} k_{4})\nonumber\\
& \times \int_{0}^{t}  e^{- \lvert k_{123} \rvert^{2} (t-s)}  k_{123}^{j_{1}} i   \int_{0}^{s} [: \hat{X}_{\sigma, u}^{\epsilon, i_{2}}(k_{1}) \hat{X}_{\sigma, b}^{\epsilon, i_{3}}(k_{2}) \hat{X}_{s,b}^{\epsilon, j_{1}}(k_{3}) \hat{X}_{t,b}^{\epsilon, j_{0}}(k_{4}): \nonumber\\
& \hspace{20mm}- : \hat{\bar{X}}_{\sigma, u}^{\epsilon, i_{2}}(k_{1}) \hat{\bar{X}}_{\sigma, b}^{\epsilon, i_{3}}(k_{2}) \hat{\bar{X}}_{s,b}^{\epsilon, j_{1}}(k_{3}) \hat{\bar{X}}_{t,b}^{\epsilon, j_{0}}(k_{4}): ] \nonumber\\ 
& \hspace{10mm} \times e^{- \lvert k_{12} \rvert^{2}  f(\epsilon k_{12})(s-\sigma)}   k_{12}^{i_{3}} g(\epsilon k_{12}^{i_{3}})d \sigma ds  \hat{\mathcal{P}}^{i_{0} i_{1}}(k_{123}) \hat{\mathcal{P}}^{i_{1} i_{2}}(k_{12}) e_{k} \rvert^{2} ] \nonumber\\
&+  \mathbb{E} [ \lvert \sum_{k\neq 0} \sum_{\lvert i-j \rvert \leq 1} \sum_{i_{1}, i_{2}, i_{3}, j_{1} =1}^{3} \sum_{k_{1}, k_{2}, k_{3}, k_{4} \neq 0: k_{1234} = k} \theta(2^{-q} k) \theta(2^{-i} k_{123}) \theta(2^{-j} k_{4})\nonumber\\
& \times \int_{0}^{t}  e^{- \lvert k_{123} \rvert^{2} (t-s)}  k_{123}^{j_{1}} i   \int_{0}^{s}  : \hat{\bar{X}}_{\sigma, u}^{\epsilon, i_{2}}(k_{1}) \hat{\bar{X}}_{\sigma, b}^{\epsilon, i_{3}}(k_{2}) \hat{\bar{X}}_{s,b}^{\epsilon, j_{1}}(k_{3}) \hat{\bar{X}}_{t,b}^{\epsilon, j_{0}}(k_{4}):  \nonumber\\ 
& \hspace{10mm} \times [e^{- \lvert k_{12} \rvert^{2}  f(\epsilon k_{12})(s-\sigma)}k_{12}^{i_{3}} g(\epsilon k_{12}^{i_{3}}) - e^{- \lvert k_{12} \rvert^{2} (s-\sigma)} k_{12}^{i_{3}} i ] \nonumber\\
& \hspace{10mm} \times d \sigma ds  \hat{\mathcal{P}}^{i_{0} i_{1}}(k_{123}) \hat{\mathcal{P}}^{i_{1} i_{2}}(k_{12}) e_{k} \rvert^{2} ] \nonumber
\end{align*} 
in which 
\begin{align}\label{estimate 266}
& \mathbb{E} [ : \hat{X}_{\sigma, u}^{\epsilon, i_{2}}(k_{1}) \hat{X}_{\sigma, b}^{\epsilon, i_{3}}(k_{2}) \hat{X}_{s,b}^{\epsilon, j_{1}}(k_{3}) \hat{X}_{t,b}^{\epsilon, j_{0}}(k_{4}):  \nonumber\\
& \hspace{10mm} \times \overline{ : \hat{X}_{\bar{\sigma}, u}^{\epsilon, i_{2}'} (k_{1}') \hat{X}_{\bar{\sigma}, b}^{\epsilon, i_{3}'} (k_{2}') \hat{X}_{\bar{s}, b}^{\epsilon, j_{1}'} (k_{3}') \hat{X}_{t,b}^{\epsilon, j_{0}}(k_{4}'):}] \lesssim \frac{1}{ \prod_{i=1}^{4} \lvert k_{i} \rvert^{2}} \sum_{i=1}^{24} R_{i} 
\end{align} 
where $R_{i}, i = 1, \hdots, 24$, are defined as
\begin{subequations}
\begin{align}
&1_{k_{1} = k_{1}', k_{2} = k_{2}', k_{3} = k_{3}', k_{4} = k_{4}'}, 1_{k_{1} = k_{1}', k_{2} = k_{2}', k_{3} = k_{4}', k_{4} = k_{3}'}, 1_{k_{1} = k_{1}', k_{2} = k_{3}', k_{3} = k_{2}', k_{4} = k_{4}'}, \\
&1_{k_{1} = k_{1}', k_{2} = k_{3}', k_{3} = k_{4}', k_{4} = k_{2}'}, 1_{k_{1} = k_{1}', k_{2} = k_{4}', k_{3} = k_{2}', k_{4} = k_{3}'}, 1_{k_{1} = k_{1}', k_{2} = k_{4}', k_{3} = k_{3}', k_{4} = k_{2}'}, \\
& 1_{k_{1} = k_{2}', k_{2} = k_{1}', k_{3} = k_{3}', k_{4} = k_{4}'}, 1_{k_{1} = k_{2}', k_{2} = k_{1}', k_{3} = k_{4}', k_{4} = k_{3}'}, 1_{k_{1} = k_{2}', k_{2} = k_{3}', k_{3} = k_{1}', k_{4} = k_{4}'} \\
& 1_{k_{1} = k_{2}', k_{2} = k_{3}', k_{3} = k_{4}', k_{4} = k_{1}'}, 1_{k_{1} = k_{2}', k_{2} = k_{4}', k_{3} = k_{1}', k_{4} = k_{3}'}, 1_{k_{1} = k_{2}', k_{2} = k_{4}', k_{3} = k_{3}', k_{4} = k_{1}'} \\
& 1_{k_{1} = k_{3}', k_{2} = k_{1}', k_{3} = k_{2}', k_{4} = k_{4}'}, 1_{k_{1} = k_{3}', k_{2} = k_{1}', k_{3} = k_{4}', k_{4} = k_{2}'}, 1_{k_{1} = k_{3}', k_{2} = k_{2}', k_{3} = k_{1}', k_{4} = k_{4}'} \\
& 1_{k_{1} = k_{3}', k_{2} = k_{2}', k_{3} = k_{4}', k_{4} = k_{1}'}, 1_{k_{1} = k_{3}', k_{2} = k_{4}', k_{3} = k_{1}', k_{4} = k_{2}'}, 1_{k_{1} = k_{3}', k_{2} = k_{4}', k_{3} = k_{2}', k_{4} = k_{1}'}, \\
& 1_{k_{1} = k_{4}', k_{2} = k_{1}', k_{3} = k_{2}', k_{4} = k_{3}'}, 1_{k_{1} = k_{4}', k_{2} = k_{1}', k_{3} = k_{3}', k_{4} = k_{2}'}, 1_{k_{1} = k_{4}', k_{2} = k_{2}', k_{3} = k_{1}', k_{4} = k_{3}'}, \\
&1_{k_{1} = k_{4}', k_{2} = k_{2}', k_{3} = k_{3}', k_{4} = k_{1}'}, 1_{k_{1} = k_{4}', k_{2} = k_{3}', k_{3} = k_{1}', k_{4} = k_{2}'}, 1_{k_{1} = k_{4}', k_{2} = k_{3}', k_{3} = k_{2}', k_{4} = k_{1}'},     
\end{align} 
\end{subequations}
respectively. Similarly, we can deduce 
\begin{align}\label{estimate 267}
& \mathbb{E} [ : \hat{\bar{X}}_{\sigma, u}^{\epsilon, i_{2}}(k_{1}) \hat{\bar{X}}_{\sigma, b}^{\epsilon, i_{3}}(k_{2}) \hat{\bar{X}}_{s,b}^{\epsilon, j_{1}}(k_{3}) \hat{\bar{X}}_{t,b}^{\epsilon, j_{0}}(k_{4}): \overline{ : \hat{\bar{X}}_{\bar{\sigma}, u}^{\epsilon, i_{2}'} (k_{1}') \hat{\bar{X}}_{\bar{\sigma}, b}^{\epsilon, i_{3}'} (k_{2}') \hat{\bar{X}}_{\bar{s}, b}^{\epsilon, j_{1}'} (k_{3}') \hat{\bar{X}}_{t,b}^{\epsilon, j_{0}}(k_{4}'):}] \nonumber\\
\lesssim& \frac{1}{ \prod_{i=1}^{4} \lvert k_{i} \rvert^{2}} \sum_{i=1}^{24} R_{i}. 
\end{align} 
We can compute 
\begin{align}\label{estimate 77}
& \mathbb{E} [\lvert \Delta_{q} L_{t, i_{0}j_{0}}^{14} \rvert^{2} ] \\
\lesssim& \sum_{k\neq 0} \sum_{\lvert i-j\rvert\leq 1, \lvert i'-j'\rvert \leq 1} \sum_{k_{1}, k_{2}, k_{3}, k_{4}, k_{1}', k_{2}', k_{3}', k_{4}' \neq 0: k_{1234} = k_{1234}' = k} \sum_{i_{2}, i_{3}, j_{1}, i_{2}', i_{3}', j_{1}' = 1}^{3} \nonumber\\
& \times \theta(2^{-q} k)^{2} \theta(2^{-i} k_{123}) \theta(2^{-i'} k_{123}') \theta(2^{-j} k_{4}) \theta(2^{-j'} k_{4}') \sum_{i=1}^{24} R_{i} \nonumber\\
& \times [\int_{[0,t]^{2}} \lvert e^{- \lvert k_{123} \rvert^{2} f(\epsilon k_{123} ) (t-s)} g(\epsilon k_{123}^{j_{1}}) - e^{- \lvert k_{123} \rvert^{2} (t-s)} i \rvert   \lvert k_{123} \rvert \nonumber\\
& \hspace{5mm} \times \lvert e^{- \lvert k_{123}' \rvert^{2} f(\epsilon k_{123}' ) (t-\bar{s})} g(\epsilon (k_{123}')^{j_{1}'}) - e^{- \lvert k_{123}' \rvert^{2} (t-\bar{s})} i \rvert \lvert k_{123} ' \rvert \nonumber\\
& \hspace{5mm} \times \int_{0}^{s} \int_{0}^{\bar{s}} \frac{1}{\prod_{i=1}^{4} \lvert k_{i} \rvert^{2}} e^{- \lvert k_{12} \rvert^{2} f(\epsilon k_{12}) (s-\sigma) - \lvert k_{12} '\rvert^{2} f(\epsilon k_{12}') (\bar{s} - \bar{\sigma})} \lvert k_{12} \rvert \lvert k_{12} ' \rvert d\bar{\sigma} d \sigma ds d \bar{s} \nonumber\\
& \hspace{5mm} + \int_{[0,t]^{2}} e^{- \lvert k_{123} \rvert^{2} (t-s) - \lvert k_{123}' \rvert^{2} (t- \bar{s})} \lvert k_{123} \rvert \lvert k_{123} ' \rvert \nonumber\\
& \hspace{5mm} \times \int_{0}^{s} \int_{0}^{\bar{s}} e^{- \lvert k_{12} \rvert^{2} f(\epsilon k_{12}) (s-\sigma) - \lvert k_{12}'\rvert^{2} f(\epsilon k_{12}') (\bar{s}-\bar{\sigma}) } \lvert k_{12} \rvert \lvert k_{12} ' \rvert \nonumber\\
& \hspace{5mm} \times ( \mathbb{E} [ \lvert : \hat{X}_{\sigma, u}^{\epsilon, i_{2}}(k_{1}) \hat{X}_{\sigma, b}^{\epsilon, i_{3}}(k_{2}) \hat{X}_{s,b}^{\epsilon, j_{1}}(k_{3}) \hat{X}_{t,b}^{\epsilon, j_{0}}(k_{4}): \nonumber\\
& \hspace{10mm} - : \hat{\bar{X}}_{\sigma, u}^{\epsilon, i_{2}}(k_{1}) \hat{\bar{X}}_{\sigma, b}^{\epsilon, i_{3}}(k_{2}) \hat{\bar{X}}_{s,b}^{\epsilon, j_{1}}(k_{3}) \hat{\bar{X}}_{t,b}^{\epsilon, j_{0}}(k_{4}): \rvert^{2} ])^{\frac{1}{2}}  \nonumber\\
& \hspace{5mm} \times ( \mathbb{E} [ \lvert : \hat{X}_{\bar{\sigma}, u}^{\epsilon, i_{2}'}(k_{1}') \hat{X}_{\bar{\sigma}, b}^{\epsilon, i_{3}'}(k_{2}') \hat{X}_{\bar{s},b}^{\epsilon, j_{1}'}(k_{3}') \hat{X}_{t,b}^{\epsilon, j_{0}}(k_{4}'): \nonumber\\
& \hspace{10mm} - : \hat{\bar{X}}_{\bar{\sigma}, u}^{\epsilon, i_{2}'}(k_{1}') \hat{\bar{X}}_{\bar{\sigma}, b}^{\epsilon, i_{3}'}(k_{2}') \hat{\bar{X}}_{\bar{s},b}^{\epsilon, j_{1}'}(k_{3}') \hat{\bar{X}}_{t,b}^{\epsilon, j_{0}}(k_{4}'):  \rvert^{2} ])^{\frac{1}{2}}  d \bar{\sigma} d\sigma ds d \bar{s} \nonumber\\
& \hspace{5mm} + \int_{[0,t]^{2}} e^{- \lvert k_{123} \rvert^{2} (t-s)} \lvert k_{123} \rvert e^{- \lvert k_{123} ' \rvert^{2} (t-\bar{s})} \lvert k_{123} ' \rvert \int_{0}^{s} \int_{0}^{\bar{s}} \frac{1}{\prod_{i=1}^{4} \lvert k_{i} \rvert^{2}} \lvert k_{12} \rvert \lvert k_{12} ' \rvert \nonumber\\
& \hspace{5mm} \times \lvert e^{- \lvert k_{12} \rvert^{2} f(\epsilon k_{12}) (s-\sigma)}g(\epsilon k_{12}^{i_{3}}) - e^{- \lvert k_{12} \rvert^{2} (s-\sigma) } i \rvert \nonumber\\
& \hspace{5mm} \times \lvert e^{- \lvert k_{12}' \rvert^{2} f(\epsilon k_{12}') (\bar{s}-\bar{\sigma})} g(\epsilon (k_{12}')^{i_{3}'}) - e^{- \lvert k_{12} '\rvert^{2} (\bar{s} - \bar{\sigma})} i \rvert d \bar{\sigma} d \sigma ds d \bar{s} ]\nonumber
\end{align} 
by \eqref{estimate 266}, \eqref{estimate 267} and H$\ddot{\mathrm{o}}$lder's inequality. At this point, and only at this point, we can take advantage of the symmetry, especially of $k_{1}$ and $k_{2}$ with $k_{1}'$ and $k_{2}'$, to reduce the 24 terms corresponding to each characteristic functions down to $U_{q, t, i_{0} j_{0}}^{1}, \hdots, U_{q, t, i_{0} j_{0}}^{6}$ corresponding to 
\begin{subequations}
\begin{align}
& 1_{k_{1} = k_{1}', k_{2} = k_{2}', k_{3} = k_{3}, k_{4} = k_{4}'}, 1_{k_{1} = k_{1}', k_{2} = k_{2}', k_{3} = k_{4}', k_{4} = k_{3}'}, 1_{k_{1} = k_{1}', k_{2} = k_{3}', k_{3} = k_{2}', k_{4} = k_{4}'}, \label{estimate 31}\\
& 1_{k_{1} = k_{3}', k_{2} = k_{2}', k_{3} = k_{4}', k_{4} = k_{1}'}, 1_{k_{1} = k_{4}', k_{2} = k_{2}', k_{3} = k_{3}', k_{4} = k_{1}'}, 1_{k_{1} = k_{3}', k_{2} = k_{4}', k_{3} = k_{1}', k_{4} = k_{2}'}, \label{estimate 32}
\end{align}
\end{subequations} 
respectively. We note that our estimate is slightly simpler than \cite[Equation (4.7s) on pg. 63]{ZZ17} which has seven cases instead of six. E.g., for $R_{4}$ we swap $k_{1}$ with $k_{2}$ and $k_{1}'$ with $k_{2}'$ to deduce $1_{k_{2} = k_{2}', k_{1} = k_{3}', k_{3} = k_{4}', k_{4} = k_{1}'}$ which is the characteristic function of $U_{q, t, i_{0} j_{0}}^{4}$; other cases may be readily reduced to one of the six cases in \eqref{estimate 31}-\eqref{estimate 32}, and for brevity we refer readers to \cite[Equation (250)]{Y19a} which is very similar. Now let us estimate e.g., 
\begin{align*}
& \mathbb{E} [ \lvert : \hat{X}_{\sigma, u}^{\epsilon, i_{2}}(k_{1}) \hat{X}_{\sigma, b}^{\epsilon, i_{3}}(k_{2}) \hat{X}_{s,b}^{\epsilon, j_{1}}(k_{3}) \hat{X}_{t,b}^{\epsilon, j_{0}}(k_{4}): \nonumber\\
& \hspace{10mm} - : \hat{\bar{X}}_{\sigma, u}^{\epsilon, i_{2}}(k_{1}) \hat{\bar{X}}_{\sigma, b}^{\epsilon, i_{3}}(k_{2}) \hat{\bar{X}}_{s,b}^{\epsilon, j_{1}}(k_{3}) \hat{\bar{X}}_{t,b}^{\epsilon, j_{0}}(k_{4}): \rvert^{2} ]= \sum_{l=1}^{4} V_{k_{1}k_{2}k_{3}k_{4}, \sigma s t, i_{2}i_{3}j_{0}j_{1}}^{l} 
\end{align*} 
where 
\begin{align*}
V_{k_{1}k_{2}k_{3}k_{4}, \sigma s t, i_{2}i_{3}j_{0}j_{1}}^{1}  \triangleq&  \mathbb{E} [ : \hat{X}_{\sigma, u}^{\epsilon, i_{2}}(k_{1}) \hat{X}_{\sigma, b}^{\epsilon, i_{3}}(k_{2}) \hat{X}_{s,b}^{\epsilon, j_{1}}(k_{3}) \hat{X}_{t,b}^{\epsilon, j_{0}} (k_{4}): \\
& \times \overline{ : \hat{X}_{\sigma, u}^{\epsilon, i_{2}}(k_{1}) \hat{X}_{\sigma, b}^{\epsilon, i_{3}}(k_{2}) \hat{X}_{s,b}^{\epsilon, j_{1}}(k_{3}) \hat{X}_{t,b}^{\epsilon, j_{0}} (k_{4}):} ], \\
V_{k_{1}k_{2}k_{3}k_{4}, \sigma s t, i_{2}i_{3}j_{0}j_{1}}^{2}  \triangleq& - \mathbb{E} [ : \hat{X}_{\sigma, u}^{\epsilon, i_{2}}(k_{1}) \hat{X}_{\sigma, b}^{\epsilon, i_{3}}(k_{2}) \hat{X}_{s,b}^{\epsilon, j_{1}}(k_{3}) \hat{X}_{t,b}^{\epsilon, j_{0}} (k_{4}): \\
& \times  \overline{ : \hat{\bar{X}}_{\sigma, u}^{\epsilon, i_{2}}(k_{1}) \hat{\bar{X}}_{\sigma, b}^{\epsilon, i_{3}}(k_{2}) \hat{\bar{X}}_{s,b}^{\epsilon, j_{1}}(k_{3}) \hat{\bar{X}}_{t,b}^{\epsilon, j_{0}} (k_{4}):} ], \\
V_{k_{1}k_{2}k_{3}k_{4}, \sigma s t, i_{2}i_{3}j_{0}j_{1}}^{3}  \triangleq& - \mathbb{E} [ : \hat{\bar{X}}_{\sigma, u}^{\epsilon, i_{2}}(k_{1}) \hat{\bar{X}}_{\sigma, b}^{\epsilon, i_{3}}(k_{2}) \hat{\bar{X}}_{s,b}^{\epsilon, j_{1}}(k_{3}) \hat{\bar{X}}_{t,b}^{\epsilon, j_{0}} (k_{4}):\\
& \times  \overline{ : \hat{X}_{\sigma, u}^{\epsilon, i_{2}}(k_{1}) \hat{X}_{\sigma, b}^{\epsilon, i_{3}}(k_{2}) \hat{X}_{s,b}^{\epsilon, j_{1}}(k_{3}) \hat{X}_{t,b}^{\epsilon, j_{0}} (k_{4}):} ],\\
V_{k_{1}k_{2}k_{3}k_{4}, \sigma s t, i_{2}i_{3}j_{0}j_{1}}^{4}  \triangleq& \mathbb{E} [ : \hat{\bar{X}}_{\sigma, u}^{\epsilon, i_{2}}(k_{1}) \hat{\bar{X}}_{\sigma, b}^{\epsilon, i_{3}}(k_{2}) \hat{\bar{X}}_{s,b}^{\epsilon, j_{1}}(k_{3}) \hat{\bar{X}}_{t,b}^{\epsilon, j_{0}} (k_{4}): \\
& \times  \overline{ : \hat{\bar{X}}_{\sigma, u}^{\epsilon, i_{2}}(k_{1}) \hat{\bar{X}}_{\sigma, b}^{\epsilon, i_{3}}(k_{2}) \hat{\bar{X}}_{s,b}^{\epsilon, j_{1}}(k_{3}) \hat{\bar{X}}_{t,b}^{\epsilon, j_{0}} (k_{4}):} ],
\end{align*}  
where by Example \ref{Example 3.1}, $V_{k_{1}k_{2}k_{3}k_{4}, \sigma s t, i_{2}i_{3}j_{0}j_{1}}^{l} $ for $l \in \{1,2,3,4\}$ have 24 terms, each of which is a quadruple of terms such as 
\begin{align*} 
&\mathbb{E} [ \hat{X}_{\sigma, u}^{\epsilon, i_{2}}(k_{1}) \overline{\hat{X}_{\sigma, u}^{\epsilon, i_{2}}(k_{1})}] \text{ in }  V_{k_{1}k_{2}k_{3}k_{4}, \sigma s t, i_{2}i_{3}j_{0}j_{1}}^{1}, \\
& \mathbb{E} [ \hat{X}_{\sigma, u}^{\epsilon, i_{2}}(k_{1}) \overline{ \hat{\bar{X}}_{\sigma, u}^{\epsilon, i_{2}}(k_{1})}] \text{ in } V_{k_{1}k_{2}k_{3}k_{4}, \sigma s t, i_{2}i_{3}j_{0}j_{1}}^{2},\\
& \mathbb{E} [ \hat{\bar{X}}_{\sigma, u}^{\epsilon, i_{2}}(k_{1}) \overline{\hat{X}_{\sigma, u}^{\epsilon, i_{2}}(k_{1})}] \text{ in } V_{k_{1}k_{2}k_{3}k_{4}, \sigma s t, i_{2}i_{3}j_{0}j_{1}}^{3} \\
\text{and } & \mathbb{E}[ \hat{\bar{X}}_{\sigma, u}^{\epsilon, i_{2}}(k_{1}) \overline{ \hat{\bar{X}}_{\sigma, u}^{\epsilon, i_{2}}(k_{1})}] \text{ in } V_{k_{1}k_{2}k_{3}k_{4}, \sigma s t, i_{2}i_{3}j_{0}j_{1}}^{4}. 
\end{align*}
Thus, a quadruple of $\frac{1}{2 \lvert k_{i} \rvert^{2} f(\epsilon k_{i})}$ in $V_{k_{1}k_{2}k_{3}k_{4}, \sigma s t, i_{2}i_{3}j_{0}j_{1}}^{1}$ gives $\frac{1}{16 \prod_{i=1}^{4} \lvert k_{i} \rvert^{2} f(\epsilon k_{i})}$, a quadruple of $\frac{1}{2\lvert k_{i} \rvert^{2}}$ in $V_{k_{1}k_{2}k_{3}k_{4}, \sigma s t, i_{2}i_{3}j_{0}j_{1}}^{4}$ gives $\frac{1}{16 \prod_{i=1}^{4}\lvert k_{i} \rvert^{2}}$, and the quadruples of $\frac{1}{\lvert k_{i} \rvert^{2}(f(\epsilon k_{i}) + 1)}$ in $V_{k_{1}k_{2}k_{3}k_{4}, \sigma s t, i_{2}i_{3}j_{0}j_{1}}^{2}$ and $V_{k_{1}k_{2}k_{3}k_{4}, \sigma s t, i_{2}i_{3}j_{0}j_{1}}^{3}$ give $\frac{2}{\prod_{i=1}^{4} \lvert k_{i} \rvert^{2} (f(\epsilon k_{i}) + 1)}$ together. Hence, 
\begin{align}\label{estimate 37}
& \mathbb{E} [ \lvert : \hat{X}_{\sigma, u}^{\epsilon, i_{2}}(k_{1}) \hat{X}_{\sigma, b}^{\epsilon, i_{3}}(k_{2}) \hat{X}_{s,b}^{\epsilon, j_{1}}(k_{3}) \hat{X}_{t,b}^{\epsilon, j_{0}}(k_{4}): \\
& \hspace{10mm} - : \hat{\bar{X}}_{\sigma, u}^{\epsilon, i_{2}}(k_{1}) \hat{\bar{X}}_{\sigma, b}^{\epsilon, i_{3}}(k_{2}) \hat{\bar{X}}_{s,b}^{\epsilon, j_{1}}(k_{3}) \hat{\bar{X}}_{t,b}^{\epsilon, j_{0}}(k_{4}): \rvert^{2} ] \nonumber\\ 
\lesssim& \lvert \frac{1}{16 ( \prod_{i=1}^{4} \lvert k_{i} \rvert^{2} f(\epsilon k_{i}))} - \frac{2}{ (\prod_{i=1}^{4} \lvert k_{i} \rvert^{2} (f(\epsilon k_{i}) + 1))} + \frac{1}{16 \prod_{i=1}^{4} \lvert k_{i} \rvert^{2}} \rvert \lesssim \frac{ \sum_{i=1}^{4} \lvert \epsilon k_{i} \rvert^{\eta}}{\prod_{i=1}^{4} \lvert k_{i} \rvert^{2}} \nonumber
\end{align}
by mean value theorem. For $U_{q, t, i_{0} j_{0}}^{1}$ corresponding to $1_{k_{1} = k_{1}', k_{2} = k_{2}', k_{3} = k_{3}, k_{4} = k_{4}'}$ from \eqref{estimate 31} in \eqref{estimate 77}, we split to 
\begin{equation}\label{estimate 271}
U_{q, t, i_{0} j_{0}}^{1} \lesssim \sum_{j=1}^{3} U_{q, t, i_{0} j_{0}}^{1j}
\end{equation}
where 
\begin{subequations} 
\begin{align} 
U_{q, t, i_{0} j_{0}}^{11} \triangleq& \sum_{k\neq 0} \sum_{\lvert i-j\rvert \leq 1, \lvert i'-j'\rvert \leq 1} \sum_{k_{1}, k_{2}, k_{3}, k_{4} \neq 0: k_{1234}  = k} \sum_{j_{1}, j_{1}' = 1}^{3} \label{estimate 33} \\
& \times \theta(2^{-q} k)^{2} \theta(2^{-i} k_{123}) \theta(2^{-i'} k_{123}) \theta(2^{-j} k_{4}) \theta(2^{-j'} k_{4})  \nonumber\\
& \times \int_{[0,t]^{2}} \lvert e^{- \lvert k_{123} \rvert^{2} f(\epsilon k_{123} ) (t-s)} g(\epsilon k_{123}^{j_{1}}) - e^{- \lvert k_{123} \rvert^{2} (t-s)} i \rvert  \nonumber\\
& \hspace{5mm} \times \lvert e^{- \lvert k_{123} \rvert^{2} f(\epsilon k_{123} ) (t-\bar{s})} g(\epsilon k_{123}^{j_{1}'}) - e^{- \lvert k_{123} \rvert^{2} (t-\bar{s})} i \rvert \lvert k_{123} \rvert^{2} \nonumber\\
& \hspace{5mm} \times \int_{0}^{s} \int_{0}^{\bar{s}} \frac{1}{\prod_{i=1}^{4} \lvert k_{i} \rvert^{2}} e^{- \lvert k_{12} \rvert^{2} f(\epsilon k_{12}) (s-\sigma + \bar{s} - \bar{\sigma})}   \lvert k_{12} \rvert^{2} d\bar{\sigma} d \sigma ds d \bar{s}, \nonumber\\
U_{q, t, i_{0} j_{0}}^{12} \triangleq& \sum_{k\neq 0} \sum_{\lvert i-j\rvert \leq 1, \lvert i'-j'\rvert \leq 1} \sum_{k_{1}, k_{2}, k_{3}, k_{4} \neq 0: k_{1234}  = k} \label{estimate 34} \\
& \times \theta(2^{-q} k)^{2} \theta(2^{-i} k_{123}) \theta(2^{-i'} k_{123}) \theta(2^{-j} k_{4}) \theta(2^{-j'} k_{4}) \nonumber\\
& \times \int_{[0,t]^{2}} e^{- \lvert k_{123} \rvert^{2} (2t-s-\bar{s})}   \lvert k_{123} \rvert^{2} \nonumber\\
& \hspace{5mm} \times \int_{0}^{s} \int_{0}^{\bar{s}} e^{- \lvert k_{12} \rvert^{2} f(\epsilon k_{12}) (s-\sigma + \bar{s} - \bar{\sigma})} \lvert k_{12} \rvert^{2} ( \frac{ \sum_{i=1}^{4} \lvert \epsilon k_{i} \rvert^{\eta}}{\prod_{i=1}^{4} \lvert k_{i} \rvert^{2}}) d \bar{\sigma} d \sigma ds d \bar{s},  \nonumber \\
U_{q, t, i_{0} j_{0}}^{13} \triangleq& \sum_{k\neq 0} \sum_{\lvert i-j\rvert \leq 1, \lvert i'-j'\rvert \leq 1} \sum_{k_{1}, k_{2}, k_{3}, k_{4} \neq 0: k_{1234}  = k} \sum_{i_{3}, i_{3}' = 1}^{3} \label{estimate 35} \\
& \times \theta(2^{-q} k)^{2} \theta(2^{-i} k_{123}) \theta(2^{-i'} k_{123}) \theta(2^{-j} k_{4}) \theta(2^{-j'} k_{4}) \nonumber\\
& \times \int_{[0,t]^{2}} e^{- \lvert k_{123} \rvert^{2} (2t-s - \bar{s})} \lvert k_{123} \rvert^{2} \int_{0}^{s} \int_{0}^{\bar{s}} \frac{\lvert k_{12} \rvert^{2}}{\prod_{i=1}^{4} \lvert k_{i} \rvert^{2}}  \nonumber\\
& \hspace{5mm} \times \lvert e^{- \lvert k_{12} \rvert^{2} f(\epsilon k_{12}) (s-\sigma)}g(\epsilon k_{12}^{i_{3}}) - e^{- \lvert k_{12} \rvert^{2} (s-\sigma) } i \rvert \nonumber\\
& \hspace{5mm} \times \lvert e^{- \lvert k_{12} \rvert^{2} f(\epsilon k_{12}) (\bar{s}-\bar{\sigma})} g(\epsilon k_{12}^{i_{3}'}) - e^{- \lvert k_{12} \rvert^{2} (\bar{s} - \bar{\sigma})} i \rvert d \bar{\sigma} d \sigma ds d \bar{s} \nonumber
\end{align} 
\end{subequations}  
where we used \eqref{estimate 37} in \eqref{estimate 34}. We estimate from \eqref{estimate 33} for any $\eta \in [0,1]$, 
\begin{align}\label{estimate 268}
U_{q, t, i_{0} j_{0}}^{11}\lesssim&  \epsilon^{\eta} \sum_{k\neq 0} \sum_{\lvert i-j\rvert \leq 1, \lvert i'-j'\rvert \leq 1} \sum_{k_{1}, k_{2}, k_{3}, k_{4}\neq 0: k_{1234}  = k}\theta(2^{-q} k)^{2} \theta(2^{-i} k_{123}) \theta(2^{-i'} k_{123})  \nonumber\\
& \hspace{2mm} \times \theta(2^{-j} k_{4}) \theta(2^{-j'} k_{4})  \frac{ ( \lvert k_{123} \rvert^{2} t)^{\eta}}{\lvert k_{123} \rvert^{2-\eta}} \frac{1}{ (\prod_{i=1}^{4} \lvert k_{i} \rvert^{2} ) \lvert k_{12} \rvert^{2}}  \lesssim \epsilon^{\eta} t^{\eta} 2^{q(3\eta + \epsilon)}
\end{align}
where we used \eqref{[Equation (4.2)][ZZ17]}, \eqref{[Equation (4.3)][ZZ17]}, Lemma \ref{Lemma 3.13} and that $2^{q} \lesssim 2^{i}$ so that $q \lesssim i$. Next, continuing our estimate on $U_{q, t, i_{0} j_{0}}^{12}$ from \eqref{estimate 34}, we compute 
\begin{align}\label{estimate 269}
U_{q, t, i_{0} j_{0}}^{12} \lesssim& \epsilon^{\eta} \sum_{k\neq 0} \sum_{\lvert i-j \rvert \leq 1, \lvert i' - j' \rvert \leq 1}\sum_{k_{1}, k_{2}, k_{3}, k_{4}\neq 0: k_{1234}  = k}\theta(2^{-q} k)^{2} \theta(2^{-i} k_{123}) \theta(2^{-i'} k_{123})  \nonumber \\
& \times  \theta(2^{-j} k_{4}) \theta(2^{-j'} k_{4})  \int_{[0,t]^{2}}  \frac{1}{ [ \lvert k_{123} \rvert^{2} (t-s) ]^{1- \frac{\eta}{2}}}\nonumber\\
& \hspace{8mm}  \times \frac{1}{ [ \lvert k_{123} \rvert^{2} (t- \bar{s})]^{1- \frac{\eta}{2}}} \frac{ \lvert k_{123} \rvert^{2}}{\lvert k_{12} \rvert^{2}} \frac{ \sum_{i=1}^{4} \lvert k_{i} \rvert^{\eta}}{\prod_{i=1}^{4} \lvert k_{i} \rvert^{2}} ds d \bar{s}  \lesssim \epsilon^{\eta} t^{\eta} 2^{q(3\eta+ \epsilon)} 
\end{align}
where we used \eqref{key estimate}, Lemma \ref{Lemma 3.13} and that $2^{q} \lesssim 2^{i}$ so that $q \lesssim i$. Finally, considering $U_{q, t, i_{0} j_{0}}^{13}$ from \eqref{estimate 35}, for any $\eta \in [0,1]$, 
\begin{align}\label{estimate 270}
&U_{q, t, i_{0} j_{0}}^{13} \lesssim \sum_{k\neq 0} \sum_{\lvert i-j\rvert \leq 1, \lvert i'-j'\rvert \leq 1} \sum_{k_{1}, k_{2}, k_{3}, k_{4}\neq 0: k_{1234}  = k}  \theta(2^{-q} k)^{2} \theta(2^{-i} k_{123}) \theta(2^{-i'} k_{123}) \nonumber\\
& \hspace{5mm} \times \theta(2^{-j} k_{4}) \theta(2^{-j'} k_{4})  \int_{[0,t]^{2}} e^{- \lvert k_{123} \rvert^{2} (2t-s - \bar{s})} \lvert k_{123} \rvert^{2} \int_{0}^{s} \int_{0}^{\bar{s}} \frac{\lvert k_{12} \rvert^{2}}{\prod_{i=1}^{4} \lvert k_{i} \rvert^{2}}  \nonumber\\
& \hspace{9mm}  \times e^{- \lvert k_{12} \rvert^{2} \bar{c}_{f} (s-\sigma)} \lvert \epsilon k_{12} \rvert^{\frac{\eta}{2}} e^{-\lvert k_{12} \rvert^{2} \bar{c}_{f} (\bar{s} - \bar{\sigma})} \lvert \epsilon k_{12} \rvert^{\frac{\eta}{2}} d \bar{\sigma} d \sigma ds d \bar{s} \lesssim \epsilon^{\eta} t^{\eta} 2^{q(3\eta + \epsilon)}
\end{align}
by \eqref{[Equation (4.2)][ZZ17]}, \eqref{[Equation (4.3)][ZZ17]}, \eqref{key estimate}, Lemma \ref{Lemma 3.13} and that $2^{q} \lesssim 2^{i}$ so that $q \lesssim i$.  Applying \eqref{estimate 268}-\eqref{estimate 270} to  \eqref{estimate 271} gives us a bound on $U_{q, t, i_{0} j_{0}}^{1}$ by a constant multiples of $2^{q (3 \eta + \epsilon)}t^{\eta} \epsilon^{\eta}$ and because the computations of $U_{q, t, i_{0} j_{0}}^{l}$ for $l\in \{2, \hdots, 6\}$ are similar, we conclude from \eqref{estimate 77}, \eqref{estimate 31}, \eqref{estimate 32}, that 
\begin{equation}\label{estimate 272}
\mathbb{E} [ \lvert \Delta_{q} L_{t, i_{0}j_{0}}^{1} \rvert^{2} ] \lesssim 2^{q(3\eta + \epsilon)} t^{\eta} \epsilon^{\eta}. 
\end{equation} 
Applying \eqref{estimate 272} on $L_{t, i_{0}j_{0}}^{1}$, \eqref{estimate 14} on $L_{t, i{0} j_{0}}^{2}$ and $L_{t, i_{0} j_{0}}^{3}$, \eqref{estimate 247} on $L_{t, i_{0} j_{0}}^{4}$, \eqref{estimate 265} on $L_{t, i_{0} j_{0}}^{5} - \bar{L}_{t, i_{0} j_{0}}^{5}$ and $L_{t, i_{0} j_{0}}^{6} - \bar{L}_{t, i_{0}j_{0}}^{6}$ and \eqref{estimate 222}, \eqref{estimate 273}, \eqref{estimate 274}, \eqref{estimate 275} on $L_{t, i_{0}j_{0}}^{7}$ to \eqref{[Equation (4.6e)][ZZ17]} and \eqref{estimate 194} gives us the necessary estimates on $\pi_{0,\diamond} (u_{32}^{\epsilon, i_{0}}, b_{1}^{\epsilon, j_{0}})(t) - \pi_{0,\diamond} (\bar{u}_{32}^{\epsilon, i_{0}}, \bar{b}_{1}^{\epsilon, j_{0}})(t)$. Similarly, we can show that for every $\eta, \epsilon, \gamma > 0$ sufficiently small and $t_{1}, t_{2} > 0$, 
\begin{align*}
& \sum_{i_{0}, j_{0} =1}^{3} \mathbb{E} [ \lvert \Delta_{q} (\pi_{0,\diamond} (u_{32}^{\epsilon, i_{0}}, b_{1}^{\epsilon, j_{0}})(t_{1}) - \pi_{0,\diamond} (u_{32}^{\epsilon, i_{0}}, b_{1}^{\epsilon, j_{0}})(t_{2}) \nonumber\\
&\hspace{5mm} - \pi_{0,\diamond} (\bar{u}_{32}^{\epsilon, i_{0}}, \bar{b}_{1}^{\epsilon, j_{0}})(t_{1}) + \pi_{0,\diamond} (\bar{u}_{32}^{\epsilon, i_{0}}, \bar{b}_{1}^{\epsilon, j_{0}})(t_{2})) \rvert^{2} ] \lesssim \epsilon^{\gamma} \lvert t_{1} - t_{2} \rvert^{\eta} 2^{q\epsilon}. 
\end{align*} 
This implies by applications of Gaussian hypercontractivity theorem and Besov embedding Lemma \ref{Lemma 3.4}, similarly to previous cases such as \eqref{estimate 36}, that for all $i_{0}, j_{0} \in \{1,2,3\}$, $\delta > 0$ sufficiently small, and $p > 1$, $\pi_{0,\diamond} (u_{32}^{\epsilon, i_{0}}, b_{1}^{\epsilon, j_{0}}) - \pi_{0,\diamond} (\bar{u}_{32}^{\epsilon, i_{0}}, \bar{b}_{1}^{\epsilon, j_{0}}) \to 0$ in $L^{p}(\Omega; C([0, T]; \mathcal{C}^{-\delta}))$ as $\epsilon \searrow 0$. 

\subsubsection{Convergence of \eqref{sixth convergence}}

W.l.o.g. let us show the necessary estimates for $\pi_{0,\diamond} ( \mathcal{P}^{i_{1} i_{2}} D_{j_{0}}^{\epsilon} K_{b}^{\epsilon, j_{0}}, b_{1}^{\epsilon, j_{1}}) - \pi_{0,\diamond} (\mathcal{P}^{i_{1}i_{2}} D_{j_{0}} \bar{K}_{b}^{\epsilon, j_{0}}, \bar{b}_{1}^{\epsilon, j_{1}})$. Considering \eqref{estimate 56}, \eqref{estimate 276}, Example \ref{Example 3.1} and \eqref{covariance}, we see that we can define 
\begin{align}
C_{34}^{\epsilon, i_{1} i_{2} j_{0} j_{1}}\triangleq& (2\pi)^{-3} \sum_{\lvert i-j\rvert \leq 1} \sum_{k_{1} \neq 0} \theta(2^{-i} k_{1}) \theta(2^{-j} k_{1}) \int_{0}^{t} e^{-2 \lvert k_{1} \rvert^{2} f(\epsilon k_{1}) (t-s)} k_{1}^{j_{0}} g(\epsilon k_{1}^{j_{0}}) \nonumber\\
& \hspace{5mm} \times \frac{ h_{b}(\epsilon k_{1})^{2}}{2\lvert k_{1} \rvert^{2} f(\epsilon k_{1})} ds \sum_{j_{2} =1}^{3} \hat{\mathcal{P}}^{j_{0} j_{2}} (k_{1}) \hat{\mathcal{P}}^{j_{1} j_{2}}(k_{1}) \hat{\mathcal{P}}^{i_{1} i_{2}}(k_{1}) \label{estimate 50}
\end{align} 
and write  from \eqref{[Equation (3.2f)][ZZ17]} and \eqref{[Equation (3.23b)][ZZ17]} that 
\begin{align}
&  \pi_{0,\diamond} ( \mathcal{P}^{i_{1} i_{2}} D_{j_{0}}^{\epsilon} K_{b}^{\epsilon, j_{0}}, b_{1}^{\epsilon, j_{1}})(t) - \pi_{0,\diamond} (\mathcal{P}^{i_{1}i_{2}} D_{j_{0}} \bar{K}_{b}^{\epsilon, j_{0}}, \bar{b}_{1}^{\epsilon, j_{1}})(t) \\
=&  (2\pi)^{-\frac{3}{2}} \sum_{k} \sum_{\lvert i-j\rvert \leq 1} \sum_{k_{1}, k_{2}: k_{12} = k} \theta(2^{-i} k_{1}) \theta(2^{-j} k_{2}) \nonumber\\
& \times [ k_{1}^{j_{0}} g(\epsilon k_{1}^{j_{0}}) \int_{0}^{t} e^{-\lvert k_{1} \rvert^{2} f(\epsilon k_{1}) (t-s)} : \hat{X}_{s,b}^{\epsilon, j_{0}} (k_{1}) \hat{X}_{t,b}^{\epsilon, j_{1}}(k_{2}): ds \nonumber\\
& \hspace{5mm} - k_{1}^{j_{0}} i \int_{0}^{t} e^{- \lvert k_{1} \rvert^{2} (t-s)} : \hat{\bar{X}}_{s,b}^{\epsilon, j_{0}} (k_{1}) \hat{\bar{X}}_{t,b}^{\epsilon, j_{1}}(k_{2}): ds] \hat{\mathcal{P}}^{i_{1}i_{2}}(k_{1}) e_{k}. \nonumber
\end{align}
We compute 
\begin{align}\label{estimate 280}
& \mathbb{E} [ \lvert \Delta_{q} [ \pi_{0,\diamond}(\mathcal{P}^{i_{1} i_{2}} D_{j_{0}}^{\epsilon}K_{b}^{\epsilon, j_{0}}, b_{1}^{\epsilon, j_{1}})(t) \nonumber\\
& \hspace{20mm} - \pi_{0,\diamond} (\mathcal{P}^{i_{1} i_{2}} D_{j_{0}} \bar{K}_{b}^{\epsilon, j_{0}}, \bar{b}_{1}^{\epsilon, j_{1}})(t) ] \rvert^{2}]  \lesssim \sum_{l=1}^{2} VI_{q, t, j_{0} j_{1} j_{1}'}^{l}
\end{align}
where 
\begin{subequations}\label{estimate 277}
\begin{align}
VI_{q, t, j_{0}j_{1}j_{1}'}^{1} &\triangleq \sum_{k} \sum_{\lvert i-j\rvert \leq 1, \lvert i' - j'\rvert \leq 1} \sum_{k_{1}, k_{2}, k_{1}', k_{2}' \neq 0: k_{12} = k_{12}' = k} \theta(2^{-q} k)^{2}\theta(2^{-i} k_{1}) \theta(2^{-i'} k_{1}') \nonumber\\
& \times  \theta(2^{-j} k_{2}) \theta(2^{-j'}k_{2}') \lvert k_{1} \rvert \lvert k_{1}'\rvert  \int_{[0,t]^{2}} \lvert e^{-\lvert k_{1} \rvert^{2} f(\epsilon k_{1})(t-s)} g(\epsilon k_{1}^{j_{0}}) - e^{- \lvert k_{1} \rvert^{2} (t-s)} i \rvert \nonumber\\
& \hspace{10mm} \times \lvert e^{-\lvert k_{1}' \rvert^{2} f(\epsilon k_{1}') (t-\bar{s})} g(\epsilon (k_{1}')^{j_{0}}) - e^{- \lvert k_{1}' \rvert^{2} (t-\bar{s})} i \rvert \nonumber\\
& \hspace{10mm} \times \mathbb{E} [: \hat{X}_{s,b}^{\epsilon, j_{0}}(k_{1}) \hat{X}_{t,b}^{\epsilon, j_{1}}(k_{2}): \overline{: \hat{X}_{\bar{s}, b}^{\epsilon, j_{0}} (k_{1}') \hat{X}_{t,b}^{\epsilon, j_{1}'} (k_{2}'):}] ds d \bar{s}, \\
VI_{q, t, j_{0}j_{1}j_{1}'}^{2} &\triangleq \sum_{k} \sum_{\lvert i-j\rvert \leq 1, \lvert i'-j'\rvert \leq 1} \sum_{k_{1}, k_{2}, k_{1}', k_{2}' \neq 0: k_{12} = k_{12}' = k} \theta(2^{-q} k)^{2}\theta(2^{-i} k_{1}) \theta(2^{-i'} k_{1}')  \nonumber\\
& \times \theta(2^{-j} k_{2}) \theta(2^{-j'} k_{2}') \int_{[0,t]^{2}}  e^{- \lvert k_{1} \rvert^{2} (t-s) - \lvert k_{1}' \rvert^{2} (t- \bar{s})} \lvert k_{1} \rvert \lvert k_{1} ' \rvert \nonumber\\
& \hspace{5mm} \times \mathbb{E} [(: \hat{X}_{s,b}^{\epsilon, j_{0}}(k_{1}) \hat{X}_{t,b}^{\epsilon, j_{1}}(k_{2}): - : \hat{\bar{X}}_{s,b}^{\epsilon, j_{0}}(k_{1}) \hat{\bar{X}}_{t,b}^{\epsilon, j_{1}}(k_{2}):)  \nonumber\\
& \hspace{5mm}  \times \overline{(: \hat{X}_{\bar{s},b}^{\epsilon, j_{0}}(k_{1}') \hat{X}_{t,b}^{\epsilon, j_{1}'}(k_{2}'): - : \hat{\bar{X}}_{\bar{s},b}^{\epsilon, j_{0}}(k_{1}') \hat{\bar{X}}_{t,b}^{\epsilon, j_{1}'}(k_{2}'):)} ]ds d \bar{s}. 
\end{align}
\end{subequations} 
For $VI_{q, t, j_{0}j_{1}j_{1}'}^{1}$ in \eqref{estimate 277}, we rely on Example \ref{Example 3.1} and \eqref{covariance} so that for any $\eta \in [0,1]$ 
\begin{align}\label{estimate 278}
VI_{q, t, j_{0}j_{1}j_{1}'}^{1} &\lesssim \epsilon^{\eta} \sum_{k} \sum_{\lvert i-j\rvert \leq 1, \lvert i'-j'\rvert \leq 1} \sum_{k_{1}, k_{2} \neq 0: k_{12} = k} \theta(2^{-q} k)^{2} \theta(2^{-i} k_{1}) \theta(2^{-j} k_{2}) \nonumber\\
& \times [ \theta(2^{-i'} k_{1}) \theta(2^{-j'} k_{2}) \frac{ \lvert k_{1} \rvert^{\eta}}{\lvert k_{2} \rvert^{2}} \int_{[0,t]^{2}} \frac{1}{[\lvert k_{1} \rvert^{2} (t-s)]^{1- \frac{\epsilon}{2}}} \frac{1}{[\lvert k_{1} \rvert^{2} (t- \bar{s})]^{1- \frac{\epsilon}{2}}} ds d \bar{s} \nonumber\\
&+ \theta(2^{-i'} k_{2}) \theta(2^{-j'} k_{1}) \frac{1}{\lvert k_{1} \rvert^{1- \frac{\eta}{2}} \lvert k_{2} \rvert^{1- \frac{\eta}{2}}} \int_{[0,t]^{2}} \frac{1}{ [ \lvert k_{1} \rvert^{2} (t-s)]^{1-\frac{\epsilon}{2}}} \nonumber\\
& \hspace{35mm} \times \frac{1}{ [\lvert k_{2} \rvert^{2} (t- \bar{s})]^{1-\frac{\epsilon}{2}}} ds d \bar{s}]  \lesssim \epsilon^{\eta} t^{\epsilon} 2^{q(2\epsilon + \eta)}
\end{align}
where we used \eqref{[Equation (4.2)][ZZ17]}, \eqref{[Equation (4.3)][ZZ17]}, \eqref{key estimate}, that $2^{q} \lesssim 2^{i}$ so that $q \lesssim i$ and Lemma \ref{Lemma 3.13}.  Similarly, we deduce from \eqref{estimate 277}, 
\begin{align}\label{estimate 279}
VI_{q, t, j_{0}j_{1}j_{1}'}^{2} \lesssim& \epsilon^{\eta} t^{\epsilon} \sum_{k} \sum_{\lvert i-j\rvert \leq 1, \lvert i'-j'\rvert \leq 1} \sum_{k_{1}, k_{2} \neq 0: k_{12} = k} \theta(2^{-q} k)^{2} \theta(2^{-i} k_{1}) \theta(2^{-j} k_{2}) \\
& \times [ \theta(2^{-i'} k_{1}) \theta(2^{-j'} k_{2}) \frac{ \lvert k_{1} \rvert^{\eta} + \lvert k_{2} \rvert^{\eta}}{\lvert k_{1} \rvert^{4- 2 \epsilon}\lvert k_{2} \rvert^{2}}  \nonumber  \\
& \hspace{10mm} + \theta(2^{-i'} k_{2}) \theta(2^{-j'} k_{1}) \frac{ \lvert k_{1} \rvert^{\eta} + \lvert k_{2} \rvert^{\eta}}{\lvert k_{1} \rvert^{3-\epsilon} \lvert k_{2} \rvert^{3-\epsilon}}] \lesssim  \epsilon^{\eta} t^{\epsilon} 2^{q(2\epsilon + \eta)}\nonumber
\end{align}
by \eqref{[Equation (4.7)][ZZ17]}, \eqref{key estimate}, Lemma \ref{Lemma 3.13} and that $2^{q} \lesssim 2^{i} \approx 2^{j}$ so that $q \lesssim i, j$. Applying \eqref{estimate 278} and \eqref{estimate 279} to \eqref{estimate 280} gives us 
\begin{equation*}
\mathbb{E} [ \lvert \Delta_{q} [ \pi_{0,\diamond}(\mathcal{P}^{i_{1} i_{2}} D_{j_{0}}^{\epsilon}K_{b}^{\epsilon, j_{0}}, b_{1}^{\epsilon, j_{1}})(t) - \pi_{0,\diamond} (\mathcal{P}^{i_{1} i_{2}} D_{j_{0}} \bar{K}_{b}^{\epsilon, j_{0}}, \bar{b}_{1}^{\epsilon, j_{1}})(t) ] \rvert^{2}]  \lesssim \epsilon^{\eta} t^{\epsilon} 2^{q(2\epsilon + \eta)}. 
\end{equation*}
Similarly, we can show that for all $\epsilon > 0, \gamma > 0$ sufficiently small, 
\begin{align*}
&\mathbb{E} [ \lvert \Delta_{q} [ \pi_{0,\diamond} (\mathcal{P}^{i_{1}i_{2}} D_{j_{0}}^{\epsilon} K_{b}^{\epsilon, j_{0}}, b_{1}^{\epsilon, j_{1}})(t_{1}) - \pi_{0,\diamond} (\mathcal{P}^{i_{1}i_{2}} D_{j_{0}}^{\epsilon} K_{b}^{\epsilon, j_{0}}, b_{1}^{\epsilon, j_{1}})(t_{2}) \nonumber\\
& \hspace{5mm} - \pi_{0,\diamond} (\mathcal{P}^{i_{1}i_{2}} D_{j_{0}} \bar{K}_{b}^{\epsilon, j_{0}}, \bar{b}_{1}^{\epsilon, j_{1}})(t_{1})  \nonumber\\
& \hspace{25mm} + \pi_{0,\diamond}(\mathcal{P}^{i_{1}i_{2}} D_{j_{0}} \bar{K}_{b}^{\epsilon, j_{0}}, \bar{b}_{1}^{\epsilon, j_{1}})(t_{2})] \rvert^{2} ]  \lesssim \epsilon^{\eta} \lvert t_{1} - t_{2} \rvert^{\eta} 2^{q(\epsilon + 3\eta)} 
\end{align*} 
so that via applications of Gaussian hypercontractivity theorem and Besov embedding Lemma \ref{Lemma 3.4}, similarly to \eqref{estimate 36}, we deduce that for all $i_{1}, i_{2}, j_{0}, j_{1} \in \{1,2,3\}$,  $\pi_{0,\diamond} (\mathcal{P}^{i_{1}i_{2}} D_{j_{0}}^{\epsilon} K_{b}^{\epsilon, j_{0}}, b_{1}^{\epsilon, j_{1}})$ $-$ $ \pi_{0,\diamond} (\mathcal{P}^{i_{1}i_{2}} D_{j_{0}} \bar{K}_{b}^{\epsilon, j_{0}}, \bar{b}_{1}^{\epsilon, j_{1}}) \to 0$ in $C([0,T]; \mathcal{C}^{-\delta})$ as $\epsilon \searrow 0$. 

\subsubsection{Convergence of \eqref{fourth convergence}}

W.l.o.g. we show the necessary estimates on $b_{2}^{\epsilon, i} \diamond b_{2}^{\epsilon, j} - \bar{b}_{2}^{\epsilon, i} \diamond \bar{b}_{2}^{\epsilon, j}$ as those on others are similar. We consider \eqref{[Equation (3.2ac)][ZZ17]}, \eqref{[Equation (3.22ac)][ZZ17]}, \eqref{[Equation (3.1ab)][ZZ17]}, \eqref{[Equation (3.21d)][ZZ17]},  \eqref{[Equation (3.2aa)][ZZ17]} and \eqref{[Equation (3.22aa)][ZZ17]} to write 
\begin{align}\label{estimate 282}
& (b_{2}^{\epsilon, i} b_{2}^{\epsilon, j} - \bar{b}_{2}^{\epsilon, i} \bar{b}_{2}^{\epsilon, j} )(t) \nonumber\\
=& \frac{ (2\pi)^{-\frac{9}{2}}}{4} \sum_{k} \sum_{k_{1}, k_{2}, k_{3}, k_{4} \neq 0: k_{1234} = k} \sum_{i_{1}, i_{2}, j_{1}, j_{2} =1}^{3} \int_{[0,t]^{2}} \nonumber\\
& \times [e^{- \lvert k_{12} \rvert^{2} f(\epsilon k_{12})(t-s) - \lvert k_{34} \rvert^{2} f(\epsilon k_{34}) (t- \bar{s})} k_{12}^{i_{2}} k_{34}^{j_{2}} g(\epsilon k_{12}^{i_{2}}) g(\epsilon k_{34}^{j_{2}}) \nonumber \\
& \hspace{2mm} \times [ \hat{X}_{s,b}^{\epsilon, i_{1}}(k_{1}) \hat{X}_{s,u}^{\epsilon, i_{2}}(k_{2}) \hat{X}_{\bar{s}, b}^{\epsilon, j_{1}}(k_{3}) \hat{X}_{\bar{s}, u}^{\epsilon, j_{2}}(k_{4}) - \hat{X}_{s,b}^{\epsilon, i_{1}}(k_{1}) \hat{X}_{s,u}^{\epsilon, i_{2}}(k_{2}) \hat{X}_{\bar{s}, u}^{\epsilon, j_{1}}(k_{3}) \hat{X}_{\bar{s}, b}^{\epsilon, j_{2}}(k_{4})  \nonumber\\
& \hspace{2mm} -  \hat{X}_{s,u}^{\epsilon, i_{1}}(k_{1}) \hat{X}_{s,b}^{\epsilon, i_{2}}(k_{2}) \hat{X}_{\bar{s}, b}^{\epsilon, j_{1}}(k_{3}) \hat{X}_{\bar{s}, u}^{\epsilon, j_{2}}(k_{4}) +  \hat{X}_{s,u}^{\epsilon, i_{1}}(k_{1}) \hat{X}_{s,b}^{\epsilon, i_{2}}(k_{2}) \hat{X}_{\bar{s}, u}^{\epsilon, j_{1}}(k_{3}) \hat{X}_{\bar{s}, b}^{\epsilon, j_{2}}(k_{4})] \nonumber\\
&  - e^{- \lvert k_{12} \rvert^{2} (t-s) - \lvert k_{34} \rvert^{2} (t- \bar{s})} k_{12}^{i_{2}} k_{34}^{j_{2}} ii \nonumber \\
& \hspace{2mm} \times [ \hat{\bar{X}}_{s,b}^{\epsilon, i_{1}}(k_{1}) \hat{\bar{X}}_{s,u}^{\epsilon, i_{2}}(k_{2}) \hat{\bar{X}}_{\bar{s}, b}^{\epsilon, j_{1}}(k_{3}) \hat{\bar{X}}_{\bar{s}, u}^{\epsilon, j_{2}}(k_{4}) - \hat{\bar{X}}_{s,b}^{\epsilon, i_{1}}(k_{1}) \hat{\bar{X}}_{s,u}^{\epsilon, i_{2}}(k_{2}) \hat{\bar{X}}_{\bar{s}, u}^{\epsilon, j_{1}}(k_{3}) \hat{\bar{X}}_{\bar{s}, b}^{\epsilon, j_{2}}(k_{4})  \nonumber\\
& \hspace{2mm} -  \hat{\bar{X}}_{s,u}^{\epsilon, i_{1}}(k_{1}) \hat{\bar{X}}_{s,b}^{\epsilon, i_{2}}(k_{2}) \hat{\bar{X}}_{\bar{s}, b}^{\epsilon, j_{1}}(k_{3}) \hat{\bar{X}}_{\bar{s}, u}^{\epsilon, j_{2}}(k_{4}) +  \hat{\bar{X}}_{s,u}^{\epsilon, i_{1}}(k_{1}) \hat{\bar{X}}_{s,b}^{\epsilon, i_{2}}(k_{2}) \hat{\bar{X}}_{\bar{s}, u}^{\epsilon, j_{1}}(k_{3}) \hat{\bar{X}}_{\bar{s}, b}^{\epsilon, j_{2}}(k_{4})]] \nonumber\\
& \times \hat{\mathcal{P}}^{ii_{1}}(k_{12}) \hat{\mathcal{P}}^{jj_{1}}(k_{34}) e_{k}ds d \bar{s} = \sum_{l=1}^{3} VII_{t, ij}^{l}
\end{align}
where 
\begin{align}\label{estimate 281}
\xi_{1} \xi_{2} \xi_{3} \xi_{4}=&  :\xi_{1} \xi_{2} \xi_{3} \xi_{4}: + \mathbb{E} [ \xi_{1} \xi_{2}] : \xi_{3} \xi_{4}: + \mathbb{E} [ \xi_{1} \xi_{3}] : \xi_{2} \xi_{4}: + \mathbb{E} [ \xi_{1} \xi_{4}] :\xi_{2} \xi_{3}: \nonumber \\
&+ \mathbb{E} [ \xi_{2} \xi_{3}]: \xi_{1} \xi_{4}: + \mathbb{E} [\xi_{2} \xi_{3}] \mathbb{E} [\xi_{1} \xi_{4}] + \mathbb{E} [\xi_{2} \xi_{4}] : \xi_{1} \xi_{3}:\nonumber \\
&+ \mathbb{E} [\xi_{2} \xi_{4}] \mathbb{E} [\xi_{1} \xi_{3}] + \mathbb{E} [\xi_{3} \xi_{4}] : \xi_{1} \xi_{2}: + \mathbb{E} [\xi_{3} \xi_{4}] \mathbb{E} [\xi_{1} \xi_{2}]
\end{align} 
due to Example \ref{Example 3.1} so that $VII_{t, ij}^{1}$ corresponds to $: \xi_{1} \xi_{2} \xi_{3} \xi_{4}:$, $VII_{t, ij}^{2}$ to those products of an expectation and a Wick product, and $VII_{t, ij}^{3}$ to the products of expectations. That is, first, 
\begin{align}\label{[Equation (4.8r)][ZZ17]} 
&VII_{t, ij}^{1} \\
&\triangleq  \frac{ (2\pi)^{-\frac{9}{2}}}{4} \sum_{k} \sum_{k_{1}, k_{2}, k_{3}, k_{4} \neq 0: k_{1234} = k} \sum_{i_{1}, i_{2}, j_{1}, j_{2} = 1}^{3} \int_{[0,t]^{2}}  \nonumber\\
& \times [ e^{- \lvert k_{12} \rvert^{2} f(\epsilon k_{12}) (t-s) - \lvert k_{34} \rvert^{2} f(\epsilon k_{34}) (t- \bar{s})} k_{12}^{i_{2}} k_{34}^{j_{2}} g(\epsilon k_{12}^{i_{2}}) g(\epsilon k_{34}^{j_{2}}) \nonumber\\
& \hspace{1mm} \times [ : \hat{X}_{s,b}^{\epsilon, i_{1}}(k_{1}) \hat{X}_{s,u}^{\epsilon, i_{2}}(k_{2}) \hat{X}_{\bar{s}, b}^{\epsilon, j_{1}}(k_{3}) \hat{X}_{\bar{s}, u}^{\epsilon, j_{2}}(k_{4}): - : \hat{X}_{s,b}^{\epsilon, i_{1}}(k_{1}) \hat{X}_{s,u}^{\epsilon, i_{2}}(k_{2}) \hat{X}_{\bar{s}, u}^{\epsilon, j_{1}}(k_{3}) \hat{X}_{\bar{s}, b}^{\epsilon, j_{2}}(k_{4}):
 \nonumber\\
 & \hspace{1mm} - : \hat{X}_{s,u}^{\epsilon, i_{1}}(k_{1}) \hat{X}_{s,b}^{\epsilon, i_{2}}(k_{2}) \hat{X}_{\bar{s}, b}^{\epsilon, j_{1}}(k_{3}) \hat{X}_{\bar{s}, u}^{\epsilon, j_{2}}(k_{4}): + : \hat{X}_{s,u}^{\epsilon, i_{1}}(k_{1}) \hat{X}_{s,b}^{\epsilon, i_{2}}(k_{2}) \hat{X}_{\bar{s}, u}^{\epsilon, j_{1}}(k_{3}) \hat{X}_{\bar{s}, b}^{\epsilon, j_{2}}(k_{4}): ] \nonumber\\
 & - e^{- \lvert k_{12} \rvert^{2} (t-s) - \lvert k_{34} \rvert^{2} (t- \bar{s})} k_{12}^{i_{2}} k_{34}^{j_{2}} ii \nonumber\\
& \hspace{1mm} \times [: \hat{\bar{X}}_{s,b}^{\epsilon, i_{1}}(k_{1}) \hat{\bar{X}}_{s,u}^{\epsilon, i_{2}}(k_{2}) \hat{\bar{X}}_{\bar{s}, b}^{\epsilon, j_{1}}(k_{3}) \hat{\bar{X}}_{\bar{s}, u}^{\epsilon, j_{2}}(k_{4}): - : \hat{\bar{X}}_{s,b}^{\epsilon, i_{1}}(k_{1}) \hat{\bar{X}}_{s,u}^{\epsilon, i_{2}}(k_{2}) \hat{\bar{X}}_{\bar{s}, u}^{\epsilon, j_{1}}(k_{3}) \hat{\bar{X}}_{\bar{s}, b}^{\epsilon, j_{2}}(k_{4}):
 \nonumber\\
 & \hspace{1mm} - : \hat{\bar{X}}_{s,u}^{\epsilon, i_{1}}(k_{1}) \hat{\bar{X}}_{s,b}^{\epsilon, i_{2}}(k_{2}) \hat{\bar{X}}_{\bar{s}, b}^{\epsilon, j_{1}}(k_{3}) \hat{\bar{X}}_{\bar{s}, u}^{\epsilon, j_{2}}(k_{4}): + : \hat{\bar{X}}_{s,u}^{\epsilon, i_{1}}(k_{1}) \hat{\bar{X}}_{s,b}^{\epsilon, i_{2}}(k_{2}) \hat{\bar{X}}_{\bar{s}, u}^{\epsilon, j_{1}}(k_{3}) \hat{\bar{X}}_{\bar{s}, b}^{\epsilon, j_{2}}(k_{4}): ] ] \nonumber\\
 & \times \hat{\mathcal{P}}^{ii_{1}}(k_{12}) \hat{\mathcal{P}}^{jj_{1}}(k_{34}) e_{k} ds d\bar{s}. \nonumber
\end{align}
Second, for $VII_{t, ij}^{2}$, among the terms corresponding to $\mathbb{E} [ \xi_{1} \xi_{2}]:\xi_{3}\xi_{4}:$, $\mathbb{E} [ \xi_{1}\xi_{3}]: \xi_{2}\xi_{4}:$, $\mathbb{E} [ \xi_{1}\xi_{4}] :\xi_{2}\xi_{3}:$, $\mathbb{E} [\xi_{2}\xi_{3}]:\xi_{1}\xi_{4}:$, $\mathbb{E} [\xi_{2}\xi_{4}]: \xi_{1}\xi_{3}:$ and $\mathbb{E} [\xi_{3}\xi_{4}]: \xi_{1}\xi_{2}:$ in \eqref{estimate 281}, those with $\mathbb{E}[\xi_{1}\xi_{2}]$ or $\mathbb{E} [\xi_{3}\xi_{4}]$ vanish due to $k_{12}$ and $k_{34}$ in \eqref{estimate 282}, and hence 
 \begin{equation}\label{estimate 217}
VII_{t, ij}^{2} \triangleq \sum_{l=1}^{4} VII_{t, ij}^{2l} 
 \end{equation} 
where corresponding to $\mathbb{E} [\xi_{1}\xi_{3}]: \xi_{2}\xi_{4}:$, $\mathbb{E} [\xi_{1}\xi_{4}]:\xi_{2}\xi_{3}:$, $\mathbb{E}[\xi_{2}\xi_{3}]:\xi_{1}\xi_{4}:$ and $\mathbb{E} [\xi_{2}\xi_{4}]: \xi_{1}\xi_{3}:$, we have 
\begin{subequations}
\begin{align}
&VII_{t, ij}^{21} \triangleq \frac{ (2\pi)^{-\frac{9}{2}}}{4} \sum_{k} \sum_{k_{1}, k_{2}, k_{4} \neq 0: k_{24} = k} \sum_{i_{1}, i_{2}, j_{1}, j_{2} = 1}^{3} \int_{[0,t]^{2}} \\
&\hspace{2mm} \times [ e^{- \lvert k_{12} \rvert^{2} f(\epsilon k_{12} )(t-s) - \lvert k_{4} - k_{1} \rvert^{2} f(\epsilon (k_{4} - k_{1})) (t- \bar{s})} k_{12}^{i_{2}} (k_{4}^{j_{2}} - k_{1}^{j_{2}}) g(\epsilon k_{12}^{i_{2}}) g(\epsilon (k_{4}^{j_{2}} - k_{1}^{j_{2}})) \nonumber\\
& \hspace{7mm} \times [: \hat{X}_{s,u}^{\epsilon, i_{2}}(k_{2}) \hat{X}_{\bar{s}, u}^{\epsilon, j_{2}}(k_{4}): \sum_{j_{3} =1}^{3} \frac{ e^{-\lvert k_{1} \rvert^{2} f(\epsilon k_{1}) \lvert s- \bar{s} \rvert} h_{b}(\epsilon k_{1})^{2}}{2 \lvert k_{1}\rvert^{2} f(\epsilon k_{1})} \hat{\mathcal{P}}^{i_{1}j_{3}}(k_{1}) \hat{\mathcal{P}}^{j_{1} j_{3}}(k_{1}) \nonumber\\
& \hspace{7mm} - : \hat{X}_{s,u}^{\epsilon, i_{2}}(k_{2}) \hat{X}_{\bar{s}, b}^{\epsilon, j_{2}}(k_{4}): \sum_{j_{3} =1}^{3} \frac{ e^{-\lvert k_{1} \rvert^{2} f(\epsilon k_{1}) \lvert s- \bar{s} \rvert} h_{b}(\epsilon k_{1})h_{u}(\epsilon k_{1})}{2 \lvert k_{1}\rvert^{2} f(\epsilon k_{1})} \hat{\mathcal{P}}^{i_{1}j_{3}}(k_{1}) \hat{\mathcal{P}}^{j_{1} j_{3}}(k_{1}) \nonumber\\
& \hspace{7mm} - : \hat{X}_{s,b}^{\epsilon, i_{2}}(k_{2}) \hat{X}_{\bar{s}, u}^{\epsilon, j_{2}}(k_{4}): \sum_{j_{3} =1}^{3} \frac{ e^{-\lvert k_{1} \rvert^{2} f(\epsilon k_{1}) \lvert s- \bar{s} \rvert} h_{u}(\epsilon k_{1})h_{b}(\epsilon k_{1})}{2 \lvert k_{1}\rvert^{2} f(\epsilon k_{1})} \hat{\mathcal{P}}^{i_{1}j_{3}}(k_{1}) \hat{\mathcal{P}}^{j_{1} j_{3}}(k_{1}) \nonumber\\
& \hspace{7mm} + : \hat{X}_{s,b}^{\epsilon, i_{2}}(k_{2}) \hat{X}_{\bar{s}, b}^{\epsilon, j_{2}}(k_{4}): \sum_{j_{3} =1}^{3} \frac{ e^{-\lvert k_{1} \rvert^{2} f(\epsilon k_{1}) \lvert s- \bar{s} \rvert} h_{u}(\epsilon k_{1})^{2}}{2 \lvert k_{1}\rvert^{2} f(\epsilon k_{1})} \hat{\mathcal{P}}^{i_{1}j_{3}}(k_{1}) \hat{\mathcal{P}}^{j_{1} j_{3}}(k_{1})] \nonumber\\
&\hspace{2mm} - e^{- \lvert k_{12} \rvert^{2} (t-s) - \lvert k_{4} - k_{1} \rvert^{2} (t- \bar{s})} k_{12}^{i_{2}} (k_{4} - k_{1})^{j_{2}} ii \nonumber\\
& \hspace{7mm} \times [: \hat{\bar{X}}_{s,u}^{\epsilon, i_{2}}(k_{2}) \hat{\bar{X}}_{\bar{s}, u}^{\epsilon, j_{2}}(k_{4}): \sum_{j_{3} =1}^{3} \frac{ e^{-\lvert k_{1} \rvert^{2} \lvert s- \bar{s} \rvert} h_{b}(\epsilon k_{1})^{2}}{2 \lvert k_{1}\rvert^{2}} \hat{\mathcal{P}}^{i_{1}j_{3}}(k_{1}) \hat{\mathcal{P}}^{j_{1} j_{3}}(k_{1}) \nonumber\\
& \hspace{7mm} - : \hat{\bar{X}}_{s,u}^{\epsilon, i_{2}}(k_{2}) \hat{\bar{X}}_{\bar{s}, b}^{\epsilon, j_{2}}(k_{4}): \sum_{j_{3} =1}^{3} \frac{ e^{-\lvert k_{1} \rvert^{2} \lvert s- \bar{s} \rvert} h_{b}(\epsilon k_{1})h_{u}(\epsilon k_{1})}{2 \lvert k_{1}\rvert^{2} } \hat{\mathcal{P}}^{i_{1}j_{3}}(k_{1}) \hat{\mathcal{P}}^{j_{1} j_{3}}(k_{1}) \nonumber\\
& \hspace{7mm} - : \hat{\bar{X}}_{s,b}^{\epsilon, i_{2}}(k_{2}) \hat{\bar{X}}_{\bar{s}, u}^{\epsilon, j_{2}}(k_{4}): \sum_{j_{3} =1}^{3} \frac{ e^{-\lvert k_{1} \rvert^{2}  \lvert s- \bar{s} \rvert} h_{u}(\epsilon k_{1})h_{b}(\epsilon k_{1})}{2 \lvert k_{1}\rvert^{2}} \hat{\mathcal{P}}^{i_{1}j_{3}}(k_{1}) \hat{\mathcal{P}}^{j_{1} j_{3}}(k_{1}) \nonumber\\
& \hspace{7mm} + : \hat{\bar{X}}_{s,b}^{\epsilon, i_{2}}(k_{2}) \hat{\bar{X}}_{\bar{s}, b}^{\epsilon, j_{2}}(k_{4}): \sum_{j_{3} =1}^{3} \frac{ e^{-\lvert k_{1} \rvert^{2} \lvert s- \bar{s} \rvert} h_{u}(\epsilon k_{1})^{2}}{2 \lvert k_{1}\rvert^{2} } \hat{\mathcal{P}}^{i_{1}j_{3}}(k_{1}) \hat{\mathcal{P}}^{j_{1} j_{3}}(k_{1})] ] \nonumber\\
&\hspace{2mm}  \times \hat{\mathcal{P}}^{ii_{1}} (k_{12}) \hat{\mathcal{P}}^{jj_{1}}(k_{4} - k_{1}) e_{k} ds d \bar{s}, \nonumber\\
&VII_{t, ij}^{22} \triangleq \frac{ (2\pi)^{-\frac{9}{2}}}{4} \sum_{k} \sum_{k_{1}, k_{2}, k_{3}\neq 0: k_{23} = k} \sum_{i_{1}, i_{2}, j_{1}, j_{2} = 1}^{3} \int_{[0,t]^{2}} \\
&\hspace{2mm} \times [ e^{- \lvert k_{12} \rvert^{2} f(\epsilon k_{12} )(t-s) - \lvert k_{3} - k_{1} \rvert^{2} f(\epsilon (k_{3} - k_{1})) (t- \bar{s})} k_{12}^{i_{2}} (k_{3}^{j_{2}} - k_{1}^{j_{2}}) g(\epsilon k_{12}^{i_{2}}) g(\epsilon (k_{3}^{j_{2}} - k_{1}^{j_{2}})) \nonumber\\
& \hspace{7mm} \times [: \hat{X}_{s,u}^{\epsilon, i_{2}}(k_{2}) \hat{X}_{\bar{s}, b}^{\epsilon, j_{1}}(k_{3}): \sum_{j_{3} =1}^{3} \frac{ e^{-\lvert k_{1} \rvert^{2} f(\epsilon k_{1}) \lvert s- \bar{s} \rvert} h_{b}(\epsilon k_{1})h_{u}(\epsilon k_{1})}{2 \lvert k_{1}\rvert^{2} f(\epsilon k_{1})} \hat{\mathcal{P}}^{i_{1}j_{3}}(k_{1}) \hat{\mathcal{P}}^{j_{2} j_{3}}(k_{1}) \nonumber\\
& \hspace{7mm} - : \hat{X}_{s,u}^{\epsilon, i_{2}}(k_{2}) \hat{X}_{\bar{s}, u}^{\epsilon, j_{1}}(k_{3}): \sum_{j_{3} =1}^{3} \frac{ e^{-\lvert k_{1} \rvert^{2} f(\epsilon k_{1}) \lvert s- \bar{s} \rvert} h_{b}(\epsilon k_{1})^{2}}{2 \lvert k_{1}\rvert^{2} f(\epsilon k_{1})} \hat{\mathcal{P}}^{i_{1}j_{3}}(k_{1}) \hat{\mathcal{P}}^{j_{2} j_{3}}(k_{1}) \nonumber\\
& \hspace{7mm} - : \hat{X}_{s,b}^{\epsilon, i_{2}}(k_{2}) \hat{X}_{\bar{s}, b}^{\epsilon, j_{1}}(k_{3}): \sum_{j_{3} =1}^{3} \frac{ e^{-\lvert k_{1} \rvert^{2} f(\epsilon k_{1}) \lvert s- \bar{s} \rvert} h_{u}(\epsilon k_{1})^{2}}{2 \lvert k_{1}\rvert^{2} f(\epsilon k_{1})} \hat{\mathcal{P}}^{i_{1}j_{3}}(k_{1}) \hat{\mathcal{P}}^{j_{2} j_{3}}(k_{1}) \nonumber\\
& \hspace{7mm} + : \hat{X}_{s,b}^{\epsilon, i_{2}}(k_{2}) \hat{X}_{\bar{s}, u}^{\epsilon, j_{1}}(k_{3}): \sum_{j_{3} =1}^{3} \frac{ e^{-\lvert k_{1} \rvert^{2} f(\epsilon k_{1}) \lvert s- \bar{s} \rvert} h_{u}(\epsilon k_{1})h_{b}(\epsilon k_{1})}{2 \lvert k_{1}\rvert^{2} f(\epsilon k_{1})} \hat{\mathcal{P}}^{i_{1}j_{3}}(k_{1}) \hat{\mathcal{P}}^{j_{2} j_{3}}(k_{1})] \nonumber\\
& \hspace{2mm}- e^{- \lvert k_{12} \rvert^{2} (t-s) - \lvert k_{3} - k_{1} \rvert^{2} (t- \bar{s})} k_{12}^{i_{2}} (k_{3} - k_{1})^{j_{2}} ii \nonumber\\
& \hspace{7mm} \times [: \hat{\bar{X}}_{s,u}^{\epsilon, i_{2}}(k_{2}) \hat{\bar{X}}_{\bar{s}, b}^{\epsilon, j_{1}}(k_{3}): \sum_{j_{3} =1}^{3} \frac{ e^{-\lvert k_{1} \rvert^{2} \lvert s- \bar{s} \rvert} h_{b}(\epsilon k_{1})h_{u}(\epsilon k_{1})}{2 \lvert k_{1}\rvert^{2}} \hat{\mathcal{P}}^{i_{1}j_{3}}(k_{1}) \hat{\mathcal{P}}^{j_{2} j_{3}}(k_{1}) \nonumber\\
& \hspace{7mm} - : \hat{\bar{X}}_{s,u}^{\epsilon, i_{2}}(k_{2}) \hat{\bar{X}}_{\bar{s}, u}^{\epsilon, j_{1}}(k_{3}): \sum_{j_{3} =1}^{3} \frac{ e^{-\lvert k_{1} \rvert^{2} \lvert s- \bar{s} \rvert} h_{b}(\epsilon k_{1})^{2}}{2 \lvert k_{1}\rvert^{2} } \hat{\mathcal{P}}^{i_{1}j_{3}}(k_{1}) \hat{\mathcal{P}}^{j_{2} j_{3}}(k_{1}) \nonumber\\
& \hspace{7mm} - : \hat{\bar{X}}_{s,b}^{\epsilon, i_{2}}(k_{2}) \hat{\bar{X}}_{\bar{s}, b}^{\epsilon, j_{1}}(k_{3}): \sum_{j_{3} =1}^{3} \frac{ e^{-\lvert k_{1} \rvert^{2}  \lvert s- \bar{s} \rvert} h_{u}(\epsilon k_{1})^{2}}{2 \lvert k_{1}\rvert^{2}} \hat{\mathcal{P}}^{i_{1}j_{3}}(k_{1}) \hat{\mathcal{P}}^{j_{2} j_{3}}(k_{1}) \nonumber\\
& \hspace{7mm} + : \hat{\bar{X}}_{s,b}^{\epsilon, i_{2}}(k_{2}) \hat{\bar{X}}_{\bar{s}, u}^{\epsilon, j_{1}}(k_{3}): \sum_{j_{3} =1}^{3} \frac{ e^{-\lvert k_{1} \rvert^{2} \lvert s- \bar{s} \rvert} h_{u}(\epsilon k_{1})h_{b}(\epsilon k_{1})}{2 \lvert k_{1}\rvert^{2} } \hat{\mathcal{P}}^{i_{1}j_{3}}(k_{1}) \hat{\mathcal{P}}^{j_{2} j_{3}}(k_{1})] ] \nonumber\\
&\hspace{2mm}  \times \hat{\mathcal{P}}^{ii_{1}} (k_{12}) \hat{\mathcal{P}}^{jj_{1}}(k_{3} - k_{1}) e_{k} ds d \bar{s}, \nonumber \\
&VII_{t, ij}^{23} \triangleq \frac{ (2\pi)^{-\frac{9}{2}}}{4} \sum_{k} \sum_{k_{1}, k_{2}, k_{4} \neq 0: k_{14} = k} \sum_{i_{1}, i_{2}, j_{1}, j_{2} = 1}^{3} \int_{[0,t]^{2}} \\
&\hspace{2mm} \times [ e^{- \lvert k_{12} \rvert^{2} f(\epsilon k_{12} )(t-s) - \lvert k_{4} - k_{2} \rvert^{2} f(\epsilon (k_{4} - k_{2})) (t- \bar{s})} k_{12}^{i_{2}} (k_{4}^{j_{2}} - k_{2}^{j_{2}}) g(\epsilon k_{12}^{i_{2}}) g(\epsilon (k_{4}^{j_{2}} - k_{2}^{j_{2}})) \nonumber\\
& \hspace{7mm} \times [: \hat{X}_{s,b}^{\epsilon, i_{1}}(k_{1}) \hat{X}_{\bar{s}, u}^{\epsilon, j_{2}}(k_{4}): \sum_{j_{3} =1}^{3} \frac{ e^{-\lvert k_{2} \rvert^{2} f(\epsilon k_{2}) \lvert s- \bar{s} \rvert} h_{u}(\epsilon k_{2})h_{b}(\epsilon k_{2})}{2 \lvert k_{2}\rvert^{2} f(\epsilon k_{2})} \hat{\mathcal{P}}^{i_{2}j_{3}}(k_{2}) \hat{\mathcal{P}}^{j_{1} j_{3}}(k_{2}) \nonumber\\
& \hspace{7mm} - : \hat{X}_{s,b}^{\epsilon, i_{1}}(k_{1}) \hat{X}_{\bar{s}, b}^{\epsilon, j_{2}}(k_{4}): \sum_{j_{3} =1}^{3} \frac{ e^{-\lvert k_{2} \rvert^{2} f(\epsilon k_{2}) \lvert s- \bar{s} \rvert} h_{u}(\epsilon k_{2})^{2}}{2 \lvert k_{2}\rvert^{2} f(\epsilon k_{2})} \hat{\mathcal{P}}^{i_{2}j_{3}}(k_{2}) \hat{\mathcal{P}}^{j_{1} j_{3}}(k_{2}) \nonumber\\
& \hspace{7mm} - : \hat{X}_{s,u}^{\epsilon, i_{1}}(k_{1}) \hat{X}_{\bar{s}, u}^{\epsilon, j_{2}}(k_{4}): \sum_{j_{3} =1}^{3} \frac{ e^{-\lvert k_{2} \rvert^{2} f(\epsilon k_{2}) \lvert s- \bar{s} \rvert} h_{b}(\epsilon k_{2})^{2}}{2 \lvert k_{2}\rvert^{2} f(\epsilon k_{2})} \hat{\mathcal{P}}^{i_{2}j_{3}}(k_{2}) \hat{\mathcal{P}}^{j_{1} j_{3}}(k_{2}) \nonumber\\
& \hspace{7mm} + : \hat{X}_{s,u}^{\epsilon, i_{1}}(k_{1}) \hat{X}_{\bar{s}, b}^{\epsilon, j_{2}}(k_{4}): \sum_{j_{3} =1}^{3} \frac{ e^{-\lvert k_{2} \rvert^{2} f(\epsilon k_{2}) \lvert s- \bar{s} \rvert} h_{b}(\epsilon k_{2}) h_{u}(\epsilon k_{2})}{2 \lvert k_{2}\rvert^{2} f(\epsilon k_{2})} \hat{\mathcal{P}}^{i_{2}j_{3}}(k_{2}) \hat{\mathcal{P}}^{j_{1} j_{3}}(k_{2})] \nonumber\\
& \hspace{2mm}- e^{- \lvert k_{12} \rvert^{2} (t-s) - \lvert k_{4} - k_{2} \rvert^{2} (t- \bar{s})} k_{12}^{i_{2}} (k_{4} - k_{2})^{j_{2}} ii \nonumber\\
& \hspace{7mm} \times [: \hat{\bar{X}}_{s,b}^{\epsilon, i_{1}}(k_{1}) \hat{\bar{X}}_{\bar{s}, u}^{\epsilon, j_{2}}(k_{4}): \sum_{j_{3} =1}^{3} \frac{ e^{-\lvert k_{2} \rvert^{2} \lvert s- \bar{s} \rvert} h_{u}(\epsilon k_{2}) h_{b}(\epsilon k_{2})}{2 \lvert k_{1}\rvert^{2}} \hat{\mathcal{P}}^{i_{2}j_{3}}(k_{2}) \hat{\mathcal{P}}^{j_{1} j_{3}}(k_{2}) \nonumber\\
& \hspace{7mm} - : \hat{\bar{X}}_{s,b}^{\epsilon, i_{1}}(k_{1}) \hat{\bar{X}}_{\bar{s}, b}^{\epsilon, j_{2}}(k_{4}): \sum_{j_{3} =1}^{3} \frac{ e^{-\lvert k_{2} \rvert^{2} \lvert s- \bar{s} \rvert} h_{u}(\epsilon k_{2})^{2}}{2 \lvert k_{2}\rvert^{2} } \hat{\mathcal{P}}^{i_{2}j_{3}}(k_{2}) \hat{\mathcal{P}}^{j_{1} j_{3}}(k_{2}) \nonumber\\
& \hspace{7mm} - : \hat{\bar{X}}_{s,u}^{\epsilon, i_{1}}(k_{1}) \hat{\bar{X}}_{\bar{s}, u}^{\epsilon, j_{2}}(k_{4}): \sum_{j_{3} =1}^{3} \frac{ e^{-\lvert k_{2} \rvert^{2}  \lvert s- \bar{s} \rvert} h_{b}(\epsilon k_{2})^{2}}{2 \lvert k_{2}\rvert^{2}} \hat{\mathcal{P}}^{i_{2}j_{3}}(k_{2}) \hat{\mathcal{P}}^{j_{1} j_{3}}(k_{2}) \nonumber\\
& \hspace{7mm} + : \hat{\bar{X}}_{s,u}^{\epsilon, i_{1}}(k_{1}) \hat{\bar{X}}_{\bar{s}, b}^{\epsilon, j_{2}}(k_{4}): \sum_{j_{3} =1}^{3} \frac{ e^{-\lvert k_{2} \rvert^{2} \lvert s- \bar{s} \rvert} h_{u}(\epsilon k_{2})h_{b}(\epsilon k_{2})}{2 \lvert k_{2}\rvert^{2} } \hat{\mathcal{P}}^{i_{2}j_{3}}(k_{2}) \hat{\mathcal{P}}^{j_{1} j_{3}}(k_{2})] ] \nonumber\\
& \times \hat{\mathcal{P}}^{ii_{1}} (k_{12}) \hat{\mathcal{P}}^{jj_{1}}(k_{4} - k_{2}) e_{k} ds d \bar{s}, \nonumber \\
&VII_{t, ij}^{24} \triangleq \frac{ (2\pi)^{-\frac{9}{2}}}{4} \sum_{k} \sum_{k_{1}, k_{2}, k_{3}: k_{13} = k} \sum_{i_{1}, i_{2}, j_{1}, j_{2} = 1}^{3} \int_{[0,t]^{2}} \\
&\hspace{2mm} \times [ e^{- \lvert k_{12} \rvert^{2} f(\epsilon k_{12} )(t-s) - \lvert k_{3} - k_{2} \rvert^{2} f(\epsilon (k_{3} - k_{2})) (t- \bar{s})} k_{12}^{i_{2}} (k_{3}^{j_{2}} - k_{2}^{j_{2}}) g(\epsilon k_{12}^{i_{2}}) g(\epsilon (k_{3}^{j_{2}} - k_{2}^{j_{2}})) \nonumber\\
& \hspace{7mm} \times [: \hat{X}_{s,b}^{\epsilon, i_{1}}(k_{1}) \hat{X}_{\bar{s}, b}^{\epsilon, j_{1}}(k_{3}): \sum_{j_{3} =1}^{3} \frac{ e^{-\lvert k_{2} \rvert^{2} f(\epsilon k_{2}) \lvert s- \bar{s} \rvert} h_{u}(\epsilon k_{2})^{2}}{2 \lvert k_{2}\rvert^{2} f(\epsilon k_{2})} \hat{\mathcal{P}}^{i_{2}j_{3}}(k_{2}) \hat{\mathcal{P}}^{j_{2} j_{3}}(k_{2}) \nonumber\\
& \hspace{7mm} - : \hat{X}_{s,b}^{\epsilon, i_{1}}(k_{1}) \hat{X}_{\bar{s}, u}^{\epsilon, j_{1}}(k_{3}): \sum_{j_{3} =1}^{3} \frac{ e^{-\lvert k_{2} \rvert^{2} f(\epsilon k_{2}) \lvert s- \bar{s} \rvert} h_{u}(\epsilon k_{2}) h_{b}(\epsilon k_{2})}{2 \lvert k_{2}\rvert^{2} f(\epsilon k_{2})} \hat{\mathcal{P}}^{i_{2}j_{3}}(k_{2}) \hat{\mathcal{P}}^{j_{2} j_{3}}(k_{2}) \nonumber\\
& \hspace{7mm} - : \hat{X}_{s,u}^{\epsilon, i_{1}}(k_{1}) \hat{X}_{\bar{s}, b}^{\epsilon, j_{1}}(k_{3}): \sum_{j_{3} =1}^{3} \frac{ e^{-\lvert k_{2} \rvert^{2} f(\epsilon k_{2}) \lvert s- \bar{s} \rvert} h_{b}(\epsilon k_{2})h_{u}(\epsilon k_{2})}{2 \lvert k_{2}\rvert^{2} f(\epsilon k_{2})} \hat{\mathcal{P}}^{i_{2}j_{3}}(k_{2}) \hat{\mathcal{P}}^{j_{2} j_{3}}(k_{2}) \nonumber\\
& \hspace{7mm} + : \hat{X}_{s,u}^{\epsilon, i_{1}}(k_{1}) \hat{X}_{\bar{s}, u}^{\epsilon, j_{1}}(k_{3}): \sum_{j_{3} =1}^{3} \frac{ e^{-\lvert k_{2} \rvert^{2} f(\epsilon k_{2}) \lvert s- \bar{s} \rvert} h_{b}(\epsilon k_{2})^{2}}{2 \lvert k_{2}\rvert^{2} f(\epsilon k_{2})} \hat{\mathcal{P}}^{i_{2}j_{3}}(k_{2}) \hat{\mathcal{P}}^{j_{2} j_{3}}(k_{2})] \nonumber\\
& \hspace{2mm}- e^{- \lvert k_{12} \rvert^{2} (t-s) - \lvert k_{3} - k_{2} \rvert^{2} (t- \bar{s})} k_{12}^{i_{2}} (k_{3} - k_{2})^{j_{2}} ii \nonumber\\
& \hspace{7mm} \times [: \hat{\bar{X}}_{s,b}^{\epsilon, i_{1}}(k_{1}) \hat{\bar{X}}_{\bar{s}, b}^{\epsilon, j_{1}}(k_{3}): \sum_{j_{3} =1}^{3} \frac{ e^{-\lvert k_{2} \rvert^{2} \lvert s- \bar{s} \rvert} h_{u}(\epsilon k_{2})^{2}}{2 \lvert k_{2}\rvert^{2}} \hat{\mathcal{P}}^{i_{2}j_{3}}(k_{2}) \hat{\mathcal{P}}^{j_{2} j_{3}}(k_{2}) \nonumber\\
& \hspace{7mm} - : \hat{\bar{X}}_{s,b}^{\epsilon, i_{1}}(k_{1}) \hat{\bar{X}}_{\bar{s}, u}^{\epsilon, j_{1}}(k_{3}): \sum_{j_{3} =1}^{3} \frac{ e^{-\lvert k_{2} \rvert^{2} \lvert s- \bar{s} \rvert} h_{u}(\epsilon k_{2})h_{b}(\epsilon k_{2})}{2 \lvert k_{2}\rvert^{2} } \hat{\mathcal{P}}^{i_{2}j_{3}}(k_{2}) \hat{\mathcal{P}}^{j_{2} j_{3}}(k_{2}) \nonumber\\
& \hspace{7mm} - : \hat{\bar{X}}_{s,u}^{\epsilon, i_{1}}(k_{1}) \hat{\bar{X}}_{\bar{s}, b}^{\epsilon, j_{1}}(k_{3}): \sum_{j_{3} =1}^{3} \frac{ e^{-\lvert k_{2} \rvert^{2}  \lvert s- \bar{s} \rvert} h_{u}(\epsilon k_{2})h_{b}(\epsilon k_{2})}{2 \lvert k_{2}\rvert^{2}} \hat{\mathcal{P}}^{i_{2}j_{3}}(k_{2}) \hat{\mathcal{P}}^{j_{2} j_{3}}(k_{2}) \nonumber\\
& \hspace{7mm} + : \hat{\bar{X}}_{s,u}^{\epsilon, i_{1}}(k_{1}) \hat{\bar{X}}_{\bar{s}, u}^{\epsilon, j_{1}}(k_{3}): \sum_{j_{3} =1}^{3} \frac{ e^{-\lvert k_{2} \rvert^{2} \lvert s- \bar{s} \rvert} h_{b}(\epsilon k_{2})^{2}}{2 \lvert k_{2}\rvert^{2} } \hat{\mathcal{P}}^{i_{2}j_{3}}(k_{2}) \hat{\mathcal{P}}^{j_{2} j_{3}}(k_{2})] ] \nonumber\\
&\hspace{2mm} \times \hat{\mathcal{P}}^{ii_{1}} (k_{12}) \hat{\mathcal{P}}^{jj_{1}}(k_{3} - k_{2}) e_{k} ds d \bar{s}. \nonumber
\end{align}
\end{subequations}  
Finally, $VII_{t, ij}^{3}$ consists of terms corresponding to $\mathbb{E} [ \xi_{2}\xi_{3}]\mathbb{E} [ \xi_{1}\xi_{4}]$, $\mathbb{E} [\xi_{2}\xi_{4}]\mathbb{E} [\xi_{1}\xi_{3}]$ and $\mathbb{E} [ \xi_{3}\xi_{4}]\mathbb{E}[\xi_{1}\xi_{2}]$ in \eqref{estimate 281} of which the last term vanishes due to $k_{12}$ and $k_{34}$ in \eqref{estimate 282}. In sum, we have 
\begin{align}\label{estimate 218}
VII_{t, ij}^{3} \triangleq & \frac{ (2\pi)^{-6}}{4} \sum_{k_{1}, k_{2} \neq 0} \sum_{i_{1}, i_{2}, j_{1}, j_{2} =1}^{3} \\
& \times [\int_{[0,t]^{2}} e^{- \lvert k_{12} \rvert^{2} f(\epsilon k_{12}) (t-s) - \lvert k_{12} \rvert^{2} f(\epsilon k_{12}) (t- \bar{s})} k_{12}^{i_{2}} (-k_{12}^{j_{2}}) g(\epsilon k_{12}^{i_{2}}) g(- k_{12}^{j_{2}})\nonumber\\
& \hspace{5mm} \times \frac{ e^{- \lvert k_{2} \rvert^{2} f(\epsilon k_{2}) \lvert s- \bar{s} \rvert - \lvert k_{1} \rvert^{2} f(\epsilon k_{1}) \lvert s - \bar{s} \rvert}}{ \prod_{i=1}^{2} (2 \lvert k_{i} \rvert^{2} f(\epsilon k_{i}))}  ds d \bar{s} \nonumber\\
& - \int_{[0,t]^{2}} e^{- \lvert k_{12} \rvert^{2} (t-s) - \lvert k_{12} \rvert^{2} (t- \bar{s})} k_{12}^{i_{2}} (- k_{12}^{j_{2}}) ii \frac{ e^{- \lvert k_{2} \rvert^{2} \lvert s - \bar{s} \rvert - \lvert k_{1} \rvert^{2} \lvert s- \bar{s} \rvert}}{\prod_{i=1}^{2} (2 \lvert k_{i} \rvert^{2})}  ds d \bar{s} ] \nonumber\\
& \times Y_{k_{1}k_{2}} \hat{\mathcal{P}}^{ii_{1}}(k_{12}) \hat{\mathcal{P}}^{jj_{1}}(k_{12}) \sum_{j_{3}, j_{4} =1}^{3} [ \hat{\mathcal{P}}^{i_{2} j_{4}}(k_{2}) \hat{\mathcal{P}}^{j_{1} j_{4}}(k_{2}) \hat{\mathcal{P}}^{i_{1} j_{3}}(k_{1}) \hat{\mathcal{P}}^{j_{2}j_{3}}(k_{1}) \nonumber\\
& \hspace{25mm}  - \hat{\mathcal{P}}^{i_{2} j_{4}}(k_{2}) \hat{\mathcal{P}}^{j_{2} j_{4}}(k_{2}) \hat{\mathcal{P}}^{i_{1} j_{3}}(k_{1}) \hat{\mathcal{P}}^{j_{1}j_{3}}(k_{1})]  \nonumber
\end{align} 
where we defined 
\begin{align}
Y_{k_{1}k_{2}} \triangleq& 2 h_{u}(\epsilon k_{1}) h_{b}(\epsilon k_{1}) h_{u}(\epsilon k_{2}) h_{b}(\epsilon k_{2}) \nonumber\\
& - h_{u}(\epsilon k_{2})^{2} h_{b}(\epsilon k_{1})^{2} - h_{u}(\epsilon k_{1})^{2} h_{b}(\epsilon k_{2})^{2}.
\end{align} 
In contrast to the case of the NS equations in \cite{ZZ17}, $VII_{t, ij}^{3}$ is complicated due to $h_{u}, h_{b}$ and the opposite signs within the last summation over $j_{3}$ and $j_{4}$, specifically $[ \hat{\mathcal{P}}^{i_{2} j_{4}}(k_{2}) \hat{\mathcal{P}}^{j_{1} j_{4}}(k_{2}) \hat{\mathcal{P}}^{i_{1} j_{3}}(k_{1}) \hat{\mathcal{P}}^{j_{2}j_{3}}(k_{1})$ $- \hat{\mathcal{P}}^{i_{2} j_{4}}(k_{2}) \hat{\mathcal{P}}^{j_{2} j_{4}}(k_{2}) \hat{\mathcal{P}}^{i_{1} j_{3}}(k_{1}) \hat{\mathcal{P}}^{j_{1}j_{3}}(k_{1})]$.\\

\emph{Term in the zeroth chaos: $VII_{t, ij}^{3}$ in \eqref{estimate 218}}\\

Applying Fubini theorem, it can readily be verified that $VII_{t, ij}^{3}$ satisfies 
\begin{equation}
\phi_{22}^{\epsilon, ij} - \bar{\phi}_{22}^{\epsilon, ij} = VII_{t, ij}^{3} - C_{22}^{\epsilon, ij} + \bar{C}_{22}^{\epsilon, ij}, 
\end{equation}  
where 
\begin{subequations}
\begin{align}
C_{22}^{\epsilon, ij} &\triangleq \frac{ (2\pi)^{-6}}{4} \sum_{k_{1}, k_{2} \neq 0} \sum_{i_{1}, i_{2}, j_{1}, j_{2} =1}^{3} Y_{k_{1}k_{2}}\hat{\mathcal{P}}^{ii_{1}}(k_{12}) \hat{\mathcal{P}}^{jj_{1}}(k_{12}) \label{estimate 43}\\
& \times \sum_{j_{3}, j_{4} =1}^{3} [ \hat{\mathcal{P}}^{i_{2} j_{4}}(k_{2}) \hat{\mathcal{P}}^{j_{1} j_{4}}(k_{2}) \hat{\mathcal{P}}^{i_{1} j_{3}}(k_{1}) \hat{\mathcal{P}}^{j_{2}j_{3}}(k_{1}) \nonumber\\
& \hspace{5mm} - \hat{\mathcal{P}}^{i_{2} j_{4}}(k_{2}) \hat{\mathcal{P}}^{j_{2} j_{4}}(k_{2}) \hat{\mathcal{P}}^{i_{1} j_{3}}(k_{1}) \hat{\mathcal{P}}^{j_{1}j_{3}}(k_{1})]   (- k_{12}^{i_{2}} k_{12}^{j_{2}}) g(\epsilon k_{12}^{i_{2}}) g(-\epsilon k_{12}^{j_{2}}) \nonumber\\
& \hspace{5mm} \times \frac{1}{ [ \prod_{i=1}^{2} (2 \lvert k_{i} \rvert^{2} f(\epsilon k_{i}))][\lvert k_{12} \rvert^{2} f(\epsilon k_{12}) + \sum_{i=1}^{2} \lvert k_{i} \rvert^{2} f(\epsilon k_{i})]  }\frac{1}{\lvert k_{12} \rvert^{2} f(\epsilon k_{12})}, \nonumber \\
\bar{C}_{22}^{\epsilon, ij} &\triangleq - \frac{ (2\pi)^{-6}}{4} \sum_{k_{1}, k_{2} \neq 0} \sum_{i_{1}, i_{2}, j_{1}, j_{2} =1}^{3} Y_{k_{1}k_{2}}\hat{\mathcal{P}}^{ii_{1}}(k_{12}) \hat{\mathcal{P}}^{jj_{1}}(k_{12}) \label{estimate 44}  \\
& \times \sum_{j_{3}, j_{4} =1}^{3} [ \hat{\mathcal{P}}^{i_{2} j_{4}}(k_{2}) \hat{\mathcal{P}}^{j_{1} j_{4}}(k_{2}) \hat{\mathcal{P}}^{i_{1} j_{3}}(k_{1}) \hat{\mathcal{P}}^{j_{2}j_{3}}(k_{1}) \nonumber\\
& \hspace{5mm} - \hat{\mathcal{P}}^{i_{2} j_{4}}(k_{2}) \hat{\mathcal{P}}^{j_{2} j_{4}}(k_{2}) \hat{\mathcal{P}}^{i_{1} j_{3}}(k_{1}) \hat{\mathcal{P}}^{j_{1}j_{3}}(k_{1})] k_{12}^{i_{2}} k_{12}^{j_{2}}ii \nonumber\\
& \times  \frac{1}{ [ \prod_{i=1}^{2} (2 \lvert k_{i} \rvert^{2} )][\lvert k_{12} \rvert^{2} + \sum_{i=1}^{2} \lvert k_{i} \rvert^{2} ]  }\frac{1}{\lvert k_{12} \rvert^{2} }, \nonumber \\
\phi_{22}^{\epsilon, ij}(t) &\triangleq \frac{ (2\pi)^{-6}}{4} \sum_{k_{1}, k_{2} \neq 0} \sum_{i_{1}, i_{2}, j_{1}, j_{2} =1}^{3} Y_{k_{1}k_{2}}  \hat{\mathcal{P}}^{ii_{1}}(k_{12}) \hat{\mathcal{P}}^{jj_{1}}(k_{12}) \label{estimate 45}  \\
& \times \sum_{j_{3}, j_{4} =1}^{3} [ \hat{\mathcal{P}}^{i_{2} j_{4}}(k_{2}) \hat{\mathcal{P}}^{j_{1} j_{4}}(k_{2}) \hat{\mathcal{P}}^{i_{1} j_{3}}(k_{1}) \hat{\mathcal{P}}^{j_{2}j_{3}}(k_{1}) \nonumber\\
&  - \hat{\mathcal{P}}^{i_{2} j_{4}}(k_{2}) \hat{\mathcal{P}}^{j_{2} j_{4}}(k_{2}) \hat{\mathcal{P}}^{i_{1} j_{3}}(k_{1}) \hat{\mathcal{P}}^{j_{1}j_{3}}(k_{1})] k_{12}^{i_{2}} k_{12}^{j_{2}} \nonumber\\
& \times g(\epsilon k_{12}^{i_{2}}) g(-\epsilon k_{12}^{j_{2}})\frac{1}{ [ \prod_{i=1}^{2} (2 \lvert k_{i} \rvert^{2} f(\epsilon k_{i}))][\lvert k_{12} \rvert^{2} f(\epsilon k_{12}) + \sum_{i=1}^{2} \lvert k_{i} \rvert^{2} f(\epsilon k_{i})]} \nonumber\\
&  \times [ \frac{e^{- \lvert k_{12} \rvert^{2} f(\epsilon k_{12}) 2t}}{\lvert k_{12} \rvert^{2} f(\epsilon k_{12})}  \nonumber\\
& \hspace{5mm} + 2 \int_{0}^{t} e^{- \lvert k_{12} \rvert^{2} f(\epsilon k_{12})2(t-s) - (\lvert k_{12} \rvert^{2} f(\epsilon k_{12}) + \sum_{i=1}^{2} \lvert k_{i}\rvert^{2} f(\epsilon k_{i})) s} ds ], \nonumber  \\
\bar{\phi}_{22}^{\epsilon, ij}(t) &\triangleq \frac{ (2\pi)^{-6}}{4} \sum_{k_{1}, k_{2} \neq 0} \sum_{i_{1}, i_{2}, j_{1}, j_{2} =1}^{3} Y_{k_{1}k_{2}} \hat{\mathcal{P}}^{ii_{1}}(k_{12}) \hat{\mathcal{P}}^{jj_{1}}(k_{12})  \label{estimate 46}\\
& \times \sum_{j_{3}, j_{4} =1}^{3} [ \hat{\mathcal{P}}^{i_{2} j_{4}}(k_{2}) \hat{\mathcal{P}}^{j_{1} j_{4}}(k_{2}) \hat{\mathcal{P}}^{i_{1} j_{3}}(k_{1}) \hat{\mathcal{P}}^{j_{2}j_{3}}(k_{1}) \nonumber\\
&  - \hat{\mathcal{P}}^{i_{2} j_{4}}(k_{2}) \hat{\mathcal{P}}^{j_{2} j_{4}}(k_{2}) \hat{\mathcal{P}}^{i_{1} j_{3}}(k_{1}) \hat{\mathcal{P}}^{j_{1}j_{3}}(k_{1})] k_{12}^{i_{2}} k_{12}^{j_{2}} ii \nonumber\\
& \times \frac{1}{ [ \prod_{i=1}^{2} (2 \lvert k_{i} \rvert^{2} )][\lvert k_{12} \rvert^{2} + \sum_{i=1}^{2} \lvert k_{i} \rvert^{2}]} \nonumber\\
&  \times [ \frac{e^{- \lvert k_{12} \rvert^{2} 2t}}{\lvert k_{12} \rvert^{2} } + 2 \int_{0}^{t} e^{- \lvert k_{12} \rvert^{2} 2(t-s) - (\lvert k_{12} \rvert^{2} + \sum_{i=1}^{2} \lvert k_{i}\rvert^{2}) s} ds ]. \nonumber  
\end{align}
\end{subequations} 
Then, for all $\rho > 0, \eta \in (0, 2\rho)$ and $i, j \in \{1,2,3\}$, from \eqref{estimate 45} and \eqref{estimate 46}
\begin{align}\label{estimate 54}
 \lvert \phi_{22}^{\epsilon, ij} - \bar{\phi}_{22}^{\epsilon, ij} \rvert  \lesssim \sum_{l=1}^{8} \delta \phi_{22}^{\epsilon, ij, l} 
\end{align}
where 
\begin{subequations}\label{estimate 219}
\begin{align}
\delta \phi_{22}^{\epsilon, ij, 1} \triangleq& \sum_{k_{1}, k_{2} \neq 0} \sum_{i_{2}, j_{2} =1}^{3} \lvert k_{12} \rvert^{2} \\
& \times  \lvert \frac{ [ g(\epsilon k_{12}^{i_{2}}) - i ] g(- \epsilon k_{12}^{j_{2}})}{ [\prod_{i=1}^{2} \lvert k_{i} \rvert^{2} f(\epsilon k_{i}) ] [ \lvert k_{12} \rvert^{2} f(\epsilon k_{12}) + \sum_{i=1}^{2} \lvert k_{i} \rvert^{2} f(\epsilon k_{i}) ]} \frac{ e^{- \lvert k_{12} \rvert^{2} f(\epsilon k_{12} ) 2t}}{\lvert k_{12} \rvert^{2} f(\epsilon k_{12})} \rvert, \nonumber\\
\delta \phi_{22}^{\epsilon, ij, 2} \triangleq&  \sum_{k_{1}, k_{2} \neq 0} \sum_{j_{2} =1}^{3} \lvert k_{12} \rvert^{2} \\
& \times  \lvert \frac{  i [g(- \epsilon k_{12}^{j_{2}}) - i]}{ [\prod_{i=1}^{2} \lvert k_{i} \rvert^{2} f(\epsilon k_{i}) ] [ \lvert k_{12} \rvert^{2} f(\epsilon k_{12}) + \sum_{i=1}^{2} \lvert k_{i} \rvert^{2} f(\epsilon k_{i}) ]} \frac{ e^{- \lvert k_{12} \rvert^{2} f(\epsilon k_{12} ) 2t}}{\lvert k_{12} \rvert^{2} f(\epsilon k_{12})} \rvert, \nonumber\\
 \delta \phi_{22}^{\epsilon, ij, 3} \triangleq& \sum_{k_{1}, k_{2} \neq 0} \lvert k_{12} \rvert^{2} \lvert \frac{1}{ [\prod_{i=1}^{2} \lvert k_{i} \rvert^{2} f(\epsilon k_{i})] [ \lvert k_{12} \rvert^{2} f(\epsilon k_{12}) + \sum_{i=1}^{2} \lvert k_{i} \rvert^{2} f(\epsilon k_{i}) ]} \\
 & \hspace{15mm}  - \frac{1}{ [\prod_{i=1}^{2} \lvert k_{i} \rvert^{2}] [\lvert k_{12} \rvert^{2} + \sum_{i=1}^{2} \lvert k_{i} \rvert^{2} ]} \rvert \frac{ e^{- \lvert k_{12} \rvert^{2} f(\epsilon k_{12}) 2t}}{\lvert k_{12}\rvert^{2} f(\epsilon k_{12})},    \nonumber\\
\delta \phi_{22}^{\epsilon, ij, 4} \triangleq& \sum_{k_{1}, k_{2}\neq 0} \frac{ \lvert k_{12} \rvert^{2}}{ [ \prod_{i=1}^{2} \lvert k_{i} \rvert^{2} ] [\lvert k_{12} \rvert^{2} + \sum_{i=1}^{2} \lvert k_{i}\rvert^{2} ]} \nonumber\\
& \hspace{25mm} \times  \lvert \frac{ e^{-2 \lvert k_{12} \rvert^{2} f(\epsilon k_{12})t}}{\lvert k_{12} \rvert^{2} f(\epsilon k_{12})} - \frac{e^{-2 \lvert k_{12} \rvert^{2} t}}{\lvert k_{12} \rvert^{2}} \rvert,  \\
\delta \phi_{22}^{\epsilon, ij, 5} \triangleq& \sum_{k_{1}, k_{2} \neq 0} \sum_{i_{2}, j_{2} =1}^{3} \lvert k_{12} \rvert^{2}  \nonumber\\
& \hspace{10mm} \times \lvert \frac{ [ g(\epsilon k_{12}^{i_{2}}) - i] g(-\epsilon k_{12}^{j_{2}})}{[\prod_{i=1}^{2} \lvert k_{i} \rvert^{2} f(\epsilon k_{i}) ][\lvert k_{12} \rvert^{2} f(\epsilon k_{12}) + \sum_{i=1}^{2} \lvert k_{i} \rvert^{2} f(\epsilon k_{i}) ]} \\
& \times \int_{0}^{t} e^{- \lvert k_{12} \rvert^{2} f(\epsilon k_{12}) 2(t-s) - (\lvert k_{12} \rvert^{2} f(\epsilon k_{12}) + \sum_{i=1}^{2} \lvert k_{i} \rvert^{2} f(\epsilon k_{i}) ) s } ds \rvert, \nonumber \\
\delta \phi_{22}^{\epsilon, ij, 6} \triangleq& \sum_{k_{1}, k_{2} \neq 0} \sum_{j_{2} =1}^{3} \lvert k_{12} \rvert^{2} \lvert \frac{ i [g(-\epsilon k_{12}^{j_{2}}) - i]}{[\prod_{i=1}^{2} \lvert k_{i} \rvert^{2} f(\epsilon k_{i}) ][\lvert k_{12} \rvert^{2} f(\epsilon k_{12}) + \sum_{i=1}^{2} \lvert k_{i} \rvert^{2} f(\epsilon k_{i}) ]} \nonumber\\
& \times \int_{0}^{t} e^{- \lvert k_{12} \rvert^{2} f(\epsilon k_{12}) 2(t-s) - (\lvert k_{12} \rvert^{2} f(\epsilon k_{12}) + \sum_{i=1}^{2} \lvert k_{i} \rvert^{2} f(\epsilon k_{i}) ) s } ds \rvert, \\ 
\delta \phi_{22}^{\epsilon, ij, 7} \triangleq& \sum_{k_{1}, k_{2} \neq 0} \lvert k_{12} \rvert^{2} \lvert \frac{1}{ [\prod_{i=1}^{2} \lvert k_{i} \rvert^{2} f(\epsilon k_{i})] [ \lvert k_{12} \rvert^{2} f(\epsilon k_{12}) + \lvert k_{1} \rvert^{2} f(\epsilon k_{1}) + \lvert k_{2} \rvert^{2} f(\epsilon k_{2})]} \nonumber\\
& \hspace{15mm} - \frac{1}{ [\prod_{i=1}^{2} \lvert k_{i} \rvert^{2} ] [ \lvert k_{12} \rvert^{2} + \sum_{i=1}^{2} \lvert k_{i} \rvert^{2} ]} \rvert \nonumber\\
& \times \int_{0}^{t} e^{-\lvert k_{12} \rvert^{2} f(\epsilon k_{12}) 2(t-s) - (\lvert k_{12} \rvert^{2} f(\epsilon k_{12}) + \sum_{i=1}^{2} \lvert k_{i} \rvert^{2} f(\epsilon k_{i}) ) s} ds, \\
\delta \phi_{22}^{\epsilon, ij, 8} \triangleq& \sum_{k_{1}, k_{2} \neq 0} \lvert k_{12} \rvert^{2} \frac{1}{ [\prod_{i=1}^{2} \lvert k_{i} \rvert^{2} ][\lvert k_{12} \rvert^{2} + \sum_{i=1}^{2} \lvert k_{i} \rvert^{2} ]} \\
& \times \int_{0}^{t} [e^{- \lvert k_{12} \rvert^{2} f(\epsilon k_{12} )(t-s) - \frac{1}{2} (\lvert k_{12} \rvert^{2} f(\epsilon k_{12}) + \sum_{i=1}^{2} \lvert k_{i} \rvert^{2} f(\epsilon k_{i})) s}  \nonumber\\
& \hspace{45mm} - e^{-\lvert k_{12} \rvert^{2} (t-s) - \frac{1}{2} (\lvert k_{12} \rvert^{2} + \sum_{i=1}^{2} \lvert k_{i} \rvert^{2} ) s}]  \nonumber\\
& \hspace{5mm} \times [e^{- \lvert k_{12} \rvert^{2} f(\epsilon k_{12} )(t-s) - \frac{1}{2} (\lvert k_{12} \rvert^{2} f(\epsilon k_{12}) + \sum_{i=1}^{2} \lvert k_{i} \rvert^{2} f(\epsilon k_{i})) s}  \nonumber\\
& \hspace{45mm} + e^{-\lvert k_{12} \rvert^{2} (t-s) - \frac{1}{2} (\lvert k_{12} \rvert^{2} + \sum_{i=1}^{2} \lvert k_{i} \rvert^{2} ) s}] ds.  \nonumber 
\end{align}
\end{subequations} 
We can estimate for any $\eta \in (0, 1), \rho \in (\frac{\eta}{2}, \frac{1}{2})$ and $\epsilon_{0} \in (0, 2 \rho - \eta)$ from \eqref{estimate 219}
\begin{align}\label{estimate 220}
\delta \phi_{22}^{\epsilon, ij, 1} \lesssim \sum_{k_{1}, k_{2} \neq 0} \frac{ \lvert k_{12} \rvert^{2} \lvert \epsilon k_{12} \rvert^{\eta}}{ [\prod_{i=1}^{2} \lvert k_{i} \rvert^{2} ][\lvert k_{12} \rvert^{2} + \sum_{i=1}^{2} \lvert k_{i} \rvert^{2} ]} \frac{1}{ \lvert k_{12} \rvert^{2} [ \lvert k_{12} \rvert^{2} t]^{\rho}}  \lesssim t^{-\rho} \epsilon^{\eta}
\end{align} 
by \eqref{[Equation (4.2)][ZZ17]}, \eqref{key estimate}, and Young's inequality. Identically we can deduce that $\delta \phi_{22}^{\epsilon, ij, 2} \lesssim t^{-\rho} \epsilon^{\eta}$. Next, from \eqref{estimate 219}, 
\begin{align}
\delta \phi_{22}^{\epsilon, ij, 3} \lesssim \epsilon^{\eta} \sum_{k_{1},k_{2} \neq 0} \frac{1}{ [ \prod_{i=1}^{2} \lvert k_{i} \rvert^{2} ][\lvert k_{12} \rvert + \sum_{i=1}^{2} \lvert k_{i} \rvert ]^{2-\eta}} \frac{1}{ [\lvert k_{12} \rvert^{2} t]^{\rho}}  \lesssim \epsilon^{\eta} t^{-\rho} 
\end{align} 
by \eqref{key estimate}. Next, from \eqref{estimate 219}
\begin{align}
\delta \phi_{22}^{\epsilon, ij, 4} \lesssim \sum_{k_{1}, k_{2}\neq 0} \frac{1}{ [ \prod_{i=1}^{2} \lvert k_{i} \rvert^{2} ] [\lvert k_{12} \rvert^{2} + \sum_{i=1}^{2} \lvert k_{i} \rvert^{2} ]} e^{-2 \lvert k_{12} \rvert^{2} \bar{c}_{f} t} \lvert \epsilon k_{12} \rvert^{\eta} \lesssim \epsilon^{\eta} t^{-\rho} 
\end{align}
by \eqref{[Equation (4.2)][ZZ17]} and \eqref{key estimate}. Next, from \eqref{estimate 219}
\begin{align}
\delta \phi_{22}^{\epsilon, ij, 5} \lesssim& \epsilon^{\eta} t^{-\rho} \sum_{k_{1}, k_{2} \neq 0} \frac{ \lvert k_{12} \rvert^{2+ \eta - 2 \rho}}{[\prod_{i=1}^{2} \lvert k_{i} \rvert^{2} ] \lvert k_{1} \rvert^{1+ \epsilon_{0}}\lvert k_{2} \rvert^{1-\epsilon_{0}} }  \nonumber\\
& \hspace{30mm} \times [ \frac{1- e^{-\lvert k_{12} \rvert^{2} f(\epsilon k_{12} ) t}}{\lvert k_{12} \rvert^{2} f(\epsilon k_{12})} ]  \lesssim \epsilon^{\eta} t^{-\rho} 
\end{align}
by \eqref{[Equation (4.3)][ZZ17]} and \eqref{key estimate}. Similarly, we can deduce $\delta \phi_{22}^{\epsilon, ij, 6} \lesssim \epsilon^{\eta} t^{-\rho}$. Next, from \eqref{estimate 219}, 
\begin{align} 
\delta \phi_{22}^{\epsilon, ij, 7} \lesssim& t^{-\rho} \sum_{k_{1}, k_{2} \neq 0} \frac{1}{ [ \prod_{i=1}^{2} \lvert k_{i} \rvert^{2} ] \lvert k_{12} \rvert^{2\rho}} \frac{1}{ [ \lvert k_{12} \rvert^{2} + \lvert k_{1} \rvert^{2} + \lvert k_{2} \rvert^{2} ]^{2}} \nonumber\\
& \hspace{10mm} \times [ \lvert k_{12} \rvert^{2} \lvert \epsilon k_{12} \rvert^{\eta} + \lvert k_{1} \rvert^{2} \lvert \epsilon k_{1} \rvert^{\eta} + \lvert k_{2} \rvert^{2} \lvert \epsilon k_{2} \rvert^{\eta} ]  \lesssim \epsilon^{\eta} t^{-\rho} 
\end{align}
by \eqref{key estimate}. Finally, from \eqref{estimate 219} we carefully estimate as follows: 
\begin{align}\label{estimate 55}
\delta \phi_{22}^{\epsilon, ij, 8} &\lesssim  \epsilon^{\eta} t^{-\rho} \sum_{k_{1}, k_{2} \neq 0} \frac{ \lvert k_{12} \rvert^{2}}{[\prod_{i=1}^{2} \lvert k_{i} \rvert^{2} ] [ \lvert k_{12} \rvert^{2} + \sum_{i=1}^{2} \lvert k_{i} \rvert^{2} ]} \\
& \times [ \frac{\lvert k_{12} \rvert^{\eta}}{\lvert k_{12} \rvert^{2\rho}} \frac{ (1- e^{-\frac{1}{2} \sum_{i=1}^{2} \lvert k_{i} \rvert^{2} \bar{c}_{f} t})}{\sum_{i=1}^{2} \lvert k_{i}\rvert^{2}} \nonumber\\
&+ \frac{\lvert k_{1} \rvert^{\eta}}{\lvert k_{12} \rvert^{2\rho}} \frac{ (1- e^{-\frac{1}{2} (\sum_{i=1}^{2} \lvert k_{i} \rvert^{2})\bar{c}_{f} t})}{ \sum_{i=1}^{2} \lvert k_{i} \rvert^{2}} + \frac{\lvert k_{2} \rvert^{\eta}}{\lvert k_{12} \rvert^{2\rho}} \frac{(1- e^{- \frac{1}{2} ( \sum_{i=1}^{2} \lvert k_{i} \rvert^{2} ) \bar{c}_{f} t})}{\sum_{i=1}^{2} \lvert k_{i} \rvert^{2}} ]\lesssim  \epsilon^{\eta} t^{-\rho} \nonumber
\end{align}  
by \eqref{[Equation (4.2)][ZZ17]} and \eqref{key estimate}. Therefore, applying \eqref{estimate 220} to \eqref{estimate 55} to \eqref{estimate 54} allows us to conclude that $\phi_{22}^{\epsilon, ij} - \bar{\phi}_{22}^{\epsilon, ij}$ converges to zero with respect to the norm $\lVert \cdot \rVert \triangleq \sup_{t\in [0,T]} t^{\rho} \lvert \cdot (t) \rvert$. 

\emph{Term in the second chaos: $VII_{t, ij}^{2}$ in \eqref{estimate 217}}\\

Among $VII_{t,ij}^{2} = \sum_{l=1}^{4} VII_{t,ij}^{2l}$, we show the necessary estimate only on $VII_{t,ij}^{21}$ as those on others are similar. 
\begin{remark}\label{Remark 2.6}
We strategically split $VII_{t,ij}^{21}$, which consists of eight terms, to four terms matching the first with fifth, second with sixth, third with seventh and fourth with eighth as follows: 
\begin{equation}\label{estimate 287}
VII_{t, ij}^{21} = \sum_{l=1}^{4} VII_{t,ij}^{21l} 
\end{equation} 
where 
\begin{subequations}\label{estimate 283}
\begin{align}   
&VII_{t,ij}^{211} \triangleq  \frac{ (2\pi)^{-\frac{9}{2}}}{4} \sum_{k} \sum_{k_{1}, k_{2}, k_{4}\neq 0: k_{24} = k} \sum_{i_{1}, i_{2}, j_{1}, j_{2} = 1}^{3} \int_{[0,t]^{2}} \nonumber\\
& \hspace{2mm} \times [ e^{- \lvert k_{12} \rvert^{2} f(\epsilon k_{12} )(t-s) - \lvert k_{4} - k_{1} \rvert^{2} f(\epsilon (k_{4} - k_{1})) (t- \bar{s})} k_{12}^{i_{2}} (k_{4}^{j_{2}} - k_{1}^{j_{2}}) g(\epsilon k_{12}^{i_{2}}) g(\epsilon (k_{4}^{j_{2}} - k_{1}^{j_{2}})) \nonumber\\
& \hspace{7mm} \times : \hat{X}_{s,u}^{\epsilon, i_{2}}(k_{2}) \hat{X}_{\bar{s}, u}^{\epsilon, j_{2}}(k_{4}): \sum_{j_{3} =1}^{3} \frac{ e^{-\lvert k_{1} \rvert^{2} f(\epsilon k_{1}) \lvert s- \bar{s} \rvert} h_{b}(\epsilon k_{1})^{2}}{2 \lvert k_{1}\rvert^{2} f(\epsilon k_{1})} \hat{\mathcal{P}}^{i_{1}j_{3}}(k_{1}) \hat{\mathcal{P}}^{j_{1} j_{3}}(k_{1}) \nonumber\\
& \hspace{2mm} - e^{- \lvert k_{12} \rvert^{2} (t-s) - \lvert k_{4} - k_{1} \rvert^{2} (t- \bar{s})} k_{12}^{i_{2}} (k_{4}^{j_{2}} - k_{1}^{j_{2}}) ii \nonumber\\
& \hspace{7mm} \times : \hat{\bar{X}}_{s,u}^{\epsilon, i_{2}}(k_{2}) \hat{\bar{X}}_{\bar{s}, u}^{\epsilon, j_{2}}(k_{4}): \sum_{j_{3} =1}^{3} \frac{ e^{-\lvert k_{1} \rvert^{2} \lvert s- \bar{s} \rvert} h_{b}(\epsilon k_{1})^{2}}{2 \lvert k_{1}\rvert^{2}} \hat{\mathcal{P}}^{i_{1}j_{3}}(k_{1}) \hat{\mathcal{P}}^{j_{1} j_{3}}(k_{1})  ] ds d \bar{s}\nonumber\\
& \hspace{2mm} \times \hat{\mathcal{P}}^{ii_{1}} (k_{12}) \hat{\mathcal{P}}^{jj_{1}}(k_{4} - k_{1}) e_{k}, \\
&VII_{t, ij}^{212} \triangleq  \frac{ (2\pi)^{-\frac{9}{2}}}{4} \sum_{k} \sum_{k_{1}, k_{2}, k_{4} \neq 0: k_{24} = k} \sum_{i_{1}, i_{2}, j_{1}, j_{2} = 1}^{3} \int_{[0,t]^{2}} \nonumber\\
& \hspace{2mm} \times [ -e^{- \lvert k_{12} \rvert^{2} f(\epsilon k_{12} )(t-s) - \lvert k_{4} - k_{1} \rvert^{2} f(\epsilon (k_{4} - k_{1})) (t- \bar{s})} k_{12}^{i_{2}} (k_{4}^{j_{2}} - k_{1}^{j_{2}}) g(\epsilon k_{12}^{i_{2}}) g(\epsilon (k_{4}^{j_{2}} - k_{1}^{j_{2}})) \nonumber\\
& \hspace{7mm} \times : \hat{X}_{s,u}^{\epsilon, i_{2}}(k_{2}) \hat{X}_{\bar{s}, b}^{\epsilon, j_{2}}(k_{4}): \sum_{j_{3} =1}^{3} \frac{ e^{-\lvert k_{1} \rvert^{2} f(\epsilon k_{1}) \lvert s- \bar{s} \rvert} h_{b}(\epsilon k_{1})h_{u}(\epsilon k_{1})}{2 \lvert k_{1}\rvert^{2} f(\epsilon k_{1})} \hat{\mathcal{P}}^{i_{1}j_{3}}(k_{1}) \hat{\mathcal{P}}^{j_{1} j_{3}}(k_{1}) \nonumber\\
& \hspace{2mm} + e^{- \lvert k_{12} \rvert^{2} (t-s) - \lvert k_{4} - k_{1} \rvert^{2} (t- \bar{s})} k_{12}^{i_{2}} (k_{4}^{j_{2}} - k_{1}^{j_{2}}) ii \nonumber\\
& \hspace{7mm} \times : \hat{\bar{X}}_{s,u}^{\epsilon, i_{2}}(k_{2}) \hat{\bar{X}}_{\bar{s}, b}^{\epsilon, j_{2}}(k_{4}): \sum_{j_{3} =1}^{3} \frac{ e^{-\lvert k_{1} \rvert^{2} \lvert s- \bar{s} \rvert} h_{b}(\epsilon k_{1})h_{u}(\epsilon k_{1})}{2 \lvert k_{1}\rvert^{2}} \hat{\mathcal{P}}^{i_{1}j_{3}}(k_{1}) \hat{\mathcal{P}}^{j_{1} j_{3}}(k_{1})  ] ds d \bar{s}\nonumber\\
& \hspace{2mm} \times \hat{\mathcal{P}}^{ii_{1}} (k_{12}) \hat{\mathcal{P}}^{jj_{1}}(k_{4} - k_{1}) e_{k}, \\
&VII_{t, ij}^{213} \triangleq  \frac{ (2\pi)^{-\frac{9}{2}}}{4} \sum_{k} \sum_{k_{1}, k_{2}, k_{4} \neq 0: k_{24} = k} \sum_{i_{1}, i_{2}, j_{1}, j_{2} = 1}^{3} \int_{[0,t]^{2}} \nonumber\\
& \hspace{2mm} \times [ -e^{- \lvert k_{12} \rvert^{2} f(\epsilon k_{12} )(t-s) - \lvert k_{4} - k_{1} \rvert^{2} f(\epsilon (k_{4} - k_{1})) (t- \bar{s})} k_{12}^{i_{2}} (k_{4}^{j_{2}} - k_{1}^{j_{2}}) g(\epsilon k_{12}^{i_{2}}) g(\epsilon (k_{4}^{j_{2}} - k_{1}^{j_{2}})) \nonumber\\
& \hspace{7mm} \times : \hat{X}_{s,b}^{\epsilon, i_{2}}(k_{2}) \hat{X}_{\bar{s}, u}^{\epsilon, j_{2}}(k_{4}): \sum_{j_{3} =1}^{3} \frac{ e^{-\lvert k_{1} \rvert^{2} f(\epsilon k_{1}) \lvert s- \bar{s} \rvert} h_{u}(\epsilon k_{1}) h_{b}(\epsilon k_{1})}{2 \lvert k_{1}\rvert^{2} f(\epsilon k_{1})} \hat{\mathcal{P}}^{i_{1}j_{3}}(k_{1}) \hat{\mathcal{P}}^{j_{1} j_{3}}(k_{1}) \nonumber\\
& \hspace{2mm} + e^{- \lvert k_{12} \rvert^{2} (t-s) - \lvert k_{4} - k_{1} \rvert^{2} (t- \bar{s})} k_{12}^{i_{2}} (k_{4}^{j_{2}} - k_{1}^{j_{2}}) ii \nonumber\\
& \hspace{7mm} \times : \hat{\bar{X}}_{s,b}^{\epsilon, i_{2}}(k_{2}) \hat{\bar{X}}_{\bar{s}, u}^{\epsilon, j_{2}}(k_{4}): \sum_{j_{3} =1}^{3} \frac{ e^{-\lvert k_{1} \rvert^{2} \lvert s- \bar{s} \rvert} h_{u}(\epsilon k_{1}) h_{b}(\epsilon k_{1})}{2 \lvert k_{1}\rvert^{2}} \hat{\mathcal{P}}^{i_{1}j_{3}}(k_{1}) \hat{\mathcal{P}}^{j_{1} j_{3}}(k_{1})  ] ds d \bar{s}\nonumber\\
& \hspace{2mm} \times \hat{\mathcal{P}}^{ii_{1}} (k_{12}) \hat{\mathcal{P}}^{jj_{1}}(k_{4} - k_{1}) e_{k}, \\
&VII_{t, ij}^{214} \triangleq \frac{ (2\pi)^{-\frac{9}{2}}}{4} \sum_{k} \sum_{k_{1}, k_{2}, k_{4}\neq 0: k_{24} = k} \sum_{i_{1}, i_{2}, j_{1}, j_{2} = 1}^{3} \int_{[0,t]^{2}} \nonumber\\
&  \hspace{2mm}\times [ e^{- \lvert k_{12} \rvert^{2} f(\epsilon k_{12} )(t-s) - \lvert k_{4} - k_{1} \rvert^{2} f(\epsilon (k_{4} - k_{1})) (t- \bar{s})} k_{12}^{i_{2}} (k_{4}^{j_{2}} - k_{1}^{j_{2}}) g(\epsilon k_{12}^{i_{2}}) g(\epsilon (k_{4}^{j_{2}} - k_{1}^{j_{2}})) \nonumber\\
& \hspace{7mm} \times : \hat{X}_{s,b}^{\epsilon, i_{2}}(k_{2}) \hat{X}_{\bar{s}, b}^{\epsilon, j_{2}}(k_{4}): \sum_{j_{3} =1}^{3} \frac{ e^{-\lvert k_{1} \rvert^{2} f(\epsilon k_{1}) \lvert s- \bar{s} \rvert} h_{u}(\epsilon k_{1})^{2}}{2 \lvert k_{1}\rvert^{2} f(\epsilon k_{1})} \hat{\mathcal{P}}^{i_{1}j_{3}}(k_{1}) \hat{\mathcal{P}}^{j_{1} j_{3}}(k_{1}) \nonumber\\
& \hspace{2mm} - e^{- \lvert k_{12} \rvert^{2} (t-s) - \lvert k_{4} - k_{1} \rvert^{2} (t- \bar{s})} k_{12}^{i_{2}} (k_{4}^{j_{2}} - k_{1}^{j_{2}}) ii \nonumber\\
& \hspace{7mm} \times : \hat{\bar{X}}_{s,b}^{\epsilon, i_{2}}(k_{2}) \hat{\bar{X}}_{\bar{s}, b}^{\epsilon, j_{2}}(k_{4}): \sum_{j_{3} =1}^{3} \frac{ e^{-\lvert k_{1} \rvert^{2} \lvert s- \bar{s} \rvert} h_{u}(\epsilon k_{1})^{2}}{2 \lvert k_{1}\rvert^{2}} \hat{\mathcal{P}}^{i_{1}j_{3}}(k_{1}) \hat{\mathcal{P}}^{j_{1} j_{3}}(k_{1})  ] ds d \bar{s}\nonumber\\
&  \hspace{2mm} \times \hat{\mathcal{P}}^{ii_{1}} (k_{12}) \hat{\mathcal{P}}^{jj_{1}}(k_{4} - k_{1}) e_{k}.
\end{align}
\end{subequations}
We emphasize that this is a difficulty that does not exist for the NS equations. Such a careful splitting will be crucial in \eqref{estimate 221} for $VII_{t,ij}^{1}$ again.  
\end{remark}
W.l.o.g. we show the necessary estimate only for $VII_{t, ij}^{214}$ as those for others are similar. First, we see that 
\begin{align*}
\mathbb{E} [ : \hat{X}_{s,b}^{\epsilon, i_{2}}(k_{2}) \hat{X}_{\bar{s}, b}^{\epsilon, j_{2}}(k_{4}): \overline{ : \hat{X}_{\sigma,b}^{\epsilon, i_{2}'}(k_{2}') \hat{X}_{\bar{\sigma}, b}^{\epsilon, j_{2}'}(k_{4}'):}]  \lesssim (1_{k_{2} = k_{2}', k_{4} = k_{4}'} + 1_{k_{2} = k_{4}', k_{4} = k_{2}'}) \frac{1}{\lvert k_{2} \rvert^{2} \lvert k_{4} \rvert^{2}} 
\end{align*} 
by Example \ref{Example 3.1} and \eqref{covariance} while for any $\eta \in [0,1]$,  
\begin{align*}
&\mathbb{E} [ (: \hat{X}_{s,b}^{\epsilon, i_{2}}(k_{2}) \hat{X}_{\bar{s}, b}^{\epsilon, j_{2}}(k_{4}): - : \hat{\bar{X}}_{s,b}^{\epsilon, i_{2}}(k_{2}) \hat{\bar{X}}_{\bar{s}, b}^{\epsilon, j_{2}}(k_{4}):) \\
& \hspace{10mm} \times \overline{ (: \hat{X}_{\sigma,b}^{\epsilon, i_{2}'}(k_{2}') \hat{X}_{\bar{\sigma}, b}^{\epsilon, j_{2}'}(k_{4}'): - : \hat{\bar{X}}_{\sigma,b}^{\epsilon, i_{2}'}(k_{2}') \hat{\bar{X}}_{\bar{\sigma}, b}^{\epsilon, j_{2}'}(k_{4}'):) } ] \\
\lesssim&  (1_{k_{2} = k_{2}', k_{4} = k_{4}'} + 1_{k_{2} = k_{4}', k_{4} = k_{2}'})  \frac{( \lvert \epsilon k_{2} \rvert^{\eta} + \lvert \epsilon k_{4} \rvert^{\eta})}{\lvert k_{2} \rvert^{2} \lvert k_{4} \rvert^{2}}
\end{align*}
by H$\ddot{\mathrm{o}}$lder's inequality and \eqref{[Equation (4.7)][ZZ17]}. Applying these estimates to \eqref{estimate 283} leads to 
\begin{equation}\label{estimate 286}
 \mathbb{E} [ \lvert \Delta_{q} VII_{t, ij}^{214} \rvert^{2}] \lesssim \sum_{l=1}^{2} VII_{t, ij}^{214l}
\end{equation} 
where 
\begin{subequations}\label{estimate 284}
\begin{align}
VII_{t, ij}^{2141} \triangleq& \sum_{k} \theta(2^{-q} k)^{2} \sum_{k_{1}, k_{2}, k_{4}, k_{1}', k_{2}', k_{4}' \neq 0: k_{24} = k_{24} ' = k} \sum_{i_{2}, j_{2}, i_{2}', j_{2}' =1}^{3} \nonumber\\
& \times  (1_{k_{2} = k_{2}', k_{4} = k_{4}'} + 1_{k_{2} = k_{4}', k_{4} = k_{2}'}) \frac{ \lvert k_{12} \rvert \lvert k_{12}' \rvert \lvert k_{4} - k_{1} \rvert \lvert k_{4}' -k_{1}' \rvert}{ \lvert k_{1} \rvert^{2} \lvert k_{1} ' \rvert^{2} \lvert k_{2} \rvert^{2} \lvert k_{4} \rvert^{2}}  \nonumber\\
& \times \int_{[0,t]^{4}}  \lvert e^{- \lvert k_{12} \rvert^{2} f(\epsilon k_{12})(t-s) - \lvert k_{4} - k_{1} \rvert^{2} f(\epsilon (k_{4} - k_{1}))(t- \bar{s})} \nonumber\\
& \hspace{35mm} \times g(\epsilon k_{12}^{i_{2}}) g(\epsilon (k_{4} - k_{1})^{j_{2}}) \frac{ e^{- \lvert k_{1} \rvert^{2} f(\epsilon k_{1}) \lvert s- \bar{s} \rvert}}{f( \epsilon k_{1})} \nonumber\\
& \hspace{15mm} - e^{- \lvert k_{12} \rvert^{2} (t-s) - \lvert k_{4} - k_{1} \rvert^{2} (t- \bar{s})} ii e^{- \lvert k_{1} \rvert^{2} \lvert s- \bar{s} \rvert} \rvert \nonumber\\ 
& \times \lvert e^{- \lvert k_{12}' \rvert^{2} f(\epsilon k_{12}')(t-\sigma) - \lvert k_{4}' - k_{1}' \rvert^{2} f(\epsilon (k_{4}' - k_{1}'))(t- \bar{\sigma})}  \nonumber\\
& \hspace{35mm} \times g(\epsilon (k_{12}')^{i_{2}'}) g(\epsilon (k_{4}' - k_{1}')^{j_{2}'}) \frac{ e^{- \lvert k_{1}' \rvert^{2} f(\epsilon k_{1}') \lvert \sigma- \bar{\sigma} \rvert}}{f( \epsilon k_{1}')} \nonumber\\
& \hspace{15mm} - e^{- \lvert k_{12}' \rvert^{2} (t-\sigma) - \lvert k_{4}' - k_{1}' \rvert^{2} (t- \bar{\sigma})} ii e^{- \lvert k_{1}' \rvert^{2} \lvert \sigma- \bar{\sigma} \rvert} \rvert d \sigma d \bar{\sigma} ds d \bar{s}, \\
VII_{t, ij}^{2142} \triangleq& \epsilon^{\eta} \sum_{k} \theta(2^{-q} k)^{2}  \sum_{k_{1}, k_{2}, k_{4}, k_{1}', k_{2}', k_{4}' \neq 0: k_{24} = k_{24} ' = k} (1_{k_{2} = k_{2}', k_{4} = k_{4}'} + 1_{k_{2} = k_{4}', k_{4} = k_{2}'}) \nonumber\\
& \times \int_{[0,t]^{4}} \frac{ \lvert k_{12} \rvert \lvert k_{12} ' \rvert \lvert k_{4} - k_{1} \rvert \lvert k_{4}' - k_{1}' \rvert ( \lvert k_{2} \rvert^{\eta} + \lvert k_{4} \rvert^{\eta})}{ \lvert k_{1} \rvert^{2} \lvert k_{1}' \rvert^{2} \lvert k_{2} \rvert^{2} \lvert k_{4} \rvert^{2} } \nonumber\\
& \times e^{- \lvert k_{12} \rvert^{2} (t-s) - \lvert k_{4} - k_{1} \rvert^{2} (t- \bar{s}) } e^{- \lvert k_{12}'\rvert^{2} (t-\sigma) - \lvert k_{4}' - k_{1}' \rvert^{2} (t-\bar{\sigma})}  d \sigma d \bar{\sigma} ds d \bar{s}.  
\end{align}
\end{subequations} 
Within $VII_{t, ij}^{2141}$, we can estimate for any $\eta \in [0,1]$, 
\begin{align*}
&  \lvert e^{- \lvert k_{12} \rvert^{2} f(\epsilon k_{12})(t-s) - \lvert k_{4} - k_{1} \rvert^{2} f(\epsilon (k_{4} - k_{1}))(t- \bar{s})} g(\epsilon k_{12}^{i_{2}}) g(\epsilon (k_{4} - k_{1})^{j_{2}}) \frac{ e^{- \lvert k_{1} \rvert^{2} f(\epsilon k_{1}) \lvert s- \bar{s} \rvert}}{f( \epsilon k_{1})} \nonumber\\
& \hspace{10mm} - e^{- \lvert k_{12} \rvert^{2} (t-s) - \lvert k_{4} - k_{1} \rvert^{2} (t- \bar{s})} ii e^{- \lvert k_{1} \rvert^{2} \lvert s- \bar{s} \rvert} \rvert \nonumber\\  
 \lesssim& \epsilon^{\frac{\eta}{2}} \frac{1}{ [ \lvert k_{12} \rvert^{2} (t-s)]^{1- \frac{\epsilon}{4}}} \frac{1}{ [ \lvert k_{4} - k_{1} \rvert^{2} ( t- \bar{s})]^{1- \frac{\epsilon}{4}}} ( \lvert k_{12} \rvert^{\frac{\eta}{2}} + \lvert k_{4} - k_{1} \rvert^{\frac{\eta}{2}} + \lvert k_{1} \rvert^{\frac{\eta}{2}}) 
\end{align*}
by \eqref{[Equation (4.2)][ZZ17]}, \eqref{[Equation (4.3)][ZZ17]} and \eqref{key estimate} so that \eqref{estimate 284} leads to  
\begin{align}\label{estimate 285}
VII_{t, ij}^{2141} &\lesssim \epsilon^{\eta} t^{\epsilon} \sum_{k} \theta(2^{-q} k)^{2} \sum_{k_{1}, k_{2}, k_{3}, k_{4} \neq 0: k_{24} = k} \frac{1}{ \prod_{i=1}^{4} \lvert k_{i} \rvert^{2}} \\
& \times [ \frac{( \lvert k_{12} \rvert^{\eta} + \lvert k_{4} - k_{1} \rvert^{\eta} + \lvert k_{1} \rvert^{\eta} + \lvert k_{23} \rvert^{\eta} + \lvert k_{4} - k_{3} \rvert^{\eta} + \lvert k_{3} \rvert^{\eta}}{ \lvert k_{12} \rvert^{1- \frac{\epsilon}{2}} \lvert k_{4} - k_{1} \rvert^{1- \frac{\epsilon}{2}} \lvert k_{123} \rvert^{1- \frac{\epsilon}{2}} \lvert k_{4} - k_{3} \rvert^{1-\frac{\epsilon}{2}}} \nonumber\\
&+ \frac{( \lvert k_{12} \rvert^{\eta} + \lvert k_{4} - k_{1} \rvert^{\eta} + \lvert k_{1} \rvert^{\eta} + \lvert k_{34} \rvert^{\eta} + \lvert k_{2} - k_{3} \rvert^{\eta} + \lvert k_{3} \rvert^{\eta}}{ \lvert k_{12} \rvert^{1- \frac{\epsilon}{2}} \lvert k_{4} - k_{1} \rvert^{1- \frac{\epsilon}{2}} \lvert k_{34} \rvert^{1- \frac{\epsilon}{2}} \lvert k_{2} - k_{3} \rvert^{1-\frac{\epsilon}{2}}}]  \lesssim \epsilon^{\eta} t^{\epsilon} 2^{q(2\epsilon + \eta)} \nonumber
\end{align}
by H$\ddot{\mathrm{o}}$lder's inequality and Lemma \ref{Lemma 3.13}. We also estimate $VII_{t, ij}^{2142}$ from \eqref{estimate 284} as follows: 
\begin{align}\label{estimate 107}
&VII_{t, ij}^{2142} \\
&\lesssim  \epsilon^{\eta} t^{\epsilon} \sum_{k} \theta(2^{-q} k)^{2} \sum_{k_{1}, k_{2}, k_{4}, k_{1}', k_{2}', k_{4}' \neq 0: k_{24} = k_{24} ' = k} (1_{k_{2} = k_{2}', k_{4} = k_{4}'} + 1_{k_{2} = k_{4}', k_{4} = k_{2}'}) \nonumber\\
& \hspace{3mm} \times \frac{ \lvert k_{2}'\rvert^{\eta} + \lvert k_{4}\rvert^{\eta}}{ \lvert k_{1} \rvert^{2} \lvert k_{1}' \rvert^{2} \lvert k_{2} \rvert^{2} \lvert k_{4} \rvert^{2}} \frac{1}{ \lvert k_{12} \rvert^{1- \frac{\epsilon}{2}} \lvert k_{4} - k_{1} \rvert^{1- \frac{\epsilon}{2}} \lvert k_{12}' \rvert^{1- \frac{\epsilon}{2}} \lvert k_{4}' - k_{1}' \rvert^{1-\frac{\epsilon}{2}}} \lesssim \epsilon^{\eta} t^{\epsilon} 2^{q(\eta + 2\epsilon)} \nonumber
\end{align}
by \eqref{key estimate}, H$\ddot{\mathrm{o}}$lder's inequality and Lemma \ref{Lemma 3.13}. Applying \eqref{estimate 285} and \eqref{estimate 107} to \eqref{estimate 286} gives us 
\begin{equation*}
 \mathbb{E} [ \lvert \Delta_{q} VII_{t, ij}^{214} \rvert^{2}] \lesssim \epsilon^{\eta} t^{\epsilon} 2^{q(2\epsilon + \eta)}.
\end{equation*} 

\emph{Term in the fourth chaos: $VII_{t, ij}^{1}$ in \eqref{[Equation (4.8r)][ZZ17]}}\\

Similarly to \eqref{estimate 287}, we strategically split $VII_{t, ij}^{11}$ to four terms matching the first with the fifth, the second with the sixth, the third with the seventh and the fourth with the eighth so that $VII_{t, ij}^{1} = \sum_{l=1}^{4} VII_{t, ij}^{1l}$ as follows: 
\begin{subequations}\label{estimate 221}
\begin{align}
VII_{t, ij}^{11} \triangleq& \frac{ (2\pi)^{-\frac{9}{2}}}{4} \sum_{k} \sum_{k_{1}, k_{2}, k_{3}, k_{4} \neq 0: k_{1234} = k} \sum_{i_{1}, i_{2}, j_{1}, j_{2} = 1}^{3} \int_{[0,t]^{2}}  \\
& \times [ e^{- \lvert k_{12} \rvert^{2} f(\epsilon k_{12}) (t-s) - \lvert k_{34} \rvert^{2} f(\epsilon k_{34}) (t- \bar{s})} k_{12}^{i_{2}} k_{34}^{j_{2}} g(\epsilon k_{12}^{i_{2}}) g(\epsilon k_{34}^{j_{2}}) \nonumber\\
& \hspace{5mm} \times : \hat{X}_{s,b}^{\epsilon, i_{1}}(k_{1}) \hat{X}_{s,u}^{\epsilon, i_{2}}(k_{2}) \hat{X}_{\bar{s}, b}^{\epsilon, j_{1}}(k_{3}) \hat{X}_{\bar{s}, u}^{\epsilon, j_{2}}(k_{4}):  \nonumber\\
& - e^{- \lvert k_{12} \rvert^{2} (t-s) - \lvert k_{34} \rvert^{2} (t- \bar{s})} k_{12}^{i_{2}} k_{34}^{j_{2}} ii \nonumber\\
& \hspace{5mm} \times : \hat{\bar{X}}_{s,b}^{\epsilon, i_{1}}(k_{1}) \hat{\bar{X}}_{s,u}^{\epsilon, i_{2}}(k_{2}) \hat{\bar{X}}_{\bar{s}, b}^{\epsilon, j_{1}}(k_{3}) \hat{\bar{X}}_{\bar{s}, u}^{\epsilon, j_{2}}(k_{4}): ]  \hat{\mathcal{P}}^{ii_{1}}(k_{12}) \hat{\mathcal{P}}^{jj_{1}}(k_{34}) e_{k} ds d\bar{s}, \nonumber \\
 VII_{t, ij}^{12} \triangleq& \frac{ (2\pi)^{-\frac{9}{2}}}{4} \sum_{k} \sum_{k_{1}, k_{2}, k_{3}, k_{4} \neq 0: k_{1234} = k} \sum_{i_{1}, i_{2}, j_{1}, j_{2} = 1}^{3} \int_{[0,t]^{2}}  \\
& \times [ -e^{- \lvert k_{12} \rvert^{2} f(\epsilon k_{12}) (t-s) - \lvert k_{34} \rvert^{2} f(\epsilon k_{34}) (t- \bar{s})} k_{12}^{i_{2}} k_{34}^{j_{2}} g(\epsilon k_{12}^{i_{2}}) g(\epsilon k_{34}^{j_{2}}) \nonumber\\
& \hspace{5mm} \times : \hat{X}_{s,b}^{\epsilon, i_{1}}(k_{1}) \hat{X}_{s,u}^{\epsilon, i_{2}}(k_{2}) \hat{X}_{\bar{s}, u}^{\epsilon, j_{1}}(k_{3}) \hat{X}_{\bar{s}, b}^{\epsilon, j_{2}}(k_{4}): \nonumber\\
& + e^{- \lvert k_{12} \rvert^{2} (t-s) - \lvert k_{34} \rvert^{2} (t- \bar{s})} k_{12}^{i_{2}} k_{34}^{j_{2}} ii \nonumber\\
& \hspace{5mm} \times : \hat{\bar{X}}_{s,b}^{\epsilon, i_{1}}(k_{1}) \hat{\bar{X}}_{s,u}^{\epsilon, i_{2}}(k_{2}) \hat{\bar{X}}_{\bar{s}, u}^{\epsilon, j_{1}}(k_{3}) \hat{\bar{X}}_{\bar{s}, b}^{\epsilon, j_{2}}(k_{4}):]  \hat{\mathcal{P}}^{ii_{1}}(k_{12}) \hat{\mathcal{P}}^{jj_{1}}(k_{34}) e_{k} ds d\bar{s}, \nonumber \\
 VII_{t, ij}^{12} \triangleq& \frac{ (2\pi)^{-\frac{9}{2}}}{4} \sum_{k} \sum_{k_{1}, k_{2}, k_{3}, k_{4} \neq 0: k_{1234} = k} \sum_{i_{1}, i_{2}, j_{1}, j_{2} = 1}^{3} \int_{[0,t]^{2}}  \\
& \times [ -e^{- \lvert k_{12} \rvert^{2} f(\epsilon k_{12}) (t-s) - \lvert k_{34} \rvert^{2} f(\epsilon k_{34}) (t- \bar{s})} k_{12}^{i_{2}} k_{34}^{j_{2}} g(\epsilon k_{12}^{i_{2}}) g(\epsilon k_{34}^{j_{2}}) \nonumber\\
& \hspace{5mm} \times : \hat{X}_{s,u}^{\epsilon, i_{1}}(k_{1}) \hat{X}_{s,b}^{\epsilon, i_{2}}(k_{2}) \hat{X}_{\bar{s}, b}^{\epsilon, j_{1}}(k_{3}) \hat{X}_{\bar{s}, u}^{\epsilon, j_{2}}(k_{4}): \nonumber\\
& + e^{- \lvert k_{12} \rvert^{2} (t-s) - \lvert k_{34} \rvert^{2} (t- \bar{s})} k_{12}^{i_{2}} k_{34}^{j_{2}} ii \nonumber\\
& \hspace{5mm} \times : \hat{\bar{X}}_{s,u}^{\epsilon, i_{1}}(k_{1}) \hat{\bar{X}}_{s,b}^{\epsilon, i_{2}}(k_{2}) \hat{\bar{X}}_{\bar{s}, b}^{\epsilon, j_{1}}(k_{3}) \hat{\bar{X}}_{\bar{s}, u}^{\epsilon, j_{2}}(k_{4}):]  \hat{\mathcal{P}}^{ii_{1}}(k_{12}) \hat{\mathcal{P}}^{jj_{1}}(k_{34}) e_{k} ds d\bar{s}, \nonumber \\
VII_{t, ij}^{14} \triangleq& \frac{ (2\pi)^{-\frac{9}{2}}}{4} \sum_{k} \sum_{k_{1}, k_{2}, k_{3}, k_{4} \neq 0: k_{1234} = k} \sum_{i_{1}, i_{2}, j_{1}, j_{2} = 1}^{3} \int_{[0,t]^{2}}  \\
& \times [ e^{- \lvert k_{12} \rvert^{2} f(\epsilon k_{12}) (t-s) - \lvert k_{34} \rvert^{2} f(\epsilon k_{34}) (t- \bar{s})} k_{12}^{i_{2}} k_{34}^{j_{2}} g(\epsilon k_{12}^{i_{2}}) g(\epsilon k_{34}^{j_{2}}) \nonumber\\
& \hspace{5mm} \times : \hat{X}_{s,u}^{\epsilon, i_{1}}(k_{1}) \hat{X}_{s,b}^{\epsilon, i_{2}}(k_{2}) \hat{X}_{\bar{s}, u}^{\epsilon, j_{1}}(k_{3}) \hat{X}_{\bar{s}, b}^{\epsilon, j_{2}}(k_{4}):  \nonumber\\
& - e^{- \lvert k_{12} \rvert^{2} (t-s) - \lvert k_{34} \rvert^{2} (t- \bar{s})} k_{12}^{i_{2}} k_{34}^{j_{2}} ii \nonumber\\
& \hspace{5mm} \times : \hat{\bar{X}}_{s,u}^{\epsilon, i_{1}}(k_{1}) \hat{\bar{X}}_{s,b}^{\epsilon, i_{2}}(k_{2}) \hat{\bar{X}}_{\bar{s}, u}^{\epsilon, j_{1}}(k_{3}) \hat{\bar{X}}_{\bar{s}, b}^{\epsilon, j_{2}}(k_{4}): ]  \hat{\mathcal{P}}^{ii_{1}}(k_{12}) \hat{\mathcal{P}}^{jj_{1}}(k_{34}) e_{k} ds d\bar{s}. \nonumber
\end{align}
\end{subequations}
W.l.o.g. we show the necessary estimates on only $VII_{t, ij}^{11}$ as those on other terms are similar. We compute 
\begin{align}
& \mathbb{E} [ \lvert \Delta_{q} VII_{t, ij}^{11} \rvert^{2} ] \nonumber\\
\lesssim& \mathbb{E} [ \lvert \sum_{k} \theta(2^{-q} k) \sum_{k_{1}, k_{2}, k_{3}, k_{4} \neq 0: k_{1234} = k} \sum_{i_{1}, i_{2}, j_{1}, j_{2} = 1}^{3} \int_{[0,t]^{2}}  \\
& \times [ e^{- \lvert k_{12} \rvert^{2} f(\epsilon k_{12}) (t-s) - \lvert k_{34} \rvert^{2} f(\epsilon k_{34}) (t- \bar{s})} k_{12}^{i_{2}} k_{34}^{j_{2}} g(\epsilon k_{12}^{i_{2}}) g(\epsilon k_{34}^{j_{2}}) \nonumber\\
& \hspace{5mm} - e^{- \lvert k_{12} \rvert^{2} (t-s) - \lvert k_{34} \rvert^{2} (t- \bar{s})} k_{12}^{i_{2}} k_{34}^{j_{2}} ii] \nonumber\\
& \hspace{5mm} \times : \hat{X}_{s,b}^{\epsilon, i_{1}}(k_{1}) \hat{X}_{s,u}^{\epsilon, i_{2}}(k_{2}) \hat{X}_{\bar{s}, b}^{\epsilon, j_{1}}(k_{3}) \hat{X}_{\bar{s}, u}^{\epsilon, j_{2}}(k_{4}):   \hat{\mathcal{P}}^{ii_{1}}(k_{12}) \hat{\mathcal{P}}^{jj_{1}}(k_{34}) e_{k} ds d\bar{s} \rvert^{2} ] \nonumber\\
&+ \mathbb{E} [ \lvert \sum_{k} \theta(2^{-q} k)  \sum_{k_{1}, k_{2}, k_{3}, k_{4} \neq 0: k_{1234} = k} \sum_{i_{1}, i_{2}, j_{1}, j_{2} = 1}^{3} \int_{[0,t]^{2}}  \nonumber\\
& \times  e^{- \lvert k_{12} \rvert^{2} (t-s) - \lvert k_{34} \rvert^{2} (t- \bar{s})} k_{12}^{i_{2}} k_{34}^{j_{2}} ii \nonumber\\
& \hspace{5mm} \times [  : \hat{X}_{s,b}^{\epsilon, i_{1}}(k_{1}) \hat{X}_{s,u}^{\epsilon, i_{2}}(k_{2}) \hat{X}_{\bar{s}, b}^{\epsilon, j_{1}}(k_{3}) \hat{X}_{\bar{s}, u}^{\epsilon, j_{2}}(k_{4}):  \nonumber\\
& \hspace{10mm} - : \hat{\bar{X}}_{s,b}^{\epsilon, i_{1}}(k_{1}) \hat{\bar{X}}_{s,u}^{\epsilon, i_{2}}(k_{2}) \hat{\bar{X}}_{\bar{s}, b}^{\epsilon, j_{1}}(k_{3}) \hat{\bar{X}}_{\bar{s}, u}^{\epsilon, j_{2}}(k_{4}): ]  \hat{\mathcal{P}}^{ii_{1}}(k_{12}) \hat{\mathcal{P}}^{jj_{1}}(k_{34}) e_{k} ds d\bar{s} \rvert^{2} ]. \nonumber
\end{align}
Now because 
\begin{align}\label{estimate 288}
& \mathbb{E} [ : \hat{X}_{s,b}^{\epsilon, i_{1}}(k_{1}) \hat{X}_{s,u}^{\epsilon, i_{2}}(k_{2}) \hat{X}_{\bar{s}, b}^{\epsilon, j_{1}}(k_{3}) \hat{X}_{\bar{s}, u}^{\epsilon, j_{2}}(k_{4}): \nonumber\\
& \hspace{5mm} \times \overline{ : \hat{X}_{\sigma,b}^{\epsilon, i_{1}'}(k_{1}') \hat{X}_{\sigma,u}^{\epsilon, i_{2}'}(k_{2}') \hat{X}_{\bar{\sigma}, b}^{\epsilon, j_{1}'}(k_{3}') \hat{X}_{\bar{\sigma}, u}^{\epsilon, j_{2}'}(k_{4}'): } \lesssim  \frac{\Theta}{\prod_{i=1}^{4} \lvert k_{i} \rvert^{2}}
\end{align} 
by Example \ref{Example 3.1} and \eqref{covariance} where 
\begin{align}\label{estimate 108}
\Theta \triangleq& 1_{k_{1} = k_{1}', k_{2}  = k_{2}', k_{3} = k_{3}', k_{4} = k_{4}'} + 1_{k_{1} = k_{1}', k_{2}  = k_{2}', k_{3} = k_{4}', k_{4} = k_{3}'} \nonumber\\
&+ 1_{k_{1} = k_{1}', k_{2}  = k_{3}', k_{3} = k_{2}', k_{4} = k_{4}'} + 1_{k_{1} = k_{1}', k_{2}  = k_{3}', k_{3} = k_{4}', k_{4} = k_{2}'} \nonumber\\
&+    1_{k_{1} = k_{1}', k_{2}  = k_{4}', k_{3} = k_{2}', k_{4} = k_{3}'} + 1_{k_{1} = k_{1}', k_{2}  = k_{4}', k_{3} = k_{3}', k_{4} = k_{2}'} \nonumber\\
&+ 1_{k_{1} = k_{2}', k_{2}  = k_{1}', k_{3} = k_{3}', k_{4} = k_{4}'} + 1_{k_{1} = k_{2}', k_{2}  = k_{1}', k_{3} = k_{4}', k_{4} = k_{3}' } \nonumber\\
&+    1_{k_{1} = k_{2}', k_{2}  = k_{3}', k_{3} = k_{1}', k_{4} = k_{4}'} + 1_{k_{1} = k_{2}', k_{2}  = k_{3}', k_{3} = k_{4}', k_{4} = k_{1}'} \nonumber\\
&+ 1_{k_{1} = k_{2}', k_{2}  = k_{4}', k_{3} = k_{1}', k_{4} = k_{3}'} + 1_{k_{1} = k_{2}', k_{2}  = k_{4}', k_{3} = k_{3}', k_{4} = k_{1}' } \nonumber\\
&+    1_{k_{1} = k_{3}', k_{2}  = k_{1}', k_{3} = k_{2}', k_{4} = k_{4}'} + 1_{k_{1} = k_{3}', k_{2}  = k_{1}', k_{3} = k_{4}', k_{4} = k_{2}'} \nonumber\\
&+ 1_{k_{1} = k_{3}', k_{2}  = k_{2}', k_{3} = k_{1}', k_{4} = k_{4}'} + 1_{k_{1} = k_{3}', k_{2}  = k_{2}', k_{3} = k_{4}', k_{4} = k_{1}' } \nonumber\\ 
&+    1_{k_{1} = k_{3}', k_{2}  = k_{4}', k_{3} = k_{1}', k_{4} = k_{2}'} + 1_{k_{1} = k_{3}', k_{2}  = k_{4}', k_{3} = k_{2}', k_{4} = k_{1}'} \nonumber\\
&+ 1_{k_{1} = k_{4}', k_{2}  = k_{1}', k_{3} = k_{2}', k_{4} = k_{3}'} + 1_{k_{1} = k_{4}', k_{2}  = k_{1}', k_{3} = k_{3}', k_{4} = k_{2}' } \nonumber\\
&+    1_{k_{1} = k_{4}', k_{2}  = k_{2}', k_{3} = k_{1}', k_{4} = k_{3}'} + 1_{k_{1} = k_{4}', k_{2}  = k_{2}', k_{3} = k_{3}', k_{4} = k_{1}'}  \nonumber\\
&+ 1_{k_{1} = k_{4}', k_{2}  = k_{3}', k_{3} = k_{1}', k_{4} = k_{2}'} + 1_{k_{1} = k_{4}', k_{2}  = k_{3}', k_{3} = k_{2}', k_{4} = k_{1}' }.
\end{align}
Thus, we obtain for any $\eta \in [0,1]$, 
\begin{align}\label{estimate 289}
& \mathbb{E} [ \lvert \Delta_{q} VII_{t, ij}^{11} \rvert^{2} ] \nonumber\\
\lesssim& \epsilon^{\eta} \sum_{k} \theta(2^{-q} k)^{2} \sum_{k_{1}, k_{2}, k_{3}, k_{4}, k_{1}', k_{2}', k_{3}', k_{4}' \neq 0: k_{1234} = k_{1234}' = k} \int_{[0,t]^{4}} \frac{\Theta}{\prod_{i=1}^{4} \lvert k_{i} \rvert^{2}}  \nonumber\\
& \times \lvert k_{12} \rvert \lvert k_{12}' \rvert \lvert k_{34} \rvert \lvert k_{34} ' \rvert e^{- \lvert k_{12} \rvert^{2} \bar{c}_{f} (t-s) - \lvert k_{34} \rvert^{2} \bar{c}_{f} (t- \bar{s}) - \lvert k_{12}' \rvert^{2} \bar{c}_{f} (t- \sigma) - \lvert k_{34}'\rvert^{2} \bar{c}_{f} (t- \bar{\sigma})}  \nonumber\\
& \times [ ( \lvert k_{12} \rvert^{\frac{\eta}{2}} + \lvert k_{34} \rvert^{\frac{\eta}{2}}) (\lvert k_{12}' \rvert^{\frac{\eta}{2}} + \lvert k_{34}' \rvert^{\frac{\eta}{2}}) + \sum_{i=1}^{4} \lvert k_{i} \rvert^{\frac{\eta}{2}} \lvert k_{i}' \rvert^{\frac{\eta}{2}} ] d \sigma d \bar{\sigma} ds d\bar{s}
\end{align}
by \eqref{estimate 288}, H$\ddot{\mathrm{o}}$lder's inequality, \eqref{estimate 37}, \eqref{[Equation (4.2)][ZZ17]}, \eqref{[Equation (4.3)][ZZ17]} and \eqref{key estimate}. Within the 24 terms of $\Theta$ in \eqref{estimate 108}, we can now take the advantage of symmetry and reduce all cases down to only two of $1_{k_{1} = k_{1}', k_{2} = k_{2}', k_{3} = k_{3}', k_{4} = k_{4}'}$ and $1_{k_{1} = k_{3}', k_{2} = k_{2}', k_{3} = k_{1}', k_{4} = k_{4}'}$. E.g., for the third case of $1_{k_{1} = k_{1}', k_{2} = k_{3}', k_{3} = k_{2}', k_{4} = k_{4}'}$ in \eqref{estimate 108}, we can swap $k_{1}$ with $k_{2}$ and $k_{1}'$ with $k_{2}'$ to deduce $1_{k_{2} = k_{2}', k_{1} = k_{3}', k_{3} = k_{1}', k_{4} = k_{4}'}$. Similarly, other cases may be reduced down to one of $1_{k_{1} = k_{1}', k_{2} = k_{2}', k_{3} = k_{3}', k_{4} = k_{4}'}$ and $1_{k_{1} = k_{3}', k_{2} = k_{2}', k_{3} = k_{1}', k_{4} = k_{4}'}$; thus, for brevity we refer readers to \cite[Equation (250)]{Y19a} to similar reductions. Therefore, \eqref{estimate 289} leads to  
\begin{align}\label{reduction 1}
& \mathbb{E} [ \lvert \Delta_{q} VII_{t, ij}^{11} \rvert^{2} ]\\
\lesssim& \epsilon^{\eta} \sum_{k} \theta(2^{-q} k)^{2} \sum_{k_{1}, k_{2}, k_{3}, k_{4}, k_{1}', k_{2}', k_{3}', k_{4}' \neq 0: k_{1234} = k_{1234}' = k} \nonumber\\
& \times (1_{k_{1} = k_{1}', k_{2} = k_{2}', k_{3} = k_{3}', k_{4} = k_{4}'} + 1_{k_{1} = k_{3}', k_{2} = k_{2}', k_{3} = k_{1}', k_{4} = k_{4}'}) \nonumber\\
& \times \frac{1}{\prod_{i=1}^{4} \lvert k_{i} \rvert^{2}} \frac{ (\lvert k_{12} \rvert^{2} t)^{\frac{\epsilon}{4}}}{\lvert k_{12} \rvert} \frac{(\lvert k_{34} \rvert^{2} t)^{\frac{\epsilon}{4}}}{\lvert k_{34} \rvert} \frac{ (\lvert k_{12}' \rvert^{2} t)^{\frac{\epsilon}{4}}}{\lvert k_{12}' \rvert} \frac{ (\lvert k_{34} '\rvert^{2} t)^{\frac{\epsilon}{4}}}{\lvert k_{34}'\rvert} (\sum_{i=1}^{4} \lvert k_{i} \rvert^{\eta}) \lesssim  \epsilon^{\eta}t^{\epsilon} 2^{q(\eta + 2 \epsilon)}   \nonumber
\end{align}
by mean value theorem. By a similar calculation, we can show that for all $\epsilon, \eta > 0$ sufficiently small, and any $t_{1}, t_{2} > 0$, 
\begin{align*}
&\mathbb{E} [ \lvert \Delta_{q} ( b_{2}^{\epsilon, i} \diamond b_{2}^{\epsilon, j} (t_{1}) - b_{2}^{\epsilon, i} \diamond b_{2}^{\epsilon, j}(t_{2})  \nonumber\\
&\hspace{5mm} - \bar{b}_{2}^{\epsilon, i} \diamond \bar{b}_{2}^{\epsilon, j} (t_{1}) + \bar{b}_{2}^{\epsilon, i} \diamond \bar{b}_{2}^{\epsilon, j} (t_{2})) \rvert^{2}] \lesssim \epsilon^{\eta} \lvert t_{1} - t_{2} \rvert^{\eta} 2^{q(\epsilon + 3\eta)}
\end{align*} 
and therefore an application of Gaussian hypercontractivity theorem and Besov embedding Lemma \ref{Lemma 3.4}, similarly to \eqref{estimate 36}, implies that for all $i, j \in \{1,2,3\}$, $p\geq 1$, and $\delta > 0$, $b_{2}^{\epsilon, i} \diamond b_{2}^{\epsilon, j} - \bar{b}_{2}^{\epsilon, i} \diamond \bar{b}_{2}^{\epsilon, j} \searrow 0$ in $L^{p}(\Omega; C([0,T]; \mathcal{C}^{-\delta}))$ as $\epsilon \searrow 0$. 

\subsection{Conclusion of the proof of Theorem \ref{Theorem 1.2}}\label{Subsection 2.3}

We define (cf. \eqref{[Equation (3.2e)][ZZ17]} and \eqref{[Equation (3.23a)][ZZ17]})
\begin{align}\label{[Equation (3.25f)][ZZ17]} 
\delta &C_{W}^{\epsilon} \triangleq  \sup_{t\in [0,T]} [ \sum_{i=1}^{3} \lVert (\bar{u}_{1}^{\epsilon, i}, \bar{b}_{1}^{\epsilon, i})(t)  -  (u_{1}^{\epsilon, i}, b_{1}^{\epsilon, i})(t) \rVert_{\mathcal{C}^{- \frac{1}{2} - \frac{\delta}{2}}} \\
& + \sum_{i,j=1}^{3} \lVert (\bar{u}_{1}^{\epsilon, i} \diamond \bar{u}_{1}^{\epsilon, j} - u_{1}^{\epsilon, i} \diamond u_{1}^{\epsilon, j}, \bar{b}_{1}^{\epsilon, i} \diamond \bar{b}_{1}^{\epsilon, j} - b_{1}^{\epsilon, i} \diamond b_{1}^{\epsilon, j},  \nonumber\\
& \hspace{25mm} \bar{u}_{1}^{\epsilon, i} \diamond \bar{b}_{1}^{\epsilon, j} - u_{1}^{\epsilon, i} \diamond b_{1}^{\epsilon, j}, \bar{b}_{1}^{\epsilon, i} \diamond \bar{u}_{1}^{\epsilon, j} - b_{1}^{\epsilon, i} \diamond u_{1}^{\epsilon, j})(t) \rVert_{\mathcal{C}^{-1 - \frac{\delta}{2}}} \nonumber\\
&+  \sum_{i,j=1}^{3} \lVert ( \bar{u}_{1}^{\epsilon, i} \diamond \bar{u}_{2}^{\epsilon, j} - u_{1}^{\epsilon, i} \diamond u_{2}^{\epsilon, j}, \bar{b}_{1}^{\epsilon, i} \diamond \bar{b}_{2}^{\epsilon, j} - b_{1}^{\epsilon, i} \diamond b_{2}^{\epsilon, j},  \nonumber\\
& \hspace{25mm} \bar{u}_{2}^{\epsilon, i} \diamond \bar{b}_{1}^{\epsilon, j}  -  u_{2}^{\epsilon, i} \diamond b_{1}^{\epsilon, j}, \bar{b}_{2}^{\epsilon, i} \diamond \bar{u}_{1}^{\epsilon, j} - b_{2}^{\epsilon, i} \diamond u_{1}^{\epsilon, j})(t) \rVert_{\mathcal{C}^{- \frac{1}{2} - \frac{\delta}{2}}} \nonumber\\
&+  \sum_{i,j=1}^{3} \lVert ( \bar{u}_{2}^{\epsilon, i} \diamond \bar{u}_{2}^{\epsilon, j} - u_{2}^{\epsilon, i} \diamond u_{2}^{\epsilon, j}, \bar{b}_{2}^{\epsilon, i} \diamond \bar{b}_{2}^{\epsilon, j} - b_{2}^{\epsilon, i} \diamond b_{2}^{\epsilon, j}, \bar{b}_{2}^{\epsilon, i} \diamond \bar{u}_{2}^{\epsilon, j} - b_{2}^{\epsilon, i} \diamond u_{2}^{\epsilon, j})(t) \rVert_{\mathcal{C}^{-\delta}} \nonumber\\
&+  \sum_{i,j=1}^{3} \lVert ( \pi_{0,\diamond} (\bar{u}_{3}^{\epsilon, j}, \bar{u}_{1}^{\epsilon, i}) -  \pi_{0,\diamond} (u_{3}^{\epsilon, j}, u_{1}^{\epsilon, i}), \pi_{0,\diamond} (\bar{b}_{3}^{\epsilon, j}, \bar{b}_{1}^{\epsilon, i}) - \pi_{0,\diamond} (b_{3}^{\epsilon, j}, b_{1}^{\epsilon, i}), \nonumber\\
& \hspace{25mm}  \pi_{0,\diamond}( \bar{u}_{3}^{\epsilon, j}, \bar{b}_{1}^{\epsilon, i}) - \pi_{0,\diamond}( u_{3}^{\epsilon, j}, b_{1}^{\epsilon, i}), \pi_{0,\diamond} (\bar{b}_{3}^{\epsilon, j}, \bar{u}_{1}^{\epsilon, i})-  \pi_{0,\diamond} (b_{3}^{\epsilon, j}, u_{1}^{\epsilon, i})) \rVert_{\mathcal{C}^{-\delta}}\nonumber\\
&+  \sum_{i, i_{1}, j, j_{1} =1}^{3} \lVert (  \pi_{0,\diamond} (\mathcal{P}^{ii_{1}} D_{j} \bar{K}_{u}^{\epsilon, j}, \bar{u}_{1}^{\epsilon, j_{1}}) -  \pi_{0,\diamond} (\mathcal{P}^{ii_{1}} D_{j}^{\epsilon} K_{u}^{\epsilon, j}, u_{1}^{\epsilon, j_{1}}),  \nonumber\\
& \hspace{25mm} \pi_{0,\diamond}(\mathcal{P}^{ii_{1}} D_{j} \bar{K}_{u}^{\epsilon, i_{1}}, \bar{u}_{1}^{\epsilon, j_{1}}) - \pi_{0,\diamond}(\mathcal{P}^{ii_{1}} D_{j}^{\epsilon} K_{u}^{\epsilon, i_{1}}, u_{1}^{\epsilon, j_{1}}), \nonumber\\
& \hspace{25mm}  \pi_{0,\diamond} (\mathcal{P}^{ii_{1}} D_{j} \bar{K}_{b}^{\epsilon, j}, \bar{u}_{1}^{\epsilon, j_{1}}) - \pi_{0,\diamond} (\mathcal{P}^{ii_{1}} D_{j}^{\epsilon} K_{b}^{\epsilon, j}, u_{1}^{\epsilon, j_{1}}), \nonumber\\
& \hspace{25mm}  \pi_{0,\diamond} (\mathcal{P}^{ii_{1}} D_{j} \bar{K}_{b}^{\epsilon, i_{1}}, \bar{u}_{1}^{\epsilon, j_{1}}) - \pi_{0,\diamond} (\mathcal{P}^{ii_{1}} D_{j}^{\epsilon} K_{b}^{\epsilon, i_{1}}, u_{1}^{\epsilon, j_{1}}), \nonumber\\
& \hspace{25mm} \pi_{0,\diamond} (\mathcal{P}^{ii_{1}} D_{j} \bar{K}_{u}^{\epsilon, j}, \bar{b}_{1}^{\epsilon, j_{1}}) -  \pi_{0,\diamond} (\mathcal{P}^{ii_{1}} D_{j}^{\epsilon} K_{u}^{\epsilon, j}, b_{1}^{\epsilon, j_{1}}), \nonumber\\
& \hspace{25mm}  \pi_{0,\diamond}(\mathcal{P}^{ii_{1}} D_{j} \bar{K}_{u}^{\epsilon, i_{1}}, \bar{b}_{1}^{\epsilon, j_{1}}) - \pi_{0,\diamond}(\mathcal{P}^{ii_{1}} D_{j}^{\epsilon} K_{u}^{\epsilon, i_{1}}, b_{1}^{\epsilon, j_{1}}), \nonumber\\
& \hspace{25mm} \pi_{0,\diamond} (\mathcal{P}^{ii_{1}} D_{j} \bar{K}_{b}^{\epsilon, j}, \bar{b}_{1}^{\epsilon, j_{1}})-  \pi_{0,\diamond} (\mathcal{P}^{ii_{1}} D_{j}^{\epsilon} K_{b}^{\epsilon, j}, b_{1}^{\epsilon, j_{1}}), \nonumber\\
& \hspace{25mm}  \pi_{0,\diamond} (\mathcal{P}^{ii_{1}} D_{j} \bar{K}_{b}^{\epsilon, i_{1}}, \bar{b}_{1}^{\epsilon, j_{1}}) -  \pi_{0,\diamond} (\mathcal{P}^{ii_{1}} D_{j}^{\epsilon} K_{b}^{\epsilon, i_{1}}, b_{1}^{\epsilon, j_{1}}) ) \rVert_{\mathcal{C}^{-\delta}}]. \nonumber
\end{align}
Then a similar argument as in the construction of \eqref{MHD} (cf. \eqref{[Equation (3.18)][ZZ17]}) implies that for a sufficiently small $T_{1} \leq \tau_{L} \wedge \tau_{L_{1}}^{\epsilon} \wedge \rho_{L_{2}}^{\epsilon} \wedge \rho_{L_{3}}^{\epsilon}$, where we recall $\tau_{L}, \tau_{L}^{\epsilon}$ and $\rho_{L}^{\epsilon}$ respectively from \eqref{estimate 104} and \eqref{estimate 290}, for any $L_{1}, L_{2}, L_{3} > 0$, and $\kappa > 0$ sufficiently small, 
\begin{align}\label{[Equation (3.25e)][ZZ17]} 
& \sup_{t\in [0,T_{1}]} \lVert ( \bar{u}_{4}^{\epsilon, i}, \bar{b}_{4}^{\epsilon, i} )(t) - (u_{4}^{\epsilon, i}, b_{4}^{\epsilon, i} )(t) \rVert_{\mathcal{C}^{-z}} \nonumber\\
\lesssim& (\delta C_{W}^{\epsilon} + \epsilon^{\frac{\kappa}{2}} + \lVert (u_{0}^{\epsilon} - u_{0}, b_{0}^{\epsilon}- b_{0}) \rVert_{\mathcal{C}^{-z}}) C(L, L_{1}, L_{2}, L_{3}). 
\end{align}
Again, extending the local existence time to $\tau_{L}\wedge \tau_{L_{1}}^{\epsilon} \wedge \rho_{L_{2}}^{\epsilon} \wedge \rho_{L_{3}}^{\epsilon}$ is standard and we omit its details (see the proof of \cite[Theorem 3.12]{ZZ15}). Due to convergence results in Subsection \ref{Subsection 2.2} we can show similarly to \eqref{[Equation (3.19)][ZZ17]} that $\delta C_{W}^{\epsilon} \searrow 0$ in probability as $\epsilon \searrow 0$. We estimate 
\begin{equation}\label{estimate 38}
\lVert y^{\epsilon}(t) - y(t) \rVert_{\mathcal{C}^{-z}} \leq \lVert y^{\epsilon}(t) - \bar{y}^{\epsilon}(t) \rVert_{\mathcal{C}^{-z}} + \lVert \bar{y}^{\epsilon}(t) - y(t) \rVert_{\mathcal{C}^{-z}} 
\end{equation} 
where we know that 
\begin{equation*}
\sup_{t\in [0, \tau_{L} \wedge \bar{\rho}_{L_{1}}^{\epsilon} ]} \lVert \bar{y}^{\epsilon}(t) - y(t) \rVert_{\mathcal{C}^{-z}} \to 0 
\end{equation*} 
in probability as $\epsilon \searrow 0$ due to \eqref{[Equation (3.21a)][ZZ17]}. Thus, we work on 
\begin{align}\label{estimate 39}
 \lVert y^{\epsilon}(t) - \bar{y}^{\epsilon}(t) \rVert_{\mathcal{C}^{-z}} \lesssim \sum_{i=1}^{4} \lVert (u_{i}^{\epsilon} - \bar{u}_{i}^{\epsilon}, b_{i}^{\epsilon} - \bar{b}_{i}^{\epsilon})(t) \rVert_{\mathcal{C}^{-z}}. 
\end{align}
First, 
\begin{equation}\label{estimate 303}
\lVert (u_{1}^{\epsilon} - \bar{u}_{1}^{\epsilon}, b_{1}^{\epsilon} - \bar{b}_{1}^{\epsilon} )(t) \rVert_{\mathcal{C}^{-z}} \lesssim \lVert (u_{1}^{\epsilon} - \bar{u}_{1}^{\epsilon}, b_{1}^{\epsilon} - \bar{b}_{1}^{\epsilon} )(t) \rVert_{\mathcal{C}^{-\frac{1}{2} - \frac{\delta}{2}}} \to 0  
\end{equation} 
in probability as $\epsilon \searrow 0$ where we used \eqref{[Equation (3.2g)][ZZ17]} and \eqref{[Equation (3.25f)][ZZ17]}. Second, 
\begin{align}\label{estimate 295}
 \lVert (u_{2}^{\epsilon} - \bar{u}_{2}^{\epsilon}, b_{2}^{\epsilon} - \bar{b}_{2}^{\epsilon})(t) \rVert_{\mathcal{C}^{-z}}   \lesssim \sum_{i=1}^{3} VIII_{t}^{i}
\end{align}
by \eqref{[Equation (3.1ab)][ZZ17]} and \eqref{[Equation (3.21d)][ZZ17]} where 
\begin{subequations}\label{estimate 291}
\begin{align}
&VIII_{t}^{1} \triangleq \sum_{i_{1}, j=1}^{3} \int_{0}^{t} \lVert ( (P_{t-s} - P_{t-s}^{\epsilon}) D_{j} (\bar{u}_{1}^{\epsilon, i_{1}} \diamond \bar{u}_{1}^{\epsilon, j}), (P_{t-s}  - P_{t-s}^{\epsilon}) D_{j} (\bar{b}_{1}^{\epsilon, i_{1}} \diamond \bar{b}_{1}^{\epsilon, j}),  \nonumber\\
&\hspace{10mm}  (P_{t-s} - P_{t-s}^{\epsilon}) D_{j} (\bar{b}_{1}^{\epsilon, i_{1}} \diamond \bar{u}_{1}^{\epsilon, j}), (P_{t-s}  - P_{t-s}^{\epsilon}) D_{j} (\bar{u}_{1}^{\epsilon, i_{1}} \diamond \bar{b}_{1}^{\epsilon, j}) ) \rVert_{\mathcal{C}^{-z}} ds,  \\
&VIII_{t}^{2} \triangleq \sum_{i_{1}, j=1}^{3} \int_{0}^{t} \lVert (P_{t-s}^{\epsilon}(D_{j} - D_{j}^{\epsilon}) (\bar{u}_{1}^{\epsilon, i_{1}} \diamond \bar{u}_{1}^{\epsilon, j}), P_{t-s}^{\epsilon} (D_{j} - D_{j}^{\epsilon}) (\bar{b}_{1}^{\epsilon, i_{1}} \diamond \bar{b}_{1}^{\epsilon, j}), \nonumber\\
& \hspace{10mm} P_{t-s}^{\epsilon} (D_{j} - D_{j}^{\epsilon})(\bar{b}_{1}^{\epsilon, i_{1}} \diamond \bar{u}_{1}^{\epsilon, j}), P_{t-s}^{\epsilon} (D_{j} - D_{j}^{\epsilon}) (\bar{u}_{1}^{\epsilon, i_{1}} \diamond \bar{b}_{1}^{\epsilon, j}) ) \rVert_{\mathcal{C}^{-z}} ds,  \\
&VIII_{t}^{3} \triangleq \int_{0}^{t} \lVert ( P_{t-s}^{\epsilon} D_{j}^{\epsilon} (\bar{u}_{1}^{\epsilon, i_{1}} \diamond \bar{u}_{1}^{\epsilon, j} - u_{1}^{\epsilon, i_{1}} \diamond u_{1}^{\epsilon, j} ), P_{t-s}^{\epsilon} D_{j}^{\epsilon}(\bar{b}_{1}^{\epsilon, i_{1}} \diamond \bar{b}_{1}^{\epsilon, j} - b_{1}^{\epsilon, i_{1}} \diamond b_{1}^{\epsilon, j} ), \nonumber   \\
& \hspace{30mm} P_{t-s}^{\epsilon} D_{j}^{\epsilon} (\bar{b}_{1}^{\epsilon, i_{1}} \diamond \bar{u}_{1}^{\epsilon, j} - b_{1}^{\epsilon, i_{1}} \diamond u_{1}^{\epsilon, j} ), \nonumber\\
& \hspace{35mm} P_{t-s}^{\epsilon} D_{j}^{\epsilon}(\bar{u}_{1}^{\epsilon, i_{1}} \diamond \bar{b}_{1}^{\epsilon, j} - u_{1}^{\epsilon, i_{1}} \diamond b_{1}^{\epsilon, j} )) \rVert_{\mathcal{C}^{-z}} ds. 
\end{align}
\end{subequations} 
We estimate from \eqref{estimate 291}
\begin{align}\label{estimate 292}
VIII_{t}^{1} \lesssim \epsilon^{\frac{1}{3}} t^{-\frac{1}{4} + \frac{z}{2} - \frac{\delta}{4}}  \to 0 
\end{align}
as $\epsilon \searrow 0$ by Lemma \ref{Lemma 3.2}, \eqref{[Equation (3.2g)][ZZ17]}, \eqref{[Equation (3.2e)][ZZ17]} and \eqref{[Equation (3.23a)][ZZ17]}. Next, from \eqref{estimate 291}
\begin{align}\label{estimate 293}
VIII_{t}^{2} \lesssim \epsilon^{\frac{1}{8}} t^{\frac{z}{2} - \frac{1}{4} - \frac{\delta}{4}}  \to 0
\end{align}
as $\epsilon \searrow 0$ by Lemmas \ref{Lemma 3.1} and \ref{Lemma 3.15} and \eqref{[Equation (3.2g)][ZZ17]}. Finally, from \eqref{estimate 291}
\begin{align}\label{estimate 294}
&VIII_{t}^{3} \\
\lesssim& t^{- \frac{1}{4} + \frac{z}{2} - \frac{\delta}{4}} \sup_{t \in [0, \tau_{L} \wedge \tau_{L_{1}}^{\epsilon} \wedge \rho_{L_{2}}^{\epsilon} \wedge \bar{\rho}_{L_{3}}^{\epsilon} ]} \lVert ( \bar{u}_{1}^{\epsilon, i_{1}} \diamond \bar{u}_{1}^{\epsilon, j} - u_{1}^{\epsilon, i_{1}} \diamond u_{1}^{\epsilon, j}, \bar{b}_{1}^{\epsilon, i_{1}} \diamond \bar{b}_{1}^{\epsilon, j} - b_{1}^{\epsilon, i_{1}} \diamond b_{1}^{\epsilon, j}, \nonumber\\
& \hspace{15mm} \times \bar{b}_{1}^{\epsilon, i_{1}} \diamond \bar{u}_{1}^{\epsilon, j} - b_{1}^{\epsilon, i_{1}} \diamond u_{1}^{\epsilon, j}, \bar{u}_{1}^{\epsilon, i_{1}} \diamond \bar{b}_{1}^{\epsilon, j} - u_{1}^{\epsilon, i_{1}} \diamond b_{1}^{\epsilon, j})(t) \rVert_{\mathcal{C}^{-1 - \frac{\delta}{2}}} \to 0 \nonumber
 \end{align}
as $\epsilon \searrow 0$ by Lemmas \ref{Lemma 3.1} and \ref{Lemma 3.3} and \eqref{[Equation (3.2g)][ZZ17]}. Thus, by applying \eqref{estimate 292}-\eqref{estimate 294} to \eqref{estimate 295}, we conclude now that 
\begin{equation}\label{estimate 296}
\lVert (u_{2}^{\epsilon} - \bar{u}_{2}^{\epsilon}, b_{2}^{\epsilon} - \bar{b}_{2}^{\epsilon})(t) \rVert_{\mathcal{C}^{-z}}   \to 0
\end{equation}
as $\epsilon \searrow 0$. Next, we consider from \eqref{[Equation (3.1ac)][ZZ17]} and \eqref{[Equation (3.21e)][ZZ17]}
\begin{equation}
\lVert (u_{3}^{\epsilon} - \bar{u}_{3}^{\epsilon}, b_{3}^{\epsilon} - \bar{b}_{3}^{\epsilon})(t) \rVert_{\mathcal{C}^{-z}} =  \frac{1}{2}(IX_{t}^{1} + IX_{t}^{2})
\end{equation}
where 
\begin{subequations}
\begin{align} 
IX_{t}^{1}&\triangleq \sum_{i=1}^{3} \lVert  \sum_{i_{1}, j=1}^{3} \mathcal{P}^{ii_{1}} D_{j}^{\epsilon} \int_{0}^{t} P_{t-s}^{\epsilon} (u_{1}^{\epsilon, i_{1}} \diamond u_{2}^{\epsilon, j} + u_{2}^{\epsilon, i_{1}} \diamond u_{1}^{\epsilon, j} \nonumber\\
& \hspace{45mm}  - b_{1}^{\epsilon, i_{1}} \diamond b_{2}^{\epsilon, j} - b_{2}^{\epsilon, i_{1}} \diamond b_{1}^{\epsilon, j}) ds \nonumber \\
& \hspace{5mm} -   \sum_{i_{1}, j=1}^{3} \mathcal{P}^{ii_{1}} D_{j} \int_{0}^{t} P_{t-s} (\bar{u}_{1}^{\epsilon, i_{1}} \diamond \bar{u}_{2}^{\epsilon, j} + \bar{u}_{2}^{\epsilon, i_{1}} \diamond \bar{u}_{1}^{\epsilon, j} \nonumber\\
& \hspace{45mm}  - \bar{b}_{1}^{\epsilon, i_{1}} \diamond \bar{b}_{2}^{\epsilon, j} - \bar{b}_{2}^{\epsilon, i_{1}} \diamond \bar{b}_{1}^{\epsilon, j}) ds \rVert_{\mathcal{C}^{-z}},\\
IX_{t}^{2}&\triangleq \sum_{i=1}^{3} \lVert  \sum_{i_{1}, j=1}^{3} \mathcal{P}^{ii_{1}} D_{j}^{\epsilon} \int_{0}^{t} P_{t-s}^{\epsilon} (b_{1}^{\epsilon, i_{1}} \diamond u_{2}^{\epsilon, j} + b_{2}^{\epsilon, i_{1}} \diamond u_{1}^{\epsilon, j} \nonumber\\
& \hspace{45mm}  - u_{1}^{\epsilon, i_{1}} \diamond b_{2}^{\epsilon, j} - u_{2}^{\epsilon, i_{1}} \diamond b_{1}^{\epsilon, j}) ds \nonumber \\
& \hspace{5mm} - \sum_{i_{1}, j=1}^{3} \mathcal{P}^{ii_{1}} D_{j} \int_{0}^{t} P_{t-s} (\bar{b}_{1}^{\epsilon, i_{1}} \diamond \bar{u}_{2}^{\epsilon, j} + \bar{b}_{2}^{\epsilon, i_{1}} \diamond \bar{u}_{1}^{\epsilon, j}  \nonumber\\
& \hspace{45mm} - \bar{u}_{1}^{\epsilon, i_{1}} \diamond \bar{b}_{2}^{\epsilon, j} - \bar{u}_{2}^{\epsilon, i_{1}} \diamond \bar{b}_{1}^{\epsilon, j}) ds \rVert_{\mathcal{C}^{-z}}.  
\end{align}
\end{subequations}
W.l.o.g. we show the necessary estimates for only $IX_{t}^{1}$ as those for $IX_{t}^{2}$ are similar. We compute 
\begin{equation}\label{estimate 301}
IX_{t}^{1} \lesssim \sum_{l=1}^{3} IX_{t}^{1l}
\end{equation}
where 
\begin{subequations}\label{estimate 297}
\begin{align}
IX_{t}^{11} \triangleq&  \sum_{i_{1}, j=1}^{3} \int_{0}^{t} \lVert (P_{t-s} - P_{t-s}^{\epsilon}) D_{j} (\bar{u}_{1}^{\epsilon, i_{1}} \diamond \bar{u}_{2}^{\epsilon, j} + \bar{u}_{2}^{\epsilon, i_{1}} \diamond \bar{u}_{1}^{\epsilon,j} \nonumber\\
& \hspace{12mm} - \bar{b}_{1}^{\epsilon, i_{1}} \diamond \bar{b}_{2}^{\epsilon, j} - \bar{b}_{2}^{\epsilon, i_{1}} \diamond \bar{b}_{1}^{\epsilon, j} )(s) \rVert_{\mathcal{C}^{-z}} ds,\\
IX_{t}^{12} \triangleq& \sum_{i_{1}, j=1}^{3} \int_{0}^{t} \lVert P_{t-s}^{\epsilon} (D_{j} - D_{j}^{\epsilon}) (\bar{u}_{1}^{\epsilon, i_{1}} \diamond \bar{u}_{2}^{\epsilon, j} + \bar{u}_{2}^{\epsilon, i_{1}} \diamond \bar{u}_{1}^{\epsilon, j} \nonumber\\
& \hspace{12mm} - \bar{b}_{1}^{\epsilon, i_{1}} \diamond \bar{b}_{2}^{\epsilon, j} - \bar{b}_{2}^{\epsilon, i_{1}} \diamond \bar{b}_{1}^{\epsilon, j})(s) \rVert_{\mathcal{C}^{-z}} ds, \\
IX_{t}^{13} \triangleq& \sum_{i_{1}, j=1}^{3} \int_{0}^{t} \lVert P_{t-s}^{\epsilon} D_{j}^{\epsilon} (\bar{u}_{1}^{\epsilon, i_{1}} \diamond \bar{u}_{2}^{\epsilon, j} - u_{1}^{\epsilon, i_{1}} \diamond u_{2}^{\epsilon, j} + \bar{u}_{2}^{\epsilon, i_{1}} \diamond \bar{u}_{1}^{\epsilon, j} - u_{2}^{\epsilon, i_{1}} \diamond u_{1}^{\epsilon, j} \nonumber\\
& \hspace{12mm} - \bar{b}_{1}^{\epsilon, i_{1}} \diamond \bar{b}_{2}^{\epsilon, j} + b_{1}^{\epsilon, i_{1}} \diamond b_{2}^{\epsilon, j} - \bar{b}_{2}^{\epsilon, i_{1}} \diamond \bar{b}_{1}^{\epsilon, j} + b_{2}^{\epsilon, i_{1}} \diamond b_{1}^{\epsilon, j}) \rVert_{\mathcal{C}^{-z}} ds.  
\end{align}
\end{subequations} 
We estimate from \eqref{estimate 297}
\begin{align}\label{estimate 298}
IX_{t}^{11} \lesssim& \sum_{i_{1}, j=1}^{3} \int_{0}^{t} \epsilon^{\frac{1}{3}} (t-s)^{- \frac{ 2- z + \frac{\delta}{2}}{2}} \lVert (\bar{u}_{1}^{\epsilon, i_{1}} \diamond \bar{u}_{2}^{\epsilon, j} + \bar{u}_{2}^{\epsilon, i_{1}} \diamond \bar{u}_{1}^{\epsilon,j} \nonumber\\
& \hspace{15mm} - \bar{b}_{1}^{\epsilon, i_{1}} \diamond \bar{b}_{2}^{\epsilon, j} - \bar{b}_{2}^{\epsilon, i_{1}} \diamond \bar{b}_{1}^{\epsilon, j} )(s) \rVert_{\mathcal{C}^{-\frac{1}{2} - \frac{\delta}{2}}} ds \lesssim \epsilon^{\frac{1}{3}} t^{\frac{ z - \frac{\delta}{2}}{2}}  \to 0 
\end{align}
as $\epsilon \searrow 0$ by Lemma \ref{Lemma 3.2}, \eqref{[Equation (3.2e)][ZZ17]} and \eqref{[Equation (3.2g)][ZZ17]}. Next, from \eqref{estimate 297}
\begin{align}\label{estimate 299}
IX_{t}^{12} \lesssim& \int_{0}^{t} (t-s)^{- \frac{ 2- z + \frac{\delta}{2}}{2}} \epsilon^{\frac{1}{8}} \lVert (\bar{u}_{1}^{\epsilon, i_{1}} \diamond \bar{u}_{2}^{\epsilon, j} + \bar{u}_{2}^{\epsilon, i_{1}} \diamond \bar{u}_{1}^{\epsilon, j}  \nonumber\\
& \hspace{20mm} - \bar{b}_{1}^{\epsilon, i_{1}} \diamond \bar{b}_{2}^{\epsilon, j} - \bar{b}_{2}^{\epsilon, i_{1}}\diamond \bar{b}_{1}^{\epsilon, j}) \rVert_{\mathcal{C}^{-\frac{1}{2} - \frac{\delta}{2}}} ds \lesssim  \epsilon^{\frac{1}{8}} t^{\frac{ z - \frac{\delta}{2}}{2}}  \to 0 
\end{align}
as $\epsilon \searrow 0$ by Lemmas \ref{Lemma 3.1} and \ref{Lemma 3.15} and \eqref{[Equation (3.2g)][ZZ17]}. Finally, from \eqref{estimate 297}
\begin{align}\label{estimate 300}
IX_{t}^{13} \lesssim& t^{\frac{ z - \frac{\delta}{2}}{2}} \sup_{t\in [0, \tau_{L} \wedge \tau_{L_{1}}^{\epsilon} \wedge \rho_{L_{2}}^{\epsilon} \wedge \bar{\rho}_{L_{3}}^{\epsilon} ]} \lVert (\bar{u}_{1}^{\epsilon, i_{1}} \diamond \bar{u}_{2}^{\epsilon, j} - u_{1}^{\epsilon, i_{1}} \diamond u_{2}^{\epsilon, j} + \bar{u}_{2}^{\epsilon, i_{1}} \diamond \bar{u}_{1}^{\epsilon, j} - u_{2}^{\epsilon, i_{1}} \diamond u_{1}^{\epsilon, j} \nonumber\\
& \hspace{4mm} - \bar{b}_{1}^{\epsilon, i_{1}} \diamond \bar{b}_{2}^{\epsilon, j} + b_{1}^{\epsilon, i_{1}} \diamond b_{2}^{\epsilon, j} - \bar{b}_{2}^{\epsilon, i_{1}} \diamond \bar{b}_{1}^{\epsilon, j} + b_{2}^{\epsilon, i_{1}} \diamond b_{1}^{\epsilon, j}) (t) \rVert_{\mathcal{C}^{-\frac{1}{2} - \frac{\delta}{2}}} \to 0 
\end{align}
as $\epsilon \searrow 0$ by Lemmas \ref{Lemma 3.1} and \ref{Lemma 3.3} and \eqref{[Equation (3.2g)][ZZ17]}.  Therefore, by applying \eqref{estimate 298}-\eqref{estimate 300} to \eqref{estimate 301},  we conclude that 
\begin{equation}\label{estimate 302}
\lVert (u_{3}^{\epsilon} - \bar{u}_{3}^{\epsilon}, b_{3}^{\epsilon} - \bar{b}_{3}^{\epsilon})(t) \rVert_{\mathcal{C}^{-z}} \to 0
\end{equation} 
as $\epsilon \searrow 0$. Along with \eqref{[Equation (3.25e)][ZZ17]} and the hypothesis that $(u_{0}^{\epsilon}, b_{0}^{\epsilon}) \to (u_{0}, b_{0})$ in $\mathcal{C}^{-z}$ as $\epsilon \searrow 0$, we conclude from \eqref{estimate 38}, \eqref{estimate 39}, \eqref{estimate 303}, \eqref{estimate 296} and \eqref{estimate 302} that  
\begin{equation}\label{[Equation (3.25h)][ZZ17]} 
\sup_{t\in [0, \tau_{L} \wedge \tau_{L_{1}}^{\epsilon} \wedge \rho_{L_{2}}^{\epsilon} \wedge \bar{\rho}_{L_{3}}^{\epsilon}]}  \lVert (y^{\epsilon} - y)(t) \rVert_{\mathcal{C}^{-z}} \to 0 
\end{equation}
in probability as $\epsilon \searrow 0$. Finally, for any fixed $\epsilon > 0, L > 0$, this leads to 
\begin{equation}
\mathbb{P} ( \{\sup_{t \in [0,\tau_{L}]} \lVert (y^{\epsilon} - y)(t) \rVert_{\mathcal{C}^{-z}} \geq \epsilon \}) \to 0 
\end{equation} 
as $\epsilon \searrow 0$. Indeed, 
\begin{align}
&\mathbb{P} ( \{ \sup_{t\in [0, \tau_{L} ]} \lVert (y^{\epsilon} - y)(t) \rVert_{\mathcal{C}^{-z}} \geq \epsilon \})  \nonumber\\
\leq& \mathbb{P} ( \{ \sup_{t\in [0, \tau_{L} \wedge \tau_{L_{1}}^{\epsilon} \wedge \rho_{L_{2}}^{\epsilon} \wedge \bar{\rho}_{L_{3}}^{\epsilon} ]} \lVert (y^{\epsilon} - y)(t) \rVert_{\mathcal{C}^{-z}} > \epsilon \}) \nonumber\\
& \hspace{5mm} + \mathbb{P} (\{\tau_{L_{1}}^{\epsilon} < \tau_{L} \wedge \rho_{L_{2}}^{\epsilon} \wedge \bar{\rho}_{L_{3}}^{\epsilon} \}) + \mathbb{P} (\{\rho_{L_{2}}^{\epsilon} < \tau_{L}\}) + \mathbb{P} (\{\bar{\rho}_{L_{3}}^{\epsilon} < \tau_{L} \})
\end{align}
where the first and second probabilities approach zero due to \eqref{[Equation (3.25h)][ZZ17]} while  $\mathbb{P} (\{\tau_{L} > \rho_{L_{2}}^{\epsilon} \}) + \mathbb{P} (\{\tau_{L} > \bar{\rho}_{L_{3}}^{\epsilon} \}) \to 0$ as $L_{2}, L_{3} \nearrow + \infty$ uniformly over $\epsilon \in (0,1)$ by \eqref{estimate 104} and \eqref{estimate 290}. This completes the proof of Theorem \ref{Theorem 1.2}. 

\section{Appendix}
\subsection{Preliminaries}\label{Preliminaries} 
The purpose of this section is to collect many useful lemmas which were used in the proof of our main result. 

\begin{lemma}\label{Lemma 3.1}
\rm{(\cite[Lemma 3.2]{ZZ17})} Let $P_{t}^{\epsilon} \triangleq e^{t \Delta_{\epsilon}}$ and $f \in \mathcal{C}^{\alpha}$ for $\alpha \in \mathbb{R}$. Then for every $\delta \geq 0, \kappa > 0$ and $t > 0$, 
\begin{equation*}
\sup_{\epsilon \in (0,1)} \lVert P_{t}^{\epsilon} f \rVert_{\mathcal{C}^{\alpha + \delta - \kappa}} \lesssim t^{- \frac{\delta}{2}} \lVert f \rVert_{\mathcal{C}^{\alpha}}.
\end{equation*}
\end{lemma}

\begin{lemma}\label{Lemma 3.2}
\rm{(\cite[Lemma 3.3]{ZZ17})} Let $P_{t}^{\epsilon} \triangleq e^{t \Delta_{\epsilon}}, P_{t} \triangleq e^{t \Delta}$, $f \in \mathcal{C}^{\alpha + \eta}$ for $\alpha \in \mathbb{R}$ and $\eta \in (0,1)$. Then for every $\delta \geq 0, \kappa > 0$ and $\epsilon \in (0,1)$, 
\begin{equation*}
\lVert (P_{t}^{\epsilon} - P_{t}) f \rVert_{\mathcal{C}^{\alpha + \delta - \kappa}} \lesssim \epsilon^{\eta} t^{-\frac{\delta}{2}} \lVert f \rVert_{\mathcal{C}^{\alpha + \eta}}. 
\end{equation*} 
\end{lemma}

\begin{lemma}\label{Lemma 3.3}
\rm{(\cite[Lemma 3.7]{ZZ17})} Let $f \in \mathcal{C}^{\alpha + 1}$ for some $\alpha \in \mathbb{R}$. Then for every $\kappa > 0$, 
\begin{equation*}
\sup_{\epsilon \in (0,1)} \lVert D_{j}^{\epsilon} f \rVert_{\mathcal{C}^{\alpha - \kappa}} \lesssim \lVert f \rVert_{\mathcal{C}^{\alpha + 1}}. 
\end{equation*} 
\end{lemma}

 \begin{lemma}\label{Lemma 3.4}
\rm{(\cite[Lemma A.2]{GIP15}, \cite[Lemma 3.1]{ZZ15})} Let $1 \leq p_{1} \leq p_{2} \leq \infty, d \in \mathbb{N}$ and $1 \leq q_{1} \leq q_{2} \leq \infty$, and $\alpha \in \mathbb{R}$. Then $B_{p_{1}, q_{1}}^{\alpha} (\mathbb{T}^{d})$ is continuously embedded in $B_{p_{2}, q_{2}}^{\alpha - d (\frac{1}{p_{1}} - \frac{1}{p_{2}})}(\mathbb{T}^{d})$. 
\end{lemma} 

\begin{lemma}\label{Lemma 3.5}
\rm{(\cite[Lemma 2.4]{GIP15}, \cite[Lemma 3.3]{ZZ15})} Suppose $\alpha \in (0, 1), \beta, \gamma \in \mathbb{R}$ satisfy $\alpha + \beta + \gamma > 0$ and $\beta + \gamma < 0$. Then for smooth $f, g, h$, the tri-linear operator 
\begin{equation*}
C(f,g,h) \triangleq \pi_{0} ( \pi_{<} (f,g), h) - f\pi_{0} (g,h) 
\end{equation*} 
satisfies 
\begin{equation*}
\lVert C(f,g,h) \rVert_{\mathcal{C}^{\alpha + \beta + \gamma}} \lesssim \lVert f \rVert_{\mathcal{C}^{\alpha}} \lVert g \rVert_{\mathcal{C}^{\beta}} \lVert h \rVert_{\mathcal{C}^{\gamma}}, 
\end{equation*} 
and thus $C$ can be uniquely extended to a bounded tri-linear operator in $L^{3} ( \mathcal{C}^{\alpha}(\mathbb{T}^{3}) \times \mathcal{C}^{\beta}(\mathbb{T}^{3}) \times \mathcal{C}^{\gamma}(\mathbb{T}^{3}), \mathcal{C}^{\alpha + \beta + \gamma}(\mathbb{T}^{3}))$. 
\end{lemma} 

\begin{lemma}\label{Lemma 3.6}
\rm{(\cite[Lemma 3.4]{ZZ15})} Let $\mathcal{P}$ be the Leray projection, $f \in \mathcal{C}^{\alpha}(\mathbb{T}^{3})$, $g \in \mathcal{C}^{\beta}(\mathbb{T}^{3})$ for $\alpha < 1$ and $\beta \in \mathbb{R}$. Then for every $k, l \in \{1,2,3\}$, 
\begin{equation*}
\lVert \mathcal{P}^{kl} \pi_{<} (f,g) - \pi_{<} (f, \mathcal{P}^{kl} g ) \rVert_{\mathcal{C}^{\alpha + \beta}} \lesssim \lVert f \rVert_{\mathcal{C}^{\alpha}} \lVert g \rVert_{\mathcal{C}^{\beta}}. 
\end{equation*} 
\end{lemma} 

\begin{lemma}\label{Lemma 3.7}
\rm{(\cite[Lemma 3.6]{ZZ15})} Let $\mathcal{P}$ be the Leray projection and $f \in \mathcal{C}^{\alpha}(\mathbb{T}^{d})$ for $\alpha \in \mathbb{R}$. Then for every $k, l \in \{1,2,3 \}$, 
\begin{equation*}
\lVert \mathcal{P}^{kl} f \rVert_{\mathcal{C}^{\alpha}} \lesssim \lVert f \rVert_{\mathcal{C}^{\alpha}}. 
\end{equation*} 
\end{lemma} 

\begin{lemma}\label{Lemma 3.8}
\rm{(\cite[Lemma 3.11]{ZZ15})} Let $\mathcal{P}$ be the Leray projection. Then for any $\eta \in (0,1), i, j, l \in \{1,2,3\}$ and $t > 0$, 
\begin{equation*}
\lvert e^{ - \lvert k_{12} \rvert^{2} t } k_{12}^{i} \hat{\mathcal{P}}^{jl} (k_{12}) - e^{- \lvert k_{2} \rvert^{2} t} k_{2}^{i} \hat{\mathcal{P}}^{jl} (k_{2}) \rvert \lesssim \lvert k_{1} \rvert^{\eta} \lvert t \rvert^{ - \frac{ (1-\eta)}{2}}. 
\end{equation*} 
\end{lemma} 

\begin{lemma}\label{Lemma 3.9}
\rm{(\cite[Lemma 4.3]{ZZ17})}  Let $\tilde{f}$ and $g$ be defined by Definition \ref{Definition of approximation}. For any $\eta \in (0,1), 0 < s \leq t,  i, j, l \in \{1,2,3\}$, if $\lvert \epsilon k_{12}^{i} \rvert \leq 3 L_{0}, \lvert \epsilon k_{2}^{i} \rvert \leq 3 L_{0}$, then 
\begin{align*}
&\lvert e^{- \lvert k_{12} \rvert^{2} \tilde{f}(\epsilon k_{12}) (t-s)} k_{12}^{i} g(\epsilon k_{12}^{i}) \hat{\mathcal{P}}^{jl} (k_{12}) - e^{- \lvert k_{2} \rvert^{2} \tilde{f}(\epsilon k_{2}) (t-s) } k_{2}^{i} g(\epsilon k_{2}^{i}) \hat{\mathcal{P}}^{jl} (k_{2}) \rvert\\
\lesssim& \lvert k_{1} \rvert^{\eta} \lvert t-s \rvert^{- \frac{1-\eta}{2}}. 
\end{align*} 
\end{lemma}

\begin{lemma}\label{Lemma 3.10}
\rm{(\cite[Lemma A.7]{GIP15}, \cite[Lemma 3.5]{ZZ15})} Let $P_{t}$ be the heat semigroup on $\mathbb{T}^{d}$. Then for $f \in \mathcal{C}^{\alpha}(\mathbb{T}^{3}), \alpha \in \mathbb{R}$ and $\delta \geq 0$, $P_{t} f$ satisfies 
\begin{equation*}
\lVert P_{t} f \rVert_{\mathcal{C}^{\alpha + \delta}} \lesssim t^{- \frac{\delta}{2}} \lVert f \rVert_{\mathcal{C}^{\alpha}}. 
\end{equation*} 
\end{lemma} 

\begin{lemma}\label{Lemma 3.11}
\rm{(\cite[Lemma 2.7]{ZZ17})} Let $f \in \mathcal{C}^{\alpha}$ for some $\alpha < 1$ and $g \in \mathcal{C}^{\beta}$ for some $\beta \in \mathbb{R}$. Then for $\delta \geq \alpha + \beta$, 
\begin{equation*}
\lVert P_{t} \pi_{<}(f,g) - \pi_{<}(f, P_{t} g) \rVert_{\mathcal{C}^{\delta}} \lesssim t^{\frac{ \alpha + \beta - \delta}{2}} \lVert f \rVert_{\mathcal{C}^{\alpha}} \lVert g \rVert_{\mathcal{C}^{\beta}}. 
\end{equation*} 
\end{lemma}

\begin{lemma}\label{Lemma 3.12}
\rm{(\cite[Lemma 2.8]{ZZ17})} Let $f \in \mathcal{C}^{\alpha + \delta}$ for some $\alpha \in \mathbb{R}$, $\delta > 0$ and $I$ an identity operator. Then for every $t \geq 0$, 
\begin{equation*}
\lVert (P_{t} - I) f\rVert_{\mathcal{C}^{\alpha}} \lesssim t^{\frac{\delta}{2}} \lVert f \rVert_{\mathcal{C}^{\alpha + \delta}}. 
\end{equation*}  
\end{lemma}

\begin{lemma}\label{Lemma 3.13}
\rm{(\cite[Lemma 3.10]{ZZ15})} For any $l, m \in (0, d)$ such that $l + m - d > 0$, 
\begin{equation*}
\sum_{k_{1}, k_{2} \in \mathbb{Z}^{d} \setminus \{0\}: k_{1} + k_{2} = k} \frac{1}{\lvert k_{1} \rvert^{l} \lvert k_{2} \rvert^{m}} \lesssim \frac{1}{\lvert k \rvert^{l + m - d}}.
\end{equation*} 
\end{lemma} 

\begin{lemma}\label{Lemma 3.14}
(Bihari's inequality) Let $u, f \geq 0$ be continuous on $[0,\infty)$, $w$ be continuous, non-decreasing on $[0,\infty)$ and satisfy $w(u) > 0$ on $(0,\infty)$. If 
\begin{equation*}
u(t) \leq \alpha + \int_{0}^{t} f(s) w(u(s)) ds 
\end{equation*} 
where $\alpha \geq 0$ is a constant, then 
\begin{equation*}
u(t) \leq G^{-1} (G(\alpha) + \int_{0}^{t} f(s) ds)
\end{equation*} 
for all $t \in [0,T]$ where $G(x) = \int_{x_{0}}^{x} \frac{dy}{w(y)}, x, x_{0} > 0$ and $T$ satisfies $G(\alpha) + \int_{0}^{t} f(s) ds \in Dom(G^{-1})$ for all $t \in [0,T]$. 
\end{lemma}

\begin{lemma}\label{Lemma 3.15}
\rm{(\cite[Lemma 3.8]{ZZ17})} Let $f \in \mathcal{C}^{\alpha + 1 + \eta}$ for some $\alpha \in \mathbb{R}$ and $\eta > 0$. Then for every $\kappa > 0$, 
\begin{equation*}
\lVert (D_{i}^{\epsilon} - D_{i}) f \rVert_{\mathcal{C}^{\alpha - \kappa}} \lesssim \epsilon^{\eta} \lVert f \rVert_{\mathcal{C}^{\alpha + 1 + \eta}}. 
\end{equation*} 
\end{lemma}

The following examples of computations of Wick products have been used extensively:
\begin{example}\label{Example 3.1}
\rm{(\cite[Examples 2.1 and 2.2]{Y19a})}
\begin{align*}
 : \xi_{1} : =& \xi_{1}, \hspace{5mm} : \xi_{1} \xi_{2}: = \xi_{1}\xi_{2} - \mathbb{E} [\xi_{1} \xi_{2}],  \\
 : \xi_{1} \xi_{2} \xi_{3}: =& \xi_{1} \xi_{2} \xi_{3} - \mathbb{E} [ \xi_{2} \xi_{3}] \xi_{1} - \mathbb{E} [\xi_{1} \xi_{3}] \xi_{2} - \mathbb{E} [ \xi_{1} \xi_{2}] \xi_{3}, \\
 : \xi_{1} \xi_{2} \xi_{3} \xi_{4} : =& \xi_{1} \xi_{2} \xi_{3} \xi_{4} - \mathbb{E} [ \xi_{1} \xi_{2}] \xi_{3} \xi_{4} - \mathbb{E} [ \xi_{1} \xi_{3}] \xi_{2} \xi_{4} - \mathbb{E}[\xi_{1} \xi_{4}] \xi_{2} \xi_{3} \\
& - \mathbb{E} [ \xi_{2} \xi_{3}] \xi_{1} \xi_{4} - \mathbb{E} [\xi_{2} \xi_{4}] \xi_{1} \xi_{3} - \mathbb{E} [\xi_{3} \xi_{4}] \xi_{1} \xi_{2} \\
&+ \mathbb{E} [ \xi_{1} \xi_{2}] \mathbb{E} [ \xi_{3} \xi_{4}] + \mathbb{E} [ \xi_{1} \xi_{3}] \mathbb{E} [\xi_{2} \xi_{4}] + \mathbb{E} [ \xi_{1} \xi_{4} ] \mathbb{E} [ \xi_{2} \xi_{3}], 
\end{align*}
and 
\begin{equation*}
\mathbb{E} [ : \xi_{11} \xi_{12}: : \xi_{21} \xi_{22}:] = \mathbb{E} [ \xi_{11} \xi_{21}] \mathbb{E} [ \xi_{12} \xi_{22}] + \mathbb{E} [ \xi_{11} \xi_{22}]\mathbb{E}[\xi_{12} \xi_{21}],
\end{equation*}
\begin{align*} 
 \mathbb{E}  [ : \xi_{11} \xi_{12} \xi_{13}: : \xi_{21} \xi_{22} \xi_{23}:] =& \mathbb{E} [ \xi_{11} \xi_{21}] \mathbb{E} [ \xi_{12} \xi_{22}] \mathbb{E} [ \xi_{13} \xi_{23}]+ \mathbb{E} [ \xi_{11} \xi_{21}] \mathbb{E} [ \xi_{12} \xi_{23}] \mathbb{E} [ \xi_{13} \xi_{22}] \\
+& \mathbb{E} [ \xi_{11} \xi_{22}] \mathbb{E} [ \xi_{12} \xi_{21}] \mathbb{E} [ \xi_{13} \xi_{23}] + \mathbb{E} [ \xi_{11} \xi_{22}] \mathbb{E} [ \xi_{12} \xi_{23}] \mathbb{E} [ \xi_{13} \xi_{21}] \\
+ &\mathbb{E} [ \xi_{11} \xi_{23}] \mathbb{E} [ \xi_{12} \xi_{21}] \mathbb{E} [ \xi_{13} \xi_{22}] + \mathbb{E} [ \xi_{11} \xi_{23}] \mathbb{E} [ \xi_{12} \xi_{22}] \mathbb{E} [ \xi_{13} \xi_{21}], 
\end{align*}
\begin{align*}
&\mathbb{E} [:\xi_{11} \xi_{12} \xi_{13} \xi_{14}: : \xi_{21} \xi_{22} \xi_{23} \xi_{24}:] \\
=& \mathbb{E} [ \xi_{11} \xi_{21}] \mathbb{E} [ \xi_{12} \xi_{22}] \mathbb{E} [ \xi_{13} \xi_{23}] \mathbb{E} [ \xi_{14} \xi_{24}]  +  \mathbb{E} [ \xi_{11} \xi_{21}] \mathbb{E} [ \xi_{12} \xi_{22}] \mathbb{E} [ \xi_{13} \xi_{24}] \mathbb{E} [ \xi_{14} \xi_{23}] \\
+& \mathbb{E} [ \xi_{11} \xi_{21}] \mathbb{E} [ \xi_{12} \xi_{23}] \mathbb{E} [ \xi_{13} \xi_{22}] \mathbb{E} [ \xi_{14} \xi_{24}]  +  \mathbb{E} [ \xi_{11} \xi_{21}] \mathbb{E} [ \xi_{12} \xi_{23}] \mathbb{E} [ \xi_{13} \xi_{24}] \mathbb{E} [ \xi_{14} \xi_{22}] \\
+& \mathbb{E} [ \xi_{11} \xi_{21}] \mathbb{E} [ \xi_{12} \xi_{24}] \mathbb{E} [ \xi_{13} \xi_{22}] \mathbb{E} [ \xi_{14} \xi_{23}]  +  \mathbb{E} [ \xi_{11} \xi_{21}] \mathbb{E} [ \xi_{12} \xi_{24}] \mathbb{E} [ \xi_{13} \xi_{23}] \mathbb{E} [ \xi_{14} \xi_{22}] \\
+& \mathbb{E} [ \xi_{11} \xi_{22}] \mathbb{E} [ \xi_{12} \xi_{21}] \mathbb{E} [ \xi_{13} \xi_{23}] \mathbb{E} [ \xi_{14} \xi_{24}]  +  \mathbb{E} [ \xi_{11} \xi_{22}] \mathbb{E} [ \xi_{12} \xi_{21}] \mathbb{E} [ \xi_{13} \xi_{24}] \mathbb{E} [ \xi_{14} \xi_{23}] \\
+& \mathbb{E} [ \xi_{11} \xi_{22}] \mathbb{E} [ \xi_{12} \xi_{23}] \mathbb{E} [ \xi_{13} \xi_{21}] \mathbb{E} [ \xi_{14} \xi_{24}]  +  \mathbb{E} [ \xi_{11} \xi_{22}] \mathbb{E} [ \xi_{12} \xi_{23}] \mathbb{E} [ \xi_{13} \xi_{24}] \mathbb{E} [ \xi_{14} \xi_{21}] \\
+& \mathbb{E} [ \xi_{11} \xi_{22}] \mathbb{E} [ \xi_{12} \xi_{24}] \mathbb{E} [ \xi_{13} \xi_{21}] \mathbb{E} [ \xi_{14} \xi_{23}]  +  \mathbb{E} [ \xi_{11} \xi_{22}] \mathbb{E} [ \xi_{12} \xi_{24}] \mathbb{E} [ \xi_{13} \xi_{23}] \mathbb{E} [ \xi_{14} \xi_{21}] \\
+& \mathbb{E} [ \xi_{11} \xi_{23}] \mathbb{E} [ \xi_{12} \xi_{21}] \mathbb{E} [ \xi_{13} \xi_{22}] \mathbb{E} [ \xi_{14} \xi_{24}]  +  \mathbb{E} [ \xi_{11} \xi_{23}] \mathbb{E} [ \xi_{12} \xi_{21}] \mathbb{E} [ \xi_{13} \xi_{24}] \mathbb{E} [ \xi_{14} \xi_{22}] \\
+& \mathbb{E} [ \xi_{11} \xi_{23}] \mathbb{E} [ \xi_{12} \xi_{22}] \mathbb{E} [ \xi_{13} \xi_{21}] \mathbb{E} [ \xi_{14} \xi_{24}]  +  \mathbb{E} [ \xi_{11} \xi_{23}] \mathbb{E} [ \xi_{12} \xi_{22}] \mathbb{E} [ \xi_{13} \xi_{24}] \mathbb{E} [ \xi_{14} \xi_{21}] \\
+& \mathbb{E} [ \xi_{11} \xi_{23}] \mathbb{E} [ \xi_{12} \xi_{24}] \mathbb{E} [ \xi_{13} \xi_{21}] \mathbb{E} [ \xi_{14} \xi_{22}]  +  \mathbb{E} [ \xi_{11} \xi_{23}] \mathbb{E} [ \xi_{12} \xi_{24}] \mathbb{E} [ \xi_{13} \xi_{22}] \mathbb{E} [ \xi_{14} \xi_{21}] \\
+& \mathbb{E} [ \xi_{11} \xi_{24}] \mathbb{E} [ \xi_{12} \xi_{21}] \mathbb{E} [ \xi_{13} \xi_{22}] \mathbb{E} [ \xi_{14} \xi_{23}]  +  \mathbb{E} [ \xi_{11} \xi_{24}] \mathbb{E} [ \xi_{12} \xi_{21}] \mathbb{E} [ \xi_{13} \xi_{23}] \mathbb{E} [ \xi_{14} \xi_{22}] \\
+& \mathbb{E} [ \xi_{11} \xi_{24}] \mathbb{E} [ \xi_{12} \xi_{22}] \mathbb{E} [ \xi_{13} \xi_{21}] \mathbb{E} [ \xi_{14} \xi_{23}]  +  \mathbb{E} [ \xi_{11} \xi_{24}] \mathbb{E} [ \xi_{12} \xi_{22}] \mathbb{E} [ \xi_{13} \xi_{23}] \mathbb{E} [ \xi_{14} \xi_{21}] \\
+& \mathbb{E} [ \xi_{11} \xi_{24}] \mathbb{E} [ \xi_{12} \xi_{23}] \mathbb{E} [ \xi_{13} \xi_{21}] \mathbb{E} [ \xi_{14} \xi_{22}]  +  \mathbb{E} [ \xi_{11} \xi_{24}] \mathbb{E} [ \xi_{12} \xi_{23}] \mathbb{E} [ \xi_{13} \xi_{22}] \mathbb{E} [ \xi_{14} \xi_{21}].
\end{align*}
\end{example} 

The following inequality is standard and was used many times: 
\begin{equation}\label{key estimate}
\sup_{a \in \mathbb{R}} \lvert a \rvert^{r} e^{-a^{2}} \leq c \hspace{3mm} \text{ for all } r \geq 0. 
\end{equation} 
The following inequalities are straight-forward generalizations of \cite[Equations (4.2), (4.3)]{ZZ17}:  for any $\eta \in [0,1]$, any $i \in \{1,2,3\}$, any $\alpha, \beta \in \{1,2,3, 12, 13, 23, 123\}$ (by which e.g. we mean $k_{12} = k_{1} + k_{2}$)
\begin{subequations}
\begin{align}
& \lvert e^{- \lvert k_{\alpha} \rvert^{2}  f(\epsilon k_{\beta})(t-s)} - e^{- \lvert k_{\alpha} \rvert^{2} \lvert t-s \rvert} \rvert \lesssim e^{- \lvert k_{\alpha} \rvert^{2} \bar{c}_{f} \lvert t-s \rvert} \lvert \epsilon k_{\beta} \rvert^{\frac{\eta}{2}}, \label{[Equation (4.2)][ZZ17]}\\
& \lvert g(\epsilon k_{\beta}^{i}) - i \rvert \lesssim \lvert \epsilon k_{\beta}^{i} \rvert^{\frac{\eta}{2}}. \label{[Equation (4.3)][ZZ17]} 
\end{align}
\end{subequations} 

\subsection{Detailed computations}\label{Subsection 3.2}
\subsubsection{Details of the proof of Proposition \ref{Proposition on F}}
Here, we leave the detailed computations of \eqref{[Equation (3.10)][ZZ17]} - \eqref{[Equation (3.17e)][ZZ17]}. In order to prove \eqref{[Equation (3.10)][ZZ17]}, we estimate $\lVert (\bar{F}_{u}^{\epsilon}, \bar{F}_{b}^{\epsilon}) (t) \rVert_{\mathcal{C}^{\delta}}$ similarly to \eqref{estimate 109}-\eqref{estimate 111} with $\lambda = \delta$. First, we estimate from \eqref{III1} with $\lambda = \delta$ similarly to \eqref{[Equation (3.7a)][ZZ17]}, 
\begin{align}\label{estimate 112}
III_{t}^{1} \lesssim& \bar{C}_{W}^{\epsilon}(t)  \int_{0}^{t} (t-s)^{- \frac{ 2 - \kappa}{2}} \lVert \bar{b}_{3}^{\epsilon}(s) + \bar{b}_{4}^{\epsilon}(s) - (\bar{b}_{3}^{\epsilon} + \bar{b}_{4}^{\epsilon})(t)\rVert_{\mathcal{C}^{- \frac{1}{2} + \frac{3\delta}{2} + \kappa}}  ds
\end{align} 
by Lemmas \ref{Lemma 3.10} and \ref{[Lemma 1.1][Y19a]} (2), \eqref{[Equation (3.3c)][ZZ17]}, \eqref{[Equation (3.2g)][ZZ17]} and \eqref{[Equation (3.2e)][ZZ17]}. Within \eqref{estimate 112} we estimate for $s \in (0,t)$ and $i \in \{1,2,3\}$, similarly to \eqref{[Equation (3.7b)][ZZ17]}, 
\begin{align}\label{estimate 113}
& \lVert \bar{b}_{3}^{\epsilon, i}(t) + \bar{b}_{4}^{\epsilon, i}(t) - \bar{b}_{3}^{\epsilon, i}(s) - \bar{b}_{4}^{\epsilon, i}(s) \rVert_{\mathcal{C}^{-\frac{1}{2} + \frac{3\delta}{2} + \kappa}} \nonumber\\
\lesssim& (t-s)^{b_{0}} s^{- \frac{ - \frac{1}{2} + \frac{3\delta}{2} + \kappa + 2b_{0} + z}{2}} \lVert b_{0} - \bar{b}_{1}^{\epsilon} (0) \rVert_{\mathcal{C}^{-z}} \nonumber\\
&+ (t-s)^{\frac{b}{2}} \int_{0}^{s} (s-r)^{- \frac{1+ 2 \delta + \frac{3\kappa}{2} + b}{2}} \lVert \bar{G}^{\epsilon}(r) \rVert_{\mathcal{C}^{- \frac{3}{2} - \frac{\delta}{2} - \frac{\kappa}{2}}} dr \nonumber\\
&+ \int_{s}^{t} (t-r)^{- \frac{1+ 2 \delta + \frac{3\kappa}{2}}{2}} \lVert \bar{G}^{\epsilon} (r) \rVert_{\mathcal{C}^{- \frac{3}{2} - \frac{\delta}{2} - \frac{\kappa}{2}}} dr  
\end{align} 
by \eqref{[Equation (3.1ac)][ZZ17]}, \eqref{[Equation (3.2)][ZZ17]}, \eqref{[Equation (3.7c)][ZZ17]}, Lemmas \ref{Lemma 3.12} and \ref{Lemma 3.10}.  Applying \eqref{estimate 113} to \eqref{estimate 112} gives  
\begin{align}\label{estimate 114}
III_{t}^{1} \lesssim& \bar{C}_{W}^{\epsilon}(t) [\int_{0}^{t} (t-s)^{-1 + \frac{\kappa}{2} + b_{0}} s^{- \frac{ - \frac{1}{2} + \frac{3\delta}{2} + \kappa + 2 b_{0} + z}{2}} \lVert b_{0} - \bar{b}_{1}^{\epsilon}(0) \rVert_{\mathcal{C}^{-z}} ds \nonumber\\
&+ \int_{0}^{t} (t-s)^{-1+ \frac{\kappa}{2} + \frac{b}{2}} \int_{0}^{s} (s-r)^{- \frac{ 1 + 2 \delta + \frac{3\kappa}{2} + b}{2}} \lVert \bar{G}^{\epsilon}(r) \rVert_{\mathcal{C}^{-\frac{3}{2} - \frac{\delta}{2} - \frac{\kappa}{2}}} dr ds \nonumber\\
&+ \int_{0}^{t} (t-s)^{-1+ \frac{\kappa}{2}} \int_{s}^{t} (t-r)^{- \frac{1+ 2 \delta + \frac{3\kappa}{2}}{2}} \lVert \bar{G}^{\epsilon}(r) \rVert_{\mathcal{C}^{- \frac{3}{2} - \frac{\delta}{2} - \frac{\kappa}{2}}} dr ds]. 
\end{align} 
We can now apply Fubini theorem similarly to \eqref{estimate 92}-\eqref{estimate 91} to deduce from \eqref{estimate 114}, together with \eqref{[Equation (3.7d)][ZZ17]},  
\begin{align}\label{estimate 115}
III_{t}^{1} &\lesssim \bar{C}_{W}^{\epsilon}(t) [t^{\frac{1}{4} - \frac{3\delta}{4} - \frac{z}{2}} \lVert b_{0} - \bar{b}_{1}^{\epsilon} (0) \rVert_{\mathcal{C}^{-z}}  \\
& \hspace{6mm} + \int_{0}^{t} (t-r)^{-\frac{1}{2} -\delta - \frac{\kappa}{2}} \nonumber\\
& \hspace{8mm} \times [(1+ (\bar{C}_{W}^{\epsilon}(t))^{3})(1+ \lVert \bar{y}^{\epsilon,\sharp}(r) \rVert_{\mathcal{C}^{\frac{1}{2} + \beta}} + \lVert \bar{y}_{4}^{\epsilon}(r) \rVert_{\mathcal{C}^{\frac{1}{2} - \delta_{0}}}) + \lVert \bar{y}_{4}^{\epsilon} (r)\rVert_{\mathcal{C}^{\delta}}^{2}]dr\nonumber  \\
&\hspace{6mm} + t^{\frac{\kappa}{2}} \int_{0}^{t} (t-r)^{- \frac{1+ 2\delta+ \frac{3\kappa}{2}}{2}} \nonumber\\
& \hspace{8mm} \times [(1+ (\bar{C}_{W}^{\epsilon}(t))^{3})(1+ \lVert \bar{y}^{\epsilon,\sharp}(r) \rVert_{\mathcal{C}^{\frac{1}{2} + \beta}} + \lVert \bar{y}_{4}^{\epsilon}(r) \rVert_{\mathcal{C}^{\frac{1}{2} - \delta_{0}}}) + \lVert \bar{y}_{4}^{\epsilon} (r)\rVert_{\mathcal{C}^{\delta}}^{2}]dr ].\nonumber 
\end{align} 
Second, similarly to \eqref{[Equation (3.7)][ZZ17]}, we estimate from \eqref{III2} with $\lambda = \delta$, 
\begin{align}\label{estimate 116}
III_{t}^{2} \lesssim \bar{C}_{W}^{\epsilon} (t) t^{\frac{1}{4} - \frac{\delta}{4}} \lVert (\bar{b}_{3}^{\epsilon} + \bar{b}_{4}^{\epsilon})(t) \rVert_{\mathcal{C}^{\delta}} 
\end{align} 
by Lemmas \ref{Lemma 3.7} and \ref{Lemma 3.11}, \eqref{[Equation (3.2g)][ZZ17]} and \eqref{[Equation (3.2e)][ZZ17]}. With \eqref{estimate 115}-\eqref{estimate 116}, we conclude \eqref{[Equation (3.10)][ZZ17]}. Next, in order to prove \eqref{[Equation (3.17b)][ZZ17]} we restart again from \eqref{III1} with $\lambda = \frac{1}{2} - \delta_{0}$ similarly to \eqref{[Equation (3.7a)][ZZ17]}, 
\begin{align}\label{estimate 117}
III_{t}^{1} \lesssim& \int_{0}^{t} (t-s)^{- \frac{2-\kappa}{2}} \lVert \bar{b}_{3}^{\epsilon}(s) + \bar{b}_{4}^{\epsilon} (s) - (\bar{b}_{3}^{\epsilon}(t) + \bar{b}_{4}^{\epsilon}(t)) \rVert_{\mathcal{C}^{-\delta_{0} + \frac{\delta}{2} + \kappa}} \nonumber\\
& \hspace{50mm} \times \lVert \bar{u}_{1}^{\epsilon}(s) \rVert_{\mathcal{C}^{-\frac{1}{2} - \frac{\delta}{2}}} 
\end{align} 
by Lemmas \ref{Lemma 3.10} and \ref{[Lemma 1.1][Y19a]} (2) and \eqref{[Equation (3.3c)][ZZ17]}. Furthermore, we estimate similarly to \eqref{[Equation (3.7b)][ZZ17]}, 
\begin{align}\label{estimate 118}
& \lVert \bar{b}_{3}^{\epsilon, i}(t) + \bar{b}_{4}^{\epsilon, i}(t) - \bar{b}_{3}^{\epsilon, i} (s) - \bar{b}_{4}^{\epsilon, i}(s) \rVert_{\mathcal{C}^{-\delta_{0} + \frac{\delta}{2} + \kappa}} \\
\lesssim& (t-s)^{b_{0}} s^{- \frac{ - \delta_{0} + \frac{\delta}{2} + \kappa + 2b_{0} + z}{2}} \lVert b_{0} - \bar{b}_{1}^{\epsilon}(0) \rVert_{\mathcal{C}^{-z}} \nonumber\\
& \hspace{10mm} + (t-s)^{\frac{b}{2}} \int_{0}^{s} (s-r)^{- \frac{ - \delta_{0} + \delta + \frac{3\kappa}{2}+ b + \frac{3}{2}}{2}}  \lVert \bar{G}^{\epsilon}(r) \rVert_{\mathcal{C}^{-\frac{3}{2} - \frac{\delta}{2} - \frac{\kappa}{2}}} dr \nonumber\\
& \hspace{10mm} + \int_{s}^{t} (t-r)^{- \frac{ - \delta_{0} + \delta + \frac{3\kappa}{2} + \frac{3}{2}}{2}} \lVert \bar{G}^{\epsilon}(r) \rVert_{\mathcal{C}^{- \frac{3}{2} - \frac{\delta}{2} - \frac{\kappa}{2}}} dr \nonumber
\end{align} 
by \eqref{[Equation (3.1ac)][ZZ17]}, \eqref{[Equation (3.2)][ZZ17]}, \eqref{b and b0}, Lemmas \ref{Lemma 3.12} and  \ref{Lemma 3.10}. Applying \eqref{estimate 118} to \eqref{estimate 117} gives 
\begin{align}\label{estimate 119}
III_{t}^{1} \lesssim& \bar{C}_{W}^{\epsilon}(t) [ \int_{0}^{t} (t-s)^{- 1 + \frac{\kappa}{2} + b_{0}} s^{- \frac{ - \delta_{0} + \frac{\delta}{2} + \kappa + 2b_{0} + z}{2}} \lVert b_{0} - \bar{b}_{1}^{\epsilon}(0) \rVert_{\mathcal{C}^{-z}} ds \\
&+ \int_{0}^{t} (t-s)^{-1 + \frac{\kappa}{2} + \frac{b}{2}} \int_{0}^{s} (s-r)^{- \frac{ - \delta_{0} + \delta + \frac{3\kappa}{2} + b + \frac{3}{2}}{2}} \lVert \bar{G}^{\epsilon}(r) \rVert_{\mathcal{C}^{- \frac{3}{2} - \frac{\delta}{2} - \frac{\kappa}{2}}} dr ds\nonumber\\
&+ \int_{0}^{t} (t-s)^{-1 + \frac{\kappa}{2}} \int_{s}^{t} (t-r)^{- \frac{ - \delta_{0} + \delta + \frac{3\kappa}{2} + \frac{3}{2}}{2}} \lVert \bar{G}^{\epsilon}(r) \rVert_{\mathcal{C}^{-\frac{3}{2} - \frac{\delta}{2} - \frac{\kappa}{2}}} drds]. \nonumber 
\end{align} 
Similarly to \eqref{estimate 92}-\eqref{estimate 91}, we apply Fubini theorem to deduce 
\begin{align}\label{estimate 120}
III_{t}^{1} &\lesssim  \bar{C}_{W}^{\epsilon}(t) [ t^{\frac{ \delta_{0} - z - \frac{\delta}{2}}{2}} \lVert b_{0} - \bar{b}_{1}^{\epsilon}(0) \rVert_{\mathcal{C}^{-z}} \\
&+ \int_{0}^{t} (t-r)^{\frac{\delta_{0}}{2} - \frac{\delta}{2} - \frac{\kappa}{4} - \frac{3}{4}} \nonumber\\
& \hspace{5mm} \times [ (1+ ( \bar{C}_{W}^{\epsilon}(t))^{3})(1+ \lVert y^{\epsilon, \sharp}(r) \rVert_{\mathcal{C}^{\frac{1}{2} + \beta}} + \lVert \bar{y}_{4}^{\epsilon}(r) \rVert_{\mathcal{C}^{\frac{1}{2} - \delta_{0}}} ) + \lVert \bar{y}_{4}^{\epsilon}(r)\rVert_{\mathcal{C}^{\delta}}^{2}] dr \nonumber \\
&+ t^{\frac{\kappa}{2}} \int_{0}^{t} (t-r)^{\frac{ \delta_{0}}{2} - \frac{\delta}{2} - \frac{3\kappa}{4} - \frac{3}{4}} \nonumber\\
& \hspace{5mm} \times  [ (1+ ( \bar{C}_{W}^{\epsilon}(t))^{3})(1+ \lVert y^{\epsilon, \sharp}(r) \rVert_{\mathcal{C}^{\frac{1}{2} + \beta}} + \lVert \bar{y}_{4}^{\epsilon}(r)\rVert_{\mathcal{C}^{\frac{1}{2} - \delta_{0}}} ) + \lVert \bar{y}_{4}^{\epsilon}(r)\rVert_{\mathcal{C}^{\delta}}^{2}] dr] \nonumber 
\end{align} 
by \eqref{[Equation (3.7d)][ZZ17]}. Second, similarly to \eqref{[Equation (3.7)][ZZ17]}, we estimate from \eqref{III2} with $\lambda = \frac{1}{2} - \delta_{0}$, 
\begin{align}\label{estimate 121}
III_{t}^{2} 
\lesssim \bar{C}_{W}^{\epsilon}(t) t^{\frac{1}{4} - \frac{\delta}{4}} \lVert (\bar{b}_{3}^{\epsilon} + \bar{b}_{4}^{\epsilon})(t) \rVert_{\mathcal{C}^{\frac{1}{2} - \delta_{0}}} 
\end{align} 
by Lemmas \ref{Lemma 3.7} and \ref{Lemma 3.11} and \eqref{[Equation (3.2e)][ZZ17]}. Thus, we now conclude \eqref{[Equation (3.17b)][ZZ17]} due to \eqref{estimate 120}-\eqref{estimate 121}. Finally, we prove \eqref{[Equation (3.17e)][ZZ17]}. We restart from \eqref{III1} with $\lambda = -z$,
\begin{align}\label{estimate 122}
III_{t}^{1}\lesssim& \bar{C}_{W}^{\epsilon}(t) [\int_{0}^{t} (t-s)^{- \frac{ \frac{3}{2} - z + \frac{\delta}{2} + \kappa}{2}} ( \bar{C}_{W}^{\epsilon} (t) + \lVert \bar{b}_{4}^{\epsilon}(s) \rVert_{\mathcal{C}^{\delta}}) ds \nonumber\\
& \hspace{30mm} + t^{\frac{1}{4} + \frac{z}{2} - \frac{\delta}{4} - \frac{\kappa}{2}} (\bar{C}_{W}^{\epsilon}(t) + \lVert \bar{b}_{4}^{\epsilon}(t) \rVert_{\mathcal{C}^{\delta}})]
\end{align} 
by Lemmas \ref{Lemma 3.10} and \ref{[Lemma 1.1][Y19a]} (2), \eqref{[Equation (3.2e)][ZZ17]}, \eqref{[Equation (3.2g)][ZZ17]}, \eqref{[Equation (3.3c)][ZZ17]} and \eqref{[Equation (3.2h)][ZZ17]}. On the other hand,  we estimate from \eqref{III2} with $\lambda = -z$,  similarly to \eqref{[Equation (3.7)][ZZ17]}, 
\begin{align}\label{estimate 123}
III_{t}^{2} \lesssim \bar{C}_{W}^{\epsilon}(t) \lVert (\bar{b}_{3}^{\epsilon} + \bar{b}_{4}^{\epsilon})(t) \rVert_{\mathcal{C}^{-z}} t^{\frac{1}{4} - \frac{\delta}{4}} 
\end{align} 
by Lemmas \ref{Lemma 3.7} and \ref{Lemma 3.11} and \eqref{[Equation (3.2e)][ZZ17]}. From \eqref{estimate 122}-\eqref{estimate 123}, we conclude \eqref{[Equation (3.17e)][ZZ17]}.

\subsection{Table of essential notations}\label{Subsection 3.3}
\begin{table}[h]
\caption{Symbols}. \label{tab2}
\vspace{1mm}
\begin{tabular}{ ll}
\hline
\hline
Symbols  & Location\\
\hline
$B_{p,q}^{\alpha}(\mathbb{T}^{3})$ & \eqref{Besov space}\\
$\bar{c}_{f}$ & Definition \ref{Definition of approximation}\\
$\bar{C}_{W}^{\epsilon}, C_{W}^{\epsilon}, \delta C_{W}^{\epsilon}$ & respectively \eqref{[Equation (3.2e)][ZZ17]}, \eqref{[Equation (3.23a)][ZZ17]}, \eqref{[Equation (3.25f)][ZZ17]} \\
$\Delta_{\epsilon}, D_{j}^{\epsilon}$ & \eqref{approximation 1}\\
$f, \tilde{f}$ & \eqref{f}\\
$\hat{f} = \mathcal{F}_{\mathbb{T}^{3}}(f)$ & \eqref{Fourier} \\
$g$ & \eqref{g}\\
$h_{u}, h_{b}$ & Definition \ref{Definition of approximation}\\
$H_{u,\epsilon}, H_{b,\epsilon}$ & \eqref{approximation 2}\\
$P_{t}, P_{t}^{\epsilon}$ & Example \ref{Burgers' equation example}\\
$\pi_{<}$, $\pi_{0}$ & \eqref{Bony's decomposition} \\
\hline
\hline
\end{tabular}
\end{table}

\section{Acknowledgements} 

The author expresses deep gratitude to Prof. Carl Mueller for stimulating discussions on various occasions.

\end{document}